\documentclass{amsart}
\usepackage{amssymb}
\usepackage[utf8]{inputenc}
\usepackage[english]{babel}

\theoremstyle{plain}
\newtheorem{theorem}{Theorem}[section]
\newtheorem{corollary}[theorem]{Corollary}
\newtheorem{lemma}[theorem]{Lemma}
\newtheorem{proposition}[theorem]{Proposition}
\theoremstyle{property}

\theoremstyle{definition}
\newtheorem{definition}[theorem]{Definition}
\newtheorem{remark}[theorem]{Remark}

\newtheorem{example}[theorem]{Example}

\newtheorem{notation}[theorem]{Notation}

\numberwithin{equation}{section}
\input{tcilatex}

\begin{document}
\title[complex manifolds with $\mathbb{C}^{\ast }$-action]{Heat kernel and
local index theorem for open complex manifolds with $\mathbb{C}^{\ast }$%
-action}
\author{Jih-Hsin Cheng}
\address{Institute of Mathematics, Academia Sinica and National Center for
Theoretical Sciences, Taipei, Taiwan}
\email{cheng@math.sinica.edu.tw}
\author{Chin-Yu Hsiao}
\address{Department of Mathematics, National Taiwan University, Taipei,
Taiwan}
\email{chinyuhsiao@ntu.edu.tw}
\author{I-Hsun Tsai}
\address{Department of Mathematics, National Taiwan University, Taipei,
Taiwan}
\email{ihtsai@math.ntu.edu.tw}
\thanks{}

\begin{abstract}
For a complex manifold $\Sigma $ with $\mathbb{C}^{\ast }$-action, we define
the $m$-th $\mathbb{C}^{\ast }$ Fourier-Dolbeault cohomology group and
consider the $m$-index on $\Sigma $. By applying the method of \textit{%
transversal} heat kernel asymptotics, we obtain a local index formula for
the $m$-index. We can reinterpret Kawasaki's Hirzebruch-Riemann-Roch formula
for a compact complex orbifold with an orbifold holomorphic line bundle by
our integral formulas over a (smooth) complex manifold and finitely many
complex submanifolds arising from singular strata. We generalize $\mathbb{C}%
^{\ast }$-action to complex reductive Lie group $G$-action on a compact or
noncompact complex manifold. Among others, we study the nonextendability of
open group action and the space of all $G$-invariant holomorphic $p$-forms.
Finally, in the case of two compatible holomorphic $\mathbb{C}^{\ast }$%
-actions, a mirror-type isomorphism is found between two linear spaces of
holomorphic forms, and the Euler characteristic associated with these spaces
can be computed by our $\mathbb{C}^{\ast }$ local index formula on the total
space. In the perspective of the equivariant algebraic cobordism theory $%
\Omega _{\ast }^{\mathbb{C}^{\ast }}(\Sigma ),$ a speculative connection is
remarked. Possible relevance to the recent development in physics and number
theory is briefly mentioned.
\end{abstract}

\maketitle




\begin{center}
{\large {Contents} }
\end{center}

\begin{eqnarray*}
&&1.\text{ Introduction and statement of the results} \\
&&2.\text{ Preliminaries on complex manifolds with }\mathbb{C}^{\ast }\text{%
-action} \\
&&3.\text{ A Hermitian metric on complex manifolds with }\mathbb{C}^{\ast }%
\text{-action} \\
&&4.\text{ A Hodge theory for }\square _{\Sigma ,m}^{(q)} \\
&&5.\text{ Transversally }spin^{c}\text{ Dirac operators} \\
&&6.\text{ Approximation of the transversal heat kernel }e^{-t\tilde{\square}%
_{m}^{c\pm }} \\
&&7.\text{ Asymptotic expansion of the transversal heat kernel} \\
&&8.\text{ Local }m\text{-index formula} \\
&&8.1.\text{ Part I of the local index formula (}k = 0\text{ in (\ref{7-46-1}%
))} \\
&&8.2.\text{ Part II of the local }m\text{-index formula (}k \geq 1\text{ in
(\ref{7-46-1})) via Lefschetz type formulas} \\
&&8.3.\text{ The local index formula completed} \\
&&8.4.\text{ Comparison with Duistermaat's formula for the K\"{a}hler case, }
\end{eqnarray*}
\begin{eqnarray*}
&&\ \ \ \ \ \text{Part I} \text{: from real to complex} \\
&&8.5.\text{ Comparison with Duistermaat's orbifold version of the index
theorem,} \\
&&\ \ \ \ \ \text{Part II} \text{: integrals over fixed point orbifolds} \\
&&9.\text{ Nonextendability of open group action; meromorphic action} \\
&&10.\text{ Complex manifolds with two holomorphic }\mathbb{C}^{\ast }\text{%
-actions}
\end{eqnarray*}

\bigskip

\section{\textbf{Introduction and statement of the results}}

We consider a complex manifold $\Sigma $ of (complex) dimension $n$ with a $%
\mathbb{C}^{\ast }$-action $\sigma (\rho e^{i\theta })$ holomorphic in $%
\mathbb{C}^{\ast }$ and $\Sigma $ jointly. For most cases in this paper, the
complex manifold is open unless specified otherwise. We assume that ($\Sigma
,$ $\sigma (\rho e^{i\theta }))$ satisfies the conditions: the action $%
\sigma $ is proper, the $\mathbb{R}^{+}$ part $\sigma (\rho )$ is globally
free, the $S^{1}$ part $\sigma $(e$^{i\theta }$) is locally free (meaning
the finite isotropy condition at any point of $\Sigma $) and the orbit space 
$\Sigma /\sigma $ (or $\Sigma /\mathbb{C}^{\ast })$ is compact. For
simplicity we sometimes write $z\circ x$ or $zx$ for the action of $\sigma
(z),$ $z\in \mathbb{C}^{\ast },$ on $x$ of $\Sigma .$

Examples satisfying our assumption include $i)$ the total space
(zero-section removed) of a holomorphic line bundle over a compact complex
manifold with the fibre-multiplication as a holomorphic $\mathbb{C}^{\ast }$%
-action and $ii)$ $X\times \mathbb{R}^{+}$ as $\Sigma $ where $X$ is a
compact CR manifold with a transversal CR locally free $S^{1}$-action,
endowed with the naturally induced complex structure and holomorphic $%
\mathbb{C}^{\ast }$-action. See Section 2 for more details. More
sophisticated and complete examples have been described by D. Gross \cite%
{Gro}, with the previous works by A. Bialynicki-Birula, A. Sommese, J.
Swiecieka and J. B. Carrell \cite{BBS}, \cite{BBSw}, \cite{CS}. See also the
survey monograph \cite{BBCM}.

Let $\Omega ^{0,q}(\Sigma )$ denote the space of all $C^{\infty }$ $(0,q)$%
-forms on $\Sigma $. For any integer $m,$ we define%
\begin{equation}
\hat{\Omega}_{m}^{0,q}(\Sigma ):=\{\omega \in \Omega ^{0,q}(\Sigma ):\sigma
(\lambda )^{\ast }\omega =\lambda ^{m}\omega \text{ for all }\lambda \in 
\mathbb{C}^{\ast }\}.  \label{1-a}
\end{equation}

\noindent Observe that $\sigma (\lambda )^{\ast }\circ \bar{\partial}$ $=$ $%
\bar{\partial}\circ \sigma (\lambda )^{\ast }$ on $\Omega ^{0,q}(\Sigma )$
since $\sigma (\lambda )$ is holomorphic, so that $\bar{\partial}_{\Sigma
,m}:=\bar{\partial}:\hat{\Omega}_{m}^{0,q}(\Sigma )\rightarrow \hat{\Omega}%
_{m}^{0,q+1}(\Sigma ).$

In this paper, with the appropriate \textit{regularity condition} we
consider only the subspace $\Omega _{m}^{0,q}(\Sigma )$ $\subset $ $\hat{%
\Omega}_{m}^{0,q}(\Sigma )$ (see Definition \ref{2m}). It follows that%
\begin{equation}
\bar{\partial}_{\Sigma ,m}:\Omega _{m}^{0,q}(\Sigma )\rightarrow \Omega
_{m}^{0,q+1}(\Sigma )  \label{1-2-2}
\end{equation}

\noindent for all $q,$ $0\leq q\leq n-1$. We can then define the cohomology
group%
\begin{equation}
H_{m}^{q}(\Sigma ,\mathcal{O}):=\frac{\text{Ker }\bar{\partial}_{\Sigma
,m}:\Omega _{m}^{0,q}(\Sigma )\rightarrow \Omega _{m}^{0,q+1}(\Sigma )}{%
\func{Im}\text{ }\bar{\partial}_{\Sigma ,m}:\Omega _{m}^{0,q-1}(\Sigma
)\rightarrow \Omega _{m}^{0,q}(\Sigma )}  \label{Hmq}
\end{equation}%
\noindent and we call it the $m$-th $\mathbb{C}^{\ast }$ \textit{%
Fourier-Dolbeault cohomology group}. Let $h_{m}^{q}(\Sigma ,\mathcal{O})$
denote the dimension of $H_{m}^{q}(\Sigma ,\mathcal{O}).$ We define the
index of the $\bar{\partial}_{\Sigma ,m}$-complex as follows (once $%
h_{m}^{q}<\infty $ is established; see Theorem \ref{t-4-2}):%
\begin{equation}
index(\bar{\partial}_{\Sigma ,m}\text{-complex)}:=%
\sum_{q=0}^{n}(-1)^{q}h_{m}^{q}(\Sigma ,\mathcal{O})  \label{1-2-3}
\end{equation}

\noindent which is metric independent.

For every $m$ $\in $ $\mathbb{Z}$ the natural map $H_{m}^{q}(\Sigma ,%
\mathcal{O})$ $\rightarrow $ $H^{q}(\Sigma ,\mathcal{O})$ into the usual
Dolbeault cohomology group is not expected to be injective in general; $%
H_{m}^{q}(\Sigma ,\mathcal{O})$ is not going to be considered as an $m$-th
component of $H^{q}(\Sigma ,\mathcal{O}).$ Compare Proposition \ref{projm}
and Remark \ref{6-3-5a}.

To show that $H_{m}^{q}(\Sigma ,\mathcal{O})$ is finite-dimensional, we
define a (non $\mathbb{C}^{\ast }$-invariant) Hermitian metric $G_{a,m}$ on $%
\Sigma $ for any fixed $m$ $\in $ $\mathbb{N\cup \{}0\}$ where $a>\frac{1}{2}%
m\geq 0$ (see (\ref{metric}) and Remark \ref{3-r}) and develop a Hodge
theory for the associated (Kodaira) Laplacian $\square _{\Sigma ,m}^{(q)}$ (%
\ref{4.0}). See (\ref{DiracLap}) for a modified version $\tilde{\square}%
_{m}^{c}$ (resp. $\tilde{\square}_{m}^{c\pm })$ of $\square _{\Sigma ,m}$
(resp.$\square _{\Sigma ,m}^{\pm })$; this modification is indispensable to
our approach. As a result, we can express the index of the $\bar{\partial}%
_{\Sigma ,m}$-complex in (\ref{1-2-3}) as%
\begin{eqnarray*}
index(\bar{\partial}_{\Sigma ,m}\text{-complex)} &=&\sum_{q:even}\dim
Ker\square _{\Sigma ,m}^{(q)}-\sum_{q:odd}\dim Ker\square _{\Sigma ,m}^{(q)}
\\
&=&\dim Ker\square _{\Sigma ,m}^{+}-\dim Ker\square _{\Sigma ,m}^{-} \\
&=&\dim Ker\tilde{\square}_{m}^{c+}-\dim Ker\tilde{\square}_{m}^{c-}
\end{eqnarray*}%
\noindent (see Corollary \ref{t-4-3}, Lemma \ref{l-5-3}, and Theorem \ref%
{t-5-1}). For all of these we construct certain $L^{2}$-spaces called $m$%
-spaces, and show that these $m$-spaces are non-trivial (Remark \ref{6-3-5a} 
$i)).$

Remark that the cohomology group (\ref{Hmq}) is metric-independent and
meaningful for any integers $m.$ But our approach starts with a fixed $m$ $%
\in $ $\mathbb{N\cup \{}0\mathbb{\}}$ (see Remark \ref{r1-1} below for $m<0$%
) and constructs the metric $G_{a,m}$ with a parameter \textquotedblleft $a$%
" adapting to $m.$ As $m$ varies and thus the metric $G_{a,m}$ might vary,
there does not exist a fixed $L^{2}$-space (with respect to a fixed metric)
that can simultaneously accommodate all these \textquotedblleft $m$%
-components"; compare Remark \ref{6-3-5a} $ii)$. This is one of the features
that distinguish the $\mathbb{C}^{\ast }$-action from our previous $S^{1}$%
-action \cite{CHT} (whose $m$-th Fourier components can naturally embed into
a fixed $L^{2}$-space and span (over $m$ $\in $ $\mathbb{Z}$) the whole
space).

We can extend the above setting to the bundle case. For later use we remark
that we can approximate the heat kernel of $\tilde{\square}_{m}^{c}$ by a
more manageable quantity $P_{m,t}^{0}$ (see (\ref{1.9-5})).

With respect to the locally free action $\sigma (e^{i\theta })$, we can talk
about the period of a point. We say $\frac{2\pi }{l}$ is the period of a
point $x$ if $l$ $=$ $\max \{$ $l^{\prime }$ $\in $ $N:$ $e^{i\frac{2\pi }{%
l^{\prime }}}\circ x$ $=$ $x\}.$ Let $\frac{2\pi }{p}$ be the largest period.

Let $L_{\Sigma }$ be the holomorphic line bundle over $\Sigma ,$ whose fibre
at $q\in \Sigma $ consists of tangents to the $\mathbb{C}^{\ast }$-orbit
through $q$ (see the lines above (\ref{3-0})). We take a $\mathbb{C}^{\ast }$%
-invariant Hermitian (fibre) metric $||\cdot ||$ on $L_{\Sigma }$ (see Step
1 in Section 3). Define the first Chern form $c_{1}(L_{\Sigma },||\cdot ||)$
of $L_{\Sigma }$ with respect to $||\cdot ||$. Note that $L_{\Sigma }$ is a
holomorphic subbundle of $T^{1,0}\Sigma $ (although the metric $||\cdot ||$
is not the induced one)$.$ The $\mathbb{C}^{\ast }$-equivariant quotient
bundle $T^{1,0}\Sigma /L_{\Sigma }$ inherits a $\mathbb{C}^{\ast }$%
-invariant metric $g_{quot}$ from the aforementioned metric $G_{a,m},$ which
is isometric to $\pi ^{\ast }g_{M}$ (see (\ref{Ga}) or (\ref{metric}), and
Lemma \ref{L-inv})$.$ The $\mathbb{C}^{\ast }$-invariant Todd form $Td_{%
\mathbb{C}^{\ast }}(T^{1,0}\Sigma /L_{\Sigma },g_{quot})$ and similarly the $%
\mathbb{C}^{\ast }$-invariant Chern character form $ch_{\mathbb{C}^{\ast
}}(E,h_{E})$ for a $\mathbb{C}^{\ast }$-equivariant holomorphic vector
bundle $E$ over $\Sigma $ with a $\mathbb{C}^{\ast }$-invariant Hermitian
metric $h_{E}$ can be defined. Finally, define $\delta _{p|m}$ $=$ $1$ if $%
p\mid m$ and $0$ if $p\nmid m.$

We have the following index theorem (Theorem \ref{main_theorem}), which is a 
\textit{local index theorem} in the sense similar to \cite{BGV} that the
index density can be formed and computed from certain heat kernel
formulation on the complex manifold $\Sigma $ (see the discussion below).
Let $\Sigma ^{\tilde{g}}$ denote the singular stratum in $\Sigma ,$
associated to $\tilde{g}$ $\in $ $\mathcal{G}$ $:=$ $\cup _{j}G_{j},$ $%
\tilde{g}\neq 1.$ See (\ref{Sigmak}) and (\ref{sing}) for the definition of $%
\Sigma ^{\tilde{g}}.$ For the associated integrand $\mathcal{F}_{\tilde{g}%
,m}(x)$ below we refer the reader to (\ref{Str-7}) for the definition and (%
\ref{FkTd}) for an expression in terms of Todd genus form and Chern
character form. Let $l(x)$ be as in (\ref{lq0}). Let $dv_{\Sigma ^{\tilde{g}%
},m}$ denote the volume form of $\Sigma ^{\tilde{g}}$ with respect to the
metric induced from $G_{a,m}.$ Let $[\cdot ]_{2n}$ denote the $2n$-form part
of a differential form, where $2n$ $=$ $\dim _{\mathbb{R}}\Sigma .$ Denote
the volume of $\Sigma $ associated to $G_{a,m}$ by $dv_{\Sigma ,m}$ (see (%
\ref{volume})). Define the following index density function $HRR_{m}(\Sigma
,G_{a,m},E)$ of Hirzebruch-Riemann-Roch type by%
\begin{eqnarray*}
&&HRR_{m}(\Sigma ,G_{a,m},E) \\
:= &&\frac{p\delta _{p|m}[Td_{\mathbb{C}^{\ast }}(T^{1,0}\Sigma /L_{\Sigma
},g_{quot})\wedge ch_{\mathbb{C}^{\ast }}(E,h_{E})\wedge
e^{-mc_{1}(L_{\Sigma },||\cdot ||)}\wedge d\hat{v}_{m}]_{2n}}{dv_{\Sigma ,m}}%
.
\end{eqnarray*}

\noindent where $\frac{2\pi }{p}$ is the largest period as aforementioned.
Here $d\hat{v}_{m}$ on $\Sigma $ restricts to a normalized area form for $%
\mathbb{C}^{\ast }$-orbits with the integral equal to 1 (see (\ref{fibrenv_0}%
) and (\ref{fibrenv})).

\textit{Most often we implicitly assume }$p=1$\textit{\ unless specified
otherwise}. (If $p>1$ the $\mathbb{C}^{\ast }$-action is not effective \cite[%
p.175]{Du}, and by redefining the action the new $\mathbb{C}^{\ast }$-action
has $p=1$ (cf. \cite[p.19]{CHT}).)

We compute $index(\bar{\partial}_{\Sigma ,m}$-complex) through the integral
of the supertrace of $P_{m,t}^{0}$ $(Str$ $P_{m,t}^{0}(x,x):=$Tr $%
P_{m,t}^{0,+}(x,x)-$Tr $P_{m,t}^{0,-}(x,x)),$ whose limit as $t\rightarrow 0$
can be expressed in terms of $HRR_{m}(\Sigma ,G_{a,m},E)$ and $\mathcal{F}_{%
\tilde{g},m}(x)$ as follows.

\begin{theorem}
\label{main_theorem} \text{(proved in Subsection \ref{Subs-8-3})} With the
notations above, suppose that $\Sigma $ is an $n$-dimensional (open) complex
manifold with a holomorphic, proper $\mathbb{C}^{\ast }$-action $\sigma
(\rho e^{i\theta }).$ Assume that the $R^{+}$ part $\sigma (\rho )$ is
globally free, the $S^{1}$ part $\sigma $(e$^{i\theta }$) is locally free
and the orbit space $\Sigma /\sigma $ is compact, and that the $\mathbb{C}%
^{\ast }$-action is effective (equivalently the largest period $\frac{2\pi }{%
p}$ above is $2\pi $). Let $(E,h_{E})$ be a $\mathbb{C}^{\ast }$-equivariant
holomorphic vector bundle over $\Sigma .$ Then for every $m\in \{0\}\cup 
\mathbb{N}$

$i)$ it holds that in the space of generalized sections%
\begin{eqnarray}
&&\lim_{t\rightarrow 0}StrP_{m,t}^{0}(x,x)  \label{Str-F} \\
&=&HRR_{m}(\Sigma ,G_{a,m},E)+\sum_{\tilde{g}\in \mathcal{G},\text{ }\tilde{g%
}\neq 1}\tilde{g}^{-m}\overline{\mathcal{F}_{\tilde{g},m}(x)}l^{m}(x)\delta
_{\Sigma ^{\tilde{g}}}\text{ };  \notag
\end{eqnarray}

$ii)$ the following index is well defined and satisfies%
\begin{eqnarray}
&&index(\bar{\partial}_{\Sigma ,m}^{E}\text{-complex})\text{ (}%
=\sum_{q=0}^{n}(-1)^{q}h_{m}^{q}(\Sigma ,\mathcal{O(}E\mathcal{)})\text{ as
in (\ref{1-2-3}))}  \label{MF} \\
&=&\int_{\Sigma }HRR_{m}(\Sigma ,G_{a,m},E)dv_{\Sigma ,m}+\sum_{\tilde{g}\in 
\mathcal{G},\text{ }\tilde{g}\neq 1}\tilde{g}^{m}\int_{\Sigma ^{\tilde{g}}}%
\mathcal{F}_{\tilde{g},m}(x)l^{m}(x)dv_{\Sigma ^{\tilde{g}},m}.  \notag
\end{eqnarray}
\end{theorem}

See Example \ref{E-IF} for an illustration. Even though this example might
be the most basic one, its associated index formula presents an algebraic
identity that does not seem to be easily discovered at first hand; see (\ref%
{ku}), (\ref{km}) and (\ref{8-59-1}).

\begin{remark}
\label{r1-1} Write $H_{m,\sigma }^{q}$ for $H_{m}^{q}$ to indicate the
dependence on the action $\sigma .$ Since $H_{-m,\tilde{\sigma}}^{q}$ $=$ $%
H_{m,\sigma }^{q}$ where the action $\tilde{\sigma}(\lambda )$ $:=$ $\sigma
(\lambda ^{-1})$ for $\lambda $ $\in $ $\mathbb{C}^{\ast }$ of cohomology
groups (with regularity conditions in Definition \ref{2m} where if $%
w_{\sigma },$ $w_{\tilde{\sigma}}$ denote the associated $w$-coordinates
then $w_{\tilde{\sigma}}$ $=$ $w_{\sigma }^{-1}$)$,$ a similar statement of
Theorem \ref{main_theorem} for $m<0$ holds true as well. Details are omitted.
\end{remark}

Remark that the RHS of (\ref{MF}) must be metric-independent, as the LHS is
so. For the verification that the integral in (\ref{MF}) is independent of
the choice of $\mathbb{C}^{\ast }$-invariant Hermitian metrics used to
compute $Td_{\mathbb{C}^{\ast }},$ $ch_{\mathbb{C}^{\ast }},$ and $c_{1},$
see Proposition \ref{pinv}

There is a link between our result and a result of Kawasaki in \cite{Ka} on
Hirzebruch-Riemann-Roch formula over complex orbifolds. Compared to
Kawasaki's, we get a Hirzebruch-Riemann-Roch formula through (\ref{MF}),
i.e. an integral over a complex manifold (with one dimension higher though)
and finitely many integrals over complex submanifolds corresponding to
singular strata $\Sigma _{\text{sing}}$ (see (\ref{1-4}) below). Moreover,
from our heat kernel approach on $\Sigma $ lying over the compact complex
orbifold $\Sigma /\sigma $ we also realize that those terms arising from the
lower-dimensional strata in \cite{Ka} correspond to the integrals over $%
\Sigma ^{\tilde{g}}$ in our situation (cf. (\ref{Fg})), via a Lefschetz type
heat kernel asymptotics on certain local slices $V_{j}$ of $\Sigma $ (see (%
\ref{Asym}), (\ref{Str-k})). Those strata-contribution of Kawasaki \cite{Ka}
are reconsidered by Duistermaat \cite[Sections 14.4, 14.6]{Du} as
integrations over \textquotedblleft fixed point orbifolds" called by him,
which can be mapped but not necessarily embedded, into the original
orbifold. It seems to us that the integral expression (\ref{MF}) is
conceptually simpler. See \cite{PV} and \cite[p.92]{RT} for related results.
For $m$ $=$ $0,$ in comparison with the formula of \cite[(14.3) on p.184]{Du}
it is perhaps interesting to note that our regions of integration $\Sigma ^{%
\tilde{g}}$ (\ref{Sigmak}) in (\ref{MF}) appear firsthand and intrinsic as
they are natural subspaces of the space $\Sigma $ itself, whereas those in 
\cite{Du} denoted by $\tilde{F}$ for the corresponding integrals are
introduced in a somewhat ad hoc manner; see (\ref{Fg}) for an integral
comparison and Subsections \ref{S-8-4}, \ref{S-8-5} for details. When $m$ $%
\neq $ $0,$ the comparison is made indirectly. One needs to convert this $m$%
-index to a \textquotedblleft $0$-index with the extra line bundle $%
(L_{\Sigma }^{\ast })^{\otimes m}$" ($L_{\Sigma }^{\ast }$ denotes the dual
of the forementioned $L_{\Sigma })$ and then compare; see Remark \ref{R-8-41}%
. In short our formula unifies the $\{m$-index$\}_{m}$ into a single
formula, whereas this interpretation is not quite the case with
Duistermaat's formula (unless $(L_{\Sigma }^{\ast })^{\otimes m}$ is added).
Remark that this comparison is in some way troubled by the convention
adopted by Duistermaat himself (see the second paragraph of Subsection \ref%
{S-8-4}). We hope that some clarification (with corrections) of the
interpretations in this and other literature is made here as is the case
with the comparison.

We remark that after Kawasaki's work as mentioned above, some other results
related to index theory on orbifolds were obtained. Among others, X. Ma
studied the analytic torsion and the Quillen metric for an orbifold K\"{a}%
hler fibration in \cite{MaX1}, \cite{MaX2}.

On the way to proving Theorem \ref{main_theorem}, we obtain an asymptotic
expansion for the diagonal of the \textit{transversal} heat kernel $e^{-t%
\tilde{\square}_{m}^{c}}$ (cf. (\ref{5.2})). Or we may regard it as another
principal result of this paper, of which Theorem \ref{main_theorem} may be
viewed as an application. See Theorem \ref{AHKE} below, and Footnote$^{1}$
(the paragraph after Remark \ref{r1-2}) for a comparison with other
approaches.

Let $\frac{2\pi }{p_{j}},$ $p$ $=$ $p_{1}$ $<$ $p_{2}$ $<$ $...<p_{k},$ be
all possible periods of the locally free action $\sigma (e^{i\gamma }).$
Define $\Sigma _{p_{j}}$ $:=$ $\{x\in \Sigma $ $:$ the period of $x$ is $%
\frac{2\pi }{p_{j}}\}$ and 
\begin{equation}
\Sigma _{\text{sing}}:=\cup _{j=2}^{k}\Sigma _{p_{j}}.  \label{1-4}
\end{equation}%
\noindent Here $\frac{2\pi }{p},$ $p$ $=$ $p_{1},$ is the largest period.
Let $\hat{d}(x,\Sigma _{\text{sing}})$ denote a certain distance between $x$
and $\Sigma _{\text{sing}}$ (see (\ref{distsing}) for the definition).

The following is proved in Sections \ref{A-THK} and \ref{AE_THK} (in
paragraphs prior to Remark \ref{7-36.5}). In $iii)$ of the following
theorem, for the meaning of \textquotedblleft $\sim "$ we refer to Remark %
\ref{r1-2} below, where the usual use of $C^{l}$-norm is modified to be
\textquotedblleft $C_{B}^{l}$-norm".

\begin{theorem}
\label{AHKE} i) \text{(Existence and uniqueness) }The heat kernel $e^{-t%
\tilde{\square}_{m}^{c}}$ for $\tilde{\square}_{m}^{c}$ exists and is unique.

ii) \text{(Asymptotic expansion (I)) Let }$x\in \Sigma \backslash \Sigma _{%
\mathrm{sing}}.$ For every $N_{0}\geq N_{0}(n)$ there exist constants $%
C_{N_{0}},$ $\delta =\delta (N_{0})>0$ (both independent of $x)$ and
functions $b_{s}$ (which are given by $b_{s}(z,\zeta )$ of (\ref{OAE}) at $%
z=\zeta ,$ $s$ $=$ $n-1-j$ with $j=0,$ $\cdot \cdot \cdot ,$ $N_{0})$ such
that%
\begin{eqnarray}
&&|e^{-t\tilde{\square}_{m}^{c}}(x,x)-p\delta
_{p|m}\sum_{j=0}^{N_{0}}t^{-(n-1)+j}b_{n-1-j}(z(x))l^{m}(x)|  \label{AHKE1}
\\
&\leq &C_{N_{0}}l^{m}(x)(t^{-(n-1)+N_{0}+1}+t^{-(n-1)}e^{-\frac{\hat{%
\varepsilon}_{0}\hat{d}(x,\Sigma _{\mathrm{sing}})^{2}}{t}})  \notag
\end{eqnarray}%
for $0<t<\delta $ and some constant $\hat{\varepsilon}_{0}>0$ (independent
of $N_{0}$ and $x).$ Here $l(x)$ is as in (\ref{lq0}) and $N_{0}(n)$ is some
explicit function in $n;$ for instance one may take $N_{0}(n)$ $=$ $n+1.$

iii) \text{(Asymptotic expansion (II)) }$e^{-t\tilde{\square}_{m}^{c}}(x,y)$
has the following asymptotic expansion:%
\begin{equation}
e^{-t\tilde{\square}_{m}^{c}}(x,y)\sim
t^{-(n-1)}a_{n-1}(t,x,y)+t^{-(n-2)}a_{n-2}(t,x,y)+\text{ }\cdot \cdot \cdot
\label{1.6-5}
\end{equation}%
where for $(x,y)\in \Sigma \times \Sigma $ \ with $x=(z,w),$ $y=(\zeta ,\eta
)$ in local coordinates \ 
\begin{eqnarray*}
&&a_{s}(t,x,y)=l(y)^{m}\sum_{j}\varphi _{j}(x)w^{m}\int_{\xi \in \mathbb{C}%
^{\ast }}\{e^{-\frac{\tilde{d}_{M}^{2}(z,\zeta )}{4t}}b_{s}(z,\zeta ) \\
&&\text{ \ \ \ \ \ \ \ \ \ \ \ \ \ \ \ \ \ \ \ \ \ }\eta ^{-m}\tau
_{j}(\zeta )\sigma _{j}(\vartheta )\xi ^{-m}\}\circ \sigma (\xi )_{\xi
^{-1}y}^{\ast }d\mu _{y,m}(\xi ),\text{ }s=n-1,\text{ }n-2,\cdot \cdot \cdot
\end{eqnarray*}%
where $d\mu _{y,m}(\xi )$ is as in (\ref{6-1e1}), $\tilde{d}_{M}(z,\zeta )$
and $b_{s}(z,\zeta )$ as in (\ref{OAE}), and to simplify notations we use $%
\zeta ,$ $\eta ^{-m}$ and $\vartheta $ to denote $\zeta (\xi ^{-1}y),$ $\eta
^{-m}(\xi ^{-1}y)$ and $\vartheta (\xi ^{-1}y)$ respectively and $\varphi
_{j},$ $\tau _{j},$ $\sigma _{j}$ are as in (\ref{1.9-5}).
\end{theorem}

\begin{remark}
\label{1.3-5} It can be shown that $a_{s}(t,x,y)$ has a nontrivial
dependence on $t$ even for $x=y$ and essentially descends to $\underline{a}%
_{s}(t,\pi (x),\pi (y))$ on the compact complex orbifold $M$ $=$ $\Sigma
/\sigma $ via $\pi $ $:\Sigma $ $\rightarrow $ $\Sigma /\sigma $ (cf. Remark %
\ref{8.1-5} and Theorem \ref{thm2-1}). Similarly $e^{-t\tilde{\square}%
_{m}^{c}}(x,y)$ on $\Sigma $ descends to $e^{-t\underline{\tilde{\square}}%
_{m}^{c}}(\pi (x),\pi (y))$ on $M,$ which coincides with an appropriate heat
kernel on $M$ (cf. Remark \ref{8.1-5}). It is worth noting that $e^{-t%
\underline{\tilde{\square}}_{m}^{c}}(\pi (x),\pi (x))$ on $M$ has an
asymptotic expansion with $t$-dependent coefficients by (\ref{1.6-5}). This
\textquotedblleft $t$-dependence" is unavoidable if one wants the asymptotic
expansion to be valid uniformly and entirely on $M$ (rather than just
piecewise valid with respect to the strata). See \cite[Remarks 1.6 and 1.7]%
{CHT} for geometrical interpretations in this regard. The intrinsic nature,
in contrast to $a_{s}(t,x,y),$ of $b_{s}(z,\zeta )$ is remarked after (\ref%
{7-5'}); see Remark \ref{8.1-5} for $a_{s}(t,x,y)$ in this regard.
\end{remark}

\begin{remark}
\label{r1-2} For the meaning of the above \textquotedblleft $\sim $" we
refer to Remark \ref{7-10-1} and \cite[Definition 5.5]{CHT} with their $%
C^{k} $-norm replaced by the $C_{B}^{k}$-norm (see (\ref{CBs})). This $%
C_{B}^{k}$-norm is perhaps a novel notion and is pervasively used in Section %
\ref{A-THK}$.$ There is an analogue of (\ref{AHKE1}) for CR manifolds with $%
S^{1} $-action (cf. \cite[(6.2) in p.92]{CHT}). The appearance of the length
function $l(y)$ in the asymptotic expansions shows a special feature of the $%
\mathbb{C}^{\ast }$-action. We can generalize (\ref{AHKE1}) to $C_{B}^{k}$
estimates by reducing the power of $t$ by $\frac{k}{2}$ on the right hand
side (Remark \ref{7-36.5}). Moreover, for $x\in \Sigma _{\text{sing}}$ an
estimate and proof similar to (\ref{AHKE1}) holds as well; we omit the
details here (cf. \cite[Theorem 6.1]{CHT}). For the corresponding results on
CR manifolds with $S^{1}$-action see \cite[(1.17) in p.10]{CHT}.
\end{remark}

The estimates in Theorem \ref{AHKE} are similar in expression to those in
the CR case \cite{CHT} in which the above length function $l(x)$ is not
existing (or viewed as reducing to the constant $1$). The fact that the
dependence of $a_{i}(t,x,x)$ on $t$ is in general nontrivial is basically a
reflection of the non-freeness of the $\mathbb{C}^{\ast }$-action. As such
the asymptotic expansion (\ref{1.6-5}), different from the classical-looking
ones which have been studied in the recent literature and involve no such $t$%
-dependence (cf. \cite[Section 7.1]{CHT} and references therein) can, with
better accuracy\footnote{%
Indeed, as $t\rightarrow 0^{+}$ our asymptotic expansion approaches the
classical-looking one in a pointwise, non-uniform manner. This is thought to
partially explain the somewhat strange discontinuity phenomenon incurred by
the conventional expansion when used across the different strata (cf. \cite[%
Section 7.1]{CHT}, \cite[(4.7)]{Rich1}). Of course, no such discontinuity
occurs if using (\ref{1.6-5}).}, find its application to the desired index
formula here. The formula (\ref{1.6-5}) does not appear feasible from the
viewpoint of $M.$ This may be due to that we mainly work on the total space $%
\Sigma $ rather than the quotient space $\Sigma /\sigma $ $=$ $M$.

Our method and local index formula have an application to the following
problems. Let us generalize $\mathbb{C}^{\ast }$-action to $G$-action $%
\sigma _{M}^{G}$ on a complex manifold $M,$ where $G$ is a connected complex
reductive Lie group. Let $\bar{G}$ be a projective compactification
(compatible with the group action in the sense that the left action of $G$
on $G$ extends holomorphically to an action of $G$ on $\bar{G})$ given in 
\cite{Som} or \cite[VIII-8]{Se}$.$ Let $H_{0,\sigma _{M}^{G}}^{0}(M,\Omega
_{M}^{p})$ denote the space of all $G$-invariant holomorphic $p$-forms via
the action $\sigma _{M}^{G}$ (see Notation \ref{N-9-1}).

\begin{theorem}
\label{propA2-1} (proof seated above Remark \ref{rext}) Let $G$ be a
connected complex reductive Lie group. Suppose that we have a holomorphic $G$%
-action $\sigma _{M}^{G}$ on a complex manifold $M$ (compact or noncompact)
admitting a meromorphic extension $\check{\sigma}_{M}^{G}$ $:$ $\bar{G}%
\times M$ - - -\TEXTsymbol{>} $M$ (meromorphic map in the sense of Remmert)$%
. $ Here no local freeness of the $G$-action is assumed. Then there holds%
\begin{equation*}
H^{0}(M,\Omega _{M}^{p})=H_{0,\sigma _{M}^{G}}^{0}(M,\Omega _{M}^{p}).
\end{equation*}
\end{theorem}

If $M$ is projective and $\sigma _{M}^{G}$ is algebraic, then $\sigma
_{M}^{G}$ automatically extends meromorphically to $\bar{G}\times M$ - - -%
\TEXTsymbol{>} $M$ (see Remark \ref{mero} for K\"{a}hler cases). Theorem \ref%
{propA2-1} generalizes a result of Carrel and Sommese \cite[Corollary IV]{CS}%
, which deals with ($\mathbb{C}^{\ast })^{d}$-action on compact K\"{a}hler
manifolds using their $\mathbb{C}^{\ast }$-invariant decomposition method.

Theorem \ref{propA2-1} can be further generalized in the following sense.
Consider a (open) complex manifold $P$ with two holomorphic, proper, locally
free $\mathbb{C}^{\ast }$-actions $\sigma _{1},$ $\sigma _{2}$. For
simplicity, we assume $\mathbb{R}^{+}$ ($\subset \mathbb{C}^{\ast })$ action
is globally free while $S^{1}$ ($\subset \mathbb{C}^{\ast })$ action is
locally free. Then $B$ $:=$ $P/\sigma _{2}$ and $M$ $:=$ $P/\sigma _{1}$ are
two complex orbifolds (see Theorem \ref{thm2-1}). For basic material on
complex orbifolds\footnote{%
For instance, in what follows $\Omega _{M}^{p}$ for the orbifold $M$ is
understood in the orbifold sense: a local section $\omega $ of $\Omega
_{M}^{p}$ means a local section $\tilde{\omega}$ of $\Omega _{\tilde{U}}^{p}$
on some (smooth) orbifold chart $\tilde{U}$, which is required to be
invariant under the associated local group.}, we refer to \cite[pp 408-410]%
{CK}, \cite[pp 206-207]{BDD} or \cite{Ka}. Assume further that $\sigma _{2}$
commutes with $\sigma _{1},$ i.e. $\sigma _{2}(\lambda )\circ \sigma
_{1}(\zeta )$ $=$ $\sigma _{1}(\zeta )\circ \sigma _{2}(\lambda )$ on $P$
for $\lambda ,$ $\zeta $ $\in $ $\mathbb{C}^{\ast }.$ It follows that $%
\sigma _{2}$ preserves $\sigma _{1}$-orbits and induces a holomorphic $%
\mathbb{C}^{\ast }$-action $\sigma _{M}$ on $P/\sigma _{1}$ $=:$ $M.$ We
also assume that $\sigma _{2}$ is \textbf{nondegenerate} in the sense that
it does not act on $\sigma _{1}$-orbits trivially (see the remarks above (%
\ref{9-0})).

\begin{theorem}
\label{BM0} (proof seated after Lemma \ref{10.4-5}) With notations and
assumptions explained above, we suppose that $\sigma _{M}$ extends
meromorphically to $\mathbb{CP}^{1}\times M$ - - -\TEXTsymbol{>} $M.$ Recall 
$B$ $:=$ $P/\sigma _{2}$ and $M$ $:=$ $P/\sigma _{1}.$ Assume that $B,$ $M$
are compact. Then we have a natural linear isomorphism%
\begin{equation}
H^{0}(B,\Omega _{B}^{p})\simeq H_{0,\sigma _{M}}^{0}(M,\Omega _{M}^{p}).
\label{BM0-1}
\end{equation}
\end{theorem}

\noindent \textit{Assume further that }$B$\textit{\ is smooth and K\"{a}%
hler. Then we have}%
\begin{equation}
\sum_{p=0}^{\dim M}(-1)^{p}\dim H_{0,\sigma _{M}}^{0}(M,\Omega
_{M}^{p})=\sum_{p=0}^{\dim P}(-1)^{p}\dim H_{0,\sigma _{2}}^{p}(P,\mathcal{O}%
_{P}\mathcal{)}  \label{BM0-2}
\end{equation}

\noindent \textit{and this can be computed through the local index formula (%
\ref{MF}) of Theorem \ref{main_theorem} for }$m$\textit{\ }$=$\textit{\ }$0.$
\textit{A generalization of (\ref{BM0-1}) to certain noncompact cases is
possible. See the comments in the proof of this theorem; compare Theorem \ref%
{propA2-1}.}

Let $\sigma _{M}$ ($=\sigma _{M}^{G}$) be as in Theorem \ref{propA2-1}. By
taking $P=\mathbb{C}^{\ast }\times M$ as a trivial, principal $\mathbb{C}%
^{\ast }$-bundle on $M$ and the obvious \textquotedblleft diagonal action"
on $P$ induced by $\sigma _{M}$ on $M$ as $\sigma _{2}$ (which is seen to be
nondegenerate in the sense above)$.$ Theorem \ref{BM0} soon brings us back
to the situation of Theorem \ref{propA2-1} for $G$ $=$ $\mathbb{C}^{\ast }.$

The above work of having two $\mathbb{C}^{\ast }$-actions might be related
to the work on a certain type of moduli spaces having two foliations, such
as the one from physics and string theory, which is briefly explained in 
\cite[(1.4) of Introduction, p. 320]{Br}. In our case, the following
phenomenon seems to be of interest:

\begin{corollary}
\label{Cor4-1} With notations and assumptions explained prior to Theorem \ref%
{BM0}, assume further that $\sigma _{1}$ is also nondegenerate and thus
induces a nontrivial holomorphic $\mathbb{C}^{\ast }$-action $\sigma _{B}$
on $P/\sigma _{2}$ $=:$ $B.$ Suppose that both $B$ and $M$ are smooth,
projective and both actions $\sigma _{M}$ and $\sigma _{B}$ are algebraic.
Then we have natural linear isomorphisms%
\begin{eqnarray*}
H^{0}(B,\Omega _{B}^{p}) &\simeq &H_{0,\sigma _{M}}^{0}(M,\Omega _{M}^{p}),
\\
H^{0}(M,\Omega _{M}^{p}) &\simeq &H_{0,\sigma _{B}}^{0}(B,\Omega _{B}^{p})
\end{eqnarray*}%
\noindent and hence the isomorphism by Theorem \ref{propA2-1}%
\begin{equation*}
H^{0}(B,\Omega _{B}^{p})\simeq H^{0}(M,\Omega _{M}^{p}).
\end{equation*}%
\textit{Moreover, we have (\ref{BM0-2}) for both }$B$\textit{\ and }$M$%
\textit{\ as the LHS while the right hand side can be computed through the
local index formula (\ref{MF}) of Theorem \ref{main_theorem} for }$m$\textit{%
\ }$=$\textit{\ }$0.$
\end{corollary}

\begin{remark}
\label{1.8-5} For a general $G$ (connected complex reductive Lie group) a
similar result as the first half of Corollary \ref{Cor4-1} holds (by method
parallel to that of Theorem \ref{propA2-1} generalizing $\mathbb{C}^{\ast }$
to $G,$ see the last paragraph of the proof of Theorem \ref{propA2-1}).
However we haven't had the local index formula in the second half of
Corollary \ref{Cor4-1} for $G$ general.
\end{remark}

Our result Theorem \ref{main_theorem} (via Theorem \ref{AHKE}) may be placed
in the context of index theorems of transversal type, which can be linked to
an extension of Atiyah-Singer index theory to the class of transversally
elliptic operators, cf. \cite{PV}, \cite{A}, \cite{Fit}. There are, however,
differences between those approaches and that of ours. For instance, our
base space $\Sigma $ and the group $\mathbb{C}^{\ast }$ are non-compact; we
aim at \textit{local index }type results in the sense closely related to 
\cite[Chap.4]{BGV}; the notion of \textquotedblleft distribution-index" as
originally advocated by Atiyah \cite{A} (see also \cite{PV} for further
results and references), is not explicitly involved in the present work. It
seems that none of those works uses the (transversal) heat kernel approach
in the same way as we did here. For potential links with other areas of
research, see the discussion later in this Introduction and the footnote
there.

In the remaining part of this Introduction, let us first outline some
ingredients involved in our proofs. Since our goal is to get a local index
density whose integral is the above index, it is natural to use heat kernel
method. Due to our setup, we are led to consider what we call
\textquotedblleft transversally" $spin^{c}$ Dirac operator; since $\Sigma $
may not be K\"{a}hler, we also need a modified version of this transversal
Dirac operator (in order to catch a local index density by an asymptotic
heat kernel for its Kodaira-type Laplacian).

For $\Sigma $ being the total space of a holomorphic line bundle $L$ (with
the zero section removed) over a compact complex manifold $M$, there is a
one-one correspondence between elements in $\Omega _{m}^{0,q}(\Sigma )$ and
sections in ($L^{\ast })^{\otimes m}$ over $M$ $=$ $\Sigma /\sigma $ (cf.
Proposition \ref{p-gue2-2}$).$ Motivated by this observation, we construct
an approximate heat kernel on $\Sigma $ by patching up local Dirichlet heat
kernels $K_{t}^{j}(z,\zeta )$ on $M$ as follows:%
\begin{eqnarray}
P_{m,t}^{0}&:=&\sum_{j\text{ (finite)}}H_{m,t}^{j}\circ \pi _{m},
\label{1.9-5} \\
H_{m,t}^{j}(x,y)&:=&\varphi _{j}(x)w^{m}K_{t}^{j}(z,\zeta )\eta ^{-m}\tau
_{j}(\zeta )\sigma _{j}(\vartheta )l(y)^{m}  \notag
\end{eqnarray}

\noindent where $\pi _{m}$ denotes the orthogonal projection onto the $m$%
-space $L_{m}^{2,\ast }(\Sigma ,E)$ ($L^{2}$-completion of $\Omega
_{m}^{0,\ast }(\Sigma ,E)),$ $x$ $=$ $(z,w),$ $w$ $=$ $|w|e^{i\phi },$ $y$ $%
= $ $(\zeta ,\eta )$, $\eta $ $=$ $|\eta |e^{i\vartheta },$ $\varphi _{j}$
being a smooth partition of unity for $\Sigma $ with $\varphi _{j}(x)$ $=$ $%
\varphi _{j}(z,\phi )$, $\tau _{j},$ $\sigma _{j}$ some cutoff functions and 
$l(y)$ $=$ $h(\zeta ,\bar{\zeta})\eta \bar{\eta}$ (see (\ref{lq}) in Section
3 and Section 6 for details).

One of our main technical tasks is to evaluate $P_{m,t}^{0}$ along the
diagonal $(x,x).$ However, this evaluation becomes nontrivial due to the
projection operator $\pi _{m}.$ More precisely (see (\ref{6-1f}))%
\begin{equation}
(H_{m,t}^{j}\circ \pi _{m})(x,x)=\int_{\mathbb{C}^{\ast }}H_{m,t}^{j}(x,\xi
^{-1}\circ x)\circ \sigma (\xi )_{\xi ^{-1}\circ x}^{\ast }\bar{\xi}^{m}d\mu
_{x,m}(\xi ).  \label{1-3}
\end{equation}

\noindent Here we have denoted $\sigma (\xi ^{-1})(x)$ by $\xi ^{-1}\circ x$
(or $\xi ^{-1}x$ for short) and $d\mu _{x,m}(\xi )$ is a certain 2-form in
the action parameter $\xi $ (see (\ref{6-1e1}))$.$

The salient fact is that the value at the diagonal element $(x,x)$ in the
LHS of (\ref{1-3}) involves those at the off-diagonal element $(x,\xi
^{-1}\circ x)$ in the RHS of (\ref{1-3}). Let us give a little more
explanation as follows.

The integral (\ref{1-3}) over the angle variable part of $\mathbb{C}^{\ast }$
gives rise to a diagonal term for small angular range (cf. (\ref{HQDI})) and
a nondiagonal term for large angular range (cf. discussions from (\ref{7-2})
onwards).\ The latter provides a term expressed in exponential to the
negative distance square over $t$ (see the last term in the RHS of (\ref%
{AHKE1})) when one tries to estimate the supertrace of the heat kernel
asymptotic expansion. It ends up that this nondiagonal term has contribution
obtained from lower dimensional strata; the detail involves
\textquotedblleft Lefschetz trace" roughly explained as follows. When one is
evaluating the supertrace around a stratum point, say $P,$ the local
isotropy group $G_{j}$ (identifiable as local orbifold structure group)
comes into play. The original \textit{transverse} supertrace at $P$ becomes
transformed to a nontransversal /ordinary supertrace \textit{twisted} by $g$ 
$\in $ $G_{j}$ (from which the above nondiagonal term arises).
Interestingly, this (as $t\rightarrow 0$) is soon recognized essentially as
the local density (at $P$) of the Lefschetz-Riemann-Roch; the local version
of LRR finds an application here (unclear to us whether any other
applications of the local LRR exist elsewhere in the literature). See (\ref%
{Asym}) for the above-mentioned twisting as well as the paragraph below it.
Let us note that this step is much inspired by a theorem of
Berline-Getzler-Vergne \cite[Theorem 6.11]{BGV}, whose proof is
difficult and whose statement is remarkable in that the asymptotic expansion
given there involves generalized functions (compare \cite{CHT-P} which
studies certain continuity issues within the parameter-dependent setting).
In this regard, compare the paragraph seated above Subsection \ref{S-8-4}
about a flaw in our previous work \cite{CHT}, \cite{CHT-E}. This analytical
implication has an effect on the algebraic result of Kawasaki's
Hirzebruch-Riemann-Roch theorem for compact complex orbifolds $\Sigma
/\sigma $ (see the second paragraph after Remark \ref{r1-1}).

Prior to the above, a more basic technical task worth mentioning is the
construction of the (non $\mathbb{C}^{\ast }$-invariant, incomplete)
Hermitian metric $G_{a,m}$ on $\Sigma $ for our purposes. The troubling
issues here are two-fold: the noncompactness of $\Sigma $ as well as that of 
$\mathbb{C}^{\ast }.$ It turns out that our metric $G_{a.m}$ is not $\mathbb{%
C}^{\ast }$-invariant, yet by using it we manage to design and work out some
geometric constructions on $\Sigma $ and on $M$ $=$ $\Sigma /\mathbb{\sigma }
$ respectively in such a way that they are mutually "compatible" in an
appropriate context (cf. Proposition \ref{madj} and Corollary \ref{Cor3-7}).
Although there are a fair amount of technicalities, let us content ourselves
with pointing out that this compatibility just mentioned, plays a crucial
role not only at a conceptual level but also leading us to technically
fulfill analytical requirements in the long process (cf. Proposition \ref%
{BoxDU} and Sections 6, 7). Fortunately, all of these is made possible via
special features of our metric $G_{a,m};$ this we can't quite see
conceptually beforehand. The question whether a different choice of metrics
can lead to similar results is far from obvious to us, but there appears to
be a certain set of conditions (not formulated in this paper) required for
the metric to do the job. Remark that this aspect presents a major
difference between points of departure in this paper and in \cite{CHT} where
the compactness of the manifold $X$ and that of the group $S^{1},$ make a
sharp contrast to the effect that their metric is simply chosen to be $S^{1}$%
-invariant, which saves a lot of work there.

In recent decades there appeared increasingly active study of heat kernels
in the transversal sense or even more generalized sense. See e.g. \cite%
{Rich1}, \cite{Rich2} and \cite[ Section 7.1]{CHT} for some comments with
extensive references. To the best of our understanding, most treatments in
the existent literature are given under the compactness (or completeness)
assumption which is either imposed on the manifold or on the group or both.
Our present work makes an attempt towards some noncompact issues. It is
likely, although technically rather unclear at this stage, that the results
here admit a generalization to complex Lie groups other than $\mathbb{C}%
^{\ast }.$ Remark also that the asymptotic expansion (in $t$) of trace
integrals $\int_{\Sigma }Tre^{-t\tilde{\square}_{m}^{c\pm }}(x,x)dv_{\Sigma
,m}$ is not discussed here (cf. some treatment in \cite[Section 7]{CHT} for
CR cases and \cite{Rich2} for foliations); some needed tools have been
developed in \cite{CHT-ECM} and for partial results in CR cases see [Ibid.,
Theorem 1.1]. We hope to come back to some of these in future publication.

Due to the noncompactness some difficulties also occur in the treatment of
the Hodge theory part; compare the introductory paragraph of Section 4. One
difficulty involves the trouble that the seemingly natural and conventional
Sobolev $s$-norm $||$ $\cdot $ $||_{s}$ (cf. (\ref{Hs})) is unsuitable. One
novelty of Section 4 is introducing slightly complicated modifications
denoted by $||$ $\cdot $ $||_{s}^{\prime }$ and $||$ $\cdot $ $%
||_{s}^{^{\prime \prime }}$ ((\ref{Als}), (\ref{4.3-5}) and \ref{4.6-5}),
whose motivations are hidden in Propositions \ref{dualLm} and \ref{madj}.
Compare the $C_{B}^{s}$-norm mentioned in Remark \ref{r1-2}. With this
modification the approach adopted in the introductory paragraph of Section %
\ref{Sec4} can be developed and finally carried out. Classical results:
Rellich compactness, elliptic estimates, elliptic regularity, etc., can find
their analogues in this transversal setting, based on the modified norms.
Among other things, the finite-dimensionality of $H_{m}^{q}(\Sigma ,\mathcal{%
O})$ can be proved here. Remark that it is possible to prove the
finite-dimensionality result independently by using Theorem \ref{thm2-1}, (%
\ref{9.17-5}) and Remark \ref{9.3} that may bring some study on $\Sigma $ to
the orbifold $M$ $=$ $\Sigma /\sigma .$ However, one purpose of this paper
is that instead of working on $M$ directly one works on $\Sigma $ itself so
that if needed $M$ is then studied via \textquotedblleft dimension
reduction" or \textquotedblleft Kaluza-Klein reduction", which in our view
is a methodology in the same spirit as ours and was already used for certain
purposes in physics. See more about it later in this Introduction.

Another feature here distinct from \cite{CHT} is the following. Given the
complex analytic equivalence $\Sigma $ $\cong $ $X_{1}\times \mathbb{R}^{+}$ 
$\cong $ $X_{2}\times \mathbb{R}^{+}$ as mentioned previously, we cannot
conclude the CR equivalence $X_{1}\cong X_{2}.$ For instance, take $\Sigma $ 
$=$ $L\backslash \{0$-section\} of a holomorphic line bundle $L$ on $M,$ and
the circle bundle $X$ $\subset $ $L\backslash \{0$-section\}. Both $\Sigma $
and $X$ have the same quotient $M$ $=$ $X/S^{1}$ $=$ $\Sigma /\mathbb{C}%
^{\ast }.$ While $X$ depends on the choice of a Hermitian metric on $L,$ $%
\Sigma $ does not. Thus, if we want to work on index theorems
\textquotedblleft upstairs" such as $\Sigma $ or $X,$ there is in general
only non-canonical choice of $X$ such that $X\times \mathbb{R}^{+}$ $\cong $ 
$\Sigma .$ Since the local index theorem is usually meant to be computable
from the associated heat kernel asymptotics and since these heat kernels are
not immediately transferable from the CR case \cite{CHT} to the complex case
(and vice versa), the present paper provides the needed technical details
precisely for the complex case.

Moreover, $\Sigma $ is akin to algebro-geometric objects. In this connection
it seems possible and of interest to formulate an analogue of the index
theorem discussed here within an algebraic setting. But then how this
formulation of results can be proved in a purely algebraic manner remains to
be seen.

Inspired by the potential algebraic interpretation above, one may be
naturally led to questions along the following line of thought. Firstly as
our index theorem may be viewed as a transversal Hirzebruch-Riemann-Roch
theorem (HRR for short), one may ask for a Grothendieck-Riemann-Roch theorem
(GRR for short) or family index theorem in the transversal sense similar to
that as considered here. Secondly, the development of the so-called
\textquotedblleft algebraic cobordism" in the last decade (see the monograph 
\cite{LM} of M. Levine and F. Morel) encodes the classical GRR theorem, cf. 
\cite[Subsection 4.2.4]{LM} for a precise explanation. In recent years an
equivariant algebraic cobordism theory for schemes $X$ with an action by a
linear algebraic group $G$ was constructed by J. Heller and J. Malag\'{o}r-L%
\'{o}pez \cite{HML} (see also \cite{Kri} and \cite{Liu}). In the case where
the geometric quotient $X\rightarrow X/G$ exists and is realized as a
principal $G$-bundle, there exists an isomorphism between the ordinary
algebraic cobordism $\Omega _{\ast }(X/G)$ of $X/G$ and the equivariant
algebraic cobordism $\Omega _{\ast +\dim G}^{G}(X)$ of $X$ (see \cite[%
Proposition 27]{HML}). In this connection and in view of our Proposition \ref%
{p-gue2-2} or Remark \ref{PLMB} (with $\Sigma $ as $X$ and $M$ $=$ $\Sigma /%
\mathbb{C}^{\ast }$ as $X/G$), the present transversal index theorem on $%
\Sigma $ in its algebraic context might be linked to a version of
\textquotedblleft equivariant GRR or HRR theorem" which by analogy with \cite%
[Subsection 4.2.4]{LM} just mentioned might be expected, or be encoded in
the theory of the equivariant algebraic cobordism $\Omega _{\ast
}^{G}(\Sigma )$ with $G$ $=$ $\mathbb{C}^{\ast }$.\footnote{%
In other related equivariant settings, approaches to Riemann-Roch using
localization techniques algebraically or analytically have been pursued in
works \cite{BV97}, \cite{EG03}, \cite{EG05}, \cite{BGV} and \cite{Gu97}. For
Riemann-Roch in (higher) equivariant K-theory, see the recent work \cite%
{Kr14}; see also \cite{EG00} for Riemann-Roch in equivariant Chow groups.
These works focus on schemes with algebraic group action using algebraic
methods, and the results there are neither valued in certain cohomology
groups nor meant for \textit{local} index theorems as considered in the
context here.} We hope to turn to it in future publication.

In addition to Corollary \ref{Cor4-1} above, we remark that the framework
set up in this paper echos certain classical constructions in physics, at
least from a philosophical point of view. In the approach of the so-called
Kaluza-Klein reduction (e.g. \cite[Section 7.1]{BBSc}, \cite[p.399]{GSW}, 
\cite[Section 4.1]{J}), a gauge field (e.g. one in electromagnetic theory)
on a space $M$ combined with a metric $g$ on $M$ can be thought of as a
certain metric (cf. gravity) on the associated principal bundle $P$ over $M,$
because the connection from this gauge field induces certain
\textquotedblleft horizontal spaces $H"$ in $P$ and hence, equipping $H$
essentially with the metric $g$ on $M$ (and also \textquotedblleft vertical
part" of $P$ with group invariant metric) leads to a natural metric on $P$;
the process here is basically reversible from $P$ to $M$ on which a gauge
field is then induced. A recent work of the physicist N. Nekrasov makes use
of such K-K picture to set up for $M$ a two-dimensional torus a framework 
\cite[(2.5)]{N} similar to (\ref{1-2-2}) of this paper. With this said the
idea is turning to the study of objects (with appropriate symmetries) on $P$
rather than the direct study of those on $M.$ Since the role played by
orbifolds in string theory is increasingly indispensable (e.g. \cite[Section
9.1]{BBSc}, \cite[Section 16.10]{GSW}, \cite[Section 4.8]{J}), it seems
conceivable that certain geometric setup, adapted to orbifolds, similar to
that of ours (arising from $\Sigma $ $\rightarrow $ $\Sigma /\sigma $ $=$ $M$
here, in particular) may appear to be of relevance in the future. It is
maybe worthwhile to mention that the above setup mostly uses compact Lie
groups for the principal bundle $P$ whereas our group of action here is $%
\mathbb{C}^{\ast },$ and that our metric $G_{a,m}$ on $\Sigma $ ($\Sigma $
thought of as a kind of \textquotedblleft orbifold principal bundle"
corresponding to $P$ above) is not $\mathbb{C}^{\ast }$-invariant whereas
the \textquotedblleft horizontal part" of $(\Sigma ,G_{a,m})$ is $\mathbb{C}%
^{\ast }$-invariant (cf. Lemma \ref{L-inv} $i)$). In this connection it
seems a natural question to generalize Theorem \ref{main_theorem} from the
HRR to LRR (Lefschetz-Riemann-Roch) under the presence of \textquotedblleft
symmetries"; see the discussion below.

Let us mention in passing that in the context of arithmetic schemes with a
finite group action there has been some significant progress on Riemann-Roch
type theorems, cf. \cite{CEPT}. It is initially of interest to study, for a
finite Galois extension $N/K$ of number fields with $G$ $=$ Gal$(N/K),$ the
ring of integers $\mathcal{O}_{N}$ as a $\mathbb{Z[}G\mathbb{]}$-module via
its class $[\mathcal{O}_{N}]$ formed in an appropriate Grothendick group
associated with the group ring $\mathbb{Z[}G\mathbb{]}$ combined with the
study of the associated Euler characteristic \cite{Chin}. A vast
generalization from this initial interest to schemes, for such $G$%
-equivariant Euler characteristics (via Riemann-Roch or LRR type theorems as
just mentioned) to yield applications to some number theory problems, can be
found in, for example, \cite{CPT}. Note that in these works the group $G$ is
a finite group and the subscheme fixed by the action of $g$ $\in $ $G$ can
be nonempty for some $g$ $\neq $ identity.

As mentioned above it appears natural to ask for a Lefschetz type index
theorem when a certain automorphism $\gamma $ of $\Sigma $ is given,
including $\gamma $ $=$ Identity of the present paper as a special case. In
our opinion the idea of this paper may be extended to such a situation, for
which one may wish to generalize Theorem \ref{AHKE} to the $\gamma $-twisted
heat kernel asymptotics in the transversal setting (compare (\ref{Asym}) for
a nontransversal, ordinary situation). The details and the appropriate
formulation are left to the interested reader.

A natural problem, closely related to that of CR manifolds with $S^{1}$%
-action already treated in \cite{HHL}, is about the existence of $\mathbb{C}%
^{\ast }$-equivariant holomorphic embeddings of $\Sigma $ when $%
c_{1}(L_{\Sigma },||\cdot ||)$ (see Section 3 for $L_{\Sigma }$) is negative
(which corresponds to the strong pseudoconvexity in the CR case \cite[p.46]%
{CHT}). Note that in the CR version of the HRR theorem as stated in \cite[%
p.16]{CHT} the term $-d\omega _{0}$ is positive when $X$ is strongly
pseudoconvex (due to the convention of the Reeb vector field $T$ given in 
\cite[Subsection 2.2]{CHT}). Moreover, for the weakly pseudoconvex situation
certain Morse-type inequalities and vanishing theorems are expected to hold
in this $\mathbb{C}^{\ast }$-version along the line similar to \cite[Theorem
2.1]{HL} and \cite[Proposition 1.21]{CHT}. As far as \textit{orbifold} 
\textit{line bundles} are concerned (whose local sections consist of those
of certain genuine line bundles $L$, that are invariant under the action of
local orbifold groups on $L$), it is of interest to ask effectivity problems
in an orbifold setting, analogous to those works in complex algebraic
geometry including some by Siu and Demailly (cf. \cite{Siu}, \cite{Siu2}, 
\cite{Dem}). For the case of orbifold cyclic singularities, working directly
on $\Sigma $ may seem a natural approach in a similar spirit to that of the
present paper. We leave these study to future publications.

The paper is organized briefly as follows. In Section \ref{S2} we discuss
some basic material for complex manifolds with holomorphic $\mathbb{C}^{\ast
}$-action. Among others we show Theorem \ref{thm2-1} that the quotient is a
complex orbifold under the condition that the action is proper, the the $%
\mathbb{R}^{+}$ part is globally free and the $S^{1}$-part is locally free.
In Section \ref{S-metric} the (non $\mathbb{C}^{\ast }$-invariant) Hermitian
metric $G_{a,m}$ is carefully constructed and its properties are examined.
In Sections \ref{Sec4} and \ref{Sec5} we develop the Hodge theory associated
to the relevant (transversal) Laplacian or modified Laplacian (necessary for
the non-K\"{a}hler case) and prove a McKean-Singer type formula (see Theorem %
\ref{t-5-1}) for the relevant index. An (transversal) approximate heat
kernel is constructed in Section \ref{A-THK} and an asymptotic expansion is
discussed in Section \ref{AE_THK} in which we give the proof of Theorem \ref%
{AHKE}. In Section \ref{LmIF} we give the proof of Theorem \ref{main_theorem}%
. Theorem \ref{propA2-1} is proved in the end of Section \ref{Sec9} while
Theorem \ref{BM0} and Corollary \ref{Cor4-1} are proved in the end of
Section \ref{Sec10}.

\medskip

\textbf{Acknowledgements.} J.-H. Cheng and C.-Y. Hsiao would like to thank
the Ministry of Science and Technology of Taiwan for the support: grant no.
MOST 108-2115-M-001-010 and grant no. MOST 108-2115-M-001-012-MY5
respectively. J.-H. Cheng would also like to thank the NCTS for constant
support. I-H. Tsai was supported in part by National Taiwan University grant
no. 106-2821-c-002-001-ES. As it takes nearly a decade to carry out this
joint work, I-H. Tsai wishes to thank Academia Sinica for the excellent
working conditions and the long-time support. We would like to thank Siye Wu
for his constant interest in our work during the preparation. 
We are grateful to Professor Mich\`{e}le Vergne for drawing our attention to the
class of examples which are partly considered in Example \ref{E-IF}. We owe
a special thanks to Professor Yum-Tong Siu for his valuable comments and
interest on this paper when we met him on the occasion of a conference in
honor of his mathematical work, held at Taipei in January 2025.

\section{\textbf{Preliminaries on} \textbf{complex manifolds with }$\mathbf{C%
}^{\ast }$-\textbf{action\label{S2}}}

Consider a complex manifold $\Sigma $ of dimension $n$ with holomorphic $%
\mathbb{C}^{\ast }$-action $\sigma .$ That is, the map%
\begin{equation}
\sigma :\mathbb{C}^{\ast }\times \Sigma \rightarrow \Sigma  \label{2-0}
\end{equation}%
\noindent defined by $(\lambda ,x)$ $=$ $(\rho e^{i\theta },x)$ $\rightarrow 
$ $\sigma (\lambda ,x)$ (also denoted as $\sigma $($\lambda )(x)$ or $\sigma 
$($\lambda )\circ x)$ is holomorphic in $\lambda $, $x$ and satisfies the
group action condition: $\sigma $($\lambda _{1}\lambda _{2})\circ x$ $=$ $%
\sigma $($\lambda _{1})\circ (\sigma (\lambda _{2})\circ x),$ $\sigma
(1)\circ x$ $=$ $x$. See \cite{BBS}, \cite{BBSw}, \cite{CS} and \cite{Gro}
for relevant information on this class of complex manifolds.%

The holomorphic $\mathbb{C}^{\ast }$-action induces a holomorphic vector
field $\digamma $ on $\Sigma .$ Near a point $q$ where $\digamma $ $\neq $ $%
0 $, we can find holomorphic coordinates $z_{1},$ $z_{2},$ $..,$ $z_{n-1},$ $%
\zeta $ such that $\digamma $ $=$ $\frac{\partial }{\partial \zeta }.$ Let $%
w $ $=$ $e^{\zeta }.$ Then we have%
\begin{equation}
\digamma =\frac{\partial }{\partial \zeta }=w\frac{\partial }{\partial w}
\label{2-2}
\end{equation}

\noindent with $w\neq 0.$

In this paper, we consider only the case of locally free action so that $%
\digamma (q)$ $\neq $ $0$ for all $q\in \Sigma .$ Here we say that the
action $\sigma $ is locally free if for any given point $q$ $\in $ $\Sigma ,$
$\sigma $($\lambda )\circ q$ $=q$ with $\lambda $ near $1$ implies $\lambda $
$=$ $1.$

\begin{proposition}
\label{p-gue2-1} \text{ (Distinguished local coordinates) }With the notation
above, \textit{suppose }$\{z_{1},$\textit{\ }$z_{2},$\textit{\ }$..,$\textit{%
\ }$z_{n-1},$\textit{\ }$w\}$\textit{\ and }$\{\tilde{z}_{1},$\textit{\ }$%
\tilde{z}_{2},$\textit{\ }$..,$\textit{\ }$\tilde{z}_{n-1},$\textit{\ }$%
\tilde{w}\}$\textit{\ are two systems of holomorphic coordinates near }$q$ 
\textit{with }$w$\textit{\ }$\neq $\textit{\ }$0,$\textit{\ }$\tilde{w}$%
\textit{\ }$\neq $\textit{\ }$0$ satisfying (\ref{2-2}) (we sometimes assume 
$w(q)$ $=$ $1,$ $\tilde{w}(q)$ $=$ $1$ for use later$\mathit{)}.$\textit{\
Then on the overlap they are related as follows:}%
\begin{equation}
\tilde{w}=w\varphi (z_{1},z_{2},..,z_{n-1}),\text{ }\tilde{z}_{j}=\mu
_{j}(z_{1},z_{2},..,z_{n-1})\text{ }(1\leq j\leq n-1)  \label{C0}
\end{equation}%
\textit{\noindent where }$\varphi $\textit{\ (vanishing nowhere) and }$\mu
_{j}$\textit{\ are holomorphic functions. The }$\mathbb{C}^{\ast }$-action $%
\sigma (\lambda )$ acts by%
\begin{equation}
\sigma (\lambda )(z_{1},\mathit{\ }z_{2},\mathit{\ }..,\mathit{\ }z_{n-1},%
\mathit{\ }w)=(z_{1},\mathit{\ }z_{2},\mathit{\ }..,\mathit{\ }z_{n-1},%
\mathit{\ \lambda }w)  \label{C0-1}
\end{equation}%
\textit{\noindent for }$\lambda \in \mathbb{C}^{\ast }$ near $1.$
\end{proposition}

\proof
From $w\frac{\partial }{\partial w}=\digamma =\tilde{w}\frac{\partial }{%
\partial \tilde{w}},$ we obtain%
\begin{equation}
w\frac{\partial \tilde{w}}{\partial w}=\tilde{w},\text{ }w\frac{\partial 
\tilde{z}_{j}}{\partial w}=0  \label{C1}
\end{equation}%
\textit{\noindent }by the chain rule. The second equation of (\ref{C1})
implies the second formula of (\ref{C0}) since $w\neq 0.$ Differentiating
the first equation of (\ref{C1}) in $w$ leads to $\partial ^{2}\tilde{w}%
/\partial w^{2}$ $=$ $0.$\textit{\ }It follows that $\tilde{w}=w\varphi
(z_{1},z_{2},..,z_{n-1})+g(z_{1},z_{2},..,z_{n-1}).$ Substituting this into
the first equation of (\ref{C1}), we get $g$ $=$ $0.$ We have shown the
first formula of (\ref{C0}). The formula (\ref{C0-1}) follows from the fact
that $w\frac{\partial }{\partial w}=\digamma .$

\endproof%

The distinguished local holomorphic coordinates $(z,w)$ $=$ $(z_{1},$ $%
z_{2}, $ $..,$ $z_{n-1},$ $w)$ of Proposition \ref{p-gue2-1} are often
adopted throughout the paper without further mention.

For our purpose, the reader may keep in mind the following typical examples.

\begin{example}
\label{2.0} i) Let $X$ be a CR manifold with locally free, transversal $%
S^{1} $-action $e^{i\theta }$: $X$ $\rightarrow $ $X$ (denoted by $%
e^{i\theta }\circ x$ for $x\in X)$ preserving the CR structure $T^{1,0}X$
(see \cite{CHT}). Define $T(x)$ $\in $ $T_{x}X$ to be the tangent to the
curve $e^{i\theta }\circ x$ $\subset $ $X$ at $\theta $ $=$ $0.$ Endow $%
X\times \mathbb{R}^{+}$ with the almost complex structure $J$ defined by $J$ 
$=$ $J_{X}$ on $T^{1,0}X\oplus T^{0,1}X$ and $JT=-r\frac{\partial }{\partial
r},$ $J(r\frac{\partial }{\partial r})=T.$ It is straightforward to check
that $J$ is integrable and hence $\Sigma $ $:=$ $X\times \mathbb{R}^{+}$ is
a complex manifold. Define a $\mathbb{C}^{\ast }$-action $\rho e^{i\theta
}:\Sigma \rightarrow \Sigma $ by $(\rho e^{i\theta })\circ (x,r)=(e^{i\theta
}\circ x,\rho r).$ We verify that this $\mathbb{C}^{\ast }$-action preserves
the complex structure on $X\times \mathbb{R}^{+}.$ Identify $X$ with $%
X\times \{1\}$ $\subset $ $X\times \mathbb{R}^{+}.$ The CR structure $J_{X}$
on $X$ is the one induced from the complex structure $J$ on $\Sigma .$

ii) Another natural class of examples arise from the total space $\hat{L}$
of a holomorphic line bundle $L$ over a compact (without boundary) complex
manifold $M.$ An obvious $\mathbb{C}^{\ast }$-action $\sigma $ on $\hat{L}$
is the nonzero fibre multiplication. One simply takes $\Sigma $ $=$ $\hat{L}%
\backslash \{0$-section\}.
\end{example}


With the local description of Proposition \ref{p-gue2-1}, we are going to
prove (see Theorem \ref{thm2-1} below) that $\Sigma $ is the union of local
holomorphic patches ($D_{j},$ ($z,w))$ satisfying (\ref{C0}) and (\ref{C0-1}%
) so that%
\begin{equation}
D_{j}\ni (z,w)=(z,\phi ,r)\in U_{j}\times (-\delta _{j},\delta _{j})\times 
\mathbb{R}^{+}\text{ }(\Sigma =\tbigcup\limits_{j=1}^{N}D_{j},\text{ }N\leq
\infty )  \label{1-0}
\end{equation}

\noindent where $w$ $=$ $re^{i\phi },$ $U_{j}$ is an open domain in $\mathbb{%
C}^{n-1},$ $\delta _{j}$ is a small positive number and $\mathbb{R}^{+}$
denotes the set of positive real numbers. Moreover, the holomorphic $\mathbb{%
C}^{\ast }$-action $\sigma (\rho e^{i\theta })$ $:$ $\Sigma $ $\rightarrow $ 
$\Sigma $ is described as%
\begin{eqnarray}
\sigma (\rho e^{i\theta })(z,w) &=&(z,\rho e^{i\theta }w)\text{ or}
\label{1-1} \\
\sigma (\rho e^{i\theta })(z,\phi ,r) &=&(z,\theta +\phi ,\rho r),\text{ }%
\rho \in \mathbb{R}^{+}  \notag
\end{eqnarray}

\noindent for those $(z,w)\in D_{j}$ such that $\sigma (\rho e^{i\theta
})(z,w)$ $\in $ $D_{j}$ for all $\theta $ with $|\theta +\phi |<\delta _{j},$
and%
\begin{equation}
\text{the }S^{1}\text{-part }\sigma \text{(e}^{i\theta }\text{) of the
action is locally free.}  \label{1-2}
\end{equation}%
\noindent Note that the $\mathbb{R}^{+}$ part $\sigma (\rho )$ of the action
is globally free if (\ref{1-1}) and (\ref{1-0}) hold.

We call the action $\sigma $ \textit{proper} if the map $\sigma :\mathbb{C}%
^{\ast }\times \Sigma \rightarrow \Sigma $ in (\ref{2-0}) is proper, i.e. $%
\sigma ^{-1}(K)$ is compact as long as $K$ $\subset $ $\Sigma $ is compact.

\begin{theorem}
\label{thm2-1} Suppose that $\Sigma $ is a complex manifold with a
holomorphic $\mathbb{C}^{\ast }$-action $\sigma (\rho e^{i\theta }).$ Assume
that the action $\sigma $ is proper, the $\mathbb{R}^{+}$ part $\sigma (\rho
)$ is globally free and the $S^{1}$-part $\sigma $(e$^{i\theta }$) is
locally free. Then (\ref{1-0}), (\ref{1-1}) and (\ref{1-2}) hold. In this
case $\Sigma /\sigma $ is a complex orbifold and a normal complex space.
Suppose further that $\Sigma /\sigma $ is compact. Then $N$ in (\ref{1-0})
can be finite.
\end{theorem}

\begin{remark}
\label{r2.3-5} The normality of $\Sigma /\sigma $ here will be useful in the
proof of Theorem \ref{BM0} given in Section \ref{Sec9}. The use of some
other results proved in later sections may simplify part of the proof below;
see Remark \ref{R-2-11}.
\end{remark}

\proof
(of Theorem \ref{thm2-1}) Given any $\bar{x}$ $\in $ $\Sigma ,$ by (\ref%
{C0-1}) there exists a neighborhood $\mathring{D}_{j}^{\varepsilon }$ $%
\subset $ $\Sigma $ of $\bar{x}$ and a local holomorphic patch
(trivialization) $\psi _{j}^{-1}$ $:$ $\mathring{D}_{j}^{\varepsilon }$ $%
\rightarrow $ $U_{j}\times (-\delta _{j},$ $\delta _{j})\times
(1-\varepsilon ,$ $1+\varepsilon )$ (since we will need to often use $\psi
_{j}^{\ast }$ in later sections, we choose to write $\psi _{j}^{-1}$ here$)$
where $U_{j}$ is an open domain in $\mathbb{C}^{n-1}$ and $\delta _{j},$ $%
\varepsilon $ are small positive numbers, such that for all $x,\tilde{x}$ $%
\in $ $\mathring{D}_{j}^{\varepsilon }$ with $\psi _{j}^{-1}(x)$ $=$ $(z,$ $%
\phi ,$ $r),$ $\psi _{j}^{-1}(\tilde{x})$ $=$ $(z,$ $\tilde{\phi},$ $\tilde{r%
}),$ we have $x=\sigma (\rho e^{i\theta })\tilde{x}$ for some complex number 
$\rho e^{i\theta }$ ($-\delta _{j}$ $<\theta +\tilde{\phi}$ $<$ $\delta
_{j}) $ and%
\begin{equation}
(z,\phi ,r)=(z,\theta +\tilde{\phi},\rho \tilde{r}).  \label{1-2.5}
\end{equation}

\noindent Furthermore, since the action is proper, we claim that we can find 
$\mathring{D}_{j}^{\varepsilon }$ so that for all $x,\tilde{x}$ $\in $ $%
\mathring{D}_{j}^{\varepsilon }$ with $\psi _{j}^{-1}(x)$ $=$ $(z,$ $\phi ,$ 
$r),$ $\psi _{j}^{-1}(\tilde{x})$ $=$ $(\tilde{z},$ $\tilde{\phi},$ $\tilde{r%
}),$%
\begin{equation}
\text{if }x=\sigma (\rho )\tilde{x}\text{ for }\rho \in \mathbb{R}^{+}\text{
(and }U_{j},\text{ }\delta _{j}\text{ sufficiently small)},\text{ then }z=%
\tilde{z}\text{ and }\phi =\tilde{\phi}.  \label{1-2.75}
\end{equation}

Proof of (\ref{1-2.75}): This essentially follows from the facts that $%
\Sigma /\mathbb{R}^{+}$ is a manifold by the properness and the global
freeness of the $\mathbb{R}^{+}$-action, and that any sufficiently small
slice in $\Sigma $ transversal to the $\mathbb{R}^{+}$-orbits gives rise to
a coordinate chart of $\Sigma /\mathbb{R}^{+}.$ We omit the details.


Let us denote $\mathring{D}_{j}^{\varepsilon }$ by $\mathring{D}_{j}$ for
simplicity$.$ Denote the set \{$\sigma (\rho )x$ $\in $ $\Sigma :$ $x\in 
\mathring{D}_{j}\}$ by $\sigma (\rho )\mathring{D}_{j}.$ Define $D_{j}$ to
be the union $\cup _{\rho \in \mathbb{R}^{+}}(\sigma (\rho )\mathring{D}%
_{j}).$ Extend $\psi _{j}^{-1}$ to the map%
\begin{equation}
\tilde{\psi}_{j}^{-1}:D_{j}\rightarrow U_{j}\times (-\delta _{j},\delta
_{j})\times \mathbb{R}^{+},\text{ }\tilde{\psi}_{j}^{-1}(\sigma (\rho
)x):=(z,\phi ,\rho r)  \label{2.10-5}
\end{equation}

\noindent for $x$ $\in $ $\mathring{D}_{j}$ with $\psi _{j}^{-1}(x)$ $=$ $%
(z, $ $\phi ,$ $r).$ We claim that $\tilde{\psi}_{j}^{-1}$ is well defined
and a holomorphic diffeomorphism. Suppose $\sigma (\bar{\rho})x$ $=$ $\sigma
(\tilde{\rho})\tilde{x}$ for $x,$ $\tilde{x}$ $\in $ $\mathring{D}_{j}$ with 
$\psi _{j}^{-1}(x)$ $=$ $(z,$ $\phi ,$ $r),$ $\psi _{j}^{-1}(\tilde{x})$ $=$ 
$(\tilde{z},$ $\tilde{\phi},$ $\tilde{r}).$ Then $x$ $=\sigma (\bar{\rho}%
^{-1}\tilde{\rho})\tilde{x}.$ By (\ref{1-2.75}) we have $z=\tilde{z}$ and $%
\phi =\tilde{\phi}.$ By the line above (\ref{1-2.5}) that says $x$ $=$ $%
\sigma (\rho e^{i\theta })\tilde{x}$ for some $\rho e^{i\theta }$ it follows
from $\phi =\tilde{\phi}$ and (\ref{1-2.5}) that $\theta $ $=$ $0.$ Further $%
\sigma (\rho ^{-1}e^{-i\theta }\bar{\rho}^{-1}\tilde{\rho})\tilde{x}$ $=$ $%
\tilde{x}$ with $\theta $ $=$ $0$ gives $\bar{\rho}^{-1}\tilde{\rho}$ $=$ $%
\rho $ since the $\mathbb{R}^{+}$ part of the action $\sigma $ is globally
free. Now we have%
\begin{equation*}
r\overset{(\ref{1-2.5})}{=}\rho \tilde{r}=\bar{\rho}^{-1}\tilde{\rho}\tilde{r%
}
\end{equation*}%
\noindent It follows that $\bar{\rho}r=\tilde{\rho}\tilde{r}.$ Together with 
$\phi $ $=$ $\tilde{\phi}$ ($\theta $ $=$ $0$ in (\ref{1-2.5})) we obtain $%
\tilde{\psi}_{j}^{-1}(\sigma (\bar{\rho})x)$ $=$ $\tilde{\psi}%
_{j}^{-1}(\sigma (\tilde{\rho})\tilde{x})$ by the definition (\ref{2.10-5})
of $\tilde{\psi}_{j}^{-1},$ giving the well-definedness of $\tilde{\psi}%
_{j}^{-1}.$ Next it is not hard to see that $\tilde{\psi}_{j}^{-1}$ is
injective and surjective. To show that $\tilde{\psi}_{j}^{-1}$ is a
holomorphic diffeomorphism, observe that $\tilde{\psi}_{j}^{-1}|_{\sigma
(\rho )\mathring{D}_{j}}=\tilde{\sigma}(\rho )\circ \psi _{j}^{-1}\circ
\sigma (\rho ^{-1})$ where $\tilde{\sigma}(\rho )$ acts on $U_{j}\times
(-\delta _{j},$ $\delta _{j})\times (1-\varepsilon _{j},$ $1+\varepsilon
_{j})$ $\subset $ $\mathbb{C}^{n-1}\times \mathbb{C}^{\ast }$ by multiplying
the third component by $\rho $. Since $\tilde{\sigma}(\rho ),$ $\psi _{j}$
and $\sigma (\rho ^{-1})$ are all holomorphic diffeomorphisms, we conclude
that $\tilde{\psi}_{j}^{-1}|_{\sigma (\rho )\mathring{D}_{j}}$ hence that $%
\tilde{\psi}_{j}^{-1}|_{D_{j}}$ is a holomorphic diffeomorphism. We have
shown that $\{D_{j}\}_{j}$ form local holomorphic charts. The assertions (%
\ref{1-0}), (\ref{1-1}) and (\ref{1-2}) follow.

To show that $M:=\Sigma /\sigma $ $=$ $(\Sigma /\mathbb{R}^{+})/S^{1}$ has a
natural orbifold structure, first note that $\Sigma /\mathbb{R}^{+}$ $=:$ $%
\tilde{M}$ is a manifold (as mentioned earlier in this proof) with a locally
free action of a compact Lie group $S^{1}$. The topological orbifold
structure of $\tilde{M}/S^{1}$ then follows from an argument in \cite[p.173]%
{Du}. To see that $M$ is a complex orbifold (note that the invariance slice
in \cite[p.173]{Du} is not necessarily a complex analytic one), let $p$ $\in 
$ $\Sigma $ and $G$ be the finite isotropy subgroup of $S^{1}(\subset 
\mathbb{C}^{\ast })$ at $p.$ Write $gq$ for $\sigma (g)q.$ For $p_{1}$ near $%
p$ and $g$ $\in $ $G$ (so $gp_{1}$ near $gp$ $=$ $p),$ $p_{1}$ and $p_{2}$ $%
= $ $gp_{1}$ are given in a coordinate chart $U\times (-\delta ,\delta
)\times \mathbb{R}^{+}$ of $p$ $=$ $(z,0,1)$ by $(z_{i},$ $\delta _{i},$ $%
r_{i}),$ $i $ $=$ $1,$ $2,$ for some $z_{i}\sim z,$ $\delta _{i}\sim 0$ and $%
r_{i}\sim 1. $ In fact $r_{i}$ $=$ $1$ by Lemma \ref{A} $i)$ (the proof of
this particular part does not use the orbifold structure of $\Sigma /\sigma
) $. Identifying $U$ with $U$ $\times $ $\{0\}$ $\times $ $\{1\},$ $g$ is
going to induce a holomorphic diffeomorphism $\tau (g)$ on $U$ (possibly
after shrinking $U$ and $\delta )$ by the composition (compare Remark \ref%
{R-2-11} below)%
\begin{equation}
p_{1}=(z_{1},0,1)\rightarrow gp_{1}=(z_{2},\delta _{2},r_{2})=(z_{2},\delta
_{2},1)\rightarrow (z_{2},0,1)\text{ }\in \mathbb{C}^{n-1}  \label{2-11a}
\end{equation}%
\noindent where the second map arises from a (local) projection $\pi _{U}$ $%
: $ $(z,\theta ,r)$ $\rightarrow $ $(z,0,1).$ We can now rewrite the action
of $\tau (g)$ at $p_{1}$ by 
\begin{equation}
\tau (g)(p_{1})=s_{-\delta _{2}}(gp_{1})  \label{2-12-0}
\end{equation}
\noindent where $s_{-\delta _{2}}$ $=$ $e^{-i\delta _{2}}$ $\in $ $S^{1}$
depends on $p_{1}.$ Note that $\tau (g):$ $U$ $\rightarrow $ $\mathbb{C}%
^{n-1}$ is holomorphic since $\sigma (g)$ and $\pi _{U}$ are so. To directly
prove that $\tau (g)$ is a diffeomorphism, one may try to control $d\tau (g)$
at $p;$ the control is not obvious (however, see Remark \ref{R-2-11}).
Instead, we are going to prove the group action property $\tau (g^{\prime
}g) $ $=$ $\tau (g^{\prime })\circ \tau (g)$ (and $\tau (1)$ $=$ $1).$ From
this it trivially follows that $\tau (g^{-1})$ is the inverse to $\tau (g).$
Set $x_{\delta _{j}}$ $=$ $(z_{j},$ $\delta _{j},$ $1)$ $j$ $=$ $1$, $2.$
Set $g^{\prime }\circ $ $(z_{2},$ $0,$ $1)$ $=$ $(z_{2}^{\prime },$ $\delta
_{2}^{\prime },$ $1)$ so 
\begin{equation}
\tau (g^{\prime })(z_{2},0,1)=(z_{2}^{\prime },0,1)=(s_{-\delta _{2}^{\prime
}}g^{\prime })(z_{2},0,1).  \label{2-12-1}
\end{equation}%
\noindent Because $gx_{\delta _{1}}$ $=$ $g\circ (s_{\delta _{1}}p_{1})$ $=$ 
$s_{\delta _{1}}\circ (gp_{1})$ and (\ref{2-11a}) we have $gx_{\delta _{1}}$ 
$=$ $(z_{2},$ $\delta _{1}+\delta _{2},$ $1)$ and similarly $g^{\prime
}x_{\delta _{2}}$ $=$ $(z_{2}^{\prime },$ $\delta _{2}^{\prime }+\delta
_{2}, $ $1).$ This gives that ($g^{\prime }g)p_{1}$ $=$ $g^{\prime }(z_{2},$ 
$\delta _{2},$ $1)$ $=$ $(z_{2}^{\prime },$ $\delta _{2}^{\prime }+\delta
_{2},$ $1)$, and then $\tau (g^{\prime }g)p_{1}$ $=$ $(z_{2}^{\prime },0,1),$
giving $\tau (g^{\prime }g)p_{1}$ $=$ $s_{-(\delta _{2}^{\prime }+\delta
_{2})}(g^{\prime }g)p_{1}.$ Further 
\begin{equation}
s_{-(\delta _{2}^{\prime }+\delta _{2})}(g^{\prime }g)p_{1}=s_{-\delta
_{2}^{\prime }}(g^{\prime }\circ (s_{-\delta _{2}}(gp_{1})))\overset{(\ref%
{2-12-0})}{=}(s_{-\delta _{2}^{\prime }}g^{\prime })\circ (z_{2},0,1).
\label{2-13-1}
\end{equation}
\noindent Inserting $(z_{2},0,1)$ $=$ $\tau (g)p_{1}$ into (\ref{2-12-1})
one has $(s_{-\delta _{2}^{\prime }}g^{\prime })\circ (z_{2},0,1)$ $=$ $\tau
(g^{\prime })\circ \tau (g)p_{1}.$ By (\ref{2-13-1}) we have proved $\tau
(g^{\prime }g)$ $=$ $\tau (g^{\prime })\circ \tau (g)$ and $\tau (G)$ is a
group ($\tau (1)$ $=$ $1$ is trivial). Consider $\tilde{U}$ $:=$ $\cup
_{g\in G}\tau (g)U$ where every $\tau (g)U$ ($\ni $ $p)$ is a domain in $%
\mathbb{C}^{n-1};$ $\tilde{U}$ is thus a domain in $\mathbb{C}^{n-1}.$ Then $%
(\tilde{U},$ $\tau (G))$ gives a complex orbifold chart (possibly after
shrinking $U$ hence $\tilde{U})$ on $M.$ We omit the discussion about the
transitions between different charts (see Remark \ref{R-2-11}).

As such, $M$ is known to be a normal complex (analytic) space (\cite[Section
IV]{Pr} or \cite[Theorem 4, p. 97]{Car}). Alternatively, by a result of \cite%
{Ho} on the normality of the quotient of a complex manifold by the proper
holomorphic action of a complex Lie group, one can also conclude the
normality of $M$. The last assertion about compactness is obvious.

\endproof%

\begin{remark}
\label{R-2-11} For later use it is shown in Proposition \ref{L-alk}\ $iii)$
that (\ref{2-11a}) above \ can be simplified: $\pi _{U}\circ \sigma (g)$ $=$ 
$\sigma (g)$ on $U$ $=$ $U\times \{0\}\times \{1\}$ $\subset $ $\Sigma $ for 
$g$ $\in $ $G$ ($\sigma $ denotes the original $\mathbb{C}^{\ast }$-action
on $\Sigma $), $i.e.$ $\delta _{2}$ $\equiv $ $0$ in (\ref{2-11a}). Upon
examination the proof of this result (including those in previous sections
on which the proof is based) uses no complex orbifold structure (of $M$)
discussed here. One can also use it to check the remaining conditions (as
recorded in, for instance, \cite[p.172]{Du}) needed for $M$ to be a complex
orbifold. Moreover $\tau $ in the above proof can be shown to be an (group)
isomorphism (see Corollary \ref{8-5-1}).
\end{remark}

Theorem \ref{thm2-1} has an application to the CR case (via $i)$ of Example %
\ref{2.0}):

\begin{corollary}
\label{2.3-5} In the notation of $i)$ of Example \ref{2.0}, the quotient
space $X/S^{1}$ of the CR manifold $X$ by the locally free $S^{1}$-action is
a complex orbifold.
\end{corollary}

\proof
Let $\Sigma $ $=$ $X\times \mathbb{R}^{+}$ by $i)$ of Example \ref{2.0}. The
assertion follows from the corresponding one for $\Sigma $ with the induced $%
\mathbb{C}^{\ast }$-action.

\endproof%

\begin{remark}
\label{2.5-5} It is now not difficult to prove the assertion that all the
compact CR manifolds with transversal, locally free, CR $S^{1}$-action as
considered in \cite{CHT}, can be regarded as "circle bundles" of orbifold
holomorphic line bundles on certain compact complex orbifolds. We omit the
details here.
\end{remark}

Let $\Sigma $ be a complex manifold of complex dimension $n$ with a locally
free holomorphic $\mathbb{C}^{\ast }$-action $\sigma (\lambda ),$ $\lambda $ 
$\in $ $\mathbb{C}^{\ast }.$ For any $m\in \mathbb{Z},$ we define the $m$-th
Fourier component $\hat{\Omega}_{m}^{0,q}$ of $\Omega ^{0,q}(\Sigma )$ by%
\begin{equation*}
\hat{\Omega}_{m}^{0,q}(\Sigma ):=\{\omega \in \Omega ^{0,q}(\Sigma ):\sigma
(\lambda )^{\ast }\omega =\lambda ^{m}\omega \text{ for all }\lambda \in 
\mathbb{C}^{\ast }\}.
\end{equation*}

\noindent Remark that we are actually interested in the subspace $\Omega
_{m}^{0,q}(\Sigma )$ $\subset $ $\hat{\Omega}_{m}^{0,q}(\Sigma )$ (see
Definition \ref{2m}).

To describe $\hat{\Omega}_{m}^{0,q}(\Sigma ),$ recalling local holomorphic
coordinates $z_{1},$\textit{\ }$z_{2},$\textit{\ }$..,$\textit{\ }$z_{n-1},$%
\textit{\ }$w$ in Proposition \ref{p-gue2-1} and using (\ref{2.10-5}) 
\begin{equation}
\sigma (\lambda )(z_{1},\mathit{\ }z_{2},\mathit{\ }..,\mathit{\ }z_{n-1},%
\mathit{\ }w)=(z_{1},\mathit{\ }z_{2},\mathit{\ }..,\mathit{\ }z_{n-1},\text{
}\lambda w)  \label{C6}
\end{equation}

\noindent for $\lambda \in \mathbb{C}_{\delta }$ with small $\delta >0$ (see
(\ref{3-0.75}) for the definition of $\mathbb{C}_{\delta }$), we write an
element $\omega $ $\in $ $\hat{\Omega}_{m}^{0,q}(\Sigma )$ as follows:%
\begin{equation}
\omega =f_{I_{q}}(z,\bar{z},w,\bar{w})d\bar{z}^{I_{q}}+g_{I_{q-1}}(z,\bar{z}%
,w,\bar{w})d\bar{z}^{I_{q-1}}\wedge d\bar{w}  \label{C6-1}
\end{equation}

\noindent where $z$ $=$ $(z_{1},\mathit{\ }z_{2},\mathit{\ }..,\mathit{\ }%
z_{n-1})$ and $I_{q}$ denotes the multi-index ($i_{1},$ $..,$ $i_{q})$, $1$ $%
\leq $ $i_{1}$ $<$ $i_{2}$ $<$ $\cdot \cdot $ $<$ $i_{q}$ $\leq $ $n.$ We
are going to simplify the expression (\ref{C6-1}); the result is given in (%
\ref{C9-1}) below.

The condition $\rho (\lambda )^{\ast }\omega =\lambda ^{m}\omega $ in $(z,w)$
reads%
\begin{eqnarray}
f_{I_{q}}(z,\bar{z},\lambda w,\bar{\lambda}\bar{w}) &=&\lambda
^{m}f_{I_{q}}(z,\bar{z},w,\bar{w}),  \label{C7} \\
g_{I_{q-1}}(z,\bar{z},\lambda w,\bar{\lambda}\bar{w})\bar{\lambda}
&=&\lambda ^{m}g_{I_{q-1}}(z,\bar{z},w,\bar{w}).  \notag
\end{eqnarray}

\noindent Differentiating the first equation of (\ref{C7}) in $\bar{\lambda}$
gives $f_{I_{q},\bar{w}}(z,\bar{z},\lambda w,\bar{\lambda}\bar{w})\bar{w}=0$
(henceforth $f_{I_{q},\bar{w}}=\partial f_{I_{q}}/\partial \bar{w}$ etc.) so
that $f_{I_{q},\bar{w}}(z,\bar{z},w,\bar{w})=0,$ $f_{I_{q}}=f_{I_{q}}(z,\bar{%
z},w).$ Similarly, differentiating it in $\lambda $ gives $f_{I_{q},w}(z,%
\bar{z},\lambda w)w$ $=$ $m\lambda ^{m-1}f_{I_{q}}(z,\bar{z},w).$ This is
solved (by setting $\lambda =1)$ to be $f_{I_{q}}(z,\bar{z},w)=f_{I_{q}}(z,%
\bar{z},1)w^{m}+h_{I_{q}}(z,\bar{z})$ for some $h_{I_{q}}(z,\bar{z}).$ It
follows from the first equation of (\ref{C7}) (with $w=1)$ that $h_{I_{q}}(z,%
\bar{z})$ $\equiv $ $0.$ Hence%
\begin{equation}
f_{I_{q}}(z,\bar{z},w)=f_{I_{q}}(z,\bar{z},1)w^{m}.  \label{C8}
\end{equation}

\noindent Differentiating the second equation of (\ref{C7}) in $\bar{\lambda}
$ gives%
\begin{equation}
\frac{\partial g_{I_{q-1}}}{\partial \bar{w}}(z,\bar{z},\lambda w,\bar{%
\lambda}\bar{w})\bar{w}\bar{\lambda}+g_{I_{q-1}}(z,\bar{z},\lambda w,\bar{%
\lambda}\bar{w})=0.  \label{2-13.5}
\end{equation}

\noindent Setting $\lambda $ $=$ $1$, we then solve (\ref{2-13.5}): $%
g_{I_{q-1}}=\bar{w}^{-1}C_{I_{q-1}}(z,\bar{z},w)$ for some function $%
C_{I_{q-1}}$ $=:$ $C.$ Substituting this into (\ref{C7}) gives $C(z,\bar{z}%
,\lambda w)=\lambda ^{m}C(z,\bar{z},w).$ In this formula, taking $w$ $=$ $1$
and rewriting $\lambda $ as $w$, we get $C(z,\bar{z},w)=C(z,\bar{z},1)w^{m}$
and conclude that 
\begin{equation}
g_{I_{q-1}}=C_{I_{q-1}}(z,\bar{z},1)\bar{w}^{-1}w^{m}.  \label{C9}
\end{equation}

\noindent From (\ref{C6-1}), (\ref{C8}) and (\ref{C9}), we obtain 
\begin{equation}
\omega =f_{I_{q}}(z,\bar{z})w^{m}d\bar{z}^{I_{q}}+C_{I_{q-1}}(z,\bar{z})w^{m}%
\bar{w}^{-1}d\bar{z}^{I_{q-1}}\wedge d\bar{w}.  \label{C9-1}
\end{equation}

It is straightforward to deduce the transformation law for $f_{I_{q}}$ and $%
C_{I_{q-1}}$ of (\ref{C9-1}) under the change of holomorphic coordinates (%
\ref{C0}). We omit the details.





Provisionally let us define 
\begin{equation}
\hat{H}_{m}^{q}(\Sigma ,\mathcal{O}):=\frac{\text{Ker}\{\bar{\partial}:\hat{%
\Omega}_{m}^{0,q}(\Sigma )\rightarrow \hat{\Omega}_{m}^{0,q+1}(\Sigma )\}}{%
\func{Im}\{\bar{\partial}:\hat{\Omega}_{m}^{0,q-1}(\Sigma )\rightarrow \hat{%
\Omega}_{m}^{0,q}(\Sigma )\}}  \label{Hmq1}
\end{equation}%
\noindent (notice the difference between (\ref{Hmq1}) and (\ref{Hmq}),
marked by tilde here).

\begin{definition}
\label{2m} (Regularity condition) For $m\in \mathbb{Z}$ let $\Omega
_{m}^{0,q}(\Sigma )$ denote the space of elements $\omega $ which satisfy 
\begin{equation}
i)\text{ }\omega \in \hat{\Omega}_{m}^{0,q}(\Sigma ),\text{ }ii)\text{ }%
\omega =f_{I_{q}}(z,\bar{z})w^{m}d\bar{z}^{I_{q}}\text{ in (one hence all)
local coordinate(s).}  \label{1-2-1}
\end{equation}
\end{definition}

It is easily seen that $\Omega _{m}^{0,q}(U)$ $\neq $ $\{0\}$ if the closure 
$\bar{U}$ of some $\mathbb{C}^{\ast }$-invariant open subset $U$ $\subset $ $%
\Sigma $ fully lies in the principal $\mathbb{C}^{\ast }$-stratum $\Sigma
_{p_{1}}$ of $\Sigma $ (see (\ref{1-4})). If $\bar{U}$ intersects the
lower-dimensional strata of $\Sigma ,$ the situation is somewhat delicate
(see the case $ii)$ stated after (\ref{6-1g})), and we resort to Proposition %
\ref{projm}\ for the related issues.

By analogy with (\ref{Hmq1}) with $\Omega _{m}^{0,q}(\Sigma )$ in place of $%
\hat{\Omega}_{m}^{0,q}(\Sigma ),$ one can define $H_{m}^{q}(\Sigma ,\mathcal{%
O})$ as given in (\ref{Hmq}). A motivation is seen in Proposition \ref%
{p-gue2-2} below; see Section 4 (cf. the discussion from (\ref{H3}) onwards)
for more.

In the case where $\Sigma $ $=$ $\hat{L}\backslash \{0$-section$\}$ $=:$ $%
\hat{L}^{\prime }$ (see $ii)$ of Example \ref{2.0})$,$ we wonder if or when $%
\hat{H}_{m}^{q}(\hat{L}^{\prime },\mathcal{O})$ is finite-dimensional. It is
easily seen that in (\ref{C9-1}), $C_{I_{q-1}}$ $=$ $0$ for $m$ $=$ $0,$ $1$
provided that $g_{I_{q-1}}$ in (\ref{C6-1}) can be continuously extended to $%
w$ $=$ $0$. Similarly, for $m$ $\geq $ $2$ we still get $C_{I_{q-1}}$ $=$ $0$
if we require that the extension of $g_{I_{q-1}}$ is $C^{m-1}$ in $\bar{w}$
at $w$ $=$ $0.$ Namely, under certain regularity assumption along
\textquotedblleft $w$ $=$ $0"$ we have $C_{I_{q-1}}$ $=$ $0$ and by (\ref%
{C9-1}) 
\begin{equation}
\omega =f_{I_{q}}(z,\bar{z})w^{m}d\bar{z}^{I_{q}},\text{ }m\geq 0.
\label{C10}
\end{equation}

\noindent Similarly for $m$ $<$ $0,$ (\ref{1-2-1}) of Definition \ref{2m}
can be regarded as a regularity condition at \textquotedblleft $w=\infty ".$
In general $\hat{H}_{m}^{q}(\Sigma ,\mathcal{O})$ in (\ref{Hmq1}) is not
expected to be linearly isomorphic to $H_{m}^{q}(\Sigma ,\mathcal{O}).$

As a matter of fact, $H_{m}^{q}(\Sigma ,\mathcal{O})$ is necessarily
finite-dimensional (see Theorem \ref{t-4-2}).

Remark that the elements of $\Omega _{m}^{0,q}(\Sigma )$ $\subset $ $\hat{%
\Omega}_{m}^{0,q}(\Sigma )$ have the following transformation law. In two
systems of holomorphic coordinates $(z,w)$ and $(\tilde{z},\tilde{w})$, we
have%
\begin{equation}
\tilde{w}=w\varphi (z_{1},z_{2},..,z_{n-1}),\text{ }\tilde{z}_{j}=\mu
_{j}(z_{1},z_{2},..,z_{n-1}),\text{ }1\leq j\leq n-1  \label{C10-1}
\end{equation}

\noindent (see (\ref{C0})). The condition $s(z,\bar{z})w^{m}$ $=$ $\tilde{s}(%
\tilde{z},\overline{\tilde{z}})\tilde{w}^{m}$ for $s,$ $\tilde{s}$ being $%
(0,q)$-forms in $z,$ $\tilde{z}$ respectively, implies%
\begin{equation}
s(z,\bar{z})=\tilde{s}(\mu _{1}(z),..,\mu _{n-1}(z),\overline{\mu _{1}(z)}%
,..,\overline{\mu _{n-1}(z)})(\varphi (z))^{m}.  \label{C11}
\end{equation}



These will help to verify that certain transversally $spin^{c}$ Dirac
operators (cf. Lemma \ref{5-1.25} and Definition \ref{5-1.5}) are globally
defined.

Let us look into the aforementioned case $\Sigma $ $=$ $\hat{L}\backslash
\{0 $-section\} =: $\hat{L}^{\prime }$ more closely. Let $L^{\ast }$ denote
the dual holomorphic line bundle of $L.$ Let $\Omega ^{0,q}(M,(L^{\ast
})^{\otimes m})$ denote the space of ($L^{\ast })^{\otimes m}$-valued $(0,q)$%
- forms on $M.$ It is straightforward to verify the following (see also
Remark \ref{PLMB}). 




\begin{proposition}
\label{p-gue2-2}\textit{The map }$\psi _{q,m}$ \textit{from }$\phi =\eta
\otimes (e^{\ast })^{\otimes m}$\textit{\ }$\in $\textit{\ }$\Omega
^{0,q}(M,(L^{\ast })^{\otimes m})$\textit{\ to }$\omega $\textit{\ }$\in $%
\textit{\ }$\Omega _{m}^{0,q}(\hat{L}^{\prime })$ (see \textit{Definition %
\ref{2m}) }given locally \textit{by}%
\begin{equation*}
\omega (p,we)=\eta (p)w^{m}
\end{equation*}%
\textit{\noindent is globally defined and a vector space isomorphism.}
Moreover $\psi _{q,m}$ commutes with the respective $\bar{\partial}$
operators, and thus $H_{m}^{q}(\hat{L}^{\prime },\mathcal{O})\simeq H_{\bar{%
\partial}}^{0,q}(M,(L^{\ast })^{\otimes m}).$
\end{proposition}

%
%
%
%
%
%





Proposition \ref{p-gue2-2} can be generalized for those $\Sigma $ other than 
$\hat{L}^{\prime };$ see Proposition \ref{dualLm}. It will be used in
Sections 5 and 9; see (\ref{5-2.5}) and Remark \ref{PLMB}.

Our next task is to define the adjoint operators of 
\begin{equation*}
\bar{\partial}_{\hat{L}^{\prime },m}:\Omega _{m}^{0,q}(\hat{L}^{\prime
})\rightarrow \Omega _{m}^{0,q+1}(\hat{L}^{\prime })\text{ \ }(\hat{L}%
^{\prime }=\hat{L}\backslash \{0\text{-section}\})
\end{equation*}%
\noindent and 
\begin{equation*}
\bar{\partial}_{M,(L^{\ast })^{\otimes m}}:\Omega ^{0,q}(M,(L^{\ast
})^{\otimes m})\rightarrow \Omega ^{0,q+1}(M,(L^{\ast })^{\otimes m}),
\end{equation*}%
\noindent and to compare (via Proposition \ref{p-gue2-2}) the two adjoint
operators so defined. For this purpose, we need first of all to endow a
metric on $\hat{L}^{\prime }$ and a fibre metric on $L$ (and hence on $%
L^{\ast })$. We will do it for general $\Sigma $ in the next section.

\section{\textbf{A Hermitian metric on complex manifolds with }$\mathbb{C}%
^{\ast }$-\textbf{action\label{S-metric}}}

Now we consider a general complex manifold $\Sigma $ with a holomorphic $%
\mathbb{C}^{\ast }$-action $\sigma $ satisfying (\ref{1-0}), (\ref{1-1}) and
(\ref{1-2})$.$ We want to construct a Hermitian metric $G_{a,m}$ on $\Sigma $
as remarked in the end of the last section. This metric is going to be $%
S^{1} $-invariant although not $\mathbb{C}^{\ast }$\textbf{-}invariant (here 
$S^{1} $ $\subset $ $\mathbb{C}^{\ast }$ naturally). For its $S^{1}$%
-invariance, see Remark \ref{7-11b}\ in Section 7.

Let $L_{\Sigma }$ be the holomorphic line bundle over $\Sigma ,$ whose fibre 
$L_{\Sigma ,q}$ at $q\in \Sigma $ consists of complex multiples of $\digamma
(q)=\frac{\partial }{\partial \zeta }|_{q}=w\frac{\partial }{\partial w}%
|_{q} $ (see (\ref{2-2})). Note that $L_{\Sigma }$ is a holomorphic
subbundle of the holomorphic tangent bundle $T^{1,0}\Sigma .$ Given $q$ $\in 
$ $\Sigma ,$ we define a nowhere vanishing holomorphic section $v$ $:\Sigma
\rightarrow $ $L_{\Sigma }$ by 
\begin{equation}
v_{q}:=\frac{d}{d\lambda }|_{\lambda =1}\sigma (\lambda )q.  \label{3-0}
\end{equation}%
\noindent Observe that $L_{\Sigma }$ is a $\mathbb{C}^{\ast }$-equivariant
bundle: A natural holomorphic $\mathbb{C}^{\ast }$-action $\tilde{\sigma}$
on $L_{\Sigma }$ is given by 
\begin{equation}
\tilde{\sigma}(\lambda )v_{q}:=\lambda ^{-1}v_{\sigma (\lambda )q}
\label{3.0}
\end{equation}

\noindent so that $\pi _{L_{\Sigma }}\circ \tilde{\sigma}(\lambda )=\sigma
(\lambda )\circ \pi _{L_{\Sigma }}$ where $\pi _{L_{\Sigma }}$ $:$ $%
L_{\Sigma }\rightarrow \Sigma $ is the projection.

We divide the construction of the metric $G_{a,m}$ into three steps.

\bigskip

\textbf{Step 1. A }$\mathbb{C}^{\ast }$\textbf{-invariant Hermitian metric
on }$L_{\Sigma }$ \textbf{and} \textbf{a global 2-form }$\partial _{z}\bar{%
\partial}_{z}\log h(z,\bar{z}).$

On each patch $D_{j}$ (see (\ref{1-0})) one can easily choose a fibre
Hermitian metric $<\cdot ,\cdot >_{j}$ on $L_{\Sigma }|_{D_{j}}$ such that $<%
\tilde{\sigma}(\lambda )s_{q},\tilde{\sigma}(\lambda )t_{q}>_{j}$ $=$ $%
<s_{q},t_{q}>_{j}$ holds whenever $q\in D_{j},$ $\lambda $ $\in $ $\mathbb{R}%
^{+}$ and any $s_{q},t_{q}$ $\in $ $L_{\Sigma ,q}.$ Take a partition of
unity $\chi _{j}$ supported on $D_{j},$ satisfying $\sigma (\lambda )^{\ast
}\chi _{j}$ $=$ $\chi _{j}$ for every $\lambda $ $\in $ $\mathbb{R}^{+}.$
Define a Hermitian metric $<\cdot ,\cdot >^{\prime }$ on $L_{\Sigma }$ by
the sum of $\chi _{j}<\cdot ,\cdot >_{j}$ (over $j),$ which is $\tilde{\sigma%
}(\lambda )$-invariant for $\lambda $ $\in $ $\mathbb{R}^{+}.$ We then take
the average of the $S^{1}$-action to get a $\mathbb{\tilde{\sigma}}$%
-invariant Hermitian metric $<\cdot ,\cdot >$ or $<\cdot ,\cdot >_{L_{\Sigma
}}$ on $L_{\Sigma }.$

For a vector $e$ $\in $ $L_{\Sigma },$ we write $||e||_{L_{\Sigma }}$ or $%
||e||$ $:=$ $\sqrt{<e,e>}.$ Define a global function $l$ :$\Sigma
\rightarrow \mathbb{R}^{+}$ by 
\begin{equation}
l(q):=||v_{q}||^{2}  \label{lq0}
\end{equation}%
\noindent for $q\in \Sigma $ and $v_{q}$ in (\ref{3-0})$.$ In local
coordinates $(z,w)$ (where $z$ $=$ $(z_{1},$ $..,$ $z_{n-1}))$ we have $%
v_{(z,\lambda )}=(w\partial /\partial w)|_{(z,\lambda )}$ and%
\begin{equation}
l(q)=h(z,\bar{z})|\lambda |^{2},\text{ }h(z,\bar{z}):=||(\partial /\partial
w)|_{(z,\lambda )}||^{2}  \label{lq}
\end{equation}

\noindent where $h(z,\bar{z})$ is independent of $\lambda .$ For, the metric 
$<\cdot ,\cdot >$ on $L_{\Sigma }$ is $\tilde{\sigma}$-invariant by
construction and $\frac{\partial }{\partial w}$ is seen to be $\tilde{\sigma}
$-invariant:%
\begin{eqnarray}
\tilde{\sigma}(\lambda )(\frac{\partial }{\partial w}|_{(z,1)}) &\overset{}{=%
}&\tilde{\sigma}(\lambda )v_{(z,1)}\overset{(\ref{3.0})}{=}\lambda
^{-1}v_{(z,\lambda )}  \label{ddw} \\
&\overset{}{=}&\lambda ^{-1}(\lambda \frac{\partial }{\partial w}%
|_{(z,\lambda )})=\frac{\partial }{\partial w}|_{(z,\lambda )}  \notag
\end{eqnarray}%
\noindent whenever $q$ $\in $ $D_{j},$ $\sigma (\lambda )q$ $\in $ $D_{j}$
and $\lambda $ $\in $ $C_{\delta _{j}}$ where%
\begin{equation}
C_{\delta _{j}}:=\{\rho e^{i\theta }\in \mathbb{C}^{\ast }:(\theta ,\rho
)\in (-\delta _{j},\delta _{j})\times \mathbb{R}^{+}\}.  \label{3-0.75}
\end{equation}

\noindent We refer to Remark \ref{7-11b} and Lemma \ref{A} $iv)$ for the
large-angle invariant property of $l(q)$ and $h(z,\bar{z}).$

Writing $\partial _{z}\bar{\partial}_{z}\log h(z,\bar{z})$ :$=$ $(\partial
_{z_{\alpha }}\partial _{\bar{z}_{\beta }}\log h(z,\bar{z}))dz_{\alpha
}\wedge d\bar{z}_{\beta },$ by using (\ref{C0}) we have 
\begin{equation}
\partial _{z}\bar{\partial}_{z}\log h(z,\bar{z})=\partial _{\tilde{z}}\bar{%
\partial}_{\tilde{z}}\log h(\tilde{z},\overline{\tilde{z}})  \label{3-6-1}
\end{equation}
\noindent which means that $\partial \bar{\partial}\log h$ is globally
defined.

\bigskip

\textbf{Step 2. A Hermitian metric }$G_{a}$\textbf{\ on }$\Sigma $\textbf{\
with local formulas.}

\begin{notation}
\label{N-3-1} Let $\pi :\Sigma \rightarrow M:=\Sigma /\mathbb{\sigma }$ be
the projection. Recall that $M$ is a compact complex orbifold by Theorem \ref%
{thm2-1}. Choose a Hermitian metric $g_{M}$ (not necessarily K\"{a}hler) on $%
M$ (in the orbifold sense; see for instance \cite[p.176]{Du}).
\end{notation}

Recall that we can choose local holomorphic patches $(D_{j},$ $(z,w))$ with $%
|w|$ extended to $\mathbb{R}^{+}$ (see (\ref{1-0}) and Theorem \ref{thm2-1}%
). We define%
\begin{eqnarray}
g_{1}:= &&\partial _{\Sigma }\bar{\partial}_{\Sigma }l-(\partial _{z}\bar{%
\partial}_{z}\log h)l,  \label{3-7.5} \\
g_{2}:= &&\partial _{\Sigma }\bar{\partial}_{\Sigma
}(l^{-2a})-(-2a)(\partial _{z}\bar{\partial}_{z}\log h)l^{-2a}  \notag
\end{eqnarray}

\noindent where \textquotedblleft $a"$ is a positive large number and $l$ is
defined in (\ref{lq0}). Let $\varphi _{1}$ be a cutoff function on $\mathbb{R%
}$ such that $\varphi _{1}(x)$ $=$ $1$ for $x$ $\in $ $[-1,1]$ and $\varphi
_{1}(x)$ $=$ $0$ for $|x|$ $\geq $ $2.$ We define a Hermitian metric $G_{a}$
on $\Sigma $ by using $g_{1},$ $g_{2}$ of (\ref{3-7.5}) and $g_{M}$ above:%
\begin{equation}
G_{a}:=\pi ^{\ast }g_{M}+(\varphi _{1}\circ l)g_{1}^{\#}+(1-\varphi
_{1}\circ l)g_{2}^{\#}  \label{Ga}
\end{equation}

\noindent where $g_{1}^{\#},$ $g_{2}^{\#}$ are metrics associated to the $2$%
-forms $g_{1},$ $g_{2}$ respectively.

In local coordinates $(z,w)$ we write (\ref{3-7.5}) as%
\begin{eqnarray}
g_{1} &=&\partial _{\Sigma }\bar{\partial}_{\Sigma }(hw\bar{w})-(\partial
_{z}\bar{\partial}_{z}\log h)hw\bar{w}  \label{M0} \\
g_{2} &=&\partial _{\Sigma }\bar{\partial}_{\Sigma }[(hw\bar{w}%
)^{-2a}]-(-2a)(\partial _{z}\bar{\partial}_{z}\log h)(hw\bar{w})^{-2a}. 
\notag
\end{eqnarray}

\noindent Denote $\frac{\partial h}{\partial z_{\alpha }},$ $\frac{\partial h%
}{\partial \bar{z}_{\alpha }}$, $\frac{\partial ^{2}h}{\partial \bar{z}%
_{\beta }\partial z_{\alpha }}$ by $h_{\alpha },$ $h_{\bar{\alpha}},$ $%
h_{\alpha \bar{\beta}}$. A direct computation shows%
\begin{eqnarray}
g_{1} &=&hdw\wedge d\bar{w}+h^{-1}h_{\alpha }h_{\bar{\beta}}w\bar{w}%
dz_{\alpha }\wedge d\bar{z}_{\beta }  \label{M0-1} \\
&&+h_{\bar{\alpha}}\bar{w}dw\wedge d\bar{z}_{\alpha }+h_{\alpha }wdz_{\alpha
}\wedge d\bar{w}  \notag
\end{eqnarray}

\noindent and%
\begin{eqnarray}
g_{2} &=&4a^{2}(hw\bar{w})^{-2a}\{(w\bar{w})^{-1}dw\wedge d\bar{w}%
+h^{-2}h_{\beta }h_{\bar{\alpha}}dz_{\beta }\wedge d\bar{z}_{\alpha }
\label{M0-2} \\
&&+h^{-1}h_{\bar{\alpha}}w^{-1}dw\wedge d\bar{z}_{\alpha }+h^{-1}h_{\alpha }%
\bar{w}^{-1}dz_{\alpha }\wedge d\bar{w}\}.  \notag
\end{eqnarray}

Given a point $p_{0}$ $\in $ $\Sigma ,$ we can find coordinates $(z,w)$
(still distinguished in the sense of Proposition \ref{p-gue2-1}) with $%
(z,w)(p_{0})$ $=$ $(z_{0},w_{0})$ such that 
\begin{equation}
h(z_{0},\bar{z}_{0})=1\text{ and }dh(z_{0},\bar{z}_{0})=0  \label{M0-3}
\end{equation}%
\noindent (cf. \cite[p.80]{Wel}).

\begin{remark}
\label{R-h} In fact we only need to change $w$ to $j(z)w$ while the
coordinate $z$ is fixed to achieve (\ref{M0-3}). So $h$ depends only on the
choice of $w$-coordinate, denoted as $h^{w}$ below. If we make a change of $%
w:$ $\tilde{w}=cw$ for a constant $c\in $ $\mathbb{C}^{\ast }$ (with $z$%
-coordinate fixed)$,$ we then have $h^{\tilde{w}}(z,\bar{z})=h^{w}(z,\bar{z}%
)|c|^{-2}.$
\end{remark}

Thus, at $p_{0}$ we simplify:%
\begin{equation}
G_{a}=(g_{M})_{\alpha \bar{\beta}}(z_{0},\bar{z}_{0})dz_{\alpha }d\bar{z}%
_{\beta }+(\varphi _{1}(w_{0}\bar{w}_{0})+\varphi _{2}(w_{0}\bar{w}%
_{0})4a^{2}(w_{0}\bar{w}_{0})^{-2a-1})dwd\bar{w}.  \label{M1-1}
\end{equation}

\noindent where $dz_{\alpha }d\bar{z}_{\beta }$ and $dwd\bar{w}$ denote the
symmetric product of $1$-forms (this way of expression for a Hermitian
metric follows the notation of \cite[p.155 (4)]{KN}) and $\varphi _{2}$ $:=$ 
$1-\varphi _{1}$.

So the metric $G_{a}$ of (\ref{M1-1}) has the property that
\textquotedblleft base" $z$-slice and \textquotedblleft fibre" $w$-slice
yield an orthogonal splitting at $p_{0}$ (here $z$-slice is noncanonical and
depends on the choice of coordinates). Furthermore, the $w$-slice (which is
always part of a $\mathbb{C}^{\ast }$-orbit, cf. (\ref{C6})) is totally
geodesic (cf. Proposition \ref{p-gue3-1} below).

\bigskip

\textbf{Step 3. The normalized metric }$G_{a,m}$\textbf{\ and its volume
form }$dv_{\Sigma ,m}$ \textbf{for }$m\geq 0.$

Assume $m\geq 0.$ Following Step 2, we have the intrinsic expression of the
volume form $dv_{G_{a}}$ or $dv_{\Sigma }$ as follows:%
\begin{equation}
dv_{\Sigma }=\pi ^{\ast }dv_{M}\wedge dv_{f}  \label{hfv}
\end{equation}%
\noindent where $\pi ^{\ast }dv_{M}(=dv(z)$ in coordinates $(z,w))$ ($dv_{M}$
denotes the volume form of $M)$ is the volume form of $\pi ^{\ast }g_{M}$
and the $2$-form $dv_{f}$ $=$ $dv_{fibre}$ on $\Sigma $ is basically the
area form on the $\mathbb{C}^{\ast }$-orbit extended to $\Sigma $ by using
the embedding of (vertical, fibrewise) forms via the orthogonal splitting
given by the metric (\ref{M1-1}).

Denote by $\mathbb{C}^{\ast }\circ p_{0}$ the $\mathbb{C}^{\ast }$-orbit $%
\{\lambda \circ p_{0}:$ $\lambda \in \mathbb{C}^{\ast }\}$ passing through $%
p_{0}.$ Define $\tau _{p_{0}}:\mathbb{C}^{\ast }$ $\rightarrow $ $\mathbb{C}%
^{\ast }\circ p_{0}$ $\subset $ $\Sigma $ by $\tau _{p_{0}}(\lambda )$ $=$ $%
\lambda \circ p_{0}.$ Define for $l$ of (\ref{lq0})%
\begin{equation}
\lambda _{m}(p_{0}):=\int_{\mathbb{C}^{\ast }}(\tau _{p_{0}}^{\ast
}l)^{m}(\tau _{p_{0}}^{\ast }dv_{f}).  \label{lambdam}
\end{equation}%
\noindent This is an integral of the function $l^{m}$ along the orbit $%
\mathbb{C}^{\ast }\circ p_{0}$ (possibly with "multiplicities") and is
easily seen to be independent of the choice of the point $p_{0}$ in the same
orbit.

Let $p_{0}\in \Sigma \backslash \Sigma _{\text{sing}},$ i.e. $p_{0}$ lies in
the principal stratum. Choosing the coordinates $(z,w)$ such that $h(z_{0},%
\bar{z}_{0})$ $=$ $1$ and $dh(z_{0},\bar{z}_{0})$ $=$ $0$ at $p_{0}$ (\ref%
{M0-3}), we have (cf. (\ref{C6}) for $\delta $ $=$ $\pi $ in $C_{\delta }$
since $p_{0}$ $\notin $ $\Sigma _{\text{sing}})$%
\begin{equation}
\tau _{p_{0}}^{\ast }dv_{f}(w)=dv(|w|)\wedge dv(\phi ),\text{ }w=|w|e^{i\phi
}\in \mathbb{C}^{\ast }  \label{dvfibre}
\end{equation}

\noindent where $dv(\phi )$ (or $dv_{S^{1}}(\phi ))$ :$=$ $d\phi $ and (cf. (%
\ref{M1-1}))%
\begin{equation}
dv(|w|)(\text{or }dv_{\mathbb{R}^{+}}(|w|)):=[\varphi _{1}(|w|^{2})+\varphi
_{2}(|w|^{2})4a^{2}|w|^{-4a-2}]|w|d|w|.  \label{dvfibre1}
\end{equation}

To compute $\lambda _{m}(p_{0})$ of (\ref{lambdam}), by (\ref{dvfibre}) and (%
\ref{lq}) that $l(q)$ $=$ $h(z_{0},\bar{z}_{0})w\bar{w}$ $=$ $|w|^{2}$ we
have (recalling $C_{\delta }$ $=$ $\mathbb{C}^{\ast }$ here)%
\begin{equation}
\lambda _{m}(p_{0})=\int_{\mathbb{C}^{\ast }}|w|^{2m}dv(|w|)\wedge dv(\phi
)=2\pi \int_{\mathbb{R}^{+}}|w|^{2m}dv(|w|).  \label{3-16.5}
\end{equation}%
\noindent It follows from (\ref{dvfibre1}) and (\ref{3-16.5}) that the
numbers $\lambda _{m}(p_{0})$ are the same for all $\mathbb{C}^{\ast }$%
-orbits (by the obvious continuity of (\ref{lambdam}) when $p_{0}$ is across 
$\Sigma _{\text{sing}})$.

\begin{notation}
\textbf{\label{3-n}} Let $\lambda _{m}$ denote the common number $\lambda
_{m}(p_{0})$ in (\ref{3-16.5})$.$ Let $dv_{m}(|w|)$ $:=$ $2\pi
dv(|w|)/\lambda _{m}$ denote the normalized volume on $\mathbb{R}^{+},$ so
that 
\begin{equation}
\int_{\mathbb{R}^{+}}|w|^{2m}dv_{m}(|w|)=1.  \label{3-16.75}
\end{equation}
\end{notation}

\bigskip

The normalized metric $G_{a,m}$ is given as%
\begin{equation}
G_{a,m}:=\pi ^{\ast }g_{M}+(\varphi _{1}\circ l)\frac{g_{1}^{\#}}{\lambda
_{m}}+(1-\varphi _{1}\circ l)\frac{g_{2}^{\#}}{\lambda _{m}}  \label{metric}
\end{equation}%
\noindent on $\Sigma ,$ where $g_{1}^{\#},$ $g_{2}^{\#}$ are as in Step 2
(cf. (\ref{Ga})). The associated volume form $dv_{\Sigma ,m}$ has the
following intrinsic expression (cf. (\ref{hfv}))%
\begin{equation}
dv_{\Sigma ,m}=\pi ^{\ast }dv_{M}\wedge dv_{f,m}  \label{volume}
\end{equation}

\noindent where $\pi ^{\ast }dv_{M}(=dv(z)$ in coordinates $(z,w))$ is the
volume form of $\pi ^{\ast }g_{M}$ (recall that $\pi :\Sigma \rightarrow M$ $%
=$ $\Sigma /\sigma $ is the natural projection) and 
\begin{equation}
dv_{f,m}=dv_{f}/\lambda _{m},\text{ }\tau _{p_{0}}^{\ast
}dv_{f,m}(w)=dv_{m}(|w|)\wedge \frac{dv(\phi )}{2\pi }.  \label{3-18.75}
\end{equation}

Writing 
\begin{equation}
dv_{f,m}=l(q)^{-m}d\hat{v}_{m}(q),  \label{fibrenv_0}
\end{equation}%
\noindent one sees, with $l=hw\bar{w},$%
\begin{equation}
\tau _{p_{0}}^{\ast }d\hat{v}_{m}=(hw\bar{w})^{m}(\tau _{p_{0}}^{\ast
}dv_{f,m})=|w|^{2m}dv_{m}(|w|)\wedge \frac{dv(\phi )}{2\pi }.  \label{3.29-5}
\end{equation}%
\noindent In summary (for $h(p_{0})$ $=$ $1$ and $dh(p_{0})$ $=$ $0)$ 
\begin{eqnarray}
(\tau _{p_{0}}^{\ast }d\hat{v}_{m})(|w|) &=&|w|^{2m}dv_{m}(|w|),
\label{3-19.5} \\
\int_{\mathbb{R}^{+}}(\tau _{p_{0}}^{\ast }d\hat{v}_{m})(|w|) &\overset{(\ref%
{3-16.75})}{=}&1.  \notag
\end{eqnarray}

\noindent Since $l(q)$ is independent of the choice of $(z,w)$ coordinates ((%
\ref{lq0}), (\ref{lq})), intrinsically we have (cf. (\ref{3.29-5}))%
\begin{equation}
\int_{\mathbb{C}^{\ast }}\tau _{p_{0}}^{\ast }d\hat{v}_{m}=\frac{1}{\lambda
_{m}}\int_{\mathbb{C}^{\ast }}(\tau _{p_{0}}^{\ast }l)^{m}\tau
_{p_{0}}^{\ast }dv_{f}=\int_{\mathbb{C}^{\ast }}(\tau _{p_{0}}^{\ast
}l)^{m}\tau _{p_{0}}^{\ast }dv_{f,m}=1.  \label{fibrenv}
\end{equation}

We will often omit the pullback notation $\tau _{p_{0}}^{\ast }$ in later
computations.

Remark that the 2-form $d\hat{v}_{m}$ above is used in the index formula (%
\ref{MF}) of Theorem \ref{main_theorem} stated in the Introduction.

\begin{remark}
\textbf{\label{3-r}} For $f\in C^{\infty }(\Sigma )$ with $f$ $=$ $%
O(|w|^{m}) $ in local coordinates $(z,w),$ it follows from (\ref{dvfibre1})
that $\int_{\Sigma }|f(x)|^{2}dv_{\Sigma ,m}(x)<\infty $ for $a$ large, say, 
$a>\frac{m}{2}\geq 0.$
\end{remark}

\begin{lemma}
\label{L-3-5} For $a>\frac{m}{2}\geq 0$ the normalized metric $G_{a,m}$ (\ref%
{metric}) is uniformly equivalent to $G_{a}$ (\ref{Ga}) in the sense that
there exists a constant $C_{m}>0$ such that $C_{m}^{-1}G_{a,m}$ $\leq $ $%
G_{a}$ $\leq $ $C_{m}G_{a,m}.$ As a consequence we have $L^{2}(\Sigma ,$ $%
G_{a,m})$ $=$ $L^{2}(\Sigma ,$ $G_{a}).$
\end{lemma}

\begin{proof}
At a point $p_{0}$ we can simultaneously \textquotedblleft diagonalize" $%
G_{a}$ and $G_{a,m}$ in view of (\ref{M1-1}). Then ($C_{m}^{\prime
})^{-1}G_{a,m}$ $\leq G_{a}$ $\leq $ $C_{m}^{\prime \prime }G_{a,m}$ where $%
C_{m}^{\prime }$ := $\max \{1,\lambda _{m}^{-1}\}$ and $C_{m}^{\prime \prime
}$ $:=$ $\max \{1,\lambda _{m}\}.$ So $C_{m}$ $:=$ $C_{m}^{\prime
}C_{m}^{\prime \prime }$ is a constant required in the lemma.
\end{proof}

The following fact seems to be of independent interest although it is not
strictly needed for our purpose. It serves as a piece of evidence for the
fact that some geometric constructions \ (to be made later) on $\Sigma $ and
on $M$ $=$ $\Sigma /\mathbb{\sigma }$ respectively are mutually "compatible"
in an appropriate context (cf. Proposition \ref{madj} and Corollary \ref%
{Cor3-7}). It is mainly this compatibility that allows us to carry out our
transversal heat kernel method for the proof of the asserted results in this
paper.

\begin{proposition}
\label{p-gue3-1}\textit{\ Let }$p_{0}\in \Sigma .$ \textit{Each }$w$-slice%
\textit{\ in }$\Sigma ,$\textit{\ described by }$\lambda \circ p_{0}$ $=$ $%
\sigma (\lambda )p_{0},$ $\lambda $ $\in $ $C_{\delta _{j}}$ in a local
patch $D_{j},$ \textit{is totally geodesic with respect to }$G_{a}$ or $%
G_{a,m}.$\textit{\ In other words, the Christoffel symbols have the
following vanishing property:}%
\begin{equation}
\Gamma _{AB}^{C}=0\text{ for }A,B\text{ tangent and }C\text{ normal to }w%
\text{-slices.}  \label{M1-2}
\end{equation}
\end{proposition}

\proof
Let $g_{AB}$ denote the component of $G_{a}$ (resp. $G_{a,m})$ with respect
to the directions $A,$ $B$. In local holomorphic coordinates $z$ $=$ $%
(z_{1}, $ $..,$ $z_{n-1})$ and $w,$ $A,B$ can be $\partial /\partial w,$ or $%
\partial /\partial \bar{w}$ and $C$ can be $\partial /\partial z_{j}$ or $%
\partial /\partial \bar{z}_{j}$. By the formula of 
\begin{equation}
\Gamma _{AB}^{C}=\frac{1}{2}g^{CD}(\frac{\partial g_{AD}}{\partial x_{B}}+%
\frac{\partial g_{BD}}{\partial x_{A}}-\frac{\partial g_{AB}}{\partial x_{D}}%
),  \label{M1-3}
\end{equation}

\noindent we can choose $w$ coordinate such that $h(z_{0},\bar{z}_{0})$ $=$ $%
1,$ $dh(z_{0},\bar{z}_{0})$ $=$ $0$ at $p_{0}$ where $(z,w)(p_{0})$ $=$ $%
(z_{0},w_{0})$ ((\ref{M0-3}))$.$ The $w$-slice is described by $z$ $=$ $%
z_{0} $ in a local patch $D_{j}$ ((\ref{1-0}))$.$ For $C$ $=$ $\partial
/\partial z_{j}$ or $\partial /\partial \bar{z}_{j}$ and $D$ $=$ $\partial
/\partial w$ or $\partial /\partial \bar{w}$ one sees $g^{CD}$ $=$ $0$ at $%
z_{0}$ by (\ref{M1-1}), so $D$ in (\ref{M1-3}) can only be left in the $z$%
-direction. Since we take $A,B$ to be $\partial /\partial w$ or $\partial
/\partial \bar{w},$ $\partial g_{AD}/\partial x_{B}$ and $\partial
g_{BD}/\partial x_{A}$ can only involve $dh$ which vanishes at $z$ $=$ $%
z_{0} $ (cf. (\ref{M0-1}), (\ref{M0-2}))$.$ Similarly $\partial
g_{AB}/\partial x_{D}(z_{0})$ can only contain the term $\partial (\varphi
_{1}\circ l)/\partial x_{D}$ $=$ $(\partial h/\partial x_{D})\varphi
_{1}^{\prime }w\bar{w}$ (in (\ref{Ga})) which vanishes at $z$ $=$ $z_{0}$
since $dh(z_{0},\bar{z}_{0})$ $=$ $0.$ Altogether, in view of (\ref{M1-3})
we have shown (\ref{M1-2}).

\endproof%

The following definition of the \textit{formal adjoint} is more or less
standard.

\begin{notation}
\textbf{\label{3-1.5}} Denote by $\vartheta _{\Sigma },$ $\vartheta _{U_{j}}$
the formal adjoint of $\bar{\partial}_{\Sigma }$ $:$ $\Omega ^{0,q}(\Sigma
)\rightarrow \Omega ^{0,q+1}(\Sigma ),$ $\bar{\partial}_{U_{j}}$ $(=$ $\bar{%
\partial}_{z}$ in $z)$ $:$ $\Omega ^{0,q}(U_{j})$ $\rightarrow $ $\Omega
^{0,q+1}(U_{j})$ (see (\ref{1-0}) for the notation $U_{j})$ with respect to $%
G_{a}$, $\pi ^{\ast }g_{M}$ (see (\ref{Ga})) respectively (cf. \cite[p.152]%
{Ko}, \cite[p.62]{ChenS}). Namely $\vartheta _{\Sigma }u$ $\in $ $\Omega
^{0,q}(\Sigma )$ for $u$ $\in $ $\Omega ^{0,q+1}(\Sigma )$ is defined to
satisfy $(\vartheta _{\Sigma }u,$ $v)_{L^{2}}$ $=$ $(u,$ $\bar{\partial}%
_{\Sigma }v)_{L^{2}}$ for any smooth $(0,q)$-form $v$ of compact support,
where the $L^{2}$-inner product is with respect to $G_{a}.$ Similarly $%
\vartheta _{U_{j}}$ is defined with $\Sigma $ (resp. $G_{a})$ replaced by $%
U_{j}$ (resp. $\pi ^{\ast }g_{M}$)$.$%
\end{notation}

For the $m$-space $\Omega _{m}^{0,q}$ the corresponding notion of formal
adjoint is less straightforward in that the conventional use of compact
support test functions $\phi $ is no longer available ($\phi $ always
involves $w^{m}$ along the $\mathbb{R}^{+}$-orbits). One way out is to
insert cut-off functions into test functions, but for later use we find it
most convenient if we simply allow the support to be noncompact. The $L^{2}$%
-inner product $(\cdot ,\cdot )_{L^{2}}$ below is with respect to $G_{a,m}.$
We define an operator $\vartheta _{\Sigma ,m}$ $:$ $\Omega
_{m}^{0,q+1}(\Sigma )$ $\rightarrow $ $\Omega _{m}^{0,q}(\Sigma )$ by $%
(\vartheta _{\Sigma ,m}u,$ $v)_{L^{2}}$ $=$ $(u,$ $\bar{\partial}_{\Sigma
,m}v)_{L^{2}}$ for all $v$ $\in $ $\Omega _{m}^{0,q}(\Sigma ),$ and $%
\vartheta _{D_{j},m}:$ $\Omega _{m}^{0,q+1}(D_{j})$ $\rightarrow $ $\Omega
_{m}^{0,q}(D_{j})$ by $(\vartheta _{D_{j},m}s,$ $t)_{L^{2}}$ $=$ $(s,$ $\bar{%
\partial}_{D_{j},m}t)_{L^{2}}$ for $s$ $=$ $s(z,\bar{z})w^{m}$ $\in $ $%
\Omega _{m}^{0,q+1}(D_{j})$, $t$ $=$ $t(z,\bar{z})w^{m}$ $\in $ $\Omega
_{m}^{0,q}(D_{j})$ with $t(z,\bar{z})$ being of compact support in $U_{j}.$
For their existence we will deduce a (local) formula for $\vartheta
_{D_{j},m}$ in Proposition \ref{madj} and that for $\vartheta _{\Sigma ,m}$
in Definition \ref{SFA} and Proposition \ref{3-11-1}. We can now make the
following definition.

\begin{definition}
\label{d-3-7} We call the above $\vartheta _{\Sigma ,m}$ (resp. $\vartheta
_{D_{j},m}$) the formal adjoint of $\bar{\partial}_{\Sigma ,m}$ (resp. $\bar{%
\partial}_{D_{j},m}$). (In the next section we need to extend their domains
of definition from the smooth elements to the $L^{2}$-elements. See lines
below Notation \ref{N-mL2}.)
\end{definition}

Remark that $\vartheta _{\Sigma ,m}$ $=$ $\pi _{m}\circ \vartheta _{\Sigma }$
on $\Omega _{m}^{0,q}(\Sigma ).$ See Proposition \ref{projm} for the
orthogonal projection $\pi _{m}$ and for its integral representation$.$ A
key point here is that this formal adjoint $\vartheta _{\Sigma ,m}$ turns
out to be a differential operator if one uses the metric $G_{a,m}$ (see
Lemma \ref{Adj}, Remark \ref{3-46.5} and Proposition \ref{p-gue3-2}). See
also Corollary \ref{Cor3-7} below for the difference between the two
Laplacians formed by the two operators $\bar{\partial}_{\Sigma ,m}$, $\bar{%
\partial}_{\Sigma }$ with their respective adjoints (the $\hat{L}^{\prime }$
there is meant $\Sigma $ here).

In the remaining of this section, we will show that modulo certain zeroth
order terms $\vartheta _{\Sigma ,m}$ equals $\vartheta _{\Sigma }|_{\Omega
_{m}^{0,q+1}(\Sigma )}$. See Proposition \ref{p-gue3-2}. During the process,
we find that our metric $G_{a,m}$ satisfies another important property (see
Proposition \ref{madj}), which is essential for an application in
Proposition \ref{BoxDU}.

Note that $\bar{\partial}_{\Sigma ,m}s=(\bar{\partial}_{z}s(z,\bar{z}))w^{m}$
where we express $s=s(z,\bar{z})w^{m}$ $\in $ $\Omega _{m}^{0,q}(\Sigma )$
locally. Recall the line bundle $L_{\Sigma }$ in the beginning of this
section. From (\ref{ddw}) we learn that $e_{w}:=\partial /\partial w$ is a $%
\tilde{\sigma}$-invariant section of $L_{\Sigma }$ over $D_{j}$ $\subset $ $%
\Sigma $ (in fact, as a local section it is only local-$\mathbb{C}^{\ast }$
invariant). Let $L_{\Sigma }^{\ast }$ denote the dual holomorphic line
bundle of $L_{\Sigma }$ and $e_{w}^{\ast }$ the local section of $L_{\Sigma
}^{\ast },$ dual to $e_{w}.$

\begin{notation}
\label{n-3-6-1} Denote by $\Omega _{0}^{0,q}(\Sigma ,(L_{\Sigma }^{\ast
})^{\otimes m})$ the space of $\mathbb{C}^{\ast }$-invariant elements $%
\varpi $ in $\Omega ^{0,q}(\Sigma ,$ $(L_{\Sigma }^{\ast })^{\otimes m}).$
\end{notation}

In a local patch $D_{j},$ write $\varpi $ $=$ $s(e_{w}^{\ast })^{\otimes m}$
where $s$ $\in $ $\Omega ^{0,q}(D_{j}).$ We have the operator $\bar{\partial}%
_{\Sigma ,(L_{\Sigma }^{\ast })^{\otimes m}}$ $:$ $\Omega _{0}^{0,q}(\Sigma
,(L_{\Sigma }^{\ast })^{\otimes m})$ $\rightarrow $ $\Omega
_{0}^{0,q+1}(\Sigma ,(L_{\Sigma }^{\ast })^{\otimes m})$ given by $\bar{%
\partial}_{\Sigma ,(L_{\Sigma }^{\ast })^{\otimes m}}(s(e_{w}^{\ast
})^{\otimes m})$ $=$ $(\bar{\partial}_{z}s(z,\bar{z}))(e_{w}^{\ast
})^{\otimes m}$.

We may identify, for $D_{j}\subset \Sigma \backslash \Sigma _{\text{sing}},$
say, $p_{1}$ $=$ $1$ and thus $\delta _{j}$ $=$ $\pi $ in (\ref{1-0}) (cf.
remarks after Definition \ref{2m}) \ in 
\begin{equation}
\Omega _{0}^{0,q}(D_{j},(L_{\Sigma }^{\ast }|_{D_{j}})^{\otimes m})\simeq
\Omega ^{0,q}(U_{j},(\psi _{j}^{\ast }L_{\Sigma }^{\ast })^{\otimes
m}|_{U_{j}\times \{0\}\times \{1\}})=:\Omega ^{0,q}(U_{j},(\psi _{j}^{\ast
}L_{\Sigma }^{\ast })^{\otimes m})  \label{3.36-25}
\end{equation}%
\noindent where $\psi _{j}^{-1}$ $:$ $D_{j}\subset \Sigma $ $\rightarrow $ $%
U_{j}\times C_{\delta _{j}}$ is a local trivialization (see (\ref{3-0.75})
for the definition of $C_{\delta _{j}}$ and (\ref{2.10-5}) for $\psi _{j})$.
Let $\bar{\partial}_{U_{j},m}$ denote the $\bar{\partial}$ operator acting
on the RHS of (\ref{3.36-25}).

\begin{definition}
\label{d-6.8-5} Let $\Omega _{m,loc}^{0,q}(\Sigma )$ (resp. $\Omega
_{m,loc}^{0,q}(D_{j}))$ denote the space of elements $u$ $\in $ $\Omega
^{0,q}(\Sigma )$ (resp. $\Omega ^{0,q}(D_{j})),$ having the form $w^{m}v(z,%
\bar{z})$ in local holomorphic coordinates $(z,w).$ Note that $\Omega
_{m}^{0,q}(\Sigma )$ $\subset $ $\Omega _{m,loc}^{0,q}(\Sigma ),$ but they
are not equal in general unless the $\mathbb{C}^{\ast }$-action on $\Sigma $
is globally free. For later use we define the space $\tilde{\Omega}%
_{m,loc}^{0,q}(\Sigma )$ consisting of elements $u$ $\in $ $\Omega
^{0,q}(\Sigma ),$ having the form $w^{m}v(z,\bar{z},w,\bar{w})$ in local
holomorphic coordinates $(z,w),$ with bounded $C_{B}^{s}$-norms for each
integer $s$ $\geq $ $0$ (see (\ref{CBs-norm}) for the definition of $%
C_{B}^{s}$-norm). We have $\Omega _{m}^{0,q}(\Sigma )$ $\subset $ $\tilde{%
\Omega}_{m,loc}^{0,q}(\Sigma ).$
\end{definition}

Let $\bar{\partial}_{D_{j},m}$ denote the $\bar{\partial}$ operator acting
on $\Omega _{m,loc}^{0,q}(D_{j}).$ With the notation above, we generalize
Proposition \ref{p-gue2-2} as follows. Compare Remark \ref{9.3}.

\begin{proposition}
\label{dualLm} Recall the line bundle $L_{\Sigma }$ defined in the lines
above (\ref{3-0}), and also Notation \ref{n-3-6-1}. The map $\tilde{\Psi}%
_{q,m}$ : $\Omega _{0}^{0,q}(\Sigma ,(L_{\Sigma }^{\ast })^{\otimes m})$ $%
\rightarrow $ $\Omega _{m}^{0,q}(\Sigma )$ given by%
\begin{equation}
\tilde{\Psi}_{q,m}(s(e_{w}^{\ast })^{\otimes m})=s(z,\bar{z})w^{m}
\label{3.36-5}
\end{equation}%
in any local patch $D_{j}$ (not necessarily in $\Sigma \backslash \Sigma _{%
\text{sing}})$ with holomorphic coordinates $(z,w),$ where $s$ $\in $ $%
\Omega _{0,loc}^{0,q}(D_{j}),$ is globally defined and a vector space
isomorphism. Moreover we have $\bar{\partial}_{\Sigma ,m}\circ \tilde{\Psi}%
_{q,m}=\tilde{\Psi}_{q+1,m}\circ \bar{\partial}_{\Sigma ,(L_{\Sigma }^{\ast
})^{\otimes m}}.$ For $(z,w)$ $\in $ $U_{j}\times C_{\delta _{j}}$ we have $%
\bar{\partial}_{D_{j},m}\circ \Psi _{q,m}=\Psi _{q+1,m}\circ \bar{\partial}%
_{U_{j},m}$ on $\Omega ^{0,q}(U_{j},(\psi _{j}^{\ast }L_{\Sigma }^{\ast
})^{\otimes m}),$ where $\Psi _{q,m}$ $:$ $\Omega ^{0,q}(U_{j},(\psi
_{j}^{\ast }L_{\Sigma }^{\ast })^{\otimes m})$ $\rightarrow $ $\Omega
_{m,loc}^{0,q}(D_{j})$ defined by 
\begin{equation}
\Psi _{q,m}(s(z,\bar{z})(\psi _{j}^{\ast }e_{w}^{\ast })^{\otimes m})=s(z,%
\bar{z})w^{m}  \label{Psi_qm}
\end{equation}

\noindent is a vector space isomorphism.
\end{proposition}

\proof
We focus on $\tilde{\Psi}_{q,m};$ the assertion for $\Psi _{q,m}$ in (\ref%
{Psi_qm}) can be proved similarly (compare Proposition \ref{p-gue2-2}).
Observe that $\tilde{\Psi}_{q,m}$ is a linear isomorphism as long as it is
well defined. Since the transformation law of $e_{w}^{\ast }$ is easily
verified to be the same as that of $w$, one sees that $\tilde{\Psi}_{q,m}$
is well defined (with image in $\Omega ^{0,q}(\Sigma ))$ . To see that the
image of $\tilde{\Psi}_{q,m}$ is actually contained in $\Omega
_{m}^{0,q}(\Sigma ),$ we restrict ourselves to the principal stratum $\Sigma
\backslash \Sigma _{\text{sing}}$ and then extend to $\Sigma $ by
continuity. That is, the image of $\tilde{\Psi}_{q,m}$ lies in $\Omega
_{m}^{0,q}(\Sigma \backslash \Sigma _{\text{sing}})$ (which is the same as $%
\Omega _{m,loc}^{0,q}(\Sigma \backslash \Sigma _{\text{sing}})$ in this
case) using (\ref{3.36-25}) and (\ref{Psi_qm}) so it must be in $\Omega
_{m}^{0,q}(\Sigma )$ since it is already in $\Omega ^{0,q}(\Sigma ).$

\endproof%

We are ready to formulate the first main result (Proposition \ref{madj}) of
this section.

The $\mathbb{C}^{\ast }$-invariant Hermitian metric $<\cdot ,\cdot >$ on $%
L_{\Sigma }$ (see Step 1 at the beginning of this section) induces a $%
\mathbb{C}^{\ast }$-invariant Hermitian metric on ($L_{\Sigma }^{\ast
})^{\otimes m},$ still denoted by the same notation if no confusion will
occur. For $s$ $=$ $s(z,\bar{z})$ $\in $ $\Omega ^{0,q+1}(U_{j}),$ by abuse
of notation, we denote 
\begin{equation}
\vartheta _{z,m}s:=\frac{\vartheta _{U_{j},m}(s(\psi _{j}^{\ast
}(e_{w}^{\ast })|_{U_{j}\times \{0\}\times \{1\}})^{\otimes m})}{(\psi
_{j}^{\ast }(e_{w}^{\ast })|_{U_{j}\times \{0\}\times \{1\}})^{\otimes m}}
\label{dbarzm}
\end{equation}%
\noindent with respect to the metrics $\pi ^{\ast }g_{M}|_{U_{j}}$ (cf. (\ref%
{Ga})) and $<\cdot ,\cdot >$, where $e_{w}^{\ast }$ is dual to $e_{w}$ $=$ $%
\partial /\partial w$ as above and $\psi _{j}$ is as in (\ref{3.36-25})$.$
According to a standard formula (see \cite[(3.142) on p.160]{Ko}) one has $%
\vartheta _{z,m}s=\vartheta _{z}s+$ (\textit{zeroth order terms in }$s$)$,$
where we recall (Definition \ref{d-3-7}) that $\vartheta _{z}$ is the formal
adjoint of $\bar{\partial}_{U_{j}}$ $:$ $\Omega ^{0,q}(U_{j})$ $\rightarrow $
$\Omega ^{0,q+1}(U_{j})$ (with respect to the metric $\pi ^{\ast }g_{M})$ in
coordinates $z$ $=$ $(z_{1},$ $..,$ $z_{n-1}).$ By choosing $w$ coordinate
such that $h(z_{0},\bar{z}_{0})$ $=$ $1,$ $dh(z_{0},\bar{z}_{0})$ $=$ $0$ at
a point $p_{0}$ $=$ $(z_{0},w_{0})$ (cf. (\ref{M0-3}))$,$ the above implies 
\begin{equation}
\vartheta _{z,m}s=\vartheta _{z}s\text{ at }p_{0}.  \label{atp0}
\end{equation}

\noindent The formula (\ref{atp0}) will be applied to (\ref{adj4}) later on.

Remark that $\vartheta _{z,m}$ is not invariantly defined while $\vartheta
_{U_{j},m}$ is (cf. Definition \ref{d-3-7}).

It is worth mentioning that the special structure of our metric $G_{a,m}$
will yield that the two operators $\vartheta _{D_{j},m}\circ \Psi _{q+1,m}$
and $\Psi _{q,m}\circ \vartheta _{U_{j},m}$ are still comparable. More
precisely, we have the following crucial fact. See Proposition \ref{BoxDU}
for an application.

\begin{proposition}
\label{madj} (The first main result of this section) Assume $m\geq 0.$ Under
the notations explained above, we have $\vartheta _{D_{j},m}(s(z,\bar{z}%
)w^{m})$ $=$ $(\vartheta _{z,m}s(z,\bar{z}))w^{m}$ and hence $\vartheta
_{D_{j},m}$ $=$ $\Psi _{q,m}\circ \vartheta _{U_{j},m}\circ \Psi
_{q+1,m}^{-1}.$
\end{proposition}

\proof
Let $t$ $\in $ $\Omega _{m,loc}^{0,q}(D_{j}),$ $s$ $\in $ $\Omega
_{m,loc}^{0,q+1}(D_{j}).$ Write $t=t(z,\bar{z})w^{m},$ $s=s(z,\bar{z})w^{m}$
where $t(z,\bar{z})$ $\in $ $\Omega ^{0,q}(U_{j}),$ $s(z,\bar{z})$ $\in $ $%
\Omega ^{0,q+1}(U_{j}).$ Here $U_{j}$ may be identified with $U_{j}\times
\{0\}\times \{1\}$ $(\subset \Sigma )$ via $\psi _{j}.$ Take $t(z,\bar{z})$
as a test function/form so it is of compact support in $U_{j}$. Write $G$
for $G_{a,m}$ and $H$ for the metric on ($L_{\Sigma }^{\ast })^{\otimes m}$
induced by $||\cdot ||$ on $L_{\Sigma }$ (cf. Step 1 given earlier in this
section). We compute, by using $<(e_{w}^{\ast })^{\otimes m},(e_{w}^{\ast
})^{\otimes m}>_{H}$ $=$ ($h^{-1})^{m}$ (see (\ref{lq}))$,$ $l(q)$ $=$ $hw%
\bar{w},$ (\ref{ddw}) and (\ref{volume}),%
\begin{eqnarray}
&&\int_{D_{j}}<\bar{\partial}_{D_{j},m}t,s>_{G}dv_{\Sigma ,m}=\int_{D_{j}}<%
\bar{\partial}_{z}t(z,\bar{z})w^{m},s(z,\bar{z})w^{m}>_{_{G}}dv_{\Sigma ,m}
\label{adj14} \\
&=&\int_{D_{j}}<\bar{\partial}_{z}t(z,\bar{z})(e_{w}^{\ast })^{\otimes
m},s(z,\bar{z})(e_{w}^{\ast })^{\otimes m}>_{G\otimes
H}h^{m}|w|^{2m}dv_{\Sigma ,m}  \notag
\end{eqnarray}

To proceed further, note first that all the integrands in (\ref{adj14}) is
invariantly defined. To integrate the above over $D_{j},$ by Fubini's
theorem we may first integrate over (part of) every $\mathbb{C}^{\ast }$%
-orbit then over the directions orthogonal to the $\mathbb{C}^{\ast }$%
-orbits. Note that the metric on the orthogonal/horizontal direction is
given by $\pi ^{\ast }g_{M}$ (see (\ref{M1-1})). With the natural projection 
$D_{j}$ $=$ $U_{j}\times C_{\delta _{j}}$ $\rightarrow $ $U_{j},$ $U_{j}$
equipped with the metric $\pi ^{\ast }g_{M}$ can be regarded as a parameter
space for horizontal directions. For the above reasoning, note however that $%
G_{a,m}|_{TU_{j}}$ $\neq $ $\pi ^{\ast }g_{M}|_{TU_{j}}$ ($U_{j}$ $\cong $ $%
U_{j}\times \{0\}\times \{1\}$ $\subset $ $\Sigma )$ and that $\pi ^{\ast
}g_{M}$ is precisely the metric we use on $TU_{j};$ see the line after (\ref%
{dbarzm}) above.

It turns out (see the last equality in (\ref{adj15}) below and remarks after
it) that (\ref{adj14}) equals (where $<\cdot ,\cdot >$ below means $<\cdot
,\cdot >_{\pi ^{\ast }g_{M}\otimes H}):$%
\begin{equation}
\int_{U_{j}}<\bar{\partial}_{U_{j},m}(t(z,\bar{z})(e_{w}^{\ast })^{\otimes
m}),s(z,\bar{z})(e_{w}^{\ast })^{\otimes m}>dv(z)\int_{C_{\delta
_{j}}}l(q)^{m}dv_{f,m}\text{.}  \label{3.46-5}
\end{equation}

\noindent Since the preceding expressions of the integrands are again
invariantly defined, for any given $z_{0}$ in $U_{j}$ we choose $(z,w)$ with 
$h(z_{0},\bar{z}_{0})$ $=$ $1$ and $dh(z_{0},\bar{z}_{0})$ $=$ $0$ (cf. (\ref%
{M0-3}))$,$ so that (see (\ref{3-16.75}))%
\begin{equation}
\int_{C_{\delta _{j}}}l(q)^{m}dv_{f,m}=\int_{\mathbb{R}%
^{+}}|w|^{2m}dv_{m}(|w|)\int_{-\delta _{j}}^{\delta _{j}}dv(\phi )=\frac{%
\delta _{j}}{\pi }.  \label{3-43.5}
\end{equation}

\noindent It is crucial that the integration (\ref{3-43.5}) results in a
constant independent of $z$-coordinates, so that for (\ref{3.46-5}) we can
now apply $\bar{\partial}_{U_{j},m}^{\ast }$ effortlessly:%
\begin{equation}
(\ref{3.46-5})=\int_{U_{j}}<t(z,\bar{z})(e_{w}^{\ast })^{\otimes
m},\vartheta _{U_{j},m}(s(z,\bar{z})(e_{w}^{\ast })^{\otimes m})>dv(z)\frac{%
\delta _{j}}{\pi }.  \label{3.32'}
\end{equation}

Let us continue with (\ref{3.32'}) and bring it back via (\ref{3-43.5}) and (%
\ref{dbarzm}) to the following (for the second equality recalling $l$ $=$ $hw%
\bar{w})$:%
\begin{eqnarray}
&&\text{RHS of }(\ref{3.32'})  \label{adj15} \\
&\overset{}{=}&\int_{U_{j}}<t(z,\bar{z}),\vartheta _{z,m}s(z,\bar{z})>_{\pi
^{\ast }g_{M}}h^{-m}dv(z)\int_{C_{\delta _{j}}}l(q)^{m}dv_{f,m}  \notag \\
&\overset{D_{j}=U_{j}\times C_{\delta _{j}}}{=}&\int_{D_{j}}<t(z,\bar{z}%
)w^{m},(\vartheta _{z,m}s(z,\bar{z}))w^{m}>_{\pi ^{\ast }g_{M}}dv(z)\wedge
dv_{f,m}  \notag \\
&\overset{(\ref{metric})+(\ref{M1-1})}{=}&\int_{D_{j}}<t(z,\bar{z}%
)w^{m},(\vartheta _{z,m}s(z,\bar{z}))w^{m}>_{G}dv_{\Sigma ,m}.  \notag
\end{eqnarray}

\noindent Here ($\vartheta _{z,m}s(z,\bar{z}))w^{m}$ $=$ $\Psi
_{q,m}(\vartheta _{U_{j},m}(s(\psi _{j}^{\ast }(e_{w}^{\ast })|_{U_{j}\times
\{0\}\times \{1\}})^{\otimes m}))$ by (\ref{dbarzm}) and (\ref{3.36-5})
(with $\psi _{j}^{\ast }$ often omitted) is invariantly defined since $\Psi
_{q,m}$ and $\vartheta _{U_{j},m}$ are. So the above $<\cdot \cdot \cdot
>_{\pi ^{\ast }g_{M}}$ $=$ $<\cdot \cdot \cdot >_{G}$ holds as one checks
that they coincide under a choice of special coordinates (at any given
point, cf. (\ref{M0-3}), (\ref{M1-1})).

In summary the LHS of (\ref{adj14}) equals the RHS of (\ref{adj15}): it
follows the first part: $\vartheta _{D_{j},m}s=(\vartheta _{z,m}s(z,\bar{z}%
))w^{m},$ also the second part by (\ref{dbarzm}) and the definition of $\Psi
_{q,m}.$

\endproof%

\begin{definition}
\label{SFA} We define a differential operator $\tilde{\vartheta}_{\Sigma
,m}: $ $\Omega _{m}^{0,q+1}(\Sigma )$ $\rightarrow $ $\Omega
_{m}^{0,q}(\Sigma )$ by ($\tilde{\vartheta}_{\Sigma ,m}u)|_{D_{j}}$ $:=$ $%
\vartheta _{D_{j},m}u|_{D_{j}}$ $=$ $(\vartheta _{z,m}u_{j}(z,\bar{z}))w^{m}$
where $u|_{D_{j}}$ $=$ $u_{j}(z,\bar{z})w^{m}.$ According to Proposition \ref%
{madj} that $\vartheta _{D_{j},m}$ is a differential operator uniquely
determined by $\bar{\partial}_{\Sigma ,m},$ $\tilde{\vartheta}_{\Sigma ,m}$
is well-defined.
\end{definition}

In globalizing Proposition \ref{madj} the cut-off functions inevitably
depend on the $\theta $-variable. Let us be specific about this point below.
Let $(\cdot ,\cdot )_{L^{2}}$ denote the inner product with respect to the
metric $G_{a,m}.$

\begin{proposition}
\label{3-11-1} For $u$ $\in \Omega _{m}^{0,q+1}(\Sigma ),$ $v$ $\in $ $%
\Omega _{m}^{0,q}(\Sigma )$ it holds that $(\tilde{\vartheta}_{\Sigma
,m}u,v)_{L^{2}}$ $=$ $(u,\bar{\partial}_{\Sigma ,m})_{L^{2}}.$ As a
consequence $\tilde{\vartheta}_{\Sigma ,m}$ $=$ $\vartheta _{\Sigma ,m}$
(Definition \ref{d-3-7}).
\end{proposition}

\begin{proof}
Write $u=\sum_{j}\varphi _{j}u=\sum_{j}\varphi _{j}u_{j}(z,\bar{z})w^{m}$ ($%
\varphi _{j}$ ($=$ $\varphi _{j}(z,\bar{z},\theta )$) being cut-off
functions introduced below Notation \ref{n-6.1}) and $v|_{D_{j}}$ $=$ $%
v_{j}(z,\bar{z})w^{m}.$ Via Proposition \ref{madj} we are going to compute
the following. Note that the $\theta $ in $\varphi _{j}$ is treated below as
a parameter (on which $\vartheta _{z,m}$ has no action). 
\begin{eqnarray*}
(\tilde{\vartheta}_{\Sigma ,m}u,v)_{L^{2}} &=&\sum_{j}\int_{D_{j}}<\vartheta
_{z,m}(\varphi _{j}u_{j})w^{m},v_{j}(z,\bar{z})w^{m}>_{G}dv_{\Sigma ,m} \\
&=&\sum_{j}\int_{D_{j}}<\varphi _{j}u_{j}w^{m},(\bar{\partial}_{z}v_{j}(z,%
\bar{z}))w^{m}>_{G}dv_{\Sigma ,m} \\
&=&\sum_{j}\int_{D_{j}}<\varphi _{j}u_{j}w^{m},(\bar{\partial}_{\Sigma
,m}v)|_{D_{j}}>_{G}dv_{\Sigma ,m} \\
&=&\int_{\Sigma }<u,\bar{\partial}_{\Sigma ,m}v>_{G}dv_{\Sigma ,m}=(u,\bar{%
\partial}_{\Sigma ,m}v)_{L^{2}}.
\end{eqnarray*}
\end{proof}

To proceed further\footnote{%
Although Corollary \ref{Cor3-7} below can be regarded as an objective of the
remaining section, it serves as a motivation for the treatment of our Hodge
theory in the coming section rather than an effective tool (cf. the
introductory paragraph of Section 4). Despite this, (\ref{M4}) with an exact
property in the complex two-dimensional case (\ref{2D}), seems a natural
question that a reader may be led to inquire (see also Remarks \ref{3-46.5}
and \ref{3-62.5}); we decide to include the details here. For the proof
below the first main result Proposition \ref{madj} will be needed (see proof
of Proposition \ref{p-gue3-2}). See also Remark \ref{3-62.5} for a
comparison with related results.}, we need one more technical lemma.

\begin{lemma}
\label{Adj} Assume $m\geq 0.$ In local holomorphic coordinates $(z,w)$ we
write $\tilde{\psi}=w^{m}\psi ,$ $\psi =\psi _{\bar{\beta}_{1}...\bar{\beta}%
_{q+1}}d\bar{z}_{\beta _{1}}\wedge ...\wedge d\bar{z}_{\beta _{q+1}}.$ Then
we have%
\begin{equation}
\vartheta _{\Sigma }\tilde{\psi}=(\vartheta _{z,m}\psi )w^{m}+\text{zeroth
order terms in }\psi _{\bar{\beta}_{1}...\bar{\beta}_{q+1}}.  \label{adj=}
\end{equation}
\end{lemma}

\begin{remark}
\label{3-46.5} The validity of the lemma relies on the specific metric $%
G_{a,m}$ on $\Sigma .$ We do not see such compatibility result for general
metric ($\Sigma $ being of one dimension higher than the $z$-space).
\end{remark}

\proof
\textbf{(of Lemma \ref{Adj})} At a point $p_{0}\in \Sigma ,$ we find
coordinates $(z,w)$ such that $z(p_{0})$ $=$ $z_{0}$ $=$ $(z_{1}^{0},$ $..,$ 
$z_{n-1}^{0})$, $w(p_{0})$ $=$ $w_{0}$, $h(z_{0},\bar{z}_{0})$ $=$ $||\frac{%
\partial }{\partial w}|_{(z_{0},w_{0})}||^{2}$ $=$ $1$ and $dh(z_{0},\bar{z}%
_{0})$ $=$ $0$ ((\ref{M0-3}))$.$ By standard formulas for $\bar{\partial}%
^{\ast }$ (cf. \cite[p.97]{MK} or \cite[p.153]{Ko}), we have%
\begin{eqnarray}
(\vartheta _{\Sigma }\tilde{\psi})^{\beta _{1}...\beta _{q}} &=&-\frac{1}{g}%
\partial _{\beta }(g\tilde{\psi}^{\beta \beta _{1}...\beta _{q}})
\label{adj0} \\
&=&-\frac{1}{g}\partial _{w}(g\tilde{\psi}^{w\beta _{1}...\beta _{q}})-\frac{%
1}{g}\partial _{z_{j}}(g\tilde{\psi}^{z_{j}\beta _{1}...\beta _{q}})  \notag
\end{eqnarray}

\noindent ($\beta _{1},...,\beta _{q}$ run for $z_{j},$ $w$ here) where $%
g:=\det (g_{\alpha \bar{\beta}})$ and $g_{\alpha \bar{\beta}}$ are
components of $G_{a}$ in coordinates $(z,w).$

Denote ($g_{\alpha \bar{\beta}})^{-1}$ by $g^{\alpha \bar{\beta}}$ (\cite[$%
g^{\bar{\beta}\alpha }$]{Ko})$,$ i.e., $g_{\alpha \bar{\beta}}g^{\gamma \bar{%
\beta}}$ $=$ $\delta _{\alpha }^{\gamma }.$ Let $G^{\alpha \bar{\beta}}$
denote the cofactor of ($g_{\alpha \bar{\beta}})$ at position ($\alpha
,\beta ),$ such that $g^{\alpha \bar{\beta}}$ $=$ $g^{-1}G^{\alpha \bar{\beta%
}}.$

First, we compute%
\begin{eqnarray}
g\tilde{\psi}^{w\beta _{1}...\beta _{q}} &=&gg^{w\bar{\gamma}_{1}}g^{\beta
_{1}\bar{\gamma}_{2}}...g^{\beta _{q}\bar{\gamma}_{q+1}}\tilde{\psi}_{\bar{%
\gamma}_{1}...\bar{\gamma}_{q+1}}  \label{adj1} \\
&=&G^{w\bar{\gamma}_{1}}g^{\beta _{1}\bar{\gamma}_{2}}...g^{\beta _{q}\bar{%
\gamma}_{q+1}}\psi _{\bar{\gamma}_{1}...\bar{\gamma}_{q+1}}w^{m}  \notag
\end{eqnarray}

\noindent (note that $\bar{\gamma}_{1},...,\bar{\gamma}_{q+1}$ only run for $%
\bar{z}_{1},$ $..,$ $\bar{z}_{n-1}).$ It can be checked that $G^{w\bar{\gamma%
}_{1}}$ and $\partial _{w}G^{w\bar{\gamma}_{1}}$ vanishes at $p_{0}$ using (%
\ref{M0-1}), (\ref{M0-2}) and (\ref{M0-3}). This vanishing and (\ref{adj1})
yield%
\begin{eqnarray}
\partial _{w}(g\tilde{\psi}^{w\beta _{1}...\beta _{q}}) &=&(\partial _{w}G^{w%
\bar{\gamma}_{1}})g^{\beta _{1}\bar{\gamma}_{2}}...g^{\beta _{q}\bar{\gamma}%
_{q+1}}\psi _{\bar{\gamma}_{1}...\bar{\gamma}_{q+1}}w^{m}  \label{adj2} \\
&&+G^{w\bar{\gamma}_{1}}\partial _{w}(g^{\beta _{1}\bar{\gamma}%
_{2}}...g^{\beta _{q}\bar{\gamma}_{q+1}}\psi _{\bar{\gamma}_{1}...\bar{\gamma%
}_{q+1}}w^{m})=0  \notag
\end{eqnarray}

\noindent at $p_{0}.$ Finally, substituting (\ref{adj2}) and $g\tilde{\psi}%
^{z_{j}\beta _{1}...\beta _{q}}$ $=$ $G^{z_{j}\bar{\gamma}_{1}}g^{\beta _{1}%
\bar{\gamma}_{2}}...g^{\beta _{q}\bar{\gamma}_{q+1}}\psi _{\bar{\gamma}%
_{1}...\bar{\gamma}_{q+1}}w^{m}$ into (\ref{adj0}) gives%
\begin{eqnarray}
(\vartheta _{\Sigma }\tilde{\psi})^{\beta _{1}...\beta _{q}} &=&-\frac{1}{g}%
\partial _{z_{j}}(G^{z_{j}\bar{\gamma}_{1}}g^{\beta _{1}\bar{\gamma}%
_{2}}...g^{\beta _{q}\bar{\gamma}_{q+1}})\psi _{\bar{\gamma}_{1}...\bar{%
\gamma}_{q+1}}w^{m}  \label{adj3} \\
&&-\frac{1}{g}G^{z_{j}\bar{\gamma}_{1}}g^{\beta _{1}\bar{\gamma}%
_{2}}...g^{\beta _{q}\bar{\gamma}_{q+1}}\partial _{z_{j}}\psi _{\bar{\gamma}%
_{1}...\bar{\gamma}_{q+1}}w^{m}.  \notag
\end{eqnarray}

On the other hand, denote $\det (g_{M\alpha \bar{\beta}})$ by $g_{M}$ (in
the same notation as the metric itself if no confusion occurs) and the
cofactor of ($g_{M\alpha \bar{\beta}})$ at position ($\alpha ,\beta )$ by $%
H^{\alpha \bar{\beta}}.$ Note that $g_{M}^{z_{j}\bar{\gamma}_{1}}$ = $%
g_{M}^{-1}H^{z_{j}\bar{\gamma}_{1}}.$ At $p_{0}$ we compute%
\begin{eqnarray}
&&(\vartheta _{z,m}\psi )^{\beta _{1}...\beta _{q}}=(\bar{\partial}%
_{z}^{\ast }\psi )^{\beta _{1}...\beta _{q}}\text{ (by (\ref{atp0}))}
\label{adj4} \\
&=&-\frac{1}{g_{M}}\partial _{z_{j}}(g_{M}\psi ^{z_{j}\beta _{1}...\beta
_{q}})\text{ (as in (\ref{adj0}))}  \notag \\
&=&-\frac{1}{g_{M}}\partial _{z_{j}}(H^{z_{j}\bar{\gamma}_{1}}g_{M}^{\beta
_{1}\bar{\gamma}_{2}}...g_{M}^{\beta _{q}\bar{\gamma}_{q+1}}\psi _{\bar{%
\gamma}_{1}...\bar{\gamma}_{q+1}})  \notag \\
&=&-\frac{1}{g_{M}}\partial _{z_{j}}(H^{z_{j}\bar{\gamma}_{1}}g_{M}^{\beta
_{1}\bar{\gamma}_{2}}...g_{M}^{\beta _{q}\bar{\gamma}_{q+1}})\psi _{\bar{%
\gamma}_{1}...\bar{\gamma}_{q+1}}  \notag \\
&&-\frac{1}{g_{M}}H^{z_{j}\bar{\gamma}_{1}}g_{M}^{\beta _{1}\bar{\gamma}%
_{2}}...g_{M}^{\beta _{q}\bar{\gamma}_{q+1}}\partial _{z_{j}}\psi _{\bar{%
\gamma}_{1}...\bar{\gamma}_{q+1}}  \notag
\end{eqnarray}

\noindent ($\beta _{1},...,\beta _{q}$ and $\gamma _{1},$ $...,$ $\gamma
_{q+1}$ only run from $z_{1}$ to $z_{n-1}$ here).

To compare (\ref{adj3}) and (\ref{adj4}) as claimed by our main task (\ref%
{adj=}), observe that at $p_{0},$ $g=g_{M}g_{w\bar{w}},$ $G^{z_{j}\bar{\gamma%
}_{1}}$ $=$ $H^{z_{j}\bar{\gamma}_{1}}g_{w\bar{w}}$ ($g_{z_{j}\bar{w}}$ = $%
g_{w\bar{z}_{j}}$ $=$ $0$ by (\ref{M0-1}), (\ref{M0-2})), and hence%
\begin{equation}
g^{z_{j}\bar{\gamma}_{1}}=\frac{1}{g}G^{z_{j}\bar{\gamma}_{1}}=\frac{1}{g_{M}%
}H^{z_{j}\bar{\gamma}_{1}}=g_{M}^{z_{j}\bar{\gamma}_{1}}.  \label{adj5}
\end{equation}

\noindent We also compute%
\begin{eqnarray}
\frac{1}{g}\partial _{z_{j}}G^{z_{j}\bar{\gamma}_{1}} &=&\frac{1}{g}\partial
_{z_{j}}(H^{z_{j}\bar{\gamma}_{1}}g_{w\bar{w}}+\text{terms involving }%
h_{\alpha }h_{\bar{\beta}})  \label{adj6} \\
&=&\frac{1}{g}(\partial _{z_{j}}H^{z_{j}\bar{\gamma}_{1}})g_{w\bar{w}}+\frac{%
1}{g}H^{z_{j}\bar{\gamma}_{1}}(\partial _{z_{j}}g_{w\bar{w}})=\frac{1}{g_{M}}%
\partial _{z_{j}}H^{z_{j}\bar{\gamma}_{1}}  \notag
\end{eqnarray}

\noindent at $p_{0}$ where $dh$ $=$ $0$ and $(0=)\partial _{z_{j}}g_{w\bar{w}%
}$ = terms involving $\partial _{z_{j}}h.$ Moreover from (\ref{adj5}), (\ref%
{adj6}) we have, at $p_{0}$%
\begin{eqnarray}
&&\partial _{z_{j}}g^{\beta _{1}\bar{\gamma}_{2}}=\partial _{z_{j}}(\frac{1}{%
g}G^{\beta _{1}\bar{\gamma}_{2}})=(-1)g^{-2}(\partial _{z_{j}}g)G^{\beta _{1}%
\bar{\gamma}_{2}}+g^{-1}\partial _{z_{j}}G^{\beta _{1}\bar{\gamma}_{2}}
\label{adj7} \\
&=&(-1)g^{-2}\partial _{z_{j}}(g_{M}g_{w\bar{w}}+\text{terms involving }%
h_{\alpha }h_{\bar{\beta}})G^{\beta _{1}\bar{\gamma}_{2}}+\frac{1}{g_{M}}%
\partial _{z_{j}}H^{\beta _{1}\bar{\gamma}_{2}}  \notag \\
&=&(-1)\frac{1}{g_{M}}(\partial _{z_{j}}g_{M})\frac{1}{g_{M}}H^{\beta _{1}%
\bar{\gamma}_{2}}+\frac{1}{g_{M}}\partial _{z_{j}}H^{\beta _{1}\bar{\gamma}%
_{2}}  \notag \\
&=&\partial _{z_{j}}(\frac{1}{g_{M}}H^{\beta _{1}\bar{\gamma}_{2}})=\partial
_{z_{j}}g_{M}^{\beta _{1}\bar{\gamma}_{2}}  \notag
\end{eqnarray}

\noindent where $\beta _{1}$ runs from $z_{1}$ to $z_{n-1}.$

We are about to compare (\ref{adj3}) and (\ref{adj4}). We use (\ref{adj5}), (%
\ref{adj6}) and (\ref{adj7}) to obtain%
\begin{equation}
(\vartheta _{\Sigma }\tilde{\psi})^{\beta _{1}...\beta _{q}}=(\vartheta
_{z,m}\psi )^{\beta _{1}...\beta _{q}}w^{m}  \label{adj8}
\end{equation}

\noindent at $p_{0}$ for $\beta _{1},...,$ $\beta _{q}$ running from $z_{1}$
to $z_{n-1}.$ By (\ref{adj8}) we lower the indices of (\ref{adj3}):%
\begin{eqnarray}
(\vartheta _{\Sigma }\tilde{\psi})_{\bar{\gamma}_{1}...\bar{\gamma}_{q}}
&=&g_{\beta _{1}\bar{\gamma}_{1}}...g_{\beta _{q}\bar{\gamma}_{q}}(\vartheta
_{\Sigma }\tilde{\psi})^{\beta _{1}...\beta _{q}}  \label{adj9} \\
&=&g_{M\beta _{1}\bar{\gamma}_{1}}...g_{M\beta _{q}\bar{\gamma}%
_{q}}(\vartheta _{z,m}\psi )^{\beta _{1}...\beta _{q}}w^{m}  \notag \\
&=&(\vartheta _{z,m}\psi )_{\bar{\gamma}_{1}...\bar{\gamma}_{q}}w^{m}  \notag
\end{eqnarray}

\noindent at $p_{0}$ for $\gamma _{1},...,$ $\gamma _{q}$ running from $%
z_{1} $ to $z_{n-1}.$ For these indices our claim (\ref{adj=}) is now shown
even without the correction of zeroth order terms.

The possible corrections by zeroth order terms occur when one (and thus the
only one) of $\gamma _{1},...,$ $\gamma _{q}$ equals $w,$ say $\gamma _{1}$ $%
=$ $w$ (with $\gamma _{2},...,$ $\gamma _{q}$ running from $z_{1}$ to $%
z_{n-1}),$%
\begin{equation}
(\vartheta _{\Sigma }\tilde{\psi})_{\bar{w}\bar{\gamma}_{2}...\bar{\gamma}%
_{q}}=g_{\beta _{1}\bar{w}}g_{\beta _{2}\bar{\gamma}_{2}}...g_{\beta _{q}%
\bar{\gamma}_{q}}(\vartheta _{\Sigma }\tilde{\psi})^{\beta _{1}...\beta
_{q}}=g_{w\bar{w}}g_{\beta _{2}\bar{\gamma}_{2}}...g_{\beta _{q}\bar{\gamma}%
_{q}}(\vartheta _{\Sigma }\tilde{\psi})^{w\beta _{2}...\beta _{q}}\text{ }
\label{adj10}
\end{equation}

\noindent at $p_{0}$ by $g_{z_{j}\bar{w}}(p_{0})$ $=$ $0$. We compute (as in
(\ref{adj0})) using $\tilde{\psi}^{ww\beta _{2}...\beta _{q}}$ $=$ $g^{w\bar{%
z}_{1}}g^{w\bar{z}_{2}}\cdot \cdot \cdot \tilde{\psi}_{\bar{z}_{1}\bar{z}%
_{2}\cdot \cdot \cdot }$ $=$ $0$ at $p_{0}$ by $g^{w\bar{z}_{1}}(p_{0})$ $=$ 
$g^{w\bar{z}_{2}}(p_{0})$ $=$ $0$ and its $w$-derivatives $=$ $0$ at $p_{0},$
for the second equality below 
\begin{eqnarray}
(\vartheta _{\Sigma }\tilde{\psi})^{w\beta _{2}...\beta _{q}} &=&-\frac{1}{g}%
\partial _{w}(g\tilde{\psi}^{ww\beta _{2}...\beta _{q}})-\frac{1}{g}\partial
_{z_{j}}(g\tilde{\psi}^{z_{j}w\beta _{2}...\beta _{q}})  \label{adj11} \\
&=&-\frac{1}{g}\partial _{z_{j}}(g\tilde{\psi}^{z_{j}w\beta _{2}...\beta
_{q}})\text{ }  \notag \\
&=&-\frac{1}{g}\partial _{z_{j}}(g^{z_{j}\bar{\gamma}_{1}}G^{w\bar{\gamma}%
_{2}}g^{\beta _{2}\bar{\gamma}_{3}}...g^{\beta _{q}\bar{\gamma}_{q+1}}\tilde{%
\psi}_{\bar{\gamma}_{1}...\bar{\gamma}_{q+1}})  \notag \\
&=&-\frac{1}{g}(\partial _{z_{j}}G^{w\bar{k}_{2}})g^{z_{j}\bar{k}%
_{1}}g^{\beta _{2}\bar{k}_{3}}...g^{\beta _{q}\bar{k}_{q+1}}\psi _{\bar{k}%
_{1}...\bar{k}_{q+1}}w^{m}  \notag
\end{eqnarray}

\noindent at $p_{0}$ by $G^{w\bar{\gamma}_{2}}(p_{0})$ $=$ $0.$ Finally,
substituting (\ref{adj11}) into (\ref{adj10}), we obtain%
\begin{eqnarray}
&&(\vartheta _{\Sigma }\tilde{\psi})_{\bar{w}\bar{\gamma}_{2}...\bar{\gamma}%
_{q}}=-\frac{1}{g}(\partial _{z_{j}}G^{w\bar{k}_{2}})g_{w\bar{w}}g^{z_{j}%
\bar{k}_{1}}\delta _{\gamma _{2}}^{k_{3}}...\delta _{\gamma
_{q}}^{k_{q+1}}\psi _{\bar{k}_{1}...\bar{k}_{q+1}}w^{m}  \label{adj12} \\
&=&-\frac{1}{g}(\partial _{z_{j}}G^{w\bar{k}_{2}})g_{w\bar{w}}g^{z_{j}\bar{k}%
_{1}}\psi _{\bar{k}_{1}\bar{k}_{2}\bar{\gamma}_{2}...\bar{\gamma}_{q}}w^{m} 
\notag
\end{eqnarray}

\noindent at $p_{0}.$ This term (\ref{adj12}) is of zeroth order in $\psi _{%
\bar{k}_{1}\bar{k}_{2}\bar{\gamma}_{2}...\bar{\gamma}_{q}}.$ Our claim (\ref%
{adj=}) in its complete form follows from (\ref{adj9}) and (\ref{adj12}).

\endproof%

\begin{remark}
\textbf{\label{dim2}} In the case where $\dim _{\mathbb{C}}\Sigma $ $=$ $2,$
observe that $(\vartheta _{\Sigma }\tilde{\psi})_{\bar{w}\bar{\gamma}_{2}...%
\bar{\gamma}_{q}}$ $=$ $0$ because the RHS of (\ref{adj12}) vanishes for
dimension reason. In this case$,$ $\vartheta _{\Sigma }\tilde{\psi}%
=\vartheta _{\Sigma ,m}\tilde{\psi}$ exactly (by using (\ref{adj9}) and
Proposition \ref{madj}).
\end{remark}

We are almost ready to arrive at the second main result of this section. Let 
$L_{0,q+1}^{2}(\Sigma ,G_{a})$ denote the space of all square-integrable $%
(0,q+1)$ forms on $\Sigma $ with respect to $G_{a}.$ For any given $m\geq 0$
and large positive $a$ (say, $a>\frac{m}{2}),$ it is not difficult to see $%
\Omega _{m}^{0,q+1}(\Sigma )$ $\subset $ $L_{0,q+1}^{2}(\Sigma ,G_{a})$ (see
Remark \ref{3-r})$.$

\begin{proposition}
\label{p-gue3-2} (The second main result of this section) For $a>\frac{m}{2}%
\geq 0,$ $\vartheta _{\Sigma ,m}$\textit{\ }$:$\textit{\ }$\Omega
_{m}^{0,q+1}(\Sigma )$\textit{\ }$\rightarrow $\textit{\ }$\Omega
_{m}^{0,q}(\Sigma )$\textit{\ is} equal to the restriction $\vartheta
_{\Sigma }|_{\Omega _{m}^{0,q+1}(\Sigma )}$ modulo zeroth order terms, where 
$\vartheta _{\Sigma }$\textit{\ }$:$\textit{\ }$\Omega ^{0,q+1}(\Sigma )$%
\textit{\ }$\rightarrow $\textit{\ }$\Omega ^{0,q}(\Sigma )$. That is, for $%
\tilde{\psi}\in \Omega _{m}^{0,q+1}(\Sigma )$ we have%
\begin{equation}
\vartheta _{\Sigma }\tilde{\psi}=\vartheta _{\Sigma ,m}\tilde{\psi}+\text{%
zeroth order terms in }\tilde{\psi}.  \label{adj=0}
\end{equation}
\end{proposition}

\proof
The formal adjoints $\vartheta _{\Sigma ,m}\tilde{\psi}$ $=$ $\vartheta
_{D_{j},m}\tilde{\psi}|_{D_{j}}$ in $D_{j}$ $\subset $ $\Sigma .$ So (\ref%
{adj=0}) follows from (\ref{adj=})\ of Lemma \ref{Adj} and\ Proposition \ref%
{madj}.

\endproof%

To streamline our ongoing presentation, let us indicate an application of
the above results to $\bar{\partial}$-Laplacians. Back to the case $\Sigma $ 
$=$ $\hat{L}\backslash \{0$-section\}=:$\hat{L}^{\prime },$ in view of
Proposition \ref{p-gue2-2} we can convert $\bar{\partial}_{M,(L^{\ast
})^{\otimes m}}$ $:$ $\Omega ^{0,q}(M,(L^{\ast })^{\otimes m})$ $\rightarrow 
$ $\Omega ^{0,q+1}(M,(L^{\ast })^{\otimes m})$ to $\bar{\partial}_{\hat{L}%
^{\prime },m}$ $:$ $\Omega _{m}^{0,q}(\hat{L}^{\prime })$ $\rightarrow $ $%
\Omega _{m}^{0,q+1}(\hat{L}^{\prime })$ given by $\bar{\partial}_{\hat{L}%
^{\prime },m}(\eta w^{m})$ $=$ ($\bar{\partial}_{M}\eta )w^{m}.$ That is, we
have $\bar{\partial}_{\hat{L}^{\prime },m}\circ \psi _{q,m}=\psi
_{q+1,m}\circ \bar{\partial}_{M,(L^{\ast })^{\otimes m}}$ (see Proposition %
\ref{p-gue2-2} for $\psi _{q,m}).$ Define $\bar{\partial}$-Laplacians $%
\square _{\hat{L}^{\prime }},$ $\square _{\hat{L}^{\prime },m}$ and $\square
_{M,(L^{\ast })^{\otimes m}}$ by 
\begin{eqnarray*}
\square _{\hat{L}^{\prime }}:= &&\vartheta _{\hat{L}^{\prime }}\circ \bar{%
\partial}_{\hat{L}^{\prime }}+\bar{\partial}_{\hat{L}^{\prime }}\circ
\vartheta _{\hat{L}^{\prime }},\text{ }\square _{\hat{L}^{\prime
},m}:=\vartheta _{\hat{L}^{\prime },m}\circ \bar{\partial}_{\hat{L}^{\prime
},m}+\bar{\partial}_{\hat{L}^{\prime },m}\circ \vartheta _{\hat{L}^{\prime
},m}, \\
\square _{M,(L^{\ast })^{\otimes m}}:= &&\vartheta _{M,(L^{\ast })^{\otimes
m}}\circ \bar{\partial}_{M,(L^{\ast })^{\otimes m}}+\bar{\partial}%
_{M,(L^{\ast })^{\otimes m}}\circ \vartheta _{M,(L^{\ast })^{\otimes m}},
\end{eqnarray*}

\noindent respectively. Now Proposition \ref{p-gue2-2}, Proposition \ref%
{p-gue3-2} and Remark \ref{dim2} yield immediately

\begin{corollary}
\label{Cor3-7} With the notation above,%
\begin{eqnarray}
(\square _{\hat{L}^{\prime }}+\text{first order operator)}\circ \psi _{q,m}
&=&\psi _{q,m}\circ \square _{M,(L^{\ast })^{\otimes m}},  \label{M4} \\
(\square _{\hat{L}^{\prime }}+\text{first order operator)\TEXTsymbol{\vert}}%
_{\Omega _{m}^{0,q}(\hat{L}^{\prime })} &=&\square _{\hat{L}^{\prime },m}. 
\notag
\end{eqnarray}%
\noindent If $\dim _{\mathbb{C}}\hat{L}^{\prime }$ $=$ $2,$ then ``first
order operator" of (\ref{M4}) vanishes. That is to say, 
\begin{equation}
\square _{\hat{L}^{\prime }}\circ \psi _{q,m}=\psi _{q,m}\circ \square
_{M,(L^{\ast })^{\otimes m}},\text{ }\square _{\hat{L}^{\prime }}|_{\Omega
_{m}^{0,q}(\hat{L}^{\prime })}=\square _{\hat{L}^{\prime },m}.  \label{2D}
\end{equation}
\end{corollary}

\begin{remark}
\label{3-62.5} This type of relation between the \textquotedblleft upstair
Laplacian $\square _{\hat{L}^{\prime }}"$ and the \textquotedblleft
downstair Laplacian $\square _{M,(L^{\ast })^{\otimes m}}"$ is also seen in
the work \cite[Proposition 5.1]{CHT} in the context of CR manifolds $X$ with 
$S^{1}$-action, where no \textquotedblleft first order corrections" (such as
the one in (\ref{M4})) is needed, due to the use of an $S^{1}$-invariant
metric on $X.$
\end{remark}

\section{\textbf{A Hodge theory for }$\square _{\Sigma ,m}^{(q)}\label{Sec4}$%
}

Let $\Sigma $ be as before. We are going to study a Hodge theory for the
Laplacian $\square _{\Sigma ,m}^{(q)}$ associated to $\bar{\partial}_{\Sigma
,m}$ acting on $\Omega _{m}^{0,q}(\Sigma ).$ For this purpose, certain 
\textit{a priori} estimates such as elliptic estimates are useful. In view
of Corollary \ref{Cor3-7} above for $\Sigma $ $=$ $\hat{L}^{\prime }$, it
appears conceivable that the desired estimates for $\square _{\hat{L}%
^{\prime },m}$ could be available from those for $\square _{\hat{L}^{\prime
}},$ and the latter is known classically (on compact manifolds). Strongly
motivated by this though, in the present section we take an alternative
approach. This approach basically aligns with Proposition \ref{madj}, and
part of the methodology will reappear in subsequent sections.

Our main result of this section is Theorem \ref{t-4-2} with an application
to the index in Corollary \ref{t-4-3}. See another application for the proof
of (\ref{At}) in Section 5. Fix a finite covering $\{D_{j}\}_{j\in J}$ of $%
\Sigma $ as in (\ref{1-0}) and a partition of unity $\varphi _{j}$ $(=$ $%
\varphi _{j}(z,\theta ))$ subordinated to $D_{j}$ as in the item $i)$ after
Notation \ref{n-6.1} with $U_{j}=V_{j}$ there. Write $\omega $ $\in $ $%
\Omega _{m}^{0,q}(\Sigma )$ as $\omega $ $=$ $\sum_{j}\varphi _{j}\omega $
with $\varphi _{j}\omega =w^{m}\varphi _{j}\mu _{j}(z,\bar{z})$ for $%
C^{\infty }$-smooth $(0,q)$-forms $\mu _{j}$ on $U_{j}$ by Definition \ref%
{2m} $ii).$

\begin{notation}
\label{N-mL2} Denote by $L_{0,q,m}^{2}(\Sigma ,G_{a,m})$ the space of $||$ $%
\cdot $ $||_{L^{2}}$-completion of $\Omega _{m}^{0,q}(\Sigma )$ with respect
to the metric $G_{a,m}$ (compare Remark \ref{N-6-3} for similar notations $%
L_{m}^{2,\ast }(\Sigma ,G_{a,m}),$ $L^{2,\ast }(\Sigma ,\pi ^{\ast }\mathcal{%
E}_{M},G_{a,m})$)$.$ (Recall that any element in $\Omega _{m}^{0,q}(\Sigma )$
is square-integrable if $a$ $>$ $\frac{m}{2};$ see Remark \ref{3-r}.)
\end{notation}

It is convenient to define $Dom(\bar{\partial}_{\Sigma ,m})$ (resp. $%
Dom(\vartheta _{\Sigma ,m}))$ to be the space of all $\omega $ $\in $ $%
L_{0,q,m}^{2}(\Sigma ,G_{a,m})$ with $\bar{\partial}_{\Sigma ,m}\omega $ $%
\in $ $L_{0,q+1,m}^{2}(\Sigma ,G_{a,m})$ (resp. $\vartheta _{\Sigma
,m}\omega $ $\in $ $L_{0,q-1,m}^{2}(\Sigma ,G_{a,m}))$ \textit{in the
distribution sense} given as follows: $(\bar{\partial}_{\Sigma ,m}\omega ,$ $%
\varphi )_{L^{2}}$ $=$ $(\omega ,$ $\vartheta _{\Sigma ,m}\varphi )_{L^{2}}$
(resp. $(\vartheta _{\Sigma ,m}\omega ,$ $\varphi )_{L^{2}}$ $=$ $(\omega ,$ 
$\bar{\partial}_{\Sigma ,m}\varphi )_{L^{2}})$ for all $\varphi $ $\in $ $%
\Omega _{m}^{0,q+1}(\Sigma )$ (resp. $\varphi $ $\in $ $\Omega
_{m}^{0,q-1}(\Sigma ))$ (note that $\vartheta _{\Sigma ,m},$ the formal
adjoint of $\bar{\partial}_{\Sigma ,m},$ as in Definition \ref{d-3-7} acting
on smooth elements is a differential operator via Proposition \ref{madj}).
In case $\omega $ is smooth the distributional $\bar{\partial}_{\Sigma
,m}\omega $ (resp. $\vartheta _{\Sigma ,m}\omega )$ coincides with the
ordinary $\bar{\partial}_{\Sigma ,m}\omega $ (resp. $\vartheta _{\Sigma
,m}\omega )$ by Proposition \ref{3-11-1}. Here both $(\omega ,$ $\vartheta
_{\Sigma ,m}\varphi )_{L^{2}}$ and $(\omega ,$ $\bar{\partial}_{\Sigma
,m}\varphi )_{L^{2}}$ are finite in view of Remark \ref{3-r}. The
distribution sense above uses test functions necessarily of noncompact
support; this is one of key features in our study. Write 
\begin{equation}
\square _{\Sigma ,m}^{(q)}:=\vartheta _{\Sigma ,m}\circ \bar{\partial}%
_{\Sigma ,m}+\bar{\partial}_{\Sigma ,m}\circ \vartheta _{\Sigma ,m}\text{ on 
}Dom(\square _{\Sigma ,m}^{(q)})\subset L_{0,q,m}^{2}(\Sigma ,G_{a,m}).
\label{4.0}
\end{equation}%
\noindent Here $Dom(\square _{\Sigma ,m}^{(q)})$ consists of elements $%
\omega \in Dom(\bar{\partial}_{\Sigma ,m})$ $\cap $ $Dom(\vartheta _{\Sigma
,m})$ $\subset $ $L_{0,q,m}^{2}(\Sigma ,G_{a})$ such that $\bar{\partial}%
_{\Sigma ,m}\omega \in Dom(\vartheta _{\Sigma ,m}^{(q+1)}),\vartheta
_{\Sigma ,m}\omega \in Dom(\bar{\partial}_{\Sigma ,m}^{(q-1)})$ (cf. \cite[%
Definition 4.2.2 ]{ChenS} with their Hilbert space adjoint replaced by the
formal adjoint $\vartheta _{\Sigma ,m}$)$.$ An alternative definition of $%
Dom(\square _{\Sigma ,m}^{(q)})$ may be that $u$ $\in $ $Dom(\square
_{\Sigma ,m}^{(q)})$ if $\square _{\Sigma ,m}^{(q)}u$ $\in $ $%
L_{0,q,m}^{2}(\Sigma ,G_{a})$ in the distribution sense as above, i.e. ($%
\square _{\Sigma ,m}^{(q)}u,\varphi )_{L^{2}}$ $=$ $(u,\square _{\Sigma
,m}^{(q)}\varphi )_{L^{2}}$ for all $\varphi $ $\in $ $\Omega
_{m}^{0,q}(\Sigma ).$ In this way, see Remark \ref{4-7-1} for disadvantages.


Let $H_{0,q}^{s}(\Sigma ,G_{a,m})$ denote the usual Sobolev space of order $%
s $ for $(0,q)$-forms on $(\Sigma ,G_{a,m})$ with $\parallel \cdot \parallel
_{s}$ its Sobolev norm$.$ Let 
\begin{equation}
H_{0,q,m}^{s}(\Sigma ,G_{a,m}):=H_{0,q}^{s}(\Sigma ,G_{a,m})\cap
L_{0,q,m}^{2}(\Sigma ,G_{a,m}),  \label{Hs}
\end{equation}

\noindent which is the completion of $\Omega _{m}^{0,q}(\Sigma )$ under $%
\parallel \cdot \parallel _{s}$ (here $\Omega _{m}^{0,q}$ $\subset $ $%
L_{0,q,m}^{2}(\Sigma ,G_{a,m}),$ cf. Remark \ref{3-r}).

However \textbf{the norm (\ref{Hs}) is not going to be adopted here}.
Instead we have the following alternative approach, which we view as a
novelty of this section:

\begin{definition}
\label{d4-0} With the notation above, by (\ref{Psi_qm}) $u\in
L_{0,q,m}^{2}(\Sigma ,G_{a,m})$ may be thought of as the form $u$ $=$ $%
\sum_{j}\varphi _{j}u$ with $\varphi _{j}u$ = $\varphi _{j}v_{j}$ for some $%
(\psi _{j}^{\ast }L_{\Sigma }^{\ast })^{\otimes m}$-valued $v_{j}$ $=$ $%
v_{j}(z,\bar{z})$, or $v_{j}(z)$ for short$,$ $\in $ $L_{0,q}^{2}(U_{j},\psi
_{j}^{\ast }G_{a,m}\otimes h^{-m})$ (see (\ref{lq}) for $L_{\Sigma }$ and $%
h, $ and (\ref{Psi_qm}) for this interpretation) with supp $\varphi
_{j}v_{j} $ $\subset $ $U_{j}$ $\times $ $(-\varepsilon _{j},\varepsilon
_{j})$ ($\times \mathbb{R}^{+}$) where $\varepsilon _{j}$ $=$ $\varepsilon $
for all $j$ (note that $\varphi _{j}$ $=$ $\varphi _{j}(z,\theta )$ where $w$
$=$ $|w|e^{i\theta }$ or $w_{j}$ = $|w_{j}|e^{i\theta _{j}}$ to indicate the
dependence on $j)$. This interpretation explains why the metric in $%
L_{0,q}^{2}$ above and in $ii)$ below involves \textquotedblleft $h^{-m}".$
For most of the time we shall write $u$ $=$ $\sum_{j}\varphi _{j}u$ where $%
\varphi _{j}u$ \ is simply $w^{m}\varphi _{j}v_{j},$ $v_{j}$ $\in $ $%
L_{0,q}^{2}(U_{j},\psi _{j}^{\ast }G_{a,m}\otimes h^{-m})$ (in this way $%
v_{j}$ not $(\psi _{j}^{\ast }L_{\Sigma }^{\ast })^{\otimes m}$-valued). The
two ways are used interchangeably.\ We define 
\begin{equation*}
u\in H_{0,q,m}^{\prime s}(\Sigma ,G_{a,m})
\end{equation*}%
if and only if the following $i),$ $ii)$ and $iii)$ hold%
\begin{eqnarray*}
&&i)\text{ }u\in L_{0,q,m}^{2}(\Sigma ,G_{a,m}), \\
&&ii)\text{ }\varphi _{j}(\cdot ,\theta _{j})v_{j}(\cdot )\in
H_{0,q}^{s}(U_{j}\times \{0\}\times \{1\},\psi _{j}^{\ast }\pi ^{\ast
}g_{M}\otimes h^{-m})
\end{eqnarray*}%
for all $j$ and $\theta _{j}$ $\in $ $(-\varepsilon _{j},\varepsilon _{j}),$
where the metric $\psi _{j}^{\ast }\pi ^{\ast }g_{M}$ is part of $\psi
_{j}^{\ast }G_{a,m}$ (see (\ref{metric})), and 
\begin{equation*}
iii)\text{ }||u||_{s}^{\prime }<\infty \text{ \ \ \ \ \ \ \ \ \ \ \ \ \ \ \
\ \ \ \ \ \ \ \ \ \ \ \ \ \ \ \ \ \ \ \ \ \ \ \ \ \ \ \ \ \ \ \ \ \ }
\end{equation*}%
where the $\parallel \cdot \parallel _{s}^{\prime }$-norm is given by%
\begin{equation}
||u||_{s}^{\prime }:=\left( \sum_{j}\int_{-\varepsilon _{j}=-\varepsilon
}^{\varepsilon _{j}=\varepsilon }||\varphi _{j}(\cdot ,\theta
_{j})v_{j}(\cdot )||_{s,U_{j}}^{2}\frac{d\theta _{j}}{2\pi }\right) ^{1/2}.
\label{Als}
\end{equation}%
\noindent Here $\parallel \cdot \parallel _{s,U_{j}}$ denotes the usual
Sobolev norm for $H_{0,q}^{s}(U_{j}$ $\times $ $\{0\}$ $\times $ $\{1\},$ $%
\psi _{j}^{\ast }\pi ^{\ast }g_{M}\otimes h^{-m})$ using local coordinates $%
(z_{j},w_{j})$ ($=(z,w)$ with \textquotedblleft $j"$ often omitted) on $%
D_{j} $ for taking derivatives. We also write, if $v$ is of support in $U$ $%
\subset $ $U_{j}$%
\begin{equation}
||v||_{s,U\subset U_{j}}:=||v||_{s,U_{j}}.  \label{4-3-1}
\end{equation}%
i.e. the norm $||$ $\cdot $ $||_{s,U\subset U_{j}}$ uses the $U_{j}$%
-coordinates. It is not hard to see that $H_{0,q,m}^{\prime s}(\Sigma
,G_{a,m})$ is a Hilbert space with the inner product $<\cdot ,\cdot
>_{H^{\prime s}}$ such that $<u,u>_{H^{\prime s}}$ $=$ ($||u||_{s}^{\prime
})^{2}.$
\end{definition}

\begin{remark}
\label{R-4-2} From an intrinsic point of view, one may want to use covariant
derivatives for $||$ $\cdot $ $||_{s,U_{j}}$ rather than ordinary
derivatives in local coordinates. But since there are $\theta _{j}$%
-coordinates, the notion \textquotedblleft family of $\theta _{j}$%
-parametrized sections" has no intrinsic meaning; that is, such a family of
sections change as soon as the coordinates change. Even though $||$ $\cdot $ 
$||_{s,U_{j}}$ can be defined using covariant derivatives, there is no
canonical choice of the family of sections in (\ref{Als}). We choose to work
with ordinary derivatives in local coordinates for the $s$-norms.
\end{remark}

\begin{remark}
\label{4.1-5} It is a fact that $H_{0,q,m}^{\prime 0}(\Sigma ,G_{a,m})$ and $%
L_{0,q,m}^{2}(\Sigma ,G_{a,m})$ are the same space with different, yet
equivalent norm if $a>\frac{m}{2}$ ($\geq 0)$ (via Remark \ref{3-r}). The
same can be said on $U_{j}\times \{0\}\times \{1\}$ with different metrics $%
\psi _{j}^{\ast }G_{a,m}\otimes h^{-m}$ and $\psi _{j}^{\ast }\pi ^{\ast
}g_{M}\otimes h^{-m}$ used in the statement of Definition \ref{d4-0}. It is
slightly tedious yet straightforward to check these statements; we omit the
details. See applications in, for instance, the proofs of Lemma \ref{l-4-1}
and Proposition \ref{p-4-1}.
\end{remark}

The notation $\parallel \cdot \parallel _{s}^{\prime }$ with superscript
\textquotedblleft prime" distinguishes itself from the usual Sobolev norm.
Different choice of coverings $\{D_{j}\}_{j}$ and $\varphi _{j}$ gives
equivalent norms$.$ Note the similarity and distinction between this norm
and the $C_{B}^{s}$-norm to be defined in (\ref{CBs-norm}).

The particular setting above is going to be crucial for us to work through a
number of technicalities and obtain a Hodge theory as just mentioned.

We may write $\varphi _{k}u|_{D_{k}\cap D_{j}}$ $=$ $\varphi
_{k}w_{j}^{m}v_{j}$ where $w_{j}$ (=$|w_{j}|e^{i\theta _{j}})$ denotes the $%
w $-coordinate in $D_{j}.$ It is natural to define another $s$-norm by (see
Remarks after (\ref{4.3-5}))%
\begin{equation}
||u||_{s}^{^{\prime \prime }}:=\left( \sum_{k,\text{ }j}\int_{-\varepsilon
}^{\varepsilon }||\varphi _{k}(\cdot ,\theta _{j})v_{j}(\cdot
)||_{s,U_{j}}^{2}\frac{d\theta _{j}}{2\pi }\right) ^{1/2};\text{ }\theta _{j}%
\text{ rewritten as }\theta .  \label{4.3-5}
\end{equation}

\noindent Here, on $D_{k}\cap D_{j}$ we write the same notation $\varphi
_{k} $ expressed in terms of coordinates on $D_{j}.$ As $\varphi _{k}$ $\in $
$C^{\infty }(\Sigma )$ we view $\varphi _{k}v_{j}$ as a function on $D_{j}.$
With this understanding the subscript \textquotedblleft $j"$ of $\theta _{j}$
in (\ref{4.3-5}) may be dropped if no confusion occurs.

\begin{lemma}
\label{normeq} With the notation above, we have the equivalence between $||$ 
$\cdot $ $||_{s}^{\prime }$-norm (\ref{Als}) and $||$ $\cdot $ $%
||_{s}^{^{\prime \prime }}$-norm (\ref{4.3-5}).
\end{lemma}

\proof
One direction is clear: $||u||_{s}^{\prime }\leq ||u||_{s}^{^{\prime \prime
}}$ by restriction to $k=j$ in the sum (\ref{4.3-5})$.$ For the other
direction, suppose in $D_{k}\cap D_{j},$ $w_{k}=w_{j}l_{jk}$ for some
holomorphic function $l_{jk}$ on $D_{j}\cap D_{k}$ in terms of $z_{j}$ or $%
z_{k}$ by (\ref{C0}). Then $w_{k}^{m}v_{k}$ $=$ $w_{j}^{m}v_{j}$ $=$ $u$
restricted on $D_{k}\cap D_{j}$ give that $v_{k}l_{jk}^{m}=v_{j},$ from
which we compute (viewing \textquotedblleft $\theta "$ ($=$ $\theta _{j},$ $%
\theta _{k})$ as a parameter; seeing also (\ref{4-3-1}))%
\begin{eqnarray}
||\varphi _{k}v_{j}||_{s,U_{j}} &=&||\varphi
_{k}v_{k}l_{jk}^{m}||_{s,U_{j}\cap U_{k}\subset U_{j}}\leq C_{1}||\varphi
_{k}v_{k}l_{jk}^{m}||_{s,U_{j}\cap U_{k}\subset U_{k}}  \label{4.3-6} \\
&\leq &C_{1}\max |l_{jk}^{m}|\text{ }||\varphi _{k}v_{k}||_{s,U_{k}}  \notag
\end{eqnarray}

\noindent where $C_{1},$ arising from the coordinate change from $U_{j}$ to $%
U_{k},$ depends on $j,$ $k,$ $m,$ $s$ and not on $\theta $ ($=$ $\theta
_{j}, $ $\theta _{k}).$ It follows from integrating the square of (\ref%
{4.3-6}) with respect to $\theta $ that $||u||_{s}^{^{\prime \prime }}\leq
(\#$ of $j)^{1/2}\cdot C$ $||u||_{s}^{\prime }.$

\endproof%


%
%
%
%
%

\begin{remark}
\label{4-4} Concerning the Sobolev norms $\parallel \cdot \parallel _{s}$ in
(\ref{Hs}) and $\parallel \cdot \parallel _{s}^{\prime }$ in (\ref{Als}),
although it is possible to study the relation between them especially when
\textquotedblleft $a$" in $G_{a,m}$ is sufficiently large (depending on $m$
and $s$), we are not going to pursue this relation in the present paper.
\end{remark}

We are going to compare local $||\cdot ||_{s,D_{j}}^{\prime }$-norm (to be
defined below) with global $||$ $\cdot $ $||_{s}^{^{\prime \prime }}$-norm.
From Propositions \ref{dualLm} and \ref{madj} it follows%
\begin{equation}
\square _{D_{j},m}^{(q)}=\Psi _{q,m}\circ \square _{U_{j},m}^{(q)}\circ \Psi
_{q,m}^{-1}.  \label{lBoxm1}
\end{equation}

\noindent Here we recall the isomorphism $\Psi _{q,m}$ $:$ $\Omega
^{0,q}(U_{j},(\psi _{j}^{\ast }L_{\Sigma }^{\ast })^{\otimes m})$ $%
\rightarrow $ $\Omega _{m,loc}^{0,q}(D_{j})$ in (\ref{Psi_qm}) and\ define $%
\square _{D_{j},m}^{(q)}$ :$=$ $\vartheta _{D_{j},m}\circ \bar{\partial}%
_{D_{j},m}$ $+$ $\bar{\partial}_{D_{j},m}\circ \vartheta _{D_{j},m}$ and $%
\square _{U_{j},m}^{(q)}$ $:=$ $\vartheta _{U_{j},m}\circ \bar{\partial}%
_{U_{j},m}$ $+$ $\bar{\partial}_{U_{j},m}\circ \vartheta _{U_{j},m}$
similarly as in (\ref{4.0}) for $\square _{\Sigma ,m}^{(q)}$.

\begin{remark}
\label{4-7-1} By the definition via test functions one sees that $i)$ if $u$ 
$\in $ $Dom(\square _{\Sigma ,m}^{(q)})$ then $u|_{D_{j}}$ $\in $ $%
Dom(\square _{D_{j},m}^{(q)})$ and $(\square _{\Sigma ,m}^{(q)}u)|_{D_{j}}$ $%
=$ $\square _{D_{j},m}^{(q)}(u|_{D_{j}}).$ The same is true of $\vartheta
_{\Sigma ,m},$ $\vartheta _{D_{j},m}$ in place of $\square _{\Sigma
,m}^{(q)},$ $\square _{D_{j},m}^{(q)}.$ $ii)$ It is easily seen that if $u$ $%
\in $ $Dom(\square _{D_{j},m}^{(q)})$ (resp. $\bar{\partial}_{\Sigma ,m},$ $%
\vartheta _{\Sigma ,m})$ and a cut-off function $\chi $ $\in $ $%
C_{c}^{\infty }(D_{j})$ then $\chi u$ $\in $ $Dom(\square _{D_{j},m}^{(q)})$
(resp. $\bar{\partial}_{D_{j},m},$ $\vartheta _{D_{j},m})$ where the $w$%
-variable in $\chi $ is regarded as \textquotedblleft parameter" without
being acted on by these operators$.$ The localization property $ii)$ is not
easily checked for the alternative choice of definition for $Dom(\square
_{\Sigma ,m}^{(q)})$ (see lines below (\ref{4.0})). Such a localization $ii)$
is crucial for the proof of Proposition \ref{ellipreg} below. Compare Remark %
\ref{4.25-1} for rejustification of this localization.
\end{remark}

Define the local $||\cdot ||_{s,D_{j}}^{\prime }$-norm by 
\begin{equation}
||\omega ||_{s,D_{j}}^{\prime }:=||v_{j}||_{s,U_{j}}\text{ for }\omega
=w_{j}^{m}v_{j}\in \Omega _{m,loc}^{0,q}(D_{j})  \label{4.6-5}
\end{equation}%
\noindent where $||\cdot ||_{s,U_{j}}$ is the usual Sobolev $s$-norm on $%
U_{j}\times $ $\{0\}$ $\times \{1\}$ with respect to the metric $\psi
_{j}^{\ast }(\pi ^{\ast }g_{M})$ together with the fibre metric $h^{-m}$ on (%
$L_{\Sigma }^{\ast })^{\otimes m}$ in (\ref{lq}). Observe that no partition
of unity is used for this local norm.

It is crucial to notice that $\mathit{\Psi }_{q,m}\mathit{,\Psi }_{q,m}^{-1}$
preserve respective Sobolev $\mathit{s}$-spaces. In fact, for $\omega \in
\Omega _{m,loc}^{0,q}(D_{j})$%
\begin{equation}
||\omega ||_{s,D_{j}}^{\prime }=||\Psi _{q,m}^{-1}\omega ||_{s,U_{j}}
\label{eqnorm}
\end{equation}

\noindent by (\ref{Psi_qm}) and (\ref{4.6-5}).

Suppose that $A$ and $B$ are functions on a set $S.$ We use the notation $%
A\lesssim B$ to mean that there is some constant $C>0$ such that $A(u)$ $%
\leq $ $CB(u)$ for all $u$ $\in $ $S.$ For instance, $||\cdot ||_{1}\lesssim
||\cdot ||_{2}$ means that there exists a constant $C>0$ such that $||\cdot
||_{1}\leq C||\cdot ||_{2}.$

\begin{lemma}
\label{4.7-5} (Localization of $||$ $\cdot $ $||_{s}^{\prime }$-norm$)$ With
the notation above, it holds that 
\begin{equation}
||\text{ }\cdot \text{ }||_{s,D_{j}}^{\prime }\lesssim ||\text{ }\cdot \text{
}||_{s}^{\prime \prime }  \label{4.3-6a}
\end{equation}%
on $H_{0,q,m}^{\prime s}(\Sigma ,G_{a,m}).$ That is $||u||_{s,D_{j}}^{\prime
}\leq C||u||_{s}^{\prime \prime }$ for some $C>0$ and every $u\in
H_{0,q,m}^{\prime s}(\Sigma ,G_{a,m}).$
\end{lemma}

\proof
Write $u|_{D_{j}}$ $=$ $w_{j}^{m}v_{j}.$ From the definition it follows that
($\theta $ as a parameter as above) 
\begin{equation}
||u||_{s,D_{j}}^{\prime }=||v_{j}||_{s,U_{j}}\leq \sum_{k;\text{ }%
|\{k\}|<\infty }||\varphi _{k}(\cdot ,\theta )v_{j}(\cdot )||_{s,U_{j}}.
\label{4.3-7}
\end{equation}

\noindent Integrating the square of (\ref{4.3-7}) over $\theta \in
(-\varepsilon ,\varepsilon )$ (cf. comments after (\ref{4.3-5}))$,$ we
obtain (\ref{4.3-6a}) by (\ref{4.3-5}).

\endproof%

Let $(\cdot ,\cdot )_{L^{2}}$ denote the $L^{2}$-inner product for $%
L_{0,q,m}^{2}(\Sigma ,G_{a,m}).$ For the notion of \textit{formally
self-adjointness} to be used below, compare \cite[p.321]{Ko}. We need the
following.

\begin{lemma}
\label{l-4-2} $\square _{\Sigma ,m}^{(q)}$ is formally self-adjoint for $a$ $%
>$ $\frac{m}{2}$ $\geq $ $0$, i.e. for $u,$ $v$ $\in $ $\Omega
_{m}^{0,q}(\Sigma )$ (being of noncompact support along the $\mathbb{R}^{+}$%
-direction) it holds that%
\begin{equation}
(\square _{\Sigma ,m}^{(q)}u,v)_{L^{2}}=(u,\square _{\Sigma
,m}^{(q)}v)_{L^{2}}.  \label{4-1-0}
\end{equation}%
For $u$ $\in $ $\Omega _{m}^{0,q}(\Sigma )$%
\begin{equation}
(\square _{\Sigma ,m}^{(q)}u,u)_{L^{2}}=||\bar{\partial}_{\Sigma
,m}u||_{L^{2}}^{2}+||\vartheta _{\Sigma ,m}u||_{L^{2}}^{2}\geq 0.
\label{pos}
\end{equation}%
Here $\square _{\Sigma ,m}^{(q)}\bullet $\ \ is the differential operator
action. As a consequence $\square _{\Sigma ,m}^{(q)}u$ coincides with the
action in the distribution sense.
\end{lemma}

\begin{proof}
This immediately follows from Proposition \ref{3-11-1} and the definition of 
$\square _{\Sigma ,m}^{(q)}.$

\end{proof}

We adopt Definition \ref{d4-0} above (using $\parallel \cdot \parallel
_{s}^{\prime })$ throughout this section. The usual Hodge theory holds true
when the underlying manifold is compact. In our case $\Sigma $ is
noncompact, so special care should be taken. It turns out that with the
special metric $G_{a,m},$ we can still build up a Hodge-type theory for $%
\square _{\Sigma ,m}^{(q)}.$

\textbf{In the following we always assume that }$m\geq 0$\textbf{\ and
\textquotedblleft }$a$\textbf{" is large, say, }$a>\frac{m}{2}$\textbf{\ (}$%
\geq 0)$\textbf{.}

\begin{proposition}
\label{Rellich} (Rellich-type compactness) With the notation above, the
inclusion map $\iota :H_{0,q,m}^{\prime s+1}(\Sigma ,G_{a,m})\rightarrow
H_{0,q,m}^{\prime s}(\Sigma ,G_{a,m}),$ $s\in \mathbb{N\cup }\{0\},$ is
compact.
\end{proposition}

\proof
Recall that a partition of unity $\varphi _{j}$ (of noncompact support, with 
$j$ in a finite index set $J$, on $\Sigma $ satisfying the item $i)$ after
Notation \ref{n-6.1} of Section \ref{A-THK}) is taken. Suppose that $f_{k}$ $%
\in $ $H_{0,q,m}^{\prime s+1}(\Sigma ,G_{a,m})$ is a bounded sequence where $%
f_{k}$ $=$ $w^{m}(v_{k})_{j}$ in $D_{j}.$ 
We compute 
\begin{equation}
||(v_{k})_{j}||_{s+1,U_{j}}\overset{(\ref{4.6-5})}{=}||f_{k}||_{s+1,D_{j}}^{%
\prime }\overset{(\ref{4.3-6a})}{\lesssim }||f_{k}||_{s+1}^{\prime \prime }%
\overset{Lemma\text{ }\ref{normeq}}{\lesssim }||f_{k}||_{s+1}^{\prime }\leq
C.  \label{4.3-8}
\end{equation}

\noindent Conditions on $\varphi _{j}$ give that $\hat{U}_{j}$ $:=$ $%
\overline{\cup _{\theta _{j}}\text{supp }\varphi _{j}(\cdot ,\theta _{j})}$ $%
\subset $ $U_{j}$ is a compact subset. Let $\chi _{j}$ be a cutoff function
with supp $\chi _{j}$ $\subset $ $U_{j}$ and $\chi _{j}$ $=$ $1$ on $\hat{U}%
_{j}$. So by (\ref{4.3-8}) $||\chi _{j}(v_{k})_{j}||_{s+1,U_{j}}$ is bounded
for all $k$. It follows from the usual Rellich's compactness lemma that
there exists a subsequence $\{$\b{k}\}$\subset \{k\}$ such that $\chi
_{j}(v_{\text{\b{k}}})_{j}$ is a Cauchy sequence in $\parallel \cdot
\parallel _{s,U_{j}}$-norm. Since $||(v_{\text{\b{k}}})_{j}||_{s,\hat{U}%
_{j}}\leq ||\chi _{j}(v_{\text{\b{k}}})_{j}||_{s,U_{j}},$ more precisely $%
||(v_{\text{\b{k}}})_{j}-(v_{\text{\b{k}}^{\prime }})_{j}||_{s,\hat{U}%
_{j}}\leq ||\chi _{j}(v_{\text{\b{k}}})_{j}-\chi _{j}(v_{\text{\b{k}}%
^{\prime }})_{j}||_{s,U_{j}},$ $(v_{\text{\b{k}}})_{j}$ is Cauchy in $%
\parallel \cdot \parallel _{s,\hat{U}_{j}}$-norm. By similar arguments with $%
||\varphi _{j}(\cdot ,\theta _{j})(v_{\text{\b{k}}})_{j}(\cdot
)||_{s,U_{j}}\lesssim ||(v_{\text{\b{k}}})_{j}||_{s,\hat{U}_{j}}$ (using the
definition of $\hat{U}_{j}$ above with the constant independent of $\theta
_{j})$ $\varphi _{j}(v_{\text{\b{k}}})_{j}$ is Cauchy in $||\cdot
||_{s,U_{j}}$-norm uniformly in $\theta _{j}$ and hence by (\ref{Als}) a
subsequence of $f_{\text{\b{k}}}$ is Cauchy in $||\cdot ||_{s}^{\prime }$
due to $j$ $\in $ a finite index set.

\endproof%

\begin{corollary}
\label{Ip} (Interpolation inequality) With the notation above and $s\in 
\mathbb{N\cup }\{0\}$, we have the following interpolation inequality: given 
$\varepsilon >0,$ there exists $C_{\varepsilon }>0$ such that for all $u\in
H_{0,q,m}^{\prime s+2}(\Sigma ,G_{a,m})$ we have $||u||_{s+1}^{\prime }\leq
\varepsilon ||u||_{s+2}^{\prime }+C_{\varepsilon }||u||_{0}^{\prime }.$
\end{corollary}

\proof
By Lemma \ref{Rellich}, both inclusions in $H_{0,q,m}^{\prime s+2}(\Sigma
,G_{a,m})$ $\subset $ $H_{0,q,m}^{\prime s+1}(\Sigma ,G_{a,m})$ $\subset $ $%
H_{0,q,m}^{\prime 0}(\Sigma ,G_{a,m})$ for $s$ $\in $ $\mathbb{N\cup }\{0\}$
are compact. The result follows from a general result in functional analysis 
\cite[Theorem 3.77, p.99]{Aub}.

\endproof%

We have elliptic estimates for $\square _{\Sigma ,m}^{(q)}$ as shown in the
following theorem. Note that $\Omega _{m}^{0,q}(\Sigma )$ $\subset $ $%
L_{0,q,m}^{2}(\Sigma ,G_{a,m})$ for $a>\frac{m}{2}$ (see Remark \ref{3-r})$.$

\begin{theorem}
\label{t-4-1} (Transversally elliptic estimate) Fix $m\geq 0$ and $a$ $>$ $%
\frac{m}{2}.$ For every $s$ $\in $ $\mathbb{N\cup }\{0\},$ there are
positive constants $C_{s}$, $C_{s}^{\prime }$ (depending on $s$ and $m$ with
the $m$-dependence suppressed in notation) such that

\noindent $i)$%
\begin{equation}
||u||_{s+2}^{\prime }\leq C_{s}\left( ||\square _{\Sigma
,m}^{(q)}u||_{s}^{\prime }+||u||_{0}^{\prime }\right)  \label{H1}
\end{equation}%
for all $u\in $ $\Omega _{m}^{0,q}(\Sigma )\subset L_{0,q,m}^{2}(\Sigma
,G_{a,m})$ and

\noindent $ii)$ $||u||_{s+2}^{\prime }\leq C_{s}^{\prime }||\square _{\Sigma
,m}^{(q)}u||_{s}^{\prime }$ for all $u\in $ $\Omega _{m}^{0,q}(\Sigma )$ $%
\cap $ $(Ker\square _{\Sigma ,m}^{(q)})^{\perp }.$ 
\end{theorem}

%
%
%
%
%
%
%
%
%
%
%
%


\begin{remark}
\label{R-4-18} It is easily seen that the similar statements and proofs work
for $u$ $\in $ $H_{0,q,m}^{^{\prime }s+2}(\Sigma ,G_{a,m})$ in place of $u$ $%
\in $ $\Omega _{m}^{0,q}(\Sigma ).$ See the proof of Proposition \ref{p-4-1}
for use of it.
\end{remark}

\proof
\textbf{(of Theorem \ref{t-4-1})} Recall that $||\cdot ||_{s}^{\prime }$
denotes the Sobolev $s$-norm on the whole space $\Sigma $ given by (\ref{Als}%
)$.$ For $u\in $ $\Omega _{m}^{0,q}(\Sigma )$ and writing $\varphi _{j}u$ $=$
$w_{j}^{m}\varphi _{j}v_{j},$ we have

\begin{equation}
||u||_{s+2}^{\prime }\overset{(\ref{Als})}{=}\left(
\sum_{j}\int_{-\varepsilon _{j}}^{\varepsilon _{j}}||\varphi _{j}(\cdot
,\theta _{j})v_{j}(\cdot )||_{s+2,U_{j}}^{2}\frac{d\theta _{j}}{2\pi }%
\right) ^{1/2}  \label{e0}
\end{equation}

\noindent From $\varphi _{j}\Psi _{q,m}^{-1}u|_{D_{j}}=\varphi _{j}(\cdot
,\theta _{j})v_{j}(\cdot )$ it follows that ($\theta _{j}$ viewed as a
parameter)%
\begin{eqnarray}
&&||\varphi _{j}(\cdot ,\theta _{j})v_{j}(\cdot )||_{s+2,U_{j}}=||\varphi
_{j}\Psi _{q,m}^{-1}u|_{D_{j}}||_{s+2,U_{j}}  \label{e1} \\
&\lesssim &||\square _{U_{j},m}^{(q)}(\varphi _{j}\Psi
_{q,m}^{-1}u|_{D_{j}})||_{s,U_{j}}+||\varphi _{j}\Psi
_{q,m}^{-1}u|_{D_{j}}||_{0,U_{j}}  \notag \\
&\lesssim &||\square _{U_{j},m}^{(q)}(\Psi
_{q,m}^{-1}u|_{D_{j}})||_{s,U_{j}}+||\Psi
_{q,m}^{-1}u|_{D_{j}}||_{s+1,U_{j}}+||\Psi _{q,m}^{-1}u|_{D_{j}}||_{0,U_{j}}
\notag
\end{eqnarray}

\noindent where the first inequality follows from classical elliptic
estimates of $\square _{U_{j},m}^{(q)}$ for smooth sections with compact
support in $U_{j}.$ For the RHS of (\ref{e1}), we have%
\begin{eqnarray}
&&||\square _{U_{j},m}^{(q)}(\Psi _{q,m}^{-1}u|_{D_{j}})||_{s,U_{j}}\overset{%
(\ref{eqnorm})+(\ref{lBoxm1})}{\lesssim }||\square
_{D_{j},m}^{(q)}u|_{D_{j}}||_{s,D_{j}}^{\prime }  \label{e2} \\
&=&||\square _{\Sigma ,m}^{(q)}u||_{s,D_{j}}^{\prime }\overset{Lem.\text{ %
\ref{4.7-5}}}{\lesssim }||\square _{\Sigma ,m}^{(q)}u||_{s}^{^{\prime \prime
}}\overset{Lem.\text{ \ref{normeq}}}{\lesssim }||\square _{\Sigma
,m}^{(q)}u||_{s}^{^{\prime }},  \notag
\end{eqnarray}

\noindent and similarly, for $l$ $=$ $s+1$ or $0$%
\begin{eqnarray}
||\Psi _{q,m}^{-1}u|_{D_{j}}||_{l,U_{j}}
&=&||v_{j}||_{l,U_{j}}=||u|_{D_{j}}||_{l,D_{j}}^{\prime }  \label{e2-1} \\
&\overset{Lem.\text{ \ref{4.7-5}}}{\lesssim }&||u||_{l}^{^{\prime \prime }}%
\overset{Lem.\text{ \ref{normeq}}}{\lesssim }||u||_{l}^{^{\prime }}.  \notag
\end{eqnarray}%
\noindent Substituting (\ref{e2}) and (\ref{e2-1}) into (\ref{e1}) and
making use of the interpolation inequality (Corollary \ref{Ip}), we obtain (%
\ref{H1}) via (\ref{e0}). 

For the second statement $ii)$, the argument is similar to the classical
one. Since we are in the transversal setting and using the modified norm $%
||\cdot ||_{s}^{\prime },$ we give details for the sake of clarity. Suppose
otherwise. That is, for each large integer $k$ there exists $u_{k}$ $\in $ $%
\Omega _{m}^{0,q}(\Sigma )$ $\cap $ $(Ker\square _{\Sigma ,m}^{(q)})^{\perp
} $ such that $||u_{k}||_{s+2}^{\prime }$ $=$ $1$ (by dividing $u_{k}$ by $%
||u_{k}||_{s+2}^{\prime }$)$,$%
\begin{equation}
||u_{k}||_{s+2}^{\prime }\geq k||\square _{\Sigma
,m}^{(q)}u_{k}||_{s}^{\prime }.  \label{e3-1}
\end{equation}%
\noindent It follows from (\ref{e3-1}) that $\square _{\Sigma ,m}^{(q)}u_{k}$
$\rightarrow $ $0$ in $||\cdot ||_{s}^{\prime }$ as $k\rightarrow \infty .$
By using the basic weak convergence result with (\ref{H1}) there exists a
subsequence (still denoted by) $u_{k}$ which weakly converges to $u_{\infty
} $ in $||\cdot ||_{s+2}^{\prime }$ and by Lemma \ref{Rellich} (the
Rellich-type compactness), strongly converges in the $||\cdot ||_{s}^{\prime
}$ norm. It follows that $u_{\infty }$ $\in $ $H_{0,q,m}^{\prime s+2}(\Sigma
,G_{a,m})$ $\cap $ $Ker\square _{\Sigma ,m}^{(q)}$ so that $<u_{k},u_{\infty
}>_{L^{2}}$ $=$ $0$ as $u_{k}$ $\perp $ $\ker \square _{\Sigma ,m}^{(q)}$ by
assumption. Taking $k\rightarrow \infty $ in $<u_{k},u_{\infty }>$ $=$ $0$
implies $u_{\infty }=0.$ On the other hand, by (\ref{H1}) we have%
\begin{equation}
1=||u_{k}||_{s+2}^{\prime }\leq C_{s}\left( ||\square _{\Sigma
,m}^{(q)}u_{k}||_{s}^{\prime }+||u_{k}||_{0}^{\prime }\right) .  \label{e3-2}
\end{equation}

\noindent Taking $k\rightarrow \infty $ in (\ref{e3-2}) and observing that $%
||u_{k}||_{0}^{\prime }$ $\rightarrow $ $||u_{\infty }||_{0}^{\prime }$ $=$ $%
0$ and $\square _{\Sigma ,m}^{(q)}u_{k}$ $\rightarrow $ $0$ in $||\cdot
||_{s}^{\prime }$ on the RHS of (\ref{e3-2}) by (\ref{e3-1}), we obtain $%
1\leq 0,$ a contradiction. 

\endproof%

The following lemma will soon be used.

\begin{lemma}
\label{Gar} (Transversal G\aa rding's inequality) Fix $m\geq 0$ and $a$ $>$ $%
\frac{m}{2}.$ There exists a constant $C$ $>$ $0$ such that for all $u\in $ $%
\Omega _{m}^{0,q}(\Sigma )\subset L_{0,q,m}^{2}(\Sigma ,G_{a,m})$ it holds
that%
\begin{equation}
(||u||_{1}^{\prime })^{2}\leq C[(\square _{\Sigma
,m}^{(q)}u,u)_{L^{2}}+(||u||_{0}^{\prime })^{2}].  \label{gar}
\end{equation}%
%
%
%
%
%
%
%
%
%
%
%
%
%
%
%
%
%
%
%
%
%
%
%
%
%
%
%
%
%
%
%
%
%
%
%
%
%
%
%
%
%
%
%
%
%
%
%
%
%
%
%
%
%
%
\noindent The term $(\square _{\Sigma ,m}^{(q)}u,u)_{L^{2}}$ can be replaced
by $||\bar{\partial}_{\Sigma ,m}u||_{L^{2}}^{2}+||\vartheta _{\Sigma
,m}u||_{L^{2}}^{2}.$
\end{lemma}

\begin{proof}
Write $u=\sum_{j}\varphi _{j}u$ $=$ $\sum_{j}w_{j}^{m}\varphi _{j}v_{j}.$ To
bound $||u||_{1}^{\prime }$ we look at $||\varphi _{j}(\cdot ,\theta
)v_{j}(\cdot )||_{1,U_{j}}^{2}.$ By the classical G\aa rding's inequality
e.g. \cite[p.348]{Tr} we have ($\theta $ viewed as a parameter ranging over
a compact interval), 
\begin{eqnarray}
&&||\varphi _{j}(\cdot ,\theta )v_{j}(\cdot )||_{1,U_{j}}^{2}=||\varphi
_{j}\Psi _{q,m}^{-1}u|_{D_{j}}||_{1,U_{j}}^{2}  \label{gar-1} \\
&\lesssim &(\square _{U_{j},m}^{(q)}(\varphi _{j}\Psi
_{q,m}^{-1}u|_{D_{j}}),\varphi _{j}\Psi
_{q,m}^{-1}u|_{D_{j}})_{L^{2}(U_{j})}+||\varphi _{j}\Psi
_{q,m}^{-1}u|_{D_{j}}||_{0,U_{j}}^{2}  \notag \\
&\lesssim &||\bar{\partial}_{U_{j},m}(\varphi _{j}\Psi
_{q,m}^{-1}u|_{D_{j}})||_{L^{2}(U_{j})}^{2}+||\vartheta _{U_{j},m}(\varphi
_{j}\Psi _{q,m}^{-1}u|_{D_{j}})||_{L^{2}(U_{j})}^{2}+||\Psi
_{q,m}^{-1}u|_{D_{j}}||_{0,U_{j}}^{2}  \notag \\
&\lesssim &||\bar{\partial}_{U_{j},m}(\Psi
_{q,m}^{-1}u|_{D_{j}})||_{L^{2}(U_{j})}^{2}+||\vartheta _{U_{j},m}(\Psi
_{q,m}^{-1}u|_{D_{j}})||_{L^{2}(U_{j})}^{2}+||\Psi
_{q,m}^{-1}u|_{D_{j}}||_{0,U_{j}}^{2}  \notag \\
&\lesssim &||\bar{\partial}_{\Sigma ,m}u||_{L^{2}(\Sigma )}^{2}+||\vartheta
_{\Sigma ,m}u||_{L^{2}(\Sigma )}^{2}+(||u||_{0}^{\prime })^{2}  \notag
\end{eqnarray}

\noindent Here we have used $||\Psi
_{q,m}^{-1}u|_{D_{j}}||_{0,U_{j}}^{2}\lesssim (||u||_{0}^{\prime })^{2}$ by
Lemmas \ref{4.7-5} and \ref{normeq}. Summing over $j$ (finitely many) and
integrating over $\theta $ in (\ref{gar-1}) we obtain (\ref{gar}) in view of
(\ref{pos}).
\end{proof}

\begin{proposition}
\label{ellipreg} (Transversally elliptic regularity) Fix $m\geq 0$ and $a$ $%
> $ $\frac{m}{2}$ $\geq $ $0$. Take $u\in Dom(\square _{\Sigma ,m}^{(q)})$ $%
\subset $ $H_{0,q,m}^{\prime 0}(\Sigma ,G_{a,m})$ $=$ $L_{0,q,m}^{2}(\Sigma
,G_{a,m}).$ Suppose $\square _{\Sigma ,m}^{(q)}u\in H_{0,q,m}^{\prime
s}(\Sigma ,G_{a,m})$ for $s$ $\in $ $\mathbb{N\cup }\{0\}$. Then $u\in
H_{0,q,m}^{\prime s+2}(\Sigma ,G_{a,m}).$
\end{proposition}

\proof
%
%
%
%
%
%
%
%
%
%
%
%
%
%
%
%
%
%
%
%
%
%
%
%
%
%
%
%
%
%
%
%
%
%
%
%
%
%
%
%
%
%
%
%
%
%
%
%
%
%
%
%
%
%
%
%
%
%
%
%
%
%
%
%
%
%
%
%
%
%
%
%
%
%
%
%
%
%
%
%


To simply the notation we use $H_{0,q,m}^{\prime s}(\Sigma )$ (resp. $%
H_{0,q,m}^{\prime s}(D_{j}),$ $H_{0,q,m}^{s}(U_{j}))$ to denote $%
H_{0,q,m}^{\prime s}(\Sigma ,G_{a,m})$ (resp. $H_{0,q,m}^{\prime
s}(D_{j},G_{a,m}),$ $H_{0,q,m}^{s}(U_{j},(\psi _{j}^{\ast }L_{\Sigma }^{\ast
})^{\otimes m}).$ First we note that the statement 
\begin{equation}
u\in H_{0,q,m}^{\prime s+1}(\Sigma )\text{ and }\square _{\Sigma
,m}^{(q)}u\in H_{0,q,m}^{\prime s}(\Sigma )\text{ then }u\in
H_{0,q,m}^{\prime s+2}(\Sigma )  \label{e3-2a}
\end{equation}%
\noindent implies \textquotedblleft $\square _{\Sigma ,m}^{(q)}u$ $\in $ $%
H_{0,q,m}^{\prime s}(\Sigma )$ then $u$ $\in $ $H_{0,q,m}^{\prime
s+2}(\Sigma )"$ as claimed in the proposition. This can be easily shown by
induction on $s$: for $s=0$ $\square _{\Sigma ,m}^{(q)}u$ $\in $ $%
L_{0,q,m}^{2}(\Sigma ,G_{a,m})$ and $u$ $\in $ $H_{0,q,m}^{\prime 0}(\Sigma
,G_{a,m})$ $=L_{0,q,m}^{2}(\Sigma ,G_{a,m})$ gives $u\in H_{0,q,m}^{\prime
1}(\Sigma )$ by G\aa rding's inequality for $H^{\prime 1}$-norms (\ref{gar})
with the usual regularization process using the partition of unity as in (%
\ref{gar-1}) and Remark \ref{4-7-1} for localization (see for instance \cite[%
p.381]{GH})$.$ So by (\ref{e3-2a}) we get $u\in H_{0,q,m}^{\prime
0+2}(\Sigma ).$ For $s=1$ we can make use of the $s=0$ case to get $u\in
H_{0,q,m}^{\prime 2}(\Sigma )$ and then apply (\ref{e3-2a}) for $s=1$ to
conclude $u\in H_{0,q,m}^{\prime 1+2}(\Sigma )$. The similar reasoning works
for $s=2,$ $3,$ $\cdot \cdot \cdot .$

In the following argument we will prove (\ref{e3-2a}).$\ $First$\ $the
assumption in (\ref{e3-2a}) and Lemmas \ref{4.7-5}, \ref{normeq} imply $%
u|_{D_{j}}$ $\in H_{0,q,m}^{\prime s+1}(D_{j})$ and $\square
_{D_{j},m}^{(q)}u|_{D_{j}}$ $\in $ $H_{0,q,m}^{\prime s}(D_{j},G_{a,m})$
with Remark \ref{4-7-1}. This yields, since $\Psi _{q,m}^{-1}$ induces
equivalent Sobolev norms (\ref{eqnorm}), $\Psi _{q,m}^{-1}u|_{D_{j}}$ $\in $ 
$H_{0,q,m}^{s+1}(U_{j}),$ $\square _{U_{j},m}^{(q)}(\Psi
_{q,m}^{-1}u|_{D_{j}})\in H_{0,q,m}^{s}(U_{j})$ by (\ref{lBoxm1}). Let $\chi
_{j}$ be the cutoff function used in the proof of Proposition \ref{Rellich}.
Observe that%
\begin{eqnarray}
\square _{U_{j},m}^{(q)}(\chi _{j}\Psi _{q,m}^{-1}u|_{D_{j}}) &=&\chi
_{j}\square _{U_{j},m}^{(q)}(\Psi _{q,m}^{-1}u|_{D_{j}})+\left[ \square
_{U_{j},m}^{(q)},\chi _{j}\right] (\Psi _{q,m}^{-1}u|_{D_{j}}),
\label{e3-2b} \\
\left[ \square _{U_{j},m}^{(q)},\chi _{j}\right] (\Psi
_{q,m}^{-1}u|_{D_{j}}) &\in &H_{0,q,m}^{s}(U_{j})  \notag
\end{eqnarray}%
\noindent since $\left[ \square _{U_{j},m}^{(q)},\chi _{j}\right] $ only
takes one derivative and $\Psi _{q,m}^{-1}u|_{D_{j}}$ $\in $ $%
H_{0,q,m}^{s+1}(U_{j})$ as noted above. From (\ref{e3-2b}) and the
assumption $\square _{U_{j},m}^{(q)}(\Psi _{q,m}^{-1}u|_{D_{j}})\in
H_{0,q,m}^{s}(U_{j}),$ it follows that $\square _{U_{j},m}^{(q)}(\chi
_{j}\Psi _{q,m}^{-1}u|_{D_{j}})$ $\in $ $H_{0,q,m}^{s}(U_{j})$. Then the
usual local elliptic regularity for $\square _{U_{j},m}^{(q)}$ (see for
instance \cite[pp.379-382]{GH}) gives $\chi _{j}\Psi
_{q,m}^{-1}u|_{D_{j}}\in H_{0,q,m}^{s+2}(U_{j}).$ Writing $\Psi
_{q,m}^{-1}u|_{D_{j}}$ $=$ $v_{j}$ we compute, for any $-\varepsilon
_{j}<\theta _{j}<\varepsilon _{j}$%
\begin{equation}
||\varphi _{j}(\cdot ,\theta _{j})v_{j}(\cdot )||_{s+2,U_{j}}^{2}\lesssim
||v_{j}||_{s+2,\hat{U}_{j}}^{2}\lesssim ||\chi
_{j}v_{j}||_{s+2,U_{j}}^{2}<\infty  \label{e3-3}
\end{equation}%
\noindent where the constants are independent of $\theta _{j},$ and $\hat{U}%
_{j}$ is defined after (\ref{4.3-8}). It follows from (\ref{e3-3}) that (\ref%
{Als}) is finite and $ii)$ holds in Definition \ref{d4-0} for $s$ replaced
by $s+2.$ We have shown $u$ $\in $ $H_{0,q,m}^{\prime s+2}(\Sigma ).$ 
\endproof%

We have the following corollary.

\begin{corollary}
\label{4.9-5} $\tbigcap\limits_{s\in \mathbb{N\cup \{}0\mathbb{\}}}$ $%
H_{0,q,m}^{\prime s}(\Sigma ,G_{a,m})=\Omega _{m}^{0,q}(\Sigma ).$
\end{corollary}

\begin{proof}
It suffices to show that the LHS of the formula is contained in the RHS.
Suppose $u$ $\in $ $H_{0,q,m}^{\prime s}(\Sigma ,G_{a,m}).$ By (\ref{e2-1}) $%
\Psi _{q,m}^{-1}u|_{D_{j}}$ is in $H^{s}$ over $U_{j}.$ So if $u$ is in the
LHS of the formula, we obtain that over $U_{j}$ $\Psi _{q,m}^{-1}u|_{D_{j}}$
is in $H^{s}$ for all $s$ $\in $ $\mathbb{N\cup \{}0\mathbb{\}}.$ By the
usual Sobolev lemma $\Psi _{q,m}^{-1}u|_{D_{j}}$ must be smooth in $U_{j}.$
It follows that $u$ $=$ $\sum_{j}\varphi _{j}\Psi _{q,m}^{-1}u|_{D_{j}}$ is
smooth and belongs to $\Omega _{m}^{0,q}(\Sigma ).$
\end{proof}

\begin{lemma}
\label{l-4-1} For $a$ $>$ $\frac{m}{2}$ $\geq $ $0,$ we have $Dom(\square
_{\Sigma ,m}^{(q)})$ $=$ $H_{0,q,m}^{\prime 2}(\Sigma ,G_{a,m}).$
\end{lemma}

\proof
For the inclusion put $v$ $=$ $\square _{\Sigma ,m}^{(q)}u$ $\in $ $%
L_{0,q,m}^{2}(\Sigma ,G_{a,m})$ $(=$ $H_{0,q,m}^{\prime 0}(\Sigma ,G_{a,m})$
by Remark \ref{4.1-5})$.$ By Proposition \ref{ellipreg} for $s$ $=$ $0$, we
have $u$ $\in $ $H_{0,q,m}^{\prime 2}(\Sigma ,G_{a,m}).$ The reverse
inclusion can be checked via (\ref{lBoxm1}) and Definition \ref{d4-0}.

\endproof%

\begin{lemma}
\label{4-18-1} Let $\mu $ be an eigenvalue of $\square _{\Sigma ,m}^{(q)}.$
Then $i)$ The eigenspace $\mathcal{E}_{m,\mu }^{q}(\Sigma ):=\{\omega \in
Dom(\square _{\Sigma ,m}^{(q)}):\square _{\Sigma ,m}^{(q)}\omega =\mu \omega
\}$ is finite-dimensional with $\mathcal{E}_{m,\mu }^{q}(\Sigma )\subset
\Omega _{m}^{0,q}(\Sigma ).$ $ii)$ In particular $Ker\square _{\Sigma
,m}^{(q)}$ $=$ $\{v$ $\in $ $\Omega _{m}^{0,q}(\Sigma )|$ $\square _{\Sigma
,m}^{(q)}v$ $=$ $0\}$ and is finite-dimensional.
\end{lemma}

\begin{proof}
The finite-dimensionality of each eigenspace follows by a similar reasoning
as in the classical case by the elliptic estimate (Theorem \ref{t-4-1}) and
the Rellich-compactness (Proposition \ref{Rellich}). The smoothness of
eigenfunctions is from Proposition \ref{ellipreg} and Corollary \ref{4.9-5}.
\end{proof}

We are now in a position to carry out a Hodge theory (in a transversal
sense) by strategically following the classical approach (nontransversal
one) using the above tools formulated in terms of the (modified) Sobolev $%
||\cdot ||_{s}^{\prime }$-norm. However we avoid using the Lax-Milgram
theorem in the proof of Lemma \ref{SolBox}; see Remark \ref{4-19-1}. To
define Green's operator we start by proving the following:

\begin{lemma}
\label{SolBox} Suppose $a$ $>$ $\frac{m}{2}$ $\geq $ $0.$ Denote by ($%
Ker\square _{\Sigma ,m}^{(q)})^{\perp }$ the orthogonal complement of $%
Ker\square _{\Sigma ,m}^{(q)}$ in $L_{0,q,m}^{2}(\Sigma ,G_{a,m}).$ $i)$ For 
$f\in $ ($Ker\square _{\Sigma ,m}^{(q)})^{\perp }$ there exists a unique
solution $u$ $\in $ $H_{0,q,m}^{\prime 2}(\Sigma ,G_{a,m})$ $\cap $ ($%
Ker\square _{\Sigma ,m}^{(q)})^{\perp }$ satisfying $\square _{\Sigma
,m}^{(q)}u$ $=$ $f$. $ii)$ If $f$ $\in $ $\Omega _{m}^{0,q}(\Sigma )$ $\cap $
($Ker\square _{\Sigma ,m}^{(q)})^{\perp }$ then $u$ $\in $ $\Omega
_{m}^{0,q}(\Sigma )$ $\cap $ ($Ker\square _{\Sigma ,m}^{(q)})^{\perp }.$
\end{lemma}

\begin{proof}
$Ker\square _{\Sigma ,m}^{(q)}$ is a closed subspace of $L_{0,q,m}^{2}(%
\Sigma ,G_{a,m})$ by Lemma \ref{4-18-1} $ii)$. On the other hand $\func{Im}%
\square _{\Sigma ,m}^{(q)}$ (= $\square _{\Sigma ,m}^{(q)}(Dom(\square
_{\Sigma ,m}^{(q)})))$ is perpendicular to $Ker\square _{\Sigma ,m}^{(q)}$
using the definition of $Dom(\square _{\Sigma ,m}^{(q)})$ and Lemma \ref%
{4-18-1} $ii)$. Moreover we claim that%
\begin{equation}
\func{Im}\square _{\Sigma ,m}^{(q)}\text{ is a closed subspace of }%
L_{0,q,m}^{2}(\Sigma ,G_{a,m})(i.e.\overline{\func{Im}\square _{\Sigma
,m}^{(q)}}=\func{Im}\square _{\Sigma ,m}^{(q)}).  \label{CIm}
\end{equation}%
\noindent Assume $\square _{\Sigma ,m}^{(q)}u_{j}$ $=$ $f_{j}$ $\rightarrow $
$f$ in $L^{2}$ with $u_{j}$ $\in $ ($Ker\square _{\Sigma ,m}^{(q)})^{\perp }$%
. We then have a Cauchy sequence $\{u_{j}\}$ in $||\cdot ||_{2}^{\prime }$
by Theorem \ref{t-4-1} $ii)$, Remark \ref{R-4-18} and Lemma \ref{l-4-1}. So $%
u_{j}$ $\rightarrow $ $u_{\infty }$ in the $||\cdot ||_{2}^{\prime }$-norm,
and in turn $\square _{\Sigma ,m}^{(q)}u_{j}$ $\rightarrow $ $\square
_{\Sigma ,m}^{(q)}u_{\infty }$ in the $||\cdot ||_{0}^{\prime }$-norm. Since 
$H^{\prime 0}$ and $L^{2}$ are essentially the same by Remark \ref{4.1-5},
we obtain $\square _{\Sigma ,m}^{(q)}u_{\infty }$ $=$ $f$, proving $f$ $\in $
$\func{Im}\square _{\Sigma ,m}^{(q)}$ as claimed in (\ref{CIm}).


We are going to show the following (orthogonal) decomposition:%
\begin{equation}
L_{0,q,m}^{2}(\Sigma ,G_{a,m})=Ker\square _{\Sigma ,m}^{(q)}\oplus \func{Im}%
\square _{\Sigma ,m}^{(q)}.  \label{Decom}
\end{equation}

\noindent Suppose not. Then there exists $f$ $\in $ $L_{0,q,m}^{2}(\Sigma
,G_{a,m})$ such that $f$ is perpendicular to $Ker\square _{\Sigma ,m}^{(q)}$
and $\func{Im}\square _{\Sigma ,m}^{(q)}.$ From $f$ $\in $ ($\func{Im}%
\square _{\Sigma ,m}^{(q)})^{\perp }$ one sees that $\square _{\Sigma
,m}^{(q)}f$ $=$ $0$ in the distribution sense since $\Omega
_{m}^{0,q}(\Sigma )$ $\subset $ $Dom(\square _{\Sigma ,m}^{(q)})$. Passing
to localization $\square _{D_{j},m}^{(q)}f|_{D_{j}}$ $=$ $0$ in the
distribution sense. By the standard regularity result using (\ref{lBoxm1}) $%
f|_{D_{j}}$ is smooth (cf. \cite[the lemma in p.379]{GH}). So $\square
_{\Sigma ,m}^{(q)}f$ $=$ $0$ strongly, giving $f$ $\in $ $Ker\square
_{\Sigma ,m}^{(q)}.$ From that $f$ is perpendicular to $Ker\square _{\Sigma
,m}^{(q)}$ by assumption, it follows $f=0.$ We have shown (\ref{Decom}). The
assertion $i)$ follows easily from (\ref{Decom}) and Lemma \ref{l-4-1}. The
assertion $ii)$ follows from Proposition \ref{ellipreg} and Corollary \ref%
{4.9-5}.
\end{proof}

\begin{remark}
\label{4-19-1} In \cite[pp.94-95]{GH} the solution $u$ in $i)$ of the above
lemma is essentially obtained by the Lax-Milgram theorem (see \cite[p.205
Lemma 23.1]{Tr}) with an intermediate operator $T.$
\end{remark}

By Lemma \ref{SolBox} we can now define a linear operator $G_{m}^{(q)}:$ $%
L_{0,q,m}^{2}(\Sigma ,G_{a,m})$ $\rightarrow $ $Dom(\square _{\Sigma
,m}^{(q)})$ $=$ $H_{0,q,m}^{\prime 2}(\Sigma ,G_{a,m}))$ such that 
\begin{eqnarray}
&&G_{m}^{(q)}(f)=u\text{ (= (}\square _{\Sigma ,m}^{(q)})^{-1}f\text{ \ \
for }f\in \func{Im}\square _{\Sigma ,m}^{(q)}\text{ (}\subset \text{ }%
L_{0,q,m}^{2}(\Sigma ,G_{a,m}))\text{ and}  \label{Gr} \\
&&\ \ \ \ \ \ \ \ \ \ =0\text{ \ \ for }f\in Ker\square _{\Sigma ,m}^{(q)} 
\notag
\end{eqnarray}%
\noindent We have the following results about $G_{m}^{(q)}$ and $%
Spec\,\square _{\Sigma ,m}^{(q)}$ $\subset $ [0,$\infty ),$ the spectrum of $%
\square _{\Sigma ,m}^{(q)}.$

\begin{proposition}
\label{p-4-1} $i)$ $G_{m}^{(q)}$ is a compact, self-adjoint, bounded linear
operator on $L_{0,q,m}^{2}(\Sigma ,G_{a,m})$. $ii)$ $Spec\,\square _{\Sigma
,m}^{(q)}$ $\subset \lbrack 0,\infty )$ consists only of discrete
eigenvalues. $iii)$ We have the orthogonal decomposition $%
L_{0,q,m}^{2}(\Sigma ,G_{a,m})=\hat{\oplus}_{\mu }\mathcal{E}_{m,\mu
}^{q}(\Sigma ),$ with $\mathcal{E}_{m,\mu }^{q}(\Sigma )$ given in Lemma \ref%
{4-18-1}.
\end{proposition}

\begin{proof}
The boundedness of $i)$ follows from Theorem \ref{t-4-1} $ii)$ and the
density of $\Omega _{m}^{0,q}(\Sigma )$ in $L_{0,q,m}^{2}(\Sigma ,G_{a,m})$
(Notation \ref{N-mL2})$.$ We are going to show that $G_{m}^{(q)}$ is
self-adjoint on the space of smooth elements. For $f,g$ $\in $ $\Omega
_{m}^{0,q}(\Sigma )$ write $f$ $=$ $Hf$ $+$ $\square _{\Sigma ,m}^{(q)}u_{f}$
($H$ being the $L^{2}$-projection onto $Ker\square _{\Sigma ,m}^{(q)}$)
where $u_{f}$ $=$ $G_{m}^{(q)}(f-Hf).$ Similarly $g$ $=$ $Hg$ $+$ $\square
_{\Sigma ,m}^{(q)}u_{g},$ $u_{g}$ $=$ $G_{m}^{(q)}(g-Hg).$ Using $u_{f},$ $%
u_{g}$ $\in $ ($Ker\square _{\Sigma ,m}^{(q)})^{\perp }$ we have%
\begin{eqnarray}
(G_{m}^{(q)}(f),g)_{L^{2}} &=&(u_{f},Hg+\square _{\Sigma
,m}^{(q)}u_{g})_{L^{2}}=(u_{f},\square _{\Sigma ,m}^{(q)}u_{g})_{L^{2}},
\label{GSA} \\
(f,G_{m}^{(q)}(g))_{L^{2}} &=&(Hf+\square _{\Sigma
,m}^{(q)}u_{f},u_{g})_{L^{2}}=(\square _{\Sigma
,m}^{(q)}u_{f},u_{g})_{L^{2}}.  \notag
\end{eqnarray}

\noindent By Lemma \ref{4-18-1} $ii)$ and Lemma \ref{SolBox} $ii)$ we learn
that $u_{f},$ $u_{g}$ $\in $ $\Omega _{m}^{0,q}(\Sigma ).$ It follows from (%
\ref{4-1-0}) ($\square _{\Sigma ,m}^{(q)}$ being formally self-adjoint) that
the right-hand sides in (\ref{GSA}) coincide, giving $%
(G_{m}^{(q)}(f),g)_{L^{2}}$ $=$ $(f,G_{m}^{(q)}(g))_{L^{2}}.$ As the space
of smooth elements is dense in $L^{2}$ and $G_{m}^{(q)}$ is bounded linear, $%
G_{m}^{(q)}$ is self-adjoint on $L_{0,q,m}^{2}(\Sigma ,G_{a,m}).$ Combining
Theorem \ref{t-4-1} $ii)$ and Proposition \ref{Rellich} yields the
compactness of $G_{m}^{(q)}.$ To prove $ii)$ we apply a general theorem \cite%
[p.10]{Lang} on a compact, self-adjoint, bounded linear operator on a
Hilbert space to conclude that $SpecG_{m}^{(q)}$ hence $Spec\,\square
_{\Sigma ,m}^{(q)}$ consists only of discrete eigenvalues. The assertion $%
iii)$ is now obvious.
\end{proof}

Define the $m$-th Fourier-Dolbeault cohomology group or $m$-th $\mathbb{C}%
^{\ast }$ $\bar{\partial}_{\Sigma ,m}$-cohomology group as follows:%
\begin{equation}
H_{m}^{q}(\Sigma ,\mathcal{O}):=\frac{\text{Ker }\bar{\partial}_{\Sigma
,m}:\Omega _{m}^{0,q}(\Sigma )\rightarrow \Omega _{m}^{0,q+1}(\Sigma )}{%
\func{Im}\bar{\partial}_{\Sigma ,m}:\Omega _{m}^{0,q-1}(\Sigma )\rightarrow
\Omega _{m}^{0,q}(\Sigma )}.  \label{H3}
\end{equation}

Denote $\bar{\partial}_{\Sigma ,m}|_{\Omega _{m}^{0,q}(\Sigma )}$ by $\bar{%
\partial}_{\Sigma ,m}^{(q)}.$ We call the complex ($\Omega _{m}^{0,\cdot },$ 
$\bar{\partial}_{\Sigma ,m}^{(\cdot )})$ the $\bar{\partial}_{\Sigma ,m}$%
-complex and define its index by%
\begin{equation*}
index(\bar{\partial}_{\Sigma ,m}\text{-complex})\text{:=}%
\sum_{q=0}^{n}(-1)^{q}\dim H_{m}^{q}(\Sigma ,\mathcal{O})
\end{equation*}

\noindent provided that each $H_{m}^{q}(\Sigma ,\mathcal{O})$ is
finite-dimensional.

We have the following Hodge theorem for $\square _{\Sigma ,m}^{(q)}$ on the
noncompact $\Sigma .$

\begin{theorem}
\label{t-4-2} For each $q\in \{0,1,2,...,n\},$ $m\geq 0$ and $a$ $>$ $\frac{m%
}{2},$ we have%
\begin{eqnarray*}
\square _{\Sigma ,m}^{(q)}G_{m}^{(q)}+P_{m,0}^{(q)} &=&I\text{ on }%
L_{0,q,m}^{2}(\Sigma ,G_{a,m}), \\
G_{m}^{(q)}\square _{\Sigma ,m}^{(q)}+P_{m,0}^{(q)} &=&I\text{ on }%
Dom(\square _{\Sigma ,m}^{(q)})\text{ }(=H_{0,q,m}^{\prime 2}(\Sigma
,G_{a,m})\text{ by Lemma \ref{l-4-1}})
\end{eqnarray*}

\noindent where $G_{m}^{(q)}$ as defined in (\ref{Gr}) is called the Green's
operator, and $P_{m,0}^{(q)}$ $:$ $L_{0,q,m}^{2}(\Sigma ,G_{a,m})$ $%
\rightarrow $ $Ker\square _{\Sigma ,m}^{(q)}$ is the orthogonal projection
(denoted by $H$ previously). Moreover, we have $Ker\square _{\Sigma
,m}^{(q)}=\mathcal{E}_{m,0}^{q}(\Sigma )\cong H_{m}^{q}(\Sigma ,\mathcal{O}%
). $ As a consequence $\dim H_{m}^{q}(\Sigma ,\mathcal{O})<\infty $ by
Proposition \ref{p-4-1}.
\end{theorem}

Note that for the case of $m$ $<$ $0$ we refer to Remark \ref{r1-1}.

Denote $\Omega _{m}^{0,+}(\Sigma )$ $=$ $\oplus _{even\text{ }q}\Omega
_{m}^{0,q}(\Sigma )$ and $\Omega _{m}^{0,-}(\Sigma )$ $=$ $\oplus _{odd\text{
}q}\Omega _{m}^{0,q}(\Sigma );$ similar notations are adopted for $%
L_{m}^{2,+}(\Sigma ,G_{a,m})$ and $L_{m}^{2,-}(\Sigma ,G_{a,m})$ out of $%
L_{0,q,m}^{2}(\Sigma ,G_{a,m}).$

Let%
\begin{equation}
D_{m}^{+}:=\bar{\partial}_{\Sigma ,m}+\vartheta _{\Sigma ,m}:\Omega
_{m}^{0,+}(\Sigma )\rightarrow \Omega _{m}^{0,-}(\Sigma )  \label{Dm}
\end{equation}%
\noindent with extension $D_{m}^{+}$ $:Dom(D_{m}^{+})(\subset
L_{m}^{2,+}(\Sigma ,G_{a,m}))\rightarrow L_{m}^{2,-}(\Sigma ,G_{a,m})$ (by
acting in the sense of distribution). Define the formal adjoint $\mathfrak{D}%
_{m}^{+}$ of $D_{m}^{+}$ in the way similar to $\vartheta _{\Sigma ,m}.$ By
similar arguments as in the classical theory we have

\begin{lemma}
\label{l-4-3} With the notation above we have%
\begin{equation*}
KerD_{m}^{+}=\oplus _{even\text{ }q}Ker\square _{\Sigma ,m}^{(q)}\subset
\Omega _{m}^{0,+}(\Sigma );\text{ }Ker\mathfrak{D}_{m}^{+}=\oplus _{odd\text{
}q}Ker\square _{\Sigma ,m}^{(q)}\subset \Omega _{m}^{0,-}(\Sigma ).
\end{equation*}
\end{lemma}

We have now that both $KerD_{m}^{+}$ and $Ker\mathfrak{D}_{m}^{+}$ are
finite-dimensional (Proposition \ref{p-4-1} and Lemma \ref{l-4-3}). The
index of $D_{m}^{+},$ denoted as $index(D_{m}^{+}),$ is defined by

\begin{equation*}
index(D_{m}^{+}):=\dim KerD_{m}^{+}-\dim Ker\mathfrak{D}_{m}^{+}.
\end{equation*}

\noindent As usual \textit{Coker}$D_{m}^{+}$ $=$ $Ker\mathfrak{D}_{m}^{+}.$
With Theorem \ref{t-4-2} and Lemma \ref{l-4-3} we have:

\begin{corollary}
\label{t-4-3} $index(\bar{\partial}_{\Sigma ,m}$-complex$)$ = $%
index(D_{m}^{+})$ $=$ $\sum_{q:even}\dim Ker\square _{\Sigma
,m}^{(q)}-\sum_{q:odd}\dim Ker\square _{\Sigma ,m}^{(q)}.$
\end{corollary}

\begin{remark}
\label{HSA} It is possible to study the \textit{Hilbert space} adjoint $%
\square _{\Sigma ,m}^{(q)\ast }$ \cite[pp.63-64]{ChenS} including its domain 
$Dom(\square _{\Sigma ,m}^{(q)\ast })$ $\subset $ $L_{0,q,m}^{2}(\Sigma
,G_{a,m})$ (resp. $\bar{\partial}_{\Sigma ,m}^{\ast }$ and $Dom(\bar{\partial%
}_{\Sigma ,m}^{\ast })).$ One may show that $\square _{\Sigma ,m}^{(q)}$ is
(Hilbert space) self-adjoint, densely defined on a Hilbert space. In an
abstract Hilbert space setting there are some basic material, for example, 
\cite[Theorem C.2.1]{MM}, \cite[Theorem 13.30, p.348]{Ru} and \cite[Lemma
8.4.1]{Da} on the spectral analysis of a general self-adjoint operator,
which might provide an alternative approach to Theorem \ref{t-4-2}. We leave
the detail to the interested reader.
\end{remark}

\begin{remark}
\label{4.25-1} In connection with Remark \ref{4-7-1} for localization,
suppose $\square _{\Sigma ,m}^{(q)}u$ $=$ $f$ in the distribution sense (see
the 5th line below (\ref{4.0})) where $u,$ $f$ $\in $ $L_{0,q,m}^{2}(\Sigma
,G_{a,m}).$ Then one sees via definition that $f$ $\perp $ $Ker\square
_{\Sigma ,m}^{(q)}$ so that $f$ $=$ $\square _{\Sigma ,m}^{(q)}v$ for some $%
v $ $\in $ $H_{0,q,m}^{\prime 2}(\Sigma ,G_{a,m})$ $\cap $ ($Ker\square
_{\Sigma ,m}^{(q)})^{\perp }$ using (\ref{Decom}). It follows that $\square
_{\Sigma ,m}^{(q)}(v-u)$ $=$ $0$ in the distribution sense, which implies $%
\square _{\Sigma ,m}^{(q)}(v-u)$ $=$ $0$ strongly (see lines below (\ref%
{Decom})), giving $u$ $\in $ $H_{0,q,m}^{\prime 2}(\Sigma ,G_{a,m})$ by
Lemma \ref{l-4-1}. So the localization $\chi u$ $\in $ $Dom(\square
_{D_{j},m}^{(q)})$ remains true. Remark that a localization result of
similar nature is claimed in \cite[p.380]{GH}; however, the detail is given
only for their first-order operator $P.$
\end{remark}

\section{\textbf{Transversally spin}$^{c}$ \textbf{Dirac operators\label%
{Sec5}}}

To compute $\sum_{q=0}^{n}(-1)^{q}h_{m}^{q}(\Sigma ,\mathcal{O})$ we are
reduced to computing $index(D_{m}^{+})$ by Corollary \ref{t-4-3}. To do it
effectively we want to modify $D_{m}^{+}$ so that the associated modified
Laplacian has a manageable heat kernel. This modification becomes
indispensable for us in dealing with the non-K\"{a}hler case. It will follow
that $index(D_{m}^{+})$ equals the index of a modified operator, to be
denoted by $\tilde{D}_{m}^{c+}$. In fact the new operator $\tilde{D}%
_{m}^{c+} $ will be taken to be an $m$-th $spin^{c}$ Dirac operator in the 
\textit{transversal} sense closely related to the one described in \cite{MM}
(yet in a different context).

The construction of $\tilde{D}_{m}^{c+}$ is first done locally; this local
part is standard as in the classical sense. However, due to our transversal
setting here some extra work will be needed to patch up those local
constructions and form a global operator $\tilde{D}_{m}^{c+}$ on $\Sigma .$
Then it turns out by computation that the chosen metrical structure in
Section \ref{S-metric} makes it possible to compare the operator $\tilde{D}%
_{m}^{c+}$ constructed here with a natural $spin^{c}$ Dirac operator $%
D_{M_{0},m}^{c+}$ at least on the principal stratum $M_{0}$ of the orbifold $%
M$ $=$ $\Sigma /\sigma ;$ see Proposition \ref{BoxDU}, also Remarks \ref%
{8.1-5}, \ref{PLMB}, \ref{9.3} for issues of descent to the entire $M$. This
part of computation is perhaps less geometrically illuminating than the
preceding local-to-global construction.

Let us now start by choosing a local orthonormal frame $\{e_{2j-1}$, $%
e_{2j}\}_{1\leq j\leq n}$ with respect to the metric $G_{a,m}$ (see (\ref%
{metric}) in Section 3) such that%
\begin{equation}
Z_{j}=\frac{1}{\sqrt{2}}(e_{2j-1}-ie_{2j});\text{ }Z_{\bar{j}}=\frac{1}{%
\sqrt{2}}(e_{2j-1}+ie_{2j})\text{ \ }1\leq j\leq n  \label{5-0}
\end{equation}%
\noindent form a local unitary frame of $T^{1,0}(\Sigma )$ and $%
T^{0,1}(\Sigma )$ respectively. As is well known, one has the
\textquotedblleft Clifford multiplication" (or action) $c(e_{k})$ on $%
\Lambda (T^{\ast 0,1}\Sigma )$ $:=$ $\oplus _{q=0}^{n}T^{\ast 0,q}\Sigma $
and \textquotedblleft Clifford connection" $\nabla _{e_{k}}^{Cl}$ acting on $%
\Omega ^{0,\ast }(\Sigma )$ (see \cite[Chapter 1]{MM}). These are given only
in the transversal parts (see (\ref{5-0})) and not in the standard Clifford
setting. The $spin^{c}$ Dirac operator $D^{c}$ on $\Sigma $ is defined by $%
D^{c}=1/\sqrt{2}\sum_{k=1}^{2n}c(e_{k})\nabla _{e_{k}}^{Cl}:\Omega ^{0,\ast
}(\Sigma )\rightarrow \Omega ^{0,\ast }(\Sigma )$ and is formally
self-adjoint. Denote by $D^{c\pm }$ the restriction $D^{c}|_{\Omega ^{0,\pm
}(\Sigma )}.$ We have 
\begin{equation}
D^{c\pm }=\bar{\partial}_{\Sigma }+\vartheta _{\Sigma }+A^{c\pm }:\Omega
^{0,\pm }(\Sigma )\rightarrow \Omega ^{0,\mp }(\Sigma )  \label{5.0}
\end{equation}%
\noindent where $A^{c\pm }:\Omega ^{0,\pm }(\Sigma )\rightarrow \Omega
^{0,\mp }(\Sigma )$ is a self-adjoint zeroth order operator and $A^{c\pm }$ $%
=$ $\frac{1}{4}^{c}(T_{as})$ (\cite[(1.4.17)]{MM})$.$

The elements of the form (\ref{C10}) are not going to be preserved under the
action of $^{c}(T_{as})$ (cf. \cite[(1.2.48) for $T_{as}$]{MM}). We would
like to replace $A^{c\pm }$ by another zeroth order operator which can
preserve $\Omega _{m}^{0,\ast }(\Sigma ).$ This is done as follows.

Let us first treat the globally free case $\Sigma $ $=$ $\hat{L}\backslash
\{ $0-section$\}$ $=:$ $\hat{L}^{\prime }$ (see Example \ref{2.0} $ii)$ for $%
\hat{L}),$ and consider the standard $spin^{c}$ Dirac operator on $M$ with ($%
L^{\ast })^{\otimes m}$-value: $D_{M,(L^{\ast })^{\otimes m}}^{c}=\bar{%
\partial}_{M,(L^{\ast })^{\otimes m}}+\vartheta _{M,(L^{\ast })^{\otimes
m}}+A_{M,m}^{c}$ where $A_{M,m}^{c}$ maps $\Omega ^{0,\pm }(M,$($L^{\ast
})^{\otimes m})$ into $\Omega ^{0,\mp }(M,$($L^{\ast })^{\otimes m})$ and is
self-adjoint (on $\Omega ^{0,\ast }$ $=$ $\Omega ^{0,+}\oplus \Omega ^{0,-})$%
. Here we adopt the metric $g_{M}$ for $M$ (cf. (\ref{Ga})); $\vartheta
_{M}, $ $\vartheta _{M,m}$ are as in Notation \ref{3-1.5} and Definition \ref%
{d-3-7}. Proposition \ref{p-gue2-2} yields a corresponding map on $\hat{L}%
^{\prime }$ 
\begin{equation}
\tilde{A}_{m}^{c}:=\psi _{\mp ,m}\circ A_{M,m}^{c}\circ \psi _{\pm
,m}^{-1}:\Omega _{m}^{0,\pm }(\hat{L}^{\prime })\rightarrow \Omega
_{m}^{0,\mp }(\hat{L}^{\prime }).  \label{5-2.5}
\end{equation}%
\noindent By abuse of notation, write%
\begin{equation*}
C_{M,m}(h_{I_{q}}(z,\bar{z})d\bar{z}^{I_{q}}):=\frac{A_{M,m}^{c}(h_{I_{q}}(z,%
\bar{z})d\bar{z}^{I_{q}}\otimes (e^{\ast })^{\otimes m})}{(e^{\ast
})^{\otimes m}}.
\end{equation*}%
\noindent Note that this depends on the choice of the local section $e^{\ast
}$ of $L^{\ast }$. In $(z,w),$ $\tilde{A}_{m}^{c}$ acts on $\Omega
_{m}^{0,q}(\hat{L}^{\prime })$ by 
\begin{equation}
\tilde{A}_{m}^{c}(w^{m}h_{I_{q}}(z,\bar{z})d\bar{z}%
^{I_{q}})=w^{m}C_{M,m}(h_{I_{q}}(z,\bar{z})d\bar{z}^{I_{q}}).  \label{Amc}
\end{equation}

Note that (\ref{Amc}) here is invariantly defined by (\ref{5-2.5}). The key
observation that makes our construction of global transversal operators well
defined, is based on the following:

\begin{lemma}
\label{5-1.25} For general $\Sigma $ as before$,$ the above (\ref{Amc})
works unchangeably if $M$ is replaced by $U_{j}$ and $z,w$ are distinguished
local holomorphic coordinates in (\ref{1-0}). Thus if now define $\tilde{A}%
_{m}^{c}$ $:$ $\Omega _{m}^{0,\pm }(\Sigma )$ $\rightarrow $ $\Omega
_{m}^{0,\mp }(\Sigma )$ by using (\ref{Amc})$,$ $\tilde{A}_{m}^{c}$ is
independent of the choice of local holomorphic coordinates and is
self-adjoint on $\Omega _{m}^{0,\ast }(\Sigma )$ $(=$ $\Omega
_{m}^{0,+}\oplus \Omega _{m}^{0,-})$.
\end{lemma}

\proof
We note that the local transformation law (\ref{C0}) of Proposition \ref%
{p-gue2-1}, together with (\ref{C11}), is similar to that of the case $%
\Sigma =\hat{L}^{\prime }$ in local expressions. Hence the first assertion
of the lemma \textquotedblleft almost" follows from the remark after (\ref%
{Amc}). The point is that it is not quite automatic that the image of $%
\tilde{A}_{m}^{c}$ so defined is contained in the $m$-space under
consideration (the image lies in $\Omega ^{0,\mp }(\Sigma )$ nevertheless),
due to the local freeness of $\sigma $ (see case $ii)$ after (\ref{6-1g}) in
Section 6). But one can bypass this issue by first considering it on the
principal stratum $\Sigma \backslash \Sigma _{\text{sing}}$ (similar to (\ref%
{5-2.5})\ for $\hat{L}^{\prime })$ then the assertion holds across $\Sigma _{%
\text{sing}}$ by argument of continuity (since $\tilde{A}_{m}^{c}$ is global
on $\Sigma $ as just mentioned). Compare the proof of Proposition \ref%
{dualLm}. For the self-adjointness of $\tilde{A}_{m}^{c},$ the treatment is
similar to and simpler than Proposition \ref{madj} because $\tilde{A}%
_{m}^{c} $ is of zeroth order.

\endproof%



\begin{definition}
\label{5-1.5} (transversally $spin^{c}$ Dirac operator) Define $\tilde{D}%
_{m}^{c\pm }$ (as a "transversally" spin$^{c}$ Dirac operator) by, with $%
\tilde{A}_{m}^{c}$ in Lemma \ref{5-1.25}%
\begin{equation}
\tilde{D}_{m}^{c\pm }:=\tilde{D}_{m}^{c}=\bar{\partial}_{\Sigma
,m}+\vartheta _{\Sigma ,m}+\tilde{A}_{m}^{c}:\Omega _{m}^{0,\pm }(\Sigma
)\rightarrow \Omega _{m}^{0,\mp }(\Sigma ).  \label{5.1}
\end{equation}%
$\tilde{A}_{m}^{c}$ is not directly linked to $A^{c}$ of (\ref{5.0}); the
\textquotedblleft tilde" in $\tilde{D}_{m}^{c}$ is used to match that of $%
\tilde{A}_{m}^{c}.$
\end{definition}





Let $\tilde{\vartheta}_{m}^{c}$ denote the formal adjoint of $\tilde{D}%
_{m}^{c}$ (cf. Notation \ref{3-1.5})$.$ The self-adjointness $\tilde{%
\vartheta}_{m}^{c}$ $=$ $\tilde{D}_{m}^{c}$ follows from Lemma \ref{5-1.25} (%
$\tilde{A}_{m}^{c}$ is self-adjoint)$.$ Let $D_{U_{j},m}^{c\pm }$ $:=$ $\bar{%
\partial}_{U_{j},m}+\vartheta _{U_{j},m}+A_{U_{j},m}^{c}$ : $\Omega
^{0,+}(U_{j},$ ($\psi _{j}^{\ast }L_{\Sigma }^{\ast })^{\otimes m})$ $%
\rightarrow $ $\Omega ^{0,-}(U_{j},$ ($\psi _{j}^{\ast }L_{\Sigma }^{\ast
})^{\otimes m})$ be the spin$^{c}$ Dirac operator on $U_{j}$ with bundle ($%
\psi _{j}^{\ast }L_{\Sigma }^{\ast })^{\otimes m}|_{U_{j}\times \{0\}\times
\{1\}}$ (see (\ref{2.10-5}) for $\psi _{j}).$ Here the metric on $U_{j}$ is $%
\pi ^{\ast }g_{M}|_{T(U_{j}\times \{0\}\times \{1\})}$ rather than $%
G_{a,m}|_{T(U_{j}\times \{0\}\times \{1\})},$ and the metric on $L_{\Sigma }$
is $<\cdot ,\cdot >$ (lines above (\ref{lq0})). We define the following $%
spin^{c}$ Laplacians of Kodaira type by%
\begin{eqnarray}
i)\text{ }\tilde{\square}_{m}^{c} &:&=\tilde{\vartheta}_{m}^{c}\tilde{D}%
_{m}^{c}=(\tilde{D}_{m}^{c})^{2}:\Omega _{m}^{0,\ast }(\Sigma )\rightarrow
\Omega _{m}^{0,\ast }(\Sigma ),  \label{DiracLap} \\
ii)\text{ }\tilde{\square}_{m}^{c\pm } &:&=\tilde{D}_{m}^{c\mp }\tilde{D}%
_{m}^{c\pm }:\Omega _{m}^{0,\pm }(\Sigma )\rightarrow \Omega _{m}^{0,\pm
}(\Sigma ),  \notag \\
iii)\text{ }\square _{U_{j},m}^{c\pm } &:&=D_{U_{j},m}^{c\mp
}D_{U_{j},m}^{c\pm }:\Omega ^{0,\pm }(U_{j},(\psi _{j}^{\ast }L_{\Sigma
}^{\ast })^{\otimes m})\rightarrow \Omega ^{0,\pm }(U_{j},(\psi _{j}^{\ast
}L_{\Sigma }^{\ast })^{\otimes m}).  \notag
\end{eqnarray}

\noindent Remark that the notation $\square _{m}$ (or $\square _{\Sigma ,m}$%
) is reserved for $\bar{\partial}_{m}\vartheta _{m}$ $+$ $\vartheta _{m}\bar{%
\partial}_{m}$\ (\ref{4.0}).

Our first main result Proposition \ref{madj} in Section 3 yields the
following

\begin{proposition}
\label{BoxDU} Restricted to $D_{j}$ $\subset $ $\Sigma ,$ we have $\tilde{D}%
_{m}^{c\pm }=\Psi _{\mp ,m}\circ D_{U_{j},m}^{c\pm }\circ \Psi _{\pm
,m}^{-1} $ (see (\ref{Psi_qm}) for the definition of $\Psi _{\mp ,m})$ and $%
\tilde{\square}_{m}^{c\pm }=\Psi _{\mp ,m}\circ \square _{U_{j},m}^{c\pm
}\circ \Psi _{\pm ,m}^{-1}$.
\end{proposition}

\proof
Restricted to $D_{j}$ $\subset $ $\Sigma ,$ the assertions readily follow
from Propositions \ref{dualLm}, \ref{madj} and Lemma \ref{5-1.25}.

\endproof%

The above result is crucial for us to construct an approximation of
transversal heat kernel (cf. (\ref{6-1})).

In the remaining of this section, we shall focus on the geometry of $\tilde{%
\square}_{m}^{c}$ (spectral aspects) and culminate in a McKean-Singer type
formula (cf. Proposition \ref{p-5-3} and Theorem \ref{t-5-1}).

We now extend $\tilde{\square}_{m}^{c\pm }:Dom$ $\tilde{\square}_{m}^{c\pm
}(\subset L_{m}^{2,\pm }(\Sigma ,G_{a,m}))\rightarrow L_{m}^{2,\pm }(\Sigma
,G_{a,m})$ (by acting in the sense of distribution). Here $L_{m}^{2,\pm
}(\Sigma ,G_{a,m})$ is as given in the line above (\ref{Dm}).

\begin{lemma}
\label{5-3.5} $\tilde{\square}_{m}^{c}$ satisfies the transversally elliptic
estimate as in the statements of Theorem \ref{t-4-1} ($\tilde{\square}%
_{m}^{c}$ in place of $\square _{\Sigma ,m}^{(q)}$ there).
\end{lemma}

\proof
The transversal ellipticity of $\tilde{\square}_{m}^{c}$ follows from the
ellipticity of $\square _{U_{j},m}^{c\pm }$ via arguments similar to those
in the proof of Theorem \ref{t-4-1}.



\endproof%

Denote by $H_{m}^{\prime s,+}(\Sigma ,G_{a,m})$ (resp. $H_{m}^{\prime
s,-}(\Sigma ,G_{a,m})$) the even (resp. odd) part of Sobolev spaces $%
H_{0,\ast ,m}^{\prime s}(\Sigma ,G_{a,m})$. In the same vein as Lemmas \ref%
{l-4-1} and \ref{l-4-2}, we have $Dom$ $\tilde{\square}_{m}^{c\pm
}=H_{m}^{\prime 2,\pm }(\Sigma ,G_{a,m})$ and $\tilde{\square}_{m}^{c\pm }$
are positive, formally self-adjoint.

\begin{proposition}
\label{l-5-2} For $\tilde{\square}_{m}^{c\pm },$ the corresponding
statements in Proposition \ref{p-4-1} and Lemma \ref{4-18-1} hold true.
Moreover, with obvious modifications of notation $Spec\,\tilde{\square}%
_{m}^{c+}\cap (0,\infty )=Spec\,\tilde{\square}_{m}^{c-}\cap (0,\infty ),$
and $\dim \mathcal{\tilde{E}}_{m,\mu }^{+}(\Sigma )=\dim \mathcal{\tilde{E}}%
_{m,\mu }^{-}(\Sigma )$ for each $0\neq \mu $ $\in Spec\,\tilde{\square}%
_{m}^{c+}.$
\end{proposition}

\begin{remark}
\label{N-6-3} As an analogue of holomorphic tangents in the CR case via
Example \ref{2.0}, we make the following definition. Let $\mathcal{E}_{M}$
denote the (orbifold) bundle of all $(0,q)$-forms on $M.$ Let $\pi ^{\ast }%
\mathcal{E}_{M}$ be the pullback bundle over $\Sigma ,$ where $\pi :\Sigma $ 
$\rightarrow $ $M$ $:=\Sigma /\sigma $ is the natural projection. Since the $%
\mathbb{C}^{\ast }$-action $\sigma $ is locally free, one sees that $\pi
^{\ast }\mathcal{E}_{M}$ embeds naturally as a subbundle of the bundle $%
\Lambda ^{0,\ast }(\Sigma )$ of all $(0,q)$-forms on $\Sigma $. Consider the 
$L^{2}$-completion of smooth sections of $\pi ^{\ast }\mathcal{E}_{M}$ with
compact support over $\Sigma $ with respect to the metric $G_{a,m},$ $a$ $>$ 
$\frac{m}{2}$ $\geq $ $0,$ denoted by $L^{2,\ast }(\Sigma ,\pi ^{\ast }%
\mathcal{E}_{M},G_{a,m})$ or $L^{2,\ast }(\Sigma ,G_{a,m})$ for short$.$
Consider $L_{m}^{2,\ast }(\Sigma ,G_{a,m})$ to be the direct sum of $%
L_{0,q,m}^{2}(\Sigma ,G_{a,m})$ for all $q$ (see Notation \ref{N-mL2} for
the definition of $L_{0,q,m}^{2}(\Sigma ,G_{a,m})).$ By the definition of $%
\Omega _{m}^{0,\ast }(\Sigma )$ (cf. Definition \ref{2m} $ii))$ we have that 
$L_{m}^{2,\ast }(\Sigma ,G_{a,m})$ $\subset $ $L^{2,\ast }(\Sigma ,\pi
^{\ast }\mathcal{E}_{M},G_{a,m})$ (see Lemma \ref{R-5-5})$.$ As an
alternative choice of metric on $\pi ^{\ast }\mathcal{E}_{M}$ it is natural
to use the $\mathbb{C}^{\ast }$-invariant Hermitian metric $\pi ^{\ast
}g_{M} $ (see Notation \ref{N-3-1}). It turns out that $\pi ^{\ast }g_{M}$
on $\pi ^{\ast }\mathcal{E}_{M}$ is the same as $G_{a,m}|_{\pi ^{\ast }%
\mathcal{E}_{M}}$ (cf. Lemma \ref{L-inv} $i)$)$.$ Let%
\begin{equation}
\pi _{m}:L^{2,\ast }(\Sigma ,\pi ^{\ast }\mathcal{E}_{M},G_{a,m})\rightarrow
L_{m}^{2,\ast }(\Sigma ,G_{a,m})\subset L^{2,\ast }(\Sigma ,\pi ^{\ast }%
\mathcal{E}_{M},G_{a,m})  \label{pim}
\end{equation}%
\noindent denote the orthogonal projection onto the $m$-space $L_{m}^{2,\ast
}(\Sigma ,G_{a,m})$ or $L_{m}^{2,\ast }(\Sigma )$ for short with respect to
the metric $G_{a,m}$ or $\pi ^{\ast }g_{M}.$ See Proposition \ref{projm}
below for more about $\pi _{m}.$

\begin{lemma}
\label{R-5-5} With the notation above, we have $L_{m}^{2,\ast }(\Sigma
,G_{a,m})$ $\subset $ $L^{2,\ast }(\Sigma ,\pi ^{\ast }\mathcal{E}%
_{M},G_{a,m}).$
\end{lemma}
\end{remark}

\begin{proof}
First we claim that $\Omega _{m}^{0,q}(\Sigma )$ $\subset $ $L^{2,\ast
}(\Sigma ,\pi ^{\ast }\mathcal{E}_{M},G_{a,m}).$ Write $u$ $\in $ $\Omega
_{m}^{0,q}(\Sigma )$ as $u=\sum_{j}\varphi _{j}u$ with $\varphi _{j}u$ $=$ $%
w^{m}\varphi _{j}v_{j}(z,\bar{z})$ for $C^{\infty }$-smooth $(0,q)$-forms $%
v_{j}$ on $U_{j}$ (see the beginning of Section \ref{Sec4} for the
notation). Let $\chi _{k}(|w|)$ be a cut-off function which equals $1$ for $%
|w|$ $\leq $ $k$ and $0$ for $|w|$ $>k+1.$ It is not difficult to see that $%
\chi _{k}(|w|)w^{m}\varphi _{j}v_{j}(z,\bar{z})$ (which is smooth and of
compact support) tends to $w^{m}\varphi _{j}v_{j}(z,\bar{z})$ $=$ $\varphi
_{j}u$ in $L^{2}$ in view of Remark \ref{3-r}. So $\varphi _{j}u$ (hence $%
u)\in $ $L^{2,\ast }(\Sigma ,\pi ^{\ast }\mathcal{E}_{M},G_{a,m})$ by
definition. We have shown the claim. It follows that $L_{m}^{2,\ast }(\Sigma
,G_{a,m}),$ the $L^{2}$-closure of $\Omega _{m}^{0,q}(\Sigma ),$ should also
be included in $L^{2,\ast }(\Sigma ,\pi ^{\ast }\mathcal{E}_{M},G_{a,m})$.
\end{proof}

For $\nu $ $\in $ $Spec\,\tilde{\square}_{m}^{c\pm }$ let $\tilde{P}_{m,\nu
}^{\pm }$ $:$ $L^{2,\pm }(\Sigma ,\pi ^{\ast }\mathcal{E}_{M},G_{a,m})$ $%
\rightarrow $ $\mathcal{\tilde{E}}_{m,\nu }^{\pm }(\Sigma )$ $\subset $ $%
L^{2,\pm }(\Sigma ,\pi ^{\ast }\mathcal{E}_{M},G_{a,m})$ denote the
orthogonal projections. Denote the distribution kernels of $\tilde{P}_{m,\nu
}^{\pm }$ by $\tilde{P}_{m,\nu }^{\pm }(x,y)$ ($\in C^{\infty }(\Sigma
\times \Sigma ,$ $T^{\ast 0,\pm }\Sigma \boxtimes (T^{\ast 0,\pm }\Sigma
)^{\ast }).$

\begin{proposition}
\label{5-5.5} Define the heat kernels of $\tilde{\square}_{m}^{c+}$ and $%
\tilde{\square}_{m}^{c-}$ to be%
\begin{equation}
e^{-t\tilde{\square}_{m}^{c\pm }}(x,y):=\tilde{P}_{m,0}^{\pm
}(x,y)+\sum_{\nu \in Spec\tilde{\square}_{m}^{c\pm },\nu >0}e^{-\nu t}\tilde{%
P}_{m,\nu }^{\pm }(x,y).  \label{5.2}
\end{equation}%
\noindent Then for a fixed $t$ $>$ $0$ $e^{-t\tilde{\square}_{m}^{c\pm }}$
is a bounded linear operator on $L^{2,\pm }(\Sigma ,\pi ^{\ast }\mathcal{E}%
_{M},G_{a,m}),$ which maps $\Omega ^{0,\pm }(\Sigma )\cap L^{2,\pm }(\Sigma
,\pi ^{\ast }\mathcal{E}_{M},G_{a,m})$ into $\Omega _{m}^{0,\pm }(\Sigma )$ $%
\subset $ $L^{2,\pm }(\Sigma ,\pi ^{\ast }\mathcal{E}_{M},G_{a,m}).$
Moreover, $e^{-t\tilde{\square}_{m}^{c\pm }}$ in (\ref{5.2}) are (Hilbert
space) self-adjoint and the kernel functions are infinitely smooth. They
satisfy%
\begin{eqnarray}
&&(\frac{\partial }{\partial t}+\tilde{\square}_{m}^{c\pm })(e^{-t\tilde{%
\square}_{m}^{c\pm }}u) = 0,\text{ }\forall t>0,  \label{5.2.1} \\
&&e^{-t\tilde{\square}_{m}^{c\pm }}u \rightarrow \pi _{m}^{\pm }u\text{ in }%
L^{2}\text{ as }t\rightarrow 0\text{ \ }\forall u\in \Omega ^{0,\pm }(\Sigma
)\cap L^{2,\pm }(\Sigma ,\pi ^{\ast }\mathcal{E}_{M},G_{a,m})  \notag
\end{eqnarray}

\noindent where $\pi _{m}^{\pm }:L^{2,\pm }(\Sigma ,\pi ^{\ast }\mathcal{E}%
_{M},G_{a,m})\rightarrow L_{m}^{2,\pm }(\Sigma ,G_{a,m})$ is the orthogonal
projection.
\end{proposition}

\begin{remark}
\label{5-8-1} Although the uniqueness part can be done here, it is postponed
until Theorem \ref{t-uniqueness} $i)$ for the sake of convenience.
\end{remark}

\proof
\textbf{(of Proposition \ref{5-5.5})} We need to show that the kernel
functions $e^{-t\tilde{\square}_{m}^{c\pm }}(x,y)$ are infinitely smooth.
The other statements are immediate (cf. the last paragraph of this proof
with references, via the first half of Proposition \ref{l-5-2} above). First
we prove that the eigenvalues $0<\nu _{1}$ $\leq $ $\nu _{2}$ $\leq $ $%
...\nu _{n}$ $\leq ...$of $\tilde{\square}_{m}^{c\pm }$ $($counting
multiplicity) satisfy the growth rate as follows: 
\begin{equation}
\nu _{n}\text{ }\geq Cn^{\delta }\text{ for a constant }C>0\text{ and an
exponent }\delta >0  \label{5-2.2}
\end{equation}

\noindent if $n$ $>$ $n_{0}$ is large (see Lemmas 1.6.3 and 1.6.5 in \cite%
{Gil}). The proof in \cite{Gil} for elliptic operators on a compact manifold
needs to be modified as shown below.

Let \{$\omega _{j}^{\pm }\}$ denote a complete orthonormal basis for $%
L_{m}^{2,\pm }(\Sigma ,G_{a,m})$ such that $\tilde{\square}_{m}^{c\pm
}\omega _{j}^{\pm }$ $=$ $\nu _{j}\omega _{j}^{\pm }$ by Proposition \ref%
{l-5-2} above. For $f$ $\in $ $\Omega _{m}^{0,\pm }(\Sigma )$ we observe the
following estimate:%
\begin{equation}
|f(x)|\leq Cl(x)^{m/2}||f||_{C_{B}^{0}}  \label{5-2.3}
\end{equation}

\noindent (see (\ref{CBs-norm}) for the definition of the norm $||\cdot
||_{C_{B}^{s}}$ and (\ref{lq}) for $l(x)).$ We have the following (recalling
the notation $\lesssim $ meaning \textquotedblleft the inequality $\leq $
holds modulo some multiplicative constant"), where for the first inequality
we are applying the usual Sobolev embedding (after choosing $k$ such that $%
k\cdot 2$ $>$ $\dim _{\mathbb{R}}(\Sigma /\mathbb{C}^{\ast })/2$ $=$ $n-1)$
together with using $(\ref{eqnorm}),$ Lemmas \ref{4.7-5} and \ref{normeq},

\begin{eqnarray}
||f||_{C_{B}^{0}} &\overset{}{\lesssim }&||f||_{2k}^{\prime }\overset{\text{%
Lemma \ref{5-3.5}}}{\lesssim }||\tilde{\square}_{m}^{c\pm
}f||_{2k-2}^{\prime }+||f||_{0}^{\prime }  \label{5-2.4} \\
&\overset{\text{Lemma \ref{5-3.5}}}{\lesssim }&||(\tilde{\square}_{m}^{c\pm
})^{k}f||_{0}^{\prime }+||(\tilde{\square}_{m}^{c\pm })^{k-1}f||_{0}^{\prime
}+...+||f||_{0}^{\prime }  \notag
\end{eqnarray}

\noindent in the Sobolev $s$-norm $||\cdot ||_{s}^{\prime }$ on $(\Sigma
,G_{a,m}).$ Remark that the bundle $(L_{\Sigma }^{\ast })^{\otimes m}$
implicitly involved in (\ref{5-2.4}), (\ref{eqnorm}) does not really matter
with the preceding estimate. \ 

The interpolation inequality (Corollary \ref{Ip}) brings (\ref{5-2.4}) to%
\begin{equation}
||f||_{C_{B}^{0}}\lesssim ||(\tilde{\square}_{m}^{c\pm })^{k}f||_{0}^{\prime
}+||f||_{0}^{\prime }.  \label{5-2-5}
\end{equation}

Taking $f$ $=$ $\sum_{j=1}^{n}c_{j}\omega _{j}^{\pm }$ in (\ref{5-2-5}) and (%
\ref{5-2.3}) gives (recalling that \{$\omega _{j}^{\pm }\}$ is orthonormal
w.r.t. the $L^{2}$-norm $||\cdot ||_{0}$ which is equivalent to $||\cdot
||_{0}^{\prime }$ by Remark \ref{4.1-5}$)$%
\begin{eqnarray}
\text{ \ \ }|\sum_{j=1}^{n}c_{j}\omega _{j}^{\pm }(x)| &\leq
&C_{1}l(x)^{m/2}(||\sum_{j=1}^{n}c_{j}\nu _{j}^{k}\omega _{j}^{\pm
}||_{0}+||\sum_{j=1}^{n}c_{j}\omega _{j}^{\pm }||_{0})  \label{5-2-6} \\
&\leq &C_{1}l(x)^{m/2}\left( (\sum_{j=1}^{n}|c_{j}\nu
_{j}^{k}|^{2})^{1/2}+(\sum_{j=1}^{n}|c_{j}|^{2})^{1/2}\right)  \notag \\
&\overset{|\nu _{j}|\leq |\nu _{n}|}{\leq }&C_{1}l(x)^{m/2}(|\nu
_{n}|^{k}+1)(\sum_{j=1}^{n}|c_{j}|^{2})^{1/2}.  \notag
\end{eqnarray}

Letting $c_{j}=\bar{\omega}_{j}^{\pm }(x)$ in (\ref{5-2-6}), squaring,
cancelling off $(\sum_{j=1}^{n}|c_{j}|^{2})^{1/2}$ on both sides and
integrating over $\Sigma ,$ we get%
\begin{equation}
n\leq C_{1}^{2}(|\nu _{n}|^{k}+1)^{2}\int_{\Sigma }l(x)^{m}dv_{\Sigma ,m}.
\label{5-2-7}
\end{equation}

\noindent By observing that%
\begin{equation*}
\int_{\Sigma }l(x)^{m}dv_{\Sigma ,m}\overset{(\ref{volume})}{=}%
\int_{M\backslash M_{\text{sing}}}dv_{M}\int_{\mathbb{C}^{\ast }}(\tau
_{x}^{\ast }l)^{m}\tau _{x}^{\ast }(dv_{f,m})\overset{(\ref{fibrenv})}{=}%
Vol(M)\cdot 1
\end{equation*}

\noindent is finite, where $M_{\text{sing}}$ denotes the set of singular
orbifold points in $M$ $=$ $\Sigma /\mathbb{\sigma }$ (of measure zero), we
reach (\ref{5-2.2}) with $\delta =\frac{1}{2k}$ from (\ref{5-2-7}).

To show that the kernel functions $e^{-t\tilde{\square}_{m}^{c\pm }}(x,y)$
are infinitely smooth from the growth rate of $\nu _{n}$ in (\ref{5-2.2}),
we imitate the arguments in \cite[pp.53-55]{Gil} by using the norm $||\cdot
||_{C_{B}^{s}}$ in place of the supreme $s$-norm in \cite{Gil}. It is seen
that the $C_{B}^{s}$-norms are suitable here, since the functions in our $m$%
-space are of the special form (\ref{6-1.5}) of Section 6. Note that
corresponding to \cite[b) of Lemma 1.6.3, p.51]{Gil} one can show similarly
that $||\omega _{j}^{\pm }||_{C_{B}^{s}}\lesssim 1+|\lambda _{j}|^{l(s)}$
from (\ref{5-2.2}) and (\ref{5-2.4}) (generalized from $C_{B}^{0}$ to $%
C_{B}^{l}$), and that (\ref{5.2}) is the analogous expression in \cite[Lemma
1.6.5, p.55]{Gil}. These observations (together with the form of expressions
(\ref{3.36-5}), (\ref{3.36-25}); see also (\ref{6-1.5}) which reduces the
study (transversal case) to that on $z$-spaces (elliptic case)) yield the
desired smoothness of (\ref{5.2}) as in \cite{Gil}. The convergence of $e^{-t%
\tilde{\square}_{m}^{c\pm }}u$ is treated similarly. We leave the details to
the reader.

For the remaining properties of $e^{-t\tilde{\square}_{m}^{c\pm }}$ claimed
in the proposition, we observe that for $u\in \Omega ^{0,\pm }(\Sigma )\cap
L^{2,\pm }(\Sigma ,\pi ^{\ast }\mathcal{E}_{M},G_{a,m})$ 
\begin{equation*}
(\frac{\partial }{\partial t}+\tilde{\square}_{m}^{c\pm })(e^{-\nu t}\tilde{P%
}_{m,\nu }^{\pm }u)=-\nu e^{-\nu t}\tilde{P}_{m,\nu }^{\pm }u+e^{-\nu t}\nu 
\tilde{P}_{m,\nu }^{\pm }u=0.
\end{equation*}

\noindent The first equation in (\ref{5.2.1}) follows since one sees that
taking differentiation in $t$ or $x$ commutes with the infinite sum of
kernel functions of projectors (for a fixed $t>0).$ For the second formula
of (\ref{5.2.1}) writing $u=\sum_{j}a_{j}\omega _{j}^{\pm }$ we have%
\begin{equation}
e^{-t\tilde{\square}_{m}^{c\pm }}u-\pi _{m}u=\sum_{\nu \in Spec\tilde{\square%
}_{m}^{c\pm }}(e^{-\nu t}-1)\tilde{P}_{m,\nu }^{\pm }u=\sum_{\nu \in Spec%
\tilde{\square}_{m}^{c\pm }}(e^{-\nu t}-1)a_{\nu }\omega _{\nu }^{\pm }.
\label{5-2-8}
\end{equation}

\noindent Note that $e^{-\nu t}-1$ in (\ref{5-2-8}) is bounded by $1$ since $%
\nu \geq 0.$ $\sum_{\nu }a_{\nu }^{2}$ is bounded, so for a large $N,$ $%
\sum_{\nu \geq N}a_{\nu }^{2}$ is small. For a finite sum $%
\lim_{t\rightarrow 0}\sum_{\nu <N}(e^{-\nu t}-1)a_{\nu }\omega _{\nu }^{\pm
} $ $=$ $\sum_{\nu <N}\lim_{t\rightarrow 0}(e^{-\nu t}-1)a_{\nu }\omega
_{\nu }^{\pm }=0.$ Altogether $e^{-t\tilde{\square}_{m}^{c\pm }}u-\pi _{m}u$ 
$\rightarrow $ $0$ in $L^{2}$ as $t\rightarrow 0.$ One sees that $||e^{-t%
\tilde{\square}_{m}^{c\pm }}u||_{L^{2}}\leq (\sum_{\nu \in Spec\tilde{\square%
}_{m}^{c\pm }}e^{-\nu t})||u||_{L^{2}},$ so $e^{-t\tilde{\square}_{m}^{c\pm
}}$ is bounded by (\ref{5-2.2}). The self-adjointness of $e^{-t\tilde{\square%
}_{m}^{c\pm }}$ follows from that each $\tilde{P}_{m,\nu }^{\pm }$ is
self-adjoint and the infinite sum of $\tilde{P}_{m,\nu }^{\pm \ast }$
converges to ($e^{-t\tilde{\square}_{m}^{c\pm }})^{\ast }.$

\endproof%

For $\nu $ $\in $ $Spec$ $\tilde{\square}_{m}^{c\pm },$ let $\{f_{1}^{\nu
},...,f_{d_{\nu }^{\pm }}^{\nu }\}$ be an orthonormal basis for $\mathcal{%
\tilde{E}}_{m,\nu }^{\pm }(\Sigma ).$ Define the trace of $\tilde{P}_{m,\nu
}^{\pm }(x,x)$ by $Tr\tilde{P}_{m,\nu }^{\pm }(x,x):=\sum_{j=1}^{d_{\nu
}^{\pm }}|f_{j}^{\nu }(x)|^{2}\in C^{\infty }(\Sigma )$ which equals $%
\sum_{j=1}^{d^{\pm }}<\tilde{P}_{m,\nu }^{\pm
}(x,x)e_{j}(x)|e_{j}(x)>_{G_{a,m}}$ where $\{e_{j}(x)\}_{j=1,...,d^{\pm }}$
is any orthonormal basis of $T_{x}^{\ast 0,\pm }\Sigma .$ From Propositions %
\ref{5-5.5}, \ref{l-5-2} and Lemma \ref{l-4-3} (for $\tilde{D}_{m}^{c\pm }$
and $\tilde{\square}_{m}^{c\pm }$) it follows

\begin{proposition}
\label{p-5-3} (Formula of McKean-Singer type for $index(\tilde{D}_{m}^{c+}))$
For each $t>0$ we have 
\begin{equation*}
index(\tilde{D}_{m}^{c+})=\int_{\Sigma }[Tre^{-t\tilde{\square}%
_{m}^{c+}}(x,x)-Tre^{-t\tilde{\square}_{m}^{c-}}(x,x)]dv_{\Sigma ,m}.
\end{equation*}
\end{proposition}

Recall $D_{m}^{+}$ of (\ref{Dm}). To compare $index(D_{m}^{+})$ with $index(%
\tilde{D}_{m}^{c+})$, we have the following homotopy invariance. Here our
Hodge theory in Section 4 provides a useful tool in the proof below.

\begin{lemma}
\label{l-5-3} (Homotopy invariance) $index(D_{m}^{+})=index(\tilde{D}%
_{m}^{c+}).$
\end{lemma}

\proof
Despite that our operators are of "transversal" type in the sense as
constructed in this subsection, the arguments are essentially classical in
spirit. The key point is to make sure that the noncompactness of $\Sigma $
endowed with our various geometric data does no essential harm to those
arguments that are valid for compact manifolds. We sketch the idea of the
proof; for more details and references, the reader is referred to the proof
of \cite[Theorem 4.7]{CHT} in a similar vein.

From (\ref{Dm}) and (\ref{5.1}), we have $\tilde{D}_{m}^{c+}=D_{m}^{+}+%
\tilde{A}_{m}^{c}$ where $\tilde{A}_{m}^{c}:\Omega _{m}^{0,+}(\Sigma
)\rightarrow \Omega _{m}^{0,-}(\Sigma )$ is a bounded linear operator of
zeroth order. A homotopy between $L_{0}$ $=$ $D_{m}^{+}$ and $L_{1}$ $=$ $%
\tilde{D}_{m}^{c+}$ can be realized by $L_{t}$ $:=$ $D_{m}^{+}+t\tilde{A}%
_{m}^{c}$ $=$ $\bar{\partial}_{\Sigma ,m}+\bar{\partial}_{\Sigma ,m}^{\ast
}+t\tilde{A}_{m}^{c}:\Omega _{m}^{0,+}(\Sigma )\rightarrow \Omega
_{m}^{0,-}(\Sigma )$ for $t\in \lbrack 0,1].$ Extending $L_{t}$ to $%
Dom(L_{t})$ $\subset $ $L_{0,+,m}^{2}$ $:=$ $\oplus
_{q:even}L_{0,q,m}^{2}(\Sigma ,G_{a,m})$, one can show that $Dom(L_{t})$ $=$ 
$H_{0,+,m}^{\prime 1}$ $:=$ $L_{0,+,m}^{2}\cap H_{0,+}^{\prime 1}$ where $%
H_{0,+}^{\prime 1}$ $:=$ $\oplus _{q:even}H_{0,q}^{\prime 1}(\Sigma
,G_{a,m}) $ (cf. (\ref{Als}) for the notation).

Now consider $\mathcal{H}_{0}$ $:=$ $H_{0,+,m}^{\prime 1}$ $\oplus $ $Ker$ $%
L_{0}^{\ast }$ and $\mathcal{H}_{1}$ $:=$ $L_{0,-,m}^{2}$ $\oplus $ $Ker$ $%
L_{0}.$ Let $A_{t}$ $:$ $\mathcal{H}_{0}$ $\rightarrow $ $\mathcal{H}_{1}$
be the bounded linear map defined by $A_{t}(u,v)=(L_{t}u+v,P_{KerL_{0}}u)\in 
\mathcal{H}_{1}$ for $(u,v)$ $\in $ $\mathcal{H}_{0},$ where $P_{KerL_{0}}$
denotes the orthogonal projection onto $Ker$ $L_{0}.$ We claim the following
fact:%
\begin{equation}
\exists \text{ }r_{0}>0\text{ such that }A_{t}\text{ is invertible for every 
}0\leq t\leq r_{0}.  \label{At}
\end{equation}

\noindent For $t$ $=$ $0,$ the fact that $A_{0}$ is invertible follows from
the Hodge theory for $L_{0}$ $=$ $D_{m}^{+}$ (cf. Theorem \ref{t-4-2})$.$
For $t$ $\neq $ $0,$ write $A_{t}$ $=$ $A_{0}+R_{t}$ so that $||R_{t}u||_{%
\mathcal{H}_{1}}$ $\leq $ $Ct||u||_{\mathcal{H}_{0}}.$ We can then construct
the inverse of $A_{t}$ by the Neuman series for small $t,$ proving (\ref{At}%
).

We claim another fact: (In the remaining of the proof, we use $``ind"$ as
abbreviation of $``index".)$%
\begin{equation}
\exists \text{ }r>0\text{ such that }ind\text{ }L_{t}=ind\text{ }L_{0}\text{
for every }0\leq t\leq r.  \label{Indt}
\end{equation}

\noindent For $0\leq t\leq r_{0}$ in (\ref{At}) we define $B_{t}$ $:$ $Ker$ $%
L_{t}^{\ast }\oplus Ker$ $L_{0}$ $\rightarrow $ $Ker$ $L_{t}\oplus Ker$ $%
L_{0}^{\ast }$ by $B_{t}(a,b):=(P_{KerL_{t}}u,v)\in Ker$ $L_{t}\oplus Ker$ $%
L_{0}^{\ast }$ where $(u,v)$ $=$ $A_{t}^{-1}(a,b).$ It is not hard to see
that $B_{t}$ is injective. It follows that $\dim KerL_{t}^{\ast }+\dim
KerL_{0}\leq \dim KerL_{t}+\dim KerL_{0}^{\ast }.$ Hence $ind$ $L_{0}:=\dim
KerL_{0}$ $-$ $\dim KerL_{0}^{\ast }$ $\leq $ $\dim KerL_{t}$ $-$ $\dim
KerL_{t}^{\ast }$ $=$ $ind$ $L_{t}.$ By a similar argument, we have $ind$ $%
L_{0}^{\ast }$ $\leq $ $ind$ $L_{t}^{\ast }$ for small t. Observe that $ind$ 
$L_{t}^{\ast }$ = $-ind$ $L_{t}.$ So we also have $ind$ $L_{0}$ $\geq $ $ind$
$L_{t}.$ We have shown (\ref{Indt}).

We shall now show that ($ind$ $\tilde{D}_{m}^{c+}$ $=)$ $ind$ $L_{1}$ $=$ $%
ind$ $L_{0}$ $(=ind$ $D_{m}^{+})$ by the continuity method. Let $\Lambda $ $%
= $ $\{t\in \lbrack 0,1]:ind$ $L_{t}=ind$ $L_{0}\}.$ Clearly $0\in \Lambda ,$
so $\Lambda $ is not empty. Suppose $t_{0}$ $\in $ $\Lambda .$ By reasoning
similar to the proof of (\ref{At}) and (\ref{Indt}) (replacing $L_{0},$ $%
A_{0}$ by $L_{t_{0}},$ $A_{t_{0}},$ respectively$),$ we can show $ind$ $%
L_{t}=ind$ $L_{t_{0}}$ for $t$ $\in $ $(t_{0}-\varepsilon ,$ $%
t_{0}+\varepsilon )$ with some $\varepsilon $ $>$ $0.$ This implies that $%
\Lambda $ is open. On the other hand, we apply the same reasoning as in (\ref%
{Indt}) to a limit point $t_{\infty }$ of $\Lambda $ and show that $ind$ $%
L_{t_{\infty }}$ $=$ $ind$ $L_{t_{n}}$ $=$ $ind$ $L_{0}$ where $t_{n}$ $\in $
$\Lambda $ is close enough to $t_{\infty }.$ So $t_{\infty }$ $\in $ $%
\Lambda .$ We have shown that $\Lambda $ is closed. Therefore $\Lambda $ $=$ 
$[0,1].$

\endproof%

From Lemma \ref{l-5-3}, Corollary \ref{t-4-3} and Proposition \ref{p-5-3},
there follows a formula of McKean-Singer type:

\begin{theorem}
\label{t-5-1} (Formula of McKean-Singer type for $index(\bar{\partial}%
_{\Sigma ,m}$-complex$))$ For $m\geq 0$ and $a$ $>$ $\frac{m}{2}$ we have
for each $t>0$%
\begin{equation}
\sum_{q=0}^{n}(-1)^{q}\dim H_{m}^{q}(\Sigma ,\mathcal{O})=\int_{\Sigma
}[Tre^{-t\tilde{\square}_{m}^{c+}}(x,x)-Tre^{-t\tilde{\square}%
_{m}^{c-}}(x,x)]dv_{\Sigma ,m}.  \label{5.5}
\end{equation}
\end{theorem}

\begin{remark}
\label{5-17.5} (Bundle case): Let $E$ be a $\mathbb{C}^{\ast }$-equivariant
holomorphic vector bundle over $\Sigma $, endowed with a $\mathbb{C}^{\ast }$%
-invariant Hermitian metric$.$ We can extend $D_{m}^{\pm },$ $\tilde{A}%
_{m}^{c}$, $\tilde{D}_{m}^{c\pm }$ and hence $\tilde{\square}_{m}^{c\pm }$
to $E$-valued $m$-spaces $\Omega _{m}^{0,\pm }(\Sigma ,E)$ in a standard
manner. By similar arguments in deducing (\ref{5.5}), we also have a
McKean-Singer type formula for $index(\bar{\partial}_{\Sigma ,m}^{E}$-complex%
$).$ That is, with $H_{m}^{q}(\Sigma ,\mathcal{O})$ replaced by $%
H_{m}^{q}(\Sigma ,E)$ and $\tilde{\square}_{m}^{c\pm }$ in the RHS replaced
by their counterparts for $E$, the resulting two sides are equal.
\end{remark}

\section{\textbf{Approximation of the transversal heat kernel }$e^{-t\tilde{%
\square}_{m}^{c\pm }}\label{A-THK}$}

In this section we are going to construct an (transversal) approximate heat
kernel by patching up local heat kernels and taking its adjoint. We then
carry out successive approximation to get a global (unique, transversal)
heat kernel. Our main results are Theorems \ref{t-adjeq}, \ref{t-uniqueness}
and \ref{p-asymp} while the most technical lemma is Proposition \ref{projm}
which brings the orthogonal projection $\pi _{m}$ as already seen in (\ref%
{pim}) to an integral representation.

The motivations for the whole setting are implicit in Propositions \ref{madj}
and \ref{BoxDU}. Let us remark that the two cases stated between (\ref{6-1g}%
) and (\ref{HQ}) yield some complication of the heat kernel evaluation as
mentioned in the Introduction.

To start with, let us choose suitable charts on $\Sigma $ as follows. For $%
\Sigma $ satisfying (\ref{1-0}), (\ref{1-1}) and (\ref{1-2}), with any point 
$q$ $\in $ $\Sigma $ write $\bar{q}$ $\in $ $\Sigma /\mathbb{R}^{+},$ the $%
\mathbb{R}^{+}$ orbit of $q.$ Choose a distance function on $\Sigma /\mathbb{%
R}^{+}$ (say, obtained from a Riemannian metric on $\Sigma /\mathbb{R}^{+}$
which is a smooth manifold as mentioned in the proof of Theorem \ref{thm2-1}$%
).$ Given small $\varepsilon $ \TEXTsymbol{>} $0$, we have a coordinate
neighborhood $V\times (-\varepsilon ,\varepsilon )$ $\subset $ $\Sigma /%
\mathbb{R}^{+}$ where $V$ $\subset $ $\mathbb{C}^{n-1}$ is a bounded domain.
Remark that we may always choose $\varepsilon $ $=$ $\pi $ if the $\mathbb{C}%
^{\ast }$ action on $\Sigma $ is globally free. Due to the compactness of $%
\Sigma /\mathbb{R}^{+}$, we can find finitely many $V_{j}\times
(-\varepsilon _{j},\varepsilon _{j}),$ $V_{j}$ $\subset $ $\mathbb{C}^{n-1}$
so that (cf. (\ref{1-0})) 
\begin{equation}
\Sigma /\mathbb{R}^{+}=\cup _{j}(V_{j}\times (-\varepsilon
_{j}/4,\varepsilon _{j}/4))\text{, }\Sigma =\cup _{j}(V_{j}\times
(-\varepsilon _{j}/4,\varepsilon _{j}/4)\times \mathbb{R}^{+})  \label{6.0}
\end{equation}

\noindent It may be useful to assume that all $\varepsilon _{j}^{\prime }s$
are the same, but we keep the subscript $j$ for the time being. The
following notations $W_{j}$ and $\hat{W}_{j}$ will be used later.

\begin{notation}
\label{n-6.1} $W_{j}:=V_{j}\times (-\varepsilon _{j},\varepsilon _{j})\times 
\mathbb{R}^{+}\subset \Sigma $ and $\hat{W}_{j}:=V_{j}\times (-\varepsilon
_{j}/4,\varepsilon _{j}/4)\times \mathbb{R}^{+}.$ See Remark \ref{R-8-37}
for explicit choices of $\varepsilon _{j}.$
\end{notation}

Below are some cut-off functions $\varphi _{j}$ and $\tau _{j}$ with values
in $[0,1].$

$i)$ $\varphi _{j}\in C_{c}^{\infty }(V_{j}\times (-\varepsilon
_{j},\varepsilon _{j}))$ with $\sum \varphi _{j}$ $=$ $1$ on $\Sigma /%
\mathbb{R}^{+}.$ Extend the domain of definition of $\varphi _{j}$ to $W_{j}$
(still denoted as $\varphi _{j})$ by $\varphi _{j}(\bar{q},r)=\varphi _{j}(%
\bar{q})$ for any $r\in \mathbb{R}^{+}$ and $\bar{q}\in V_{j}\times
(-\varepsilon _{j},\varepsilon _{j}).$ So we have $\sum \varphi _{j}$ $=$ $1$
on $\Sigma $ in view of (\ref{6.0}). Note that supp $\varphi _{j}$ $\subset $
$\Sigma $ must be noncompact while if $\varphi _{j}$ is regarded as
functions on $V_{j}\times (-\varepsilon _{j},\varepsilon _{j}),$ supp $%
\varphi _{j}$ is compact.

Let ($z,\phi )$ $\in $ $V_{j}\times (-\varepsilon _{j},\varepsilon _{j})$
denote the coordinates for $\bar{q}$ $\in $ $V_{j}\times (-\varepsilon
_{j},\varepsilon _{j}).$ Put%
\begin{equation*}
A_{j}=\left\{ z\in V_{j}:\text{there is a }\phi \in (-\varepsilon
_{j},\varepsilon _{j})\text{ such that }\varphi _{j}(z,\phi )\neq 0\right\}
\subset \subset V_{j}.
\end{equation*}

$ii)$ $\tau _{j}(z)\in C_{c}^{\infty }(V_{j})$ with $\tau _{j}\equiv 1$ on
some neighborhood of $A_{j}$ (and $\equiv $ $0$ outside $V_{j}).$

$iii)$ $\sigma _{j}$ $\in $ $C_{c}^{\infty }((-\frac{\varepsilon _{j}}{4},%
\frac{\varepsilon _{j}}{4}))$ with 
\begin{equation}
\int_{-\varepsilon _{j}}^{\varepsilon _{j}}\sigma _{j}(\phi )\frac{%
dv_{S^{1}}(\phi )}{2\pi }=\int_{-\varepsilon _{j}/4}^{\varepsilon
_{j}/4}\sigma _{j}(\phi )\frac{dv_{S^{1}}(\phi )}{2\pi }=1.  \label{6-1a}
\end{equation}

It is possible to adapt the seemingly unusable formulas in \cite[(5.39) on
p.81]{CHT} to the present situation. Let us first set up the following. For
a (regular) domain $\Omega $ $\subset $ $\mathbb{C}^{n-1},$ we have the
Dirichlet heat kernel for $\square _{\Omega ,m}^{c}$ (see (\ref{DiracLap}) $%
iii),$ cf. (\ref{6-11a}), (\ref{6-11.75})), denoted by $K_{t}^{\Omega
}(z,\zeta )$ for $z,$ $\zeta $ $\in $ $\Omega .$ See the preceding section
for the definition of such spin$^{c}$ Laplacians (cf. (\ref{DiracLap})), and 
\cite{Chavel} or \cite{CHT} for the Dirichlet heat kernel construction
(under suitable regularity conditions on $\partial V_{j}).$ We set%
\begin{equation}
K_{t}^{j}(z,\zeta )=K_{t}^{\Omega }(z,\zeta )\text{ with }\Omega =V_{j}.
\label{6-1'}
\end{equation}

\noindent Note that we have identified $V_{j}$ $\subset $ $\mathbb{C}^{n-1}$
with $V_{j}\times \{0\}\times \{1\}$ $\subset $ $V_{j}\times (-\varepsilon
_{j},$ $\varepsilon _{j})\times \mathbb{R}^{+}$ as embedded in $\Sigma $
with the induced metric $\pi ^{\ast }g_{M}|_{V_{j}}$.

Remark that we should have considered the \textit{adjoint} heat kernel in (%
\ref{6-1'}); but the operator here is self-adjoint the associated kernel
functions are the same: ($K_{t}^{j})^{\ast }(z,\zeta )$ $=$ $%
K_{t}^{j}(z,\zeta )$ (acting to the left on an element in $\zeta ;$ compare (%
\ref{endo})$).$

For $u$ $\in $ $L^{2,\pm }(\Sigma ,\pi ^{\ast }\mathcal{E}_{M},G_{a,m})$ we
have (see Notation \ref{n-6.1} for $W_{j})$%
\begin{equation}
\pi _{m}u|_{W_{j}}(\zeta ,\eta )=\eta ^{m}\upsilon _{j}(\zeta );\text{
subscript \textquotedblleft }j\text{" omitted in }(\zeta ,\eta )
\label{6-1.5}
\end{equation}%
\noindent for some $\upsilon _{j}(\zeta )$ $\in $ $L^{2,\pm }(V_{j}).$ This $%
L^{2}$ property of $\upsilon _{j}$ can be checked using (\ref{3-19.5}). It
is easily verified, with the metric $G_{a,m},$ that $L_{m}^{2,q}(\Sigma )$
is orthogonal to $L_{m^{\prime }}^{2,q}(\Sigma )$ if $m\neq m^{\prime }.$
But these spaces $\{L_{m}^{2,q}(\Sigma )\}_{m\in \mathbb{Z}}$ do not
constitute a complete decomposition of $L^{2,q}(\Sigma ).$

Let $\Psi _{\pm ,m}$ (resp. $\Psi _{\ast ,m}$) denote $\Psi _{q,m}$ in (\ref%
{Psi_qm}) for $q$ even/odd (resp. for all $q$) to identify bundle-valued
elements on $V_{j}$ with $m$-space elements on $W_{j}$ (probably with an
extra bundle $E$; $U_{j},$ $\psi _{j}^{-1}(D_{j})$ in (\ref{Psi_qm}) taken
to be $V_{j},$ $W_{j}$ in Notation \ref{n-6.1}).

We schematically define the operator $H_{m,t}^{j}$ to be, with cut-off
functions omitted, $\Psi _{\ast ,m}\circ $ $K_{t}^{j}\circ $ $\Psi _{\ast
,m}^{-1}$ (see (\ref{Psi_qm}) and (\ref{6-1'})) and its adjoint to be $\Psi
_{\ast ,m}\circ $ $K_{t}^{j\ast }\circ $ $\Psi _{\ast ,m}^{-1}$ where $%
K_{t}^{j\ast }$ is the usual heat kernel adjoint (see Remark \ref{R-6-3}
below)$.$ More precisely one has the expression in terms of local
coordinates (with the induced trivializations on bundles) and of those
cut-off functions after Notation \ref{n-6.1}%
\begin{equation}
H_{m,t}^{j}(x,y)\overset{(\ref{6-1'})}{:=}\varphi
_{j}(x)w^{m}K_{t}^{j}(z,\zeta )\eta ^{-m}\tau _{j}(\zeta )\sigma
_{j}(\vartheta )l(y)^{m}  \label{6-1}
\end{equation}

\noindent where $x$ $=$ $(z,w),$ $y$ $=$ $(\zeta ,\eta )$ ($z,\zeta $ $\in $ 
$V_{j},$ $w,\eta $ $\in $ $(-\varepsilon _{j},\varepsilon _{j})\times 
\mathbb{R}^{+}$ with the subscript \textquotedblleft $j$" omitted in these
coordinates$)$, $\eta $ $=$ $|\eta |e^{i\vartheta }$. This slightly tedious
expression (\ref{6-1}) is basically motivated (modulo the cutoff functions)
by Propositions \ref{BoxDU} and \ref{madj} (see also \cite[(5.38) and (5.39)
in p.81]{CHT} and Remark \ref{6-9-1} below for the differences) with a
seemingly extra factor $l(y)^{m}$ in the end. This factor $l(y)^{m}$ plays a
role similar to $\sigma _{j}$ (cf. (\ref{6-1a})) due to its
normalization/unity property (see (\ref{fibrenv}), (\ref{lambdam})). We note
that the metric $G_{a,m}$ is used for $H_{m,t}^{j}$ while the metric $\pi
^{\ast }g_{M}$ is used for $K_{t}^{j},$ but they coincide on $\pi ^{\ast }%
\mathcal{E}_{M}$ (see Remark \ref{N-6-3}). See also Lemma \ref{L-7iso} for
the $L^{2}$-isometry property of $\Psi _{\pm ,m}$ that partly justifies the
reason for why the expression (\ref{6-1}) is formed in this way (cf. \cite[%
top lines on p.74]{CHT}). See more in Remark \ref{R-6-3}

\begin{remark}
\label{R-6-3} Regarding the adjoint of $H_{m,t}^{j}$ as given before (\ref%
{6-1}), it is shown in Lemma \ref{L-7iso} that the above $\Psi _{\pm ,m}$
preserve $L^{2}$-norms (up to a multiplicative constant), giving
isomorphisms between respective Hilbert spaces. So the adjoint of $%
H_{m,t}^{j}$ in the usual sense and the one as defined above coincide. The
kernel function ($\Psi _{\ast ,m}\circ $ $K_{t}^{j}\circ $ $\Psi _{\ast
,m}^{-1})(x,y)$ on $W_{j}\times W_{j}$ is $\frac{\pi }{\varepsilon _{j}}%
w^{m}K_{t}^{j}(z,\zeta )\eta ^{-m}l(y)^{m}$ because for $u(y)$ $=$ $\eta
^{m}v(\zeta ,\bar{\zeta}),$ $\frac{\pi }{\varepsilon _{j}}%
\int_{W_{j}}w^{m}K_{t}^{j}(z,\zeta )\eta ^{-m}l(y)^{m}u(y)dv_{\Sigma ,m}(y)$
equals 
\begin{equation*}
w^{m}\int_{V_{j}}K_{t}^{j}(z,\zeta )v(\zeta ,\bar{\zeta})\pi ^{\ast
}dv_{M}(\zeta )
\end{equation*}%
(see the proof of Lemma \ref{L-7iso}), and this is seen to be ($\Psi
K_{t}^{j}\Psi ^{-1}(u))(x)$ (via definitions)$.$ This motivates (\ref{6-1}).
Denote by $L_{m,loc}^{2,q}(W_{j},\pi ^{\ast }\mathcal{E}_{M},G_{a,m})$ (or $%
L_{m,loc}^{2,q}(W_{j},G_{a,m})$ for short) the space of $L^{2}$-completion
of $\Omega _{m.loc}^{0,q}(W_{j})$ (see Definition \ref{d-6.8-5}). Similarly
denote by $L^{2,q}(W_{j},\pi ^{\ast }\mathcal{E}_{M},G_{a,m})$ (or $%
L^{2,q}(W_{j},G_{a,m})$ for short) the space of $L^{2}$-completion of
square-integrable smooth sections of $\pi ^{\ast }\mathcal{E}_{M}$ over $%
W_{j}$ with respect to the metric $G_{a,m}$ (cf. Remark.\ref{N-6-3}). In the
similar spirit as in Lemma \ref{R-5-5}, we have $%
L_{m,loc}^{2,q}(W_{j},G_{a,m})$ $\subset $ $L^{2,q}(W_{j},G_{a,m}).$ Now if $%
\tilde{u}$ $\in $ ($L_{m,loc}^{2,q}(W_{j},G_{a,m}))^{\perp }$ $(\subset $ $%
L^{2,q}(W_{j},G_{a,m}))$ then the similar argument as above implies that $%
\Psi K_{t}^{j}\Psi ^{-1}(\tilde{u})$ $=$ $0.$ Hence $\Psi _{\ast ,m}\circ $ $%
K_{t}^{j}$ $\circ \ \Psi _{\ast ,m}^{-1}$ extends to $L^{2,q}(W_{j},G_{a,m})$
with image in $L_{m,loc}^{2,q}(W_{j},G_{a,m})$ $\subset $ $%
L^{2,q}(W_{j},G_{a,m}).$
\end{remark}

We are going to form an approximate heat kernel: 
\begin{equation}
P_{m,t}^{0}:=\sum_{j\text{ (}finite)}P_{m,t}^{0,j}\text{\ \ \ \ where }%
P_{m,t}^{0,j}:=H_{m,t}^{j}\circ \pi _{m}  \label{6-1a1}
\end{equation}

\noindent where $H_{m,t}^{j}$ is the operator associated with the kernel
function $H_{m,t}^{j}(x,y)$ of (\ref{6-1})$.$

To formulate our next result, the following setup is needed. For any integer 
$s$ $\geq $ $0$, we define the $C_{B}^{s}$-norm of an element $\omega $ in $%
\tilde{\Omega}_{m,loc}^{0,\ast }(\Sigma )$ or $\tilde{\Omega}_{m,loc}^{0,\pm
}(\Sigma )$ (see Definition \ref{d-6.8-5}) as follows.

Fix an integer $m\geq 0$ and a partition of unity $\varphi _{j}$ (see $i)$
below Notation \ref{n-6.1}). Writing $\omega (x)=\Sigma _{j}\varphi
_{j}(x)\omega (x)$ $=$ $\Sigma _{j}\varphi _{j}(x)w^{m}h_{I_{q}}(z,\bar{z},w,%
\bar{w})d\bar{z}^{I_{q}}$ where $z$ $\in $ $V_{j}$ and $w$ $\in $ $%
(-\varepsilon _{j},\varepsilon _{j})\times \mathbb{R}^{+}$ are local
coordinates of $x,$ for an integer $s$ $\geq $ $0$ we define ($w$ $=$ $%
|w|e^{i\phi })$%
\begin{equation}
||\omega ||_{C_{B}^{s}}:=\sum_{j}\sum_{k=0}^{s}\sup_{w}||\varphi _{j}(\cdot
,\phi )h_{I_{q}}(\cdot ,\bar{\cdot},w,\bar{w})d\bar{z}%
^{I_{q}}||_{C^{k}(V_{j})}\text{ for }\omega \in \tilde{\Omega}%
_{m,loc}^{0,\ast }(\Sigma )  \label{CBs-norm}
\end{equation}

\noindent which means the supremum over $x$ in the domain, of all partial
derivatives in $z\in V_{j}$ up to order $s.$ Compare the $\parallel \cdot
\parallel _{s}^{\prime }$-norm given in (\ref{Als}). Similarly for an
element of the form (cf. Proposition \ref{l-6-0a} or (\ref{6-19.5})) 
\begin{equation}
K(x,y)=\sum_{j}w^{m}k^{j}(x,y)\bar{\eta}^{m}  \label{6.7-5}
\end{equation}%
\noindent where $x$ $=$ $(z,w),$ $y$ $=$ $(\zeta ,\eta ),$ assuming $%
k^{j}(x,y)$ are $C^{\infty }$-smooth we define the $C_{B}^{s}$-norm of $K$
as follows: 
\begin{equation}
||K(\cdot ,\cdot )||_{C_{B}^{s}(\Sigma \times \Sigma
)}:=\sum_{j}\sup_{w,\eta }||k^{j}((\cdot ,w),(\cdot ,\eta
))||_{C^{s}(V_{j}\times V_{j})}  \label{CBs}
\end{equation}

\noindent which means the supremum over all $x,$ $y$ in the domain$,$ of all
partial derivatives of $k^{j}$ in $z$, $\zeta $ $\in $ $V_{j}$ up to order $%
s.$ Note that the $C_{B}^{s}$-norm (\ref{CBs}) depends on the choice of the
expression (\ref{6.7-5}) with \textquotedblleft $m",$ but we do not put on
such dependence whenever no confusion occurs.



\begin{lemma}
\label{l-6-0} With the notation above, $\lim_{t\rightarrow
0+}P_{m,t}^{0}(u)=\pi _{m}u$ (pointwise) for every $u$ $\in $ $\Omega
^{0,\pm }(\Sigma )$ $\cap $ $L^{2,\pm }(\Sigma ,\pi ^{\ast }\mathcal{E}%
_{M},G_{a,m}).$ Moreover, for every integer $s\geq 0$ it holds that $%
P_{m,t}^{0}(u)$ $\rightarrow $ $\pi _{m}u$ as $t\rightarrow 0$ in the norm $%
||\cdot ||_{C_{B}^{s}}$ (hence in $L^{2}$) for $u\in \Omega ^{0,\pm }(\Sigma
)\cap L^{2,\pm }(\Sigma ,\pi ^{\ast }\mathcal{E}_{M},G_{a,m}).$ In
particular $P_{m,t}^{0}(u)$ $\in $ $\tilde{\Omega}_{m,loc}^{0,\pm }(\Sigma
). $
\end{lemma}

\begin{remark}
\label{6-4-1} For $u_{t},$ $u_{0}$ $\in $ $\tilde{\Omega}_{m,loc}^{0,\pm
}(\Sigma )$ assume $u_{t}$ $\rightarrow $ $u_{0}$ in $C_{B}^{0}.$ Then it is
easy to see that $u_{t}$ $\rightarrow $ $u_{0}$ pointwise on $\Sigma $ and
in $L^{2,\pm }(\Sigma )$ (in view of Remark \ref{3-r}). In particular, if $u$
$\in $ $\tilde{\Omega}_{m,loc}^{0,\pm }(\Sigma )$ with $||u||_{C_{B}^{0}}$ $%
= $ $0$ then $u=0.$
\end{remark}

\begin{remark}
\label{6.2-5} The presence of the projection in Lemma \ref{l-6-0} (rather
than the identity operator in the usual case) reflects the transversal
feature of $P_{m,t}^{0}.$
\end{remark}

\proof
\textbf{(of Lemma \ref{l-6-0})} For $u$ $\in $ $L^{2,\pm }(\Sigma ,\pi
^{\ast }\mathcal{E}_{M},G_{a,m}),$ $\pi _{m}u$ $\in $ $L_{m}^{2,\pm }(\Sigma
,G_{a,m})$ and by (\ref{6-1.5}) $\pi _{m}u|_{W_{j}}$ $=$ $\eta ^{m}\upsilon
_{j}(\zeta )$ for some $\upsilon _{j}(\zeta )$ $\in $ $L^{2,\pm }(V_{j}),$
we compute%
\begin{eqnarray}
&&\lim_{t\rightarrow 0+}(H_{m,t}^{j}(\pi _{m}u))(x)\overset{(\ref{6-1.5})+(%
\ref{6-1})}{=}\lim_{t\rightarrow 0+}\int_{\Sigma }\{\varphi
_{j}(x)w^{m}K_{t}^{j}(z,\zeta )  \label{6-1c} \\
&&\text{ \ \ \ \ \ \ \ \ \ \ \ \ \ \ \ \ \ \ \ \ \ \ \ \ \ }\eta ^{-m}\tau
_{j}(\zeta )\sigma _{j}(\vartheta )l(y)^{m}\eta ^{m}\upsilon _{j}(\zeta
)\}dv_{\tilde{V}_{j}}(\zeta )dv_{f,m}(\eta )  \notag \\
&&\overset{}{=}\lim_{t\rightarrow 0+}\int_{V_{j}}\{\varphi
_{j}(x)w^{m}K_{t}^{j}(z,\zeta )\tau _{j}(\zeta )\upsilon _{j}(\zeta
)\}dv_{V_{j}}(\zeta )\int_{C_{\varepsilon _{j}}}l(y)^{m}\sigma
_{j}(\vartheta )dv_{f,m}(\eta )  \notag
\end{eqnarray}

\noindent By (\ref{6-1a}) and hence 
\begin{equation}
\int_{C_{\varepsilon _{j}}}l(y)^{m}\sigma _{j}(\vartheta )dv_{f,m}(\eta )=1
\label{6-1d}
\end{equation}

\noindent (cf. (\ref{3-19.5}) with $\tau _{p_{0}}^{\ast }$ dropped as
remarked after (\ref{fibrenv}))$,$ the above simplifies to 
\begin{equation*}
\lim_{t\rightarrow 0+}\int_{V_{j}}\{\varphi _{j}(x)w^{m}K_{t}^{j}(z,\zeta
)\tau _{j}(\zeta )\upsilon _{j}(\zeta )\}dv_{V_{j}}(\zeta ),
\end{equation*}%
\noindent then, since $K_{t}^{j}$ $\rightarrow $ $I$ in the distribution
sense as $t$ $\rightarrow $ $0$ (cf. \cite[Definition 2.15 (4) on p.75 ]{BGV}%
) it further simplifies to $\varphi _{j}(x)\tau _{j}(z)(\pi _{m}u)(z,w)$ by (%
\ref{6-1.5}). From this, (\ref{6-1a1}), $\tau _{j}(z)$ $=$ $1$ on supp $%
\varphi _{j}$ and $\Sigma _{j}\varphi _{j}(x)=1,$ the first assertion of the
lemma follows. It also follows from the expression in the last paragraph of 
\cite[p.85]{BGV} that $P_{m,t}^{0}(u)$ $=$ $\sum_{j}H_{m,t}^{j}(\pi
_{m}u)\in $ $\tilde{\Omega}_{m,loc}^{0,\pm }(\Sigma ).$

By the argument similar to (\ref{6-1c})%
\begin{eqnarray*}
&&P_{m,t}^{0}(u)(z,w)-(\pi _{m}u)(z,w) \\
&=&\sum_{j}w^{m}\varphi _{j}(x)[\int_{V_{j}}K_{t}^{j}(z,\zeta )\tau
_{j}(\zeta )\upsilon _{j}(\zeta )dv_{V_{j}}(\zeta )-\tau _{j}(z)\upsilon
_{j}(z)].
\end{eqnarray*}

\noindent Hence for some constant $C(s)$ $>$ $0,$ as $t$ $\rightarrow $ $0$%
\begin{equation*}
||P_{m,t}^{0}(u)-(\pi _{m}u)||_{C_{B}^{s}}\leq C(s)\sum_{(j,k)\in
I}||K_{t}^{j}(\tau _{j}\upsilon _{j})-\tau _{j}\upsilon
_{j}||_{C^{s}(V_{j})}\rightarrow 0
\end{equation*}

\noindent by directly applying \cite[Theorem 2.20(2)]{BGV}. This proves the
second assertion (convergence in $L^{2}$ by Remark \ref{6-4-1})..

\endproof%

The first task is aimed to express the kernel function $(H_{m,t}^{j}\circ
\pi _{m})(x,y)$ (see the RHS of (\ref{6-1a1})) in a more manageable form.
This is given in Proposition \ref{l-6-0a} below.

To start with, suppose that $(z,w)$ denotes local coordinates around $x$ $%
\in $ $\Sigma .$ Let $\mathbb{C}^{\ast }\circ x$ denote the $\mathbb{C}%
^{\ast }$-orbit of $x.$ We define $\tau _{x}$ $:$ $\eta \in \mathbb{C}^{\ast
}$ $\rightarrow $ $\eta \circ x$ $\in $ $\mathbb{C}^{\ast }\circ x.$

Recall the definition of $L^{2,\ast }(\Sigma ,\pi ^{\ast }\mathcal{E}%
_{M},G_{a,m})$ (resp. $L_{m}^{2,\ast }(\Sigma ,G_{a,m}),$ $\pi _{m})$ in
Remark \ref{N-6-3}. We first express $\pi _{m}(u)(x)$ in the following
proposition:

\begin{proposition}
\label{projm} For the orthogonal projection (recall that $a>\frac{m}{2}\geq
0 $)%
\begin{equation*}
\pi _{m}:L^{2,\pm }(\Sigma ,\pi ^{\ast }\mathcal{E}_{M},G_{a,m})\equiv
L^{2,\pm }(\Sigma ,G_{a,m})\rightarrow L_{m}^{2,\pm }(\Sigma ,G_{a,m})
\end{equation*}%
\noindent it holds that for $u\in L^{2,\pm }(\Sigma ,\pi ^{\ast }\mathcal{E}%
_{M},G_{a,m})$%
\begin{equation}
\pi _{m}(u)(x)=l(x)^{m}\int_{\xi \in \mathbb{C}^{\ast }}\sigma (\xi
)_{x}^{\ast }u(\xi \circ x)(\bar{\xi})^{m}(\tau _{x}^{\ast }dv_{f,m})(\xi )
\label{6-1e}
\end{equation}%
\noindent where $\tau _{x}$ is defined precedingly. Moreover if $u$ is
smooth, $\pi _{m}(u)$ remains smooth. Here $\sigma (\xi )_{x}^{\ast }:$ ($%
\pi ^{\ast }\mathcal{E}_{M})_{\xi \circ x}$ $\rightarrow $ ($\pi ^{\ast }%
\mathcal{E}_{M})_{x}$ denotes the pullback of forms.
\end{proposition}

\begin{remark}
\label{6.3-5} $i)$ Note that since $u$ is $\pi ^{\ast }\mathcal{E}_{M}$%
-valued and the action $\sigma $ leaves $\pi ^{\ast }\mathcal{E}_{M}$
invariant, we can trivialize $\sigma (\xi )_{x}^{\ast }u(\xi \circ x)$ and
view it as coefficient function(s) with respect to a basis of global
sections ($\mathbb{C}^{\ast }$-equivariant) of $\pi ^{\ast }\mathcal{E}_{M}$
along the orbit $\mathbb{C}^{\ast }\circ x.$ We then write $\sigma (\xi
)_{x}^{\ast }u(\xi \circ x)$ as $u(\xi \circ x)$ by abuse of notation in the
proof below. See Footnote$^{6}$ (seated above (\ref{proj1})) for further
details.

$ii)$ The action $\sigma $ preserves the metric $\pi ^{\ast }g_{M}$ endowed
on $\pi ^{\ast }\mathcal{E}_{M}$. There is another metric $G_{a,m}|_{\pi
^{\ast }\mathcal{E}_{M}}$ on $\pi ^{\ast }\mathcal{E}_{M}$ induced from $%
G_{a,m}$ when $\pi ^{\ast }\mathcal{E}_{M}$ is viewed as a subbundle of $%
\Lambda ^{0,\ast }(\Sigma )$ (see Remark \ref{N-6-3}). We have that $\pi
^{\ast }g_{M}$ $=$ $G_{a,m}|_{\pi ^{\ast }\mathcal{E}_{M}}$ (cf. Lemma \ref%
{L-inv} $i)$) and thus $\sigma $ preserves the metric $G_{a,m}|_{\pi ^{\ast }%
\mathcal{E}_{M}}$ too.
\end{remark}

\begin{remark}
\label{6-3-5a} $i)$ $L_{m}^{2,\pm }(\Sigma ,\pi ^{\ast }\mathcal{E}
_{M},G_{a,m})$ is not trivial: Using a local bump function/section (cf.
Remark \ref{6.3-5}) $u$ with value $u(\xi \circ x)$ close to $\xi ^{m}$ and
support near $x$ $=$ $(z,w=1)$ in coordinates such that $\pi _{m}(u)(x)$ $>$ 
$0$ as $l(x)$ and the measure $(\tau _{x}^{\ast }dv_{f,m})(\xi )$ are
positive, yields the desired result.

$ii)$ There exists a fixed $u$ contained in $L^{2,\pm }(\Sigma ,\pi ^{\ast }%
\mathcal{E}_{M},G_{a})$ for any $a>0$ and $a>\frac{m}{2}\geq 0$ such that $%
\pi _{m}(u)$ $=$ $0$ in $L^{2,\pm }(\Sigma ,\pi ^{\ast }\mathcal{E}%
_{M},G_{a})$ (which by Lemma \ref{L-3-5} equals $L^{2,\pm }(\Sigma ,\pi
^{\ast }\mathcal{E}_{M},G_{a,m})$): Let $u$ be a local function/section
supported near $w=1/2$ in a chart $(z,w),$ depending only on $z$, $r=|w|$
such that for $m=0$ the integral (\ref{6-1e}) of $u(z,r)$ over $r$ equals $0$
with respect to $G_{a}$ for any $a>0;$ this is made possible because of (\ref%
{dvfibre1}) where the parameter \textquotedblleft $a$" has the impact on $%
G_{a}$ only outside a neighborhood of $\{w=1/2\}.$ Since the angular
integral of $u$ is seen to vanish for $m>0$ (cf. (\ref{6-1e}) and (\ref%
{7-2-3}))$,$ altogether we obtain the claim.
\end{remark}

\proof
\textbf{(of Proposition \ref{projm})} Recall that $l(x)$ is defined in (\ref%
{lq0}). Locally $l(x)$ $=$ $h(z(x),\bar{z}(x))w(x)\bar{w}(x)$ or $hw\bar{w}$
in short for $x$ $\in $ $(W_{j},$ $(z,w)).$ The orthogonal projection $\pi
_{m}$ is characterized by the following conditions:

\medskip

\noindent $i)$ $\pi _{m}$ is a bounded linear operator on $L^{2,\pm }(\Sigma
,\pi ^{\ast }\mathcal{E}_{M},G_{a,m}),$

\noindent $ii)$ $\pi _{m}(u)\in L_{m}^{2,\pm }(\Sigma ,G_{a,m})$ for $u$ $%
\in $ $L^{2,\pm }(\Sigma ,\pi ^{\ast }\mathcal{E}_{M},G_{a,m}),$

\noindent $iii)$ $\pi _{m}\circ \pi _{m}=\pi _{m}$

\noindent $iv)$ $\pi _{m}^{\ast }=\pi _{m},$ i.e., $\pi _{m}$ is
self-adjoint.

\noindent $v)$ $\pi _{m}(u)=u$ for $u\in L_{m}^{2,\pm }(\Sigma ,G_{a,m}).$

\medskip

Now we claim that $\pi _{m}(u)$ defined by (\ref{6-1e}) satisfies the above
conditions. For $x\in \Sigma \backslash \Sigma _{\text{sing}}$ (see (\ref%
{1-4}) for the definition of $\Sigma _{\text{sing}}),$ in local coordinates $%
(z,w)$ for $x$ and $(z,\eta )$ for $\xi \circ x$ where $\eta $ $=$ $\xi w$
(see Footnote\footnote{%
Strictly speaking, the validity of $\eta $ $=$ $\xi w$ and $z(\xi \circ x)=z$
requires the smallness of the angle-difference between $x$ and $\xi \circ x$
(see (\ref{1-1}) and cases $i)$, $ii)$ after (\ref{6-1g})). But since $%
x\notin \Sigma _{\text{sing}}$, we can assume that $\delta _{j}=\pi $ in (%
\ref{1-0}) for suitable holomorphic coordinates $(z,w)$ at $x$. Hence $\eta $
$=$ $\xi w$ and $z(\xi \circ x)=z$ are valid for all $\xi \in \mathbb{C}%
^{\ast }$ by (\ref{1-1}).})$,$ we write (see Footnote\footnote{%
In view of Remark \ref{6.3-5} $ii)$: for $x\in \Sigma \backslash \Sigma _{%
\text{sing}}$, choose a local basis $\eta ^{I_{q}}$ for $(0,q)$-forms near $%
\pi (x) $ in $M,$ set $\tilde{\eta}^{I_{q}}$ $:=$ $\pi ^{\ast }\eta ^{I_{q}}$
and compute $\sigma (\xi )^{\ast }\tilde{\eta}^{I_{q}}$ $=$ $\tilde{\eta}%
^{I_{q}} $ since $\pi \circ \sigma (\xi )$ $=$ $\pi $ on $\Sigma ,$ i.e. $%
\tilde{\eta}^{I_{q}}$ is invariant under the action $\sigma .$ 
Write a global section $u$ of $\pi ^{\ast }\mathcal{E}_{M}$ along $\mathbb{C}%
^{\ast }\circ x$ as $u$ $=$ $u_{I_{q}}\tilde{\eta}^{I_{q}}$ (summing over $%
I_{q}$ and $q$). Compute 
\begin{eqnarray*}
\sigma (\xi )_{x}^{\ast }u(\xi \circ x) &=&\sigma (\xi )_{x}^{\ast
}u_{I_{q}}(\xi \circ x)\sigma (\xi )_{x}^{\ast }\tilde{\eta}^{I_{q}}(\xi
\circ x) \\
&=&u_{I_{q}}(\xi \circ x)\tilde{\eta}^{I_{q}}(x)
\end{eqnarray*}%
\noindent by the $\sigma $-invariance of $\tilde{\eta}^{I_{q}}$. We can
therefore identify $\sigma (\xi )_{x}^{\ast }u(\xi \circ x)$ with the
coefficient functions $u_{I_{q}}(\xi \circ x)$ still denoted as $u(\xi \circ
x)$ or $u(z,\xi w)$ $=$ $u(z,\eta )$ in coordinates. For later use note that 
$\sigma (\xi )_{x}^{\ast }u(\xi \circ x) = (\sigma (\xi )^{\ast }u)(x)$.})%
\begin{eqnarray}
&&\pi _{m}(u)(x)\overset{l=hw\bar{w}}{=}h^{m}|w|^{2m}\int_{\eta \in \mathbb{C%
}^{\ast }}u(z,\eta )(\bar{\eta})^{m}(\bar{w})^{-m}(\tau _{(z,1)}^{\ast
}dv_{f,m})(\eta )  \label{proj1} \\
&\overset{|w|^{2m}(\bar{w})^{-m}=w^{m}}{=}&w^{m}h^{m}(z,\bar{z})\int_{\eta
\in \mathbb{C}^{\ast }}u(z,\eta )(\bar{\eta})^{m}(\tau _{(z,1)}^{\ast
}dv_{f,m})(\eta ).  \notag
\end{eqnarray}

\noindent By Remark \ref{6.3-5} $i)$ the norm of the coefficient
(vector-valued) function $u(z,\eta )$ or $u(y)$ below should include the
norms of $(0,q)$-forms; these norms of forms are nevertheless nonzero
constants along $\mathbb{C}^{\ast }\circ x$ by the $\mathbb{C}^{\ast }$%
-invariance of the metric $\pi ^{\ast }g_{M}$ (by Remark \ref{6.3-5} $ii))$.
For simplicity we drop these constant norms henceforth (they can be
uniformly bounded as $M$ $=$ the space of $\mathbb{C}^{\ast }$-orbits is
compact).

To prove $i)$, upon applying the Cauchy-Schwarz inequality to the integral
on the RHS of (\ref{proj1}) we get \ 
\begin{eqnarray}
&&|\int_{\eta \in \mathbb{C}^{\ast }}u(z,\eta )(\bar{\eta})^{m}(\tau
_{(z,1)}^{\ast }dv_{f,m})(\eta )|^{2}  \label{proj2} \\
&\leq &\int_{\eta \in \mathbb{C}^{\ast }}|u(z,\eta )|^{2}(\tau
_{(z,1)}^{\ast }dv_{f,m})(\eta )\int_{\eta \in \mathbb{C}^{\ast }}|\eta
|^{2m}(\tau _{(z,1)}^{\ast }dv_{f,m})(\eta ).  \notag
\end{eqnarray}

\noindent We now estimate, by (\ref{proj1}) and (\ref{proj2}),%
\begin{eqnarray}
|\pi _{m}(u)(x)|^{2} &\leq &h^{m}|w|^{2m}\int_{\eta \in \mathbb{C}^{\ast
}}|u(z,\eta )|^{2}(\tau _{(z,1)}^{\ast }dv_{f,m})(\eta )  \label{proj3} \\
&&\cdot \int_{\eta \in \mathbb{C}^{\ast }}h^{m}|\eta |^{2m}(\tau
_{(z,1)}^{\ast }dv_{f,m})(\eta )  \notag \\
&=&h^{m}|w|^{2m}\int_{\eta \in \mathbb{C}^{\ast }}|u(z,\eta )|^{2}(\tau
_{(z,1)}^{\ast }dv_{f,m})(\eta )  \notag
\end{eqnarray}

\noindent because%
\begin{equation}
\int_{\eta \in \mathbb{C}^{\ast }}h^{m}|\eta |^{2m}(\tau _{(z,1)}^{\ast
}dv_{f,m})(\eta )=1  \label{proj4}
\end{equation}

\noindent (cf. (\ref{fibrenv})) in (\ref{proj3}).

Before proceeding further let us set up the following. Let $\bar{N}_{j}$ ($%
j=1,$ $2,$ $3,$ $\cdot \cdot \cdot )$ be an open neighborhood of $M_{\text{%
sing}}$ $:=$ $\pi (\Sigma _{\text{sing}})$ $\subset $ $M$ by $\pi :$ $\Sigma 
$ $\rightarrow $ $\Sigma /\sigma $ $=$ $M.$ Write $N_{j}$ $=$ $\pi ^{-1}(%
\bar{N}_{j}$). We assume $\cap _{j}\bar{N}_{j}$ $=$ $M_{\text{sing}}$ and $%
\bar{N}_{1}$ $\supset $ $\bar{N}_{2}$ $\supset $ $\cdot \cdot \cdot .$ By
the compactness of $M$ we have that for $j^{\prime }$ in a finite index set $%
W_{j^{\prime }},$ $V_{j^{\prime }}$ (see Notation \ref{n-6.1}) can be formed
to satisfy $\cup _{j^{\prime }}\pi (V_{j^{\prime }})$ $(=$ $\pi
(W_{j^{\prime }}))$ $\supset $ $M\backslash \bar{N}_{j}$ with $\pi
(V_{j^{\prime }})$ $\subset $ $M\backslash M_{\text{sing}}.$

Assume from now on that the support of $u$ is contained in ($\Sigma
\backslash N_{j}$ $\subset )$ $\tilde{W}_{j}$ $:=$ $\cup _{j^{\prime
}}W_{j^{\prime }}$ ($\subset $ $\Sigma \backslash \Sigma _{\text{sing}})$
with $\tilde{V}_{j}$ $:=$ $\cup _{j^{\prime }}V_{j^{\prime }},$ meaning that 
$u=0$ $a.e.$ on a neighborhood of $\Sigma \backslash \tilde{W}_{j}.$ We will
come back to the general case later. Then (via Footnote$^{5}$ as just
mentioned$)$%
\begin{eqnarray}
&&\int_{\tilde{W}_{j}}|\pi _{m}(u)(x)|^{2}dv_{\Sigma ,m}(x)\overset{(\ref%
{proj3})+(\ref{volume})}{\leq }\int_{z\in \tilde{V}_{j}}{\big \{}\int_{w\in 
\mathbb{C}^{\ast }}h^{m}|w|^{2m}\tau _{(z,1)}^{\ast }dv_{f,m}(w)
\label{proj5} \\
&&\int_{\eta \in \mathbb{C}^{\ast }}|u(z,\eta )|^{2}(\tau _{(z,1)}^{\ast
}dv_{f,m})(\eta ){\big \}}dv(z)\overset{(\ref{proj4})}{\leq }\int_{\tilde{W}%
_{j}}|u(y)|^{2}dv_{\Sigma ,m}(y)  \notag
\end{eqnarray}

\noindent where $y$ $=$ $(z,\eta )$ $\in $ $\Sigma \backslash \Sigma _{\text{%
sing}}$ and the last term equals $\int_{\Sigma }|u|^{2}dv_{\Sigma ,m}$ by
the support condition, implies that 
\begin{equation*}
||\pi _{m}(u)||_{L^{2,\pm }(\Sigma ,G_{a,m})}\leq ||u||_{L^{2,\pm }(\Sigma
,G_{a,m})}
\end{equation*}%
\noindent since (\ref{proj5}) holds for every $j,$ $\cup _{j}\tilde{W}_{j}$ $%
=$ $\Sigma \backslash \Sigma _{\text{sing}}$ and $\Sigma _{\text{sing}}$ is
of measure 0. Condition $i)$ follows.

For condition $ii)$ observe that by $\lambda \circ (\xi \circ x)$ $=$ $%
(\lambda \xi )\circ x,$ we have $\tau _{\xi \circ x}^{\ast }dv_{f,m}(\lambda
)=\tau _{x}^{\ast }dv_{f,m}(\lambda \xi )$ so that 
\begin{eqnarray}
&&\pi _{m}(u)(\xi \circ x)=l(\xi \circ x)^{m}\int_{\lambda \in \mathbb{C}%
^{\ast }}u((\lambda \xi )\circ x)\bar{\lambda}^{m}(\tau _{x}^{\ast
}dv_{f,m})(\lambda \xi )  \label{proj6} \\
&&\overset{\eta =\lambda \xi }{=}l(\xi \circ x)^{m}(\bar{\xi}%
)^{-m}\int_{\eta \in \mathbb{C}^{\ast }}u(\eta \circ x)\bar{\eta}^{m}(\tau
_{x}^{\ast }dv_{f,m})(\eta ).  \notag
\end{eqnarray}

Condition $ii)$ amounts to proving the equality%
\begin{equation}
(\sigma (\xi )^{\ast }\pi _{m}(u))(x)=\pi _{m}(u))(\xi \circ x)=\xi ^{m}\pi
_{m}(u)(x)\text{ \ }a.e.  \label{proj6-1}
\end{equation}%
\noindent This immediately follows by applying $l(\xi \circ x)^{m}$ $=$ $%
|\xi |^{2m}l(x)^{m}$ (see (\ref{7-1-1})) to (\ref{proj6}) and using (\ref%
{proj1}) (for $x$ $\in $ $\Sigma \backslash \Sigma _{\text{sing}})$.

Next we compute%
\begin{eqnarray*}
&&\pi _{m}(\pi _{m}(u))(x)\overset{(\ref{6-1e})}{=}l(x)^{m}\int_{\xi \in 
\mathbb{C}^{\ast }}\pi _{m}(u)(\xi \circ x)\bar{\xi}^{m}(\tau _{x}^{\ast
}dv_{f,m})(\xi ) \\
&\overset{(\ref{proj6})}{=}&l(x)^{m}\int_{\xi \in \mathbb{C}^{\ast }}l(\xi
\circ x)^{m}(\tau _{x}^{\ast }dv_{f,m})(\xi )\cdot \int_{\eta \in \mathbb{C}%
^{\ast }}u(\eta \circ x)\bar{\eta}^{m}(\tau _{x}^{\ast }dv_{f,m})(\eta ) \\
&\overset{(\ref{6-1e})}{=}&\pi _{m}(u)(x)\int_{\xi \in \mathbb{C}^{\ast
}}l(\xi \circ x)^{m}(\tau _{x}^{\ast }dv_{f,m})(\xi )\overset{(\ref{fibrenv})%
}{=}\pi _{m}(u)(x),
\end{eqnarray*}

\noindent which gives condition $iii).$

To show $iv)$ let us compare $(\pi _{m}(u),v)$ and $(u,\pi _{m}(v))$ for any 
$v$ $\in $ $L^{2,\pm }(\Sigma ,G_{a,m}).$ In local holomorphic coordinates $%
(z,w)$ for $x$ $\in $ $\Sigma \backslash \Sigma _{\text{sing}}$ and $(z,\eta
)$ for $\xi \circ x$ $\in $ $\Sigma \backslash \Sigma _{\text{sing}}$ where $%
\eta $ $=$ $\xi w,$ we have%
\begin{eqnarray}
&&\pi _{m}(u)(x)\overline{v(x)}dv_{\Sigma ,m}(x)\overset{(\ref{proj1})+(\ref%
{volume})}{=}w^{m}h^{m}(z,\bar{z}){\big \{}\int_{\eta \in \mathbb{C}^{\ast
}}u(z,\eta )(\bar{\eta})^{m}  \label{proj7} \\
&&\text{ \ \ \ \ \ \ \ \ \ \ \ \ \ \ \ \ \ }(\tau _{(z,1)}^{\ast
}dv_{f,m})(\eta ){\big \}}\overline{v(z,w)}dv(z)(\tau _{(z,1)}^{\ast
}dv_{f,m})(w).  \notag
\end{eqnarray}

\noindent On the other hand, we have, for $y=\xi \circ x$ $\in $ $\Sigma
\backslash \Sigma _{\text{sing}}$ (in fact for $y$ $\in $ $\tilde{W}_{j}),$%
\begin{eqnarray*}
&&u(y)\overline{\pi _{m}(v)(y)}=u(y)l(y)^{m}\int_{\xi ^{-1}\in \mathbb{C}%
^{\ast }}\overline{v(\xi ^{-1}\circ y)}\xi ^{-m}\tau _{y}^{\ast
}dv_{f,m}(\xi ^{-1}) \\
&=&u(z,\eta )h^{m}(z,\bar{z})|\eta |^{2m}\int_{w\in \mathbb{C}^{\ast }}%
\overline{v(z,w)}w^{m}\eta ^{-m}\tau _{(z,1)}^{\ast }dv_{f,m}(w) \\
&=&u(z,\eta )h^{m}(z,\bar{z})\bar{\eta}^{m}\int_{w\in \mathbb{C}^{\ast }}%
\overline{v(z,w)}w^{m}\tau _{(z,1)}^{\ast }dv_{f,m}(w)
\end{eqnarray*}

\noindent Inserting $dv_{\Sigma ,m}(y)$ gives%
\begin{eqnarray}
&&u(y)\overline{\pi _{m}(v)(y)}dv_{\Sigma ,m}(y)=u(z,\eta )h^{m}(z,\bar{z})%
\bar{\eta}^{m}  \label{proj8} \\
&&\text{ \ \ \ \ \ \ \ }\cdot {\LARGE \{}\int_{w\in \mathbb{C}^{\ast }}%
\overline{v(z,w)}w^{m}\tau _{(z,1)}^{\ast }dv_{f,m}(w){\LARGE \}}(\tau
_{(z,1)}^{\ast }dv_{f,m})(\eta )dv(z).  \notag
\end{eqnarray}

\noindent It is slightly tedious to write out both expressions $%
\int_{x=(z,w)}$ (RHS of (\ref{proj7})) and $\int_{y=(z,\eta )}($RHS of (\ref%
{proj8})); once it is done, it is not difficult to see (without computing
out any integration) that they are exactly the same. Hence $\pi _{m}^{\ast
}=\pi _{m}$ proving condition $iv)$.

To show $v),$ since supp $u$ $\subset $ $\tilde{W}_{j}$ $\subset $ $\Sigma
\backslash N_{l}$ for $l>>j$ by assumption, we have that $u(\xi \circ x)$ $=$
$\xi ^{m}u(x)$ $a.e.$ if $u$ is further assumed to be in $L_{m}^{2,\pm
}(\Sigma ,G_{a,m});$ see the Footnote$^{6}$ above for explanation of $u(\xi
\circ x).$ For $x$ $\in $ $\Sigma \backslash \Sigma _{\text{sing}}$%
\begin{eqnarray*}
\pi _{m}(u)(x) &=&l(x)^{m}\int_{\xi \in \mathbb{C}^{\ast }}\xi ^{m}u(x)(\bar{%
\xi})^{m}(\tau _{x}^{\ast }dv_{f,m})(\xi ) \\
&=&u(x)\int_{\xi \in \mathbb{C}^{\ast }}(\tau _{x}^{\ast }l)^{m}(\xi )(\tau
_{x}^{\ast }dv_{f,m})(\xi )=u(x)
\end{eqnarray*}

\noindent where (\ref{7-1-1}) and $l^{m}(\xi \circ x)$ $=$ $(\tau _{x}^{\ast
}l)^{m}(\xi )$ via definitions (resp. (\ref{fibrenv})) are used for the
second (resp. third) equality. Condition $v)$ follows$.$

For the general case let us first prove the following:%
\begin{eqnarray}
&&\text{For the sake of clarity let us denote by }\tilde{\pi}_{m}(u)\text{
the RHS of }(\ref{6-1e}).  \label{proj9} \\
&&\text{Let }u_{j}\rightarrow u\text{ in }L^{2,\pm }(\Sigma ,G_{a,m})\text{.
Then }\tilde{\pi}_{m}(u_{j})\rightarrow \tilde{\pi}_{m}(u)\text{ in }L^{2}%
\text{.}  \notag
\end{eqnarray}

\noindent \textbf{Proof of (\ref{proj9}):} The $L^{2}$-difference $||\tilde{%
\pi}_{m}(u_{j})-\tilde{\pi}_{m}(u)||_{L^{2}}^{2}$ is bounded by%
\begin{equation*}
I_{j}:=\int_{\Sigma }l(x)^{2m}\left[ \int_{\xi \in \mathbb{C}^{\ast
}}|u_{j}(\xi \circ x)-u(\xi \circ x)||\xi |^{m}(\tau _{x}^{\ast
}dv_{f,m})(\xi )\right] ^{2}dv_{\Sigma ,m}(x).
\end{equation*}

\noindent Applying the Cauchy-Schwarz inequality to the integral over $\xi $
gives%
\begin{eqnarray}
I_{j} &\leq &\int_{\Sigma }{\big \{}l(x)^{m}\int_{\xi \in \mathbb{C}^{\ast
}}|u_{j}(\xi \circ x)-u(\xi \circ x)|^{2}(\tau _{x}^{\ast }dv_{f,m})(\xi )
\label{proj10} \\
&&\text{ \ \ \ \ \ \ }\cdot \int_{\xi \in \mathbb{C}^{\ast }}l(x)^{m}|\xi
|^{2m}(\tau _{x}^{\ast }dv_{f,m})(\xi ){\big \}}dv_{\Sigma ,m}(x).  \notag
\end{eqnarray}

\noindent Observe that $l(x)^{m}|\xi |^{2m}$ $=$ $l(\xi \circ x)^{m}$ by (%
\ref{7-1-1}) and hence $\int_{\xi \in \mathbb{C}^{\ast }}l(x)^{m}|\xi
|^{2m}(\tau _{x}^{\ast }dv_{f,m})(\xi )$ $=$ $1$ by (\ref{fibrenv}). It is
possible to simplify (\ref{proj10}) further: Substituting this into (\ref%
{proj10}) and making a change of variables $y=\xi \circ x$ over $\Sigma
\backslash \Sigma _{\text{sing}}$ (to ensure bijectivity), we can write the
RHS of (\ref{proj10}) as ($\Sigma _{\text{sing}}$ being of measure $0$)%
\begin{equation}
\int_{y\in \Sigma \backslash \Sigma _{\text{sing}}}\{|u_{j}(y)-u(y)|^{2}%
\int_{\xi \in \mathbb{C}^{\ast }}l(\xi ^{-1}\circ y)^{m}(\tau _{y}^{\ast
}dv_{f,m})(\xi ^{-1})\}dv_{\Sigma ,m}(y)  \label{proj12}
\end{equation}

\noindent because $(\tau _{x}^{\ast }dv_{f,m})(\xi )dv_{\Sigma ,m}(x)$ in (%
\ref{proj10}) equals ($\tilde{\sigma}^{\ast }[(\tau _{y}^{\ast
}dv_{f,m})(\xi ^{-1})dv_{\Sigma ,m}(y)]$)$(\xi ,x)$ (where $\tilde{\sigma}%
:(\xi ,x)$ $\rightarrow $ $(\xi ^{-1},y$ $=$ $\xi \circ x))$ as proved in
Lemma \ref{A} $iii).$ Again $\int_{\xi \in \mathbb{C}^{\ast }}l(\xi
^{-1}\circ y)^{m}(\tau _{y}^{\ast }dv_{f,m})(\xi ^{-1})$ $=$ $1$ for any $y$
by (\ref{fibrenv}) reduces (\ref{proj12}) to $\int_{\Sigma
}|u_{j}(y)-u(y)|^{2}dv_{\Sigma ,m}$ which tends to zero as $j\rightarrow
\infty $ by assumption$.$ We have shown (\ref{proj9}).

We can now finish the proof for the general $\tilde{u}$. Let $\bar{\chi}_{j%
\text{ }}$be cut-off functions with support in $M\backslash M_{\text{sing}}$
(cf. \cite[p.37]{Cara})$,$ which are $\equiv $ 1 on $\pi (\tilde{W}_{j})$ ($%
= $ $\pi (\tilde{V}_{j}))$ $\supset $ $M\backslash \bar{N}_{j}$ and $\bar{%
\chi}_{j}$ $\rightarrow $ $1$ on $M\backslash M_{\text{sing}}$ via $\cap _{j}%
\bar{N}_{j}$ $=$ $M_{\text{sing}}$ as already assumed$.$ Write $\chi _{j}$ $%
= $ $\pi ^{\ast }\bar{\chi}_{j}$ and consider $\chi _{j}\tilde{u}.$ Then one
has the following: $a)$ $\chi _{j}\tilde{u}$ $\rightarrow $ $\tilde{u}$ in $%
L^{2,\pm }(\Sigma ,G_{a,m});$ $b)$ if $\tilde{u}$ $\in $ $L_{m}^{2,\pm
}(\Sigma ,G_{a,m}),$ since $\chi _{j}$ is $\mathbb{C}^{\ast }$-invariant by
construction $\chi _{j}\tilde{u}(\xi \circ x)$ $=$ $\xi ^{m}(\chi _{j}\tilde{%
u})(x),$ $i.e.$ $\chi _{j}\tilde{u}$ $\in $ $L_{m}^{2,\pm }(\Sigma
,G_{a,m}). $ Using $a),$ $b)$ and applying (\ref{proj9}) with the previously
proved special case $u$ $\equiv $ $\chi _{j}\tilde{u}$ one checks that $%
\tilde{\pi}_{m}$ easily satisfies the above conditions $i)$ to $v)$ as the
original orthogonal projection $\pi _{m}$ does, giving that $\tilde{\pi}_{m}$
$=$ $\pi _{m}$ as desired.

For the last statement of the lemma, the continuity property of $\pi _{m}$
in $C_{B}^{s}$-norm (see (\ref{CBs-norm})) is postponed to Lemma \ref{l-bdp}.

\endproof%




Define $d\mu _{y,m}(\xi ),$ a 2-form in $\xi $ for $(\xi ,y)$ $\in $ $%
\mathbb{C}^{\ast }\times \Sigma $ (i.e., it is of the form $f(y,\bar{y},\xi ,%
\bar{\xi})d\xi \wedge d\bar{\xi})$ by%
\begin{equation}
d\mu _{y,m}(\xi ):=l(\xi ^{-1}\circ y)^{m}dv_{\Sigma ,m}(\xi ^{-1}\circ
y)\wedge dv_{f,m}(y)/dv_{\Sigma ,m}(y)  \label{6-1e1}
\end{equation}%
\noindent where $dv_{\Sigma ,m}(\xi ^{-1}\circ y)$ denotes the pullback on $%
\mathbb{C}^{\ast }\times \Sigma $ of $dv_{\Sigma ,m}$ (see (\ref{volume}))
by the map $(\xi ,$ $y)$ $\rightarrow $ $\xi ^{-1}\circ y.$ Notice that $%
dv_{\Sigma ,m}(\xi ^{-1}\circ y)$ in (\ref{6-1e1}) contains a 2-form in $\xi 
$ wedging a "horizontal" $2n-2$ form in $y,$ and one sees that it is this $2$%
-form that survives in the RHS of (\ref{6-1e1}).

To express $(H_{m,t}^{j}\circ \pi _{m})(x,y)$ as mentioned earlier, via
Proposition \ref{projm} above we view $(H_{m,t}^{j}\circ \pi _{m})(x,y)$ and 
$H_{m,t}^{j}(x,\xi ^{-1}\circ y)\circ \sigma (\xi )_{\xi ^{-1}\circ y}^{\ast
}$ as linear transformations from $(\pi ^{\ast }\mathcal{E}_{M})_{y}$ to $%
(\pi ^{\ast }\mathcal{E}_{M})_{x}$, where $\sigma (\xi )_{\xi ^{-1}\circ
y}^{\ast }$ $:$ $(\pi ^{\ast }\mathcal{E}_{M})_{y}$ $\rightarrow $ $(\pi
^{\ast }\mathcal{E}_{M})_{\xi ^{-1}\circ y}$ denotes the pullback of forms.

\begin{proposition}
\label{l-6-0a} Let the notation $\sigma (\xi )_{\xi ^{-1}\circ y}^{\ast }$
be as above.

\noindent $i)$ The kernel function for $P_{m,t}^{0.j}$ $\equiv $ $%
H_{m,t}^{j}\circ \pi _{m}$ is given by%
\begin{equation}
(H_{m,t}^{j}\circ \pi _{m})(x,y)=\int_{\xi \in \mathbb{C}^{\ast
}}(H_{m,t}^{j}(x,\xi ^{-1}\circ y)\circ \sigma (\xi )_{\xi ^{-1}\circ
y}^{\ast })(\bar{\xi})^{m}d\mu _{y,m}(\xi ),\text{ }\forall (x,y)\in \Sigma
\times \Sigma .  \label{6-1f}
\end{equation}%
\noindent $ii)$ $(H_{m,t}^{j}\circ \pi _{m})(x,y)$ is of the form, with $%
x=(z,w),$ $y=(\zeta ,\eta )$ as in (\ref{6-1}),%
\begin{equation}
(H_{m,t}^{j}\circ \pi _{m})(x,y)=w^{m}p_{m,t}^{0,j}(x,y)\bar{\eta}^{m}
\label{6-19.5}
\end{equation}%
\noindent for some smooth and $C_{B}^{s}$-bounded ($t$-dependent) linear
transformation $p_{m,t}^{0,j}(x,y)$ from $(\pi ^{\ast }\mathcal{E}_{M})_{y}$
to $(\pi ^{\ast }\mathcal{E}_{M})_{x}$ (see (\ref{CBs}) for the definition
of $C_{B}^{s}$-norm).

\noindent $iii)$ $(H_{m,t}^{j}\circ \pi _{m})(x,y)$ is $L^{2}$ in two
variables $(x,y)$ $\in $ $\Sigma \times \Sigma $ with respect to the metric $%
G_{a,m}\times G_{a,m}.$ Compare Lemma \ref{6-10-1} and its proof for similar
consequences.
\end{proposition}

\proof
Compute $(H_{m,t}^{j}\circ \pi _{m})u$ as follows:%
\begin{eqnarray}
&&[(H_{m,t}^{j}\circ \pi _{m})u](x)=H_{m,t}^{j}(\pi _{m}(u))(x)
\label{6-1f1} \\
&=&\int_{p\in \Sigma }H_{m,t}^{j}(x,p)\pi _{m}(u)(p)dv_{\Sigma ,m}(p)  \notag
\\
&\overset{(\ref{6-1e})}{=}&\int_{p\in \Sigma }H_{m,t}^{j}(x,p)\int_{\xi \in 
\mathbb{C}^{\ast }}\sigma (\xi )_{p}^{\ast }u(\xi \circ p)(\bar{\xi}%
)^{m}(\tau _{p}^{\ast }dv_{f,m})(\xi )l(p)^{m}dv_{\Sigma ,m}(p)  \notag
\end{eqnarray}

\noindent We now make a change of variables: $(p,\xi )\rightarrow (y,\xi )$
by $y$ $=$ $\xi \circ p$ (on $\Sigma \backslash \Sigma _{\text{sing}}$)$.$
After a careful examination, (\ref{6-1f1}) equals ($\Sigma _{\text{sing}}$
being of measure zero)%
\begin{equation*}
\int_{y\in \Sigma }\int_{\xi \in \mathbb{C}^{\ast }}H_{m,t}^{j}(x,p)\circ
(\sigma (\xi )_{p}^{\ast }u(y))(\bar{\xi})^{m}dv_{f,m}(y)l(\xi ^{-1}\circ
y)^{m}\wedge dv_{\Sigma ,m}(\xi ^{-1}\circ y)\}
\end{equation*}%
\noindent By rearranging terms, the above becomes (by $(\ref{6-1e1}))$
(noting that $p$ $=$ $\xi ^{-1}\circ y$ and $H_{m,t}^{j}(x,p)\circ \sigma
(\xi )_{p}^{\ast }$ acts on $u(y))$%
\begin{equation}
\int_{y\in \Sigma }\left[ \int_{\xi \in \mathbb{C}^{\ast
}}(H_{m,t}^{j}(x,\xi ^{-1}\circ y)\circ \sigma (\xi )_{\xi ^{-1}\circ
y}^{\ast })(\bar{\xi})^{m}d\mu _{y,m}(\xi )\right] u(y)dv_{\Sigma ,m}(y)
\label{6-11.5}
\end{equation}%
\noindent Now (\ref{6-1f}) follows from the integrand $\left[ ...\right] $
in (\ref{6-11.5}).

The proof of (\ref{6-19.5}) will be postponed to Section 7 after studying
the properties of $d\mu _{y,m}(\xi )$. See the paragraph prior to Remark \ref%
{7-8-5} (containing (\ref{kmt}) through (\ref{7K-25})). The assertion $iii)$
follows from the $C_{B}^{s}$-bound ($C_{B}^{0}$ enough) and Remark \ref{3-r}.

\endproof%

Next we want to compute $(\partial _{t}P_{m,t}^{0})u+P_{m,t}^{0}\tilde{%
\square}_{m}^{c\pm }\pi _{m}u$ (see (\ref{DiracLap}) for $\tilde{\square}%
_{m}^{c\pm }$, (\ref{6-1a1}) for $P_{m,t}^{0}$). For use in the following
lemma, as similar to $\tilde{\square}_{m}^{c\pm }$ we let the local operator
(which is a notation used interchangeably with $\square _{U_{j},m}^{c\pm }$
in (\ref{DiracLap}))%
\begin{equation}
\square _{z,m}^{c\pm }:=D_{z,m}^{c\mp }D_{z,m}^{c\pm }  \label{6-11a}
\end{equation}%
\noindent denote the Laplacian of (the standard $m$-th) $spin^{c}$ Dirac
operator%
\begin{equation}
D_{z,m}^{c}=D_{z,m}+A_{z,m}^{c}=\bar{\partial}_{z,m}+\bar{\partial}%
_{z,m}^{\ast }+A_{z,m}^{c}  \label{6-11.75}
\end{equation}%
\noindent on $V_{j}$ $\subset $ $\mathbb{C}^{n-1}$ (recall (\ref{dbarzm})
for $\bar{\partial}_{z,m}^{\ast }$)$.$ $P_{m,t}^{0}$ satisfies the following
adjoint type heat equation asymptotically; compare \cite[Lemma 5.12]{CHT}
and see Remark \ref{6-9-1} for a significant difference.

\begin{lemma}
\label{l-adj} It holds that%
\begin{equation}
(\partial _{t}P_{m,t}^{0})u+P_{m,t}^{0}(\tilde{\square}_{m}^{c\pm }\circ \pi
_{m})u=R_{t}u\text{ \ for }u\in \Omega ^{0,\pm }(\Sigma )\cap L^{2,\pm
}(\Sigma ,\pi ^{\ast }\mathcal{E}_{M},G_{a,m})  \label{Adj-eq}
\end{equation}%
where $R_{t}$ $:$ $\Omega ^{0,\pm }(\Sigma )\rightarrow \Omega ^{0,\pm
}(\Sigma )$ is an operator with distribution kernel $R_{t}(x,y)$ (= $%
R(t,x,y) $ $\in $ $C^{\infty }(\mathbb{R}^{+}\times \Sigma \times \Sigma ,$ $%
T^{\ast 0,+}\Sigma \otimes (T^{\ast 0,+}\Sigma )^{\ast }))$ of the form (\ref%
{6.7-5}) satisfying that for every $s\in \mathbb{N}\cup \{0\},$ there exist $%
\varepsilon _{0}>0,$ $C_{s}>0$ independent of $t$ such that%
\begin{equation}
||R(t,x,y)||_{C_{B}^{s}(\Sigma \times \Sigma )}\leq C_{s}e^{-\frac{%
\varepsilon _{0}}{t}}\text{ \ for }t\in \mathbb{R}^{+}.  \label{Adj-0}
\end{equation}
\end{lemma}

\proof
Write $\pi _{m}u|_{W_{j}}(\zeta ,\eta )=\eta ^{m}\upsilon _{j}(\zeta )$ (see
(\ref{6-1.5}); we drop subscripts \textquotedblleft $j"$ on the
coordinates). Recalling that $\tilde{\square}_{m,y}^{c\pm }$ acts on $y$ (=$%
(\zeta ,\eta ))$ while $\square _{\zeta ,m}^{c\pm }$ acts on $\zeta $ via
Proposition \ref{BoxDU}$,$ we compute ($P_{m,t}^{0}$ is the sum of $%
H_{m,t}^{j}\circ \pi _{m}$ over $j$ (\ref{6-1}))%
\begin{eqnarray}
&&\partial _{t}H_{m,t}^{j}(\pi _{m}u)+H_{m,t}^{j}\tilde{\square}_{m,y}^{c\pm
}(\pi _{m}u)  \label{Adj-1} \\
&=&\int_{W_{j}}\{\varphi _{j}(x)w^{m}\partial _{t}K_{t}^{j}(z,\zeta
)\upsilon _{j}(\zeta )\tau _{j}(\zeta )  \notag \\
&&+\varphi _{j}(x)w^{m}K_{t}^{j}(z,\zeta )(\square _{\zeta ,m}^{c\pm
}\upsilon _{j}(\zeta ))\tau _{j}(\zeta )\}\sigma _{j}(\vartheta
)l(y)^{m}dv_{\Sigma ,m}(y)  \notag \\
&=&\int_{V_{j}}\varphi _{j}(x)w^{m}\{\partial _{t}K_{t}^{j}(z,\zeta
)\upsilon _{j}(\zeta )+K_{t}^{j}(z,\zeta )(\square _{\zeta ,m}^{c\pm
}\upsilon _{j}(\zeta ))\}\tau _{j}(\zeta )dv(\zeta )  \notag
\end{eqnarray}

\noindent in which we have used $\int_{C_{\varepsilon _{j}}}\sigma
_{j}(\vartheta )l(y)^{m}dv_{f,m}(\eta )$ $=$ $1$ by (\ref{6-1d}). Noting
that 
\begin{equation*}
\int_{V_{j}}\{\partial _{t}K_{t}^{j}(z,\zeta )+K_{t}^{j}(z,\zeta )\square
_{\zeta ,m}^{c\pm }\}(\upsilon _{j}(\zeta )\tau _{j}(\zeta ))dv(\zeta )=0
\end{equation*}%
\noindent since $K_{t}^{j}$ is the Dirichlet heat kernel (\ref{6-1'})$,$ we
reduce the RHS of (\ref{Adj-1}) to%
\begin{eqnarray}
&&\int_{V_{j}}\varphi _{j}(x)S_{j}(t,x,\zeta )\bar{\eta}^{-m}\upsilon
_{j}(\zeta )dv(\zeta )  \label{Adj-2-0} \\
&&\overset{(\ref{6-1d})}{=}\int_{W_{j}}\varphi _{j}(x)S_{j}(t,x,\zeta )\bar{%
\eta}^{-m}\eta ^{-m}(\pi _{m}u)(y)\sigma _{j}(\vartheta )l(y)^{m}dv_{\Sigma
,m}(y)  \notag
\end{eqnarray}

\noindent where%
\begin{equation}
S_{j}(t,x,\zeta )(\eta ^{-m}(\pi _{m}u)(y))=w^{m}K_{t}^{j}(z,\zeta )\bar{\eta%
}^{m}[\tau _{j}(\zeta ),\square _{\zeta ,m}^{c\pm }](\upsilon _{j}(\zeta )).
\label{Adj-2a-0}
\end{equation}

\noindent Note that $\tau _{j}(\zeta )$ = $1$ for $\zeta $ in some small
neighborhood of the $z$-part of supp $\varphi _{j}.$ It follows that $%
\varphi _{j}(x)S_{j}(t,x,\zeta )$ $=$ $0$ if $(z(x),\zeta )$ is in some
small neighborhood of $(z,z)$ (due to $[\tau ,\square ]$ $=$ $0$ there$).$
The idea is that the singular part of $\varphi _{j}S_{j}$ originally caused
by $K_{t}^{j}$ along the diagonal ($\sim \frac{1}{t^{n-1}}$) is now
dismissed.

Note that $S_{j}(t,x,\zeta )$ in (\ref{Adj-2-0}) is a differential operator
acting on $\upsilon _{j}(\zeta ).$ We can convert it into a kernel function
via an integration by parts as follows: by (\ref{Adj-2a-0}), the
self-adjointness of $\square _{\zeta ,m}^{c\pm }$ with $\tau _{j}(\zeta )$
being of compact support, and $\int_{C_{\varepsilon _{j}}}\sigma
_{j}(\vartheta )l(y)^{m}dv_{f,m}(\eta )$ $=$ $1$ \ 
\begin{eqnarray}
&&\text{LHS of }(\ref{Adj-2-0})=\int_{V_{j}}\varphi
_{j}(x)w^{m}K_{t}^{j}(z,\zeta )[\tau _{j}(\zeta ),\square _{\zeta ,m}^{c\pm
}](\upsilon _{j}(\zeta ))dv(\zeta )  \label{Adj-2} \\
&=&\int_{V_{j}}\varphi _{j}(x)w^{m}\{\square _{\zeta ,m}^{c\pm
}(K_{t}^{j}(z,\zeta )\tau _{j}(\zeta ))-(\square _{\zeta ,m}^{c\pm
}K_{t}^{j}(z,\zeta ))\tau _{j}(\zeta )\}\upsilon _{j}(\zeta )dv(\zeta ) 
\notag \\
&=&\int_{V_{j}}\varphi _{j}(x)\hat{S}_{j}(t,x,y)\bar{\eta}^{-m}\upsilon
_{j}(\zeta )dv(\zeta )  \notag \\
&=&\int_{W_{j}}\varphi _{j}(x)\hat{S}_{j}(t,x,y)\bar{\eta}^{-m}\eta
^{-m}(\pi _{m}u)(y)\sigma _{j}(\vartheta )l(y)^{m}dv_{\Sigma ,m}(y)  \notag
\end{eqnarray}

\noindent where (the $[\square _{\zeta ,m}^{c\pm },\tau _{j}(\zeta )]$ below
as a first order differential operator acts on the $\zeta $-variable of $%
K_{t}^{j}(z,\zeta ))$%
\begin{equation}
\hat{S}_{j}(t,x,y):=w^{m}[\square _{\zeta ,m}^{c\pm },\tau _{j}(\zeta
)](K_{t}^{j}(z,\zeta ))\bar{\eta}^{m}.  \label{Adj-2a}
\end{equation}

Therefore using an analogous argument of \cite[(5.47)]{CHT} for (\ref{Adj-2}%
) (it is essential that the Gaussian factor $\exp (-\frac{\tilde{d}%
_{M}^{2}(z,\zeta )}{4t})$ encoded in $K_{t}^{j},$ see (\ref{OAE}), absorbs
the singular part $\frac{1}{t^{n-1}}$ if $(z,\zeta )$ is off the diagonal
where $[\tau ,\square ]$ $=$ $0$ as mentioned above), we conclude that for
every $s\in \mathbb{N}\cup \{0\},$ there exist $\varepsilon >0,$ $C_{s}>0$
independent of $t$ such that (noting that $\bar{\eta}^{-m}\eta ^{-m}l(y)^{m}$
$=$ $h^{m}(\zeta ,\bar{\zeta})$ by (\ref{hm}) below)%
\begin{equation}
||\varphi _{j}(x)\hat{S}_{j}(t,x,y)h^{m}(\zeta ,\bar{\zeta})\sigma
_{j}(\vartheta )||_{C_{B}^{s}(\Sigma \times \Sigma )}\leq C_{s}e^{-\frac{%
\varepsilon }{t}}\text{ for }t\in \mathbb{R}^{+}.  \label{Adj-3}
\end{equation}

\noindent From (\ref{Adj-1}) and (\ref{Adj-2}) it follows that 
\begin{equation}
(\partial _{t}P_{m,t}^{0})u+P_{m,t}^{0}(\tilde{\square}_{m}^{c\pm }\circ \pi
_{m})u=\hat{R}_{t}(\pi _{m}u)  \label{Adj-4}
\end{equation}%
\noindent where%
\begin{equation}
\hat{R}_{t}(x,y):=\sum_{j}\varphi _{j}(x)\hat{S}_{j}(t,x,y)\bar{\eta}%
^{-m}\eta ^{-m}\sigma _{j}(\vartheta )l(y)^{m}.  \label{Adj-4a}
\end{equation}%
\noindent So (\ref{Adj-eq}) holds for $R_{t}$ $:=$ $\hat{R}_{t}\circ \pi
_{m} $ in view of (\ref{Adj-4}) and (an analogue of (\ref{6-1f})) 
\begin{equation}
R(t,x,y)\text{ }(=R_{t}(x,y))=\int_{\xi \in \mathbb{C}^{\ast }}(\hat{R}%
_{t}(x,\xi ^{-1}\circ y)\circ \sigma (\xi )_{\xi ^{-1}\circ y}^{\ast })(\bar{%
\xi})^{m}d\mu _{y,m}(\xi ).  \label{Adj-4b}
\end{equation}

\noindent Now that 
\begin{equation}
\bar{\eta}^{-m}\eta ^{-m}l(y)^{m}=h^{m}(\zeta ,\bar{\zeta})  \label{hm}
\end{equation}%
\noindent from (\ref{Adj-4a}) is bounded in $\zeta $ $\in $ $M=\Sigma /%
\mathbb{C}^{\ast }$ by Lemma \ref{A} $iv)$, which is compact, $\bar{\eta}%
^{m}(\xi ^{-1}\circ y)\bar{\xi}^{m}$ (from (\ref{Adj-2a}) and (\ref{Adj-4b}%
)) is also bounded for $\xi \in \mathbb{C}^{\ast }$ (using (\ref{1-1}) for $%
\rho $ $=$ $|\xi |)$ and $\int_{\xi \in \mathbb{C}^{\ast }}d\mu _{y,m}(\xi )$
$=$ $1$ by Corollary \ref{7.6-5}, we conclude (\ref{Adj-0}) for $s=0$ by (%
\ref{Adj-3}) via (\ref{Adj-4a}), (\ref{Adj-4b}). For $s>0$ note the
definition of $C_{B}^{s}$-norms that concern mainly the $z$ or $\zeta $%
-variables (i.e. the horizontal ones). Observe that 
\begin{equation}
\bar{\eta}^{m}(\xi ^{-1}\circ y)\bar{\xi}^{m}=\chi (y)\bar{\eta}^{m}(y)
\label{6.40-5a}
\end{equation}

\noindent where%
\begin{eqnarray*}
\chi (y) &=&\frac{\bar{\eta}^{m}(e^{-i\gamma }\circ y)}{\bar{\eta}^{m}(y)}%
|\xi |^{-m}\bar{\xi}^{m}\text{ (}\xi =|\xi |e^{i\gamma }) \\
&=&e^{im\theta (e^{-i\gamma }\circ y)}e^{-im\gamma }e^{-im\theta (y)}\text{ (%
}\eta (y)=|\eta (y)|e^{-i\theta (y)}).
\end{eqnarray*}

\noindent Since $C_{B}^{s}$-norms involve no $\bar{\eta}^{m}(y)$ (see (\ref%
{CBs-norm})), for the $C_{B}^{s}$-norms of $R(t,x,y)$ it suffices to prove,
with (\ref{Adj-3}), the smoothness of $h^{m}$ and $\chi (y)$ by using the
compactness of their $\zeta $-domains. We discuss this smoothness issue as
follows. For the smoothness of $\theta (e^{-i\gamma }\circ y)$ (in $\chi (y)$%
) see a similar claim on $C_{B}^{s}$-norm in (\ref{6-19.5}) (as it involves
the similar expression $(\xi ^{-1}\circ y)$ by (\ref{6-1f})), whose proof is
placed after (\ref{HQDIIa}). In this proof $\theta (e^{-i\gamma }\circ y)$
actually lies in (\ref{7K-21}), whose smoothness is proved in the last
paragraph after (\ref{7K-25}) involving the smoothness of $\alpha _{k}.$ The
point is that the \textquotedblleft large$"$ angle action (see Case $ii)$
after (\ref{6-1g})) due to the local freeness of the $\mathbb{C}^{\ast }$%
-action makes the treatment less direct and leads to the consideration of $%
\alpha _{k}$. Lastly, to deal with $C_{B}^{s}$-norms for $y$ in $d\mu
_{y,m}(\xi ),$ simply notice the expression (\ref{7-2-2}), (\ref{7-2-3})
with $x$ replaced by $y$, which can be simplified further by the $\mathbb{C}%
^{\ast }$-invariance of $dv_{M}$ (Lemma \ref{A} $ii)$) . Our proof is now
completed.

\endproof%

\begin{remark}
\label{6-9-1} The difference between the $S_{j}$ of (\ref{Adj-2a-0}) and the
corresponding term in the same notation $S_{j}$ in \cite[(5.46) and p.84]%
{CHT} arises from that between $H_{m,t}^{j}(x,y)$ of (\ref{6-1}) and $%
H_{j}(t,x,y)$ of \cite[(5.38)]{CHT}. This difference is partly due to the
fact in Lemma \ref{L-7iso} that holds only at the Hilbert space level, while
some similar identifications in the proof of \cite[Proposition 5.1, p.73]%
{CHT} hold at the pointwise level. That these identifications, adapted to
their own contexts, are only similar in nature leads to the different
results. Compare Remark \ref{3-62.5} for different effects caused by the
metrics here and \cite{CHT}. Note that the $S_{j}(t,x,w)$ of \cite[p.84]{CHT}
is actually a (first-order) differential operator, so the presentation in 
\cite[p.84]{CHT}, especially the derivation of \cite[(5.47)]{CHT}, need be
fixed in a similar way as the part from $S_{j}$ of (\ref{Adj-2a-0}) to $\hat{%
S}_{j}$ of (\ref{Adj-2a}).
\end{remark}

If one tries to directly show that $P_{m,t}^{0}$ is an approximate heat
kernel one may encounter a difficulty which we refer to \cite[the paragraph
after (1.81), p.36]{CHT} for a similar situation and explanation. Suffice it
to say that it becomes easier to show that its adjoint $P_{m,t}^{0\ast }$ is
an approximate heat kernel. To express $P_{m,t}^{0\ast }$ in terms of $%
H_{m,t}^{j\ast }$ we need the following lemma that $H_{m,t}^{j}$ (hence $%
H_{m,t}^{j\ast })$ is a bounded linear operator on $L^{2,\pm }(\Sigma ,\pi
^{\ast }\mathcal{E}_{M},G_{a,m}),$ defined through the kernel function $%
H_{m,t}^{j}(x,y)$ (see (\ref{6-1})).

\begin{lemma}
\label{6-10-1} The kernel function $H_{m,t}^{j}(x,y)$ as in (\ref{6-1})
defines a bounded linear operator on $L^{2,\pm }(\Sigma ,\pi ^{\ast }%
\mathcal{E}_{M},G_{a,m}).$ In fact $H_{m,t}^{j}(x,y)$ is $L^{2}$ in two
variables ($x,y)$ $\in $ $\Sigma \times \Sigma $ with respect to the metric $%
G_{a,m}$ $\times $ $G_{a,m}.$ The kernel function $H_{m,t}^{j\ast }(x,y)$
defined by $H_{m,t}^{j\ast }(x,y)$ $:=$ $\overline{H_{m,t}^{j}(y,x)}$
represents the Hilbert space adjoint operator $H_{m,t}^{j\ast }$. (A
superscript \textquotedblleft $t"$ meant as transpose may be placed on $%
\overline{H_{m,t}^{j}(y,x)},$ but we omit it.)
\end{lemma}

\begin{proof}
A \textquotedblleft local version$"$ of this (with cut-off functions in (\ref%
{6-1}) removed) is seen in Remark \ref{R-6-3}. Now in (\ref{6-1}) observe
that $\eta ^{-m}l(y)^{m}$ $=$ $h(\zeta ,\bar{\zeta})^{m}\bar{\eta}^{m}.$
From this and Remark \ref{3-r} the assertion on the $L^{2}$-condition
follows. This yields the first and third assertions (see \cite[Theorem 2,
p.13 ]{Lang}).
\end{proof}

Using Lemma \ref{6-10-1} we have now 
\begin{eqnarray}
&&P_{m,t}^{0\ast }=\sum_{j}(H_{m,t}^{j}\circ \pi _{m})^{\ast }=\sum_{j}\pi
_{m}^{\ast }\circ H_{m,t}^{j\ast }  \label{Adj-5z} \\
&=&\sum_{j}\pi _{m}\circ H_{m,t}^{j\ast }:L^{2,\pm }(\Sigma ,\pi ^{\ast }%
\mathcal{E}_{M},G_{a,m})\rightarrow L^{2,\pm }(\Sigma ,\pi ^{\ast }\mathcal{E%
}_{M},G_{a,m}).  \notag
\end{eqnarray}

It is worth noting that while the action of $P_{m,t}^{0}$ may not preserve
the space $\Omega _{m}^{0,\pm }(\Sigma )$ (cf. Lemma \ref{l-6-0}), the image
of $P_{m,t}^{0\ast }$ on $\Omega ^{0,\pm }(\Sigma )\cap L^{2,\pm }(\Sigma
,\pi ^{\ast }\mathcal{E}_{M},G_{a,m})$ is nonetheless seated in $\Omega
_{m}^{0,\pm }(\Sigma )$ because of $\pi _{m}$. By taking the adjoints $%
P_{m,t}^{0\ast },$ $R_{t}^{\ast }$ of $P_{m,t}^{0},$ $R_{t}$ respectively,
we are going to prove the following. Denote by $\Omega _{c}^{0,\pm }(\Sigma
) $ $\subset $ $\Omega ^{0,\pm }(\Sigma )$ the set of those elements (smooth
sections of $\pi ^{\ast }\mathcal{E}_{M}$) of compact support in $\Sigma .$
Note that $\Omega _{c}^{0,\pm }(\Sigma )$ is dense in $L^{2,\pm }(\Sigma
,\pi ^{\ast }\mathcal{E}_{M},G_{a,m})$ (see Remark \ref{N-6-3}).

\begin{theorem}
\label{t-adjeq} In the preceding notation, we have

\noindent $i)$%
\begin{equation}
\lim_{t\rightarrow 0^{+}}P_{m,t}^{0\ast }u=\pi _{m}u\text{ in }||\cdot
||_{C_{B}^{s}}\text{ \ for }u\in \Omega _{c}^{0,\pm }(\Sigma ).
\label{Adj-5}
\end{equation}%
$ii)$%
\begin{eqnarray}
&&\frac{\partial P_{m,t}^{0\ast }}{\partial t}u+\tilde{\square}_{m}^{c\pm
}P_{m,t}^{0\ast }u=R_{t}^{\ast }u\text{ \ for }u\in L^{2,\pm }(\Sigma ,\pi
^{\ast }\mathcal{E}_{M},G_{a,m})\cap \Omega ^{0,\pm }(\Sigma )
\label{Adj-5-0} \\
&&\text{ \ \ \ \ \ \ \ \ \ \ \ \ \ \ \ \ \ \ \ \ \ \ \ \ \ \ \ \ \ \ \ \ \ \
\ \ \ \ \ }(\text{not necessarily in }\Omega _{c}^{0,\pm }(\Sigma )).  \notag
\end{eqnarray}%
$iii)$ The distribution kernel $R^{\ast }(t,x,y)$ of $R_{t}^{\ast }$ is
given by $\overline{R(t,y,x)}$ (cf. (\ref{Adj-0})); it satisfies a similar
estimate as $R_{t}$ in Lemma \ref{l-adj}.
\end{theorem}

Although $\lim_{t\rightarrow 0^{+}}P_{m,t}^{0}u=\pi _{m}u$ in $||\cdot
||_{C_{B}^{s}}$ for $u$ $\in $ $\Omega ^{0,\pm }(\Sigma )\cap L^{2,\pm
}(\Sigma ,\pi ^{\ast }\mathcal{E}_{M},G_{a,m})$ holds by Lemma \ref{l-6-0},
it is a bit surprising that the corresponding statement for $P_{m,t}^{0\ast
} $ is not obviously true as it might appear to be at first sight, unless $u$
is further restricted$.$ This difference is essentially due to the
noncompactness of $\Sigma .$ Compared to the compact CR case \cite[Theorem
5.13, p.84]{CHT} the following proof is less trivial in that certain
elements \textquotedblleft $\alpha _{l}\in S^{1}"$ related to the local
orbifold group will be introduced (and also used in some later parts of the
paper).

For the proof of Theorem \ref{t-adjeq} we start by proving the following
lemma. Writing $\tilde{\Omega}_{m,loc}^{0,\ast }(\Sigma )$ $:=$ $\oplus _{q}%
\tilde{\Omega}_{m,loc}^{0,q}(\Sigma )$ (see Definition \ref{d-6.8-5}) we
want the boundedness of $\pi _{m}:\tilde{\Omega}_{m,loc}^{0,\ast }(\Sigma )$ 
$\rightarrow $ $\tilde{\Omega}_{m}^{0,\ast }(\Sigma )$ with respect to $%
C_{B}^{s}$-norm. In the CR case \cite{CHT} these $C_{B}^{s}$-norms were
neither needed nor were they pursued because the compactness of the total
space there simplifies the picture. The technicalities here lie in the local
freeness of the $\mathbb{C}^{\ast }$-action.

\begin{lemma}
\label{l-bdp} With the notation above, $||\pi _{m}(u)||_{C_{B}^{s}}\leq
C_{1}||u||_{C_{B}^{s}}$ for every $u\in \tilde{\Omega}_{m,loc}^{0,\ast
}(\Sigma ),$ $s\in \mathbb{N\cup }\{0\}.$
\end{lemma}

\proof
Write $u(x)$ $=$ $\sum_{j}\varphi _{j}(x)u(x)=\sum_{j}\varphi
_{j}(x)w^{m}v_{j}(x)$ where $(z,w)$ are local coordinates (with $j$%
-dependence suppressed), $v_{j}$ is a $(0,q)$-form and $\varphi _{j}$ is as
in item $i)$ after Notation \ref{n-6.1}. Substituting it into (\ref{6-1e}),
we obtain (omitting \textquotedblleft $\circ "$ in $\xi \circ x)$%
\begin{equation}
\pi _{m}(u)(x)=l(x)^{m}\sum_{j}\int_{\xi \in \mathbb{C}^{\ast }}\sigma (\xi
)_{x}^{\ast }(\varphi _{j}(\xi x)w(\xi x)^{m}v_{j}(\xi x))\bar{\xi}^{m}\tau
_{x}^{\ast }dv_{f,m}(\xi ).  \label{7K-5}
\end{equation}

\noindent We first want to extract $w(x)^{m}$ out of $w(\xi x)^{m}$ in (\ref%
{7K-5}). Given a local chart $W_{j}$ $:=$ $V_{j}$ $\times $ $(-\varepsilon
_{j},\varepsilon _{j})$ $\times $ $\mathbb{R}^{+}$ as in Notation \ref{n-6.1}
and suppose $x\in W_{j},$ there exist at most finitely many $\alpha _{l}$'s $%
\in $ $S^{1}$ $(\subset $ $\mathbb{C}^{\ast })$ dependent on $x$ such that $%
\alpha _{l}x$ $\in $ $W_{j}$ with $\phi (\alpha _{l}x)$ $=$ $0$ (recall $w$ $%
=$ $re^{i\phi })$ because if $\alpha _{l}$ is such an element then any $%
\alpha _{l}^{\prime }$ $\in $ $S^{1}$ near $\alpha _{l}$ will give $\phi
(\alpha _{l}^{\prime }x)$ $\neq $ $0$ by small angle action (see Case $i)$
after (\ref{6-1g})). Let $\alpha _{0}$ $=$ $e^{-i\phi (x)}$ so that $%
w(\alpha _{0}x)$ $=$ $\alpha _{0}w(x)$ $(=$ $r(x)).$ Denote $J_{l}$ $:=$ $%
\{\xi $ $\in $ $\mathbb{C}^{\ast }$ $:$ $-\varepsilon _{j}$ $<$ $\arg (\xi
\alpha _{l}^{-1})$ $<$ $\varepsilon _{j}\}.$ It follows that for $\xi $ $\in 
$ $J_{l},$ $-\varepsilon _{j}$ $<$ $\phi (\xi x)$ $=$ $\arg (\xi \alpha
_{l}^{-1})+\phi (\alpha _{l}x)$ $=$ $\arg (\xi \alpha _{l}^{-1})$ $<$ $%
\varepsilon _{j}$ hence that $w(\xi x)=\xi \alpha _{l}^{-1}w(\alpha _{l}x)$
since $\xi \alpha _{l}^{-1}$ is of small angle (Case $i)$ after (\ref{6-1g}%
)). Using $\phi (\alpha _{l}x)$ $=$ $\phi (\alpha _{0}x)$ $=$ $0$ together
with Lemma \ref{A} $i)$ gives $w(\alpha _{l}x)$ $=$ $w(\alpha _{0}x)$ $=$ $%
\alpha _{0}w(x).$ In sum, for $\xi $ $\in $ $J_{l}$ we have%
\begin{equation}
w(\xi x)=\xi \alpha _{l}^{-1}\alpha _{0}w(x)  \label{7K-5a}
\end{equation}%
\noindent extracting $w(x)$ from $w(\xi x),$ as mentioned earlier$.$ Remark
that $\alpha _{l}^{-1}\alpha _{0}$ is independent of $x$ and $\{\alpha
_{l}^{-1}\alpha _{0}\}_{l}$ forms a group (see Proposition \ref{gk}), but we
need not use this fact here.

Write $z_{l}$ $=$ $z(\xi x)$ for $\xi $ $\in $ $J_{l}.$ Note that $z_{l}$ is
independent of $\xi $ in $J_{l}$ (similarly implied by the small angle
condition as above)$.$ For any fixed $j$ in (\ref{7K-5}), by the cut-off $%
\varphi _{j}$ we can now reduce the integral in (\ref{7K-5}) to%
\begin{eqnarray}
&&\sum_{l}\int_{\xi \in J_{l}}\xi ^{m}\alpha _{l}^{-m}\alpha
_{0}^{m}w(x)^{m}\varphi _{j}(z_{l},\phi (\xi x))\sigma (\xi )_{x}^{\ast
}v_{j}(z_{l},w(\xi x))\bar{\xi}^{m}\tau _{x}^{\ast }dv_{f,m}(\xi )
\label{7K-6} \\
&=&\sum_{l}\alpha _{l}^{-m}\alpha _{0}^{m}w(x)^{m}\sigma (\alpha
_{l})_{x}^{\ast }\int_{\xi \in J_{l}}\varphi _{j}(z_{l},\phi (\xi
x))v_{j}(z_{l},w(\xi x))|\xi |^{2m}\tau _{x}^{\ast }dv_{f,m}(\xi ).  \notag
\end{eqnarray}

\noindent Observe that 
\begin{equation}
l(x)^{m}\int_{\xi \in J_{l}}|\xi |^{2m}\tau _{x}^{\ast }dv_{f,m}(\xi )%
\overset{(\ref{7-1-1})}{=}\int_{\xi \in J_{l}}l(\xi x)^{m}\tau _{x}^{\ast
}dv_{f,m}(\xi )\overset{(\ref{3-43.5})}{=}\frac{\varepsilon _{j}}{\pi }.
\label{6.40-5}
\end{equation}

\noindent The fact that $\alpha _{l}$ smoothly depends on $x$ is proved in
remarks after (\ref{7K-25}). This leads, via $u$ $=$ $\sum_{j}\varphi
_{j}w^{m}v_{j}$ and the definition of $C_{B}^{s}$-norm, to (noting that $%
|\alpha _{l}|=1,$ $w=|w|e^{i\phi }$ and $\varphi _{j}$ is bounded)%
\begin{equation}
||\pi _{m}(u)||_{C_{B}^{s}}\leq
C_{0}\sum_{l}\sum_{j}\sum_{k=0}^{s}\sup_{w}||\varphi _{j}(\cdot ,\phi
)v_{j}(\cdot ,w)||_{C^{k}(V_{j})}=C_{0}\cdot (\#\text{ of }l)\text{ }%
||u||_{C_{B}^{s}}.  \label{6.40-7}
\end{equation}

\endproof%

Using the above lemma we continue with the proof of Theorem \ref{t-adjeq}.

\begin{proof}
\textbf{(of Theorem \ref{t-adjeq})} 
Let us first prove that $P_{m,t}^{0\ast }u$ $=$ $\pi _{m}(H_{m,t}^{j\ast }u)$
converges in the $||\cdot ||_{C_{B}^{s}}$-norm as $t\rightarrow 0.$ The
kernel function of $H_{m,t}^{j\ast }$ reads as (see (\ref{6-1}) with $x$ $=$ 
$(z,w),$ $w$ $=$ $|w|e^{i\phi },$ $y$ $=$ $(\zeta ,\eta ),$ $\eta $ $=$ $%
|\eta |e^{i\vartheta })$%
\begin{eqnarray}
H_{m,t}^{j\ast }(x,y) &=&\overline{w^{-m}(x)\tau _{j}(z(x))\sigma _{j}(\phi
(x))l(x)^{m}}K_{t}^{j\ast }(z(x),\zeta )\overline{\varphi _{j}(y)\eta ^{m}(y)%
}  \label{Adj-5-1} \\
&=&w^{m}\tau _{j}(z)\sigma _{j}(\phi )h^{m}(z,\bar{z})K_{t}^{j\ast }(z,\zeta
)\varphi _{j}(y)\overline{\eta ^{m}(y)}.  \notag
\end{eqnarray}%
\noindent For given $t$ it is easy to verify that $H_{m,t}^{j\ast }u$ $\in $ 
$\tilde{\Omega}_{m,loc}^{0,\ast }(\Sigma )$. The convergence of $%
H_{m,t}^{j\ast }u$ in the $||\cdot ||_{C_{B}^{s}}$-norm as $t\rightarrow 0$
follows simply because $K_{t}^{j\ast }$ $=$ $K_{t}^{j}$ involved in the
above expression of $H_{m,t}^{j\ast }$ has the property that $K_{t}^{j}(u)$ $%
\rightarrow $ $u$ in $C^{s}$-norm (with respect to $z)$ uniformly in
\textquotedblleft parameter$"\eta $ since $u$ is assumed to be of compact
support (see the bottom paragraph for the variable change to absorb the
singular part $\frac{1}{t^{n-1}}$ (of $K_{t}^{j}$) in \cite[p.85]{BGV}).
This together with Lemma \ref{l-bdp} yields $P_{m,t}^{0\ast }u$ $\rightarrow 
$ (say) $P_{m,0}^{0\ast }u$ $\in $ $\tilde{\Omega}_{m,loc}^{0,\ast }(\Sigma
) $ in the $||\cdot ||_{C_{B}^{s}}$-norm as $t\rightarrow 0$. To prove that $%
P_{m,0}^{0\ast }u$ $=$ $\pi _{m}(u)$ consider for $v\in \Omega _{c}^{0,\ast
}(\Sigma )$%
\begin{eqnarray}
(P_{m,t}^{0\ast }u-\pi _{m}u,v)_{L^{2}} &=&(P_{m,t}^{0\ast }u-P_{m,0}^{0\ast
}u+P_{m,0}^{0\ast }u-\pi _{m}u,v)_{L^{2}}  \label{Adj-5-2} \\
&=&(P_{m,t}^{0\ast }u-P_{m,0}^{0\ast }u,v)_{L^{2}}+(P_{m,0}^{0\ast }u-\pi
_{m}u,v)_{L^{2}}  \notag \\
&\rightarrow &0\text{ + }(P_{m,0}^{0\ast }u-\pi _{m}u,v)_{L^{2}}\text{ as }%
t\rightarrow 0.  \notag
\end{eqnarray}

\noindent On the other hand, 
\begin{eqnarray}
((P_{m,t}^{0\ast }-\pi _{m})u,v)_{L^{2}} &=&(u,(P_{m,t}^{0}-\pi _{m}^{\ast
})v)_{L^{2}}  \label{Adj-5-3} \\
&=&(u,(P_{m,t}^{0}-\pi _{m})v)_{L^{2}}\rightarrow 0\text{ as }t\rightarrow 0
\notag
\end{eqnarray}

\noindent by Lemma \ref{l-6-0}. It follows from (\ref{Adj-5-2}) and (\ref%
{Adj-5-3}) that the limit $P_{m,0}^{0\ast }u=\pi _{m}u.$ We have shown (\ref%
{Adj-5}).

To show (\ref{Adj-5-0}) we take the adjoint of (\ref{Adj-eq}) in Lemma \ref%
{l-adj} to get%
\begin{equation*}
\partial _{t}P_{m,t}^{0\ast }+[P_{m,t}^{0}(\tilde{\square}_{m}^{c\pm }\circ
\pi _{m})]^{\ast }=R_{t}^{\ast }\text{ on }\Omega ^{0,\pm }(\Sigma )\cap
L^{2,\pm }(\Sigma ,\pi ^{\ast }\mathcal{E}_{M},G_{a,m})
\end{equation*}

\noindent where ($\tilde{\square}_{m}^{c\pm }$ is formally self-adjoint on $%
\Omega _{m}^{0,\pm }(\Sigma )$ $\supset $ $P_{m,t}^{0\ast }(\Omega ^{0,\pm
}(\Sigma )\cap L^{2,\pm }(\Sigma ));$ see Lemma \ref{l-4-2}) 
\begin{eqnarray*}
\lbrack P_{m,t}^{0}(\tilde{\square}_{m}^{c\pm }\circ \pi _{m})]^{\ast }
&=&\pi _{m}^{\ast }\circ (\tilde{\square}_{m}^{c\pm })^{\ast }\circ
P_{m,t}^{0\ast } \\
&=&\pi _{m}\circ \tilde{\square}_{m}^{c\pm }\circ P_{m,t}^{0\ast }=\tilde{%
\square}_{m}^{c\pm }\circ P_{m,t}^{0\ast },
\end{eqnarray*}

\noindent giving (\ref{Adj-5-0}). The last claim of the theorem for $R^{\ast
}(t,x,y)$ follows from a similar estimate for $R(t,x,y)$ of (\ref{Adj-0}),
so that $R^{\ast }(t,x,y)$ is in $L^{2}(\Sigma \times \Sigma )$ and
represents the kernel function of the adjoint operator $R_{t}^{\ast }$%
(compare Lemma \ref{6-10-1}).
\end{proof}

Before solving our heat equation let us show that $P_{m,t}^{0\ast }$ is a
bounded linear operator on $\tilde{\Omega}_{m,loc}^{0,\ast }(\Sigma )$ in
the $||\cdot ||_{C_{B}^{s}}$-norm uniformly for $t$ near $0.$ Recall the
notation $\tilde{\Omega}_{m,loc}^{0,\ast }(\Sigma )$ in Definition \ref%
{d-6.8-5}. For $u$ $\in $ $\tilde{\Omega}_{m,loc}^{0,\ast }(\Sigma )$ recall
the definition of $||u||_{C_{B}^{s}}$ in (\ref{CBs-norm}).

\begin{proposition}
\label{l-ubK0} Given $\delta $ $>$ $0$ and $s\in \mathbb{N\cup }\{0\},$
there exists a constant $C_{s}$ independent of $\delta $ and $t$ (but may
depend on $s)$ such that (see (\ref{CBs-norm}) for $||\cdot ||_{C_{B}^{s}}$)%
\begin{equation}
||P_{m,t}^{0\ast }(u)||_{C_{B}^{s}}\leq C_{s}||u||_{C_{B}^{s}}  \label{Adj-6}
\end{equation}%
\noindent for $0\leq t\leq \delta $ and $u\in \tilde{\Omega}_{m,loc}^{0,\ast
}(\Sigma )$ $\subset $ $L^{2,\ast }(\Sigma ,\pi ^{\ast }\mathcal{E}%
_{M},G_{a,m}).$
\end{proposition}

\proof
By (\ref{6-1a1}) $P_{m,t}^{0\ast }$ $=$ $\sum_{j}\pi _{m}^{\ast }\circ
H_{m,t}^{j\ast }$ $=$ $\sum_{j}\pi _{m}\circ H_{m,t}^{j\ast }.$ We claim
that $H_{m,t}^{j\ast }$ satisfies the following estimate: there exist $%
\delta _{0}$ $>$ $0$ and $C_{s}^{\prime }$ $>$ $0$ (independent of $\delta
_{0}$ and $t)$ such that for $0$ $<$ $t$ $<$ $\delta _{0}$%
\begin{equation}
||H_{m,t}^{j\ast }u||_{C_{B}^{s}}\leq C_{s}^{\prime }||u||_{C_{B}^{s}}.
\label{Adj-7}
\end{equation}%
\noindent for $u$ $\in $ $\tilde{\Omega}_{m,loc}^{0,,\ast }(\Sigma ).$
Writing $x=(z,w)$ $(w=|w|e^{i\phi }),$ $y=(\zeta ,\eta )$ in local
coordinates and $u(y)$ $=$ $\sum_{j}\varphi _{j}(y)u(y)$ $=$ $%
\sum_{j}\varphi _{j}(y)\eta ^{m}v_{j}(y),$ we have by (\ref{Adj-5-1}) 
\begin{eqnarray}
&&(H_{m,t}^{j\ast }u)(x)=\int_{\Sigma }H_{m,t}^{j\ast }(x,y)u(y)dv_{\Sigma
,m}(y)  \label{Kstar} \\
&=&w^{m}h^{m}(z,\bar{z})\sum_{j}\tau _{j}(z)\sigma _{j}(\phi )\int_{\Sigma
}K_{t}^{j\ast }(z,\zeta )\varphi _{j}(y)\bar{\eta}^{m}\eta
^{m}v_{j}(y)dv_{\Sigma ,m}(y).  \notag
\end{eqnarray}

\noindent Observe that for each $t$ $>$ $0$ $\func{Im}H_{m,t}^{j\ast }|_{%
\tilde{\Omega}_{m,loc}^{0,\ast }(\Sigma )}$ $\subset $ $\tilde{\Omega}%
_{m,loc}^{0,\ast }(\Sigma )$ $($Definition \ref{d-6.8-5}). To estimate
(uniformly in $t$) the $C_{B}^{s}$-norm of the RHS of (\ref{Kstar}) is
reduced to estimating the usual $C^{s}$-norm in $z$ and the supremum norm in 
$w$ of 
\begin{equation}
h^{m}(z,\bar{z})\tau _{j}(z)\sigma _{j}(\phi )\int_{\Sigma }|\eta
|^{2m}K_{t}^{j\ast }(z,\zeta )\varphi _{j}(y)v_{j}(y)dv_{\Sigma ,m}(y).
\label{Kstar1}
\end{equation}

\noindent The following inequality follows from \cite[Theorem 2.20 or 2.29]%
{BGV} (cf. comments below (\ref{Adj-5-1}) in the proof of Theorem \ref%
{t-adjeq}\textbf{)}$:$ there exists $\delta _{0}$ $>$ $0$ and $C_{s}^{\prime
\prime }$ $>$ $0$ (independent of $\delta _{0}$ and $t)$ such that for $0$ $%
< $ $t$ $<$ $\delta _{0}$ 
\begin{eqnarray}
&&||\int K_{t}^{j\ast }(z,\zeta )\overline{\varphi _{j}(\zeta ,\eta
)v_{j}(\zeta ,\eta )}dv(\zeta )||_{C^{s}(z)}  \label{Kstar2a} \\
&\leq &C_{s}^{\prime \prime }||\overline{\varphi _{j}(\cdot ,\eta
)v_{j}(\cdot ,\eta )}||_{C^{s}(\zeta )}\leq C_{s}^{\prime \prime }\sup_{\eta
}||\varphi _{j}(\cdot ,\eta )v_{j}(\cdot ,\eta )||_{C^{s}(\zeta )}.  \notag
\end{eqnarray}

\noindent Observe that $|\eta |^{2m}$ in the integrand of (\ref{Kstar1}) is
independent of $z$ and its integral with respect to the fibre measure $%
dv_{f,m}(y)$ is bounded. This together with (\ref{Kstar2a}) gives 
\begin{equation*}
\sup_{w}||\text{the term (\ref{Kstar1})\TEXTsymbol{\vert}\TEXTsymbol{\vert}}%
_{C^{s}(z)}\leq \text{(constant)}\cdot ||u||_{C_{B}^{s}}.
\end{equation*}

\noindent Now (\ref{Adj-7}) follows. This together with Lemma \ref{l-bdp}
implies (\ref{Adj-6}).

\endproof%

One way to solve our heat equation, based on Theorem \ref{t-adjeq} and
Proposition \ref{l-ubK0}, resorts to the method of successive approximation
(cf. \cite{BGV}, \cite{CHT}). The convolution of two operators $A,$ $B$ is
defined through its distribution kernel as usual:%
\begin{equation}
(A\text{ }\sharp \text{ }B)_{t}(x,y):=\int_{0}^{t}\int_{\Sigma
}A(t-s,x,p)B(s,p,y)dv_{\Sigma ,m}(p)ds.  \label{6.49-5}
\end{equation}

\noindent The method of successive approximation results in a solution to
our heat equation; see Proposition \ref{p-existence} below. But since $%
\Sigma $ is noncompact, to have convergence requires a special class of
operators. Fortunately the operators $P_{m}^{0\ast }$ and $R^{\ast }$ belong
to this class ($P_{m}^{0\ast },$ $R^{\ast }$ denote the operators with
distribution kernels $P_{m,t}^{0\ast }(x,y),$ $R_{t}^{\ast }(x,y)$
respectively). The point is basically that although $\Sigma $ is noncompact,
we have only one direction which is noncompact, and the integration along
this noncompact direction can be controlled by the choice of our metric $%
G_{a,m}$ (see Remark \ref{3-r}). Another ingredient to be used here is (\ref%
{7K-5a}) in the proof of Lemma \ref{l-bdp} above; see the proof of Lemma \ref%
{l-6.7-5} below. These features give more complexities than the previous
work \cite[Proposition 5.14]{CHT}.

\begin{lemma}
\label{l-6.7-5} The kernel functions associated to $R_{t}^{\ast },\cdot
\cdot \cdot ,$ ($R^{\ast k})_{t}$ ($=$ ($R^{\ast }\sharp R^{\ast }$ $\cdot
\cdot \cdot $ $\sharp R^{\ast })_{t},$ $k$ copies) and $P_{m,t}^{0\ast },$ $%
(P_{m}^{0\ast }\sharp R^{\ast })_{t},$ $(P_{m}^{0\ast }\sharp R^{\ast
}\sharp R^{\ast })_{t},\cdot \cdot \cdot $ are of the form (\ref{6.7-5}).
Moreover, given $s\in \mathbb{N},$ there are $1>\delta _{0},$ $\delta _{1}>0$
and $C_{s}$ $>$ $0$ (independent of $\delta _{0},$ $\delta _{1}$ and $t$)
such that for all $t\in (0,\delta _{0})$%
\begin{equation}
||R_{t}^{\ast }||_{C_{B}^{s}}\leq \frac{1}{2}e^{-\frac{\delta _{1}}{t}%
},\cdot \cdot \cdot ,\text{ }||(R^{\ast k})_{t}||_{C_{B}^{s}}\leq \frac{1}{%
2^{k}}e^{-\frac{\delta _{1}}{t}},  \label{conv}
\end{equation}%
\begin{equation}
||P_{m}^{0\ast }\sharp R^{\ast }||_{C_{B}^{s}}\leq \frac{C_{s}}{2}e^{-\frac{%
\delta _{1}}{t}},\cdot \cdot \cdot ,\text{ }||P_{m}^{0\ast }\sharp R^{\ast
k}||_{C_{B}^{s}}\leq \frac{C_{s}}{2^{k}}e^{-\frac{\delta _{1}}{t}}.
\label{conv-1}
\end{equation}%
Here $||R_{t}^{\ast }||_{C_{B}^{s}}$ means $||R_{t}^{\ast
}(x,y)||_{C_{B}^{s}(\Sigma \times \Sigma )}$ (as given in (\ref{CBs})), etc.
\end{lemma}

\proof
By (\ref{6-1e}) we easily obtain%
\begin{equation}
(\pi _{m}\circ H_{m,t}^{j\ast })(x,y)=l(x)^{m}\int_{\xi \in \mathbb{C}^{\ast
}}\sigma (\xi )_{x}^{\ast }\circ H_{m,t}^{j\ast }(\xi x,y)\bar{\xi}^{m}(\tau
_{x}^{\ast }dv_{f,m})(\xi )  \label{Kstar3}
\end{equation}

\noindent and by (\ref{Adj-5-1}) we have ($w=|w|e^{i\phi })$%
\begin{eqnarray}
&&\sigma (\xi )_{x}^{\ast }\circ H_{m,t}^{j\ast }(\xi x,y)  \label{Kstar4} \\
&=&\sigma (\xi )_{x}^{\ast }\circ w^{m}(\xi x)h^{m}(z(\xi x),\bar{z}(\xi
x))\tau _{j}(z(\xi x))\sigma _{j}(\phi (\xi x))K_{t}^{j\ast }(z(\xi x),\zeta
)\varphi _{j}(y)\bar{\eta}^{m}(y).  \notag
\end{eqnarray}

\noindent Next we compute%
\begin{eqnarray}
&&w^{m}(\xi x)h^{m}(z(\xi x),\bar{z}(\xi x))  \label{Kstar5} \\
&&\overset{(\ref{7K-5a})}{=}h^{m}(z(\xi x),\bar{z}(\xi x))(\xi \alpha
_{k}^{-1}\alpha _{0})^{m}w^{m}(x)\ for\ \xi \in J_{k}.  \notag
\end{eqnarray}

\noindent Substituting (\ref{Kstar5}) into (\ref{Kstar4}) we see that $\pi
_{m}\circ H_{m,t}^{j\ast }(x,y)$ is of the form (\ref{6.7-5}) (containing
factors $w^{m}(x)$ and $\bar{\eta}^{m}(y))$ after the parameter $\xi $ is
integrated out in (\ref{Kstar3}). Hence $P_{m,t}^{0\ast }(x,y)=\sum_{j}(\pi
_{m}\circ H_{m,t}^{j\ast })(x,y)$ is of the form (\ref{6.7-5}).
Alternatively we can take the adjoint of $P_{m,t}^{0}(x,y)$ $=$ $%
\sum_{j}w^{m}p_{m,t}^{0,j}(x,y)\bar{\eta}^{m}$ (see (\ref{7K-25}) for the
explicit form of $p_{m,t}^{0,j}(x,y)$) to obtain $P_{m,t}^{0\ast }(x,y)$ $=$ 
$\sum_{j}w(x)^{m}p_{m,t}^{0,j\ast }(x,y)\overline{\eta (y)}^{m}$ where we
have, via (\ref{7K-25}) using $K_{t}^{j\ast }=K_{t}^{j},\overline{(\alpha
_{k}\alpha _{0}^{-1})}$ $=$ $\alpha _{k}^{-1}\alpha _{0}$ and $z_{k}$ $=$ $%
z(\alpha _{k}x)$ 
\begin{equation}
p_{m,t}^{0,j\ast }(x,y)=h^{m}(z(x),\bar{z}(x))\sum_{k=0}^{\Lambda }(\alpha
_{k}^{-1}\alpha _{0})^{m}\sigma (\alpha _{k})_{x}^{\ast }\{\tau
_{j}(z_{k})K_{t}^{j}(z_{k},\zeta (y))\}\varphi _{j}(y).  \label{Kstar6}
\end{equation}

For $R_{t}^{\ast }=\pi _{m}\circ \hat{R}_{t}^{\ast }$ we can also get its
kernel function through a direct computation using the formulas for $\pi
_{m} $ (\ref{6-1e}) and $\hat{R}_{t}$ (\ref{Adj-4a}) in a way parallel to (%
\ref{Kstar3}) and (\ref{Kstar6}) (with $\xi $ integrated out using (\ref%
{fibrenv}) and (\ref{6-1a})). Putting $R_{t}^{\ast }(x,y)$ $=$ $%
\sum_{j}w^{m}r_{t}^{j\ast }(x,y)\bar{\eta}^{m}$ we have%
\begin{equation}
r_{t}^{j\ast }(x,y)=h^{m}(z(x),\bar{z}(x))\sum_{k=0}^{\Lambda }(\alpha
_{k}^{-1}\alpha _{0})^{m}\sigma (\alpha _{k})_{x}^{\ast }\big(\lbrack
\square _{z_{k},m}^{c\pm },\tau _{j}(z_{k})]K_{t}^{j}(z_{k},\zeta (y))\big)%
\varphi _{j}(y).  \label{Rstar1}
\end{equation}

\noindent By $||R(t,x,y)||_{C_{B}^{s}(\Sigma \times \Sigma )}\leq C_{s}e^{-%
\frac{\varepsilon _{0}}{t}}$ \ for $t$ $>$ $0$ (\ref{Adj-0}) and Theorem \ref%
{t-adjeq} $iii)$ we conclude 
\begin{equation}
||R_{t}^{\ast }||_{C_{B}^{s}}\leq \frac{1}{2}e^{-\frac{\delta _{1}}{t}}
\label{6-68-1}
\end{equation}%
\noindent in (\ref{conv}).

Next we compute the convolution 
\begin{equation}
(P_{m}^{0\ast }\sharp R^{\ast })_{t}(x,y)=\sum_{j,\text{ }j^{\prime
}}(P_{m}^{0,j\ast }\sharp R^{j^{\prime }\ast })_{t}(x,y)=\sum_{j,\text{ }%
j^{\prime }}w^{m}(x)(p_{m}^{0,j\ast }\tilde{\sharp}r^{j^{\prime }\ast
})_{t}(x,y)\bar{\eta}^{\prime m}(y)  \label{PR-1}
\end{equation}

\noindent where, with $y=(\zeta ^{\prime },\eta ^{\prime }),$ $q=(\beta
^{\prime },\gamma ^{\prime })$ in the $j^{\prime }$-chart and $q=(\beta
,\gamma )$ in the $j$-chart%
\begin{equation}
(p_{m}^{0,j\ast }\tilde{\sharp}r^{j^{\prime }\ast
})_{t}(x,y):=\int_{0}^{t}\int_{\Sigma }p_{m,t-s}^{0,j\ast }(x,q)\bar{\gamma}%
(q)^{m}\gamma ^{\prime }(q)^{m}r_{s}^{j^{\prime }\ast }(q,y)dv_{\Sigma
,m}(q)ds.  \label{pr-1}
\end{equation}

\noindent The integrand $p_{m,t-s}^{0,j\ast }(x,q)\bar{\gamma}(q)^{m}\gamma
^{\prime }(q)^{m}r_{s}^{j^{\prime }\ast }(q,y)$ in (\ref{pr-1}) has the
following explicit expression: (denoting by $\alpha _{l}^{\prime }$ the
corresponding $\alpha _{k}$ in the $j^{\prime }$-chart where $\beta
_{l}^{\prime }$ :$=$ $\beta ^{\prime }(\alpha _{l}^{\prime }q))$%
\begin{eqnarray}
&&\ \ \ \ \ \ \ \ h^{m}(z(x),\bar{z}(x))\sum_{k=0}^{\Lambda }(\alpha
_{k}^{-1}\alpha _{0})^{m}\sigma (\alpha _{k})_{x}^{\ast }\{\tau
_{j}(z_{k})K_{t-s}^{j}(z_{k},\beta (q))\}\varphi _{j}(q)\bar{\gamma}(q)^{m}
\label{pr-2} \\
&&\cdot \gamma ^{\prime }(q)^{m}h^{\prime m}(\beta ^{\prime }(q),\bar{\beta}%
^{\prime }(q))\sum_{l=0}^{\Lambda ^{\prime }}(\alpha _{l}^{\prime -1}\alpha
_{0}^{\prime })^{m}\sigma (\alpha _{l}^{\prime })_{q}^{\ast }\big(\lbrack
\square _{\beta _{l}^{\prime },m}^{c\pm },\tau _{j^{\prime }}(\beta
_{l}^{\prime })]K_{s}^{j^{\prime }}(\beta _{l}^{\prime },\zeta ^{\prime }(y))%
\big)\varphi _{j^{\prime }}(y).  \notag
\end{eqnarray}

\noindent To bound $||P_{m}^{0\ast }\sharp R^{\ast }||_{C_{B}^{s}}$ by $e^{-%
\frac{\delta _{1}}{t}}$ note first that for the singular term $\frac{1}{%
(t-s)^{n-1}}$ along the diagonal of $K_{t-s}^{j}$ in (\ref{pr-2}) we can get
rid of it using the change of variable on $t-s$ as in the proof of Theorem %
\ref{t-adjeq} above. This and $||r_{s}^{j^{\prime }\ast }||_{C^{l}}$ $\leq $ 
$\frac{1}{2}e^{-\frac{\delta _{1}}{t}}$ (\ref{6-68-1}) (note that $e^{-c/s}$ 
$\leq $ $e^{-c/t}$ for $0<s\leq t)$ give that $K_{t-s}^{j}$ hence $%
p_{m,t-s}^{0,j\ast }$ and $r_{s}^{j^{\prime }\ast }$ in (\ref{pr-1}) are
jointly controlled by the factor $e^{-c/t}$ after integrating out
\textquotedblleft $q",$ i.e. the $z$-part $\beta (q)$ of $q$. Using this one
sees that, as $\gamma ^{\prime }$ and $\gamma $ differ by a bounded
holomorphic transition function and $|\gamma (q)|^{2m}$ is integrable by
Remark \ref{3-r}, the convolution (\ref{pr-2}) in (\ref{pr-1}) is integrable
since $|\alpha _{k}|$ $=$ $|\alpha _{l}^{\prime }|$ $=$ $1$ and $h^{m}$, $%
h^{\prime m}$ are bounded. Altogether the first inequality in (\ref{conv-1})
for the $C_{B}^{0}$-norm follows. For the $C_{B}^{s}$-norm, $s$ $>$ $0,$ we
note that the derivatives of (\ref{pr-2}) in $z(x)$ and $\zeta (y)$ do not
change the above basic structure. So we also have (\ref{conv-1}) for the $%
C_{B}^{s}$-norm. From the structure of $P_{m,t}^{0\ast }$ and $R_{t}^{\ast }$
we see without difficulty that the similar conclusions hold with all the
other convolutions mentioned in the lemma.

\endproof%

We can now adapt \cite[Proposition 5.14]{CHT} here and reach the following
result of similar nature, with the difference that we are adopting $%
C_{B}^{s} $-norms here.

\begin{proposition}
\label{p-existence} $i)$ (Existence) Given $s\in \mathbb{N},$ there exists $%
\varepsilon >0$ such that for every fixed $t$ $\in $ $(0,\varepsilon )$ the
kernel function $\Lambda _{t}(x,y)$ given by 
\begin{equation}
\Lambda _{t}(x,y):=P_{m,t}^{0\ast }(x,y)-(P_{m}^{0\ast }\sharp R^{\ast
})_{t}(x,y)+(P_{m}^{0\ast }\sharp R^{\ast }\sharp R^{\ast })_{t}(x,y)-\cdot
\cdot \cdot  \label{conv-0}
\end{equation}%
exists, and converges in $C_{B}^{s}(\Sigma \times \Sigma )$ (including its
t-derivatives up to any given order$)$. $ii)$ Let ($R^{\ast k})_{t}$ be as
in Lemma \ref{l-6.7-5}, $k$ $\geq $ $0.$ Suppose $u$ $\in $ $L^{2,\pm
}(\Sigma ,\pi ^{\ast }\mathcal{E}_{M},G_{a,m}).$ Then $(P_{m}^{0\ast }\sharp
R^{\ast k})_{t}u$ $\in $ $L_{m}^{2,\pm }(\Sigma ,\pi ^{\ast }\mathcal{E}%
_{M},G_{a,m}).$ In particular the image of $\Lambda _{t}$ lies in the $m$%
-space. Moreover for $u\in \Omega _{c}^{0,\pm }(\Sigma )$%
\begin{eqnarray}
\frac{\partial \Lambda _{t}}{\partial t}u+\tilde{\square}_{m}^{c\pm }\Lambda
_{t}u &=&0,  \label{heat-eq} \\
\lim_{t\rightarrow 0^{+}}\Lambda _{t}u &=&\pi _{m}u\text{ in }||\cdot
||_{C_{B}^{s}}  \notag
\end{eqnarray}%
$iii)$ (Approximation) Given $s\in \mathbb{N},$ there exists $\varepsilon
_{0}>0$ independent of $t$ such that%
\begin{equation}
||\Lambda _{t}(\cdot ,\cdot )-P_{m,t}^{0\ast }(\cdot ,\cdot
)||_{C_{B}^{s}(\Sigma \times \Sigma )}\leq e^{-\frac{\varepsilon _{0}}{t}}%
\text{ \ for all }t\in (0,\varepsilon _{0}).  \label{approx}
\end{equation}
\end{proposition}

\proof
It follows from (\ref{conv-1}) that the sequence (\ref{conv-0}) converges in
the $C_{B}^{s}$-norm and (\ref{approx}) holds. By the definition of the $%
C_{B}^{s}$-norm associated with a fixed $m$ (\ref{CBs}), the function $%
\Lambda _{t}(x,y)$ of (\ref{conv-0}) exists on $\Sigma \times \Sigma .$

For $ii)$ we observe that the image of the convolution lies in the image of
its first operator. In our case the first operator is $P_{m,t}^{0\ast }$ so
the image is in the $m$-space. To verify (\ref{heat-eq}) takes slightly more
work. Let $q_{t}^{k}$ denote the $(k+1)$-th term in (\ref{conv-0}). A direct
computation shows that%
\begin{equation}
\frac{\partial q_{t}^{k}}{\partial t}(x,y)+\tilde{\square}_{m,x}^{c\pm
}q_{t}^{k}(x,y)=(R^{\ast k})_{t}(x,y)+(R^{\ast (k+1)})_{t}(x,y)
\label{conv-2}
\end{equation}

\noindent (cf. \cite[(2) of Lemma 2.22]{BGV}). Noting that $\Lambda _{t}$ is
the alternating sum of these $q_{t}^{k},$ one interchanges the order of the
action of $\partial _{t}+\tilde{\square}_{m}^{c\pm }$ on $\Lambda _{t}$ with
the summation in view of (\ref{conv-1}) and similar estimates on their $t$%
-derivatives. The first equation of (\ref{heat-eq}) follows from telescoping
with (\ref{conv-2}) (cf. \cite[Theorem 2.23]{BGV}). The second equality of (%
\ref{heat-eq}) follows from (\ref{Adj-5}) (cf. (\ref{Adj-6})) and (\ref{conv}%
).

\endproof%

We are now back to discuss the properties of $P_{m,t}^{0,\pm }.$ Recall that 
$e^{-t\tilde{\square}_{m}^{c\pm }}(x,y)$ is the heat kernel that we obtain
in Proposition \ref{5-5.5}. Similarly let $\Lambda _{t}^{\pm }$ denote $%
\Lambda _{t}$ acting on the even/odd elements. Suppose $B_{t}^{\pm },$ $t>0,$
is any bounded linear operator on $L^{2,\pm }(\Sigma ,\pi ^{\ast }\mathcal{E}%
_{M},G_{a,m})$ such that $i)$ for $\psi $ $\in $ $\Omega ^{0,\pm }(\Sigma )$ 
$\cap $ $L^{2,\pm }(\Sigma ,\pi ^{\ast }\mathcal{E}_{M},G_{a,m}),$ $%
B_{t}^{\pm }\psi $ $\in $ $\Omega _{m}^{0,\pm }(\Sigma )$; $ii)$ $B_{t}^{\pm
}\psi $ satisfies the heat equation 
\begin{eqnarray}
(\partial _{t}+\tilde{\square}_{m}^{c\pm })B_{t}^{\pm }\psi &=&0\text{ \
(differentiability in }t\text{ is assumed),}  \label{GHE} \\
B_{t}^{\pm }\psi &\rightarrow &\pi _{m}\psi \text{ in }L^{2}\text{ as }%
t\rightarrow 0;  \notag
\end{eqnarray}

\noindent $iii)$ $B_{t}^{\pm }\psi $ $\rightarrow $ $B_{t_{0}}^{\pm }\psi $
in $L^{2}$ as $t\rightarrow t_{0}$ for any fixed $t_{0}$ $>$ $0.$ The
uniqueness part is the following.

\begin{theorem}
\label{t-uniqueness} $i)$ (Uniqueness) It holds that 
\begin{equation}
B_{t}^{\pm }=e^{-t\tilde{\square}_{m}^{c\pm }},\text{ in particular}
\label{ident-0}
\end{equation}%
\begin{equation}
\Lambda _{t}^{\pm }(x,y)=e^{-t\tilde{\square}_{m}^{c\pm }}(x,y)
\label{ident}
\end{equation}%
and as a consequence $\Lambda _{t}^{\pm }$ are self-adjoint (Proposition \ref%
{5-5.5}).

$ii)$ (Approximation) For every $s\in \mathbb{N}$ there exist $\varepsilon
_{0}>0$ and $\varepsilon >0$ such that%
\begin{equation}
||e^{-t\tilde{\square}_{m}^{c\pm }}(\cdot ,\cdot )-P_{m,t}^{0,\pm }(\cdot
,\cdot )||_{C_{B}^{s}(\Sigma \times \Sigma )}\leq e^{-\frac{\varepsilon _{0}%
}{t}}\text{ \ for all }t\in (0,\varepsilon ).  \label{approx-1}
\end{equation}%
As a consequence $e^{-t\tilde{\square}_{m}^{c\pm }}(x,y)$ and $%
P_{m,t}^{0}(x,y)$ are the same in the sense of asymptotic expansion (as
defined in \cite[Definition 5.5]{CHT} with $C^{l}$-norms replaced by $%
C_{B}^{l}$-norms). Along the diagonal it holds that%
\begin{equation}
\left\vert e^{-t\tilde{\square}_{m}^{c\pm }}(x,x)-P_{m,t}^{0,\pm
}(x,x)\right\vert =O(l(x)^{m}e^{-\frac{\varepsilon _{0}}{t}})
\label{approx-2}
\end{equation}%
($l(x)$ being unbounded on $\Sigma $).
\end{theorem}

\proof
The idea of the proof can now follow that in \cite{BGV}. We compute for $%
\psi ,$ $\varphi $ $\in $ $\Omega _{c}^{0,\pm }(\Sigma ),$ $\partial _{\tau
}<B_{t-\tau }^{\pm }\psi ,e^{-\tau \tilde{\square}_{m}^{c\pm }}\varphi >$ $=$
$0$ ($0$ $<$ $\tau $ $<$ $t$) by using heat equations (\ref{GHE}) and (\ref%
{5.2.1}). By the initial condition in (\ref{GHE}) and (\ref{5.2.1}), and the
images of $e^{-t\tilde{\square}_{m}^{c\pm }}$ and $B_{t}^{\pm }$ belonging
to the $m$-space (Proposition \ref{5-5.5} and the property $i)$ of $%
B_{t}^{\pm }$), we compute%
\begin{eqnarray*}
0 &=&\int_{0}^{t}\partial _{\tau }<B_{t-\tau }^{\pm }\psi ,e^{-\tau \tilde{%
\square}_{m}^{c\pm }}\varphi >d\tau =<\pi _{m}^{\pm }\psi ,e^{-t\tilde{%
\square}_{m}^{c\pm }}\varphi >-<B_{t}^{\pm }\psi ,\pi _{m}\varphi > \\
&=&<\psi ,e^{-t\tilde{\square}_{m}^{c\pm }}\varphi >-<B_{t}^{\pm }\psi
,\varphi >=<e^{-t\tilde{\square}_{m}^{c\pm }}\psi ,\varphi >-<B_{t}^{\pm
}\psi ,\varphi >.
\end{eqnarray*}

\noindent Here the property $iii)$ of $B_{t}^{\pm }$ has been used for the
second equality above. So (\ref{ident-0}) hence (\ref{ident}) follows. From (%
\ref{approx}), (\ref{ident}) and $e^{-t\tilde{\square}_{m}^{c\pm }}$ being
self-adjoint (cf. Proposition \ref{5-5.5}), (\ref{approx-1}) follows.

As for the factor $l(x)^{m}$ in (\ref{approx-2}) we first observe that $%
P_{m,t}^{0\ast }$ $=$ $\sum_{j}\pi _{m}\circ H_{m,t}^{j\ast }$ contains the
factor $l(x)^{m}$ in view of (\ref{Kstar3}). Therefore $P_{m,t}^{0\ast
}(x,x),$ hence $P_{m,t}^{0}(x,x)$, contains the factor $l(x)^{m}$ so do $%
\Lambda _{t}^{\pm }(x,x)$ and $e^{-t\tilde{\square}_{m}^{c\pm }}(x,x)$ in
view of (\ref{conv-0}) (the convolution led by $P_{m,t}^{0\ast }$ always has
the factor $l(x)^{m}$) and (\ref{ident}). This together with (\ref{approx-1}%
) gives (\ref{approx-2}).

\endproof%

In the remaining of this section, we shall treat $P_{m,t}^{0}$ more closely
as will be needed in the next section via (\ref{ident}) above. The main
result is Theorem \ref{p-asymp} below. Writing $\xi ^{-1}y$ for $\xi
^{-1}\circ y,$ we have\footnote{%
The $\mathbb{C}^{\ast }$-orbit $\{\xi ^{-1}y\}_{\xi \in \mathbb{C}^{\ast }}$
of $y$ could be delicate (Cases $i),\ ii)$ below (\ref{6-1g})); nevertheless
the local expressions in (\ref{6-1g}) involving $\zeta (\xi ^{-1}y)$ (which
is meaningless if $\xi ^{-1}y$ lies outside the chart of $\zeta $) makes
good sense due to cutoff functions there. The same is true in many places
throughout this paper without explicit mention.} (see (\ref{6-1}) with $%
x=(z,w),$ $y=(\zeta ,\eta )$)%
\begin{eqnarray}
&&H_{m,t}^{j}(x,\xi ^{-1}y)\circ \sigma (\xi )_{\xi ^{-1}\circ y}^{\ast
}=\varphi _{j}(x)w^{m}K_{t}^{j}(z,\zeta (\xi ^{-1}y))\eta ^{-m}(\xi ^{-1}y)
\label{6-1g} \\
&&\text{ \ \ \ \ \ \ \ \ \ \ \ \ \ \ \ \ \ \ \ \ \ \ \ \ \ }\tau _{j}(\zeta
(\xi ^{-1}y))\sigma _{j}(\vartheta (\xi ^{-1}y))l(\xi ^{-1}y)^{m}\circ
\sigma (\xi )_{\xi ^{-1}y}^{\ast }.  \notag
\end{eqnarray}

\noindent Due to $\xi ^{-1}y$ in the arguments above, we note the following
two cases (cf. (\ref{1-0}), (\ref{1-1})):

Case $i):$ If $\xi $ is close to $1$ such that $\xi ^{-1}y$ is still in the
same chart as $y$ $($where $\zeta (y)$ $=$ $\zeta ),$ then $\zeta (\xi
^{-1}y)$ = $\zeta ,$ $\bar{\eta}^{m}(\xi ^{-1}y)$ $=$ $\bar{\xi}^{-m}\bar{%
\eta}^{m}$ and $\vartheta (\xi ^{-1}y)$ $=$ $\vartheta (y)-\gamma $ where we
write $\xi $ $=$ $|\xi |e^{i\gamma };$

Case $ii):$ For $\xi $ general$,$ $\xi ^{-1}y$ and $y$ do not necessarily
lie in the same chart$.$ Even if $\xi ^{-1}y$ and $y$ lie in the same chart$%
, $ unlike case $i)$ the $\zeta $-values of $\xi ^{-1}y$ and $y$ may not be
the same. Indeed $\zeta (\xi ^{-1}y)$ $\neq $ $\zeta (y)$ for $y$ near the $%
\mathbb{C}^{\ast }$-strata $\Sigma _{\text{sing}}$ of $\Sigma .$ See (\ref%
{claim}) for the detail. This fact will be rather crucial for us in the
subsequent sections.

Remark that if the $\mathbb{C}^{\ast }$-action $\sigma $ on $\Sigma $ is
globally free then only Case $i)$ will occur (and in this case it is valid
to take all $\xi \in $ $\mathbb{C}^{\ast }$ rather than $\xi \sim 1).$

Substituting (\ref{6-1g}) into the RHS of (\ref{6-1f}) gives, recalling $%
P_{m,t}^{0}(x,y)$ $=$ $\sum_{j}(H_{m,t}^{j}\circ \pi _{m})(x,y),$%
\begin{eqnarray}
&&(H_{m,t}^{j}\circ \pi _{m})(x,y)=\varphi _{j}(x)\int_{\xi \in \mathbb{C}%
^{\ast }}{\Large \{}w^{m}K_{t}^{j}(z,\zeta (\xi ^{-1}y))\eta ^{-m}(\xi
^{-1}y)  \label{HQ} \\
&&\tau _{j}(\zeta (\xi ^{-1}y))\sigma _{j}(\vartheta (\xi ^{-1}y))l(\xi
^{-1}y)^{m}\bar{\xi}^{m}{\Large \}}\circ \sigma (\xi )_{\xi ^{-1}y}^{\ast
}d\mu _{y,m}(\xi ).  \notag
\end{eqnarray}

It is well known that the (ordinary, local) heat kernel $K_{t}^{j}(z,\zeta )$
has the asymptotic expansion (see for instance \cite[(5.19) on p.76]{CHT})

\begin{eqnarray}
K_{t}^{j}(z,\zeta ) &=&e^{-\frac{\tilde{d}_{M}^{2}(z,\zeta )}{4t}%
}K^{j}(t,z,\zeta )  \label{OAE} \\
K^{j}(t,z,\zeta ) &\sim &t^{-n+1}b_{n-1}(z,\zeta )+t^{-n+2}b_{n-2}(z,\zeta
)+...  \notag
\end{eqnarray}%
\noindent as $t\rightarrow 0^{+}$ for $z,\zeta $ in $V_{j},$ where $\tilde{d}%
_{M}$ denotes the distance function associated with the metric $\pi ^{\ast
}g_{M}|_{V_{j}}$ (cf. (\ref{Ga}) with lines above). Note that $\tilde{d}_{M}$
may depend on the choice of charts $V_{j}$ \textit{a priori}. Recalling that 
$V_{j}$ is endowed with the metric induced by $\pi ^{\ast }g_{M},$ we can
assume $V_{j}$ (as a Riemannian manifold) to be convex for every $j$
(possibly after shrinking). Then it is seen that $\tilde{d}_{M}$ here can be
independent of the choice of charts. Note that $\tilde{d}_{M}$ is not
necessarily the same as the distance function $d_{M}$ on the complex
orbifold $M=\Sigma /\sigma $ since two distinct points in $V_{j}$ may
project to the same point in $M.$ Compare remarks after (\ref{7-5'}).

Via (\ref{OAE}) we then get an asymptotic expansion of (\ref{HQ}) for ($%
H_{m,t}^{j}\circ \pi _{m})(x,y)$ hence for $P_{m,t}^{0}(x,y)$ without
difficulty (the fact that $l(\xi ^{-1}y)^{m}\bar{\xi}^{m}$ $=$ $l(y)^{m}\xi
^{-m}$ by (\ref{7-1-1}) has been used here).

\begin{theorem}
\label{p-asymp} (Asymptotic expansion) With the notation above, we have that 
$P_{m,t}^{0}(x,y)$ is of the form (\ref{6.7-5}) and, via (\ref{approx-1})%
\begin{equation}
P_{m,t}^{0}(x,y),\text{ }e^{-t\tilde{\square}_{m}^{c\pm }}(x,y)\sim
t^{-(n-1)}a_{n-1}(t,x,y)+t^{-(n-2)}a_{n-2}(t,x,y)+\text{ }\cdot \cdot \cdot
\label{7K-4a}
\end{equation}%
(for the meaning of the above \textquotedblleft $\sim $" we refer to \cite[%
Definition 5.5, p.75]{CHT} with $C^{l}$-norms replaced by $C_{B}^{l}$-norms)
where for $s=n-1,$ $n-2,$ $\cdot \cdot \cdot ,$%
\begin{eqnarray}
a_{s}(t,x,y) &=&l(y)^{m}\sum_{j}\varphi _{j}(x)w^{m}\int_{\xi \in \mathbb{C}%
^{\ast }}\{e^{-\frac{\tilde{d}_{M}^{2}(z,\zeta )}{4t}}b_{s}(z,\zeta )
\label{6.66-5} \\
&&\text{ \ \ \ \ \ \ \ \ \ \ \ \ \ \ \ \ \ \ }\eta ^{-m}\tau _{j}(\zeta
)\sigma _{j}(\vartheta )\xi ^{-m}\}d\mu _{y,m}(\xi )  \notag
\end{eqnarray}%
where to simplify notations, we use $\zeta ,$ $\eta ^{-m}$ and $\vartheta $
to denote $\zeta (\xi ^{-1}y),$ $\eta ^{-m}(\xi ^{-1}y)$ and $\vartheta (\xi
^{-1}y)$ respectively.
\end{theorem}

\begin{remark}
\label{8.1-5} Even for $x$ $=$ $y,$ $a_{s}(t,x,x)$ still depends on $t.$ See 
\cite[Remark 1.6]{CHT} for details. Further $a_{s}(t,x,y)$ are not uniquely
determined; indeed they depend on the various data in, e.g. (\ref{6-1}), (%
\ref{6.0}) and (\ref{6-1a}). In contrast $b_{s}(z,\zeta )$ in (\ref{6.66-5})
is intrinsic (cf. remarks after (\ref{7-5'})). Note that $P_{m,t}^{0}$ may
not preserve the $m$-space. However if we consider $\tilde{P}_{m,t}^{0}(x,y)$
$:=$ ($\pi _{m}\circ P_{m,t}^{0})(x,y)$ it is not difficult to see that the
associated $\tilde{a}_{s}(t,x,y)$ (resp. $\tilde{P}_{m,t}^{0}(x,y))$
descends to \underline{$\tilde{a}$}$_{s}(t,\pi (x),\pi (y))$ (resp. $%
\widetilde{\text{\b{P}}}_{m,t}^{0}(\pi (x),\pi (y)))$ on the compact complex
orbifold $M=\Sigma /\sigma $ (as one can show ($\sigma _{\alpha ,\beta
}^{\ast }\tilde{a}_{s}(t,\cdot ,\cdot ))(x,y)$ $=$ $\alpha ^{m}(\bar{\beta}%
)^{m}\tilde{a}_{s}(t,x,y)$ where $\sigma _{\alpha ,\beta }(x,y)$ $:=$ $%
(\alpha x,\beta y)$ and similar formulas for $\tilde{P}_{m,t}^{0}(x,y)).$ It
can be shown that both \underline{$\tilde{a}$}$_{s}(t,$ $\pi (x),$ $\pi
(y))\ $and $\widetilde{\text{\b{P}}}_{m,t}^{0}(\pi (x),\pi (y))$ are
associated with $\underline{\tilde{\square}}_{m}^{c\pm },$ as the extension
of $\square _{U_{j},m}^{c\pm }$ to $M$ in (\ref{DiracLap}) (acting on
sections of the $m$-th power of orbifold line bundle $L_{\Sigma }^{\ast
}/\sigma $, cf. Remarks \ref{PLMB} and \ref{9.3}).
\end{remark}

\section{\textbf{Asymptotic expansion of the transversal heat kernel\label%
{AE_THK}}}

The goal of this section is to prove Theorem \ref{AHKE}$.$ Our notation
follows that of the introductory paragraph of the last section. Recall that
(see (\ref{6-1a1})) 
\begin{equation}
P_{m,t}^{0}=\sum_{j}H_{m,t}^{j}\circ \pi _{m};  \label{7-1.5}
\end{equation}%
\noindent $\hat{W}_{j}$ $\subset $ $W_{j}$ ($=V_{j}\times (-\varepsilon
_{j},\varepsilon _{j})\times \mathbb{R}^{+})$ and $\varepsilon
_{j}(=\varepsilon ,$ $\forall j)$ small, such that 
\begin{equation}
\hat{W}_{j}=V_{j}\times (-\frac{\varepsilon _{j}}{4},\frac{\varepsilon _{j}}{%
4})\times \mathbb{R}^{+}  \label{7-1.75}
\end{equation}%
\noindent and $\Sigma $ is still a union of finitely many $\hat{W}_{j}.$

Assume that $x$ $\in $ $\Sigma \backslash \Sigma _{\text{sing}}$ is in the
chart $\hat{W}_{j}$ and $x$ has coordinates $(z,w)$ with $z$ $\in $ $V_{j},$ 
$w$ $=$ $|w|e^{i\phi }.$ We write $\xi ^{-1}x$ (resp. ($\sigma _{\xi
^{-1}}^{\ast })_{x})$ for $\xi ^{-1}\circ x$ (resp. $\sigma (\xi
^{-1})_{x}^{\ast }).$ From (\ref{HQ}) for $x=y$ we write\footnote{%
See the footnote attached to (\ref{6-1g}) of Section 6.}%
\begin{eqnarray}
&&\ \ (H_{m,t}^{j}\circ \pi _{m})(x,x)=\varphi _{j}(x)\int_{\xi \in \mathbb{C%
}^{\ast }}{\Large \{}w^{m}(x)K_{t}^{j}(z(x),z(\xi ^{-1}x))\bar{w}^{m}(\xi
^{-1}x)  \label{HQD} \\
&&h^{m}(z(\xi ^{-1}x),\bar{z}(\xi ^{-1}x))\tau _{j}(z(\xi ^{-1}x))\sigma
_{j}(\phi (\xi ^{-1}x))\bar{\xi}^{m}{\Large \}}\circ (\sigma _{\xi }^{\ast
})_{\xi ^{-1}x}d\mu _{x,m}(\xi ).  \notag
\end{eqnarray}

Let $\frac{2\pi }{p}$ be the largest period of the action $\sigma |_{S^{1}}.$
By assumption $x$ has the period $\frac{2\pi }{p}$ since $x\in \Sigma
\backslash \Sigma _{\text{sing}}.$ To facilitate the computation of (\ref%
{HQD}), we divide $\xi $ $\in $ $\mathbb{C}^{\ast }$ into two parts:

Part I: $\xi \in C:=(-\varepsilon _{j},\varepsilon _{j})$ $\times $ $\mathbb{%
R}^{+}$;

Part II: $\xi \in C^{\prime }:=(\varepsilon _{j},\frac{2\pi }{p}-\varepsilon
_{j})$ $\times $ $\mathbb{R}^{+}.$

\bigskip

\textbf{Part I of the RHS of (\ref{HQD}): }Let us first compute the RHS of (%
\ref{HQD}) for $\xi $ in Part I. The net result will be given in (\ref%
{HQDI-1}) and (\ref{HQDI-2}). For $\xi $ in Part I, by case $i)$ after (\ref%
{6-1g}) we have%
\begin{eqnarray}
&&\text{Part I of }(H_{m,t}^{j}\circ \pi _{m})(x,x)=\varphi _{j}(x)\int_{\xi
\in (-\varepsilon _{j},\varepsilon _{j})\times \mathbb{R}^{+}}{\Large \{}%
w^{m}K_{t}^{j}(z,z)  \label{HQDI} \\
&&\text{ \ \ \ \ \ \ \ \ \ \ \ \ \ \ \ \ \ }\xi ^{m}w^{-m}l(\xi
^{-1}x)^{m}\tau _{j}(z)\sigma _{j}(\phi (\xi ^{-1}x))\bar{\xi}^{m}{\Large \}}%
d\mu _{x,m}(\xi )  \notag \\
&&\overset{(\ref{7-1-1})}{=}\varphi _{j}(x)l^{m}(x)K_{t}^{j}(z,z)\tau
_{j}(z)\int_{\xi \in (-\varepsilon _{j},\varepsilon _{j})\times \mathbb{R}%
^{+}}\sigma _{j}(\phi (\xi ^{-1}x))d\mu _{x,m}(\xi ).  \notag
\end{eqnarray}

\noindent Plugging in $d\mu _{x,m}(\xi )$ for $\xi $ $\in $ $(-\varepsilon
_{j},\varepsilon _{j})$ $\times $ $\mathbb{R}^{+}$ (see (\ref{6-1e1}))$,$ we
have%
\begin{eqnarray}
&&d\mu _{x,m}(\xi )=\frac{l(\xi ^{-1}x)^{m}dv_{\Sigma ,m}(\xi ^{-1}x)\wedge
dv_{f,m}(x)}{dv_{\Sigma ,m}(x)}  \label{meas} \\
&\overset{(\ref{volume})}{=}&\frac{h^{m}(z,\bar{z})|\xi ^{-1}w|^{2m}dv(z(\xi
^{-1}x))\wedge (\tau _{x}^{\ast }dv_{f,m})(\xi ^{-1})\wedge dv_{f,m}(x)}{%
dv(z(x))\wedge dv_{f,m}(x)}  \notag \\
&=&l^{m}(x)|\xi |^{-2m}(\tau _{x}^{\ast }dv_{f,m})(\xi ^{-1})\frac{dv(z(\xi
^{-1}x))}{dv(z(x))}  \notag
\end{eqnarray}

\noindent See also (\ref{7-2-2}) for $d\mu _{x,m}(\xi )$.

Write $dv(z(\xi ^{-1}x))/dv(z(x))$ $=:$ $f(\xi ^{-1})$ (which equals 1 since 
$z(\xi ^{-1}x)$ $=$ $z(x)$ for the Part I case but we keep the notation)$.$
In (\ref{HQDI}) the following term simplifies (via (\ref{1-1}) for the $%
\mathbb{R}^{+}$-action)

\begin{eqnarray}
&&\int_{\xi \in (-\varepsilon _{j},\varepsilon _{j})\times \mathbb{R}%
^{+}}\sigma _{j}(\phi (\xi ^{-1}x))d\mu _{x,m}(\xi )  \label{7-0} \\
&=&l^{m}(x)\int_{\xi \in (-\varepsilon _{j},\varepsilon _{j})\times \mathbb{R%
}^{+}}\sigma _{j}(\phi (\xi ^{-1}x))|\xi |^{-2m}(\tau _{x}^{\ast
}dv_{f,m})(\xi ^{-1})f(\xi ^{-1})  \notag \\
&\overset{\eta =\xi ^{-1}}{=}&l^{m}(x)\int_{\eta \in (-\varepsilon
_{j},\varepsilon _{j})\times \mathbb{R}^{+}}\sigma _{j}(\phi (\eta x))|\eta
|^{2m}(\tau _{x}^{\ast }dv_{f,m})(\eta )f(\eta )  \notag \\
&\overset{\eta =|\eta |e^{i\gamma },(\ref{3-18.75})}{=}&\int_{|\eta |\in 
\mathbb{R}^{+}}\int_{I}\sigma _{j}(\gamma +\phi (x))f(\eta )\frac{%
dv_{S^{1}}(\gamma )}{2\pi }l^{m}(x)|\eta |^{2m}dv_{m}(|\eta ||w|)\text{.} 
\notag
\end{eqnarray}

\noindent where $I\supset \lbrack -\frac{\varepsilon _{j}}{2},\frac{%
\varepsilon _{j}}{2}]$ since $\phi (x)$ $\in $ $(-\frac{\varepsilon _{j}}{4},%
\frac{\varepsilon _{j}}{4})$ by (\ref{7-1.75}). Now by (\ref{6-1a}) for $%
\int_{I}\sigma _{j}$ $=$ $1$ and choosing coordinates with $h(z(x),\bar{z}%
(x))$ $=$ $1$ at this $x$ so that $l^{m}(x)$ $=$ $h^{m}|w|^{2m}$ $=$ $%
|w|^{2m}$ and using $f(\eta )$ $=$ $1$ (cf. Lemma \ref{A} $ii)$ below)$,$
the RHS of (\ref{7-0}) equals:%
\begin{equation}
\int_{|\eta |\in \mathbb{R}^{+}}|w|^{2m}|\eta |^{2m}dv_{m}(|\eta ||w|)%
\overset{\tilde{\eta}=w\eta }{=}\int_{|\tilde{\eta}|\in \mathbb{R}^{+}}|%
\tilde{\eta}|^{2m}dv_{m}(|\tilde{\eta}|)\overset{(\ref{3-16.75})}{=}1.
\label{7-5.5}
\end{equation}

Remark that the resulting constant $1$ in (\ref{7-5.5}) hence in (\ref{7-0})
plays an implicit yet crucial role in many places of our computation (cf. (%
\ref{HQDI-1}), (\ref{STrSE}), (\ref{7-38.5}) and (\ref{7-38.75}); also (\ref%
{3-43.5}), (\ref{adj15})). We are not going to elaborate on the question
whether it would still be possible to obtain the existence of a local index
density of Theorem \ref{main_theorem} if the metric $G_{a,m}$ used here did
not possess this unity-property.

Substituting (\ref{7-0}) and (\ref{7-5.5}) into (\ref{HQDI}), we obtain%
\begin{equation}
\text{Part I of }(H_{m,t}^{j}\circ \pi _{m})(x,x)=\varphi
_{j}(x)l^{m}(x)K_{t}^{j}(z,z)\text{ (}\tau _{j}(z)=1\text{ on supp }\varphi
_{j}).  \label{7.8-5}
\end{equation}

\noindent Recall the asymptotic expansion of the (ordinary, local) heat
kernel $K_{t}^{j}(z,z^{\prime })$ (see (\ref{OAE}) or \cite[(5.19) on p.76]%
{CHT}):%
\begin{eqnarray}
K_{t}^{j}(z,z^{\prime }) &=&e^{-\frac{\tilde{d}_{M}^{2}(z,z^{\prime })}{4t}%
}K^{j}(t,z,z^{\prime })  \label{7-5'} \\
K^{j}(t,z,z^{\prime }) &\sim &t^{-n+1}b_{n-1}(z,z^{\prime
})+t^{-n+2}b_{n-2}(z,z^{\prime })+...  \notag
\end{eqnarray}%
\noindent as $t\rightarrow 0^{+}$ for $z,z^{\prime }$ in $V_{j},$ where $%
\tilde{d}_{M}$ denotes the distance function associated with the metric $\pi
^{\ast }g_{M}|_{V_{j}}$ (cf. (\ref{Ga}) with lines above).

It is not difficult to see via Theorem \ref{thm2-1} (with its proof) that $%
\tilde{d}_{M}$ is the distance function on orbifold charts, associated with
the metric $g_{M}$ on $M=\Sigma /\mathbb{\sigma }$ (via the projection $%
\Sigma \rightarrow \Sigma /\mathbb{\sigma }).$ The reader is warned that $%
\tilde{d}_{M}$ is in general not the distance function on $M$ (unless the
globally free case).

Further, the coefficients $b_{s}(z,z^{\prime })$ $(s=n-1,$ $n-2,$ $\cdot
\cdot \cdot )$ in (\ref{7-5'}) depend only on $(x,y)$ (with $z(x)=z,$ $%
z(y)=z^{\prime })$ and are independent of the choice of charts $D_{j}$ $\ni $
$x,$ $y.$ For, it is well known (\cite[Chapter 2]{BGV}) that the
coefficients depend only on the local geometry; since the local geometry we
use consists of $\pi ^{\ast }g_{M}$ and the $\mathbb{C}^{\ast }$-invariant
metric on $L_{\Sigma }^{\ast }$ (see (\ref{6-1'}), (\ref{DiracLap}) and Step
1 of Section 3), these can be regarded as local geometry data on the
orbifold $M$ $=$ $\Sigma /\sigma $ (with $\pi :\Sigma $ $\rightarrow $ $M)$
thus intrinsic in nature.

For every compact set $K$ $\subset $ $V_{j},$ there is a constant $C_{K}>1$
such that%
\begin{equation}
\frac{1}{C_{K}}|z-z^{\prime }|\leq \tilde{d}_{M}(z,z^{\prime })\leq
C_{K}|z-z^{\prime }|.  \label{7-5"}
\end{equation}%
\noindent Set $b_{s}(z):=b_{s}(z,z),$ $s\leq n-1.$ Note that $b_{s}(z)$
(dependent on $z=z(x))$ are independent of $j$ as just mentioned. Plugging (%
\ref{7-5'}) into (\ref{7.8-5}) and noting $\tilde{d}_{M}^{2}(z,z)$ $=$ $0$
yields 
\begin{eqnarray}
&&\text{Part I of }\sum_{j}(H_{m,t}^{j}\circ \pi _{m})(x,x)  \label{HQDI-1}
\\
&\sim &t^{-n+1}\alpha _{n-1}(x)+t^{-n+2}\alpha _{n-2}(x)+\text{ }\cdot \cdot
\cdot \text{ as }t\rightarrow 0^{+}  \notag
\end{eqnarray}

\noindent where, for $s\leq n-1,$ 
\begin{equation}
\alpha _{s}(x)=b_{s}(z(x))l^{m}(x).  \label{HQDI-2}
\end{equation}

\noindent This finishes the computation of (\ref{HQD}) for $\xi $ in Part I.

\bigskip

\textbf{Part II of the RHS of (\ref{HQD}): }The computation for $\xi $ in
Part II of (\ref{HQD}) takes some extra work for which we start with the
following set-up. First recall that $\frac{2\pi }{p_{j}},$ $p$ $=$ $p_{1}$ $%
< $ $p_{2}$ $<$ $...<p_{k},$ denote all possible periods of the locally free
action $\sigma |_{S^{1}}.$ Define $\Sigma _{p_{j}}$ $:=$ $\{x\in \Sigma $ $:$
the period of $x$ is $\frac{2\pi }{p_{j}}\}$ and recall $\Sigma _{\text{sing}%
}$ $:=$ $\cup _{j=2}^{k}\Sigma _{p_{j}}$ (cf. (\ref{1-4})). Let $d(\cdot
,\cdot )$ denote the distance function on $\Sigma $ with respect to the
metric $G_{a,m}.$

Recall that we have the case $ii)$ stated after (\ref{6-1g}). We are going
to be more precise about it in (\ref{claim}) below. Let us start with

\begin{definition}
\label{d7-1.5} (cf. Remark \ref{7.3-5} for geometrical aspects) $S$ $:=$ $%
\{le^{-i\hat{\gamma}}$ $:$ $0<\frac{\varepsilon _{j}}{2}\leq \hat{\gamma}%
\leq \frac{2\pi }{p}-\frac{\varepsilon _{j}}{2},$ $l\in \mathbb{R}^{+}\}.$%
\begin{equation}
\hat{d}(x,\Sigma _{\text{sing}}):=\inf \{d(s\circ x,x):s\in S\}\geq 0.
\label{distsing}
\end{equation}
\end{definition}

We claim the existence of a constant $\hat{\varepsilon}_{0}$ $>$ $0$
satisfying the following. Let $x$ $\in $ $\Sigma \backslash \Sigma _{\text{%
sing}}$ and $x$ $=(z,w)\in $ $\hat{W}_{j}$ (see (\ref{7-1.75}))$.$%
\begin{eqnarray}
\text{Suppose }e^{-i\gamma }\circ x &=&(\tilde{z},\tilde{w})\in \hat{W}_{j}%
\text{ for some }\gamma \in \lbrack \varepsilon _{j},\frac{2\pi }{p}%
-\varepsilon _{j}].  \label{claim} \\
\text{Then }|\tilde{z}-z|\text{ } &\geq &\hat{\varepsilon}_{0}\hat{d}%
(x,\Sigma _{\text{sing}})\text{ }>0\text{.}  \notag
\end{eqnarray}

\noindent The Case $ii)$ after (\ref{6-1g}) has mentioned that the above $z,$
$\tilde{z}$ could be different. Here (\ref{claim}) confirms a positive lower
bound for $|\tilde{z}-z|$ (see \cite[(6.5)]{CHT} for a statement similar to (%
\ref{claim})).

\proof%
\textbf{\ (of (\ref{claim})) }First it can be verified that there is an $%
\hat{\varepsilon}_{0}>0$ independent of $w$ such that $|\tilde{z}-z|$ $\geq $
$\hat{\varepsilon}_{0}d((\tilde{z},w),(z,w)).$ For this verification, we
content ourselves with referring to (\ref{metric}) and (\ref{M1-1}) where $%
G_{a,m}$ is seen to be uniformly bounded along the $w$-direction. For
instance, choosing the curve ($(1-t)z+t\tilde{z},w)$ with estimation of its
length implies this inequality. With $\tilde{w}$ in (\ref{claim}) writing $%
\xi _{1}$ $=$ $w/\tilde{w}$ $=$ $|\xi _{1}|e^{i\phi }$ with $-\varepsilon
_{j}/2<\phi <\varepsilon _{j}/2$ (by the $\hat{W}_{j}$-condition) $,$ we
note that $\xi _{1}\circ (\tilde{z},\tilde{w})$ $=$ $(\tilde{z},\xi _{1}%
\tilde{w})$ $=$ $(\tilde{z},w)$ $\in $ $W_{j}$ $\supset $ $\hat{W}_{j}$ (see
cases $i)$ and $ii)$ after (\ref{6-1g}))$.$ Then we have%
\begin{eqnarray}
|\tilde{z}-z|\text{ } &\geq &\text{ }\hat{\varepsilon}_{0}d((\tilde{z}%
,w),(z,w))=\hat{\varepsilon}_{0}d(\xi _{1}\circ (\tilde{z},\tilde{w}),(z,w))
\label{Dist} \\
&=&\hat{\varepsilon}_{0}d((\xi _{1}e^{-i\gamma })\circ x,x)\geq \hat{%
\varepsilon}_{0}\hat{d}(x,\Sigma _{\text{sing}})\text{ (using (}\ref%
{distsing}),\text{ (\ref{claim})).}  \notag
\end{eqnarray}

To see $\hat{d}(x,\Sigma _{\text{sing}})$ $>$ $0,$ since $S$ in Definition %
\ref{d7-1.5} is clearly disjoint from the isotropy group at $x,$ i.e. $x$ $%
\notin $ $S\circ x$ and further, the orbit $S\circ x$ is closed (cf. (\ref%
{1-0}), (\ref{1-1})), by the definition of $\hat{d}$ in (\ref{distsing}) the
desired strict positivity follows.

\endproof%

\begin{remark}
\label{7.3-5} One expects that the function $\hat{d}(x,\Sigma _{\text{sing}%
}) $ is comparable to the genuine distance function $d(x,\Sigma _{\text{sing}%
}). $ See \cite[Theorem 6.7]{CHT} in the CR context.
\end{remark}

To proceed further, we need the following fact:

\begin{lemma}
\label{lemma7-1} It holds that $($see (\ref{lq0}) for $l(x))$%
\begin{eqnarray}
&&l(e^{i\gamma }\circ x) = l(x)\text{ provided that }e^{i\gamma }\circ x%
\text{ (}e^{i\gamma }\in S^{1}\subset \mathbb{C}^{\ast })\text{ and }x\text{
lie in}  \label{R+A} \\
&&\text{ \ \ the same chart }W_{j}\text{ with coordinates }(z,w)\text{ (cf.
Notation \ref{n-6.1})}.  \notag
\end{eqnarray}
\end{lemma}

\proof
Let $D_{0}$ $=$ $W_{j},$ $D_{1},$ $...$ $,$ $D_{L\text{ }},$ $D_{L+1}$ be a
sequence of patches with coordinates $(z_{k},w_{k})$ on each $D_{k},$ and $%
x_{0},$ $x_{1},$ $...$ , $x_{L},$ $x_{L+1}$ be a sequence of points. Assume
the following: $D_{L+1}\equiv D_{0}$, $(z_{0},w_{0})$ $\equiv $ $%
(z_{L+1},w_{L+1})$ $\equiv $ $(z,w)$ and $x_{0}:=x$ $\in $ $D_{0}$, $x_{k}$ $%
\in $ $D_{k}\cap D_{k-1},$ $1$ $\leq $ $k$ $\leq $ $L,$ $x_{L+1}$ $\in $ $%
D_{0}\cap D_{L}$ such that 
\begin{equation}
x_{k+1}=e^{i\gamma _{k}}\circ x_{k},\text{ }0\leq k\leq L.  \label{angle}
\end{equation}%
\noindent We further assume: for every $0$ $\leq $ $k$ $\leq $ $L,$ if $%
\theta \in \lbrack 0,\gamma _{k}]$ then $e^{i\theta }\circ x_{k}\in D_{k}.$
This gives%
\begin{equation}
w_{k}(e^{i\theta }\circ x_{k})=e^{i\theta }w_{k}(x_{k})  \label{7-12.5}
\end{equation}

\noindent by (\ref{1-1}). Since $l(q)=h(z(q),\bar{z}(q))w(q)\bar{w}(q)$ in (%
\ref{lq0}), (\ref{lq}) is independent of the choice of coordinates, from (%
\ref{angle}) and the fact that $h$ is invariant under rotation (by $\gamma
_{k})$ by using (\ref{7-12.5}) and (\ref{lq}), it follows that%
\begin{eqnarray}
l(x_{k+1}) &=&h(z_{k}(x_{k+1}),\bar{z}_{k}(x_{k+1}))w_{k}(x_{k+1})\bar{w}%
_{k}(x_{k+1})  \label{lxk} \\
&=&h(z_{k}(x_{k}),\bar{z}_{k}(x_{k}))w_{k}(x_{k})\bar{w}_{k}(x_{k})  \notag
\\
&=&l(x_{k}),\text{ }0\leq k\leq L.  \notag
\end{eqnarray}

\noindent Clearly (\ref{R+A}) follows from (\ref{lxk}).

\endproof%

\begin{remark}
\label{7-11b} The above proof actually shows that (\ref{R+A}) holds
unconditionally. As an important application, one sees that the metric $%
G_{a,m}$ on $\Sigma $ (in Section 3) is $S^{1}$-invariant via (\ref{3-7.5}),
(\ref{Ga}), (\ref{3-6-1}) and Lemma \ref{A} $iv)$.
\end{remark}

\begin{corollary}
\label{7.5-5} It holds that for all $\xi \in \mathbb{C}^{\ast }$ and for all 
$x\in \Sigma ,$%
\begin{equation}
l^{m}(\xi ^{-1}\circ x)=|\xi |^{-2m}l^{m}(x).  \label{7-1-1}
\end{equation}
\end{corollary}

\proof
By Lemma \ref{lemma7-1} and writing $\xi $ $=$ $|\xi |e^{i\gamma }$, we see
that (via Remark \ref{7-11b})%
\begin{equation*}
l(\xi ^{-1}\circ x)=l(|\xi |^{-1}\circ (e^{-i\gamma }\circ x))=|\xi
|^{-2}l(e^{-i\gamma }\circ x)=|\xi |^{-2}l(x)
\end{equation*}

\noindent and hence the lemma.

\endproof%

In the remaining part of this section, we will omit $``\circ "$ in the
notation of $\mathbb{C}^{\ast }$-action.

Let us now continue with Part II of (\ref{HQD}). Recall $H_{m,t}^{j}$ in (%
\ref{6-1}) (with the footnote as in (\ref{6-1g}))$:$%
\begin{eqnarray}
&&(H_{m,t}^{j}(x,\xi ^{-1}x)\circ (\sigma _{\xi }^{\ast })_{\xi ^{-1}x})\bar{%
\xi}^{m}=\varphi _{j}(x)w^{m}(x)K_{t}^{j}(z(x),z(\xi ^{-1}x))\circ (\sigma
_{\xi }^{\ast })_{\xi ^{-1}x}  \label{7-2} \\
&&\text{ \ \ }w^{-m}(\xi ^{-1}x)l^{m}(\xi ^{-1}x)\tau _{j}(z(\xi
^{-1}x))\sigma _{j}(\phi (\xi ^{-1}x))\bar{\xi}^{m}  \notag
\end{eqnarray}

To integrate (\ref{7-2}) over $I$ $=$ $[\varepsilon _{j},\frac{2\pi }{p}%
-\varepsilon _{j}]$ for Part II$,$ we shall now divide $I$ $=$ $J\cup
J^{\prime }$ where $J$ is the subset of those $\gamma $ $\in $ $[\varepsilon
_{j},\frac{2\pi }{p}-\varepsilon _{j}]$ such that $e^{-i\gamma }x$ $\in $ $%
\hat{W}_{j}$ (Notation \ref{n-6.1}) and $J^{\prime }$ is the complement of $%
J $ in $[\varepsilon _{j},\frac{2\pi }{p}-\varepsilon _{j}].$ If we denote
by $(\tilde{z},\tilde{w})$ the coordinates of $e^{-i\gamma }x$ with $\gamma $
$\in $ $J,$ then $\tilde{z}$ $\neq $ $z$ by (\ref{claim}) (noting that $%
\frac{2\pi }{p}$ is the period of $x$). We suppress the dependence of $J$ on 
$x$ since $x$ is fixed throughout.

From (\ref{6-1f}) of Proposition \ref{l-6-0a} and the notation there, it
follows that with $\xi $ $=$ $|\xi |e^{i\gamma }$%
\begin{eqnarray}
&&\text{Part II of }(H_{m,t}^{j}\circ \pi _{m})(x,x)  \label{7-2-1} \\
&{:=}&\int_{\varepsilon _{j}}^{\frac{2\pi }{p}-\varepsilon _{j}}\int_{%
\mathbb{R}^{+}}(H_{m,t}^{j}(x,\xi ^{-1}x)\circ (\sigma _{\xi }^{\ast })_{\xi
^{-1}x})\bar{\xi}^{m}d\mu _{x,m}(\xi )  \notag \\
&=&(\int_{J}+\int_{J^{\prime }})\int_{\mathbb{R}^{+}}(H_{m,t}^{j}(x,\xi
^{-1}x)\circ (\sigma _{\xi }^{\ast })_{\xi ^{-1}x})\bar{\xi}^{m}d\mu
_{x,m}(\xi )  \notag \\
&=&\int_{J}\int_{\mathbb{R}^{+}}(H_{m,t}^{j}(x,\xi ^{-1}x)\circ (\sigma
_{\xi }^{\ast })_{\xi ^{-1}x})\bar{\xi}^{m}d\mu _{x,m}(\xi )  \notag
\end{eqnarray}%
\noindent where the integral over $J^{\prime }$ vanishes since in $%
H_{m,t}^{j}$ (\ref{6-1}) the term $\tau _{j}(z(\xi ^{-1}x))\sigma _{j}(\phi
(\xi ^{-1}x))$ $=$ $0$ for $e^{-i\gamma }x$ $\notin $ $\hat{W}_{j}$ where $%
\gamma \in J^{\prime }$ (see (\ref{6-1}) and items $ii),$ $iii)$ after
Notation \ref{n-6.1}).

Next for the integral over $J$ in the RHS of (\ref{7-2-1}) we need some
preparations, which go from (\ref{7-2-2}) until Corollary \ref{7.6-5}.

Recall from (\ref{6-1e1}) and remarks below it that $d\mu _{x,m}(\xi )$
equals ($\tau _{x}$ $:$ $\mathbb{C}^{\ast }\rightarrow \Sigma $ defined by $%
\tau _{x}(\lambda )=\lambda \circ x$)%
\begin{eqnarray}
&&d\mu _{x,m}(\xi )=\frac{l^{m}(\xi ^{-1}x)dv_{\Sigma ,m}(\xi ^{-1}x)\wedge
dv_{f,m}(x)}{dv_{\Sigma ,m}(x)}  \label{7-2-2} \\
&\overset{(\ref{7-1-1})(\ref{volume})}{=}&|\xi |^{-2m}l^{m}(x)(\tau
_{x}^{\ast }dv_{f,m})(\xi ^{-1})\frac{dv_{M}(z(\xi ^{-1}x))}{dv_{M}(z(x))}. 
\notag
\end{eqnarray}%
\noindent Note also (cf. (\ref{3-18.75}) and $i)$ of Lemma \ref{A}): 
\begin{equation}
(\tau _{x}^{\ast }dv_{f,m})(\xi ^{-1})=dv_{m}(|\xi |^{-1}|w|)\wedge
dv(-\gamma )/2\pi ,\text{ }\xi =|\xi |e^{i\gamma }  \label{7-2-3}
\end{equation}%
\noindent and the invariantly defined integral:%
\begin{equation}
\int_{|\xi |\in \mathbb{R}^{+}}l^{m}(|\xi |^{-1}x)dv_{m}(|\xi |^{-1}x)=1
\label{7-20.5}
\end{equation}

\noindent (cf. (\ref{3-19.5}) where we chose $h(z_{0},\bar{z}_{0})$ $=$ $1$
and the notation $dv_{m}(|\xi |^{-1}x)$ is for $dv_{m}(|\xi |^{-1}|w|)$). By
our choice of the holomorphic coordinate $w,$ we have (see (\ref{1-1}) for $%
\lambda $ $\in $ $\mathbb{R}^{+}):$%
\begin{equation}
|w|^{-m}(\xi ^{-1}x)|\xi |^{-m}=\frac{1}{|w|^{m}(e^{-i\gamma }x)},\text{ }%
e^{-i\gamma }=\xi ^{-1}|\xi |\in S^{1}\subset \mathbb{C}^{\ast }.
\label{7-2-4}
\end{equation}

For one more preparation, a technical lemma is in order. Items $iii),$ $iv)$
below have been used in (\ref{proj12}), (\ref{hm}) respectively.

\begin{lemma}
\label{A} It holds that

$i)$ $\frac{|w(x)|}{|w(e^{-i\gamma }x)|}=1$ for $e^{-i\gamma }\in S^{1}$
such that $e^{-i\gamma }x$ $\in $ $W_{j}$ of Notation \ref{n-6.1}$.$

$ii)$ $\frac{dv_{M}(z(\xi ^{-1}x))}{dv_{M}(z(x))}$ $=$ $1$ for $\xi \in 
\mathbb{C}^{\ast }$ such that $\xi ^{-1}x$ $\in $ $W_{j}.$

$iii)$ In the product space $\mathbb{C}^{\ast }\times \Sigma ,$ under the
transformation $\tilde{\sigma}:(\xi ,x)$ $\rightarrow $ $(\xi ^{-1},y$ $=$ $%
\xi \circ x)$ we have 
\begin{equation}
(\tau _{x}^{\ast }dv_{f,m})(\xi )dv_{\Sigma ,m}(x)=\{\tilde{\sigma}^{\ast
}[(\tau _{y}^{\ast }dv_{f,m})(\xi ^{-1})dv_{\Sigma ,m}(y)]\}(\xi ,x).
\label{VT}
\end{equation}

$iv)$ For $h$ of (\ref{lq}) and notations there, $h(z^{\prime },\overline{%
z^{\prime }})=h(z,\bar{z})$ if $\{p,$ $e^{-i\gamma }\circ p\}$ $\subset $ $%
D_{j}$ and $z^{\prime }=z(e^{-i\gamma }\circ p).$ Moreover the global 2-form 
$\partial \bar{\partial}\log h$ on $\Sigma $ ($=\partial _{z}\bar{\partial}%
_{z}\log h,$ see (\ref{3-6-1})$)$ is $S^{1}$-invariant.
\end{lemma}

\proof
For $i)$, if $\gamma $ is small this is automatic.(cf. Case $i)$ after (\ref%
{6-1g})). In general, by (\ref{7-1-1}) and for $\xi ^{-1}x\in W_{j}$ with $%
|\arg \xi |$ small%
\begin{equation}
l(\xi ^{-1}x)=h(z(\xi ^{-1}x),\bar{z}(\xi ^{-1}x))|w|^{2}(\xi ^{-1}x)
\label{7-15.5}
\end{equation}%
\noindent (cf. (\ref{lq})) we obtain%
\begin{equation}
\frac{|w|(x)}{|w|(e^{-i\gamma }x)}=\frac{h^{1/2}(z(e^{-i\gamma }x),\bar{z}%
(e^{-i\gamma }x))}{h^{1/2}(z(x),\bar{z}(x))}.  \label{7-2-5}
\end{equation}

\noindent Resorting to the fact that $h$ is invariant under local rotations
as remarked in the proof of Lemma \ref{lemma7-1}, in a similar way we use
the local rotations step by step as in (\ref{lxk}). This, together with (\ref%
{7-2-5}), leads to the first equality.

For the second equality, since $dv_{M}$ $=$ the volume form on $V_{j}\times
\{0\}\times \{1\}$ $\subset $ $\Sigma $ induced by $\pi ^{\ast }g_{M}$ (cf. (%
\ref{volume})) which is trivially $\mathbb{C}^{\ast }$-invariant (see the
beginning of step 2 in Section 3), the conclusion follows.

For $iii),$ first observe that (recalling $\sigma :\mathbb{C}^{\ast }\times
\Sigma \rightarrow \Sigma $ defined by $\sigma (\lambda ,p)$ $:=$ $\lambda
\circ p)$%
\begin{equation}
LHS\text{ of }(\ref{VT})=(\sigma ^{\ast }dv_{f,m})(\xi ,x)\wedge dv_{\Sigma
,m}(x)\text{,}  \label{7-2-6}
\end{equation}%
\begin{equation}
RHS\text{ of }(\ref{VT})=\{\tilde{\sigma}^{\ast }[(\sigma ^{\ast
}dv_{f,m})(\xi ^{-1},y)\wedge dv_{\Sigma ,m}(y)]\}(\xi ,x).  \label{7-2-7}
\end{equation}

\noindent It then follows from (\ref{volume}) that%
\begin{equation}
RHS\text{ of (\ref{7-2-6})}=(\sigma ^{\ast }dv_{f,m})(\xi ,x)\wedge \pi
^{\ast }dv_{M}(x)\wedge dv_{f,m}(x)\text{,}  \label{7-2-8}
\end{equation}

\begin{eqnarray}
RHS\text{ of (\ref{7-2-7})} &=&\{\tilde{\sigma}^{\ast }[(\sigma ^{\ast
}dv_{f,m})(\xi ^{-1},y)]\wedge \tilde{\sigma}^{\ast }(dv_{\Sigma
,m}(y))\}(\xi ,x)  \label{7-2-9} \\
&=&dv_{f,m}(x)\wedge (\sigma ^{\ast }\pi ^{\ast }dv_{M})(\xi ,x)\wedge
(\sigma ^{\ast }dv_{f,m})(\xi ,x)  \notag
\end{eqnarray}

\noindent Comparing (\ref{7-2-8}) with (\ref{7-2-9}) and noting that $%
(\sigma ^{\ast }\pi ^{\ast }dv_{M})(\xi ,x)$ $=$ $\pi ^{\ast }dv_{M}(x)$
since $\pi ^{\ast }dv_{M}$ is $\mathbb{C}^{\ast }$-invariant (alternatively,
computing it using $(z,w)$ coordinates is recommended), we conclude that $%
RHS $ of (\ref{7-2-8}) equals $RHS$ of (\ref{7-2-9}). Hence (\ref{VT})
follows from (\ref{7-2-6}) and (\ref{7-2-7}).

For $iv)$ the first assertion follows from $i)$ of this lemma, (\ref{lq})
and (\ref{R+A}). It is clear that $e^{\pm i\varepsilon }$ preserves the
global form $\partial \bar{\partial}\log h$ for $|\varepsilon |<<1$ so does $%
e^{i\varepsilon q}$ for any $q$ $\in $ $\mathbb{Z},$ giving the second
assertion.



\endproof%

\begin{corollary}
\label{7.6-5} It holds that%
\begin{equation}
\int_{\xi \in \mathbb{C}^{\ast }}d\mu _{x,m}(\xi )=1.  \label{A-1}
\end{equation}
\end{corollary}

\proof
(\ref{A-1}) follows from (\ref{7-2-2})$,$ $ii)$ of Lemma \ref{A}, (\ref%
{7-2-3}) and (\ref{7-20.5}).

\endproof%

We can now estimate Part II of $(H_{m,t}^{j}\circ \pi _{m})(x,x)$ as follows.

\begin{proposition}
\label{P-7-35} For any $N_{0}$ $\geq $ $n+1$ there exists $\delta =\delta
(N_{0})$ $>$ $0$ and $C_{N_{0}}$ $>$ $0$ such that for $0<t<\delta $ it
holds that%
\begin{equation}
|\text{Part II of }(H_{m,t}^{j}\circ \pi _{m})(x,x)|\leq
C_{N_{0}}l^{m}(x)t^{-(n-1)}e^{-\frac{\hat{\varepsilon}_{0}^{\prime }\hat{d}%
(x,\Sigma _{\text{sing}})^{2}}{t}}.  \label{HQDIIa}
\end{equation}%
for $x$ $\in $ $(\Sigma \backslash \Sigma _{\text{sing}})\cap \hat{W}_{j}$,
where $C_{N_{0}}$ and $\hat{\varepsilon}_{0}^{\prime }$ are independent of $%
x.$
\end{proposition}

\begin{proof}
In view of (\ref{7-2-1}) we need to estimate a certain integral of $%
(H_{m,t}^{j}(x,\xi ^{-1}x)\circ (\sigma _{\xi }^{\ast })_{\xi ^{-1}x})\bar{%
\xi}^{m}d\mu _{x,m}(\xi ).$ First we deal with $H_{m,t}^{j}(x,\xi ^{-1}x)$.
Let both $x$ and $\xi ^{-1}x$ be in ($\Sigma \backslash \Sigma _{\text{sing}%
})$ $\cap $ $\hat{W}_{j}$ as required by the reduction of (\ref{HQDIIa}) to
the $J$-part integral in (\ref{7-2-1})$.$ By Lemma \ref{A} $i)$ and (\ref%
{7-1-1}) we obtain the modulus of the RHS of $H_{m,t}^{j}(x,\xi ^{-1}x)$ (%
\ref{6-1}) for terms except $K_{t}^{j}$ and cutoff functions: (noting that ($%
\zeta ,\eta )$ $=$ $(z,w)$ as $\xi ^{-1}x$ lies in $\hat{W}_{j}$ by
assumption$)$%
\begin{eqnarray}
&&|w^{m}(x)||w(\xi ^{-1}x)^{-m}|l(\xi ^{-1}x)^{m}  \label{7-b} \\
&=&|w^{m}(x)||\xi |^{m}|w(x)^{-m}||\xi |^{-2m}l(x)^{m}=l(x)^{m}|\xi |^{-m}. 
\notag
\end{eqnarray}

\noindent By (\ref{6-1}) and (\ref{7-b}) we can then estimate the integrand (%
\ref{7-2}) of (\ref{7-2-1}): noting that $\sigma _{\xi }^{\ast }$ leaves $%
\pi ^{\ast }\mathcal{E}_{M}$ invariant (with respect to the basis $\tilde{%
\eta}^{I_{q}}$ in Footnote$^{6}$ $\sigma _{\xi }^{\ast }$ is the identity
matrix)%
\begin{eqnarray}
&&|(H_{m,t}^{j}(x,\xi ^{-1}x)\circ (\sigma _{\xi }^{\ast })_{\xi ^{-1}x})%
\bar{\xi}^{m}d\mu _{x,m}(\xi )|  \label{7-c} \\
&\leq &C_{1}l(x)^{m}|K_{t}^{j}(z(x),z(\xi ^{-1}x))|d\mu _{x,m}(\xi ).  \notag
\end{eqnarray}

\noindent where $C_{1}$ is a constant independent of $x$ and $\xi .$ Now
from the asymptotic expansion (\ref{7-5'}) of $K_{t}^{j}$ it follows that
for $N_{0}$ $>$ $n$ $=$ $\frac{1}{2}\dim _{\mathbb{R}}V_{j}$ $+$ $1$ (cf. 
\cite[Theorem 2.23, p.81]{BGV}) there exists $\delta =\delta (N_{0})$ $>$ $0 
$ and $C_{N_{0}}^{\prime }$ $>$ $0$ such that for $0<t<\delta $ and $x,$ $%
\xi ^{-1}x$ $\in $ $\Sigma \backslash \Sigma _{\text{sing}}\cap \hat{W}_{j}$%
\begin{equation}
|K_{t}^{j}(z(x),z(\xi ^{-1}x))|\leq C_{N_{0}}^{\prime }e^{-\frac{\tilde{d}%
_{M}^{2}(z,z^{\prime })}{4t}}t^{-(n-1)}  \label{7-e}
\end{equation}

\noindent where we have written $z$ $=$ $z(x)$, $z^{\prime }$ $=$ $z(\xi
^{-1}x)$ and $C_{N_{0}}^{\prime }$ is independent of $x$ and $\xi $ (such
that $\xi ^{-1}x$ $\in $ $\hat{W}_{j}).$ By $(\ref{7-5"})$ and $(\ref{claim}%
) $ we obtain%
\begin{equation}
e^{-\frac{\tilde{d}_{M}^{2}(z,z^{\prime })}{4t}}\leq e^{-\frac{\hat{%
\varepsilon}_{0}^{\prime }\hat{d}(x,\Sigma _{\text{sing}})^{2}}{t}}
\label{7-a}
\end{equation}

\noindent with $\hat{\varepsilon}_{0}^{\prime }=\hat{\varepsilon}%
_{0}^{2}/C_{K}^{2}.$ Using (\ref{7-2-1}), (\ref{HQDIIa}) follows from (\ref%
{7-c}), (\ref{7-e}), (\ref{7-a}) and (\ref{A-1}) with $C_{N_{0}}$ $=$ $%
C_{1}C_{N_{0}}^{\prime }.$
\end{proof}

Having just worked out (\ref{HQD}), we can now prove Theorem \ref{AHKE}
stated in the Introduction. Before going on, we pause to give a proof of (%
\ref{6-19.5}) as an interlude since we have now learned many properties
about $d\mu _{x,m}(\xi ).$ The following proof uses $\alpha _{k}$
constructed in the proof of Lemma \ref{l-bdp} in an essential way, but the
reader may skip this proof and come back to it in due course.

\proof%
\textbf{\ (of (\ref{6-19.5}) in Proposition \ref{l-6-0a}) }Substituting (\ref%
{7-2-2}) (with $x$ replaced by $y)$ into (\ref{6-1f}) and using (\ref{7-2-3}%
), (\ref{6-1g}) one is able to obtain $p_{m,t}^{0,j}(x,y)$ in (\ref{6-19.5})
as follows:%
\begin{eqnarray}
&&\text{ \ \ \ \ \ \ \ \ \ \ \ \ \ \ }p_{m,t}^{0,j}(x,y)\overset{(\ref{7-1-1}%
)}{=}\int_{\xi \in \mathbb{C}^{\ast }}\bigl\{\varphi
_{j}(x)K_{t}^{j}(z(x),\zeta (\xi ^{-1}y))  \label{kmt} \\
&&\text{ \ \ \ \ \ \ \ \ \ \ \ \ \ \ \ \ \ \ \ \ }\frac{\bar{\eta}^{m}(\xi
^{-1}y)}{\bar{\eta}^{m}(y)}\bar{\xi}^{m}h^{m}(\zeta (\xi ^{-1}y),\bar{\zeta}%
(\xi ^{-1}y))\tau _{j}(\zeta (\xi ^{-1}y))\sigma _{j}(\vartheta (\xi ^{-1}y))
\notag \\
&&\text{ \ \ \ \ \ \ \ \ \ \ \ \ \ \ \ \ \ \ \ \ \ \ \ \ }l^{m}(|\xi |^{-1}y)%
\bigr\}(\sigma _{\xi }^{\ast })_{\xi ^{-1}y}dv_{m}(|\xi |^{-1}|\eta
(y)|)\wedge \frac{dv(-\gamma )}{2\pi }\frac{dv_{M}(\zeta (\xi ^{-1}y))}{%
dv_{M}(\zeta (y))};  \notag
\end{eqnarray}

\noindent for the above noting that one replaces $l^{m}(x)$ in (\ref{7-2-2})
by 
\begin{equation*}
l^{m}(y)=h^{m}(\zeta (\xi ^{-1}y),\bar{\zeta}(\xi ^{-1}y))\eta ^{m}(y)\bar{%
\eta}^{m}(y)
\end{equation*}
\noindent using Lemma \ref{A} $iv).$\ Observe that in (\ref{kmt})

\begin{eqnarray}
&&\bar{\eta}^{m}(\xi ^{-1}y)h^{m}(\zeta (\xi ^{-1}y),\bar{\zeta}(\xi
^{-1}y))/\bar{\eta}^{m}(y)\overset{l=h\eta \bar{\eta}}{=}\frac{l^{m}(\xi
^{-1}y)}{\eta ^{m}(\xi ^{-1}y)\bar{\eta}^{m}(y)}  \label{kmt1} \\
&\overset{(\ref{1-2.75})+(\ref{7-1-1})}{=}&\frac{|\xi |^{-2m}l^{m}(y)}{|\xi
|^{-m}\eta ^{m}(e^{-i\gamma }y)\bar{\eta}^{m}(y)}=\frac{1}{\eta
^{m}(e^{-i\gamma }y)\bar{\eta}^{m}(y)}|\xi |^{-m}l^{m}(y)  \notag \\
&=&(\frac{\eta ^{m}(y)}{\eta ^{m}(e^{-i\gamma }y)}\frac{h^{m}(\zeta ,\bar{%
\zeta})}{l^{m}(y)})|\xi |^{-m}l^{m}(y)=h^{m}(\zeta ,\bar{\zeta})|\xi |^{-m}%
\frac{\eta ^{m}(y)}{\eta ^{m}(e^{-i\gamma }y)}.  \notag
\end{eqnarray}

\noindent Here $y$ is omitted in $\zeta $ $=$ $\zeta (y).$ By (\ref{kmt1}),
Lemma \ref{A} $ii)$ we reduce (\ref{kmt}) to (recalling $\xi =|\xi
|e^{i\gamma },$ $\eta (y)=|\eta |(y)e^{i\vartheta (y)})$ 
\begin{eqnarray}
&&\ \ \ \ \ \ p_{m,t}^{0,j}(x,y)=\varphi _{j}(x)h^{m}(\zeta (y),\bar{\zeta}%
(y))\int_{\xi \in \mathbb{C}^{\ast }}\bigl\{K_{t}^{j}(z(x),\zeta (\xi
^{-1}y))|\xi |^{-m}\bar{\xi}^{m}  \label{7K-20} \\
&&\frac{\eta ^{m}(y)}{\eta ^{m}(e^{-i\gamma }y)}\tau _{j}(\zeta (\xi
^{-1}y))\sigma _{j}(\vartheta (\xi ^{-1}y))l^{m}(|\xi |^{-1}y)\bigr\}(\sigma
_{\xi }^{\ast })_{\xi ^{-1}y}dv_{m}(|\xi |^{-1}|\eta (y)|)\frac{dv(-\gamma )%
}{2\pi }.  \notag
\end{eqnarray}

\noindent Here one trouble is $\eta (e^{-i\gamma }y)$ because for $\xi $ $%
\in $ $\mathbb{C}^{\ast }$ its angle $\gamma $ may be \textquotedblleft large%
$";$ one runs into the large angle action. To figure out $\eta (e^{-i\gamma
}y)$ we apply the construction of $\alpha _{k}$ in the proof of Lemma \ref%
{l-bdp} for coordinates ($\zeta ,\eta )$ of $y$ (replacing coordinates $%
(z,w) $ of $x$ there) so that $\vartheta (\alpha _{k}y)$ $=$ $0,$ $\alpha
_{0}=e^{-i\vartheta (y)},$ 
\begin{equation}
\zeta _{k}:=\zeta (\alpha _{k}y),  \label{7-38-1}
\end{equation}
\ \noindent $J_{k}$ $:=$ $\{\xi \in $ $\mathbb{C}^{\ast }:$ $-\varepsilon
_{j}<$ $\arg (\xi \alpha _{k}^{-1})$ $<\varepsilon _{j}\}$ ($\varepsilon
_{j} $ as in the local chart $D_{j}:=$ $V_{j}\times $ $(-\varepsilon
_{j},\varepsilon _{j})$ $\times \mathbb{R}^{+}).$ Similar to (\ref{7K-5a})
we now have%
\begin{equation}
\eta (\xi y)=\xi \alpha _{k}^{-1}\alpha _{0}\eta (y)\text{ for }\xi \in J_{k}
\label{7.38-1}
\end{equation}

\noindent where $k=$ $0,1,\cdot \cdot \cdot ,$ $\Lambda $ (some nonnegative
integer). Using the above formula gives%
\begin{equation}
\frac{\eta ^{m}(y)}{\eta ^{m}(e^{-i\gamma }y)}=e^{im\gamma }(\alpha
_{k}\alpha _{0}^{-1})^{m}\text{ for }e^{-i\gamma }\in J_{k}.  \label{7K-21}
\end{equation}

\noindent Two more formulas for use: For $\xi ^{-1}\in J_{k}$ we have in the
notation above (with Case $i)$ after (\ref{6-1g}) applied to $J_{k})$%
\begin{equation}
\zeta (\xi ^{-1}y)=\zeta ((\xi ^{-1}\alpha _{k}^{-1})\alpha _{k}y)=\zeta
(\alpha _{k}y)=\zeta _{k}\text{.}  \label{7K-22}
\end{equation}

\noindent Writing $\alpha _{k}=e^{-i\gamma _{k}}$ and $\xi =|\xi |e^{i\gamma
}$ we have$,$ for $\xi ^{-1}\in J_{k}$%
\begin{eqnarray}
\vartheta (\xi ^{-1}y) &=&\vartheta ((\xi ^{-1}\alpha _{k}^{-1})\alpha
_{k}y)=\vartheta (\xi ^{-1}\alpha _{k}^{-1})+\vartheta (\alpha _{k}y)
\label{7K-23} \\
&=&\vartheta (\xi ^{-1}\alpha _{k}^{-1})=-\gamma +\gamma _{k}\text{.}  \notag
\end{eqnarray}

\noindent Here $-\varepsilon _{j}<-\gamma +\gamma _{k}<\varepsilon _{j}$
since $\xi ^{-1}$ $\in $ $J_{k}.$ Namely (\ref{7K-22}) and (\ref{7K-23}) are
part of the coordinates of $\xi ^{-1}y.$

We are ready to substitute (\ref{7K-22}), $|\xi |^{-m}\bar{\xi}%
^{m}=e^{-im\gamma },$ (\ref{7K-21}) and (\ref{7K-23}) into (\ref{7K-20}),
giving (via the cut-off functions reducing it to a summation over smaller
regions of integration)

\begin{eqnarray}
&&p_{m,t}^{0,j}(x,y)=\varphi _{j}(x)h^{m}(\zeta (y),\bar{\zeta}%
(y))\sum_{k=0}^{\Lambda }\{K_{t}^{j}(z(x),\zeta _{k})\tau _{j}(\zeta
_{k})(\alpha _{k}\alpha _{0}^{-1})^{m}\}  \label{7K-24} \\
&&\circ (\sigma _{\alpha _{k}^{-1}}^{\ast })_{\alpha _{k}y}\cdot
\int_{\gamma =\gamma _{k}+\varepsilon _{j}}^{\gamma =\gamma _{k}-\varepsilon
_{j}}\sigma _{j}(\gamma _{k}-\gamma )\frac{dv(-\gamma )}{2\pi }\int_{|\xi
|\in \mathbb{R}^{+}}l^{m}(|\xi |^{-1}y)dv_{m}(|\xi |^{-1}y)  \notag
\end{eqnarray}

\noindent where $(\sigma _{\alpha _{k}^{-1}}^{\ast })_{\alpha _{k}y}$ is
from the fact that $(\sigma _{\xi }^{\ast })_{\xi ^{-1}y}$ is the map $\xi
^{-1}y$ $\rightarrow $ $y$ that pulls back a section at $y$ to one at $\xi
^{-1}y$ so that the resulting section over $\{\xi ^{-1}y\}_{\xi \in \mathbb{C%
}^{\ast }}$ behaves as a \textquotedblleft constant section" (in a piecewise
sense due to the large angle action). By $\bar{\gamma}=\gamma _{k}-\gamma ,$
the angular integral above gives $\int_{-\varepsilon _{j}}^{\varepsilon
_{j}}\sigma _{j}(\bar{\gamma})\frac{dv(\bar{\gamma})}{2\pi }=1$ by $(\ref%
{6-1a}),$ while by (\ref{7-20.5}) the last integral in (\ref{7K-24}) equals 1%
\footnote{%
It might seem that the computation here involves confusing sign issues. Let
us note that the top-forms involved are positive (see also (\ref{7-2-3}) and
the change of variable (if any) switches $\xi $ to $\xi ^{-1}$ (compare (\ref%
{7K-24})) which remains orientation preserving. The overall plus-sign is
thus obtained.}. Thus (\ref{7K-24}) is reduced to%
\begin{equation}
p_{m,t}^{0,j}(x,y)=\varphi _{j}(x)h^{m}(\zeta (y),\bar{\zeta}%
(y))\sum_{k=0}^{\Lambda }\{K_{t}^{j}(z(x),\zeta _{k})\tau _{j}(\zeta
_{k})(\alpha _{k}\alpha _{0}^{-1})^{m}\}\circ (\sigma _{\alpha
_{k}^{-1}}^{\ast })_{\alpha _{k}y}.  \label{7K-25}
\end{equation}

\noindent Note that $\zeta _{k}$ and $\alpha _{k}$ in (\ref{7K-25}) depend
on $y$ and that for the $x$-part $p_{m,t}^{0,j}(x,y)$ is independent of $%
|w|(x).$ Moreover, differentiating $\vartheta (\alpha y)$ in $\alpha $ $\in $
$S^{1}$ $(\subset \mathbb{C}^{\ast })$ is not zero by the locally freeness.
From this it follows that $\alpha _{k}$ defined by $\vartheta (\alpha _{k}y)$
$=$ $0$ is smooth in $y$ by the implicit function theorem. Therefore $\zeta
_{k}$ $=$ $\zeta (\alpha _{k}y)$ is also smooth in $y.$ Clearly $\alpha _{k}$
does not depend on the variable $|\eta |$ (the radial part of $y),$ neither
does $\zeta _{k}$ since by (\ref{7K-22}) $\zeta _{k}(\lambda y)$ $=$ $\zeta
(\alpha _{k}(\lambda y)\lambda y)$ $=$ $\zeta (\alpha _{k}(y)\lambda y)$ $=$ 
$\zeta (\alpha _{k}(y)y)$ $=$ $\zeta _{k}(y)$ for any $\lambda \in \mathbb{R}%
^{+}.$ Since $\mathbb{R}^{+}$ is the only noncompact direction, it follows
from the above independence that $\zeta _{k}$ and $\alpha _{k}$ are $%
C_{B}^{s}$-bounded. We conclude that $p_{m,t}^{0,j}$ is smooth in $x,y$ in
view that $\zeta _{k}$ and $\alpha _{k}$ are smooth in $y,$ and $C_{B}^{s}$%
-bounded since $\zeta _{k}$ and $\alpha _{k}$ are $C_{B}^{s}$-bounded (cf. (%
\ref{CBs})).

\endproof%

\begin{remark}
\label{7-8-5} In the above proof suppose that the action $\sigma $ is
globally free everywhere on the local chart $D_{j}$. Then $\Lambda =0,$ $%
\zeta _{0}$ $=$ $\zeta (\alpha _{0}y)$ $=$ $\zeta (y)$ and we have $%
p_{m,t}^{0,j}(x,y)$ $=$ $\varphi _{j}(x)$ $h^{m}(\zeta (y),\bar{\zeta}(y))$ $%
K_{t}^{j}(z(x),\zeta (y))$ $\tau _{j}(\zeta (y))$ which depends only on $%
z(x) $ and $\zeta (y)$ (except the cutoff function $\varphi _{j}(x)).$
\end{remark}

\proof
\textbf{(of Theorem \ref{AHKE})} The assertion $i)$ of the theorem follows
from Proposition \ref{p-existence} and $i)$ of Theorem \ref{t-uniqueness}$.$
To prove the formula (\ref{AHKE1}) for $ii)$ of the theorem, we reduce the
estimate to that of $P_{m,t}^{0}(x,x)$ by (\ref{approx-2}). Let us estimate $%
P_{m,t}^{0}(x,x)$ which is essentially (\ref{HQD}). First suppose the
simplest situation $p$ ($=$ $p_{1})$ $=$ $1.$ By (\ref{HQDIIa}) and (\ref%
{HQDI-1}), (\ref{HQDI-2}) (using the meaning of \textquotedblleft $\sim ")$
for every $N_{0}\geq N_{0}(n)$ (\cite[p.81]{BGV}) there exist constants $%
C_{N_{0}},$ $\delta =\delta (N_{0})$ $>$ $0$ such that 
\begin{eqnarray}
&&|P_{m,t}^{0}(x,x)-\sum_{j=0}^{N_{0}}t^{-(n-1)+j}b_{n-1-j}(z(x))l^{m}(x)|
\label{K0D} \\
&\leq &C_{N_{0}}l^{m}(x)(t^{-(n-1)+N_{0}+1}+t^{-(n-1)}e^{-\frac{\hat{%
\varepsilon}_{0}\hat{d}(x,\Sigma _{\text{sing}})^{2}}{t}}),\text{ }0<t<\delta
\notag
\end{eqnarray}

\noindent for some constant $\hat{\varepsilon}_{0}>0$ (independent of $N_{0}$
and $x).$ Here we may take $N_{0}(n)$ to be $[\dim _{R}(\Sigma /\sigma )/2$ $%
+1]+1$ ($=$ $\frac{2(n-1)}{2}+2)$ $=$ $n+1.$ We have proved (\ref{AHKE1})
for $p=1$. We remark that (\ref{K0D}) has an analogue for CR manifolds with $%
S^{1}$-action (cf. \cite[(6.2) in p.92]{CHT}).

Now suppose $p$ $>$ $1.$ Then, the angular sectors in $[0,2\pi ]$ over which
the integrals correspond to the two types (\ref{HQDI}) and (\ref{7-2-1})
denoted as $a)$ and $b)$, have the extra $p-1$ pairs of sectors and are
given respectively by%
\begin{equation}
a^{\prime })\text{ }[(s-1)\frac{2\pi }{p}-\varepsilon _{j},(s-1)\frac{2\pi }{%
p}+\varepsilon _{j}]\text{ and }b^{\prime })\text{ }[(s-1)\frac{2\pi }{p}%
+\varepsilon _{j},s\frac{2\pi }{p}-\varepsilon _{j}],  \label{sectors}
\end{equation}

\noindent where $s=1,...,p$ ($s=p+1$ identified with $s=1)$. (The sectors in
(\ref{sectors}) are obtained by successively shifting the first pair of
sectors $s=1$ by a common amount $\frac{2\pi }{p};$the union of these pairs
gives $[0,2\pi ].)$

To evaluate the two types (\ref{sectors}) of integrals, a linear change of
variable for the angular part $\gamma $ brings the intervals of the
integration on these sectors (\ref{sectors}) back to those in (\ref{HQDI})
and (\ref{7-2-1}) with the extra multiplicative factor $\sum_{s=1}^{p}e^{%
\frac{2\pi (s-1)}{p}mi}.$ This number equals $p$ if $p\mid m$ and $0$ if $%
p\nmid m,$ which amounts to $p\delta _{p|m}.$ This concludes (\ref{AHKE1})
proving the assertion $ii)$. The assertion $iii)$ of the theorem follows
from Theorem \ref{p-asymp}.

\endproof%

\begin{remark}
\label{7-36.5} To generalize the $C^{0}$ estimate here to the $C^{l}$ ($%
C_{B}^{l}$ more precisely) estimate presents no serious problem. We skip the
details, and content ourselves with referring to \cite[Corollary 6.3]{CHT}
for a closely related treatment.
\end{remark}

One sees that the RHS of (\ref{K0D}) (for general $p$) blows up as $%
t\rightarrow 0$ and $x$ $\rightarrow $ $\Sigma _{\text{sing}}$ at various
speeds (due to $t^{-(n-1)}$ in the second term). Let $b_{s}^{\pm }(z,\zeta )$
be coefficients in the asymptotic expansion of $K_{t}^{j,\pm }(z,\zeta )$ ($%
\pm $ means acting on even/odd degree of elements as usual) (cf. (\ref{7-5'}%
)):%
\begin{eqnarray}
K_{t}^{j,\pm }(z,\zeta ) &=&e^{-\frac{\tilde{d}_{M}^{2}(z,\zeta )}{4t}%
}K^{j,\pm }(t,z,\zeta )  \label{KjkE} \\
K^{j,\pm }(t,z,\zeta ) &\sim &t^{-n+1}b_{n-1}^{\pm }(z,\zeta
)+t^{-n+2}b_{n-2}^{\pm }(z,\zeta )+\text{ }{\tiny \cdot \cdot \cdot }. 
\notag
\end{eqnarray}

\begin{remark}
\label{7-10-1} Note that the notion of the above asymptotic expansion (\ref%
{KjkE}) (see \cite[(5.19) on p.76]{CHT}) is different from the one in \cite[%
p.87]{BGV} in which the meaning of $\sim $ is given in such a way that it
includes the Gaussian term $e^{-\frac{\tilde{d}_{M}^{2}(z,\zeta )}{4t}};$
our above meaning of $\sim ,$ excluding the Gaussian term, is basically
equivalent to the one in Chavel's book \cite[(45) on p.154]{Chavel}.
\end{remark}

Let $\psi _{j}^{-1}:D_{j}$ $\subset $ $\Sigma $ $\rightarrow $ $W_{j}$ $%
=V_{j}\times (-\varepsilon _{j},\varepsilon _{j})\times \mathbb{R}^{+}$
denote a local trivialization (cf. the line below (\ref{3.36-25})).

\begin{notation}
\label{n-8-1}

\noindent $i)$ Let $\mathcal{\tilde{E}}^{m}$ := $\pi ^{\ast }\mathcal{E}_{M}$
$\otimes E\otimes (L_{\Sigma }^{\ast })^{\otimes m}$ and $\mathcal{E}^{m}$
:= $\psi _{j}^{\ast }(\mathcal{\tilde{E}}^{m})|_{V_{j}\times \{0\}\times
\{1\}}$ $=$ $\psi _{j}^{\ast }(\pi ^{\ast }\mathcal{E}_{M}$ $\otimes
E\otimes (L_{\Sigma }^{\ast })^{\otimes m})|_{V_{j}\times \{0\}\times \{1\}}$
be a complex vector bundle over $V_{j},$ where $\pi :\Sigma \rightarrow
M=\Sigma /\sigma $ denotes the natural projection, $\mathcal{E}_{M}$ denotes
the (orbifold) bundle of all $(0,q)$-forms on $M$, $E$ denotes a $\mathbb{C}%
^{\ast }$-equivariant holomorphic vector bundle over $\Sigma ,$ equipped
with a $\mathbb{C}^{\ast }$-invariant Hermitian metric $h_{E}$ (constructed
similarly as for $L_{\Sigma }$ in Step 1 of Section \ref{S-metric}), and $%
L_{\Sigma }$ is defined before (\ref{3-0}). Let $e^{E}$ denote a \textit{%
locally} $\mathbb{C}^{\ast }$-invariant section of $E$ over $D_{j}$ $\subset 
$ $\Sigma $ (meaning that it is invariant under $\mathbb{R}^{+}$ and
small-angle action)$.$ Similarly let $\mathcal{E}^{m\pm }$ denote the
even/odd part $\psi _{j}^{\ast }(\pi ^{\ast }\mathcal{E}_{M}^{\pm }$ $%
\otimes E\otimes (L_{\Sigma }^{\ast })^{\otimes m})$ of $\mathcal{E}^{m}$.
Although $\mathcal{E}_{M}$ is only an orbifold bundle, the \textquotedblleft
pullback $\pi ^{\ast }\mathcal{E}_{M}"$ having local sections of the form $%
f_{I_{q}}(z,\bar{z})d\bar{z}^{I_{q}}$ can be identified as a vector bundle.
Recall that the metrics for $\pi ^{\ast }\mathcal{E}_{M}$ and $L_{\Sigma }$
are $\pi ^{\ast }g_{M}$ ($=G_{a,m}|_{\pi ^{\ast }\mathcal{E}_{M}},$ see
Lemma \ref{L-inv} $i))$ and $<\cdot ,\cdot >_{L_{\Sigma }}$ (see lines above
(\ref{lq0})) respectively.

$ii)$ Let $\Psi _{\pm ,m}$ denote $\Psi _{q,m}$ in (\ref{Psi_qm}) for $q$
even/odd that identify bundle elements with $m$-space elements (with an
extra bundle $\psi _{j}^{\ast }E|_{V_{j}\times \{0\}\times \{1\}}$; $U_{j},$ 
$\psi _{j}^{-1}(D_{j})$ there taken to be $V_{j},$ $W_{j}$ in Notation \ref%
{n-6.1}). Namely 
\begin{equation}
\Psi _{\pm ,m}:\Omega ^{0,\pm }(V_{j},(\psi _{j}^{\ast }(E\otimes (L_{\Sigma
}^{\ast })^{\otimes m})|_{V_{j}\times \{0\}\times \{1\}})\rightarrow \Omega
_{m,loc}^{0,\pm }(D_{j},E)  \label{PsiE}
\end{equation}%
is defined by%
\begin{equation}
\Psi _{\pm ,m}(s(z,\bar{z})\psi _{j}^{\ast }(e^{E}\otimes (e_{w}^{\ast
})^{\otimes m})|_{V_{j}\times \{0\}\times \{1\}}))=s(z,\bar{z})w^{m}e^{E}.
\label{PsiE-a}
\end{equation}
\end{notation}

Let us now be specific about the various metrics: We define the metric $||$ $%
\cdot $ $||_{\mathcal{E}^{m}}$ (hence $<\cdot ,\cdot >_{\mathcal{E}^{m}}$)
at $\mathcal{E}^{m}|_{(z,0,1)}$ by%
\begin{eqnarray}
&&||\psi _{j}^{\ast }(\pi ^{\ast }\eta ^{I_{q}}\otimes e^{E}\otimes
(e_{w}^{\ast })^{\otimes m})|_{_{(z,0,1)}}||_{\mathcal{E}^{m}}^{2}
\label{Emm} \\
&{:=}&||\eta ^{I_{q}}||_{g_{M}}^{2}||e^{E}||_{h_{E}}^{2}h(z,\bar{z})^{-m} 
\notag
\end{eqnarray}%
where the notation $\eta ^{I_{q}}$ is as in Footnote$^{6}$ (the line above (%
\ref{proj1})) and $e_{w}$ is $\partial /\partial w$ in local coordinates $%
(z,w)$ $\in $ $W_{j}$ with $||e_{w}||^{2}$ $=$ $h(z,\bar{z})$ (see (\ref{lq}%
)). For the $m$-space bundle $\Lambda _{m,loc}^{0,\ast }(D_{j})$ whose
sections are just $\Omega _{m,loc}^{0,\ast }(D_{j})$ (Definition \ref%
{d-6.8-5}) we define the metric $||$ $\cdot $ $||_{\Lambda _{m}^{0,\ast }}$
(hence $<\cdot ,\cdot >_{\Lambda _{m}^{0,\ast }})$ at $q$ $=$ $\psi
_{j}((z,w))$ $\in $ $D_{j}$ by%
\begin{equation}
||\psi _{j}^{\ast }(\pi ^{\ast }\eta ^{I_{q}}\otimes e^{E})w^{m}||_{\Lambda
_{m}^{0,\ast }}^{2}:=||\eta
^{I_{q}}||_{g_{M}}^{2}||e^{E}||_{h_{E}}^{2}|w|^{2m}.  \label{msm}
\end{equation}

\noindent We define the $L^{2}$-inner product $(\cdot ,\cdot )_{\mathcal{E}%
^{m}}$ (resp. $(\cdot ,\cdot )_{\Lambda _{m}^{0,\ast }}$) on sections of $%
\mathcal{E}^{m}$ (resp. $\Lambda _{m,loc}^{0,\ast }(D_{j}))$ by integrating
the fibrewise inner product $<\cdot ,\cdot >_{\mathcal{E}^{m}}$ (resp. $%
<\cdot ,\cdot >_{\Lambda _{m}^{0,\ast }})$ over $V_{j}$ (resp. $W_{j}).$

\begin{lemma}
\label{L-7iso} With the notation above, it holds that $\Psi _{\pm ,m}$
preserves $L^{2}$-inner product up to a constant, i.e.%
\begin{equation}
(\Psi _{\pm ,m}(s),\Psi _{\pm ,m}(t))_{\Lambda _{m}^{0,\ast }}=\frac{%
\varepsilon _{j}}{\pi }(s,t)_{\mathcal{E}^{m}}  \label{7iso}
\end{equation}%
for sections $s,$ $t$ of $\mathcal{E}^{m}.$
\end{lemma}

\begin{proof}
We observe that $h(z,\bar{z})^{-m}$ in (\ref{Emm}) and $|w|^{2m}$ in (\ref%
{msm}) are related in the following fibre integration: (writing $dv_{f,m}$ $%
= $ $l(q)^{-m}d\hat{v}_{m}(q)$ (\ref{fibrenv_0}); omitting the pullback $%
\tau _{p_{0}}^{\ast }$)%
\begin{eqnarray}
\int_{(-\varepsilon _{j},\varepsilon _{j})\times \mathbb{R}%
^{+}}|w|^{2m}dv_{f,m} &=&\int_{(-\varepsilon _{j},\varepsilon _{j})\times 
\mathbb{R}^{+}}|w|^{2m}l(q)^{-m}d\hat{v}_{m}(q)  \label{iso-a} \\
&&\overset{(\ref{lq})}{=}h(z,\bar{z})^{-m}\int_{(-\varepsilon
_{j},\varepsilon _{j})\times \mathbb{R}^{+}}d\hat{v}_{m}(q)  \notag \\
&&\overset{(\ref{3.29-5})+(\ref{3-19.5})}{=}h(z,\bar{z})^{-m}\frac{%
2\varepsilon _{j}}{2\pi }.  \notag
\end{eqnarray}

\noindent In view of (\ref{msm}), (\ref{Emm}) and $||e^{E}||_{h_{E}}^{2}$
being invariant under the action of $(-\varepsilon _{j},\varepsilon
_{j})\times \mathbb{R}^{+}$, (\ref{7iso}) follows easily from (\ref{iso-a})
and the relation of the measures as given in (\ref{hfv}).
\end{proof}

To proceed further, let us set up some more notation. In (\ref{7K-25}) the
expression given by $K_{t}^{j,\pm }(z(x),z(\xi ^{-1}x))$ \textquotedblleft
composed" with $\sigma (\xi )_{\xi ^{-1}x}^{\ast }$ can be interpreted as
(see the line above (\ref{6-1}))%
\begin{eqnarray}
&&(\Psi _{\pm ,m}\circ K_{t}^{j,\pm }\circ \Psi _{\pm ,m}^{-1})\circ \sigma
(\xi )_{\xi ^{-1}x}^{\ast }  \label{7end} \\
&=&\Psi _{\pm ,m}\circ (K_{t}^{j,\pm }\circ \gamma _{\xi ^{-1}}^{\mathcal{E}%
^{m\pm }})\circ \Psi _{\pm ,m}^{-1}  \notag
\end{eqnarray}

\noindent where we have written (see Notation \ref{n-8-1} for $\mathcal{E}%
^{m\pm })$ 
\begin{equation}
\gamma _{\xi ^{-1}}^{\mathcal{E}^{m\pm }}:=\Psi _{\pm ,m}^{-1}\circ \sigma
(\xi )_{\xi ^{-1}x}^{\ast }\circ \Psi _{\pm ,m}:\mathcal{E}^{m\pm
}|_{(z(x),0,1)}\rightarrow \mathcal{E}^{m\pm }|_{(z(\xi ^{-1}x),0,1)}
\label{rE}
\end{equation}%
\noindent in order to be consistent with the notation used in \cite{BGV} and
Section \ref{LmIF}. In later use of (\ref{rE}) (cf. lines below (\ref{7-46-1}%
)) we take $\xi $ $=$ $\alpha _{k}$ $\in $ $S^{1}$ such that $\alpha
_{k}^{-1}x$ $\in $ $V_{j}\times \{0\}\times \{1\}$ for $x$ $\in $ $%
V_{j}\times \{0\}\times \{1\}$ (compare (\ref{7K-5a}) and (\ref{7-38-1}) for
a similar notation)$.$ For $e^{E}$ $\in $ $\psi _{j}^{\ast }E$ in Notation %
\ref{n-8-1} $i),$ $\sigma (\xi )_{\xi ^{-1}x}^{\ast }$ hence $\gamma _{\xi
^{-1}}^{\mathcal{E}^{m\pm }}$ sends it to $\sigma ^{E}(\xi ^{-1})\circ
e_{x}^{E}$ $\in $ $\psi _{j}^{\ast }E|_{\xi ^{-1}x}$ where $\sigma ^{E}$ is
a lifted action on the bundle $\psi _{j}^{\ast }E.$ Note that $\sigma
^{E}(\xi ^{-1})\circ e_{x}^{E}$ may not be $e_{\xi ^{-1}x}^{E}$ for $\xi
^{-1}x$ $\in $ $D_{j},$ especially when $x$ $\in $ $\Sigma _{\text{sing}}$
and $\xi $ gives a large angle action as in Case $ii)$ after (\ref{6-1g}),
but $||\sigma ^{E}(\xi ^{-1})\circ e_{x}^{E}||_{h_{E}}$ $=$ $%
||e_{x}^{E}||_{h_{E}}$ remains true. Writing $z=z(x),$ $\zeta =z(\xi ^{-1}x)$
and using \textquotedblleft $\dagger "$ to denote the adjoint (to
distinguish it from the pullback notation), for $|\xi |$ $=$ $1$ we have, by 
$\gamma _{\xi ^{-1}}^{\mathcal{E}^{m\pm }}$ $=$ $(\gamma _{\xi }^{\mathcal{E}%
^{m\pm }})^{\dagger }$ (see Lemma \ref{L-adj} below)%
\begin{eqnarray}
K_{t}^{j,\pm }(z(x),z(\xi ^{-1}x))\circ \gamma _{\xi ^{-1}}^{\mathcal{E}%
^{m\pm }} &=&(\gamma _{\xi }^{\mathcal{E}^{m\pm }}\circ K_{t}^{j,\pm
}(z,\zeta )^{\dagger })^{\dagger }  \label{endo} \\
&=&(\gamma _{\xi }^{\mathcal{E}^{m\pm }}\circ K_{t}^{j,\pm }(\zeta
,z))^{\dagger }  \notag
\end{eqnarray}

\noindent where $K_{t}^{j,\pm }(z,\zeta )^{\dagger }$ $=$ $(K_{t}^{j,\pm
})^{\dagger }(\zeta ,z)$ $=$ $K_{t}^{j,\pm }(\zeta ,z)$ and the associated
kernel function is always acting to the left on an element in $z$ by our
convention: $K_{t}^{j,\pm }(\zeta ,z)$ $:$ $\mathcal{E}^{m\pm }|_{(z,0,1)}$ $%
\rightarrow $ $\mathcal{E}^{m\pm }|_{(\zeta ,0,1)}$ satisfies $K_{t}^{j,\pm
}(z,\zeta )^{\dagger }$ $=$ $K_{t}^{j,\pm }(\zeta ,z)$. (\ref{endo}) will be
used in (\ref{Kmt-z}) and more importantly, in (\ref{Asym}).

It is important to remark that for $x\in $ $\Sigma _{p}$ (=$\Sigma
\backslash \Sigma _{\text{sing}})$ so that $z(x)=z(\xi ^{-1}x)$ in the above 
$K_{t}^{j,\pm }$ for every $\xi $ $\in $ $\mathbb{C}^{\ast },$ $\gamma _{\xi
}^{\mathcal{E}^{m\pm }}$ is just the identity endomorphism (at $z(x));$
however for $x\in \Sigma _{\text{sing}}$ and $\xi $ $\in $ $G_{x}$ $:=$ $%
\{\xi \in S^{1}$ $\subset $ $\mathbb{C}^{\ast }$ $:$ $\sigma (\xi )x=x\}$
the (finite) isotropy group at $x,$ $\gamma _{\xi }^{\mathcal{E}^{m\pm }}$
may not be the identity. This feature is crucial to our supertrace
evaluation later on.

\begin{lemma}
\label{L-adj} With the notation above, it holds that for $x\in D_{j},$ $|\xi
|$ $=$ $1$ and $\xi ^{-1}x\in D_{j},$ $\gamma _{\xi }^{\mathcal{E}^{m\pm }}$
is an isometry on the bundle part; see also Corollary \ref{C-alk}). It holds
that 
\begin{equation}
\gamma _{\xi ^{-1}}^{\mathcal{E}^{m\pm }}=(\gamma _{\xi }^{\mathcal{E}^{m\pm
}})^{\dagger }.  \label{7-adj}
\end{equation}
\end{lemma}

\begin{proof}
Since the action may involve the large angle one, let us be specific about
the proof. Observe that 
\begin{eqnarray}
&&\sigma (\xi )_{\xi ^{-1}x}^{\ast }(\psi _{j}^{\ast }(\pi ^{\ast }\eta
^{I_{q}}\otimes e^{E})w^{m})_{x}  \label{7-ele} \\
&=&\psi _{j}^{\ast }(\pi ^{\ast }\eta ^{I_{q}})_{\xi ^{-1}x}\otimes \psi
_{j}^{\ast }(\sigma ^{E}(\xi ^{-1})\circ e_{x}^{E})_{\xi ^{-1}x}\xi
^{m}w^{m}(\xi ^{-1}x)  \notag
\end{eqnarray}%
\noindent by the $\sigma $-invariance of the sections $\pi ^{\ast }\eta
^{I_{q}}.$ Note that $(\sigma ^{E}(\xi ^{-1})\circ e_{x}^{E})_{\xi ^{-1}x}$
may not be $e_{\xi ^{-1}x}^{E}$ unless $\sigma (\xi )$ is a small-angle
action. However it remains true that%
\begin{equation}
||(\sigma ^{E}(\xi ^{-1})\circ e_{x}^{E})_{\xi ^{-1}x}||_{h_{E}}=||e_{\xi
^{-1}x}^{E}||_{h_{E}}  \label{7-ele-z}
\end{equation}%
\noindent by the $\sigma ^{E}$-invariance of the metric $h_{E}$ and the
choice of local invariant section $e^{E}.$ Also by the $\sigma $-invariance
of the metrics $\pi ^{\ast }g_{M}$ for $\pi ^{\ast }\mathcal{E}_{M}$ and (%
\ref{7-ele-z}), we conclude from (\ref{7-ele}) that%
\begin{equation}
||\sigma (\xi )_{\xi ^{-1}x}^{\ast }(\psi _{j}^{\ast }(\pi ^{\ast }\eta
^{I_{q}}\otimes e^{E})w^{m})_{x}||_{\Lambda _{m}^{0,\ast }}=||(\psi
_{j}^{\ast }(\pi ^{\ast }\eta ^{I_{q}}\otimes e^{E})w^{m})_{\xi
^{-1}x}||_{\Lambda _{m}^{0,\ast }}\text{ }|\xi |^{m}.  \label{7-ele-a}
\end{equation}%
\noindent Hence for $|\xi |$ $=$ $1$ we learn that $\sigma (\xi )_{\xi
^{-1}x}^{\ast }$ is an isometry with respect to the metric $||$ $\cdot $ $%
||_{\Lambda _{m}^{0,\ast }}.$ Together with (\ref{PsiE})/(\ref{PsiE-a}) it
follows that for $|\xi |$ $=$ $1$ $\gamma _{\xi ^{-1}}^{\mathcal{E}^{m\pm
}}=\Psi _{\pm ,m}^{-1}\circ \sigma (\xi )_{\xi ^{-1}x}^{\ast }\circ \Psi
_{\pm ,m}$ is an isometry with respect to the metric $||$ $\cdot $ $||_{%
\mathcal{E}^{m}}$ (hence $<\cdot ,\cdot >_{\mathcal{E}^{m}}$). So taking $u$ 
$\in $ $\mathcal{E}^{m\pm }|_{x},$ $v$ $\in $ $\mathcal{E}^{m\pm }|_{\xi x}$
we have 
\begin{equation*}
<\gamma _{\xi }^{\mathcal{E}^{m\pm }}u,v>_{\mathcal{E}^{m}}=<\gamma _{\xi
^{-1}}^{\mathcal{E}^{m\pm }}(\gamma _{\xi }^{\mathcal{E}^{m\pm }}u),\gamma
_{\xi ^{-1}}^{\mathcal{E}^{m\pm }}v>_{\mathcal{E}^{m}}=<u,\gamma _{\xi
^{-1}}^{\mathcal{E}^{m\pm }}v>_{\mathcal{E}^{m}}
\end{equation*}

\noindent which gives (\ref{7-adj}).
\end{proof}

\section{Local $m$-index formula\label{LmIF}}

In this long section we want to get an explicit expression of the local
index density, to which the first three subsections are devoted. The
remaining two subsections compare our index formula with the one -- which we
view as a pure orbifold result, of Duistermaat.

Before proceeding to Subsection \ref{Sub-8-1} let us discuss a number of
auxiliary results, which may be regarded as background material for the
subsequent subsections. First, we have the following \textquotedblleft
supertrace" integral equality by the McKean-Singer formula (Theorem \ref%
{t-5-1}) together with (\ref{approx-2}) of Theorem \ref{t-uniqueness} using
the approximate heat kernel $P_{m,t}^{0,\pm }$ (noting that the integral of $%
l(x)^{m}dv_{\Sigma ,m}$ over $\Sigma $ is finite as in Remark \ref{3-r}):

\begin{theorem}
\label{t-MS} For $m\geq 0$, $a$ $>$ $\frac{m}{2}$ (on which the metric $%
G_{a,m}$ depends)$,$ we have 
\begin{eqnarray}
&&\int_{\Sigma }[Tre^{-t\tilde{\square}_{m}^{c+}}(x,x)-Tre^{-t\tilde{\square}%
_{m}^{c-}}(x,x)]dv_{\Sigma ,m}  \label{MS} \\
&=&\lim_{t\rightarrow 0}\int_{\Sigma
}[TrP_{m,t}^{0,+}(x,x)-TrP_{m,t}^{0,-}(x,x)]dv_{\Sigma ,m}.  \notag
\end{eqnarray}
\end{theorem}

Thus, we are reduced to computing the supertrace of $P_{m,t}^{0}(x,x),$
which is done in the first two subsections (cf. (\ref{Part-I}) and (\ref%
{Str-13}) below). We prove Theorem \ref{main_theorem} in the third
subsection.

Our actual computation starts with (\ref{Kmt}). But before doing it we need
some preparatory work. From Section 2 we learn that the $\mathbb{C}^{\ast }$%
-action $\sigma $ on $\Sigma $ gives rise to a complex orbifold structure on 
$\Sigma /\sigma $ (see Theorem \ref{thm2-1}). Write $W_{j}$ $=$ $V_{j}\times
(-\varepsilon _{j},\varepsilon _{j})\times \mathbb{R}^{+}$ for a chart of $%
\Sigma $ (see Notation \ref{n-6.1}) where $V_{j}$ is chosen to be a complex
orbifold chart of $\Sigma /\sigma $ (see the proof of Theorem \ref{thm2-1}).
Moreover the finite group associated to the orbifold chart $V_{j}$ is a
cyclic subgroup of $S^{1}\subset \mathbb{C}^{\ast }$, denoted by $G_{j}$; $%
W_{j}$ is suitably shrinked so that $G_{j}$ is the largest isotropy group at
some point $x$ $\in $ $W_{j}$, containing isotropy groups at any $y$ $\in $ $%
W_{j}$ as subgroups, cf. Corollary \ref{8-5-1}. Let $g_{0}\in G_{j}$ be a
generator of order $N_{j}+1$%
\begin{equation}
g_{0}^{N_{j}+1}=1\text{ and }g_{k}:=g_{0}^{k},\text{ \ \ }k\geq 1\text{ \
(note }g_{1}=g_{0}\text{).}  \label{g1}
\end{equation}%
\noindent Each $g_{k}$ depends on $j$ but for the simplicity of notation we
omit \textquotedblleft $j$" in the expression of the symbol $g_{k}.$ Define
(possibly after shrinking $V_{j})$%
\begin{equation}
\gamma _{k}^{-1}(\in \sigma (S^{1})):=\pi _{V_{j}}\circ \sigma
(g_{k}^{-1}):V_{j}\times \{0\}\times \{1\}\rightarrow V_{j}\times
\{0\}\times \{1\}  \label{Str-5}
\end{equation}

\noindent for all $k$ $=$ $1,$ $\cdot \cdot ,$ $N_{j},$ where $\pi _{V_{j}}$
is the natural projection from $W_{j}$ onto $V_{j}\times \{0\}\times \{1\}$;
note that%
\begin{equation}
\gamma _{k}\text{ is denoted by }\tau (g_{k})\in \sigma (S^{1})\text{ in the
proof of Theorem \ref{thm2-1}}.  \label{8.3-1}
\end{equation}

\begin{lemma}
\label{L-7.6a} With notations in the proof of Lemma \ref{lemma7-1} suppose $%
x,$ $x^{\prime }$ $\in $ $W_{j}$ with $w(x)=w(x^{\prime }).$ Assume that $%
e^{i\beta }\circ x$ and $e^{i\beta }\circ x^{\prime }$ lie in $W_{j}$ for
some $e^{i\beta }\in S^{1}.$ Then $w(e^{i\beta }\circ x)$ $=$ $w(e^{i\beta
}\circ x^{\prime }).$
\end{lemma}

\begin{proof}
First assume that $x^{\prime }$ is close to $x.$ As in the proof of Lemma %
\ref{lemma7-1}, we take a sequence of points $x_{l}$ $\in $ $D_{l}\cap
D_{l-1}$ (resp. $x_{l}^{\prime }$ $\in $ $D_{l}\cap D_{l-1}),$ $l=0,$ $1,$ $%
\cdot \cdot ,$ $L$ such that $x_{0}=x,$ $\cdot \cdot ,$ $x_{L+1}$ $=$ $%
e^{i\beta }\circ x$ $\in $ $D_{0}\cap D_{L}$ (resp. $x_{0}^{\prime
}=x^{\prime },$ $\cdot \cdot ,$ $x_{L+1}^{\prime }$ $=$ $e^{i\beta }\circ
x^{\prime }$ $\in $ $D_{0}\cap D_{L}).$ Let ($z_{l},w_{l})$ denote the
coordinates of the patch $D_{l}$ with $D_{0}$ $=$ $D_{L+1}$ $=$ $W_{j}.$
Note that $w_{0}$ $=$ $w_{L+1}$ $=$ $w.$ By (\ref{C0}) in Proposition \ref%
{p-gue2-1} we have%
\begin{equation}
w_{l}(x_{l})=w_{l-1}(x_{l})f_{l}(z_{l}(x_{l})),\text{ }l=1,\cdot \cdot ,L+1
\label{wh}
\end{equation}

\noindent where $f_{l}$ is holomorphic in $z_{l}.$ From (\ref{angle}) and (%
\ref{7-12.5}) (replacing $\gamma _{l}$ by $\beta _{l}$) it follows that $%
w_{l}(x_{l+1})$ $=$ $w_{l}(e^{i\beta _{l}}x_{l})$ $=$ $e^{i\beta
_{l}}w_{l}(x_{l})$. Together with (\ref{wh}) and $\beta $ :$=$ $%
\sum_{l=0}^{L}\beta _{l}$ we obtain%
\begin{eqnarray}
w(e^{i\beta }\circ x) &=&w_{L+1}(x_{L+1})  \label{wh-a} \\
&=&e^{i\beta }w(x)f_{1}(z_{1}(x_{1}))\cdot \cdot \cdot
f_{L+1}(z_{L+1}(x_{L+1})).  \notag
\end{eqnarray}

\noindent Similarly we have%
\begin{equation}
w(e^{i\beta }\circ x^{\prime })=e^{i\beta }w(x^{\prime
})f_{1}(z_{1}(x_{1}^{\prime }))\cdot \cdot \cdot
f_{L+1}(z_{L+1}(x_{L+1}^{\prime })).  \label{wh-b}
\end{equation}

By Lemma \ref{A} $i)$ we have \TEXTsymbol{\vert}$w(e^{i\beta }\circ x)|$ $=$ 
$|w(x)$\TEXTsymbol{\vert} (resp. \TEXTsymbol{\vert}$w(e^{i\beta }\circ
x^{\prime })|$ $=$ $|w(x^{\prime })$\TEXTsymbol{\vert}) and hence 
\begin{eqnarray}
|f_{1}(z_{1}(x_{1}))\cdot \cdot \cdot f_{L+1}(z_{L+1}(x_{L+1}))| &=&1
\label{wh-c} \\
(\text{resp. }|f_{1}(z_{1}(x_{1}^{\prime }))\cdot \cdot \cdot
f_{L+1}(z_{L+1}(x_{L+1}^{\prime }))| &=&1)  \notag
\end{eqnarray}

\noindent in view of (\ref{wh-a}) (resp. (\ref{wh-b})). The actions $%
e^{i\beta _{l}}$ are holomorphic and hence, in terms of $x$ 
\begin{eqnarray*}
&&f_{1}(z_{1}(x_{1}))\cdot \cdot \cdot f_{L+1}(z_{L+1}(x_{L+1})) \\
&=&f_{1}(z_{1}(e^{i\beta _{0}}\circ x))f_{2}(z_{2}(e^{i(\beta _{1}+\beta
_{0})}\circ x))\cdot \cdot \cdot f_{L+1}(z_{L+1}(e^{i\beta }\circ x))
\end{eqnarray*}%
\noindent is holomorphic in $x.$ This together with (\ref{wh-c}) implies
that $f_{1}(z_{1}(x_{1}))\cdot \cdot \cdot f_{L+1}(z_{L+1}(x_{L+1}))$ is
independent of $x$ so that it is the same as $f_{1}(z_{1}(x_{1}^{\prime
}))\cdot \cdot \cdot f_{L+1}(z_{L+1}(x_{L+1}^{\prime }))$ (here $x\sim
x^{\prime }$ so $\beta _{l},$ $\beta _{l}^{\prime }$ can be chosen to be the
same)$.$ In view of (\ref{wh-a}) and (\ref{wh-b}) we conclude that $%
w(e^{i\beta }\circ x)$ $=$ $w(e^{i\beta }\circ x^{\prime }).$ For the
general case where $x$ and $x^{\prime }$ are not necessarily close, one
connects $x$ and $x^{\prime }$ by a path $\alpha (t)$ $\subset $ $W_{j}$
with the same $w(\alpha (t))$ values. The proof follows by the usual
continuity argument.
\end{proof}

For the fixed point set $V_{j}^{\gamma _{k}}$ ($\subset V_{j}$) of $\gamma
_{k}$ or $\gamma _{k}^{-1}$ ($k\geq 1)$ in (\ref{Str-5}) we write 
\begin{equation}
\Sigma _{j,k}:=\mathbb{C}^{\ast }\circ (V_{j}^{\gamma _{k}}\times
\{0\}\times \{1\}),\text{ \ }k\geq 1  \label{sing}
\end{equation}%
\noindent for the $\mathbb{C}^{\ast }$-orbit of $V_{j}^{\gamma _{k}}\times
\{0\}\times \{1\}$ in $\Sigma $. The local trivialization $\psi _{j}$ is
often omitted.

\begin{remark}
\label{r-1} $i)$ We can choose the above chart $W_{j}$ ($=$ $V_{j}\times $ $%
(-\varepsilon _{j},\varepsilon _{j})\times $ $\mathbb{R}^{+})$ such that $a)$
$\cup _{k}\Sigma _{j,k}$ from (\ref{sing}) is connected and $b)$ $\gamma
_{k} $ acts on $V_{j}\times \{0\}\times \{1\}$ itself for all $k$ as in (\ref%
{Str-5}). $ii)$ Note that the slice $V_{j}\times \{0\}\times \{1\}$ $\subset 
$ $\Sigma $ depends on the choice of local coordinates of $\Sigma $.
\end{remark}

Define $V_{j}^{g_{k}}$ $:=$ $\{z^{\prime }$ $\in $ $V_{j}$ $|$ $\sigma
(g_{k})(z^{\prime },0,1)$ $=$ $(z^{\prime },0,1)\}$ and $\Sigma ^{g_{k}}$ $%
\subset $ $\Sigma ,$ the fixed point set of $\sigma (g_{k}),$ so $%
V_{j}^{g_{k}}$ $=$ $\Sigma ^{g_{k}}$ $\cap $ $V_{j}\times \{0\}\times \{1\}$%
. Let $V_{j}^{F}$ :$=$\ $\cup _{k=1}^{N_{j}}V_{j}^{g_{k}}.$ Note that for $%
1\leq k\leq N_{j}$%
\begin{equation}
V_{j}^{g_{k}}\neq V_{j}.  \label{8-9-1}
\end{equation}

\noindent Otherwise $\sigma (g_{k})$ will fix all points in $V_{j}$ and
hence one sees that $W_{j}$ $\subset $ $\Sigma ^{g_{k}}.$ Since $W_{j}$ is
open in $\Sigma ,$ it follows by holomorphicity that $g_{k}$ $=$ $\{1\},$ a
contradiction to $1\leq k\leq N_{j}$.

We have from (\ref{Str-5}) and Lemma \ref{A} $i)$ that%
\begin{equation}
\gamma _{k}^{-1}x={\Large \{}%
\begin{array}{c}
\sigma (e^{i\eta _{k}(x)})x\text{ \ \ \ \ }x\in (V_{j}\backslash
V_{j}^{F})\times \{0\}\times \{1\} \\ 
x\text{ \ \ \ \ \ \ \ \ \ \ \ \ \ \ \ \ \ }x\in V_{j}^{g_{k}}\times
\{0\}\times \{1\}\text{ \ \ \ \ }%
\end{array}
\label{rk}
\end{equation}%
\noindent where $\eta _{k}$ is a real valued, continuous function (at least
locally defined at a given $x)$. Based on Lemma \ref{L-7.6a} we have

\begin{proposition}
\label{L-alk} With the chart $W_{j}$ chosen in Remark \ref{r-1} and the
notations above, we have $i)$ $e^{i\eta _{k}}$ is a constant, independent of
the choice of $x\in (V_{j}\backslash V_{j}^{F})\times \{0\}\times \{1\};$ $%
ii)$ $\gamma _{k}^{-1}$ $=$ $\sigma (e^{i\eta _{k}})$ acting on $V_{j}\times
\{0\}\times \{1\}$ itself is an isometry with respect to the metric induced
from $G_{a,m}$; $iii)$ $e^{i\eta _{k}}=g_{k}^{-1}.$ In particular $\gamma
_{k}^{-1}$ equals $\sigma (e^{i\eta _{k}})$ $=$ $\sigma (g_{k}^{-1})$ on $%
V_{j}\times \{0\}\times \{1\}$; namely $\pi _{V_{j}}$ in the definition (\ref%
{Str-5}) of $\gamma _{k}^{-1}$ can be dropped.
\end{proposition}

\begin{proof}
We will simply write $e^{i\eta _{k}(x)}x$ for $\sigma (e^{i\eta _{k}(x)})x.$
For $x,x^{\prime }\in (V_{j}\backslash V_{j}^{F})\times \{0\}\times \{1\},$ $%
w(x)=w(x^{\prime })=1$ which by Lemma \ref{L-7.6a} gives that $w(e^{i\eta
_{k}(x)}x)$ $=$ $w(e^{i\eta _{k}(x)}x^{\prime })$ (for $x^{\prime }$ close
to $x$ so that $e^{i\eta _{k}(x)}x^{\prime }$ falls in $W_{j}$ and is close
to $e^{i\eta _{k}(x)}x).$ By the definition of $\eta _{k}$ $w(e^{i\eta
_{k}(x)}x)=1=w(e^{i\eta _{k}(x^{\prime })}x^{\prime })$ and we get $%
w(e^{i\eta _{k}(x^{\prime })}x^{\prime })$ $=$ $w(e^{i\eta _{k}(x)}x^{\prime
}).$ By Lemma \ref{L-7.6a} again, applying $e^{-i\eta _{k}(x)}$ to the
arguments of $w$ we conclude $w(e^{i(\eta _{k}(x^{\prime })-\eta
_{k}(x))}x^{\prime })=w(x^{\prime })$ which is $1.$ In view of the action by
a small angle (by the continuity of $\eta _{k}$ and $x^{\prime }\sim x$) it
follows that for $x^{\prime }$ near $x$, $e^{i(\eta _{k}(x^{\prime })-\eta
_{k}(x))}x^{\prime }$ $=$ $x^{\prime }$ so that $e^{i\eta _{k}(x^{\prime })}$
$=$ $e^{i\eta _{k}(x)}$ since $x^{\prime }$ $\notin $ $V_{j}^{F}$. Therefore 
$e^{i\eta _{k}}$ is constant in each connected component of $%
(V_{j}\backslash V_{j}^{F})\times \{0\}\times \{1\}$ by continuity. Now $%
V_{j}\backslash V_{j}^{F}$ is connected since $V_{j}^{g_{k}}$ is of real
codimension $\geq 2$ in $V_{j}$ by the holomorphicity of $g_{k}$. We have
shown $i)$.

From $i)$ and (\ref{rk}) we have now that $\gamma _{k}^{-1}$ $=$ $\sigma
(e^{i\eta _{k}})$ on $V_{j}\times \{0\}\times \{1\}$ (including $%
V_{j}^{g_{k}}\times \{0\}\times \{1\})$ by continuity from $V_{j}\backslash
V_{j}^{F}$ to its closure $V_{j}.$ That $\gamma _{k}^{-1}$ is an isometry
follows from the fact that the $\sigma (S^{1})$-action is an isometry (see
Remark \ref{7-11b}). We have shown $ii)$.

To show $iii),$ we fix an $x_{0}=(z^{\prime },0,1)$ with $z^{\prime }$ $\in $
$V_{j}^{g_{k}}.$ Then $\sigma (e^{i\eta _{k}})x_{0}=\gamma _{k}^{-1}x_{0}$
by $ii)$ and $\gamma _{k}^{-1}x_{0}$ $=$ ($\pi _{V_{j}}\circ \sigma
(g_{k}^{-1}))x_{0}$ = $\sigma (g_{k}^{-1})x_{0}$ $=$ $x_{0}$ since $\pi
_{V_{j}}$ is trivially the identity on $V_{j}\times \{0\}\times \{1\}$ from
its definition$.$ It follows that we can put\ 
\begin{equation}
e^{i\eta _{k}}=h_{x_{0}}g_{k}^{-1}  \label{h0}
\end{equation}%
\noindent for some $h_{x_{0}}$ $\in $ $G_{j}$ $\subset $ $S^{1}.$ Next, we
apply the action%
\begin{equation*}
\sigma (e^{i\eta _{k}})\overset{ii)}{=}\gamma _{k}^{-1}\overset{(\ref{Str-5})%
}{=}\pi _{V_{j}}\circ \sigma (g_{k}^{-1})
\end{equation*}%
\noindent to any $x$ near $x_{0}.$ We have $\sigma (e^{i\eta
_{k}})(x)=\sigma (e^{i\epsilon _{k}}g_{k}^{-1})(x)$ for a small angle $%
\epsilon _{k}$ (depending on $x$ \textit{a priori}) obtained by $\pi
_{V_{j}};$ compare (\ref{2-11a}) and the lines below it. From the global
freeness of the $\sigma (S^{1})$-action at $x$ $\in $ $(V_{j}\backslash
V_{j}^{F})\times \{0\}\times \{1\}$ it follows that $e^{i\eta _{k}}$ $=$ $%
e^{i\epsilon _{k}}g_{k}^{-1}$ for such $x.$ Comparing this with (\ref{h0})
gives 
\begin{equation}
h_{x_{0}}=e^{i\epsilon _{k}}.  \label{8-11-1}
\end{equation}%
\noindent The contradiction is arising: The number of $h_{x_{0}}$ is at most 
$|G_{j}|,$ while by choosing $x$ sufficiently near $x_{0}$ and applying the
continuity argument at $x_{0}$ as just mentioned one can make the angle $%
\epsilon _{k}(x)$ arbitrarily small. This would violate (\ref{8-11-1})
unless $\epsilon _{k}$ $=$ $0$ thus $h_{x_{0}}$ $=$ $1,$ giving $e^{i\eta
_{k}}$ $=$ $g_{k}^{-1}$ by (\ref{h0}). The claim $iii)$ of the proposition
is now proved.
\end{proof}

The following two corollaries are needed for later use.

\begin{corollary}
\label{C-alk} $\gamma _{k}^{-1}:V_{j}\rightarrow V_{j}$ is also an isometry
with respect to the metric $\pi ^{\ast }g_{M}|_{V_{j}}.$
\end{corollary}

\begin{proof}
Let $i:V_{j}\rightarrow \Sigma $ be the natural embedding. By Proposition %
\ref{L-alk} one has $\gamma _{k}$ $=$ $\sigma (g_{k})|_{V_{j}},$ giving that 
$\gamma _{k}^{\ast }i^{\ast }\pi ^{\ast }g_{M}$ $=$ $i^{\ast }\sigma
(g_{k})^{\ast }\pi ^{\ast }g_{M}$ $=$ $i^{\ast }\pi ^{\ast }g_{M}$ using $%
\pi \circ \sigma (g_{k})$ $=$ $\pi .$
\end{proof}

\begin{corollary}
\label{8-5-1} The map $\tau :G_{j}\rightarrow \tau (G_{j})$ used in Theorem %
\ref{thm2-1} is a group isomorphism.
\end{corollary}

\begin{proof}
By Proposition \ref{L-alk} $iii)$ and (\ref{8.3-1}) we learn that $\tau
(g_{k})$ $=$ $\sigma (g_{k}).$ If $\tau (g_{k})$ is the identity on $V_{j},$
this contradicts (\ref{8-9-1})$.$ So $\ker \tau $ $=$ $\{1\}$ hence $\tau $
is a group isomorphism since in Theorem \ref{thm2-1} we have shown that $%
\tau $ is a group homomorphism.
\end{proof}

Recall (in the proof of Lemma \ref{l-bdp}) that $\alpha _{k}(x)$ $\in $ $%
S^{1}$ $\subset $ $\mathbb{C}^{\ast }$ is chosen to satisfy the property
that for $x\in W_{j},$ $\alpha _{k}(x)x$ ($=\alpha _{k}(x)\circ x)$ $\in
W_{j}$ and $\phi (\alpha _{k}(x)x)$ $=$ $0.$ For $k=0$ set 
\begin{equation}
\alpha _{0}(x)=e^{-i\phi (x)}\text{ for }x=(z,\phi ,r).  \label{a0x}
\end{equation}%
\noindent Clearly $\phi (\alpha _{0}(x)x)$ $=$ $0.$ We describe the
properties of $\alpha _{k}(x)$ below.

\begin{proposition}
\label{gk} (group property of $\alpha _{k}$ in the proof of Lemma \ref{l-bdp}%
) With the chart $W_{j}$ chosen as in Remark \ref{r-1} and the notation
above, it holds that $\alpha _{k}(x)\alpha _{0}^{-1}(x)$ is independent of $%
x\in W_{j}$ and%
\begin{equation}
\{\alpha _{k}(x)\alpha _{0}^{-1}(x):k=0,1,\cdot \cdot ,\Lambda \}=G_{j}
\label{Gj}
\end{equation}%
\noindent (see (\ref{7.38-1}) for $\Lambda $ and (\ref{g1}) for $G_{j}$). So 
$\Lambda =N_{j}.$ Moreover we can arrange $1$ $\leq $ $k$ $\leq $ $N_{j}$
such that%
\begin{equation}
\alpha _{k}(x)\alpha _{0}^{-1}(x)=g_{k}^{-1}.  \label{alk}
\end{equation}%
In particular, the set $\{\alpha _{k}(x)\}_{k=0,1,\cdot \cdot \cdot ,\Lambda
}$ is independent of those $x\in W_{j}$ with $\phi (x)=0$ and this set forms
a group equal to the local orbifold group $G_{j}.$
\end{proposition}

\begin{proof}
Let $x_{1}=(z,0,1)$ $\in $ $V_{j}\times \{0\}\times \{1\}.$ It follows from
the definition of $\alpha _{k}(x_{1})$ and Lemma \ref{A} $i)$ that $\sigma
(\alpha _{k}(x_{1}))x_{1}$ $=$ $(z_{k},0,1)$ $\in $ $V_{j}\times \{0\}\times
\{1\}.$ Since $V_{j}$ is an orbifold chart, there is $g_{l}(=g_{0}^{l})$ $%
\in $ $G_{j}$ such that $\tau (g_{l})z=z_{k}$ (see the proof of Theorem \ref%
{thm2-1}). By Proposition \ref{L-alk} $iii)$ (noting $\gamma _{l}=\tau
(g_{l}),$ cf. (\ref{8.3-1})) $\sigma (g_{l})x_{1}$ $=$ $\sigma (\alpha
_{k}(x_{1}))x_{1}.$ So $g_{l}^{-1}\alpha _{k}(x_{1})$ lies in the isotropy
group of $x_{1},$ which is a subgroup of $G_{j}.$ It follows that $\alpha
_{k}(x_{1})\in G_{j}.$ Denote the set $\{\alpha _{k}(x)\alpha
_{0}^{-1}(x):k=0,1,\cdot \cdot ,\Lambda \}$ by $\Gamma _{x}.$ Observing $%
\alpha _{0}(x_{1})$ $=$ $1$ by (\ref{a0x}), we have shown%
\begin{equation}
\Gamma _{x_{1}}\subset G_{j}.  \label{riG}
\end{equation}%
\noindent Conversely by Proposition \ref{L-alk} $iii)$ again any element of $%
G_{j}$ acts on $V_{j}\times \{0\}\times \{1\}$ itself via $\sigma ,$ so from
the definition of $\alpha _{k}(x_{1}),$ it follows that $G_{j}\subset \Gamma
_{x_{1}}.$ Together with (\ref{riG}) we conclude 
\begin{equation}
\Gamma _{x_{1}}=G_{j}.  \label{reG}
\end{equation}%
\noindent For $x=(z,\phi ,r)$ $\in $ $W_{j}$ we compute ($\sigma $ omitted
for simplicity of notation)%
\begin{eqnarray*}
\alpha _{0}(x)\alpha _{k}(x_{1})x &=&e^{-i\phi (x)}\alpha
_{k}(x_{1})(re^{i\phi (x)}x_{1}) \\
&=&r\alpha _{k}(x_{1})x_{1}=(z_{k},0,r),
\end{eqnarray*}

\noindent so $\phi (\alpha _{0}(x)\alpha _{k}(x_{1})x)=0.$ It follows that
as sets $\{\alpha _{0}(x)\alpha _{k}(x_{1})\}_{k}$ $\subset $ $\{\alpha
_{k}(x)\}_{k}$ from the definition of $\alpha _{k}(x),$ so up to
multiplication by $\alpha _{0}^{-1}(x)$ on both sets $\{\alpha
_{k}(x_{1})\}_{k}$ $=$ $\Gamma _{x_{1}}$ $(\alpha _{0}(x_{1})$ $=$ $1)$ $%
\subset $ $\{\alpha _{0}^{-1}(x)\alpha _{k}(x)\}_{k}$ $=$ $\Gamma _{x}.$
Similarly one proves $\phi (\alpha _{0}^{-1}(x)\alpha _{k}(x)x_{1})$ $=$ $0,$
giving $\Gamma _{x}$ $\subset \Gamma _{x_{1}}$ so $\Gamma _{x_{1}}$ $=$ $%
\Gamma _{x}.$ This, together with (\ref{reG}), gives (\ref{Gj}). By the
continuity of $\alpha _{k}(x)\alpha _{0}^{-1}(x)$ in $x$ and the
discreteness of $G_{j},$ $\alpha _{k}(x)\alpha _{0}^{-1}(x)$ for each $k$
must be independent of $x\in W_{j}.$
\end{proof}

An immediate corollary to Propositions \ref{L-alk} and \ref{gk} is the
following lemma.

\begin{lemma}
\label{L-8-6} It holds that for $1$ $\leq $ $k$ $\leq $ $N_{j}$%
\begin{equation}
\sigma (\alpha _{k}(x)\alpha _{0}^{-1}(x))=\sigma (g_{k}^{-1})=\gamma
_{k}^{-1}\text{ maps }V_{j}\times \{0\}\times \{1\}\text{ (}\subset \Sigma )%
\text{ to itself.}  \label{Str-5c}
\end{equation}
\end{lemma}

We can now compute the supertrace of the approximate heat kernel $%
P_{m,t}^{0}(x,x)$ in (\ref{6-1a1}) 
\begin{equation}
P_{m,t}^{0}:=\sum_{j\text{ (finite)}}H_{m,t}^{j}\circ \pi _{m}  \label{Kmt}
\end{equation}

\noindent where by (\ref{HQD}), (\ref{7end}), (\ref{endo}) and the notation
\textquotedblleft $\dagger "$ (see the lines before and after (\ref{endo}))%
\begin{eqnarray}
&&(H_{m,t}^{j}\circ \pi _{m})(x,x)=\varphi _{j}(x)w^{m}(x)\int_{\xi \in 
\mathbb{C}^{\ast }}\Psi _{\ast ,m}{\Large \{}(\gamma _{\xi }^{\mathcal{E}%
^{m}}\circ K_{t}^{j}(\zeta ,z))^{\dagger }\Psi _{\ast ,m}^{-1}  \label{Kmt-z}
\\
&&\text{ \ \ \ \ \ \ \ }\bar{w}^{m}(\xi ^{-1}x)h^{m}(z(\xi ^{-1}x),\bar{z}%
(\xi ^{-1}x))\tau _{j}(z(\xi ^{-1}x))\sigma _{j}(\phi (\xi ^{-1}x))\bar{\xi}%
^{m}{\Large \}}d\mu _{x,m}(\xi ).  \notag
\end{eqnarray}

\begin{remark}
\label{R-8-9} We equip $\pi ^{\ast }\mathcal{E}_{M}$ with the metric $\pi
^{\ast }g_{M},$ cf. the first term of the RHS in (\ref{mrp}) in local
coordinates. Associated to the metric $h_{E}$ on $E,$ we consider the 
\textit{Chern connection} $\nabla ^{h_{E}}$ (see \cite{MM}) for later use.
For the heat kernel $K_{t}^{j}$ below in (\ref{7-46-1}) we use the metric $%
\pi ^{\ast }g_{M}|_{V_{j}}$ (see the note below (\ref{6-1'})).
\end{remark}

Note that $\pi ^{\ast }\mathcal{E}_{M}$ is not a Clifford module over $%
\Sigma $ although $\mathcal{E}_{M}$ is a Clifford module over $M$.\ Recall
that in (\ref{7K-25}) (in Section \ref{AE_THK}) for $y=x,$ ($\sigma _{\alpha
_{k}}^{\ast })_{y}$ $=$ $\sigma (\alpha _{k}(x))_{x}^{\ast }$ $:$ $\mathcal{%
\tilde{E}}_{\alpha _{k}(x)x}^{m}$ $(=$ $\mathcal{E}_{\alpha
_{k}(x)x}^{m})\rightarrow $ $\mathcal{\tilde{E}}_{x}^{m}$ where $\mathcal{%
\tilde{E}}^{m}$ = $\pi ^{\ast }\mathcal{E}_{M}$ $\otimes E\otimes (L_{\Sigma
}^{\ast })^{\otimes m}$ (Notation \ref{n-8-1})$.$ Since $\alpha _{k}(x)x$ $%
\in $ $V_{j}\times \{0\}\times \{1\},$ $\mathcal{\tilde{E}}_{\alpha
_{k}(x)x}^{m}$ $=$ $\mathcal{E}_{\alpha _{k}(x)x}^{m}.$ For the integral in (%
\ref{Kmt-z}), using (\ref{6-19.5}) and (\ref{7K-25}) we can write%
\begin{equation}
(H_{m,t}^{j}\circ \pi _{m})(x,x)=p_{m,t}^{0,j}(x,x)|w|^{2m}  \label{Kmt-a}
\end{equation}

\noindent where we recall: (\ref{Str-5}) for the definition of $\gamma _{k}$%
, (\ref{rE}) for $\gamma _{k}^{\mathcal{E}^{m}}$ with $\xi ^{-1}$ $=$ $g_{k}$
and $z_{k}=z(\alpha _{k}(x)x),$%
\begin{equation}
p_{m,t}^{0,j}(x,x)=\varphi _{j}(x)h^{m}(z,\bar{z})\sum_{k=0}^{N_{j}}\Psi
_{\ast ,m}(\gamma _{k}^{\mathcal{E}^{m}}\circ K_{t}^{j}(z_{k},z))^{\dagger
}\Psi _{\ast ,m}^{-1}\tau _{j}(z_{k})(\alpha _{k}(x)\alpha _{0}^{-1}(x))^{m}
\label{7-46-1}
\end{equation}%
\noindent by a slightly tedious verification using (\ref{Kmt-z}), (\ref%
{7K-25}) and (\ref{Str-5c}): For this verification we are contented with
pointing out that in (\ref{Str-5c}) $\alpha _{0}^{-1}$ is involved and
noting that $\sigma _{\alpha _{0}^{-1}}^{\ast }$ acts on $\mathcal{\tilde{E}}%
^{m}$ \textquotedblleft trivially"\footnote{%
Note that $\alpha _{0}(x)$ (cf. (\ref{a0x})) acts as a small-angle rotation $%
e^{-i\phi (x)}$ whose action on any $y=(z,\phi ,r)$ keeps $z$-coordinates
unchanged so that the (induced) action of $\sigma _{\alpha _{0}(x)}$ on the
bundle $\pi ^{\ast }\mathcal{E}_{M}$ (as well as the $\mathbb{C}^{\ast }$%
-equivariant $E\otimes (L_{\Sigma }^{\ast })^{\otimes m}$) is regarded as
\textquotedblleft trivial" under the natural trivialization using the
pullback sections (see also $Footnote^{6}$ associated with (\ref{proj1})).}
so using $\sigma _{\alpha _{k}^{-1}\alpha _{0}}^{\ast }$ $=$ $\sigma
_{\alpha _{k}^{-1}}^{\ast }$ and $\gamma _{k}$ $=$ $\sigma _{g_{k}}$ $=$ $%
\sigma _{\alpha _{k}^{-1}\alpha _{0}}$ (\ref{Str-5c}) it follows from $%
\sigma _{\alpha _{k}^{-1}}^{\ast }$ in (\ref{7K-25}), (\ref{endo}) the
expression $\gamma _{k}^{\mathcal{E}^{m}}$ in (\ref{7-46-1}). With $z=z(x)$
and $z_{k}=z(\alpha _{k}(x)x)$ we have an endomorphism 
\begin{equation}
\gamma _{k}^{\mathcal{E}^{m}}=\gamma _{\alpha _{k}^{-1}}^{\mathcal{E}^{m}}:%
\mathcal{E}_{(z_{k},0,1)}^{m}\rightarrow \mathcal{E}_{(z,0,1)}^{m}
\label{8.22-a}
\end{equation}%
\noindent (with $\xi $ $=$ $\alpha _{k}$ in (\ref{rE})) of $\mathcal{E}^{m}$ 
$\mathcal{=}$ $\psi _{j}^{\ast }(\pi ^{\ast }\mathcal{E}_{M}\otimes E\otimes
(L_{\Sigma }^{\ast })^{\otimes m})|_{V_{j}\times \{0\}\times \{1\}}$,
induced by the pullback of $\sigma (\alpha _{k})$ $=$ $\gamma _{k}^{-1}$ $:$ 
$(z,0,1)$ $\rightarrow $ $(z_{k},0,1)$ for $k$ $=$ $1,$ $\cdot \cdot ,$ $%
N_{j}$; note that a local equivariant section of $\mathcal{\tilde{E}}^{m}$
formed from those of $\pi ^{\ast }\mathcal{E}_{M},$ $E$ and $(L_{\Sigma
}^{\ast })^{\otimes m}$ has been used to lift $\gamma _{k}$ to $\gamma _{k}^{%
\mathcal{E}^{m}}.$ For $k$ $=$ $0$ we define $\sigma (\alpha _{0}(x))$ $=:$ $%
\gamma _{0}^{-1}$ at $x$ (see Footnote$^{10}$ at general $x$) and then for $%
x $ $=$ $(z,0,1)$ (thus $\alpha _{0}(x)$ $=$ $1$) we have%
\begin{equation}
\gamma _{0}^{\mathcal{E}^{m}}:\mathcal{E}_{(z,0,1)}^{m}\rightarrow \mathcal{E%
}_{(z,0,1)}^{m}\text{ is the identity.}  \label{r0Em}
\end{equation}

\subsection{Part I of the local index formula ($k=0$ in (\protect\ref{7-46-1}%
))\label{Sub-8-1}}

When $k=0$ thus $z_{0}=z$ and $\gamma _{0}^{\mathcal{E}^{m}}$ $=$ identity
endomorphism (\ref{r0Em}), the $k=0$ term in (\ref{7-46-1}) equals%
\begin{equation}
\varphi _{j}(x)h^{m}(z,\bar{z})K_{t}^{j}(z,z)\tau _{j}(z)  \label{Main}
\end{equation}

\noindent and the corresponding supertrace (denoted as Part I of $%
StrP_{m,t}^{0}(x,x))$ of $P_{m,t}^{0}(x,x)$ in (\ref{Kmt}), (\ref{Kmt-a})
reads (by noting that $\tau _{j}(z)$ $=$ $1$ on supp $\varphi _{j}$ and $%
h^{m}(z,\bar{z})|w|^{2m}$ $=$ $l^{m}(x)$) as%
\begin{eqnarray}
&&\text{Part I of }StrP_{m,t}^{0}(x,x)  \label{PI} \\
&&\overset{(\ref{Kmt-a})+(\ref{Main})}{=}\sum_{j}\varphi
_{j}(x)l^{m}(x)Str(K_{t}^{j}(z,z))  \notag
\end{eqnarray}%
\noindent which gives the major contribution for the index of $\bar{\partial}%
_{\Sigma ,m}^{E}$-complex in Theorem \ref{main_theorem} (see Theorem \ref%
{P-main} below). With (\ref{PI}) we are almost ready to derive
\textquotedblleft Part I" of the local index density of Theorem \ref%
{main_theorem} as stated in the Introduction. To fix the notation let $E$ be
a holomorphic vector bundle on $\Sigma $ with a connection $\nabla .$ Let $%
ch(\nabla ,E)$ and $Td(\nabla ,E)$ be the $\nabla $-induced Chern character
form and the Todd form respectively. For the Chern connection $\nabla
^{h_{E}}$ of $E$ (see Remark \ref{R-8-9}) we write $Td(E,h_{E}):=Td(\nabla
^{h_{E}},E)$ ; $ch(E,h_{E}):=ch(\nabla ^{h_{E}},E).$ Recall that the metric $%
G_{a}$ or $G_{a,m}$ on $\Sigma $ has the property that \textquotedblleft
base" $z$-slices and \textquotedblleft fibre" $w$-slices are orthogonally
splitting (see (\ref{M1-1}) for a precise discussion). We summarize the
results as follows.

\begin{lemma}
\label{L-inv} With notations above and those in the Introduction, we have $%
i) $ the quotient (Hermitian) metric $g_{quot}$ on $T^{1,0}(\Sigma
)/L_{\Sigma } $ defined by the restriction of $G_{a,m}$ on the orthogonal
complement of the tangent space to the \textquotedblleft fibre" (the orbit
of the $\mathbb{C}^{\ast }$-action) is isometric to $\pi ^{\ast }g_{M}$ and
is $\mathbb{C}^{\ast }$-invariant; $ii)$ for $E$ being $\mathbb{C}^{\ast }$%
-equivariant with a $\mathbb{C}^{\ast }$-invariant Hermitian metric $h_{E},$ 
$Td(E,h_{E})$ and $ch(E,h_{E})$ are $\mathbb{C}^{\ast }$-invariant, denoted
by $Td_{\mathbb{C}^{\ast }}(E,h_{E})$ and $ch_{\mathbb{C}^{\ast }}(E,h_{E})$
respectively.
\end{lemma}

Recall that $||\cdot ||$ denotes the $\mathbb{C}^{\ast }$-invariant
Hermitian metric on $L_{\Sigma }$ (Step 1 in Section \ref{S-metric})$,$
which induces $||\cdot ||^{\ast }$ and $||\cdot ||_{m}^{\ast }$ on $%
L_{\Sigma }^{\ast }$ and $(L_{\Sigma }^{\ast })^{m}$ respectively. As above
we write the first Chern form $c_{1}(L_{\Sigma },||\cdot ||).$

We are going to establish \textquotedblleft Part I" of our transversal local
index density in Theorem \ref{P-main} via the local adaptation of classical
local index density arguments (non-transversal ones) with (\ref{PI}).

\begin{theorem}
\label{P-main} With notations and assumptions as in Theorem \ref%
{main_theorem}, we have%
\begin{eqnarray}
&&\lim_{t\rightarrow 0}\text{Part I of }StrP_{m,t}^{0}(x,x)dv_{\Sigma ,m}
\label{Part-Ia} \\
&=&p\delta _{p|m}[Td_{\mathbb{C}^{\ast }}(T^{1,0}\Sigma /L_{\Sigma
},g_{quot})\wedge ch_{\mathbb{C}^{\ast }}(E,h_{E})\wedge
e^{-mc_{1}(L_{\Sigma },||\cdot ||)}\wedge d\hat{v}_{m}]_{2n}(x)  \notag
\end{eqnarray}%
(in a pointwise, non-uniform manner) for $x$ $\in $ $\Sigma _{p}$ (=$\Sigma
\backslash \Sigma _{\text{sing}}$)$.$ Moreover%
\begin{eqnarray}
&&\lim_{t\rightarrow 0}\int_{\Sigma }\text{Part I of }StrP_{m,t}^{0}(x,x)%
\text{ }dv_{\Sigma ,m}  \label{Part-I} \\
&=&p\delta _{p|m}\int_{\Sigma }Td_{\mathbb{C}^{\ast }}(T^{1,0}\Sigma
/L_{\Sigma },g_{quot})\wedge ch_{\mathbb{C}^{\ast }}(E,h_{E})\wedge
e^{-mc_{1}(L_{\Sigma },||\cdot ||)}\wedge d\hat{v}_{m}.  \notag
\end{eqnarray}
\end{theorem}

\begin{proof}
Let $N_{0}(n)$ be as in $ii)$ of Theorem \ref{AHKE}. We first claim that for
every positive integer $N_{0}$ $\geq $ $N_{0}(n)$ ($=n+1)$, there exist $%
\delta $ $>$ $0$ and $C_{N_{0}}$ $>$ $0$ such that%
\begin{eqnarray}
&&|\text{Part I of (}TrP_{m,t}^{0,+}(x,x)-TrP_{m,t}^{0,-}(x,x))
\label{STrSE} \\
&&\text{ \ \ \ \ }-p\delta
_{p|m}l^{m}(x)%
\sum_{j=0}^{N_{0}}t^{-(n-1)+j}(Trb_{n-1-j}^{+}(z)-Trb_{n-1-j}^{-}(z))| 
\notag \\
&\leq &C_{N_{0}}l^{m}(x)t^{-(n-1)+N_{0}+1}  \notag
\end{eqnarray}%
\noindent for any $t,$ $0<t<\delta $ and any $x$ $\in $ $\Sigma _{p}$ (=$%
\Sigma \backslash \Sigma _{\text{sing}}$)$.$ For $p=1$ the proof of (\ref%
{STrSE}) follows from (\ref{KjkE}) (with $\zeta $ $=$ $z)$ and (\ref{PI}).
Since $l^{m}(x)$ in (\ref{STrSE}) blows up as $|w(x)|$ $\rightarrow $ $%
\infty ,$ no uniform convergence follows from (\ref{STrSE}). In fact the
convergence cannot be uniform; this follows from an inspection of the factor 
$l^{m}(x)$ in (\ref{KjkE}) and (\ref{PI}) used for deriving (\ref{STrSE}).
For $p$ $>$ $1$ there is an extra factor $p\delta _{p|m}$ as shown before
(see (\ref{sectors}) and the paragraph after it). Hence the claim. It is
known by using rescaling techniques that%
\begin{eqnarray}
&&\sum_{j=0}^{N_{0}}t^{-(n-1)+j}(Trb_{n-1-j}^{+}(z)-Trb_{n-1-j}^{-}(z))
\label{7-58a} \\
&=&\sum_{j=n-1}^{N_{0}}t^{-(n-1)+j}(Trb_{n-1-j}^{+}(z)-Trb_{n-1-j}^{-}(z)) 
\notag
\end{eqnarray}

\noindent as is from \cite[Proposition 3.21 and Theorem 4.1 (1)]{BGV} (about
the vanishing of the supertrace for degrees strictly less than $\dim _{%
\mathbb{R}}M$ $=$ $2(n-1)$). Note that only the $t^{0}$-term $%
Trb_{0}^{+}(z)-Trb_{0}^{-}(z)$ ($j=n-1$) in (\ref{7-58a}) would survive as $%
t\rightarrow 0$\textbf{. }We obtain via (\ref{STrSE}) and (\ref{7-58a})%
\begin{eqnarray}
&&\lim_{t\rightarrow 0}\text{Part I of }StrP_{m,t}^{0}(x,x)\text{ }%
dv_{\Sigma ,m}  \label{7-38.5} \\
&&\overset{}{=}\sum_{i}\varphi _{i}(x)p\delta
_{p|m}l^{m}(x)(Trb_{0}^{+}(z)-Trb_{0}^{-}(z))dv_{V_{i}}(z)dv_{m}(|w|)dv(\phi
)/2\pi .  \notag
\end{eqnarray}

The RHS of (\ref{7-38.5}) is seen to be related to the classical local index
density. For, in view of the heat kernel $K_{t}^{i}(z,\zeta )$ for $\square
_{V_{i},m}^{c}$ (see (\ref{6-1'}) and (\ref{DiracLap})) the well-known local
index density computation (in connection with Hirzebruch-Riemann-Roch
theorem, cf. \cite{BGV}) can be adapted and applied on $V_{i}$ (with the
extra bundle $E$), and then, since the background metric $g_{M}$ on $V_{i}$
for $\square _{V_{i},m}^{c}$ (cf. Step 3 of Section \ref{S-metric} and
Definition \ref{d-3-7}) has been identified with the above-mentioned metric $%
g_{quot}$ on $T^{1,0}(\Sigma )/L_{\Sigma }$ (Lemma \ref{L-inv})$,$ we arrive
at the following equalities:%
\begin{eqnarray}
&&(Trb_{0}^{+}(z)-Trb_{0}^{-}(z))dv_{V_{i}}(z)  \label{7-38.75} \\
&=&[Td(T^{1,0}V_{i},\pi ^{\ast }g_{M})ch(\psi _{i}^{\ast }(E\otimes
(L_{\Sigma }^{\ast })^{m})|_{V_{i}},h_{E}\otimes ||\cdot ||_{m}^{\ast
})]_{2(n-1)}(z)  \notag \\
&=&[Td(T^{1,0}(\Sigma )/L_{\Sigma },g_{quot})ch(E,h_{E})ch((L_{\Sigma
}^{\ast })^{m},||\cdot ||_{m}^{\ast })]_{2(n-1)}(z,w)  \notag \\
&=&[Td_{\mathbb{C}^{\ast }}(T^{1,0}(\Sigma )/L_{\Sigma },g_{quot})ch_{%
\mathbb{C}^{\ast }}(E,h_{E})e^{-mc_{1}(L_{\Sigma },||\cdot
||)}]_{2(n-1)}(z,w)  \notag
\end{eqnarray}%
\noindent where $\psi _{i}$ $:$ $(z,w)$ $\in $ $V_{i}\times C_{\varepsilon
_{i}}$ $\rightarrow \Sigma $ is the local trivialization (see (\ref{3-0.75})
and (\ref{2.10-5})) and $V_{i}$ may be identified with $V_{i}\times
\{0\}\times \{1\}$ $\subset $ $\Sigma .$ By (\ref{3-18.75}) we write (\ref%
{7-38.5}) as%
\begin{eqnarray}
&&\sum_{i}\varphi _{i}(x)p\delta
_{p|m}l^{m}(x)(Trb_{0}^{+}(z)-Trb_{0}^{-}(z))dv_{V_{i}}(z)dv_{f,m}
\label{7-39} \\
&\overset{(\ref{7-38.75})}{=}&p\delta _{p|m}[Td_{\mathbb{C}^{\ast
}}(T^{1,0}(\Sigma )/L_{\Sigma },g_{quot})ch_{\mathbb{C}^{\ast }}(E,h_{E}) 
\notag \\
&&\ \ \ \ \ \ \ \ e^{-mc_{1}(L_{\Sigma },||\cdot ||)}]_{2(n-1)}(z,w)\wedge
l^{m}(x)dv_{f,m}\text{ (also by }\sum_{i}\varphi _{i}(x)=1)  \notag \\
&\overset{(\ref{fibrenv_0})}{=}&p\delta _{p|m}[Td_{\mathbb{C}^{\ast
}}(T^{1,0}\Sigma /L_{\Sigma },g_{quot})\wedge ch_{\mathbb{C}^{\ast
}}(E,h_{E})\wedge e^{-mc_{1}(L_{\Sigma },||\cdot ||)}\wedge d\hat{v}%
_{m}]_{2n}(x).  \notag
\end{eqnarray}%
\noindent The claim (\ref{Part-Ia}) follows from (\ref{7-38.5}) and (\ref%
{7-39}). To exchange the limit $t\rightarrow 0$ and the integral sign we
note that $l^{m}(x)$ in the RHS of (\ref{STrSE}) is an $L^{1}$ function in
view of Remark \ref{3-r}, implying the second claim (\ref{Part-I}).
\end{proof}

\subsection{Part II of the local $m$-index formula ($k\geq 1$ in (\protect
\ref{7-46-1})) via Lefschetz type formulas.\label{Subs-8-2}}

In contrast to the $k=0$ case$,$ $\gamma _{k}^{\mathcal{E}^{m}}$ is not
necessarily the identity endomorphism if $k\geq 1$ in (\ref{7-46-1}). In
view of local equivariant index theorems, the fixed points set $%
V_{j}^{\gamma _{k}}$ $($in $V_{j},$ identified with $(V_{j}\times
\{0\}\times \{1\})^{\gamma _{k}}$ $\subset $ $V_{j}\times \{0\}\times \{1\})$
of $\gamma _{k}$ and hence the singular stratum $\Sigma _{j,k}$ (see (\ref%
{sing})) (which may cover $|\Sigma _{j,k}|$--the support of $\Sigma _{j,k},$
several times) are expected to play a role in the final index formula. See
the remark below for the support $|\Sigma _{j,k}|$ of $\Sigma _{j,k}$.

\begin{remark}
\label{R-orbit} Let $O$ $:=$ $\mathbb{C}^{\ast }\circ S$ be the $\mathbb{C}%
^{\ast }$-orbit of a set $S$. We write $|O|$ for the \textquotedblleft
support of $O$", that is, the set-theoretical image of the $\mathbb{C}^{\ast
}$-action on $S$. This counts the points in the orbit only once.
\end{remark}

\begin{notation}
\label{N-8-2} In the remaining of this section $\sigma _{\bullet }$ denotes
the symbol map as in \cite[Definition 3.4 and Proposition 3.6]{BGV} and for $%
\sigma _{\bullet }(endomorphisms$) see \cite[Lemma 6.10 and the top two
lines on p.193]{BGV} which is basically the symbol of the Clifford algebra
part of $endomorphisms$.
\end{notation}

To proceed with (\ref{7-46-1}), let us first write the asymptotic expansion
of ($\gamma _{k}^{\mathcal{E}^{m}}\circ K_{t}^{j}(z_{k},z))^{\dagger }$
which appears in (\ref{7-46-1}) and (\ref{8.22-a}), as follows (without
\textquotedblleft $\dagger "$ below): as $t\rightarrow 0$ (cf. (\ref{7-5'})
for $k=0$)%
\begin{equation}
\gamma _{k}^{\mathcal{E}^{m}}\circ K_{t}^{j}(z_{k},z)\sim (4\pi t)^{-\dim _{%
\mathbb{R}}V_{j}^{\gamma _{k}}/2}\sum_{i=0}^{\infty }t^{i}\Phi _{i}^{\gamma
_{k}^{\mathcal{E}^{m}}}(z,z),\text{ \ }z\in V_{j}  \label{Asym}
\end{equation}%
\noindent (cf. \cite[Theorem 6.11]{BGV} with notations parallelly used yet
slightly modified here; $\gamma _{k}^{\mathcal{E}^{m}}$ acts at $z_{k}$).
Notice that (\ref{Asym}) holds true for an open domain $V_{j}$ although \cite%
[Theorem 6.11]{BGV} is applicable for a compact manifold. See \cite{CHT-ECM}
for some detailed explanation.

In the spirit of deducing Lefschetz fixed point theorem we compute the
following limit via (\ref{Asym}) (in the space of generalized sections, see 
\cite[Theorem 6.16]{BGV} for details by noting that $\gamma _{k}^{\mathcal{E}%
^{m}}$ on $(V_{j},$ $\mathcal{E}^{m})$ is an isometry by Corollary \ref%
{C-alk} and Lemma \ref{L-adj}$)$%
\begin{eqnarray}
&&\lim_{t\rightarrow 0}Str(\gamma _{k}^{\mathcal{E}^{m}}\circ
K_{t}^{j}(z_{k},z))  \label{Str-k} \\
&=&c_{j,k}T_{V_{j}}\{Str_{\mathcal{E}^{m}/S}[\sigma _{2\dim _{\mathbb{C}%
}V_{j}}(\Phi _{\dim _{\mathbb{R}}V_{j}^{\gamma _{k}}/2}^{\gamma
_{k}}(z,z))]\}  \notag \\
&=&c_{j,k}T_{V_{j}}\{I_{j}(\gamma _{k})Str_{\mathcal{E}^{m}/S}[\sigma
_{2\dim _{\mathbb{C}}V_{j}-\dim _{\mathbb{R}}V_{j}^{\gamma _{k}}}(\gamma
_{k}^{\mathcal{E}^{m}})\exp (-F_{0}^{_{\mathcal{E}^{m}/S}})]\}\delta
_{V_{j}^{\gamma _{k}}}  \notag
\end{eqnarray}

\noindent where the Berezin integral denoted by $T_{V_{j}}\{I_{j}(\gamma
_{k})$ $\cdot \cdot \cdot \}$ gives a smooth function on $V_{j}^{\gamma
_{k}} $ \cite[p.196 and p.54]{BGV}, $F_{0}^{_{\mathcal{E}^{m}/S}}$ denotes
the restriction to $V_{j}^{\gamma _{k}}$ of the twisting curvature \cite[%
p.195 and p.120]{BGV} ($V_{j}^{\gamma _{k}}$ endowed with the metric induced
from $\pi ^{\ast }g_{M}|_{V_{j}}$), $\sigma _{\bullet }$ is the symbol map
(not to be confused with the $\mathbb{C}^{\ast }$-action $\sigma ,$ see also
(\ref{S-1})), $c_{j,k}$ $:=$ $(4\pi )^{-\dim _{\mathbb{R}}V_{j}^{\gamma
_{k}}/2}(-2i)^{\dim _{\mathbb{C}}V_{j}}$ with $(4\pi )^{-\dim _{\mathbb{R}%
}V_{j}^{\gamma _{k}}/2}$ from (\ref{Asym}) and $(-2i)^{\dim _{\mathbb{C}%
}V_{j}}$ from the formula in \cite[Proposition 3.21]{BGV} and finally%
\begin{equation}
I_{j}(\gamma _{k}):=\frac{\hat{A}_{BGV}(V_{j}^{\gamma _{k}})}{%
\det^{1/2}(1-(\gamma _{k})_{1})\det^{1/2}(1-(\gamma _{k})_{1}\exp (-R^{1}))}
\label{Str-0}
\end{equation}%
\noindent (see \cite[Theorem 6.11]{BGV} and Notation \ref{N-8-2a} below).
Note that the use of $\hat{A}_{BGV}$ here according to \cite{BGV} is
different from the usual $\hat{A}$-genus form (\ref{3-11}) by a constant
factor involving $2\pi $. About these expressions, see related discussions
prior to (and in the proof of) Proposition \ref{P-Fk}.

Note also that following \cite[the line above Theorem 6.16]{BGV} we use the
notation $T_{V_{j}}$ in (\ref{Str-k}) although it is applied to a
delta-function like object supported on $V_{j}^{\gamma _{k}}$. For the
insertion of $\delta _{V_{j}^{\gamma _{k}}}$ in (\ref{Str-k}) in the end,
see \cite[Theorem 6.11]{BGV}. With the notation above, set%
\begin{equation}
\mathcal{F}_{k,m}^{j}(z,\bar{z}):=c_{j,k}T_{V_{j}}\{I_{j}(\gamma _{k})Str_{%
\mathcal{E}^{m}/S}[\sigma _{2\dim _{\mathbb{C}}V_{j}-\dim _{\mathbb{R}%
}V_{j}^{\gamma _{k}}}(\gamma _{k}^{\mathcal{E}^{m}})\exp (-F_{0}^{_{\mathcal{%
E}^{m}/S}})]\}.  \label{Str-1a}
\end{equation}%
\noindent for $z\in V_{j}^{\gamma _{k}}.$ Rewrite (\ref{Str-k}) as%
\begin{equation}
\lim_{t\rightarrow 0}Str(\gamma _{k}^{\mathcal{E}^{m}}\circ
K_{t}^{j}(z_{k},z))=\mathcal{F}_{k,m}^{j}(z,\bar{z})\delta _{V_{j}^{\gamma
_{k}}}.  \label{Str-1}
\end{equation}

\noindent It follows from (\ref{Str-1}), (\ref{7-46-1}) and $l(x)=h(z,\bar{z}%
)|w|^{2}$ that (by $Str(\bullet ^{\dagger })$ $=$ $\overline{Str(\bullet )}$
for the first equality then separating the $k=0$ term where $\gamma _{0}^{%
\mathcal{E}^{m}}$ is the identity (\ref{r0Em}), from $k\geq 1$ terms in the
second equality below)%
\begin{eqnarray}
&&\lim_{t\rightarrow 0}Str\text{ }(H_{m,t}^{j}\circ \pi _{m})(x,x)
\label{Str-2} \\
&=&\varphi _{j}(x)l^{m}(x)\sum_{k=0}^{N_{j}}\lim_{t\rightarrow 0}\overline{%
Str(\gamma _{k}^{\mathcal{E}^{m}}\circ K_{t}^{j}(z_{k},z))}\tau
_{j}(z_{k})(\alpha _{k}\alpha _{0}^{-1})^{m}  \notag \\
&=&\varphi _{j}(x)l^{m}(x)\lim_{t\rightarrow 0}Str(K_{t}^{j}(z,z))\tau
_{j}(z)  \notag \\
&&+\varphi _{j}(x)l^{m}(x)\sum_{k=1}^{N_{j}}\overline{\mathcal{F}%
_{k,m}^{j}(z,\bar{z})}\tau _{j}(z_{k})(\alpha _{k}\alpha
_{0}^{-1})^{m}\delta _{V_{j}^{\gamma _{k}}}.  \notag
\end{eqnarray}

\noindent Here notice that $\alpha _{k}\alpha _{0}^{-1}$ ($=$ $\alpha
_{k}(x)\alpha _{0}^{-1}(x)$ $=$ $g_{k}^{-1}$ by (\ref{alk})) are independent
of $x$ and that $V_{j}^{\gamma _{k}}$ $=$ $V_{j}^{g_{k}}$ by $iii)$ of
Proposition \ref{L-alk}$.$ We now compute the integral of the supertrace:%
\begin{eqnarray}
&&\lim_{t\rightarrow 0}\int_{\Sigma }Str\text{ }P_{m,t}^{0}(x,x)dv_{\Sigma
,m}\overset{(\ref{7-1.5})}{=}\sum_{j}\lim_{t\rightarrow 0}\int_{\Sigma }Str%
\text{ }H_{m,t}^{j}\circ \pi _{m}(x,x)dv_{\Sigma ,m}  \label{Str-3} \\
&&\overset{(\ref{Str-2})}{=}\sum_{j}\lim_{t\rightarrow 0}\int_{\Sigma
}\varphi _{j}(x)l^{m}(x)Str(K_{t}^{j}(z,z))\tau _{j}(z)dv_{\Sigma ,m}  \notag
\\
&&+\sum_{j}\sum_{k=1}^{N_{j}}(g_{k}^{-1})^{m}\int_{W_{j}}\varphi
_{j}(x)l^{m}(x)\overline{\mathcal{F}_{k,m}^{j}(z,\bar{z})}\tau
_{j}(z_{k})\delta _{V_{j}^{\gamma _{k}}}dv_{\Sigma ,m}  \notag
\end{eqnarray}

\noindent the last term of which can be computed by (\ref{Str-8}) below to
become%
\begin{equation}
\sum_{j}\sum_{k=1}^{N_{j}}(g_{k}^{-1})^{m}\int_{V_{j}^{\gamma _{k}}\times
(-\varepsilon _{j},\varepsilon _{j})\times \mathbb{R}^{+}}\varphi
_{j}(z,\phi )\overline{\mathcal{F}_{k,m}^{j}(z,\bar{z})}d\tilde{v}%
_{V_{j}^{\gamma _{k}}}(z)\wedge d\hat{v}_{m}(z,\phi ,|w|)  \label{Str-4}
\end{equation}

\noindent by $l^{m}(x)dv_{\Sigma ,m}$ $=$ $d\tilde{v}_{V_{j}^{\gamma
_{k}}}(z)\wedge d\hat{v}_{m}(z,\phi ,|w|)$ in $W_{j}$ with $\sigma (\tilde{g}%
)$ $=$ $\gamma _{k}$ in $V_{j}$ (\ref{Str-5c}) and $\tau _{j}=1$ on $%
\{\varphi _{j}(z,\phi )\neq 0\},$ see (\ref{Str-8}) and Remark \ref{R-z-vol}
for $d\tilde{v}_{V_{j}^{\gamma _{k}}}$ here.

We are going to look for an integrand defined on $\Sigma $ such that its
integral over singular strata $\Sigma _{\text{sing}}$ (see (\ref{1-4}))
equals (\ref{Str-4}), cf. Theorem \ref{main_theorem} and (\ref{Str-7}). To
this aim (see Proposition \ref{L-sing} below) first recall (\ref{g1}) for
the definition of $g_{k}$ and $G_{j}$ ($\subset S^{1}\subset \mathbb{C}%
^{\ast }$) and let 
\begin{equation}
\mathcal{G}:=\dbigcup\limits_{j(\text{finite})}G_{j}.  \label{8-41-1}
\end{equation}%
\noindent Note that the finite index set $\mathcal{G}$ may not be a group
and $g_{k}$ in (\ref{g1}) depends on $j.$

We use $g_{k}$ or $g_{k}^{(j)}$ interchangeably below. For $\tilde{g}\in 
\mathcal{G}$ let $\Sigma ^{\tilde{g}}$ denote the set of points fixed by $%
\tilde{g}$ (via $\sigma )$ in $\Sigma .$ By (\ref{sing}) for $\Sigma _{j,k}$
we then have%
\begin{equation}
\Sigma ^{\tilde{g}}=\dbigcup\limits_{(j,k):g_{k}^{(j)}=\tilde{g}}|\Sigma
_{j,k}|.  \label{Sigmak}
\end{equation}%
\noindent Observe that $\Sigma ^{\tilde{g}}$ is a (complex) submanifold of $%
\Sigma $. The following basically follows from definitions.

\begin{lemma}
\label{L-S} With the notation above, it holds that ($\mathcal{G}$ being a
finite set)%
\begin{equation}
\Sigma _{\text{sing}}=\dbigcup\limits_{\tilde{g}\in \mathcal{G},\text{ }%
\tilde{g}\neq 1}\Sigma ^{\tilde{g}}.  \label{sk}
\end{equation}
\end{lemma}

\begin{proof}
Recall that $\Sigma _{p_{j}}$ denotes the set of points having period $2\pi
/p_{j}$ and that $\Sigma _{\text{sing}}$ is by definition the union of $%
\Sigma _{p_{j}},$ $j\geq 2.$ If a point has the period less than $2\pi
/p_{1},$ then its isotropy group $\subset $ $S^{1}$ is nontrivial so that it
is a fixed point of some element $\tilde{g}$ of $S^{1}.$ Conversely, any
point in $\Sigma ^{\tilde{g}}$ for $\tilde{g}$ $\neq $ $1$ has a nontrivial
isotropy group rendering its period less than $2\pi /p_{1}.$%
\end{proof}

\begin{notation}
\label{N-8-2a} $i)$ Let $\mathcal{N}_{\mathbb{R}}$ denote the \textit{real}
normal bundle of $\Sigma ^{\tilde{g}},$ but for notational convenience we
drop the subscript $\mathbb{R}$. We equip it with the metric induced from $%
G_{a,m}$ or $\pi ^{\ast }g_{M};$ see Remark \ref{R-8-16} for the equivalence
in this case$.$ $ii)$ We follow the notation adopted by \cite{BGV}: Let $%
R^{1}$ denote the curvature of $\mathcal{N}$ and by $\tilde{g}_{1}$ the
naturally induced action of $\tilde{g}$ on $\mathcal{N}$.
\end{notation}

In view of (\ref{Str-0}) we set ($T\Sigma ^{\tilde{g}}$ and $L_{\Sigma ^{%
\tilde{g}}}$ below viewed as real tangent bundle and subbundle respectively
with Lemma \ref{L-inv} $i)$ for the metric $g_{quot}$)%
\begin{equation}
I(\tilde{g}):=\frac{\hat{A}_{BGV}(T\Sigma ^{\tilde{g}}/L_{\Sigma ^{\tilde{g}%
}},g_{quot})}{\det^{1/2}(1-\tilde{g}_{1})\det^{1/2}(1-\tilde{g}_{1}\exp
(-R^{1}))}  \label{Ig}
\end{equation}%
\noindent on $\Sigma ^{\tilde{g}}$. Note that for $\tilde{g}%
=g_{k}^{(j)}=g_{k}$ locally in $W_{j}$ as in (\ref{Sigmak})%
\begin{equation}
\tilde{g}^{\mathcal{E}^{m}}=\gamma _{k}^{\mathcal{E}^{m}}\text{ on }%
V_{j}\times \{0\}\times \{1\}  \label{8-16c}
\end{equation}%
\noindent by (\ref{Str-5c}) for $\sigma (g_{k})=\gamma _{k}$ on $V_{j}\times
\{0\}\times \{1\}$. We prefer the use of two separate notations $\tilde{g}^{%
\mathcal{E}^{m}}$ and $\gamma _{k}^{\mathcal{E}^{m}}$ because the former is
meant to be global and intrinsic while the latter is for the restriction of
the former on $V_{j}$ (which depends on the trivialization $\psi _{j})$. In
view of (\ref{Str-1a}) and (\ref{8-16c}) we define a \textit{complex-valued}
function $\mathcal{F}_{\tilde{g},m}$ on $\Sigma ^{\tilde{g}}$ $\subset $ $%
\Sigma $ via \textquotedblleft projections" $\pi _{V_{j}}$ $:$ $(z_{j},$ $%
w_{j})$ $\in $ $W_{j}$ $\rightarrow $ $(z_{j},$ $1)$ $\in $ $V_{j}$ ($%
T_{V_{j}}$ below as in (\ref{Str-k})): for $q$ $\in $ $\Sigma ^{\tilde{g}%
}\cap D_{j}$%
\begin{eqnarray}
&&\mathcal{F}_{\tilde{g},m}|_{\Sigma ^{\tilde{g}}\cap D_{j}}(q)
\label{Str-7} \\
:= &&c_{j,k}T_{V_{j}}\{I(\tilde{g})Str_{\mathcal{E}/S}[\sigma
_{2n-(l_{j,k}+2)}(\tilde{g}^{\mathcal{E}^{m}})\exp (-F_{0}^{_{\mathcal{E}%
^{m}/S}})]\}|_{\pi _{V_{j}}(\psi _{j}^{-1}(q))}  \notag
\end{eqnarray}%
\noindent where $l_{j,k}$ $=$ $\dim _{\mathbb{R}}V_{j}^{\gamma _{k}}.$ In
Lemma \ref{P-8-16} below we show that (\ref{Str-7}) is independent of the
choice of $D_{j}$ that contains $q.$

It follows (cf. (\ref{Str-1a}) for $\mathcal{F}_{k,m}^{j}$) that for $q\in
\Sigma ^{\tilde{g}}\cap D_{j}$%
\begin{equation}
\mathcal{F}_{\tilde{g},m}|_{\Sigma ^{\tilde{g}}\cap D_{j}}(q)=\mathcal{F}%
_{k,m}^{j}(z(\psi _{j}^{-1}(q)),\bar{z}(\psi _{j}^{-1}(q)))\text{.}
\label{Str-7a}
\end{equation}

\noindent Here we think of $\mathcal{F}_{\tilde{g},m}$ as a global function
and $\mathcal{F}_{k,m}^{j}$ as its local expression.

\begin{lemma}
\label{8-10a} With the notation above $\hat{A}_{BGV}(T\Sigma ^{\tilde{g}%
}/L_{\Sigma ^{\tilde{g}}},g_{quot})$, $R^{1}$ and $\tilde{g}_{1}$are locally 
$\mathbb{C}^{\ast }$-invariant (meaning that it is invariant under the
action of $\rho e^{i\phi }$ with $\rho $ $\in $ $\mathbb{R}^{+}$ and $|\phi
| $ small). Moreover $\det^{1/2}(1-\tilde{g}_{1})$ and $\det^{1/2}(1-\tilde{g%
}_{1}\exp (-R^{1}))$ in (\ref{Ig}) are also locally $\mathbb{C}^{\ast }$%
-invariant.
\end{lemma}

\begin{proof}
By (\ref{metric}) and (\ref{M1-1}) we write the metric $G_{a,m}$ in special
local coordinates $(z,w)$ for (\ref{M0-3}) at $\sigma (c)q_{0}$ $=$ $%
(z_{0},cw_{0})$ with $q_{0}$ $=$ $(z_{0},w_{0})$ $\in $ $\Sigma ^{\tilde{g}%
}\cap D_{j}$ and $c$ $=$ $|c|e^{i\varphi }$ $\in $ $\mathbb{C}^{\ast }$ for $%
|c|$ $\in $ $\mathbb{R}^{+}$ arbitrary, $\varphi $ sufficiently near $0$
such that $\sigma (c)q_{0}$ still lies in $\Sigma ^{\tilde{g}}\cap D_{j}$,
as follows (with the same coordinates (\ref{M0-3}) holds at $\sigma (c)q_{0}$
when $c$ varies) 
\begin{eqnarray}
&&G_{a,m}|_{\sigma (c)q_{0}}=(g_{M})_{\alpha \bar{\beta}}(z_{0},\bar{z}%
_{0})dz_{\alpha }d\bar{z}_{\beta }+  \label{mrp} \\
&&{\Large (}\varphi _{1}(|c|^{2}w_{0}\bar{w}_{0})+\varphi _{2}(|c|^{2}w_{0}%
\bar{w}_{0})4a^{2}|c|^{-4a-2}(w_{0}\bar{w}_{0})^{-2a-1}{\large )}\frac{dwd%
\bar{w}}{\lambda _{m}}.  \notag
\end{eqnarray}

\noindent It is not difficult to see from (\ref{mrp}) that the normal space $%
\mathcal{N}_{\sigma (c)q_{0}},$ which is perpendicular to $\Sigma ^{\tilde{g}%
}$ at $\sigma (c)q_{0},$ consists of vectors depending only on $z$%
-coordinate as $c$ varies. It follows that $\sigma (c)_{\ast }:\mathcal{N}%
_{q_{0}}$ $\rightarrow $ $\mathcal{N}_{\sigma (c)q_{0}}$ is an isometry with
respect to $\pi ^{\ast }g_{M}$ $=$ $(g_{M})_{\alpha \bar{\beta}}(z_{0},\bar{z%
}_{0})$ $dz_{\alpha }d\bar{z}_{\beta }$ in $W_{j}$ $=$ $\psi
_{j}^{-1}(D_{j}) $, and hence the curvature $R^{1}$ of $\mathcal{N}$ is
locally $\mathbb{C}^{\ast }$-invariant. For $\hat{A}_{BGV}$ on $\Sigma ^{%
\tilde{g}}$ we observe that $g_{quot}$ is identified with the metric $\pi
^{\ast }g_{M}|_{_{\Sigma ^{\tilde{g}}\cap D_{j}}}$ on the $W_{j}$-chart (see
Lemma \ref{L-inv} $i)$) so that $\hat{A}_{BGV}(T\Sigma ^{\tilde{g}%
}/L_{\Sigma ^{\tilde{g}}},g_{quot}) $ is also locally $\mathbb{C}^{\ast }$%
-invariant. That $\tilde{g}_{1}$ is locally $\mathbb{C}^{\ast }$-invariant
follows from that the action of $\tilde{g}$ commutes with the (local) $%
\mathbb{C}^{\ast }$-action. The last statement easily follows via the
isometry $\sigma (c)_{\ast }$ just mentioned.
\end{proof}

\begin{remark}
\label{R-8-16} By (\ref{mrp}) the metric $\pi ^{\ast }g_{M}$ and the metric $%
G_{a,m}$ coincide on the normal bundle $\mathcal{N},$ cf. Remark \ref%
{R-z-vol}.
\end{remark}

\begin{corollary}
\label{R+Inv} With the notation above, $\mathcal{F}_{\tilde{g},m}|_{\Sigma ^{%
\tilde{g}}\cap D_{j}}$ (= $\mathcal{F}_{k,m}^{j}(z,\bar{z})$ by (\ref{Str-7a}%
)$)$ viewed as a function of $(z,w)$ $\in $ ($\psi _{j}^{-1})\Sigma ^{\tilde{%
g}}\cap D_{j}$ is a function of $z$ (and $\bar{z})$ only, $z$ $\in $ $%
V_{j}^{\gamma _{k}}.$ In particular $\mathcal{F}_{\tilde{g},m}|_{\Sigma ^{%
\tilde{g}}\cap D_{j}}$ is locally $\mathbb{C}^{\ast }$-invariant.
\end{corollary}

\begin{proof}
From Lemma \ref{8-10a} $I(\tilde{g})$ in (\ref{Ig}) is independent of the
local $\mathbb{C}^{\ast }$-action. The curvature $F_{0}^{_{\mathcal{E}%
^{m}/S}}$ (\ref{Str-k}) and the symbol $\sigma _{2n-(l_{j,k}+2)}(\tilde{g}^{%
\mathcal{E}^{m}})$ work on $V_{j}^{\gamma _{k}}$ and so involve the $z$%
-coordinate(s) only. Altogether in view of (\ref{Str-7}) $\mathcal{F}_{%
\tilde{g},m}|_{\Sigma ^{\tilde{g}}\cap D_{j}}$ is a function of $z$ in $%
V_{j}^{\gamma _{k}}$.
\end{proof}

Using Lemma \ref{8-10a} we prove

\begin{lemma}
\label{P-8-16} For $q$ $\in $ $\Sigma ^{\tilde{g}}\cap D_{j}\cap
D_{j^{\prime }}$ we have 
\begin{equation}
\mathcal{F}_{\tilde{g},m}|_{\Sigma ^{\tilde{g}}\cap D_{j}}(q)=\mathcal{F}_{%
\tilde{g},m}|_{\Sigma ^{\tilde{g}}\cap D_{j^{\prime }}}(q).  \label{Fgm}
\end{equation}%
Namely $\mathcal{F}_{\tilde{g},m}$ is a globally defined smooth function on $%
\Sigma ^{\tilde{g}}.$
\end{lemma}

\begin{proof}
Suppose that $g_{k\text{ }}^{(j)}$ $=$ $g_{k^{\prime }}^{(j^{\prime })}$ $=$ 
$\tilde{g}$ thus $V_{j}^{\gamma _{k}}$ $=$ $V_{j}^{\tilde{g}}$ $=$ $\Sigma ^{%
\tilde{g}}$ $\cap $ ($V_{j}\times \{0\}\times \{1\})$ (resp. $V_{j^{\prime
}}^{\gamma _{k^{\prime }}}$ $=$ $V_{j^{\prime }}^{\tilde{g}}$ $=$ $\Sigma ^{%
\tilde{g}}$ $\cap $ ($V_{j^{\prime }}\times \{0\}\times \{1\}))$. We leave
the reader to check the equality of the constants $l_{j,k}$ $=$ $%
l_{j^{\prime },k^{\prime }},$ $c_{j,k}$ $=$ $c_{j^{\prime },k^{\prime }}.$ 
To compare the two sides of (\ref{Fgm}) we notice that by (\ref{1-1}) and
Theorem \ref{thm2-1} one finds $\zeta _{j}$ $\in $ $\mathbb{C}^{\ast }$ with 
$\zeta _{j}$ $=$ $|\zeta _{j}|e^{i\phi _{j}},$ $|\phi _{j}|$ $<$ $%
\varepsilon _{j}$ (resp. $\zeta _{j^{\prime }}$ $\in $ $\mathbb{C}^{\ast }$
\noindent with $\zeta _{j^{\prime }}$ $=$ $|\zeta _{j^{\prime }}|e^{i\phi
_{j^{\prime }}},$ $|\phi _{j^{\prime }}|$ $<$ $\varepsilon _{j^{\prime }})$
such that 
\begin{equation*}
\psi _{j}\pi _{V_{j}}(\psi _{j}^{-1}(q))\overset{\sigma (\zeta _{j})}{%
\longrightarrow }q\text{ (resp. }q\overset{\sigma (\zeta _{j^{\prime }})}{%
\longrightarrow }\psi _{j^{\prime }}\pi _{V_{j^{\prime }}}(\psi _{j^{\prime
}}^{-1}(q)))
\end{equation*}%
so that $\sigma (\zeta _{j^{\prime }}\zeta _{j})$ sends $\psi _{j}\pi
_{V_{j}}(\psi _{j}^{-1}(q))$ to $\psi _{j^{\prime }}\pi _{V_{j^{\prime
}}}(\psi _{j^{\prime }}^{-1}(q))$. Since $\mathcal{F}_{\tilde{g},m}$ in both
sides of (\ref{Fgm}) arise from the quantity $I(\tilde{g})Str_{\mathcal{E}%
/S} $ $[\sigma _{2n-(l_{j,k}+2)}(\tilde{g}^{\mathcal{\tilde{E}}^{m}})$ $\exp
(-F_{0}^{_{\mathcal{\tilde{E}}^{m}/S}})]$ on $\Sigma ^{\tilde{g}}$ (see
Notation \ref{n-8-1} $i)$ for $\mathcal{\tilde{E}}^{m},$ $\mathcal{E}^{m}$
and (\ref{Str-7}) with no \textquotedblleft tilde$"$ on $\mathcal{E}^{m})$,
we see from (\ref{Str-7}) that the values in (\ref{Fgm}) differ by the $%
\mathbb{C}^{\ast }$-action $\sigma (\zeta _{j^{\prime }}\zeta _{j})$ and
then the local $\mathbb{C}^{\ast }$-invariance (Corollary \ref{R+Inv}) gives
the same value, proving (\ref{Fgm})$.$
\end{proof}

In the above discussion we refer to the local $\mathbb{C}^{\ast }$%
-invariance of $\mathcal{F}_{\tilde{g},m}.$ In fact $\mathcal{F}_{\tilde{g}%
,m}$ is (globally) $\mathbb{C}^{\ast }$-invariant: (This fact is not
strictly needed until (\ref{Fg}); see remarks after (\ref{Fg}).)

\begin{corollary}
\label{C-Fgm} For all $\lambda $ $\in $ $\mathbb{C}^{\ast }$ and $q$,
(hence) $\sigma (\lambda )q$ $\in $ $\Sigma ^{\tilde{g}}$ it holds that $%
\mathcal{F}_{\tilde{g},m}(\sigma (\lambda )q)$ $=$ $\mathcal{F}_{\tilde{g}%
,m}(q)$. In particular, the function $\mathcal{F}_{\tilde{g},m}$ takes the
same value if both $q$ and $\sigma (\lambda )q$ lie in the same local chart $%
\Sigma ^{\tilde{g}}\cap D_{j}.$
\end{corollary}

\begin{proof}
Since $\mathcal{F}_{\tilde{g},m}$ is well defined on the whole $\Sigma ^{%
\tilde{g}}$ by Lemma \ref{P-8-16} and is locally $\mathbb{C}^{\ast }$%
-invariant by Corollary \ref{R+Inv}, the global $\mathbb{C}^{\ast }$%
-invariance follows from a composition of finite number of local $\mathbb{C}%
^{\ast }$-actions (cf. proof of Lemma \ref{L-7.6a}).
\end{proof}

Note that $G_{a,m}$ is not $\mathbb{C}^{\ast }$-invariant in general but it
is $\mathbb{C}^{\ast }$-invariant after restricting to some subbundles such
as the above-mentioned normal bundles (Remark \ref{R-8-16}). In $M$ $=$ $%
\Sigma /\sigma $ let $\tilde{F}$ (resp. $\tilde{F}_{j})$ denote a fixed
point orbifold (resp. the $j$-chart part of $\tilde{F}$); see (\ref{Fix})
and the paragraph there for the precise meaning of $\tilde{F}.$ Let $dv_{%
\tilde{F}_{j}},$ $dv_{\tilde{F}}$ denote the (induced) volume forms on $%
\tilde{F}_{j}$ $\subset $ $M,$ $\tilde{F}$ $\subset $ $M$ respectively with
respect to the metric $g_{M}$ (Notation \ref{N-3-1}). Let $d\tilde{v}%
_{V_{j}^{\tilde{g}}}$ be the pullback of $dv_{\tilde{F}_{j}}$ by $V_{j}^{%
\tilde{g}}$ $\rightarrow $ $\tilde{F}_{j}$ (the $j$-part orbifold chart for
some $\tilde{g}$ $\in $ $\mathcal{G}$ in (\ref{8-41-1}) under the
restriction of $\pi $ $:$ $\Sigma $ $\rightarrow $ $\Sigma /\sigma $ to $%
V_{j}^{\tilde{g}}\times \{0\}\times \{1\}.$ Let $dv_{\Sigma ^{\tilde{g}},m}$
denote the volume form of $\Sigma ^{\tilde{g}}$ with respect to the metric
induced from $G_{a,m}.$ With respect to the metrics $g_{M}$ and $G_{a,m},$
the relation between these volume forms is given by

\begin{lemma}
\label{L-VFub} With the notation above and in Section \ref{S-metric}, we
have $i)$ 
\begin{eqnarray}
dv_{\Sigma ^{\tilde{g}},m} &=&\pi ^{\ast }dv_{\tilde{F}}\wedge
dv_{f,m}|_{\Sigma ^{\tilde{g}}}  \label{Str-8} \\
&=&d\tilde{v}_{V_{j}^{\tilde{g}}}\wedge (l^{-m}d\hat{v}_{m})|_{|\Sigma
_{j,k}|}\text{ \ \ (}\tilde{g}=g_{k}^{(j)}\text{ in (\ref{Sigmak}));}  \notag
\end{eqnarray}%
$ii)$ for $m=0,$ given $p_{0}=\{z\}$ $\times $ \{$0\}$ $\times $ $\{1\}\in
W_{j},$ choose new coordinate $w$ $=$ $|w|e^{i\phi }$ with $z$-coordinate
fixed so that%
\begin{eqnarray}
dv_{\Sigma ^{\tilde{g}},0}|_{\{z\}\times (-\varepsilon _{j},\varepsilon
_{j})\times \mathbb{R}^{+}} &=&\pi ^{\ast }dv_{\tilde{F}_{j}}(z)\wedge d\hat{%
v}_{0}|_{\Sigma ^{\tilde{g}}}(w)  \label{vm0} \\
&=&d\tilde{v}_{V_{j}^{\tilde{g}}}(z)\wedge dv_{0}(|w|)\wedge \frac{dv(\phi )%
}{2\pi }.  \notag
\end{eqnarray}%
Here note that $V_{j}$ is determined by the \textquotedblleft old"
coordinate $w$ ($=1$).
\end{lemma}

\begin{proof}
For $i)$ the first equality of (\ref{Str-8}) follows from (\ref{volume})
(see also (\ref{mrp})). In view of (\ref{Sigmak}), (\ref{sing}) the second
equality of (\ref{Str-8}) follows from the definition of $d\tilde{v}_{V_{j}^{%
\tilde{g}}}$ and (\ref{fibrenv_0}). For $ii)$ with the metric on $\Sigma ^{%
\tilde{g}}$ induced from the metric $G_{a,m}$ ($m$ $=$ $0$) on $\Sigma ,$
the volume form $dv_{\Sigma ^{\tilde{g}},0}|_{\{z\}\times (-\varepsilon
_{j},\varepsilon _{j})\times \mathbb{R}^{+}}$along a local $\mathbb{C}^{\ast
}$-orbit of $p_{0}$ reads as (\ref{vm0}) by (\ref{volume}), (\ref{3-18.75})
with the reasoning there for $m=0:$ given $p_{0}=\{z\}$ $\times $ \{$0\}$ $%
\times $ \{1\}$,$ choose new $w$ $=$ $|w|e^{i\phi }$ such that $h(z,\bar{z})$
$=$ $1,$ $dh(z,\bar{z})$ $=$ $0$ at $p_{0}.$ Hence (\ref{vm0}) holds.
\end{proof}

The notation $d\tilde{v}_{V_{j}^{\tilde{g}}}$ is the volume form on $V_{j}^{%
\tilde{g}}$ with respect to the metric $\pi ^{\ast }g_{M};$ this metric is
to be distinguished from the metric $G_{a,m}.$ It is worthwhile noting the
following.

\begin{remark}
\label{R-z-vol} Let $dv_{V_{j}^{\tilde{g}}}$ denote the volume form induced
by $G_{a,m}.$ Then $d\tilde{v}_{V_{j}^{\tilde{g}}}$ may not equal $%
dv_{V_{j}^{\tilde{g}}}$ in general in view of (\ref{M0-1}), (\ref{M0-2}) and
(\ref{metric}). Compare Remark \ref{R-8-16}.
\end{remark}

With reference to (\ref{Str-4}), the integrand to be desired (cf. the
paragraph after (\ref{Str-4})) is now seen by the following.

\begin{proposition}
\label{L-sing} With the notation above, we have%
\begin{eqnarray}
&&\tilde{g}^{-m}\int_{\Sigma ^{\tilde{g}}}\overline{\mathcal{F}_{\tilde{g}%
,m}(x)}l^{m}(x)dv_{\Sigma ^{\tilde{g}},m}  \label{Str-9} \\
&=&\sum_{(j,k):g_{k\text{ }}^{(j)}=\tilde{g}}(g_{k}^{-1})^{m}\int_{V_{j}^{%
\gamma _{k}}\times (-\varepsilon _{j},\varepsilon _{j})\times \mathbb{R}%
^{+}}\varphi _{j}(z,\phi )\overline{\mathcal{F}_{k,m}^{j}(z,\bar{z})}  \notag
\\
&&\text{ \ \ \ \ \ \ \ \ \ \ \ \ \ \ \ \ \ \ \ \ \ \ \ \ \ \ \ \ \ \ }d%
\tilde{v}_{V_{j}^{\gamma _{k}}}(z)\wedge d\hat{v}_{m}(z,\phi ,|w|).  \notag
\end{eqnarray}
\end{proposition}

\begin{proof}
Observe that by (\ref{Str-8}) $l^{m}(x)dv_{\Sigma ^{\tilde{g}},m}$ $=$ $d%
\tilde{v}_{V_{j}^{\gamma _{k}}}(z)\wedge d\hat{v}_{m}$ on $\Sigma ^{\tilde{g}%
}\cap D_{j}.$ From this, $\Sigma _{j}\varphi _{j}=1$ and Corollary \ref%
{R+Inv}, (\ref{Str-9}) follows.
\end{proof}

Recall that the largest period is $\frac{2\pi }{p}$ (see the paragraph after
(1.5)). We can now prove the main result Theorem \ref{main_theorem}.

\subsection{The local index formula completed\label{Subs-8-3}}

\begin{proof}
\textbf{(of Theorem \ref{main_theorem})} To show $i)$ of the theorem,
observe that the formula (\ref{Part-Ia}) gives the $HRR_{m}$ term in the RHS
of (\ref{Str-F}). To compute the singular part of $StrP_{m,t}^{0}(x,x)$ we
sum up over $j$ the second term in the RHS of (\ref{Str-2}) in view of (\ref%
{Kmt}). This results in getting the term $\overline{\mathcal{F}_{\tilde{g},m}%
}$ (the complex conjugate of $\mathcal{F}_{\tilde{g},m}$ in (\ref{Str-7}))
by noting that $\alpha _{k}\alpha _{0}^{-1}$ $=$ $\alpha _{k}(x)\alpha
_{0}^{-1}(x)$ $=$ $g_{k}^{-1}$ by (\ref{alk}) and $\tau _{j}$ $=$ $1$ on
supp $\varphi _{j}.$ We have shown (\ref{Str-F}) for $StrP_{m,t}^{0}(x,x)$
as $t\rightarrow 0$.

To show $ii)$ for the index formula$,$ from the McKean-Singer type formula (%
\ref{5.5}) for $\tilde{\square}_{m}^{c\pm }$ (with $E$ added by Remark \ref%
{5-17.5}) together with Theorem \ref{t-MS}, it follows that%
\begin{equation}
index(\bar{\partial}_{\Sigma ,m}^{E}\text{-complex})=\lim_{t\rightarrow
0}\int_{\Sigma }[TrP_{m,t}^{0,+}(x,x)-TrP_{m,t}^{0,-}(x,x)]dv_{\Sigma ,m}.
\label{Str-12}
\end{equation}

\noindent By (\ref{Str-3}), (\ref{Str-4}) and (\ref{Str-9}), we have 
\begin{eqnarray}
RHS\text{ of (\ref{Str-12}) } &=&\lim_{t\rightarrow 0}\int_{\Sigma }\text{%
Part I of }StrP_{m,t}^{0}(x,x)\text{ }dv_{\Sigma .m}  \label{Str-13} \\
&&+\sum_{\tilde{g}\in \mathcal{G},\text{ }\tilde{g}\neq 1}\tilde{g}%
^{-m}\int_{\Sigma ^{\tilde{g}}}\overline{\mathcal{F}_{\tilde{g},m}(x)}%
l^{m}(x)dv_{\Sigma ^{\tilde{g}},m}.  \notag
\end{eqnarray}

\noindent The first term in the RHS of (\ref{Str-13}) is reduced to the RHS
of (\ref{Part-I}) by Theorem \ref{P-main}. The second term in the RHS of (%
\ref{Str-13}) is real-valued since the other terms in (\ref{Str-13}) are
real-valued. Thus (\ref{MF}) follows from (\ref{Str-12}) and taking the
complex conjugate of the second term in the RHS of (\ref{Str-13}).
\end{proof}


\begin{remark}
\label{R-sum} In comparison with \cite[(14.3) and (14.4)]{Du} a sum over $%
\tilde{g}$ or $k$ in our formula (\ref{Str-13}) is anticipated; see also
Remark \ref{R-8-40}. A detailed comparison is made in the next subsection.
\end{remark}

In these subsections devoted to the local $m$-index density, we have given
an expression based on the language of \cite{BGV}, which are written in the
setting of Riemannian geometry. For complex manifolds here, it is desirable
to express $\mathcal{F}_{\tilde{g},m}$ in (\ref{Str-7}) in terms of Todd
form $Td(T\Sigma ^{\tilde{g}}/L_{\Sigma ^{\tilde{g}}},g_{quot})$, \textit{%
twisted} Chern character form $ch(\gamma _{k}^{E},E\mathcal{)}$ (it is the
usual Chern character form twisted by $\gamma _{k}^{E}$ in the sense of (\ref%
{ch})) where $E$ is a $\mathbb{C}^{\ast }$-equivariant holomorphic vector
bundle and the (usual/untwisted) Chern character form $ch(L_{\Sigma }^{\ast
})^{\otimes m}$. The expression that we will end up with is the following:
for $(j,k)$ such that $g_{k\text{ }}^{(j)}=\tilde{g},$ 
\begin{eqnarray}
&&\ \ \ \mathcal{F}_{\tilde{g},m}|_{\Sigma ^{\tilde{g}}\cap D_{j}}\overset{(%
\ref{Str-7a})}{=}\mathcal{F}_{k,m}^{j}(z,\bar{z})=  \label{FkTd} \\
&&T_{V_{j}^{\gamma _{k}}}\frac{Td(T\Sigma ^{\tilde{g}}/L_{\Sigma ^{\tilde{g}%
}},g_{quot})ch(\gamma _{k}^{\psi _{j}^{\ast }E|_{V_{j}}},\psi _{j}^{\ast
}E|_{V_{j}}\mathcal{)}ch(\psi _{j}^{\ast }(L_{\Sigma }^{\ast })^{\otimes
m}|_{V_{j}}\mathcal{)}}{\det (1-(\tilde{g}^{-1})_{1}^{c}\exp (-\frac{i}{2\pi 
}R_{c}^{1}))}  \notag
\end{eqnarray}

\noindent (see the paragraph before Notation \ref{N-b-1} below for notations
involving the superscript/subscript \textquotedblleft $c$" in $(\tilde{g}%
^{-1})_{1}^{c}$ and $R_{c}^{1}$ above). Some details for deducing (\ref{FkTd}%
) is given in Subsection \ref{S-8-4} below. Formula (\ref{FkTd}) allows us
to compare our result with \cite[p.184, (14.4)]{Du} which corresponds to the 
$m=0$ case of (\ref{FkTd}).

An application of the $m$-index on some two-dimensional $\Sigma $ yields
algebraic identities that are perhaps interesting and nontrivial (see (\ref%
{ku})).

\begin{example}
\label{E-IF} Consider $\tilde{M}=\mathbb{CP}^{1}$ also viewed as $S^{2},$
the unit sphere in $\mathbb{R}^{3}.$ Let $l$ be an integer larger than or
equal to $2.$ Let $g=e^{\frac{2\pi i}{l}}\in G$ $=$ $\mathbb{Z}_{l}$ $%
\subset $ $S^{1}$ $\subset \mathbb{C}^{\ast }$ act on $S^{2}$ by a rotation
of $\frac{2\pi }{l}$ degree around the $z$-axis. So the north pole $%
N=(0,0,1) $ and the south pole $S=(0,0,-1)$ are the only fixed points of $g$%
, $g^{2},$ $\cdot \cdot ,$ $g^{l-1}.$ Let $K_{\tilde{M}}$ denote the
canonical line bundle over $\tilde{M}.$ The $G$-action on $\tilde{M}$
induces an action on $K_{\tilde{M}}$ by pulling back the forms. Observe that%
\begin{equation}
\Sigma :=(K_{\tilde{M}}\backslash \{0\text{-section}\})/G  \label{Ex1}
\end{equation}%
is a complex surface with a $\mathbb{C}^{\ast }$-action induced by the
natural $\mathbb{C}^{\ast }$-action on $K_{\tilde{M}}\backslash \{0$-section$%
\},$ which becomes a locally free action denoted as $\sigma _{s}$ on $\Sigma 
$ in the sense of Theorem \ref{main_theorem}. It follows that 
\begin{equation*}
\Sigma /\mathbb{C}^{\ast }(or\text{ }\Sigma /\sigma _{s})\cong \tilde{M}/G=%
\mathbb{CP}^{1}/\mathbb{Z}_{l}
\end{equation*}%
is a (1-dimensional) compact complex orbifold with the two orbifold points $%
N $ and $S.$

Take $m=0$ in (\ref{MF}). We first compute $index$ $\bar{\partial}_{\Sigma
,m}$ for $m$ $=$ $0.$ Using (\ref{9.17-5}) with $P$ $=$ $\Sigma ,$ $M$ $=$ $%
\Sigma /\mathbb{\sigma }_{s}$ $=$ $\Sigma /\mathbb{C}^{\ast }$ by Remark \ref%
{9.3} we get $h_{m=0}^{0}(\Sigma ,\mathcal{O)}$ $\mathcal{=}$ $h^{0}(\Sigma /%
\mathbb{C}^{\ast },\mathcal{O}_{\Sigma /\mathbb{C}^{\ast }})$ $=$ $1$ and $%
h_{m=0}^{1}(\Sigma ,\mathcal{O)}$ $\mathcal{=}$ $h^{1}(\Sigma /\mathbb{C}%
^{\ast },\mathcal{O}_{\Sigma /\mathbb{C}^{\ast }})$ $=$ $h^{1}(\mathbb{CP}%
^{1}/\mathbb{Z}_{l},\mathcal{O}_{\mathbb{CP}^{1}/\mathbb{Z}_{l}}).$ The fact
that $H^{1}(\mathbb{CP}^{1}/\mathbb{Z}_{l},$ $\mathcal{O}_{\mathbb{CP}^{1}/%
\mathbb{Z}_{l}})$ is easily seen to be a $\mathbb{Z}_{l}$-invariant subspace
of $H^{1}(\mathbb{CP}^{1},\mathcal{O}_{\mathbb{CP}^{1}}\mathcal{)}$ which
equals $0,$ gives the LHS of (\ref{MF}):%
\begin{equation*}
index(\bar{\partial}_{\Sigma ,m=0})=h_{m=0}^{0}(\Sigma ,\mathcal{O)-}%
h_{m=0}^{1}(\Sigma ,\mathcal{O)}=1.
\end{equation*}%
By $h^{0}(\mathbb{CP}^{1},\mathcal{O)-}h^{1}(\mathbb{CP}^{1},\mathcal{O)}$ $%
= $ $1$ given by the similar index formula, it is seen that the first term
on the RHS of (\ref{MF}) equals $1/l.$ For the remaining terms of (\ref{MF})
with $\tilde{g}$ $=$ $g,$ $g^{2},$ $\cdot \cdot ,$ $g^{l-1}$ $\in $ $G,$ $%
\Sigma ^{\tilde{g}}$ $(\subset $ $\Sigma )$ consists of two fibres of $%
\Sigma $ at $N$ and $S$ each with area $\frac{1}{l}$ with respect to the
measure $dv_{0}(|w|)$ $\frac{dv(\phi )}{2\pi }$ (see (\ref{Sigmak}) and (\ref%
{vm0})). It is not difficult to see that the contribution from $\Sigma ^{%
\tilde{g}}$ associated with $(N,g)$, $(N,g^{2}),$ $\cdot \cdot \cdot ,$ $%
(N,g^{l-1}),$ $\tilde{g}$ $=$ $g^{k},$ $1\leq k\leq l-1$ in the RHS of (\ref%
{MF}) using (\ref{FkTd}) without $E$ gives, where $1-g^{-k}$ below is from
the denominator of (\ref{FkTd}) and found to be ($\tilde{g}^{-1})_{1}^{c}$ $%
= $ $g^{-k}$,%
\begin{eqnarray}
\frac{1}{l}\sum_{k=1}^{l-1}\frac{1}{1-g^{-k}} &=&\frac{1}{l}\sum_{k=1}^{l-1}(%
\frac{1}{2}-i\frac{\sin \frac{2\pi k}{l}}{2(1-\cos \frac{2\pi k}{l})})
\label{m0} \\
&=&\frac{1}{l}\sum_{k=1}^{l-1}\frac{1}{2}=\frac{l-1}{2l}  \notag
\end{eqnarray}%
in view of $\Sigma _{k=1}^{l-1}\sin \frac{2\pi k}{l}/(1-\cos \frac{2\pi k}{l}%
)$ $=$ $0$ since the complex conjugate of $\sum_{k=1}^{l-1}\frac{1}{1-g^{-k}}
$ equals itself ($\bar{g}^{-k}$ $=$ $g^{-l+k})$. Similarly the contribution
from $\Sigma ^{\tilde{g}}$ associated with $(S,g)$, $(S,g^{2}),$ $\cdot
\cdot ,$ $(S,g^{l-1})$ also gives $\frac{l-1}{2l}.$ Altogether we get $\frac{%
1}{l}$ $+$ $\frac{l-1}{2l}$ $+$ $\frac{l-1}{2l}$ $=$ $1$ for the RHS of (\ref%
{MF}) and hence have verified (\ref{MF}) for $m$ $=$ $0$.

We turn now to $m>0.$ Write $G$ $=$ $\mathbb{Z}_{l}$. The situation is now
equivalent to adding an orbifold line bundle ($K_{\tilde{M}/G}^{\ast
})^{\otimes m}$; this follows from Remark \ref{9.3} using (\ref{9.17-5})
with $P$ $=$ $\Sigma ,$ $M$ $=$ $\Sigma /\mathbb{\sigma }_{s}$ $=$ $\tilde{M}%
/G$ and (\ref{10.30-5}) and (\ref{Ex1}) with $L_{M}$ $=$ $\Sigma \times
_{\sigma _{s}}\mathbb{C}$ $=$ $K_{\tilde{M}}/G$ $=$ $K_{\tilde{M}/G}$ as a
holomorphic orbifold line bundle over $M$ $=$ $\tilde{M}/G.$ Denote the
holomorphic line bundle of degree $d$ over $\tilde{M}$ $=$ $\mathbb{CP}^{1}$
by $\mathcal{O}(d).$ By $K_{\tilde{M}}$ $=$ $\mathcal{O}(-2)$%
\begin{equation}
H^{1}(\tilde{M},(K_{\tilde{M}}^{\ast })^{\otimes m})=H^{0}(\tilde{M},%
\mathcal{O}(-2m-2))=0  \label{h1}
\end{equation}%
for $m$ $\geq $ $0.$ It follows from (\ref{9.17-5}) (which holds for
orbifolds via Remark \ref{9.3}) that $H_{m}^{1}(\Sigma ,\mathcal{O})$ $%
\simeq $ $H^{1}(\tilde{M}/G,$($K_{\tilde{M}/G}^{\ast })^{\otimes m}),$ a $G$%
-invariant subspace of $H^{1}(\tilde{M},(K_{\tilde{M}}^{\ast })^{\otimes m})$
as in the $m$ $=$ $0$ case since $G$ is a finite group, vanishes via (\ref%
{h1}). Now we are going to compute the dimension of $H_{m}^{0}(\Sigma ,%
\mathcal{O})$ $\simeq $ $H^{0}(\tilde{M}/G,$($K_{\tilde{M}/G}^{\ast
})^{\otimes m})$ $\simeq $ $G$-invariant elements of $H^{0}(\tilde{M},(K_{%
\tilde{M}}^{\ast })^{\otimes m}).$ Let $[z:w]$ denote the homogeneous
coordinates of $\tilde{M}=\mathbb{CP}^{1}$ with $G$ $=$ $\mathbb{Z}_{l}$
action by $[z:w]$ $\rightarrow $ $[e^{2\pi i/l}z:w].$ Write a holomorphic
section of $T\tilde{M}^{\otimes m}$ $=$ $(K_{\tilde{M}}^{\ast })^{\otimes m}$
in $[z:1]$ as $f(z)(\frac{\partial }{\partial z})^{m}$ for a polynomial $%
f(z) $ of degree $\leq $ $2m.$ Its $G$-invariance implies that 
\begin{equation}
f(e^{2\pi i/l}z)=e^{2\pi mi/l}f(z).  \label{h0m}
\end{equation}%
Writing $f(z)$ $=$ $\sum_{k=0}^{2m}c_{k}z^{k}$ and $m$ $\equiv $ $r$ mod $l$
for some $0$ $\leq $ $r$ $<$ $l$.$,$ we obtain by (\ref{h0m}) that for $%
c_{k} $ $\neq $ $0,$ $e^{2\pi ik/l}$ must be $e^{2\pi ir/l}$. Write 
\begin{equation}
\kappa (l,m):=\text{the number of nonnegative integers }n\text{ satisfying }%
r+l\cdot n\leq 2m.  \label{8-59-1}
\end{equation}%
So $h_{m}^{0}(\Sigma ,\mathcal{O})$ $=$ $h^{0}(\tilde{M}/G,$ ($K_{\tilde{M}%
/G}^{\ast })^{\otimes m})$ equals $\kappa (l,m).$ Thus%
\begin{equation}
index\text{ }\bar{\partial}_{\Sigma ,m}\equiv h_{m}^{0}(\Sigma ,\mathcal{O)-}%
h_{m}^{1}(\Sigma ,\mathcal{O)=}\text{ }\kappa (l,m)  \label{ind-m}
\end{equation}%
\noindent as the LHS of (\ref{MF}). For instance, if $l$ $|$ $m$ then $r=0$
and $\kappa (l,m)$ $=$ $\frac{2m}{l}$ $+$ $1.$

We now compute the RHS of (\ref{MF}) for $m$ $\geq $ $0$, the first term of
which being the integral of HRR$_{m}$ equals $\frac{2m+1}{l}$ by similar
arguments as in the $m=0$ case above. The contribution from $\Sigma ^{\tilde{%
g}}$ associated with $(N,g)$, $(N,g^{2}),$ $\cdot \cdot ,$ $(N,g^{l-1})$
(resp. $(S,g)$, $(S,g^{2}),$ $\cdot \cdot ,$ $(S,g^{l-1}))$ in the RHS of (%
\ref{MF}) using (\ref{FkTd}) without $E$ gives, where the numerator $g^{km}$
below is from $\tilde{g}^{m}$ of (\ref{MF}),%
\begin{equation}
\frac{1}{l}\sum_{k=1}^{l-1}\frac{g^{km}}{1-g^{-k}}=:\mu _{N}(l,m)\text{
(resp. }\mu _{S}(l,m)\text{)},\text{ \ }g=e^{\frac{2\pi i}{l}}.  \label{km}
\end{equation}%
\noindent Clearly $\mu _{N}(l,m)$ $=$ $\mu _{S}(l,m)$ denoted by $\mu (l,m).$
To verify (\ref{MF}) for any integers $l\geq 2,$ $m\geq 0,$ is the same as
to show the following identity%
\begin{equation}
\kappa (l,m)=\frac{2m+1}{l}+2\mu (l,m).  \label{ku}
\end{equation}%
\noindent Write $x_{k}=g^{k}.$ So $\frac{g^{km}}{1-g^{-k}}$ in (\ref{km})
equals $\frac{x_{k}^{m+1}}{x_{k}-1}.$ It follows from (\ref{km}) that 
\begin{eqnarray}
l\cdot \mu (l,m) &=&\sum_{k=1}^{l-1}\frac{x_{k}^{m+1}}{x_{k}-1}  \label{xk}
\\
&=&\sum_{k=1}^{l-1}\frac{x_{k}^{m+1}-1}{x_{k}-1}+\sum_{k=1}^{l-1}\frac{1}{%
x_{k}-1}.  \notag
\end{eqnarray}%
\noindent The first term of the RHS in (\ref{xk}) equals%
\begin{eqnarray}
&&\sum_{k=1}^{l-1}(x_{k}^{m}+x_{k}^{m-1}+\cdot \cdot \cdot +x_{k}+1)
\label{xks} \\
&=&\sum_{a=0}^{m}\sum_{k=1}^{l-1}x_{k}^{a}=\sum_{a=0}^{m}%
\sum_{k=1}^{l-1}(g^{a})^{k}  \notag \\
&=&\sum_{0\leq a\leq m,g^{a}=1}\sum_{k=1}^{l-1}1+\sum_{0\leq a\leq
m,g^{a}\neq 1}\frac{g^{a}-(g^{a})^{l}}{1-g^{a}}\text{ (note that }g^{l}=1) 
\notag \\
&=&(l-1)\cdot \#\{a\text{ }|0\leq a\leq m,g^{a}=1\}+(-1)\cdot \#\left\{ a%
\text{ }|0\leq a\leq m,g^{a}\neq 1\right\}  \notag \\
&=&(l-1)(1+q)+(-1)(m-q)=-m+l-1+l\cdot q  \notag
\end{eqnarray}%
\noindent where $q$ is the nonnegative integer such that $m=q\cdot l+r,$ $%
0\leq r<l.$ In the RHS of (\ref{xk}) the second term via (\ref{m0}) reads%
\begin{equation}
\sum_{k=1}^{l-1}\frac{1}{x_{k}-1}=-\frac{l-1}{2}.  \label{xk2}
\end{equation}%
\noindent Substituting (\ref{xks}) and (\ref{xk2}) into (\ref{xk}) we obtain%
\begin{eqnarray}
\mu (l,m) &=&\frac{1}{l}[(-m+l-1+l\cdot q)-\frac{l-1}{2}]  \label{ulp} \\
&=&\frac{1}{l}(-m+\frac{l-1}{2}+l\cdot q).  \notag
\end{eqnarray}%
\noindent By (\ref{ulp}) the RHS of (\ref{ku}) reads as%
\begin{equation}
\frac{2m+1}{l}+2\mu (l,m)=1+2q.  \label{rhs}
\end{equation}%
\noindent From $r$ $=$ $m-q\cdot l$ and $r+l\cdot n$ $\leq $ $2m$ it follows
that $n$ $\leq $ $2q$ $+$ $\frac{r}{l}$ hence $n$ $=$ $0,$ $1,$ $\cdot \cdot
,$ $2q.$ So $\kappa (l,m)$--the number of the nonnegative integers $n$
satisfying $r+l\cdot n$ $\leq $ $2m$, is exactly $1+2q.$ By this and (\ref%
{rhs}) we have proved (\ref{ku}).
\end{example}

Finally let us indicate the following fact which is of topological nature.

\begin{proposition}
\label{pinv} The first integral of (\ref{MF}) in Theorem \ref{main_theorem}
is independent of the choice of $\mathbb{C}^{\ast }$-invariant connections
on $T^{1,0}(\Sigma )/L_{\Sigma },$ $E$ and $L_{\Sigma }$ respectively, used
for computing the associated Todd form and Chern character forms in (\ref{MF}%
). Furthermore the similar conclusion holds for those integrals over $\Sigma
^{\tilde{g}}$ of (\ref{MF})$.$
\end{proposition}

\proof
We follow the notation in the preceding proof. From \cite[Appendix B.5]{MM}
we see that the "d-exact" objects, resulting from the difference between the
characteristic forms associated with different $\mathbb{C}^{\ast }$%
-invariant connections, can be chosen to be "d($\mathbb{C}^{\ast }$%
-invariant forms)". We are then reduced to checking the following vanishing
on the noncompact space $\Sigma $ (for the first integral of (\ref{MF}))%
\begin{equation}
\int_{\Sigma }dQ\wedge d\hat{v}_{m}=0  \label{dQ}
\end{equation}

\noindent where $Q$ is a $\mathbb{C}^{\ast }$-invariant $(2n-3)$-form. That
the integrand in (\ref{dQ}) is $L^{1}$-integrable is easily checked (cf.
Remark \ref{3-r}).

Some preparations are in order. Take a $\mathbb{C}^{\ast }$-invariant
distance function $\rho (x,\Sigma _{\text{sing}})$, i.e. $\rho (\sigma
(\lambda )(x),\Sigma _{\text{sing}})$ $=$ $\rho (x,\Sigma _{\text{sing}}),$ $%
\lambda $ $\in $ $\mathbb{C}^{\ast },$ which can be constructed from a
distance function on $M:=\Sigma /\sigma $ using $g_{M}$ (thus degenerate
along the $\mathbb{C}^{\ast }$-orbits)$.$ Let $0$ $\leq $ $\chi
_{\varepsilon }$ $\leq $ $1$ on $\Sigma $ be a $\mathbb{C}^{\ast }$%
-invariant $C^{\infty }$ cut-off function: $\chi _{\varepsilon }(x)$ $=$ $1$
if $\rho (x,\Sigma _{\text{sing}})$ $\geq $ $2\varepsilon ,$ $\chi
_{\varepsilon }(x)$ $=$ $0$ if $\rho (x,\Sigma _{\text{sing}})$ $\leq $ $%
\varepsilon $ and $\chi _{\varepsilon }(\sigma (\lambda )(x))$ $=$ $\chi
_{\varepsilon }(x),$ $|d\chi _{\varepsilon }|_{G_{a,m}}$ $=$ $O(\frac{1}{%
\varepsilon }).$ For the last condition we use (\ref{metric}) for $G_{a,m}$
(see also (\ref{M1-1})) with the $\mathbb{C}^{\ast }$-invariance of $\chi
_{\varepsilon }.$ Since the action $\sigma $ is globally free on $\Sigma
\backslash \Sigma _{\text{sing}}$ ($p=p_{1}=1$ by adjusting $\sigma $ to $%
\tilde{\sigma}(le^{i\theta })$ $=$ $\sigma (le^{i\theta /p}))$ and supp$%
(\chi _{\varepsilon })$ $\subset $ $\Sigma \backslash \Sigma _{\text{sing}}$%
, it follows that $D_{\varepsilon }$ $:=$ supp$(\chi _{\varepsilon })/\sigma 
$ ($\subset $ $M:=\Sigma /\sigma )$ is a smooth manifold with the boundary $%
\partial D_{\varepsilon }.$

Back to (\ref{dQ}) which equals%
\begin{eqnarray}
&&\int_{\Sigma \backslash \Sigma _{\text{sing}}}dQ\wedge d\hat{v}_{m}
=\lim_{\varepsilon \rightarrow 0}\int_{\Sigma \backslash \Sigma _{\text{sing}%
}}\chi _{\varepsilon }dQ\wedge d\hat{v}_{m}  \label{int3} \\
&=&\lim_{\varepsilon \rightarrow 0}\int_{\Sigma \backslash \Sigma _{\text{%
sing}}}d(\chi _{\varepsilon }Q)\wedge d\hat{v}_{m}-\lim_{\varepsilon
\rightarrow 0}\int_{\Sigma \backslash \Sigma _{\text{sing}}}(d\chi
_{\varepsilon })Q\wedge d\hat{v}_{m}  \notag
\end{eqnarray}

Now we compute $\int_{\Sigma \backslash \Sigma _{\text{sing}}}d(\chi
_{\varepsilon }Q)\wedge d\hat{v}_{m}$ on the RHS of (\ref{int3}), which
equals%
\begin{eqnarray}
&&\int_{D_{\varepsilon }}d_{M}(\chi _{\varepsilon }Q)\int_{\mathbb{C}^{\ast }%
\text{-orbit}}d\hat{v}_{m}\text{ (by }\mathbb{C}^{\ast }\text{-invariance of 
}\chi _{\varepsilon }Q\text{)}  \label{int1} \\
&\overset{\text{(\ref{fibrenv})}}{=}&(\int_{\partial D_{\varepsilon }}\chi
_{\varepsilon }Q)\cdot 1=0\text{ (}\chi _{\varepsilon }|_{\partial
D_{\varepsilon }}=0)  \notag
\end{eqnarray}

For the last term in (\ref{int3}), using $N_{\varepsilon }$ $=$ $\{x\in
\Sigma :$ $\varepsilon \leq \rho (x,\Sigma _{\text{sing}})\leq 2\varepsilon
\}/\sigma $ ($\subset $ $M:=\Sigma /\sigma )$ as a $C^{\infty }$ manifold
with boundary, we compute (recalling $|d\chi _{\varepsilon }|$ $=$ $O(\frac{1%
}{\varepsilon }))$:%
\begin{equation}
\int_{\Sigma \backslash \Sigma _{\text{sing}}}(d\chi _{\varepsilon })Q\wedge
d\hat{v}_{m}=\int_{N_{\varepsilon }}(d\chi _{\varepsilon })Q\int_{\mathbb{C}%
^{\ast }\text{-orbit}}d\hat{v}_{m}=O(\frac{1}{\varepsilon }%
)vol(N_{\varepsilon })\cdot 1\rightarrow 0  \label{int2}
\end{equation}

\noindent as $\varepsilon \rightarrow 0$ since $vol(N_{\varepsilon })$ $=$ $%
O(\varepsilon ^{2})$ in view that the real codimensin of $\Sigma _{\text{sing%
}}$ (resp. $\Sigma _{\text{sing}}/\sigma )$ in $\Sigma $ (resp. $\Sigma
/\sigma $ $=$ $M)$ is larger or equal to $2.$ The assertion (\ref{dQ})
follows from (\ref{int3}), (\ref{int1}) and (\ref{int2}). To get similar
conclusion for those integrals on $\Sigma ^{\tilde{g}},$ we notice that $V^{%
\tilde{g}}$ $=$ $V^{H_{k+1}}$ by (\ref{VgH}) in Lemma \ref{L-Vg} $i).$ It
follows (cf. Lemma \ref{L-Vg}) that $S^{1}/H_{k+1}$ $\times $ $\mathbb{R}%
^{+} $ acts on $\Sigma ^{\tilde{g}}\backslash \Sigma _{\text{sing}}^{\tilde{g%
}}$ (globally) freely, where $\Sigma _{\text{sing}}^{\tilde{g}}$ consists of
lower dimensional strata. The remaining arguments are then similar to those
from (\ref{int3}) to (\ref{int2}).

\endproof%

We remark that in \cite{CHT} the statement and proof of the off-diagonal
estimate (ODE for short) \cite[Theorem 5.10, p.78]{CHT} are correct but
unfortunately ODE is not properly applied to the \textquotedblleft
supertrace" computation --- i.e. to the proof of \cite[Theorem 6.4, cf.
(6.21) on p.98]{CHT}. So the resulting index formula as stated there (cf. 
\cite[Theorem 1.10, Corollary 1.13, Theorem 1.28]{CHT}) is not entirely
correct (see \cite{CHT-E} for an erratum to \cite{CHT}) unless certain
conditions are imposed on the underlying CR manifolds. The misuse of ODE
occurs in \cite[(6.21)]{CHT} where the supertrace computation involves
\textquotedblleft pullbacks", for which our application of ODE is not quite
valid because the pullback operation may produce nontrivial endomorphisms of
the bundles under consideration. Nevertheless we refer to \cite{CT-1} for
special situations to which the original index formulas of \cite{CHT} as
just mentioned do apply.

\subsection{Comparison with Duistermaat's formula for the K\"{a}hler case,
Part I: from real to complex\label{S-8-4}}

In this subsection we are going to convert the real expression of $\mathcal{F%
}_{\tilde{g},m}$ in (\ref{Str-7}) into the complex version (\ref{FkTd}). The
formula (\ref{FkTd}) allows us to compare our result with that of
Duistermaat \cite[p.184, (14.4)]{Du}. The main result of this section is
Proposition \ref{P-Fk}, which proves (\ref{FkTd}) claimed in the last
subsection.

Our main references for this and the next subsections are \cite{BGV} and 
\cite{Du}. The former is mainly on the real situation while the latter is on
the (almost)-complex case. It is hoped that our presentation here may help
to clarify some points; see for instance Remarks \ref{R-BGV}, \ref{R-8-B}, %
\ref{R-8-C} and Footnote$^{12}$ below.

We start with the general setup and fix the notation. Let $X$ be a complex
manifold which plays the role as our $M=\Sigma /\sigma .$ For simplicity we
assume that $X$ is K\"{a}hler. Let $E$ be a holomorphic Hermitian vector
bundle over $X$ with trivial Clifford action. Consider the Clifford module $%
\mathcal{E}$ $\equiv $ $\mathcal{E}_{X}$ :$=$ $\Lambda ^{0,\ast }T^{\ast
}X\otimes E$ with the Clifford connection obtained by twisting the
Levi-Civita connection with the canonical (Chern) connection of $E.$ The
endomorphisms of a complex vector bundle are meant to be $\mathbb{C}$-linear.

Recall that the canonical $\hat{A}$-genus form $\hat{A}(X)$ and the Todd
genus form $Td(X)$ of the complex manifold $X$ are defined as follows: ($%
R^{+}$ denotes the curvature of $T^{1,0}X$ as in \cite[p.152]{BGV}) 
\begin{equation}
\hat{A}(X):=\det \left( \frac{(\frac{i}{2\pi }R^{+})/2}{\sinh [(\frac{i}{%
2\pi }R^{+})/2]}\right) =\det \left( \frac{(\frac{1}{2\pi i}R^{+})/2}{\sinh
[(\frac{1}{2\pi i}R^{+})/2]}\right)  \label{3-11}
\end{equation}%
\begin{equation}
Td(X):=\det \left( \frac{\frac{i}{2\pi }R^{+}}{1-e^{-\frac{i}{2\pi }R^{+}}}%
\right) =\det \left( \frac{\frac{1}{2\pi i}R^{+}}{e^{\frac{1}{2\pi i}R^{+}}-1%
}\right)  \label{3-10}
\end{equation}%
\noindent where the above convention involving $2\pi $-factors is different
from that in \cite[p.152]{BGV}, cf. (\ref{8-69a}). Let $\gamma _{X}$ be a
biholomorphic map and an isometry on $X.$ Let $X^{\gamma _{X}}$ $\subset $ $%
X $ denote the fixed point set of $\gamma _{X}$. Then $\gamma _{X}$ induces\
a complex (bundle) endomorphism $\gamma _{X}^{\Lambda ^{0,\ast }T^{\ast }X}$
acting on $\Lambda ^{0,\ast }T^{\ast }X$ over $X^{\gamma _{X}}.$ For $E$
above, assume that $\gamma _{X}^{E}$ is a holomorphic bundle map of $E$
covering the action of $\gamma _{X},$ preserving the Hermitian metric of $E$%
. Together we have a bundle map $\gamma _{X}^{\mathcal{E}}$ $=$ $\gamma
_{X}^{\Lambda ^{0,\ast }T^{\ast }X}\otimes \gamma _{X}^{E}$ on $\mathcal{E}%
|_{X^{\gamma _{X}}}.$ Note that $\gamma _{X}^{\mathcal{E}}$ is compatible
with the Clifford action and the Clifford connection. We have the following
for use in Lemma \ref{L-B-7}.

\begin{lemma}
\label{L-End} With the notation above, there is a canonical isomorphism $%
j:End(E)\rightarrow End_{C(X)}(\mathcal{E}_{X})$ over $X^{\gamma _{X}},$
where $C(X)$ denotes the real Clifford algebra of $X.$
\end{lemma}

\begin{proof}
We can embed $End(E)$ into $End_{C(X)}(\Lambda ^{0,\ast }T^{\ast }X\otimes
E)=End_{C(X)}(\mathcal{E}_{X})$ by extending the action on $\Lambda ^{0,\ast
}T^{\ast }X$ identically (note that the Clifford action on $E$ is trivial by
default). We denote this canonical embedding by $j.$ At each point $p$ of $%
X^{\gamma _{X}}$, $C(X)\otimes \mathbb{C\cong }End(\Lambda ^{0,\ast }T^{\ast
}X)$ at $p$ by \cite[Proposition 3.19]{BGV} and the center of $End(\Lambda
^{0,\ast }T^{\ast }X)$ at $p$ is $\mathbb{C}.$ It follows that $%
End_{C(X)}(\Lambda ^{0,\ast }T^{\ast }X)$ at $p$ is $\mathbb{C}$, so $%
End_{C(X)}(\mathcal{E}_{X})$ $\cong $ $End_{C(X)}(\Lambda ^{0,\ast }T^{\ast
}X)\otimes End(E)$ at $p$ is $End(E)$ at $p$. Therefore $j$ is surjective,
hence an isomorphism$.$
\end{proof}

Let $K^{\ast }$ (resp. $K_{X^{\gamma _{X}}}^{\ast }$) denote the dual of the
canonical line bundle of $X$ (resp. $X^{\gamma _{X}}$). Let $K_{N}^{\ast }$
denote the dual of the complex line bundle of the $(\frac{1}{2}\dim _{%
\mathbb{R}}\mathcal{N},0)$-forms on (the complexification of) the real 
\textit{normal bundle} $\mathcal{N}$ of $X^{\gamma _{X}}$ in $X$ (cf. \cite[%
p.153]{Du})$.$ Let $R^{K_{N}^{\ast }}$ (resp. $R^{K^{\ast }},$ $%
R^{K_{X^{\gamma _{X}}}^{\ast }}$) denote the curvature operator of $%
K_{N}^{\ast }$ (resp. $K^{\ast },$ $K_{X^{\gamma _{X}}}^{\ast }$)$.$ In view
of Lemma \ref{L-End} we define a twisting complex endomorphism $\gamma _{X}^{%
\mathcal{E}/\mathcal{S}}$ $\in $ $End_{C(X)}(\mathcal{E}_{X})$ by $\gamma
_{X}^{\mathcal{E}/\mathcal{S}}$ $:=$ $j(\gamma _{X}^{E}).$ Let $F_{0}^{%
\mathcal{E}/\mathcal{S}}$ (resp. $F_{0}^{E}$) denote the restriction of the
curvature $F^{\mathcal{E}/\mathcal{S}}$ (resp. $F^{E}$) to the fixed point
submanifold $X^{\gamma _{X}}$ as in (\ref{Str-k})$.$ We then define the $%
\gamma $-twisted Chern character forms in the sense of \cite[pp. 194-195]%
{BGV} (In strict conformity with \cite[pp.194-195]{BGV} we will define $%
ch_{BGV}$ later; see (\ref{chBGV}).): 
\begin{eqnarray}
ch(\gamma _{X}^{\mathcal{E}/\mathcal{S}},\mathcal{E}/\mathcal{S)}&{:=}&\text{%
Str}_{\mathcal{E}/\mathcal{S}}\text{ }(\gamma _{X}^{\mathcal{E}/\mathcal{S}%
}\exp (\frac{i}{2\pi }F_{0}^{\mathcal{E}/\mathcal{S}}))  \label{ch} \\
ch(\gamma _{X}^{E},E)&{:=}&\text{Str}_{E}\text{ }(\gamma _{X}^{E}\exp (\frac{%
i}{2\pi }F_{0}^{E})),\text{ both on }X^{\gamma _{X}}.  \notag
\end{eqnarray}

\noindent See \cite[p.113]{BGV} for the definition of the relative
supertrace Str$_{\mathcal{E}/\mathcal{S}}$ in (\ref{ch}).

\begin{lemma}
\label{L-B-7} With the notation above, we have, on $X^{\gamma _{X}}$%
\begin{equation}
\hat{A}(X^{\gamma _{X}})ch(\gamma _{X}^{\mathcal{E}/\mathcal{S}},\mathcal{E}/%
\mathcal{S})=Td(X^{\gamma _{X}})ch(\gamma _{X}^{E},E)\exp (\frac{\frac{i}{%
2\pi }R^{K_{N}^{\ast }}}{2})\text{.}  \label{A-Td}
\end{equation}
\end{lemma}

\begin{proof}
From \cite[p.152]{BGV} using the K\"{a}hler assumption on $X$ it follows
that 
\begin{equation}
F^{\mathcal{E}/\mathcal{S}}=\frac{1}{2}Tr_{T^{1,0}X}(R^{+})+F^{E}
\label{3-9}
\end{equation}%
where $R^{+}$ ($=(R_{X})^{+}$ $=$ $R_{X}^{+}$) denotes the curvature of the
bundle $T^{1,0}X$ (here we have identified $End(E)$ with $End_{C(X)}(%
\mathcal{E}_{X})$ through $j$ by Lemma \ref{L-End}). Observe that we have,
over $X^{\gamma _{X}}$%
\begin{equation}
Tr_{T^{1,0}X}(R^{+})=R^{K^{\ast }}=R^{K_{X^{\gamma _{X}}}^{\ast
}}+R^{K_{N}^{\ast }}.  \label{3-9a}
\end{equation}%
\noindent By (\ref{3-9}) and (\ref{3-9a}) restricted to $X^{\gamma _{X}}$ we
obtain%
\begin{equation}
\exp (\frac{i}{2\pi }F_{0}^{\mathcal{E}/\mathcal{S}})=\exp (\frac{\frac{i}{%
2\pi }R^{K_{X^{\gamma _{X}}}^{\ast }}}{2})\exp (\frac{\frac{i}{2\pi }%
R^{K_{N}^{\ast }}}{2})\exp (\frac{i}{2\pi }F_{0}^{E}).  \label{3-13}
\end{equation}%
\noindent On the other hand, from comparing (\ref{3-10}) with (\ref{3-11})
and applying to $X^{\gamma _{X}}$ it follows that 
\begin{equation}
\hat{A}(X^{\gamma _{X}})=Td(X^{\gamma _{X}})\exp [-Tr(\frac{\frac{i}{2\pi }%
R_{X^{\gamma _{X}}}^{+}}{2})].  \label{3-12}
\end{equation}%
\noindent Multiplying (Wedging) (\ref{3-12}) by (\ref{3-13}) gives%
\begin{equation}
\hat{A}(X^{\gamma _{X}})\exp \left( \frac{i}{2\pi }F_{0}^{\mathcal{E}/%
\mathcal{S}}\right) =Td(X^{\gamma _{X}})\exp (\frac{\frac{i}{2\pi }%
R^{K_{N}^{\ast }}}{2})\exp \left( \frac{i}{2\pi }F_{0}^{E}\right) .
\label{3-14}
\end{equation}

\noindent Here we have used $TrR_{X^{\gamma _{X}}}^{+}$ $=$ $R^{K_{X^{\gamma
_{X}}}^{\ast }}.$ Applying $\gamma _{X}^{E}$ (resp. $\gamma _{X}^{\mathcal{E}%
/\mathcal{S}}$ $:=$ $j(\gamma _{X}^{E}))$ to the RHS (resp. LHS) of (\ref%
{3-14}) and then taking the supertrace, we obtain (\ref{A-Td}) by \cite[%
(3.10)]{BGV} and the definition of the relative supertrace in \cite[p.113]%
{BGV}.
\end{proof}

The authors Berline, Getzler and Vergne in \cite[p.152]{BGV} define $\hat{A}$%
-genus form and Todd genus form without the factor $\frac{1}{2\pi i}$ in (%
\ref{3-11}) and (\ref{3-10}): (we put the subscript \textquotedblleft BGV"
below)%
\begin{eqnarray}
\hat{A}_{BGV}(X)&{:=}&{\det }^{1/2}\left( \frac{R/2}{\sinh [R/2]}\right)
\label{8-69a} \\
&=&\det \left( \frac{R^{+}/2}{\sinh [R^{+}/2]}\right)  \notag \\
Td_{BGV}(X)&{:=}&\det \left( \frac{R^{+}}{e^{R^{+}}-1}\right) .  \notag
\end{eqnarray}

The following Corollary will be used in the proof of Proposition \ref{P-Fk},
for which we write the previous lemma with some signs opposite to the ones
given in (\ref{ch}) and (\ref{A-Td}).

\begin{corollary}
\label{C-B-3} With the notation above, we have, on $X^{\gamma _{X}}$%
\begin{eqnarray}
&&\hat{A}_{BGV}(X^{\gamma _{X}})\text{Str}_{\mathcal{E}/\mathcal{S}}\text{ }%
(\gamma _{X}^{\mathcal{E}/\mathcal{S}}\exp \left( -F_{0}^{\mathcal{E}/%
\mathcal{S}}\right) )  \label{B-3a} \\
&=&Td_{BGV}(X^{\gamma _{X}})\text{Str}_{E}\text{ }(\gamma _{X}^{E}\exp
\left( -F_{0}^{E}\right) )\exp (\frac{-R^{K_{N}^{\ast }}}{2})\text{.}  \notag
\end{eqnarray}
\end{corollary}

We are now ready to return to our situation. Recall that $M:=\Sigma /\sigma $
is equipped with the Hermitian metric $g_{M}.$ Henceforth we assume that $%
g_{M}$ is K\"{a}hler. Consider the above $X$ as $V_{j}$ where $V_{j}$ is
equipped with $\pi ^{\ast }g_{M}|_{V_{j}}$ which we warn is not $%
G_{a,m}|_{V_{j}};$ see the line above (\ref{Str-0}) and Remark \ref{R-z-vol}
for the warning. Identify the above $\gamma _{X}$ with $\gamma _{k}$ (see (%
\ref{8-16c})) which is an isometry with respect to $\pi ^{\ast
}g_{M}|_{V_{j}}$ by Corollary \ref{C-alk}. Let the above $\mathcal{E}$ be $%
\mathcal{E}^{m}$ over $V_{j}$ with the metric given in Notation \ref{n-8-1},
on which we see that $\gamma _{k}$ acts as an isometry denoted by $\gamma
_{k}^{\mathcal{E}^{m}}$ (see also (\ref{rkE}) below$.$ Set $\dim _{\mathbb{C}%
}V_{j}=n-1$, $\dim _{\mathbb{R}}V_{j}^{\gamma _{k}}$ $=:$ $l_{j,k}$. Let $%
n_{j,k}$ $:=$ $n-1-l_{j,k}/2,$ half the real dimension of the real normal
bundle $\mathcal{N}_{j,k}$ to $V_{j}^{\gamma _{k}}$ (resp. $\Sigma _{j,k})$
in $V_{j}$ (resp. $\Sigma ).$ Let $(\gamma _{k})_{1}^{c}$ denote the
complex-linear transformation of $\mathcal{N}_{j,k}$ induced by $\gamma
_{k}; $ see Notation \ref{N-b-1} below$.$ Let $vol_{\mathcal{N}_{j,k}}$
denote the standard section of $\Lambda ^{2n_{j,k}}\mathcal{N}_{j,k}^{\ast }$
of unit length over $\Sigma _{j,k}$ (\cite[(12.8)]{Du}). For notational
simplicity we frequently identify $V_{j}\times \{0\}\times \{1\}$ with $%
V_{j}.$ We have (cf. Notation \ref{n-8-1}, (\ref{8-16c}))%
\begin{equation}
\gamma _{k}^{\mathcal{E}^{m}}=\gamma _{k}^{\psi _{j}^{\ast }(\pi ^{\ast }%
\mathcal{E}_{M})|_{V_{j}}}\otimes \gamma _{k}^{\psi _{j}^{\ast }(E\otimes
(L_{\Sigma }^{\ast })^{\otimes m})|_{V_{j}}}  \label{rkE}
\end{equation}%
\noindent where $\gamma _{k}^{\mathcal{E}^{m}}$ (resp. $\gamma _{k}^{\psi
_{j}^{\ast }(\pi ^{\ast }\mathcal{E}_{M})|_{V_{j}}},$ $\gamma _{k}^{\psi
_{j}^{\ast }(E\otimes (L_{\Sigma }^{\ast })^{\otimes m})|_{V_{j}}}$) is the
complex endomorphism induced by the action of $\gamma _{k}$ on $\mathcal{E}%
^{m}$ (resp. $\psi _{j}^{\ast }(\pi ^{\ast }\mathcal{E}_{M})|_{V_{j}}$,\ $%
\psi _{j}^{\ast }(E\otimes (L_{\Sigma }^{\ast })^{\otimes m})|_{V_{j}}).$
Note that $\gamma _{k}^{\mathcal{E}^{m}}$ is induced by $(\sigma _{\alpha
_{k}^{-1}}^{\ast })_{\alpha _{k}\circ x}$ (\ref{rE}), which equals $(\sigma
_{\alpha _{k}^{-1}}^{\ast })_{x}$ if $x$ $\in $ $V_{j}^{\gamma _{k}}$.

\begin{notation}
\label{N-b-1} The superscript (resp. subscript) \textquotedblleft $c$" is
not used by \cite{BGV}, but here we use it to indicate a complex
endomorphism. Suppose that $(V,J)$ is a real vector space of even dimension
with an almost complex structure $J.$ Let $\varphi $ $\in $ $End(V)$ be a $J$%
-preserving endomorphism on $V.$ Then we use $\varphi ^{c}$ or $\varphi _{c}$
to denote the corresponding complex endomorphism on a complex space of
complex dimension $\frac{1}{2}\dim _{\mathbb{R}}V$.
\end{notation}

The symbol of $\gamma _{k}^{\mathcal{E}^{m}}$ in (\ref{rkE}) which will be
used for (\ref{B-Fjk}) below is understood to be%
\begin{equation}
\sigma _{2n_{j,k}}(\gamma _{k}^{\mathcal{E}^{m}})=\sigma _{2n_{j,k}}(\gamma
_{k}^{\psi _{j}^{\ast }(\pi ^{\ast }\mathcal{E}_{M})|_{V_{j}}})\gamma
_{k}^{\psi _{j}^{\ast }(E\otimes (L_{\Sigma }^{\ast })^{\otimes
m})|_{V_{j}}},  \label{S-1}
\end{equation}

\noindent cf. \cite[top two lines on p.193]{BGV}. Thus we need to compute
the first term in the RHS of (\ref{S-1}). See Lemma \ref{L-sN} below.

To proceed, note that $(\gamma _{k})_{1}^{c}$ on $\mathcal{N}_{j,k}$ can be
diagonalized with eigenvalues $e^{i2\theta _{l}},$ where $\theta _{l}$ is
the unique angle such that $0$ $<$ $\theta _{l}$ $<$ $\pi $ for $1$ $\leq $ $%
l$ $\leq $ $n_{j,k}$ \cite[(12.5) on p.150]{Du} since no eigenvalue here can
be 1. We define the square root of the determinant of $(\gamma _{k})_{1}^{c}$
as follows \cite[the middle of p.154]{Du}:%
\begin{equation}
(\det (\gamma
_{k})_{1}^{c})^{1/2}:=\dprod\nolimits_{l=1}^{n_{j,k}}e^{i\theta _{l}}.
\label{cn}
\end{equation}

\begin{lemma}
\label{L-sN} With the notation above, we compute the symbol seated in (\ref%
{S-1}):%
\begin{equation}
\sigma _{2n_{j,k}}(\gamma _{k}^{\psi _{j}^{\ast }(\pi ^{\ast }\mathcal{E}%
_{M})|_{V_{j}}})=2^{-n_{j,k}}{\det }^{1/2}(1-(\gamma _{k})_{1})(\det (\gamma
_{k})_{1}^{c})^{1/2}vol_{\mathcal{N}_{j,k}}  \label{S-2}
\end{equation}%
over $V_{j}^{\gamma _{k}}$. See the last paragraph of the proof for $vol_{%
\mathcal{N}_{j,k}}.$
\end{lemma}

\begin{remark}
\label{R-BGV} The authors of \cite{BGV} consider general Clifford modules $%
\mathcal{E}$ rather than the specific $\mathcal{E}_{M}.$ As such, they did
not go further with the explicit computation as done here. Moreover our $\pi
^{\ast }\mathcal{E}_{M}$ is not a Clifford module from the point of view of $%
\Sigma $ (although $\mathcal{E}_{M}$ is so from that of $M$ $=$ $\Sigma
/\sigma );$ one may think of $\pi ^{\ast }\mathcal{E}_{M}$ as a Clifford
module \textit{of transversal type}.
\end{remark}

\begin{proof}
(of Lemma \ref{L-sN}) Write $\mathcal{N}_{j,k}$ as $\mathcal{N}$ and
decompose $\mathcal{N}^{\ast }\mathcal{\otimes \mathbb{C}}$ $=$ $\mathcal{N}%
^{\ast 0,1}\oplus \mathcal{N}^{\ast 1,0}.$ Note that $(\gamma
_{k}^{-1})^{\ast }$ $\in $ $U(\mathcal{N}^{\ast 0,1})$ (meaning the action
on $\mathcal{N}^{\ast 0,1}$ induced by $\gamma _{k}$ on $\mathcal{N)},$ a
unitary transformation on $\mathcal{N}^{\ast 0,1},$ extends naturally to an
endomorphism $\gamma _{k}^{\Lambda \mathcal{N}^{\ast 0,1}}$of the exterior
algebra $\Lambda \mathcal{N}^{\ast 0,1}.$ By \cite[p.192, Lemma 6.10]{BGV}
(with our Lemma \ref{L-End} without $E$), along $V_{j}^{\gamma _{k}}$ we
identify $\gamma _{k}^{\psi _{j}^{\ast }\pi ^{\ast }\mathcal{E}%
_{M}|_{V_{j}}}|_{\mathcal{N}^{\ast 0,1}}$ ($=$ the preceding $(\gamma
_{k}^{-1})^{\ast }$) with a section of $C(\mathcal{N}^{\ast })\otimes 
\mathbb{C}$ corresponding to $\gamma _{k}^{\Lambda \mathcal{N}^{\ast 0,1}}$ $%
\in $ $End(\Lambda \mathcal{N}^{\ast 0,1})$ via the isomorphism $c:C(%
\mathcal{N}^{\ast })\otimes \mathbb{C}$ $\rightarrow $ $End(\Lambda \mathcal{%
N}^{\ast 0,1})$ (\cite[Proposition 3.19]{BGV} or \cite[p.37]{Du}), i.e. $%
\gamma _{k}^{\psi _{j}^{\ast }\pi ^{\ast }\mathcal{E}_{M}|_{V_{j}}}$ $=$ $%
c^{-1}(\gamma _{k}^{\Lambda \mathcal{N}^{\ast 0,1}})$ on $V_{j}^{\gamma
_{k}} $. Let $e_{l},$ $Je_{l},$ $l=1,$ $\cdot \cdot ,$ $n_{j,k}$ be an
orthonormal basis of $\mathcal{N}^{\ast }.$ We claim%
\begin{equation}
\gamma _{k}^{\psi _{j}^{\ast }\pi ^{\ast }\mathcal{E}_{M}|_{V_{j}}}=\exp
_{C}\sum_{l=1}^{n_{j,k}}\theta _{l}(e_{l}\cdot Je_{l}+i)\in C(\mathcal{N}%
^{\ast })\otimes \mathbb{C}  \label{rC}
\end{equation}

\noindent where $\exp _{C}$ means the exponential \cite[(4.3)]{Du}. By \cite[%
p.37]{Du} we see that 
\begin{equation}
c(e_{l}\cdot Je_{l}+i)=2ie_{e_{l}-iJe_{l}}\circ \iota _{e_{l}},  \label{2e}
\end{equation}

\noindent on the RHS of which $e_{\bullet }$ (resp. $\iota _{\bullet })$
means the operators taking exterior $($resp. interior$)$ product with $%
\bullet $. Note that $e_{l}-iJe_{l}$ $\in $ $\mathcal{N}^{\ast 0,1}.$ By
acting on $e_{k}-iJe_{k}$ it follows from (\ref{2e}) that%
\begin{equation}
c(e_{l}\cdot Je_{l}+i)(e_{k}-iJe_{k})=2i\delta _{lk}(e_{k}-iJe_{k})\text{
and hence}  \label{8-89-1}
\end{equation}%
\begin{equation}
c{\big (}\exp _{C}\sum_{a=1}^{n_{j,k}}\theta _{a}(e_{a}\cdot Je_{a}+i){\big )%
}(e_{l}-iJe_{l})=e^{2\theta _{l}i}(e_{l}-iJe_{l}).  \label{8-77a}
\end{equation}

\noindent On the other hand, we can easily show that $(\gamma
_{k}^{-1})^{\ast }$ on $\mathcal{N}^{\ast 0,1}$ is also diagonalized with
eigenvalues $e^{i2\theta _{l}}$ (when $(\gamma _{k})_{1}^{c}$ on $\mathcal{N}
$ is diagonalized with eigenvalues $e^{i2\theta _{l}}$ by our assumption
lying above Lemma \ref{L-sN}). So, by (\ref{8-77a}) the RHS of (\ref{rC}) is 
$(\gamma _{k}^{-1})^{\ast }$. To show that it is $\gamma _{k}^{\Lambda 
\mathcal{N}^{\ast 0,1}}$ one computes as in (\ref{8-89-1}), (\ref{8-77a})
the action on two-forms, three-forms, etc. This can be done similarly using
the action (\ref{2e}). The final result gives (\ref{rC}); we omit the
details (see \cite{CT2}).

Now by using \cite[(4.3) and Subsection 4.3]{Du} we obtain%
\begin{equation}
\text{The RHS of (\ref{rC}) = }\dprod\limits_{l=1}^{n_{j,k}}(\cos \theta
_{l}+\sin \theta _{l}e_{l}\cdot Je_{l})e^{i\theta _{l}}.  \label{rC-1}
\end{equation}

\noindent Combining (\ref{rC}) and (\ref{rC-1}) we can now compute the
symbol of $\gamma _{k}^{\psi _{j}^{\ast }(\pi ^{\ast }\mathcal{E}%
_{M})|_{V_{j}}}:$ (After converting $e_{l}\cdot Je_{l}$ in (\ref{rC-1}) into 
$e_{l}\wedge Je_{l}$ and keeping the top $2n_{j,k}$-form,)%
\begin{equation*}
\sigma _{2n_{j,k}}(\gamma _{k}^{\psi _{j}^{\ast }(\pi ^{\ast }\mathcal{E}%
_{M})|_{V_{j}}})=\dprod\limits_{l=1}^{n_{j,k}}e^{i\theta _{l}}\sin \theta
_{l}(\wedge _{l=1}^{n_{j,k}}(e_{l}\wedge Je_{l}))\text{.}
\end{equation*}

\noindent From (\ref{cn}), \cite[the third formula on p.154]{Du}\footnote{%
To avoid possible confusion, we record this formula as follows (in the
notation of \cite{Du}): 
\begin{equation*}
\det (1-\gamma _{N})^{-1/2}=(2^{l}\dprod\limits_{j=n-l+1}^{n}\sin \theta
_{j})^{-1}.
\end{equation*}%
} and $\wedge _{l=1}^{n_{j,k}}(e_{l}\wedge Je_{l})=:vol_{\mathcal{N}_{j,k}}$
it is straightforward to see that%
\begin{equation*}
\sigma _{2n_{j,k}}(\gamma _{k}^{\psi _{j}^{\ast }(\pi ^{\ast }\mathcal{E}%
_{M})|_{V_{j}}})=(\det (\gamma _{k})_{1}^{c})^{1/2}2^{-n_{j,k}}{\det }%
^{1/2}(1-(\gamma _{k})_{1})vol_{\mathcal{N}_{j,k}}.
\end{equation*}

\noindent We have proved (\ref{S-2}).
\end{proof}

Rewrite $\mathcal{F}_{k,m}^{j}$ of (\ref{Str-1a}) as $\mathcal{F}_{k,m}^{j}$ 
$=$ $T_{V_{j}}F_{j,k}$ where in view of (\ref{8-16c}) for $\tilde{g},$%
\begin{equation}
F_{j,k}:=c_{j,k}I(\tilde{g})Str_{\mathcal{E}^{m}/S}[\sigma
_{2n-(l_{j,k}+2)}(\gamma _{k}^{\mathcal{E}^{m}})\exp (-F_{0}^{_{\mathcal{E}%
^{m}/S}})].  \label{B-Fjk}
\end{equation}

\noindent We are going to express $F_{j,k}$ in terms of the Todd and Chern
character forms, where we recall (cf. (\ref{Ig}))%
\begin{equation}
I(\tilde{g})=\frac{\hat{A}_{BGV}(T\Sigma ^{\tilde{g}}/L_{\Sigma ^{\tilde{g}%
}},g_{quot})}{\det^{1/2}(1-\tilde{g}_{1})\det^{1/2}(1-\tilde{g}_{1}\exp
(-R^{1}))}.  \label{Ig-a}
\end{equation}

Using Lemma \ref{L-sN} we now come to the main result of this section and
prove (\ref{FkTd}) above. Recall the notation $T_{V_{j}}$ in (\ref{Str-1a})
but for the sake of clarity we will write $T_{V_{j}}|_{V_{j}^{\gamma _{k}}}$
instead of $T_{V_{j}}$ to denote the Berezin integral of the bundle $%
TV_{j}|_{V_{j}^{\gamma _{k}}}$ (see also the line after (\ref{Bd})) and to
get a function on $V_{j}^{\gamma _{k}}.$ This notation $T_{V_{j}}|_{V_{j}^{%
\gamma _{k}}}$ is not to be confused with $T_{V_{j}^{\gamma _{k}}}$ below.

\begin{proposition}
\label{P-Fk} With the notation above and $g_{M}$ being K\"{a}hler, we have%
\begin{eqnarray}
&&\text{ \ \ \ }\mathcal{F}_{k,m}^{j}\equiv T_{V_{j}}|_{V_{j}^{\gamma
_{k}}}F_{j,k}=  \label{TFk} \\
&&T_{V_{j}^{\gamma _{k}}}\frac{Td(T^{1,0}\Sigma ^{\tilde{g}}/L_{\Sigma ^{%
\tilde{g}}},g_{quot})ch(\gamma _{k}^{\psi _{j}^{\ast }E|_{V_{j}}},\psi
_{j}^{\ast }E|_{V_{j}}\mathcal{)}ch(\psi _{j}^{\ast }(L_{\Sigma }^{\ast
})^{\otimes m}|_{V_{j}}\mathcal{)}}{\det (1-(\tilde{g}^{-1})_{1}^{c}\exp (-%
\frac{i}{2\pi }R_{c}^{1}))}  \notag
\end{eqnarray}%
where $\tilde{g}_{1}^{c}$ (resp. $R_{c}^{1})$ denotes the complex-linear
transformation (resp. complex endomorphism valued curvature form)
corresponding to the real transformation $\tilde{g}_{1}$ on $\mathcal{N}%
|_{V_{j}^{\tilde{g}}}$ (resp. real endomorphism valued curvature form $%
R^{1}).$ Compare the remark below.
\end{proposition}

\begin{remark}
\label{R-8-32} For the above $R_{c}^{1}$ note that $R^{1}$ is complex linear
($J$-linear) with respect to the complex structure $J$ by our K\"{a}hler
assumption on $V_{j}$ in this subsection. Similarly $\tilde{g}_{1}^{c}$ can
be viewed as an $n_{j,k}\times n_{j,k}$ complex matrix, $n_{j,k}$ $=$ $\frac{%
1}{2}\dim _{\mathbb{R}}\mathcal{N}.$
\end{remark}

\begin{proof}
(of Proposition \ref{P-Fk}) By (\ref{S-1}), (\ref{S-2}) of Lemma \ref{L-sN}
and noting that $\tilde{g}_{1}$ $=$ $(\gamma _{k})_{1}$ along $V_{j}^{\gamma
_{k}}$ since $\sigma (\tilde{g})$ $=$ $\gamma _{k}$ on $V_{j}\times
\{0\}\times \{1\}$ by (\ref{8-16c}) or (\ref{Str-5c})$,$ we obtain%
\begin{equation}
\frac{\sigma _{2n-(l_{j,k}+2)}(\gamma _{k}^{\mathcal{E}^{m}})}{\det^{1/2}(1-%
\tilde{g}_{1})}=2^{-n_{j,k}}(\det (\gamma _{k})_{1}^{c})^{1/2}vol_{\mathcal{N%
}_{j,k}}\gamma _{k}^{\psi _{j}^{\ast }(E\otimes (L_{\Sigma }^{\ast
})^{\otimes m})|_{V_{j}}}.  \label{Ba}
\end{equation}

\noindent Substituting (\ref{Ba}) into (\ref{B-Fjk}) with $I(\tilde{g})$ in (%
\ref{Ig-a}) and making use of (\ref{B-3a}) with $X$ $=$ $V_{j}$, we get, by
setting (compare (\ref{ch}) and (\ref{8-69a}) for the subscripts
\textquotedblleft $BGV$" below; the difference between them involves
multiplicative factors $\frac{i}{2\pi }$)%
\begin{equation}
\text{Str}_{E^{\prime }}(\gamma _{k}^{E^{\prime }}\exp (-F_{0}^{E^{\prime
}}))=:ch_{BGV}(\gamma _{k}^{E^{\prime }},E^{\prime })  \label{chBGV}
\end{equation}%
\begin{eqnarray}
F_{j,k} &=&c_{j,k}\frac{2^{-n_{j,k}}(\det (\gamma _{k})_{1}^{c})^{1/2}vol_{%
\mathcal{N}_{j,k}}}{\det^{1/2}(1-\tilde{g}_{1}\exp (-R^{1}))}  \label{Bb} \\
&&\cdot Td_{BGV}(V_{j}^{\gamma _{k}})ch_{BGV}(\gamma _{k}^{E^{\prime
}},E^{\prime })\exp (\frac{-R^{K_{N}^{\ast }}}{2})  \notag
\end{eqnarray}%
\noindent with $E^{\prime }$ $=$ $\psi _{j}^{\ast }(E\otimes (L_{\Sigma
}^{\ast })^{\otimes m})|_{V_{j}}.$ Here $n_{j,k}=n-1-l_{j,k}/2$ as before.
The main computation of this proposition is the following: 
\begin{eqnarray}
&&{\det }^{-1/2}(1-\tilde{g}_{1}\exp (-R^{1}))i^{-n_{j,k}}(\det (\gamma
_{k})_{1}^{c})^{1/2}e^{-\frac{1}{2}R^{K_{\mathcal{N}}^{\ast }}}  \label{Bc}
\\
&=&(\det (1-(\tilde{g}^{-1})_{1}^{c}\exp (R_{c}^{1})))^{-1}.  \notag
\end{eqnarray}

Before we proceed, a warning is in order. With the square root $(\det
(\gamma _{k})_{1}^{c})^{1/2}$ chosen by (\ref{cn}), special care should be
taken when using the usual rules for further computation or else an unwanted
minus sign for (\ref{Bc}) may occur in the end\footnote{%
For the related computations, see \cite[p.156]{Du}. However, the treatment
there does not quite lead to the conclusion here because the sign issue as
warned above was not dealt with in sufficient details and the opposite sign
seems to occur there. It is desirable to carry out the computation of our
own to ensure the ultimately correct sign.}. To show (\ref{Bc}) we may
assume that both $(\tilde{g}^{-1})_{1}^{c}$ and $R_{c}^{1}$ are
simultaniously diagonalized with respective eigenvalues $e^{-2i\theta _{l}}$
(this being consistent with the lines above (\ref{cn}) as $\tilde{g}$ $=$ $%
g_{k},$ $\gamma _{k}$ here) and $R_{c,l}^{1}$ (which are two-forms) for $1$ $%
\leq $ $l$ $\leq $ $n_{j,k}$ (cf. \cite[second paragraph on p.167]{Du})$.$
Note that $R^{K_{\mathcal{N}}^{\ast }}$ $=$ $TrR_{c}^{1}$ $=$ $%
\sum_{l=1}^{n_{j,k}}R_{c,l}^{1}.$ Observe that the $l$th real $2\times 2$
matrix block of $\exp (-R^{1})$ corresponding to $\exp (-R_{c,l}^{1})$ reads
as (noting that $R_{c,l}^{1}$ is purely imaginary)%
\begin{equation}
\left( 
\begin{array}{cc}
\cos iR_{c,l}^{1} & \sin iR_{c,l}^{1} \\ 
-\sin iR_{c,l}^{1} & \cos iR_{c,l}^{1}%
\end{array}%
\right) .  \label{cs}
\end{equation}%
\noindent Using \cite[p.154]{Du} or Footnote$^{11}$ in the proof of Lemma %
\ref{L-sN}, and (\ref{cs}) it is not difficult to show that%
\begin{equation}
{\det }^{-1/2}(1-\tilde{g}_{1}\exp (-R^{1}))=\dprod\limits_{l=1}^{n_{j,k}}%
\frac{1}{2\sin (\theta _{l}+\frac{1}{2}iR_{c,l}^{1})},  \label{cs-a}
\end{equation}%
\noindent and that by $(\ref{cn})$%
\begin{equation}
(\det (\gamma _{k})_{1}^{c})^{1/2}e^{-\frac{1}{2}R^{K_{\mathcal{N}}^{\ast }}}%
\overset{}{=}\dprod\limits_{l=1}^{n_{j,k}}e^{i(\theta _{l}+\frac{1}{2}%
iR_{c,l}^{1})}.  \label{cs-b}
\end{equation}%
\noindent Thus by (\ref{cs-a}) and (\ref{cs-b}) to show (\ref{Bc}) is
reduced to verifying 
\begin{equation}
\frac{1}{2\sin (\theta _{l}+\frac{1}{2}iR_{c,l}^{1})}i^{-1}e^{i(\theta _{l}+%
\frac{1}{2}iR_{c,l}^{1})}=\frac{1}{1-e^{-2i(\theta _{l}+\frac{1}{2}%
iR_{c,l}^{1})}}  \label{Bc-a}
\end{equation}%
\noindent for each $l$. Now (\ref{Bc-a}) holds true by a direct computation,
proving (\ref{Bc}).

Substituting (\ref{Bc}) into (\ref{Bb}) and noting that $c_{j,k}$ $=$ $(2\pi
i)^{-l_{j,k}/2}2^{n_{j,k}}(-i)^{n_{j,k}}$, $(-i)^{n_{j,k}}=i^{-n_{j,k}},$ we
reduce (\ref{Bb}) to%
\begin{eqnarray}
&F_{j,k}=(2\pi i)^{-l_{j,k}/2}&  \label{Bd} \\
&\cdot \frac{Td_{BGV}(V_{j}^{\gamma _{k}})ch_{BGV}(\gamma _{k}^{\psi
_{j}^{\ast }E|_{V_{j}}}\otimes \gamma _{k}^{\psi _{j}^{\ast }(L_{\Sigma
}^{\ast })^{\otimes m}|_{V_{j}}},\psi _{j}^{\ast }(E\otimes (L_{\Sigma
}^{\ast })^{\otimes m})|_{V_{j}})vol_{\mathcal{N}_{j,k}}}{\det (1-(\tilde{g}%
^{-1})_{1}^{c}\exp (R_{c}^{1}))}&  \notag
\end{eqnarray}%
\noindent by (\ref{rkE}). We now see from (\ref{Bd}) that taking $%
T_{V_{j}}|_{V_{j}^{\gamma _{k}}}$ of $F_{j,k}$ is the same as taking $%
T_{V_{j}^{\gamma _{k}}}$ of $F_{j,k}/vol_{\mathcal{N}_{j,k}}$. The factor $%
(2\pi i)^{-l_{j,k}/2}$ ($l_{j,k}$ $=$ $\dim _{\mathbb{R}}X^{\gamma _{X}}$)
is absorbed when one changes the curvature form by a multiple $\frac{1}{2\pi
i}$ so that the above subscripts \textquotedblleft $BGV$" drop out. Also
note that $g_{quot}$ is identified with $\pi ^{\ast }g_{M}$ restricted to $%
V_{j}^{\gamma _{k}}.$ Altogether, with the \textit{multiplicative property}
of the $\gamma $-twisted Chern character forms (this property can be checked
by the trace formula for tensor product \cite[(11.2) on p.133]{Du} and by
using simultaneous diagonalization \cite[p.167]{Du}), we conclude (\ref{TFk}%
). Here note that $ch(\gamma _{k}^{\psi _{j}^{\ast }(L_{\Sigma }^{\ast
})^{\otimes m}|_{V_{j}}},\psi _{j}^{\ast }(L_{\Sigma }^{\ast })^{\otimes
m}|_{V_{j}})$ $=$ $ch(\psi _{j}^{\ast }(L_{\Sigma }^{\ast })^{\otimes
m}|_{V_{j}})$ on $V_{j}^{\gamma _{k}}$ because it is not difficult to see
via definitions that $\gamma _{k}^{\psi _{j}^{\ast }(L_{\Sigma }^{\ast
})^{\otimes m}|_{V_{j}}}$ on $V_{j}^{\gamma _{k}}$ is the identity action.
\end{proof}

\begin{remark}
\label{R-8-B} Duistermaat's formula \cite[(11.17) on p.144]{Du} in the
statement of his Theorem 11.1 seems not to be consistent with the usual form
of the index theorem due to his possibly wrong sign for $R^{L}$ involved in
the Chern character term $ch(L)$ (after the replacement $R^{L}$ $\rightarrow 
$ $R^{L}/2\pi i,$ cf. \cite[p.145]{Du}). Similar confusion occurs also for
the sign of $\frac{1}{2}R^{K^{\ast }}$ in the same formula because the term $%
-\frac{1}{2}R^{K^{\ast }}$ is needed for 
\begin{equation*}
\left( \det \frac{1-e^{-R}}{R}\right) ^{-1/2}\cdot e^{-\frac{1}{2}R^{K^{\ast
}}}=\det{}_{\mathbb{C}}\frac{R}{e^{R}-1}\left( =\det \frac{R_{c}}{e^{R_{c}}-1%
}\text{ in our notation above}\right)
\end{equation*}%
to give the usual Todd class term after the replacement $R$ $\rightarrow $ $%
\frac{R}{2\pi i}$ (at least for the K\"{a}hler case).
\end{remark}

\subsection{Comparison with Duistermaat's orbifold version of the index
theorem, Part II: integrals over fixed point orbifolds\label{S-8-5}}

In comparison with \cite[Theorem 14.1 on p.184]{Du} we take $M=\Sigma
/\sigma ,$ a compact complex orbifold by Theorem \ref{thm2-1}. For
simplicity we assume no extra $\mathbb{C}^{\ast }$-equivariant holomorphic
vector bundle $E$ over $\Sigma $ i.e. no extra complex orbifold vector
bundle $L$ over $M$ in \cite[Theorem 14.1 on p.184]{Du}. To see what $\tilde{%
F}$ in the notation of \cite[pp. 184, 180]{Du} is, we take $\gamma _{V}$ = $%
Id$ in \cite[p.180]{Du}. Recalling the chart $W_{j}$ $=$ $V_{j}\times
(-\varepsilon _{j},\varepsilon _{j})\times \mathbb{R}^{+}$ of $\Sigma $ (see
the paragraph preceding (\ref{g1})), the local orbifold structure group $%
H=G_{j},$ a finite cyclic subgroup of $S^{1}$ $\subset $ $\mathbb{C}^{\ast }$
acts on the orbifold chart $V_{j}$ $=:$ $V$ through (\ref{Str-5}) for which
recall that $\pi _{V_{j}}$ plays no role as shown in Proposition \ref{L-alk} 
$iii),$ i.e. $\gamma _{k}^{-1}$ $=$ $\sigma (g_{k}^{-1})$ on $V_{j}\times
\{0\}\times \{1\}$ for any $g_{k}$ $\in $ $H=G_{j}.$ Let $g_{0}$ $\in $ $H$
be a generator of order $N+1$ $=$ $N_{j}+1$ thus $g_{0}^{N+1}=1.$

Take an element $\tilde{g}$ $\in $ $H,$ $\tilde{g}$ $\neq $ $1.$ Consider $%
V^{\tilde{g}}$ $\times $ $\{\tilde{g}\}$ $\subset $ $\hat{V}$ in \cite[p.180]%
{Du} with the action $(v,\tilde{g})$ $\rightarrow $ $(\sigma (a)v,\tilde{g})$
for $a\in H$ (note that the action of $H$ preserves $V)$ arising from the
original action in \cite[p.180]{Du} (since $H$ is abelian and $\gamma _{V}$
= $Id$ in our case)$.$ Piece together these (local) ($V^{\tilde{g}}$ $\times 
$ $\{\tilde{g}\})/H$ $\cong $ $V^{\tilde{g}}/H$ (here the subscript $j $
omitted already) to form a \textquotedblleft fixed point" orbifold $\tilde{F}
$ as in \cite[p.180]{Du} with the embedding $\tilde{F}$ $\subset $ $M$ in
our case$.$ We assume that $\tilde{F}$ is the only connected component
(otherwise just take a connected component of it). We need an explicit
description of the orbit type stratification for $(V,H)$ \cite[p.174]{Du}.
Let $H^{\prime }$ $\subset $ $H$ be a subgroup of $H.$ Let $V^{H^{\prime }} $
$\subset $ $V$ denote the set of points fixed by $H^{\prime }.$ For a point $%
v\in V$ let $H_{v}$ denote the isotropy subgroup at $v.$

\begin{lemma}
\label{L-ots} With the notation above (and $V_{j}$ possibly shrinked), there
are unique subgroups $H_{l}$ of $H,$ $l$ $=$ $0,$ $1,$ $\cdot \cdot \cdot ,$ 
$K,$ with the properties $i)$ $H_{0}:=\{1\}$ $\subset $ $H_{1}$ $\subset $ $%
\cdot \cdot \cdot $ $\subset $ $H_{K-1}$ $\subset $ $H_{K}:=H;$ $ii)$ for $%
v\in V^{H_{l}}\backslash V^{H_{l+1}}$ $H_{v}$ $=$ $H_{l}$ ($V^{H_{K+1}}$ $:=$
$\emptyset $)$.$
\end{lemma}

\begin{proof}
By the orbit type stratification \cite[pp. 174-175]{Du} two points $x,$ $y$ $%
\in $ $V$ belong to the same orbit type if $H_{x}$ $=$ $H_{y}$ since $H$ is
Abelian. Let $S_{l}$ be the set of points having the same isotropy subgroup
which we denote by $H_{l}.$ Then the stratification structure in \cite{Du}
gives, possibly after shrinking $V,$ the \textit{orbit type stratification}
(OTS for short$)$%
\begin{equation}
V=\bar{S}_{0}\supset \bar{S}_{1}\supset \cdot \cdot \cdot \supset \bar{S}%
_{K}=\{x\in V:H_{x}=H\}  \label{V}
\end{equation}

\noindent for some sequence $\bar{S}_{l}$ $=$ $S_{l}\cup S_{l+1}\cup \cdot
\cdot \cdot \cup S_{K},$ $l=0,1,\cdot \cdot \cdot ,K.$ $i)$ follows. For $%
ii) $ observe that 
\begin{equation}
V^{H_{l}}=\bar{S}_{l},  \label{VS}
\end{equation}%
\noindent so $V^{H_{l}}\backslash V^{H_{l+1}}=S_{l}.$ From this and the
definition of $S_{l},$ the statement $ii)$ follows.
\end{proof}

Let $H_{i},$ $H_{0}$ $\subset $ $H_{1}$ $\subset $ $H_{2}$ $\subset $ $\cdot
\cdot \cdot $ $\subset $ $H,$ be the sequence of (finitely many) isotropy
subgroups of $H$ as in the OTS given in the proof of Lemma \ref{L-ots}$,$
and $V^{H_{i}}$ $\subset $ $V$ the set of points fixed by $H_{i}$ so $%
V^{H_{1}}$ $\supset $ $V^{H_{2}}$ $\supset $ $\cdot \cdot \cdot $ $\supset $ 
$V^{H}.$ Let $<\tilde{g}>$ $\subset $ $H$ denote the subgroup generated by $%
\tilde{g}$, and $H_{V^{\tilde{g}}}$ $\subset $ $H$ the subgroup of $H$
consisting of elements which act on $V^{\tilde{g}}$ ($=$ $V^{<\tilde{g}>}$)
as the identity.

\begin{lemma}
\label{L-Vg} With the notation above, there is a subscript $k$ ($0$ $\leq $ $%
k$ $\leq $ $K-1$) such that $i)$ 
\begin{equation}
V^{\tilde{g}}=V^{H_{k+1}};  \label{VgH}
\end{equation}%
$ii)$ for $v\in V^{\tilde{g}}\backslash V^{H_{k+2}}$ the isotropy subgroup $%
H_{v}=H_{k+1}$ $\supset $ $<\tilde{g}>;$ $iii)$ for $v\in V^{\tilde{g}%
}\backslash V^{H_{k+2}}$ it holds that $H_{V^{\tilde{g}}}=H_{v}.$
\end{lemma}

\begin{proof}
For some $k$, $<\tilde{g}>$ $\subset $ $H_{k+1}$ and $<\tilde{g}>$ $%
\nsubseteq $ $H_{k}.$ We now claim $V^{\tilde{g}}=V^{H_{k+1}}.$ That $V^{%
\tilde{g}}\supset V^{H_{k+1}}$ is obvious. If $v\in $ $V^{\tilde{g}%
}\backslash V^{H_{k+1}}$ then by (\ref{V}) and (\ref{VS}) $v\in
V^{H_{i^{\prime }}}\backslash V^{H_{i^{\prime }+1}}$ for some $i^{\prime
}\leq k,$ and one has the isotropy subgroup $H_{v}$ $=$ $H_{i^{\prime }}$ $%
\subset $ $H_{k}$ by Lemma \ref{L-ots} $ii).$ Clearly $<\tilde{g}>$ $\subset 
$ $H_{v}$ hence $<\tilde{g}>$ $\subset $ $H_{k}$ contradicts the choice of $%
k.$ Thus $V^{\tilde{g}}\backslash V^{H_{k+1}}$ $=$ $\emptyset $ namely $V^{%
\tilde{g}}$ $\subset $ $V^{H_{k+1}}.$ We have shown (\ref{VgH}). $ii)$
follows from (\ref{VgH}) and Lemma \ref{L-ots} $ii)$. $iii)$ follows from $%
i) $ and $ii)$.
\end{proof}

For generic $v$ $\in $ $V^{\tilde{g}}$ we have $H_{v}=H_{k+1}$ by Lemma \ref%
{L-Vg} $ii)$, which is the largest subgroup that acts trivially on $V^{%
\tilde{g}}$ $=$ $V^{H_{k+1}}$ (by (\ref{VgH})) by Lemma \ref{L-Vg} $iii)$.
It follows that the multiplicity $m(\tilde{F})$ of $\tilde{F}$ \cite[p.175
with $S=\tilde{F}$]{Du} reads as ($j$ below denotes our chart index)%
\begin{equation}
m(\tilde{F})=\#H_{V^{\tilde{g}}}=\#H_{v}=\#H_{k+1}=:\frac{N_{j}+1}{h_{j}};
\label{Nk}
\end{equation}

\noindent $\#H_{k+1}\mid N_{j}+1$ so $h_{j}$ $=$ $h_{j}(H_{k+1})$ $\in $ $%
\mathbb{N}$ (originally dependent on $\tilde{g})$. Note that according to 
\cite[p.175]{Du} $m(\tilde{F})$ is independent of local data $(h_{j},N_{j})$%
\footnote{%
In our context this independence can be seen as follows. Denote $H_{v}$ by $%
H_{v}^{(j)}$ to indicate the dependence on the chart $W_{j}.$ We claim that
given $\tilde{g}$ one has 
\begin{equation}
H_{v}^{(j)}=H_{v}^{(l)}  \label{HjHl}
\end{equation}%
\noindent for $v$ $\in $ $W_{j}\cap W_{l}\cap V^{\tilde{g}}.$ By Proposition %
\ref{L-alk} $iii)$, $H_{v}^{(j)}$ $\subset $ $\tilde{H}_{v}$ where $\tilde{H}%
_{v}$ $:=$ $\{g$ $\in $ $S^{1}$ $\subset $ $\mathbb{C}^{\ast }|$ $\sigma
(g)v $ $=$ $v\}.$ On the other hand $\tilde{H}_{v}$ $\subset $ $G_{j}$ for $%
v $ $\in $ $W_{j}$ since $G_{j}$ is the local orbifold structure group,
giving that $\tilde{H}_{v}$ $\subset $ $H_{v}^{(j)}$ from the definition of $%
H_{v}^{(j)}$. Similarly we also have $H_{v}^{(l)}$ $=$ $\tilde{H}_{v}$,
giving (\ref{HjHl}). Now by (\ref{HjHl}) and (\ref{Nk}) once $\tilde{g}$ is
chosen, then%
\begin{equation}
(N_{j}+1)/h_{j}=(N_{l}+1)/h_{l}  \label{hjhl}
\end{equation}%
\noindent for different charts $W_{j}$ and $W_{l}.$ (\ref{hjhl}) will be
used in (\ref{Fg}).}. Let 
\begin{equation}
\tilde{F}_{j}:=\tilde{F}\cap \lbrack (\Sigma ^{\tilde{g}}\cap W_{j})/\sigma
]\subset M  \label{Fix}
\end{equation}%
\noindent which equals $\Sigma _{j,l}/\sigma $ $=$ $V_{j}^{\tilde{g}}/H$,
cf. (\ref{sing}) for $\Sigma _{j,l}$ with $\tilde{g}$ $=$ $g_{0}^{l}$ and $H$
stands for the local orbifold structure group $G_{j}$. So we can express the
integral in \cite[(14.3) on p.184]{Du} via (\ref{Nk}) and (\ref{VgH}) as
follows (cf. \cite[(14.1) on p.175]{Du}) provided that the integrand $(\cdot
)$ below has certain \textquotedblleft descent property":%
\begin{equation}
\int_{\tilde{F}_{j}}(\cdot )=\frac{1}{h_{j}}\int_{V_{j}^{\tilde{g}}}(\cdot )%
\text{ and equals }\frac{1}{h_{j}(H_{k+1})}\int_{V_{j}^{H_{k+1}}}(\cdot ).
\label{mF}
\end{equation}

Recall (cf. (\ref{Sigmak})) that $\Sigma ^{\tilde{g}}$ denotes the set of
points fixed by $\tilde{g}$ $\in $ $S^{1}$ $\subset $ $\mathbb{C}^{\ast }$
on $\Sigma .$ Let $\Sigma _{j}^{\tilde{g}}$ $:=$ $|\mathbb{C}^{\ast }\circ $
($V_{j}^{\tilde{g}}\times $ $\{0\}\times $ $\{1\})|,$ $|\cdot \cdot \cdot |$
meaning the support (without multiplicity), cf. Remark \ref{R-orbit}. Remark
that $\Sigma _{j}^{\tilde{g}}$ is not necessarily $\Sigma ^{\tilde{g}}\cap
W_{j}$ since the \textquotedblleft angle part" of $W_{j}$ is restricted to $%
(-\varepsilon _{j},$ $\varepsilon _{j}).$ The following technical lemma is
crucial:

\begin{lemma}
\label{L-orb} With the notation above and $\tilde{g}$ $\in $ $G_{j}$ ($%
\tilde{g}$ $=$ $1$ allowed), it holds that $V_{j}^{\tilde{g}}\times (0,$ $%
\frac{2\pi }{N_{j}+1})$ $\times $ $\mathbb{R}^{+}$ is diffeomorphic to $%
\Sigma _{j}^{\tilde{g}}$\TEXTsymbol{\backslash}$(V_{j}^{\tilde{g}}\times $ $%
\{0\}\times $ $\{1\})$ $\subset $ $\Sigma $.
\end{lemma}

\begin{proof}
Write $\tilde{g}$ $=$ $g_{0}^{l}$ for some $l$ between $0$ and $N_{j},$ $%
g_{0}$ $=$ $e^{2\pi i/(N_{j}+1)}$. Observe that $\sigma (g_{0})$ leaves $%
V_{j}\times $ $\{0\}\times $ $\{1\}$ set-invariant by Lemma \ref{L-8-6} and
hence leaves $V_{j}^{\tilde{g}}\times $ $\{0\}\times $ $\{1\}$
set-invariant. Define the map $\Psi :$ $[0,2\pi )\times $ $\mathbb{R}%
^{+}\times $ $(V_{j}^{\tilde{g}}\times \{0\}\times \{1\})$ $\rightarrow $ $|%
\mathbb{C}^{\ast }\circ $ $(V_{j}^{\tilde{g}}\times $ $\{0\}\times $ $%
\{1\})| $ by%
\begin{equation*}
\Psi (\phi ,r,\{p\}\times \{0\}\times \{1\})=(re^{i\phi })\circ (\{p\}\times
\{0\}\times \{1\})\text{ }\in \text{ }\Sigma .
\end{equation*}%
\noindent We claim that $\Psi $ maps $(0,$ $\frac{2\pi }{N_{j}+1})$ $\times $
$\mathbb{R}^{+}$ $\times $ ($V_{j}^{\tilde{g}}\times \{0\}\times \{1\}$)
into $|\mathbb{C}^{\ast }\circ (V_{j}^{\tilde{g}}\times \{0\}\times \{1\})|$%
\TEXTsymbol{\backslash}$(V_{j}^{\tilde{g}}\times \{0\}\times \{1\})$ and is
an embedding$.$ To show this, suppose that there are $\phi _{1},$ $\phi _{2}$
$\in $ $(0,$ $\frac{2\pi }{N_{j}+1})$ and two points $x_{1},$ $x_{2}$ $\in $ 
$V_{j}^{\tilde{g}}$ such that $e^{i\phi _{1}}x_{1}$ $=$ $e^{i\phi
_{2}}x_{2}. $ So $e^{i(\phi _{1}-\phi _{2})}x_{1}$ $=$ $x_{2}$ and hence $%
\pi (x_{1})$ $= $ $\pi (x_{2})$ $\in $ $M$ $=$ $\Sigma /\sigma $ where $\pi
: $ $\Sigma \rightarrow $ $M$ is the natural projection. That $x_{1}$ and $%
x_{2}$ represent the same orbifold point implies that $hx_{1}$ $=$ $x_{2}$
for some $h$ $\in $ $H$ $=$ $G_{j}$ (via Theorem \ref{thm2-1}). It follows
that $e^{i(\phi _{2}-\phi _{1})}hx_{1}$ $=$ $x_{1}$ so $e^{i(\phi _{2}-\phi
_{1})}h $ $\in $ $G_{j}.$ Hence we write $e^{i(\phi _{2}-\phi _{1})}$ $=$ $%
g_{0}^{l^{\prime }}$ $\in $ $G_{j}$. But this yields a contradiction since $%
0 $ $\leq $ $|\phi _{1}$ $-$ $\phi _{2}|$ $<$ $\frac{2\pi }{N_{j}+1}$ $=$ $%
\arg g_{0},$ unless $\phi _{1}$ $-$ $\phi _{2}$ $=$ $0$ which gives $x_{1}$ $%
=$ $x_{2}.$ This implies the embedding part of the claim. Using $\phi _{2}$ $%
=$ $0$ the similar argument yields the into part of the claim. Moreover $%
\Psi $ is indeed a diffeomorphism since one sees that\footnote{%
For any $z\in \mathbb{C}^{\ast }$ written as $z=re^{2\pi
mi/(N_{j}+1)}e^{i\delta }$ with $0\leq m\leq $ $N_{j}$ and $0\leq \delta <$ $%
\frac{2\pi }{N_{j}+1},$ via the observation mentioned earlier in the proof
one sees that $|z\circ $ $(V_{j}^{\tilde{g}}\times $ $\{0\}\times $ $\{1\})|$
$=$ $|re^{i\delta }\circ $ $(V_{j}^{\tilde{g}}\times $ $\{0\}\times $ $%
\{1\})|,$ which implies the claim.} 
\begin{equation*}
\Psi ([0,\frac{2\pi }{N_{j}+1})\times \mathbb{R}^{+}\times (V_{j}^{\tilde{g}%
}\times \{0\}\times \{1\}))=|\mathbb{C}^{\ast }\circ (V_{j}^{\tilde{g}%
}\times \{0\}\times \{1\})|.
\end{equation*}
\end{proof}

\begin{remark}
\label{R-8-37} The essence of the lemma implies (by choosing $\tilde{g}$ $=$ 
$1$) that for the chart $W_{j}$ $=$ $V_{j}\times $ $(-\varepsilon
_{j},\varepsilon _{j})\times $ $\mathbb{R}^{+}$ in Notation \ref{n-6.1}, a
possible choice of $\varepsilon _{j}$ can be $\frac{\pi }{N_{j}+1}.$
\end{remark}

From Lemma \ref{L-orb}, Lemma \ref{L-VFub} $ii)$ (\ref{vm0}) and $\int_{%
\mathbb{R}^{+}}dv_{0}(|w|)$ $=$ $1$ by (\ref{3-16.75}) for $m$ $=$ $0$ it is
not difficult to see that if the integrand $(\cdot )$ below depends only on $%
z,\bar{z}$ then%
\begin{equation}
\int_{|\mathbb{C}^{\ast }\circ (V_{j}^{\tilde{g}}\times \{0\}\times
\{1\})|}(\cdot )dv_{\Sigma ^{\tilde{g}},0}=\int_{V_{j}^{\tilde{g}}\times (0,%
\frac{2\pi }{N_{j}+1})\times \mathbb{R}^{+}}(\cdot )dv_{\Sigma ^{\tilde{g}%
},0}=\frac{1}{N_{j}+1}\int_{V_{j}^{\tilde{g}}}(\cdot )d\tilde{v}_{V_{j}^{%
\tilde{g}}}(z)  \label{sjk}
\end{equation}

\noindent in view of Fubini's theorem\footnote{%
The $w$-coordinates here may be changing all the time, i.e. the choice of $w$
is $p$-dependent for $p$ $\in $ $V_{j}$ when using Lemma \ref{L-VFub} $ii)$.
This dependence however yields no big problem in applying Fubini's theorem:
Imaging that one is integrating over a fibre bundle $E$ $\rightarrow $ $B,$
one can first do so on each fibres $E_{t},$ for which the choice of
coordinates on $E_{t}$ is immaterial. The situation here is similar (with
the support of $\mathbb{C}^{\ast }$-orbits playing the role of fibres $E_{t}$%
).}; here notice that $d\tilde{v}_{V_{j}^{\tilde{g}}}$ is the pullback of $%
dv_{\tilde{F}_{j}}$ by $V_{j}^{\tilde{g}}$ $\rightarrow $ $\tilde{F}_{j}$
under the restriction of $\pi $ $:$ $\Sigma $ $\rightarrow $ $\Sigma /\sigma 
$ to $V_{j}^{\tilde{g}}\times \{0\}\times \{1\}$. For the orbifold charts $%
V_{j}^{\tilde{g}}$ of the fixed point orbifold $\tilde{F}$ ($\subset $ $M$)$%
, $ we now take $\bar{\chi}_{j}$ a partition of unity of $\tilde{F}$
subordinated to $V_{j}^{\tilde{g}}/H$ ($=$ $\tilde{F}_{j})$ covering $\tilde{%
F}$ (for the existence of $\bar{\chi}_{j}$, see e.g. \cite[p.37]{Cara} and
references therein), and treat this as a \textquotedblleft partition of
unity" $\{\chi _{j}(z,\bar{z})\}_{j}$ adapted to $\{V_{j}^{\tilde{g}}\}_{j}$
although $\cup _{j}V_{j}^{\tilde{g}}$ $=:$ $\tilde{V}^{\tilde{g}}$ does not
necessarily admit a manifold structure (cf. $\cup _{j}\Sigma _{j}^{\tilde{g}%
} $ $=$ $\Sigma ^{\tilde{g}}(\supsetneqq $ $\tilde{V}^{\tilde{g}})$ is a
genuine submanifold of $\Sigma )$.

To start with the comparison with Duistermaat's formula in \cite{Du}, the
identification between our $H_{m}^{q}(\Sigma ,\mathcal{O}_{\Sigma })$ for $m$
$=$ $0$ and his $H^{q}(M,\mathcal{O}_{M})$ is immediate via (\ref{9.17-5})
and Remark \ref{9.3}. The case where an extra $\mathbb{C}^{\ast }$%
-equivariant holomorphic vector bundle $E$ $\rightarrow $ $\Sigma $ is
present yields no essential problem; the details are omitted. Now we are
going to devote ourselves to comparing the integral formulas given in the
RHS of the index theorems. Setting $m=0$ in (\ref{Str-9}) (for $m$ $\neq $ $%
0 $ see Remark \ref{R-8-41}) we compute: For the below $\Sigma ^{\tilde{g}}$
means the fixed point set of $\tilde{g},$ $\Sigma _{j}^{\tilde{g}}$ $=$ $|%
\mathbb{C}^{\ast }\circ (V_{j}^{\tilde{g}}\times \{0\}\times \{1\})|$ and $%
\mathcal{F}_{\tilde{g},0}$ in the LHS of (\ref{Str-9}) is rewritten as $%
\mathcal{F}_{\tilde{g}}$.%
\begin{eqnarray}
&&\int_{\Sigma ^{\tilde{g}}}\mathcal{F}_{\tilde{g}}(x)dv_{\Sigma ^{\tilde{g}%
},0}\overset{Cor.\ref{R+Inv}}{=}\sum_{j}\int_{\Sigma _{j}^{\tilde{g}}}\chi
_{j}(z,\bar{z})\mathcal{F}_{\tilde{g}}(z,\bar{z})dv_{\Sigma ^{\tilde{g}},0}
\label{Fg} \\
&&\overset{(\ref{sjk})}{=}\sum_{j}\frac{1}{N_{j}+1}\int_{V_{j}^{\tilde{g}%
}}\chi _{j}(z,\bar{z})\mathcal{F}_{\tilde{g}}(z,\bar{z})d\tilde{v}_{V_{j}^{%
\tilde{g}}}(z)\text{ (see also Cor. \ref{C-Fgm})}  \notag \\
&&\overset{(\ref{mF})}{=}\sum_{j}\frac{1}{N_{j}+1}h_{j}\int_{\tilde{F}_{j}}%
\bar{\chi}_{j}\mathcal{F}_{\tilde{g}}dv_{\tilde{F}_{j}}\text{ \ \ (see
remarks below)}  \notag \\
&&\overset{(\ref{Nk})+(\ref{hjhl})}{=}\frac{1}{m(\tilde{F})}\int_{\tilde{F}}%
\mathcal{F}_{\tilde{g}}(x)dv_{\tilde{F}}.  \notag
\end{eqnarray}

\noindent Here for using (\ref{mF}) above one requires the descent property
of $\mathcal{F}_{\tilde{g}}$ (and $\chi _{j})$ that $\mathcal{F}_{\tilde{g}}$
be invariant under the action of $H=G_{j}$ $\subset $ $S^{1},$ which holds
as seen in Corollary \ref{C-Fgm}$.$ By (\ref{Fg}) we can now identify terms
(with $m$ $=$ $0)$ in (\ref{MF}) with those in \cite[(14.3) on p.184 for the
case $\gamma $ $=$ Id]{Du} and identify our $\mathcal{F}_{\tilde{g}}\mathcal{%
\ }$in (\ref{Fg}) with the characteristic class $\alpha _{\tilde{F}}$ of
Duistermaat in \cite[(14.4)]{Du} (using (\ref{FkTd}) and Corollary \ref%
{C-Fgm}), modulo certain sign differences between $\mathcal{F}_{\tilde{g}}$
and $\alpha _{\tilde{F}}.$ These sign issues are discussed in the following
remark.

\begin{remark}
\label{R-8-C} The definition of the Todd class given in \cite[p.163]{Du} is
basically $\det_{\mathbb{C}}$($\frac{i}{2\pi }\Omega /(1-e^{-(i/2\pi )\Omega
}))$ which is indeed consistent with the usual definition of the Todd class
provided that $\Omega $ is put in the form of curvature (at least for the K%
\"{a}hler case). Unfortunately, Duistermaat points out that the matrix of $R$
(cf. \cite[p.54]{Du}) is equal to \textit{minus} the matrix of $\Omega $ 
\cite[p.160]{Du}, which seems to render his Todd class different from the
usual one. Duistermaat's adoption of such an opposite sign convention above
is explained by himself in \cite[pp. 56-57]{Du} where his remarks end up
with \textquotedblleft This is one of the numerous sources of sign confusion
in differential geometry $\cdot \cdot \cdot $". In spite of his effort for
clarification, his Proposition 13.1 \cite[p.163]{Du}, which is based on \cite%
[(11.17)]{Du}, seems to be confusing in view of the remark above and of the
previous Remark \ref{R-8-B}. The similar can be said with his orbifold
version of the index theorem \cite[Theorem 14.1 on p.184]{Du}, which is
based on his Proposition 13.1 (via his Proposition 13.2). Given these
confusions, if Duistermaat's characteristic classes were presumed to be the
same as the usual ones (as ours here), then his written form of the orbifold
index theorem \cite[Theorem 14.1]{Du} would agree with that of ours as shown
in our proof above.
\end{remark}

\begin{remark}
\label{R-8-40} For a slight simplification, recalling that by (\ref{Fg}) it
is going to be summing over $\tilde{F}$'$s$ with each $\tilde{F}$ associated
with $\tilde{g}$ $=$ $g_{0}^{k},$ $k=1,$ $\cdot \cdot ,$ $N,$ one can first
group those $\tilde{g}$'$s$ with the same $V^{H_{k+1}}$ (see Lemma \ref{L-Vg}%
) associated to $H_{0}$ $\subset $ $H_{1}$ $\subset $ $\cdot \cdot $ $%
\subset $ $H_{k+1}$ $\subset $ $\cdot \cdot $ $\subset $ $H$ in the OTS (see
the proof of Lemma \ref{L-ots}) and then group this summation over
\textquotedblleft types $V^{H_{k+1}}/H$" (with integrands still $\tilde{g}$%
-dependent).
\end{remark}

\begin{remark}
\label{R-8-41} ($m>0$ case) In the previous and present subsections, the
comparison between the formula of Duistermaat and that of ours is most
naturally set up and made when $m$ $=$ $0.$ If $m$ $>$ $0$ (without the
extra bundle $E$ $\rightarrow $ $\Sigma $ as before)$,$ our formula (\ref{MF}%
) involves extra factors $\tilde{g}^{m}$ (for integrals over $\Sigma ^{%
\tilde{g}}$). Using Remark \ref{9.3} this $m$-index on $\Sigma $ can be
converted to a natural index problem with the additional (orbifold) line
bundle $(L_{\Sigma }^{\ast })^{\otimes m}$ (let us call this $0$-index for
short) on the orbifold $M;$ compare Example \ref{E-IF}. After reaching such
a reduction, one can alternatively use Duistermaat's formula for this $0$%
-index computation. By similar computations as in this subsection (for $m$ $%
= $ $0$ above), it turns out that the relevant term $\lambda _{\tilde{L}}ch(%
\tilde{L})$ in \cite[(14.4) on p.184]{Du} due to the extra bundle $%
(L_{\Sigma }^{\ast })^{\otimes m}$ (i.e. $L$ of \cite{Du} is $(L_{\Sigma
}^{\ast })^{\otimes m}$ viewed as an orbifold line bundle on $M$ via descent
from $\Sigma $) produces the contribution similar to that of the term $%
\tilde{g}^{m}ch(\psi _{j}^{\ast }(L_{\Sigma }^{\ast })^{\otimes m}|_{V_{j}})$
in our formula (\ref{MF})\ where we have (\ref{FkTd}) inserted; here the
agreement between $\tilde{g}^{m}$ and $\lambda _{\tilde{L}}$ is from (\ref%
{3.0}) and \cite[the second paragraph of Section 14.5]{Du}. At this point
the two formulas yield the same answer.
\end{remark}

\section{\textbf{Nonextendability of open group action; meromorphic action 
\label{Sec9}}}

Let $M$ be a complex manifold (not necessarily compact) with a holomorphic $%
\mathbb{C}^{\ast }$-action $\sigma _{M}.$ That is, the map $\sigma _{M}:%
\mathbb{C}^{\ast }\times M\rightarrow M$ denoted as $\sigma _{M}(\lambda ,x)$
(also as $\sigma _{M}(\lambda )\circ x$ or $\lambda \circ x)$ is holomorphic
and satisfies the group action condition: $\sigma _{M}(\lambda _{1}\lambda
_{2})\circ x$ $=$ $\sigma _{M}(\lambda _{1})\circ (\sigma _{M}(\lambda
_{2})\circ x),$ $\sigma _{M}(1)\circ x$ $=$ $x$. Note that no other
condition such as freeness or local freeness is assumed on $\sigma _{M}.$

We say that $\sigma _{M}$ extends holomorphically to $0$ (resp.$\infty $)
provided that there exists a holomorphic map $\tilde{\sigma}_{M}:\mathbb{C}%
\times M\rightarrow M$ $($resp. $\tilde{\sigma}_{M}:(\mathbb{CP}%
^{1}\backslash \{0\})\times M\rightarrow M$) such that $\tilde{\sigma}_{M}$
equals $\sigma _{M}$ on $\mathbb{C}^{\ast }\times M.$ Both conditions hold
if and only if the holomorphic action $\sigma _{M}$ extends to a holomorphic
action $\tilde{\sigma}_{M}$ on $\mathbb{C}\mathbb{P}^{1}$:%
\begin{equation}
\tilde{\sigma}_{M}:\mathbb{CP}^{1}\times M\rightarrow M.  \label{Ext}
\end{equation}

\noindent We say that $\sigma _{M}$ is trivial if $\sigma _{M}(\xi ,x)$ $=$ $%
x$ for all $\xi $ $\in $ $\mathbb{C}^{\ast },$ $x$ $\in $ $M.$

We are going to show that the two-sided extension is impossible (even in the 
$C^{0}$ category; see Proposition \ref{lem-1}) unless the original action is
trivial. The above action and extension conditions can obviously be defined
in the $C^{\infty }$ or $C^{0}$ category.

\begin{proposition}
\label{propA1} Suppose that $M$ is a manifold with a smooth $\mathbb{C}%
^{\ast }$-action $\sigma _{M}$. Then it is impossible to extend $\sigma _{M}$
to a smooth map $\tilde{\sigma}_{M}$ of (\ref{Ext}) unless the action $%
\sigma _{M}$ is trivial.
\end{proposition}

Let us examine the simplest case: the topological group $G$ $:=$ $\mathbb{R}%
^{+}.$ Let $M$ be a topological space with a continuous $G$-action. Recall
that a compactification of a topological Hausdorff space $X$ is a pair $(%
\hat{X},h)$ consisting of a compact Hausdorff space $\hat{X}$ and a
homeomorphism $h$ of $X$ onto a dense subset of $\hat{X}$ (see \cite[p.242]%
{Dug}). We often view $X$ as a subspace of $\hat{X}$ by identifying $X$ with 
$h(X)$ $\subset $ $\hat{X}.$

\begin{proposition}
\label{lem-1} Suppose that the topological group $G$ is $\mathbb{R}^{+}$ and
that $M$ is a topological space with a continuous $G$-action $\phi $. Let $%
\bar{G}$ be any compactification of $G$ in the sense above ($\bar{G}$ need
not be a topological group)$,$ such that $\bar{G}\backslash G$ is a
countable set. Then $\phi $ as a map cannot be extended continuously to $%
\bar{G}.$ That is to say, there does not exist a continuous map%
\begin{equation*}
\tilde{\phi}:\bar{G}\times M\rightarrow M
\end{equation*}%
\noindent such that $\tilde{\phi}=\phi $ on $G\times M$ unless the action $%
\phi $ is trivial.
\end{proposition}

\begin{remark}
\label{rkR} One can use $\sin (\frac{1}{x})$-like graphs to easily construct
examples $\overline{\mathbb{R}^{+}}$ such that both cases where $\overline{%
\mathbb{R}^{+}}\backslash \mathbb{R}^{+}$ is countable or is uncountable can
occur.
\end{remark}

\proof
\textbf{(of Proposition \ref{lem-1})} Take a sequence $\lambda _{n}$ $\in $ $%
G$ $=$ $\mathbb{R}^{+}$ such that%
\begin{equation}
\lim_{n\rightarrow \infty }\lambda _{n}=a\in \bar{G}\backslash G\text{ and }%
\lim_{n\rightarrow \infty }\lambda _{n}^{-1}=b\in \bar{G}\backslash G
\label{8-4a}
\end{equation}%
\noindent by compactness of $\bar{G}$. For any $\mu $ $\in $ $(1-\delta ,1)$
with small $\delta $ $>$ $0,$ there exists a subsequence $\lambda _{n(\mu )}$
of $\lambda _{n}$ such that 
\begin{equation}
\lim_{n(\mu )\rightarrow \infty }\mu \lambda _{n(\mu )}=\alpha (\mu )\in 
\bar{G}\backslash G.  \label{8-4b}
\end{equation}

Since $\bar{G}\backslash G$ is countable by assumption, the map $\alpha :$ $%
(1-\delta ,1)$ $\rightarrow $ $\bar{G}\backslash G$ in (\ref{8-4b}) cannot
be injective. There exist $\mu _{1},$ $\mu _{2}$ $\in $ $(1-\delta ,1)$ such
that $\alpha (\mu _{2})$ $=$ $\alpha (\mu _{1}),$ $\mu _{2}$ $<$ $\mu _{1}.$
For $x$ $\in $ $M$ we consider%
\begin{equation}
\mu _{1}\circ x=(\mu _{1}\lambda _{n(\mu _{1})})\circ (\lambda _{n(\mu
_{1})}^{-1}\circ x)  \label{A1}
\end{equation}

\noindent where we denote $\phi (g,x)$ by $g\circ x.$

Assuming the extension $\tilde{\phi}$ (or $\tilde{\circ}$ for convenience)
exists, we are going to prove that the action $\phi $ is trivial. Taking the
limit $n(\mu _{1})$ $\rightarrow $ $\infty $ in (\ref{A1}) we get, via (\ref%
{8-4a}) and (\ref{8-4b}), $\mu _{1}\circ x=\alpha (\mu _{1})\tilde{\circ}(b%
\tilde{\circ}x).$ Similarly we have $\mu _{2}\circ x=\alpha (\mu _{2})\tilde{%
\circ}(b\tilde{\circ}x).$ We obtain $\mu _{1}\circ x$ $=$ $\mu _{2}\circ x$
since $\alpha (\mu _{2})$ $=$ $\alpha (\mu _{1})$ by assumption$.$ It
follows that%
\begin{equation}
(\mu _{1}^{-1}\mu _{2})\circ x=x.  \label{A4}
\end{equation}

Since given any small $\varepsilon >0$ there exists a $\delta >0$ such that $%
\mu _{1}^{-1}\mu _{2}$ $\in $ $(1-\varepsilon ,1)$ if $\mu _{2},$ $\mu _{1}$ 
$\in $ $(1-\delta ,1)$ and $\mu _{2}$ $<$ $\mu _{1}$ ($\alpha (\mu _{2})$ $=$
$\alpha (\mu _{1})),$ it follows from (\ref{A4}) that \{$1\}$ cannot be a
connected component of the closed isotropy subgroup $G_{0}$ $:=$ $\{g$ $\in $
$G$ \TEXTsymbol{\vert} $g\circ x$ $=$ $x\}$, which is a Lie subgroup of $G$
by \cite[Ch.II, Theorem 2.3]{He} (or Remark in the end of its proof). This
implies that $\dim G_{0}$ $\geq $ $1$ and hence $G_{0}=$ $G$ since $\dim G$ $%
=$ $1.$ We have shown that the isotropy subgroup of any $x$ $\in $ $M$ is $%
G; $ this amounts to the triviality of the action $\phi $.

\endproof%

\proof
\textbf{(of Proposition \ref{propA1}) }Clearly this proposition follows
immediately from Proposition \ref{lem-1}.

\endproof%

We say that the holomorphic action $\sigma _{M}$ extends \textit{pointwise
holomorphically} to $\mathbb{CP}^{1}$ if for any point $p$ $\in $ $M,$ there
is a holomorphic map $\varphi $ $:$ $\mathbb{CP}^{1}\times \{p\}$ $%
\rightarrow $ $M$ such that 
\begin{equation*}
\varphi (\lambda ,p)=\sigma _{M}(\lambda ,p)
\end{equation*}

\noindent for $\lambda $ $\in $ $\mathbb{C}^{\ast }.$

By Proposition \ref{lem-1} there is even no continuous extension of a
nontrivial holomorphic $\mathbb{C}^{\ast }$-action to $\mathbb{CP}^{1}$ $%
\times $ $M$ $\rightarrow $ $M$ . However, the \textit{meromorphic extension}
does possibly exist. We say that $\sigma _{M}$ extends meromorphically to $%
\mathbb{CP}^{1}\times M$ if $\sigma _{M}$ extends to a meromorphic map $%
\check{\sigma}_{M}$ $:$ $\mathbb{CP}^{1}\times M$ - - -\TEXTsymbol{>} $M$
(in the sense of Remmert \cite{GPR})$.$ Note that the singular set of a
meromorphic map is of complex codimension $\geq 2$ (\cite{GPR}). See Remark %
\ref{mero} below for examples of meromorphic extension.

In the remaining of this section we mainly assume that the meromorphic
extension exists. Let $\Omega _{M}^{p}$ on $M$ denote the holomorphic vector
bundle of holomorphic $p$-forms$.$ $\sigma _{M}$ induces a holomorphic
action on $\Omega ^{p}$ by pulling back. Let $H^{0}(M,\Omega _{M}^{p})$
denote the space of all global holomorphic $p$-forms.

\begin{notation}
\label{N-9-1} Let $H_{k}^{0}(M,\Omega _{M}^{p})$ or $H_{k,\sigma
_{M}}^{0}(M,\Omega _{M}^{p})$ denote the space of all global holomorphic $p$%
-forms $\omega $ such that $\sigma _{M}(\lambda )^{\ast }(\omega )=\lambda
^{k}\omega ,$ $\lambda \in \mathbb{C}^{\ast }$ where $\sigma _{M}(\lambda )$ 
$:$ $M\rightarrow M$ is given by%
\begin{equation}
\sigma _{M}(\lambda )(p):=\sigma _{M}(\lambda ,p).  \label{8-8.5}
\end{equation}
\end{notation}

\noindent Here there is no need to talk about any regularity condition as in
Definition \ref{2m}.

\begin{proposition}
\label{propA2}. With the notation as above, suppose that $\sigma _{M}$
extends meromorphically to $\mathbb{CP}^{1}\times M$. Then we have%
\begin{equation}
H_{k}^{0}(M,\Omega _{M}^{p})=\{0\}\text{ for }k\neq 0;\text{ }H^{0}(M,\Omega
_{M}^{p})=H_{0}^{0}(M,\Omega _{M}^{p})\text{ }(=H_{0,\sigma
_{M}}^{0}(M,\Omega _{M}^{p})).  \label{A5}
\end{equation}
\end{proposition}

\proof
Let $\omega $ $\in $ $H^{0}(M,\Omega _{M}^{p}).$ Since $\sigma _{M}$ extends
meromorphically, its pullback $\sigma _{M}^{\ast }\omega $ on $\mathbb{C}%
^{\ast }\times M$ (may contain the factor $d\lambda )$ extends to a
holomorphic $p$-form on $\mathbb{CP}^{1}\times M$ by Hartogs' theorem (e.g. 
\cite[p.81]{GPR}). In particular, the parameter $\lambda $ in the
holomorphic $p$-form $\sigma _{M}(\lambda )^{\ast }\omega $ on $M$ (see (\ref%
{8-8.5})) extends holomorphically to $\mathbb{CP}^{1}.$ One is thus allowed
to expand $\sigma _{M}(\lambda ^{-1})^{\ast }\omega $ near $\lambda $ $=$ $0$
to get%
\begin{equation}
\sigma _{M}(\lambda ^{-1})^{\ast }\omega =\sum_{k=0}^{\infty }\lambda
^{k}\omega _{k}  \label{A6-a}
\end{equation}

\noindent where $\omega _{k}$ $\in $ $H^{0}(M,\Omega _{M}^{p})$ does not
depend on $\lambda .$ Then, combining%
\begin{eqnarray*}
\sum_{k=0}^{\infty }\lambda ^{k}\zeta ^{k}\omega _{k} &=&\sigma
_{M}((\lambda \zeta )^{-1})^{\ast }\omega \overset{(\ref{A6-a})}{=}\sigma
_{M}(\zeta ^{-1})^{\ast }(\sum_{k=0}^{\infty }\lambda ^{k}\omega _{k}) \\
&\overset{}{=}&\sum_{k=0}^{\infty }\lambda ^{k}(\sigma _{M}(\zeta
^{-1})^{\ast }\omega _{k})
\end{eqnarray*}

\noindent with (\ref{A6-a}), we have 
\begin{equation}
\sigma _{M}(\zeta ^{-1})^{\ast }\omega _{k}=\zeta ^{k}\omega _{k}  \label{A7}
\end{equation}%
\noindent i.e., $\omega _{k}$ $\in $ $H_{-k}^{0}(M,\Omega _{M}^{p}).$ For $k$
$>$ $0,$ the LHS of (\ref{A7}) is finite as $\zeta \rightarrow \infty $ and
the RHS goes to infinity (unless $\omega _{k}=0).$ We conclude that $\sigma
_{M}(\lambda ^{-1})^{\ast }\omega $ $=$ $\omega _{0}$ in (\ref{A6-a}). This
yields that $\omega $ is $\mathbb{C}^{\ast }$-invariant. Now that $\omega $ $%
\in $ $H^{0}(M,\Omega _{M}^{p})$ is arbitrarily chosen, we are led to (\ref%
{A5}).

\endproof%

\begin{remark}
\label{mero} (Examples for the meromorphic extension) The assumption that $%
\sigma _{M}$ extends meromorphically to $\mathbb{CP}^{1}\times M$ in
Proposition \ref{propA2} holds for any compact K\"{a}hler manifold $M$ with $%
\sigma _{M}$ having a fixed point (on $M)$. For, by \cite[Proposition II]%
{Som} $\mathbb{C}^{\ast }$ (through $\sigma _{M})$ acts projectively on $M.$
By \cite[Lemma II-B]{Som} for $Y$ $=$ $X$ $=$ $M$, $\sigma _{M}$ extends
meromorphically to $\check{\sigma}_{M}$ $:$ $\mathbb{CP}^{1}\times M$ - - -%
\TEXTsymbol{>} $M.$ Another natural class of examples consists of $M$ that
is algebraic and $\sigma _{M}$ that is an algebraic action. Then $\sigma
_{M} $ automatically extends to $\mathbb{CP}^{1}\times M$ meromorphically.
See, e.g., \cite[p.777]{BBS} and Remark \ref{rext} below.
\end{remark}

Proposition \ref{propA2} has an application to the study of $S^{1}$-action
on a complex manifold $M$ via biholomorphisms. Denote such an action by $%
\sigma _{M}^{S^{1}}$ $:$ $S^{1}\times M$ $\rightarrow $ $M.$ Assume that we
can pass $\sigma _{M}^{S^{1}}$ to a holomorphic $\mathbb{C}^{\ast }$-action $%
\sigma _{M}^{\mathbb{C}^{\ast }}$ $:$ $\mathbb{C}^{\ast }\times M$ $%
\rightarrow $ $M.$ This $\mathbb{C}^{\ast }$-extension follows automatically
if $M$ is compact (cf. \cite[p. 50]{CS})$.$ For an integer $k,$ we define $%
H_{k,\sigma _{M}^{S^{1}}}^{0}(M,\Omega _{M}^{p})$ to be the space of all
holomorphic $p$-forms $\omega $ such that $\sigma _{M}^{S^{1}}(e^{i\vartheta
})^{\ast }\omega $ $=$ $e^{ik\vartheta }\omega .$ By Remark \ref{mero}, if $%
M $ is algebraic with an algebraic action $\sigma _{M}^{\mathbb{C}^{\ast }}$%
, then $\sigma _{M}^{\mathbb{C}^{\ast }}$ admits a meromorphic extension on $%
\mathbb{CP}^{1}\times M.$

The following result might be known to the experts, but we are unable to
find a precise reference.

\begin{corollary}
\label{corS1} With the notation above, suppose that $\sigma _{M}^{\mathbb{C}%
^{\ast }}$ extends meromorphically to $\mathbb{CP}^{1}\times M$. Then we have%
\begin{equation}
H_{k,\sigma _{M}^{S^{1}}}^{0}(M,\Omega _{M}^{p})=0\text{ for }k\neq 0;\text{ 
}H^{0}(M,\Omega _{M}^{p})=H_{0,\sigma _{M}^{S^{1}}}^{0}(M,\Omega _{M}^{p}).
\label{S1}
\end{equation}
\end{corollary}

\proof
Observe that $H_{k,\sigma _{M}^{S^{1}}}^{0}(M,\Omega _{M}^{p})$ $\subset $ $%
H^{0}(M,\Omega _{M}^{p})=H_{0,\sigma _{M}^{\mathbb{C}^{\ast }}}^{0}(M,\Omega
_{M}^{p})$ $\subset $ $H_{0,\sigma _{M}^{S^{1}}}^{0}(M,$ $\Omega _{M}^{p})$
by Proposition \ref{propA2}, and hence (\ref{S1}).

\endproof%

As another application of Proposition \ref{propA2}, we now want to relate
Proposition \ref{propA2} to \cite[Corollary IV]{CS}. The main result is
Proposition \ref{CCS} below.

Some preparations are in order. Denote $\sigma _{M}(\lambda )\circ x$ by $%
\lambda \circ x$. Suppose that $\sigma _{M}$ extends meromorphically to $%
\check{\sigma}_{M}$ $:$ $\mathbb{C}\mathbb{P}^{1}\times M$ - - -\TEXTsymbol{>%
} $M.$ Then $\sigma _{M}$ extends pointwise holomorphically to $\mathbb{C}%
\mathbb{P}^{1},$ 
i.e. (in particular) $\lim_{\lambda \rightarrow 0}\lambda \circ x$ exists
for any fixed $x\in M$ (cf. \cite[Lemma 2.4.1]{CG})$.$ Moreover, the
singular set of the meromorphic extension $\check{\sigma}_{M}$ is contained
in ($\{0\}\times S_{0})$ $\cup $ ($\{\infty \}\times S_{\infty })$ for some
subvarieties $S_{0},$ $S_{\infty }$ of codimension $\geq $ 1 in $M.$ Observe
that 
\begin{equation*}
\sigma _{M,0}:=\check{\sigma}_{M}(0,\cdot ):M\backslash S_{0}\rightarrow M
\end{equation*}%
\noindent is meromorphic on $M$ since $\{0\}\times M\nsubseteq \{0\}\times
S_{0}$ (cf. \cite[pp.35-36]{St}). So $\sigma _{M,0}$ is actually defined and
holomorphic on $M\backslash T$ where $T$ $\subset $ $S_{0}$ is of
codimension at least 2 in $M$. Denote this extension by $\tilde{\sigma}%
_{M,0}:M\backslash T\rightarrow M.$ Note that $\tilde{\sigma}_{M,0}$ $=$ $%
\sigma _{M,0}$ on $M\backslash S_{0},$ but $\sigma _{M,0}$ is not defined on 
$S_{0}\backslash T$ where $\tilde{\sigma}_{M,0}$ is defined.

We define $F^{0}$ :$=$ $\tilde{\sigma}_{M,0}(M\backslash T)$ and $G^{0}$ :$=$
$\sigma _{M,0}(M\backslash S_{0})$. Obviously $G^{0}\subset F^{0}.$ Let $%
F^{\sigma _{M}}$ be the fixed point set of $\sigma _{M}$ on $M,$ i.e. $%
\{x\in M$ $|$ $\sigma _{M}(\lambda )x$ $=$ $x,$ $\forall \lambda $ $\in $ $%
\mathbb{C}^{\ast }\}.$ $G^{0}$ equals $\{\lim_{\lambda \rightarrow 0}\lambda
\circ x$ : $x\in M\backslash S_{0}\}$ (while the same statement for $F^{0}$
via $M\backslash T$ may not hold, \textit{a priori}$)$ which is contained in 
$F^{\sigma _{M}}.$ Observe that $F^{0}$ (resp. $G^{0})$ is connected since $%
\tilde{\sigma}_{M,0}$ (resp. $\sigma _{M,0})$ is continuous on the connected
space $M\backslash T$ (resp. $M\backslash S_{0}).$

\begin{lemma}
\label{LF0} With the notation above, it holds that (a) $G^{0}$ is a complex
submanifold in $M;$ (b) $G^{0}\subset M\backslash S_{0};$ (c) $G^{0}=F^{0};$
(d) $G^{0}$ is closed in $M.$
\end{lemma}

The proof of Lemma \ref{LF0} is postponed below.

Define%
\begin{equation}
\pi :M\backslash T\rightarrow F^{0}\text{ by }\pi (x):=\tilde{\sigma}%
_{M,0}(x)  \label{8-15.5}
\end{equation}

\noindent Note that for $x\in M\backslash S_{0},$ $\pi (x)$ $=$ $%
\lim_{\lambda \rightarrow 0}\lambda \circ x.$ We can extend $\pi $ to $M$ by 
$\pi (x)$ $:=$ $\lim_{\lambda \rightarrow 0}\lambda \circ x$ $\in $ $%
F^{\sigma _{M}}$ $\subset $ $M$ (cf. the pointwise holomorphic extension as
mentioned earlier)$.$ Note that $\pi $ in (\ref{8-15.5}) is holomorphic on $%
M\backslash T$ $\subset $ $M$ (since $\pi |_{M\backslash T}$ $=$ $\tilde{%
\sigma}_{M,0}|_{M\backslash T}),$ yet $\pi $ on $M$ could be discontinuous
across $T.$ We call $\pi $ on $M$ a canonical extension of $\sigma _{M,0}$
and $\tilde{\sigma}_{M,0}.$

\begin{lemma}
\label{LFM} Let $\sigma _{M}$ be a holomorphic $\mathbb{C}^{\ast }$-action
on a complex manifold $M$ (not necessarily compact). Suppose $\sigma _{M}$
extends meromorphically to $\mathbb{C}\mathbb{P}^{1}\times M$ - - -%
\TEXTsymbol{>} $M$. Then there is a linear isomorphism: 
\begin{equation}
\pi ^{\ast }:H^{0}(F^{0},\Omega _{F^{0}}^{p})\rightarrow H_{0}^{0}(M,\Omega
_{M}^{p})  \label{8-17.5}
\end{equation}%
\noindent where $\pi ^{\ast }$ is essentially defined via the "pullback" of
the canonical map $\pi $ of (\ref{8-15.5}).
\end{lemma}

\proof
For $\omega $ $\in $ $H^{0}(F^{0},\Omega _{F^{0}}^{p})$ (which makes sense
by Lemma \ref{LF0} (a), (c))$,$ we can now define $\pi ^{\ast }\omega $ \ of
(\ref{8-17.5}) to be $\pi ^{\ast }\omega $ := $\tilde{\sigma}_{M,0}^{\ast
}\omega $ $\in $ $H^{0}(M,\Omega _{M}^{p})$ by Hartogs' extension theorem
since $T$ is of codimension at least 2 in $M$. From $\pi =\pi \circ (\sigma
(\lambda ))$ on $M\backslash S_{0},$ it follows that $\pi ^{\ast }\omega $ $%
= $ $\sigma (\lambda )^{\ast }(\pi ^{\ast }\omega )$ on a dense open subset $%
U_{\lambda }$ (depending on $\lambda )$ of $M$ (since $\pi $ could be
discontinuous), giving that $\pi ^{\ast }\omega $ is $\sigma (\lambda )$%
-invariant on $U_{\lambda }$ hence on $M$ ($\pi ^{\ast }\omega $ being
globally holomorphic as just defined) i.e. $\pi ^{\ast }\omega $ $\in $ $%
H_{0}^{0}(M,\Omega _{M}^{p}).$ The map $\pi ^{\ast }$ of (\ref{8-17.5}) is
now well defined.

Suppose $\pi ^{\ast }\omega $ $=$ $0$ on $M.$ By $\iota _{G^{0}}:$ $%
G^{0}\hookrightarrow $ $M$ and $M\backslash S_{0}$ $\overset{\pi }{%
\rightarrow }$ $G^{0}$ and $\pi |_{G^{0}}$ $=$ $Id,$ one has $\pi
|_{M\backslash S_{0}}\circ \iota _{G^{0}}$ $=$ $Id$ on $G^{0}$ (note that $%
G^{0}\subset M\backslash S_{0}$ by Lemma \ref{LF0} (b)$).$ So (recalling
that $\omega $ is on $F^{0}$ $\supset $ $G^{0})$ $\omega |_{G^{0}}$ $=$ $%
\iota _{G^{0}}^{\ast }\pi |_{M\backslash S_{0}}^{\ast }\omega |_{G^{0}}$ $=$ 
$\iota _{G^{0}}^{\ast }(\pi ^{\ast }\omega )|_{M\backslash S_{0}}$ $=$ $0$
if $\pi ^{\ast }\omega $ $=$ $0,$ giving $\omega $ $=$ $0$ on $G^{0}.$ Hence 
$\omega =0$ on $F^{0}$ ($=G^{0}$ by Lemma \ref{LF0} (c)). That is, $\pi
^{\ast }$ of (\ref{8-17.5}) is injective.

For the surjectivity of $\pi ^{\ast }$ let $\mathring{\pi}$ $:$ $M\backslash
S_{0}$ $\rightarrow $ $G^{0}$ be the restriction of $\pi $ to $M\backslash
S_{0}$ $\subset $ $M\backslash T$ (namely $\mathring{\pi}$ $=$ $\sigma
_{M,0}).$ At a regular point $(\lambda ,x)$ $\in $ $\mathbb{C}\mathbb{P}%
^{1}\times (M\backslash S_{0}),$ for $v$ $\in $ $T_{x}M$ one has $\mathring{%
\pi}_{\ast }v$ $=$ $\lim_{\lambda \rightarrow 0}\sigma (\lambda )_{\ast }v$ $%
\in $ $T_{\mathring{\pi}(x)}G^{0}$ (which might not hold for $\mathring{\pi}%
, $ $S$ and $G^{0}$ replaced by $\pi ,$ $T$ and $F^{0}$ respectively$)$. Now
given $\eta $ $\in $ $H_{0}^{0}(M,\Omega _{M}^{p}),$ that is $\sigma
(\lambda )^{\ast }\eta =\eta $ ($\forall \lambda $ $\in $ $\mathbb{C}^{\ast
}),$ this invariance yields at $x$ $\in $ $M\backslash S_{0}$ with $%
v_{1},..,v_{p}$ $\in $ $T_{x}(M\backslash S_{0})$%
\begin{eqnarray}
\eta _{x}(v_{1},..,v_{p}) &=&\eta _{\lambda \circ x}(\sigma (\lambda )_{\ast
}v_{1},\cdot \cdot ,\sigma (\lambda )_{\ast }v_{p})  \label{L1} \\
&=&\eta _{\mathring{\pi}(x)}(\lim_{\lambda \rightarrow 0}\sigma (\lambda
)_{\ast }v_{1},\cdot \cdot ,\lim_{\lambda \rightarrow 0}\sigma (\lambda
)_{\ast }v_{p})  \notag \\
&=&\eta _{\mathring{\pi}(x)}(\mathring{\pi}_{\ast }v_{1},\cdot \cdot ,%
\mathring{\pi}_{\ast }v_{p}).  \notag
\end{eqnarray}

\noindent The equality of (\ref{L1}) amounts to asserting that $\mathring{\pi%
}^{\ast }(\iota _{G^{0}}^{\ast }\eta )$ $=$ $\eta $ on $M\backslash S_{0}.$
In view that $M\backslash S_{0}$ is dense (and open) and $\mathring{\pi}%
^{\ast }(\iota _{G^{0}}^{\ast }\eta )$ (and $\eta )$ is actually holomorphic
on the whole $M$ (by $\pi ^{\ast }$ of (\ref{8-17.5}))$,$ it follows that $%
\mathring{\pi}^{\ast }(\iota _{G^{0}}^{\ast }\eta )$ $=$ $\eta $ holds on $%
M, $ giving in turn that $\pi ^{\ast }(\iota _{F^{0}}^{\ast }\eta )$ $=$ $%
\eta $ on $M\backslash S_{0}$ thus on $M$ (or, using $G^{0}$ $=$ $F^{0}$ in
Lemma \ref{LF0} (c)). This amounts to yielding $\eta $ $\in $ $\func{Im}\pi
^{\ast }.$ We have shown that $\pi ^{\ast }$ of (\ref{8-17.5}) is surjective
hence in turn, an isomorphism.

\endproof%

By Proposition \ref{propA2} and Lemma \ref{LFM}, we immediately obtain

\begin{proposition}
\label{CCS} (cf. \cite[Corollary IV]{CS} for $M$ compact K\"{a}hler) Let $%
\sigma _{M}$ be a holomorphic $\mathbb{C}^{\ast }$-action on a complex
manifold $M$ which can be noncompact. Suppose that $\sigma _{M}$ extends
meromorphically to $\mathbb{C}\mathbb{P}^{1}\times M$ - - -\TEXTsymbol{>} $M$%
. Then we have a natural linear isomorphism:%
\begin{equation*}
H^{0}(F^{0},\Omega _{F^{0}}^{p})\simeq H^{0}(M,\Omega _{M}^{p}).
\end{equation*}
\end{proposition}

Let us compare Proposition \ref{CCS} with \cite[Proposition II]{CS} where $M$
is assumed to be a connected compact K\"{a}hler manifold and $\sigma _{M}$
has at least one fixed point. In our notation above, if $M$ is compact, the
closure $\overline{F^{0}}$ in the usual complex topology is known to be a
complex subvariety of $M$ by standard argument (or, one may see this via
(c), (d) of Lemma \ref{LF0})$.$ Note that $\overline{F^{0}},$ contained in $%
F^{\sigma _{M}}$ (fixed point set), is connected. In such a special
situation $\overline{F^{0}}$ is equal/reduced to the \textit{source, }%
denoted by\textit{\ }$F_{1},$ as is introduced in \cite[Proposition II,
pp.55-56]{CS}; compare the proof of Lemma \ref{LF0} (d) below.

Our result and proof above differ from those of \cite[Corollary IV]{CS} in
that in \cite{CS} the complex manifold $M$ is assumed to be compact and its
proof relies on the so-called\textit{\ invariant decomposition} of $M$
(associated with the $\mathbb{C}^{\ast }$-action), which was originally
discovered by A. Bialynicki-Birula in the algebraic setting (cf. \cite{BB}).
The comparison mentioned above is established by the following:

\begin{lemma}
\label{9.9-5} Let $M,$ $\sigma _{M}$ be as in Proposition \ref{CCS} (so $M$
can be noncompact). Assume that $\overline{F^{0}}$ is an analytic subvariety
of $M.$ We have $H^{0}(F^{0},\Omega _{F^{0}}^{p})$ $=$ $H^{0}(\overline{F^{0}%
},\Omega _{\overline{F^{0}}}^{p}).$ Here, with $\overline{F^{0}}$ being
possibly singular $\Omega _{\overline{F^{0}}}^{p}$ is defined in the sense
of algebraic geometry (cf. \cite[Chap.II, Section 8]{Ha}). In particular,
for $M$ compact we obtain, together with Proposition \ref{CCS}, $%
H^{0}(M,\Omega _{M}^{p})$ $\cong $ $H^{0}(F_{1},\Omega _{F_{1}}^{p})$ by $%
\overline{F^{0}}=F_{1},$ as originally stated in \cite{CS}.
\end{lemma}

\begin{remark}
\label{9.10-5} In fact $F^{0}=\overline{F^{0}}$ as can be seen in the proof
of Lemma \ref{LF0} below (for (c) that $G^{0}=F^{0},$ whose argument applies
here similarly).
\end{remark}

\proof
(of Lemma \ref{9.9-5}) The natural map $H^{0}(\overline{F^{0}},\Omega _{%
\overline{F^{0}}}^{p})$ $\rightarrow $ $H^{0}(F^{0},\Omega _{F^{0}}^{p})$
(induced by $F^{0}$ $\hookrightarrow $ $\overline{F^{0}})$ is injective, so
it suffices to show that every $\omega $ $\in $ $H^{0}(F^{0},\Omega
_{F^{0}}^{p})$ on $F^{0}$ can be extended to $\overline{F^{0}}$ (hence the
map is also surjective). This is in general not possible unless $\overline{%
F^{0}}\backslash F^{0}$ is of codimension $\geq 2$ in $\overline{F^{0}}.$ We
are not going to need such a codimension condition (cf. Remark \ref{9.10-5}%
). Instead, making use of Proposition \ref{CCS} and its proof we have that $%
\omega $ must be of the form $\iota _{F^{0}}^{\ast }\eta $ ($\iota _{F^{0}}$ 
$:$ $F^{0}$ $\hookrightarrow $ $M)$ for a unique $\eta $ $\in $ $%
H^{0}(M,\Omega _{M}^{p}).$ It is now seen that $\iota _{\overline{F^{0}}%
}^{\ast }\eta $ ($\iota _{\overline{F^{0}}}:\overline{F^{0}}$ $%
\hookrightarrow $ $M)$ $\in $ $H^{0}(\overline{F^{0}},\Omega _{\overline{%
F^{0}}}^{p})$ is the desired extension of $\omega .$

\endproof%

\proof
(\textbf{of Lemma \ref{LF0}}) Compared to the proofs above, the proof below
is less conceptual as it mostly relies on use of local coordinates. Recall $%
G^{0}$ $\subset $ the fixed point set $F^{\sigma _{M}}$ of $\sigma _{M}$.
Near any fixed point $q$ there exists a chart $U$ of local holomorphic
coordinates $(w_{1},$ $\cdot \cdot \cdot ,$ $w_{n})$ such that if $\lambda
\in \mathbb{C}^{\ast }$, then (\cite[p.56]{CS}, however, see Remark \ref%
{9-12-1} below)%
\begin{equation}
\sigma _{M}(\lambda )(w_{1},\cdot \cdot \cdot ,w_{n})=(\lambda
^{k_{1}}w_{1},\cdot \cdot ,\lambda ^{k_{s}}w_{s},w_{s+1},\cdot \cdot
,w_{t},\lambda ^{l_{t+1}}w_{t+1},\cdot \cdot ,\lambda ^{l_{n}}w_{n})
\label{lhc}
\end{equation}%
\noindent for integers $k_{j}>0,$ $l_{j}<0$ ($s$ may be $0$ and $t$ may be $%
n)$ with $(w_{1},$ $\cdot \cdot \cdot ,$ $w_{n})(q)$ $=$ $(0,$ $\cdot \cdot
\cdot $, $0).$

Given $q\in G^{0},$ there is a point $\tilde{q}\in M\backslash S_{0}$ such
that $\sigma _{M,0}(\tilde{q})$ $=$ $q.$ By $\sigma _{M,0}(\tilde{q})$ $=$ $%
\lim_{\lambda \rightarrow 0}\lambda \tilde{q}$ we assume that for some 0 $%
\neq $ $\lambda _{0}\sim 0$, $\lambda _{0}\tilde{q}$ lies in a local chart
where (\ref{lhc}) is valid. Without loss of generality we may assume $%
\lambda _{0}\tilde{q}$ $\notin $ $S_{0}$ using $\tilde{q}$ $\notin $ $S_{0}$
and analyticity. Moreover $\lambda _{0}\tilde{q}$ has coordinates $%
(w_{1}^{0},$ $\cdot \cdot ,w_{t}^{0},$ $0,$ $\cdot \cdot ,$ $0)$, proved by
contradiction using the negative power $l_{j}<0$ $($compare Remark \ref%
{9-12-1} below), and similarly $w_{s+1}^{0}$ $=$ $\cdot \cdot \cdot $ $=$ $%
w_{t}^{0}$ $=$ $0$ by $\lim_{\lambda \rightarrow 0}\lambda (\lambda _{0}%
\tilde{q})$ $=$ $q$ $=$ $(0,$ $0,$ $\cdot \cdot \cdot ,$ $0)$ and (\ref{lhc})%
$.$ Take an open connected neighborhood $\tilde{V}$ $\subset $ $M\backslash
S_{0}$ of $\tilde{q}$. Then $\sigma (\lambda _{0})\tilde{V}$ ($\sigma
(\lambda _{0})$ is a biholomorphism of $M$) is an open connected
neighborhood of $\lambda _{0}\tilde{q}$ in $M\backslash S_{0}$ (if $\tilde{V}
$ small enough)$,$ hence contains $V\times \{0_{n-t}\}$ where $V$ is an open
neighborhood of $(w_{1}^{0},$ $\cdot \cdot ,w_{t}^{0})$ in $\mathbb{C}^{t}$
and $0_{m}$ denotes the origin in $\mathbb{C}^{m}$. Then it follows from (%
\ref{lhc}) using $\lambda $ $\rightarrow $ $0$ that $\sigma _{M,0}(\tilde{V}$
$\cap $ $\{w_{t+1}$ $=$ $w_{t+2}$ $=$ $\cdot \cdot \cdot $ $=$ $w_{n}$ $=$ $%
0\})$ $\subset $ $\sigma _{M,0}(\tilde{V}$) $\subset $ $G^{0}$ covers an
open neighborhood of $q$ in $\{0_{s}\}\times \mathbb{C}^{t-s}\times
\{0_{n-t}\}.$ Conversely only points in the local chart having coordinates $%
(0_{s},w_{s+1},\cdot \cdot ,w_{t},$ $0_{n-t})$ can belong to $G^{0}.$ Thus $%
(w_{1},$ $\cdot \cdot \cdot ,$ $w_{n})$ in (\ref{lhc}) provide a complex
submanifold structure of $G^{0}$ in $M$ near $q.$ We have proved (a) of the
lemma.

Suppose $t$ $<$ $n$ in (\ref{lhc}) for $q\in G^{0}.$ For any $\tilde{q}_{1}$
sufficiently near $\tilde{q}$ with $\tilde{q}_{1},$ $\lambda _{0}\tilde{q}%
_{1}$ $\notin $ $S_{0},$ $\lambda _{0}\tilde{q}_{1}$ has coordinates with
vanishing $w_{t+1},$ $\cdot \cdot \cdot ,$ $w_{n}$ by the similar argument
as above, giving that $\sigma (\lambda _{0})\tilde{V}$ becomes degenerate, a
contradiction to the biholomorphism of $\sigma (\lambda _{0}).$ Thus%
\begin{equation}
t=n\text{ in (\ref{lhc}) for }q\in G^{0}.  \label{tn}
\end{equation}

Let \textit{Proj} $:U\rightarrow G^{0}$ be the coordinate projection via (%
\ref{lhc}) for $t=n,$ i.e. $(w_{1},$ $\cdot \cdot \cdot ,$ $w_{n})$ $%
\rightarrow $ $(0,$ $\cdot \cdot ,$ $0,$ $w_{s+1},$ $\cdot \cdot ,$ $w_{n})$%
. Clearly \textit{Proj }holomorphically\textit{\ }extends $\sigma _{M,0}$
(originally defined on $U\backslash S_{0})$ to $U$. Thus $\sigma _{M,0}$ is
regular at $q$ $=$ $(0,$ $\cdot \cdot ,$ $0)$ since \textit{Proj }is so,
suggesting that $G^{0}$ $\subset $ $M\backslash S_{0}$. To prove this
inclusion rigorously, by considering $\phi :$ $(w_{1},$ $\cdot \cdot \cdot ,$
$w_{n})$ $\rightarrow $ $(\lambda ^{k_{1}}w_{1},$ $\cdot \cdot \cdot ,$ $%
\lambda ^{k_{s}}w_{s},$ $w_{s+1},$ $\cdot \cdot \cdot ,$ $w_{n})$ one shows
in a similar way that $\sigma _{M}$ is regular at $(0,q)$ $\in $ $\mathbb{C}%
\mathbb{P}^{1}\times $ $M,$ and hence the claim $G^{0}$ $\subset $ $%
M\backslash S_{0}$ in (b) of the lemma is proved.

From the definition of the extension $\tilde{\sigma}_{M,0}$ it is not
difficult to see $F^{0}$ $\subset $ $\overline{G^{0}}$ (via $\tilde{\sigma}%
_{M,0}$ $=$ $\sigma _{M,0}$ on $M\backslash S_{0}).$ Obviously $\overline{%
G^{0}}$ $\subset $ $F^{\sigma _{M}}.$ So given $\bar{q}\in F^{0}$ there is a
neighborhood $U$ in $M,$ which is contained in a local chart first having
the property (\ref{lhc}) with $\bar{q}$ set to be $q$ there then having $t=n$
in view of $G^{0}$ $\subset $ $F^{0}$ $\subset $ $\overline{G^{0}}$ and (\ref%
{tn}). We claim $G^{0}\cap U$ $=$ $F^{0}\cap U.$ The coordinate projection 
\textit{Proj} above defined on $U$ also holomorphically extends $\tilde{%
\sigma}_{M,0}$ (on $U\backslash T)$ to $U$ meaning that $U$ is disjoint from
the singular set $T$ of $\tilde{\sigma}_{M,0}.$ So $T\cap U$ $=$ $\emptyset $
together with $S_{0}\cap U$ $=$ $\emptyset $ (as shown in the proof of (b)
above), implies that $\tilde{\sigma}_{M,0}$ $=$ $\sigma _{M,0}$ $=$ \textit{%
Proj }on $U.$ Thus $G^{0}\cap U$ $=$ $F^{0}\cap U$ ($=$ \textit{Proj(}$U$))
hence $G^{0}=F^{0}$ proving (c) of the lemma.

The proof of (c) actually shows that $G^{0}$ $=$ $\overline{G^{0}}$ (compare
the preceding sentence) hence (d) of the lemma. Recall the aforementioned $%
F_{1}$ after Proposition \ref{CCS}, which is defined in \cite{CS} for $M$
compact. Since $G^{0}$ is of the same complex dimension $t-s$ ($=$ $n-s)$ as
that of $F_{1}$ \cite{CS} which is connected, for $M$ compact we conclude $%
G^{0}(=$ $F^{0}=\overline{F^{0}}$ by (c), (d) above$)$ $=$ $F_{1}$ as
asserted earlier. In this connection, $F^{0}(=\overline{F^{0}})$ in our
treatment here can be regarded as a replacement of $F_{1}$ when $M$ is not
necessarily compact.

\endproof%

\begin{remark}
\label{9-12-1} Since (\ref{lhc}) for $\lambda $ $\in $ $\mathbb{C}^{\ast }$
is claimed without proof in \cite{CS}, let us assume that $U$ is a \textit{%
small} neighborhood of $q$ $=$ $(0,$ $0,$ $\cdot \cdot \cdot $, $0),$ in
which case (\ref{lhc}) only holds for $\lambda \sim 1.$ It might happen that
a point $p$ $\in $ $U$ with some $w_{k}(p)$ $\neq $ $0$ ($t+1$ $\leq $ $k$ $%
\leq $ $n)$ lying outside $U$ after the action by some $|\lambda |$ $<$ $1,$
travels back to $U$ after another action by $|\lambda |$ $<<$ $1.$ This
would not occur if (\ref{lhc}) were true for $\lambda $ $\in $ $\mathbb{C}%
^{\ast }$ or if $U$ were large. For the remedy here we put $q_{1}$ $=$ $%
\lambda _{0}\tilde{q}$ $\in $ $U.$ The fact $\lim_{\lambda \rightarrow
0}\lambda q_{1}$ $=$ $q$ gives that $\lambda q_{1}$ $\in $ $U$ for all $%
|\lambda |\leq \delta $ for some $\delta $ $>$ $0.$ Then we have, by setting 
$q_{2}$ $=$ $\delta q_{1},$ $\lambda q_{2}$ $\in $ $U$ for all $|\lambda |$ $%
\leq $ $1.$ This yields $w_{t+1}(q_{2})$ $=$ $\cdot \cdot \cdot $ $=$ $%
w_{n}(q_{2})$ $=$ $0$ otherwise it follows the contradiction that $\lambda
q_{2}$ $\notin $ $U$ for some $|\lambda |$ $<$ $1$ (if $w_{k}(q_{2})$ $\neq $
$0$ for some $t+1$ $\leq $ $k$ $\leq $ $n$ and $\lambda q_{2}$ $\in $ $U$
for all $|\lambda |$ $<$ $1$ then $U$ cannot be small). Put differently the
aforementioned scenario that a point frequently/always comes in and out
through $U$ under the actions $|\lambda |$ $<$ $1$ is intuitively seen to
get nowhere and therefore violates the foregoing existence of limit. Remark
that the above enables us to simply replace $\lambda _{0}$ in the original
argument by $\delta \lambda _{0}.$ We are now done.
\end{remark}

Another generalization of \cite{CS} from a quite different perspective is
given in the next section.

For a general holomorphic action $\sigma _{M}^{G}$ given by a connected
complex reductive Lie group $G$, we can show (below) that (\ref{A5}) of
Proposition \ref{propA2} still holds, without knowing the detailed $G$%
-invariant decomposition of $M$ as conjectured in \cite[p.115]{Som}.

To fix the notation for use shortly, let $\mathfrak{g}$ denote a complex
simple Lie algebra$,$ $\mathfrak{h}$ a fixed Cartan subalgebra of $\mathfrak{%
g}$ and $\Delta $ the set of all nonzero roots. In the root space
decomposition $\mathfrak{g}=\mathfrak{h+}\sum_{\alpha \in \Delta }\mathfrak{g%
}^{\alpha }$ one can choose $\alpha _{j}$ $\in $ $\Delta ,$ $j=1,...,$ $l,$ $%
X_{j}$ $\in $ $\mathfrak{g}^{\alpha _{j}}$ and $Y_{j}$ $\in $ $\mathfrak{g}%
^{-\alpha _{j}}$ such that%
\begin{equation}
\lbrack X_{j},Y_{j}]=H_{j}\in \mathfrak{h,}\text{ }[H_{j},X_{j}]=2X_{j},%
\text{ }[H_{j},Y_{j}]=-2Y_{j}  \label{sl(2)}
\end{equation}

\noindent and that $\mathfrak{g}$ is generated by $X_{j},$ $Y_{j},$ $H_{j},$ 
$1\leq j\leq l$ (cf. \cite[p.482]{He})$.$ Also $X_{j},$ $Y_{j},$ $H_{j}$
form a canonical basis for the Lie algebra $sl(2,\mathbb{C}).$ The following
is basic:

\begin{lemma}
\label{SLdense} Let $T=\left\{ \left( 
\begin{array}{cc}
\lambda & 0 \\ 
0 & \lambda ^{-1}%
\end{array}%
\right) \text{ }|\text{ }\lambda \in \mathbb{C}^{\ast }\right\} $ $\subset $ 
$H$ $=$ $SL(2,\mathbb{C}).$ Then the set $\tbigcup\limits_{h\in H}Ad(h)T$%
\textit{\ is dense in }$SL(2,\mathbb{C}).$
\end{lemma}

We are now ready to prove Theorem \ref{propA2-1} stated in the Introduction.

\proof
\textbf{(of Theorem \ref{propA2-1})} Let us first assume that $G$ is
semisimple. Associated to the Lie subalgebra $\mathfrak{g}_{j}$ $\cong $ $%
sl(2,\mathbb{C})$ spanned by $X_{j},$ $Y_{j},$ $H_{j}$ of (\ref{sl(2)})$,$
we have exactly one connected Lie subgroup $G_{j}$ of $G$ with its Lie
algebra equal to $\mathfrak{g}_{j}$ (cf. \cite[p.112]{He}). Let $\pi _{j}$ : 
$SL(2,\mathbb{C})$ $\rightarrow $ $G_{j}$ ($\subset $ $G)$ be the covering
map (since $SL(2,\mathbb{C})$ is simply connected). We have a Lie group
homomorphism $\psi $ : $\lambda $ $\in $ $\mathbb{C}^{\ast }$ $\rightarrow $ 
$G$ defined by%
\begin{equation*}
\psi :\lambda \longrightarrow \left( 
\begin{array}{cc}
\lambda & 0 \\ 
0 & \lambda ^{-1}%
\end{array}%
\right) \overset{\pi _{j}}{\longrightarrow }G_{j}\overset{incl}{%
\longrightarrow }G.
\end{equation*}

\noindent With the projective compactification $\bar{G}$ of $G$ (cf. Remark %
\ref{rext} below) and all the maps being natural, we conclude that $\phi $ $%
: $ $SL(2,\mathbb{C})$ $\rightarrow $ $\bar{G}$ defined by the composition
of $\pi _{j}$ and the inclusion map \textquotedblleft $incl"$ extends
meromorphically to $\overline{SL(2,\mathbb{C})}.$ Composed with the map $%
\lambda \rightarrow \left( 
\begin{array}{cc}
\lambda & 0 \\ 
0 & \lambda ^{-1}%
\end{array}%
\right) ,$ we obtain that $\psi $ extends meromorphically (holomorphically
in fact) to $\bar{\psi}:\mathbb{C}\mathbb{P}^{1}$ $=$ $\overline{\mathbb{C}%
^{\ast }}$ $\overset{}{\text{- - -\TEXTsymbol{>}}}$ $\bar{G}$ (cf. \cite[%
Lemma II-C]{Som})$.$ It implies that the holomorphic $\mathbb{C}^{\ast }$%
-action on $M$ via $\bar{\psi}$ extends to a meromorphic map as the
following composite:%
\begin{equation*}
\mathbb{C}\mathbb{P}^{1}\times M\text{ }\overset{\bar{\psi}\times id}{\text{%
- - -\TEXTsymbol{>}}}\text{ }\bar{G}\times M\overset{\check{\sigma}_{M}^{G}}{%
\text{- - -\TEXTsymbol{>}}}M.
\end{equation*}%
\noindent Here we have used the facts that $\check{\sigma}_{M}^{G}$ is
meromorphic as assumed in the theorem and that the image of $\bar{\psi}%
\times id$ is not contained in the singular set of $\check{\sigma}_{M}^{G}$
(cf. \cite[pp.35-36]{St}). Similarly the holomorphic $\mathbb{C}^{\ast }$
action on $M$ through the map $\psi _{h}$: 
\begin{equation*}
\lambda \longrightarrow h\left( 
\begin{array}{cc}
\lambda & 0 \\ 
0 & \lambda ^{-1}%
\end{array}%
\right) h^{-1}\overset{\pi _{j}}{\longrightarrow }G_{j}\overset{incl}{%
\longrightarrow }G
\end{equation*}

\noindent also extends meromorphically to $\mathbb{C}\mathbb{P}^{1}\times M$
- - -\TEXTsymbol{>} $M.$

By Proposition \ref{propA2} and Lemma \ref{SLdense}, we conclude that each
holomorphic $p$-form $\omega $ on $M$ is $G_{j}$-invariant via the action $%
\sigma _{M}^{G}.$ This implies that $\omega $ is also invariant under $%
\sigma _{M}^{G}(g)$ for $g$ in an open neighborhood $V$ of the identity of $%
G $ (\cite[p.115]{He}) since the Lie algebras of $G_{j}$ span the Lie
algebra $\mathfrak{g}$ of $G$ as indicated earlier in (\ref{sl(2)}). Now
that $\cup _{k=1}^{\infty }V^{k}$ $=$ $G$ since $G$ is connected (cf. \cite[%
p.181]{Ma}), it follows that $\omega $ is $G$-invariant. This amounts to the
inclusion $H^{0}(M,\Omega _{M}^{p})$ $\subset $ $H_{0,\sigma
_{M}^{G}}^{0}(M,\Omega _{M}^{p}).$ The converse is obvious. Hence $%
H^{0}(M,\Omega _{M}^{p})=H_{0,\sigma _{M}^{G}}^{0}(M,\Omega _{M}^{p})$ for
semisimple $G.$ Since it is well known that a connected complex reductive
group is the product of a connected semisimple group and $(\mathbb{C}^{\ast
})^{k}$ for some $k$ $\in $ $\mathbb{N}\cup \{0\}$ (e.g. \cite[p.168]{H}, 
\cite[p.21]{Mar})$,$ the similar reasoning as above (by using Proposition %
\ref{propA2}) concludes the proof.

\endproof%

\begin{remark}
\label{rext} According to \cite[Remarks II-C]{Som}, a \textit{good}
compactification $\bar{G}$ exists for $G$ that is reductive. Suppose $G$
acts \textit{projectively} (see \cite[p.107]{Som} for this definition) on a
compact K\"{a}hler manifold $M$ through $\sigma _{M}^{G}.$ Then $\sigma
_{M}^{G}$ extends meromorphically to $\bar{G}\times M$ by \cite[Proposition I%
]{Som}. For other examples of the meromorphic extension, see Remark \ref%
{mero}.
\end{remark}

\section{\textbf{Complex manifolds with two holomorphic }$\mathbb{C}^{\ast }$%
\textbf{-actions\label{Sec10}}}

The goal of this section is to prove Theorem \ref{BM0} stated in the
Introduction. This result combined with those of Section 8 proves Corollary %
\ref{Cor4-1}. The main assertion $H^{0}(M,\Omega _{M}^{p})=H_{0,\sigma
_{M}}^{0}(M,\Omega _{M}^{p})$ in Proposition \ref{propA2} (where $\sigma
_{M} $ admits meromorphic extension) becomes a special case of $%
H^{0}(B,\Omega _{B}^{p})\cong H_{0,\sigma _{M}}^{0}(M,\Omega _{M}^{p})$ in
Theorem \ref{BM0}, which is to be discussed below. See also the paragraph
prior to Corollary \ref{Cor4-1} in the Introduction, that refers to moduli
spaces with two fibrations; the results of this section might have
applications to these spaces.

Our strategy is to consider the globally free case first, i.e. a principal $%
\mathbb{C}^{\ast }$-bundle $P$ over a complex manifold $M$ (not necessarily
compact). We turn to the locally free case later.

Some setup first. Denote the standard holomorphic $\mathbb{C}^{\ast }$%
-action on $P$ by $\sigma _{s}$ and $M$ $=$ $P/\sigma _{s}.$ Let $\sigma
_{d} $ be another globally free holomorphic $\mathbb{C}^{\ast }$-action on $%
P $, which maps fibre to fibre of the fibration $\pi :P\rightarrow M$
(perhaps different fibres). Let $\sigma _{M}$ be the $\sigma _{d}$-induced
holomorphic $\mathbb{C}^{\ast }$-action on $M,$ i.e. the commutative
relation $\pi \circ \sigma _{d}(\lambda )$ $=$ $\sigma _{M}(\lambda )\circ
\pi $ holds for any $\lambda $ $\in $ $\mathbb{C}^{\ast }.$ Assume that $B$ $%
:=$ $P/\sigma _{d}$ is a complex manifold and the natural projection $\pi
_{2}:P$ $\rightarrow $ $B$ realizes $P$ as a principal $\mathbb{C}^{\ast }$%
-bundle on $B$ (via $\sigma _{d}).$

A typical situation that motivates us is the one in Section \ref{Sec9}; here 
$P$ $=$ $\mathbb{C}^{\ast }\times M$ as a trivial principal $\mathbb{C}%
^{\ast }$-bundle on $M,$ the \textquotedblleft diagonal" action $\sigma
_{d}(\lambda )(\xi ,p)$ $:=$ $(\lambda \xi ,$ $\sigma _{M}(\lambda )p)$ for
a given holomorphic $\mathbb{C}^{\ast }$-action $\sigma _{M}$ on $M,$ and $%
B=P/\sigma _{d}$ ($=\{[(\xi ,x)]$ $|$ $(\xi ,x)$ $\in $ $P\}).$ Define $\psi
:B$ $\rightarrow $ $M$ by $\psi ([(\xi ,x)])$ $=$ $\sigma _{M}(\xi ^{-1})x,$
and $\psi ^{-1}(y)=[(1,y)]$ $\in $ $B$ for $y$ $\in $ $M,$ i.e. $\psi $ is a
biholomorphism.

Fix a local trivialization $\mathbb{C}^{\ast }\times $ $U$ for $P$ $%
\rightarrow $ $M.$ We may write $1$ for the local holomorphic section $(1,$ $%
y)$ on $U$ $\subset $ $M$ if no confusion occurs$.$ Write $(\zeta ,y)$ $:=$ $%
(\zeta \circ _{s}1,y)$ in $\mathbb{C}^{\ast }\times U$ where $\zeta \circ
_{s}1$ $=$ $\sigma _{s}(\zeta )1$ for short$;$ similar notations
\textquotedblleft $\circ _{d}",$ $``\circ _{M}"$ also apply below. For $%
\lambda $ $\in $ $\mathbb{C}^{\ast },$ $\sigma _{s}$ acts by $\sigma
_{s}(\lambda )(\zeta \circ _{s}1,y)$ $=$ $((\lambda \zeta )\circ _{s}1,y)$
or $\lambda \circ _{s}(\zeta ,y)=(\lambda \zeta ,y).$

\textbf{Throughout this section we assume that }$\sigma _{d}$\textbf{\
commutes with }$\sigma _{s},$ i.e. $\sigma _{d}(\lambda )\circ \sigma
_{s}(\zeta )$ $=$ $\sigma _{s}(\zeta )\circ \sigma _{d}(\lambda )$ or $%
\lambda \circ _{d}(\zeta \circ _{s}q)$ $=$ $\zeta \circ _{s}(\lambda \circ
_{d}q)$ for $q\in P,$ $\lambda ,$ $\zeta $ $\in $ $\mathbb{C}^{\ast }.$ We
say that $\sigma _{d}$ is \textit{degenerate} if $\sigma _{d}(\lambda
)(\zeta ,y)$ $=$ $(\zeta ,\lambda \circ _{M}y)$ for $\lambda $ close to $1$
in some local chart. Otherwise $\sigma _{d}$ is said to be \textit{%
nondegenerate.} This definition is easily seen to be intrinsic. If $\sigma
_{d}$ is nondegenerate, we claim the existence of $0\neq l$ $\in $ $\mathbb{Z%
},$ such that 
\begin{equation}
\lambda \circ _{d}(\zeta ,y)=(\lambda ^{l}\zeta ,\lambda \circ _{M}y).
\label{9-0}
\end{equation}%
\noindent Here we assume that $\lambda $ is close to $1$ to ensure that the
image of $\sigma _{d}(\lambda )$ is contained in the same trivialization.
Write $\lambda \circ _{d}(\zeta ,y)=(\Phi _{\lambda }(\zeta ),\lambda \circ
_{M}y).$ Then it is not hard to see%
\begin{equation}
\Phi _{\lambda _{1}\lambda _{2}}(\zeta )=\Phi _{\lambda _{1}}(\Phi _{\lambda
_{2}}(\zeta ))\text{ }(=\Phi _{\lambda _{2}}(\Phi _{\lambda _{1}}(\zeta ))),%
\text{ }\Phi _{\lambda }(\zeta _{1}\zeta _{2})=\zeta _{1}\Phi _{\lambda
}(\zeta _{2}).  \label{9-0.5}
\end{equation}%
\noindent Note that the second equality of (\ref{9-0.5}) follows from the
commutativity of $\sigma _{d}$ and $\sigma _{s}:$%
\begin{eqnarray*}
(\Phi _{\lambda }(\zeta ),\lambda \circ _{M}y) &=&\lambda \circ _{d}(\zeta
,y)=\lambda \circ _{d}(\zeta \circ _{s}(1,y)) \\
&=&\zeta \circ _{s}(\lambda \circ _{d}(1,y))=\zeta \circ _{s}((\Phi
_{\lambda }(1),\lambda \circ _{M}y) \\
&=&(\zeta \Phi _{\lambda }(1),\lambda \circ _{M}y).
\end{eqnarray*}%
\noindent Letting $\zeta =1$ in the first equality of (\ref{9-0.5}) gives 
\begin{equation*}
\Phi _{\lambda _{1}\lambda _{2}}(1)=\Phi _{\lambda _{1}}(\Phi _{\lambda
_{2}}(1))=\Phi _{\lambda _{2}}(1)\Phi _{\lambda _{1}}(1),
\end{equation*}%
\noindent so the map $\lambda \rightarrow \Phi _{\lambda }(1)$ is a
holomorphic character. It follows (see Remark \ref{10-2-1} below) that $\Phi
_{\lambda }(1)$ $=$ $\lambda ^{l}$ for $0\neq l$ $\in $ $\mathbb{Z}$ if $%
\sigma _{d}$ is nondegenerate.

\begin{remark}
\label{10-2-1} In the reasoning above, $\lambda $ is originally close to $1.$
Under different local trivializations $1$ and $1^{a},$ for the same point $%
\zeta \circ _{s}1$ and $\zeta ^{a}\circ _{s}1^{a}$ lying over $y\in M,$ we
have the same point $\Phi _{\lambda }(\zeta )\circ _{s}1$ and $\Phi
_{\lambda }^{a}(\zeta ^{a})\circ _{s}1^{a}$ lying over $\lambda \circ _{M}y$ 
$\in $ $M.$ Writing $1^{a}$ $=$ $\mu 1$ thus $\zeta $ $=$ $\zeta ^{a}\mu $
and using $\Phi _{\lambda }^{a}(\zeta \mu ^{-1})$ $=$ $\mu ^{-1}\Phi
_{\lambda }^{a}(\zeta )$ (\ref{9-0.5}) one sees that%
\begin{equation}
\Phi _{\lambda }(\zeta )=\Phi _{\lambda }^{a}(\zeta ).  \label{10.2-5}
\end{equation}%
From $\Phi _{\lambda }$ $=$ $\Phi _{\lambda }^{a}$ in (\ref{10.2-5}) it
follows that $\Phi _{\lambda }$ can be extended to every $\lambda $ $\in $ $%
\mathbb{C}^{\ast }$ even if $\lambda \circ _{M}y$ may leave the original
trivialization. In sum (\ref{9-0.5}) remains valid for all $\lambda _{1},$ $%
\lambda _{2}$ $\in $ $\mathbb{C}^{\ast }.$
\end{remark}

We will sometimes drop the subscripts $``s",$ $``d"$ and $``M"$ in $``\circ
_{s}",$ \textquotedblleft $\circ _{d}"$ and $``\circ _{M}"$ respectively if
there is no danger of confusion in the context.

Suppose that $\sigma _{M}$ extends meromorphically to $\mathbb{C}\mathbb{P}%
^{1}\times M$ - - -\TEXTsymbol{>} $M.$ As usual, $H^{0}(B,\Omega _{B}^{p})$
denotes the space of holomorphic $p$-forms on $B.$ Take $\omega _{B}$ $\in $ 
$H^{0}(B,\Omega _{B}^{p}).$ Consider the global holomorphic $p$-form $\pi
_{2}^{\ast }\omega _{B}$ on $P$ (where $\pi _{2}:P$ $\rightarrow $ $B).$




To proceed further, some preparations are in order. From the proof of
Proposition \ref{propA2}, we can localize the reasoning and restrict
ourselves to a local trivialization $\mathbb{C}^{\ast }\times $ $U$ of $P$ $%
\rightarrow $ $M.$ The form $\pi _{2}^{\ast }\omega _{B}|_{\mathbb{C}^{\ast
}\times U}$ can now be expanded at $\zeta $ $=$ $0$ $\in $ $\mathbb{C}%
\mathbb{P}^{1}$ as before; this is dependent only on the local
trivialization $\mathbb{C}^{\ast }\times U$ hence on $\zeta $ and
independent of local coordinates at $x$ $\in $ $M.$ With the previous
notation we write%
\begin{equation}
(\pi _{2}^{\ast }\omega _{B})(\zeta \circ _{s}1,x)=\sum_{k=0}^{\infty }\zeta
^{k}\tilde{\omega}_{k}(1,x,dx,d\zeta )  \label{A9}
\end{equation}

\noindent where $\tilde{\omega}_{k}$ is a holomorphic $p$-form in $\zeta $
and $x,$ whose coefficients are independent of $\zeta .$

\begin{remark}
\label{10-4-1} An intrinsic description for the regularity (\ref{A9}) is the
following. Regarding the principal $\mathbb{C}^{\ast }$-bundle $P$ $%
\rightarrow $ $M$ as seated inside the associated holomorphic line bundle $L$
$\rightarrow $ $M$ (cf. (\ref{10.30-5}) below). The above regularity (\ref%
{A9}) is the same as to say that $\pi _{2}^{\ast }\omega _{B}$ is regular at
the zero section of $L.$
\end{remark}

Define the projection Proj$_{x}$ by dropping the terms involving $d\zeta :$%
\begin{equation}
\omega _{k}(1,x,dx):=\text{Proj}_{x}\tilde{\omega}_{k}(1,x,dx,d\zeta ),\text{
}k=0,1,\cdot \cdot \cdot .  \label{9-1.1}
\end{equation}

\noindent Note that in $\omega _{k}(1,x,dx)$ there is no dependence on $%
\zeta $ and $d\zeta .$ As Proj$_{x}$ is presumably coordinate-dependent (on $%
\zeta $ as aforementioned), we are going to examine the patching property of 
$\omega _{k}(1,x,dx)$ below.

\begin{lemma}
\label{10.0} As given in (\ref{9-1.1}), the $p$-form $\omega _{0}(1,x,dx)$
is globally defined on $M,$ and for every $k\geq 1$ ($\tilde{\omega}%
_{k}-\omega _{k})/d\zeta $ induces a global $L^{-(k+1)}$-valued $(p-1)$-form 
$\mathfrak{f}_{k}$ on $M,$ where $L$ is some holomorphic line bundle on $M$
(in fact it is the same $L$ as the one in Remark \ref{10-4-1})$.$
\end{lemma}

\proof
For later use, the following proof discusses more than what is needed in the
lemma. Let $1^{\prime }$ be another local holomorphic section such that $%
(\zeta \circ _{s}1,x)=(\zeta ^{\prime }\circ _{s}1^{\prime },x).$ Write $%
1^{\prime }=c(x)^{-1}\circ _{s}1$ for some local (nowhere vanishing)
holomorphic function $c(x),$ so $\zeta ^{\prime }=c(x)\zeta .$ Write $\pi
_{2}^{\ast }\omega _{B}(\zeta \circ _{s}1,x)$ $=$ $\pi _{2}^{\ast }\omega
_{B}(\zeta ^{\prime }\circ _{s}1^{\prime },x),$ and thus%
\begin{equation}
\sum_{k=0}^{\infty }\zeta ^{k}\tilde{\omega}_{k}(1,x,dx,d\zeta
)=\sum_{k=0}^{\infty }(\zeta ^{\prime })^{k}\tilde{\omega}_{k}^{\prime
}(1^{\prime },x,dx,d\zeta ^{\prime }).  \label{A11}
\end{equation}

\noindent Via (\ref{9-1.1}) one rewrites $\tilde{\omega}_{k}:$%
\begin{equation}
\tilde{\omega}_{k}(1,x,dx,d\zeta )=\omega _{k}(1,x,dx)+\eta
_{k}(1,x,dx,d\zeta )  \label{A16}
\end{equation}

\noindent where $\eta _{k}$ is of the form (with $J$ denoting the
multi-indices $I_{p-1})$%
\begin{equation}
\eta _{k}(1,x,dx,d\zeta )=\sum_{J}f_{k,J}(x)dx^{J}\wedge d\zeta .
\label{9.3b}
\end{equation}

\noindent Similar notation with a \textquotedblleft prime" applies to $%
\tilde{\omega}_{k}^{\prime },$ $\eta _{k}^{\prime }.$

Substituting $d\zeta ^{\prime }$ $=$ $\zeta c_{j}(x)dx^{j}$ $+$ $c(x)d\zeta $
(by differentiating $\zeta ^{\prime }$ $=$ $c(x)\zeta )$ into $\eta
_{k}^{\prime }$ we obtain%
\begin{equation}
\eta _{k}^{\prime }(1^{\prime },x,dx,d\zeta ^{\prime })=\zeta \sum_{J,\text{ 
}j}f_{k,J}^{\prime }c_{j}(x)dx^{J}\wedge dx^{j}+c(x)\sum_{J}f_{k,J}^{\prime
}dx^{J}\wedge d\zeta ,  \label{9.3c}
\end{equation}

\noindent and hence by (\ref{A11}) for the terms involving $\zeta ^{k}d\zeta 
$ 
\begin{equation}
\sum_{J}f_{k,J}(x)dx^{J}\wedge d\zeta =c(x)^{k+1}\sum_{J}f_{k,J}^{\prime
}(x)dx^{J}\wedge d\zeta .  \label{9.3d}
\end{equation}

\noindent It follows from (\ref{9.3d}) that 
\begin{equation}
f_{k,J}^{\prime }(x)dx^{J}=c(x)^{-(k+1)}f_{k,J}(x)dx^{J}.  \label{9.3e}
\end{equation}

Let $L^{-(k+1)}$ be the holomorphic line bundle on $M$ associated to the
transition function $c(x)^{k+1}.$ By (\ref{9.3e}) the collection 
\begin{equation}
\mathfrak{f}_{k}(p):=\left\{ \sum_{J}f_{k,J}(x)dx^{J}\right\}  \label{9.3e1}
\end{equation}%
\noindent for $p\in M$ with coordinates $x$ $=$ $(x^{j})$ is a global $%
L^{-(k+1)}$-valued $(p-1)$-form on $M.$

Similar to (\ref{9.3d}), by (\ref{A11}) for the terms involving $\zeta ^{k}$
(but excluding $\zeta ^{k}d\zeta )$ one has (see also (\ref{9.3c}))%
\begin{eqnarray}
&&\ \ \ \ \ \omega _{0}(1,x,dx)=\omega _{0}^{\prime }(1^{\prime },x,dx),
\label{9.3f} \\
&&\omega _{k}(1,x,dx)=c(x)^{k}\omega _{k}^{\prime }(1^{\prime
},x,dx)+c(x)^{k-1}\sum_{J,\text{ }j}f_{(k-1),J}^{\prime
}c_{j}(x)dx^{J}\wedge dx^{j},\text{ }k\geq 1.  \notag
\end{eqnarray}

\noindent The first equality of (\ref{9.3f}) shows that $\omega _{0}(1,x,dx)$
is a globally-defined holomorphic $p$-form on $M$. The statement about ($%
\tilde{\omega}_{k}-\omega _{k})/d\zeta $ is from (\ref{A16}), (\ref{9.3b})
and (\ref{9.3e}).

\endproof%

Recall that $H_{0}^{0}(M,\Omega _{M}^{p})$ denote the space of $\sigma _{M}$%
-invariant holomorphic $p$-forms on $M$ (cf. Section 8). The crucial result
of this section is the following:

\begin{proposition}
\label{propA3} Let $P$ be a principal $\mathbb{C}^{\ast }$-bundle over a
complex manifold $M$, not necessarily compact. Assume that there exists a
finite open covering $\{U^{a}\}_{a}$ of $M$ such that $P$ is holomorphically
trivial over a neighborhood $V^{a}$ of the closure $\overline{U^{a}}.$ Let $%
\sigma _{d}$ be a nondegenerate, i.e. $l\neq 0$ in (\ref{9-0}), globally
free holomorphic $\mathbb{C}^{\ast }$-action on $P$ which maps fibre to
fibre of the fibration $P\rightarrow M$ (perhaps different fibres), and
induces a holomorphic $\mathbb{C}^{\ast }$-action $\sigma _{M}$ on $M.$
Suppose that $\sigma _{d}$ commutes with $\sigma _{s}$ (the standard
holomorphic $\mathbb{C}^{\ast }$-action on $P)$ and $\sigma _{M}$ extends
meromorphically to $\mathbb{C}\mathbb{P}^{1}\times M$ \ - - -\TEXTsymbol{>} $%
M$ (see Remark \ref{mero}). Let $B$ $:=$ $P/\sigma _{d},$ possibly noncompact%
$.$ Then the following map from (\ref{A9}) and (\ref{9-1.1}) in local
coordinates: 
\begin{equation}
\psi :\omega _{B}\in H^{0}(B,\Omega _{B}^{p})\rightarrow \omega
_{0}(1,x,dx)\in H_{0}^{0}(M,\Omega _{M}^{p}),\text{ }p=0,1,2,\cdot \cdot
\cdot  \label{9-4.5}
\end{equation}%
is globally well defined and gives a linear isomorphism.
\end{proposition}

\begin{remark}
\label{9.2-5} If $\sigma _{d}$ is degenerate, the assertion of the
proposition may fail. For the trivial product $\mathbb{C}^{\ast }\times M$ $%
=:$ $P$ with $\sigma _{d}(\lambda )(\xi ,$ $x)$ $=$ $(\xi ,$ $\sigma
_{M}(\lambda )x)$ we have $\omega _{B}$ $:=$ $d\xi $ for $p=1$ projecting to
zero under $\psi .$ The argument in the following proof breaks down when $l$
in (\ref{10d}) equals $0$ (i.e. $\sigma _{d}$ is degenerate), and in this
case no vanishing of $\eta _{k}$ ($k$ $\geq $ $0)$ is guaranteed (see Lemma %
\ref{10.1-6}); the foregoing $d\xi $ $\neq $ $0$ corresponds $\eta _{k=0}).$
\end{remark}

\proof
\textbf{(of Proposition \ref{propA3})} By Lemma \ref{10.0} that $\psi
(\omega _{B})$ in (\ref{9-4.5}) is a globally defined $p$-form on $M$, it
remains to show that it is $\sigma _{M}$-invariant. Writing $\pi _{2}^{\ast
}\omega _{B}$ $=$ $\sigma _{d}(\lambda )^{\ast }\pi _{2}^{\ast }\omega _{B}$
on the LHS of (\ref{A9}) via $\pi _{2}\circ \sigma _{d}(\lambda )$ $=$ $\pi
_{2}$ and applying (\ref{A9}) again$,$ we have%
\begin{eqnarray}
&&\sum_{k=0}^{\infty }\zeta ^{k}\tilde{\omega}_{k}(1,x,dx,d\zeta )=(\sigma
_{d}(\lambda )^{\ast }\pi _{2}^{\ast }\omega _{B})(\zeta ,x)  \label{9-1.5}
\\
&=&\sigma _{d}(\lambda )^{\ast }((\pi _{2}^{\ast }\omega _{B})(\lambda
^{l}\zeta ,\lambda \circ x))  \notag \\
&\overset{(\ref{A9})}{=}&\sum_{k=0}^{\infty }(\lambda ^{l}\zeta )^{k}\sigma
_{d}(\lambda )^{\ast }(\tilde{\omega}_{k}(1,\lambda \circ x,d(\lambda \circ
x),d(\lambda ^{l}\zeta )))  \notag \\
&=&\sum_{k=0}^{\infty }(\lambda ^{l}\zeta )^{k}(\sigma _{d}(\lambda )^{\ast }%
\tilde{\omega}_{k})(1,x,dx,d\zeta ).  \notag
\end{eqnarray}

\noindent It follows from (\ref{9-1.5}) that 
\begin{equation}
(\sigma _{d}(\lambda )^{\ast }\tilde{\omega}_{k})(1,x,dx,d\zeta )=\lambda
^{-lk}\tilde{\omega}_{k}(1,x,dx,d\zeta ).  \label{A9-1}
\end{equation}

\noindent By (\ref{A9-1}), (\ref{A16}), (\ref{9.3b}), (\ref{9-0}) and (\ref%
{9-1.1}), it is not difficult to convince oneself that, with $\omega _{k}:=$%
Proj$_{x}\tilde{\omega}_{k},$%
\begin{equation}
(\sigma _{M}(\lambda )^{\ast }\omega _{k})(1,x,dx)=\lambda ^{-lk}\omega
_{k}(1,x,dx).  \label{A10}
\end{equation}

\noindent It is tempting to let $\lambda \rightarrow 0$ (resp. $\infty $)
for $l>0$ (resp. $l<0$) in (\ref{A10}) and conclude that%
\begin{equation}
\omega _{k}(1,x,dx)\equiv 0\text{ \ \ }(k\geq 1)  \label{A10-1}
\end{equation}%
\noindent since the LHS of (\ref{A10}) is finite, which follows from the
assumption on the meromorphic extension of $\sigma _{M}$ and Hartogs'
extension theorem.

However, for $k$ $\geq $ $1$ (\ref{A10}) is meaningful only for $\lambda $
close to $1$ since \textit{a priori }$\omega _{k}(1,x,dx)$ is local (see the
second equality of (\ref{9.3f})) and $\sigma _{M}(\lambda )$ may carry away
the local chart if $\lambda \nsim 1$. For $k=0$ (\ref{A10}) together with
the first equality of (\ref{9.3f}) did show that $\omega _{0}(1,x,dx)$ is a
global $\sigma _{M}$-invariant holomorphic $p$-form on $M.$ So the map $\psi 
$ in (\ref{9-4.5}) of this proposition is now well defined.

The claim (\ref{A10-1}) enters into the proof that $\psi $ in (\ref{9-4.5})
is an isomorphism. To prove (\ref{A10-1}) rigorously, the idea is to extend
the action $\sigma _{M}(\lambda )$ in (\ref{A10}) to all $\lambda $ $\in $ $%
\mathbb{C}^{\ast }$ (then taking $\lambda \rightarrow 0$ as reasoning above).

\begin{lemma}
\label{10.1-5} For every $k\geq 1,$ $\omega _{k}(1,x,dx)$ in (\ref{9.3f})
represents a global $L^{-k}$-valued holomorphic $p$-form on $M,$ where the
holomorphic line bundle $L^{-k}$ $\rightarrow $ $M$ is as similar to (\ref%
{9.3e1}) and admits a natural $\mathbb{C}^{\ast }$-action $\tilde{\sigma}%
_{M,-k}$ compatible with $\sigma _{M}.$
\end{lemma}

Lemma \ref{10.1-5} will be obtained from the following vanishing result:

\begin{lemma}
\label{10.1-6} (Vanishing result (I)) In (\ref{9.3f}) one has $%
\sum_{J}f_{k,J}(x)dx^{J}$ $\equiv $ $0,$ $k\geq 0.$ Thus $\eta _{k}$ in (\ref%
{9.3b}) vanishes identically.
\end{lemma}

\proof
(of Lemma \ref{10.1-6}) For a fixed $k\geq 0$ write $\varpi
_{k}:=\sum_{J}f_{k,J}(x)dx^{J}$ (recalling that this local $|J|$-form is
independent of local coordinates at $x,$ cf. remarks prior to (\ref{A9}))
and denote by $L$ $\rightarrow $ $M$ the line bundle associated with $P$ $%
\rightarrow $ $M$ (so that $e^{\prime }$ $=$ $c(x)^{-1}e$ with the
transition function $c(x)$ already specified, cf.$\,$(\ref{9.3e})); see also
(\ref{10.30-5}) below. We claim that the $\mathbb{C}^{\ast }$-action $\sigma
_{d}$ on $P$ induces a $\mathbb{C}^{\ast }$-action, still denoted by $\sigma
_{d},$ on $L.$ The action $\sigma _{d}$ has been explicated in (\ref{9-0}),
from which the claim follows. More generally, for every $q$ $\in $ $\mathbb{Z%
}$, $L^{q}$ $\rightarrow $ $M$ has a natural $\mathbb{C}^{\ast }$-action $%
\tilde{\sigma}_{M,q}(\lambda ):$ 
\begin{equation}
\tilde{\sigma}_{M,q}(\lambda )\text{ sending }e_{p}^{q}\in L_{p}^{q}\text{ }%
(p\in M)\text{ to }(\sigma _{d}(\lambda )e_{p})^{q}\in L_{\lambda \circ
_{M}p}^{q}  \label{9.18-5a}
\end{equation}%
\noindent induced by $\sigma _{d}(\lambda )$ $:$ $L$ $\rightarrow $ $L$
equivariant with respect to $\sigma _{M}(\lambda )$ (thus $\tilde{\sigma}%
_{M,1}$ $=$ $\sigma _{d}$ on $L$)$.$ Since $\sigma _{M}(\lambda )$ acts on
the bundle of holomorphic $p$-forms on $M$ via pull-back, a moment's thought
yields that $\tilde{\sigma}_{M,q}(\lambda )$ acts, still denoted by $\tilde{%
\sigma}_{M,q}(\lambda ),$ on the space of global sections under
consideration, namely global $L^{q}$-valued ($p-1)$-forms on $M.$ Here, a
typical situation occurs with the global section $\mathfrak{f}_{k}$ in (\ref%
{9.3e1}) ($L^{-(k+1)}$-valued) for $q$ $=$ $-(k+1).$ We have now that $%
\tilde{\sigma}_{M,q}(\lambda )\mathfrak{f}_{k}$ makes perfect sense for
every $\lambda $ $\in $ $\mathbb{C}^{\ast }.$

Let us first derive a scaling property for $\varpi _{k}$ (on some open
subset $U$ $\subset $ $M),$ regarded as local expressions of the global
object $\mathfrak{f}_{k}.$ From (\ref{A9-1}), (\ref{A10}) and (\ref{A16}) we
get%
\begin{equation}
\sigma _{d}(\lambda )^{\ast }\eta _{k}=\lambda ^{-lk}\eta _{k}\text{.}
\label{A17}
\end{equation}

\noindent for $\lambda $ near $1.$ Writing out (\ref{A17}) via (\ref{9-0})
and (\ref{9.3b}), one is in a position to compare the terms involving $%
d\zeta .$ Then it is not difficult to find, via $\Phi _{\lambda }^{\ast
}(d\zeta )$ $=$ $\lambda ^{l}d\zeta $, that%
\begin{equation}
(\sigma _{M}(\lambda )^{\ast }\varpi _{k})(x)=\lambda ^{-l(k+1)}\varpi
_{k}(x),\text{ }\lambda \sim 1.  \label{9-12.5}
\end{equation}

Considering $\tilde{\sigma}_{M,-(k+1)}(\lambda )\mathfrak{f}_{k}$ for $%
\lambda \sim 1,$ one infers from (\ref{9-12.5}) (with $\lambda $ replaced by 
$\lambda ^{-1}$ since $\sigma _{M}(\lambda )$ acts as $\sigma _{M}(\lambda
^{-1})^{\ast }$ on forms) and (\ref{9-0}) (see also (\ref{9.18-5a}) and
remarks below (\ref{10a})) that $\tilde{\sigma}_{M,-(k+1)}(\lambda )%
\mathfrak{f}_{k}$ $=$ $\mathfrak{f}_{k}$ ($\lambda \sim 1);$ slightly more
precisely the contribution from the form-part of $\mathfrak{f}_{k}$ acted on
by $\tilde{\sigma}_{M,-(k+1)}(\lambda )$ yields a factor ($\lambda
^{-1})^{-l(k+1)}$ while the contribution from the same action on the $%
L^{-(k+1)}$-part of $\mathfrak{f}_{k}$ gives another factor ($\lambda
^{l})^{-(k+1)}$ that cancels out the preceding one. As mentioned above $%
\tilde{\sigma}_{M,-(k+1)}(\lambda )\mathfrak{f}_{k}$ is well defined for
every $\lambda $ $\in $ $\mathbb{C}^{\ast }$. The crucial property due to
the analyticity in $\lambda $ then leads to the important conclusion:%
\begin{equation}
\tilde{\sigma}_{M,-(k+1)}(\lambda )\mathfrak{f}_{k}=\mathfrak{f}_{k}\text{
not only for }\lambda \sim 1\text{ but also for all }\lambda \in \mathbb{C}%
^{\ast }.  \label{10a}
\end{equation}

Denote by $\mathfrak{f}_{k}^{a}$ $=$ $\varpi _{k}^{a}s^{a}$ where $s^{a}$ :$%
= $ ($e^{a})^{-(k+1)}$ and $\varpi _{k}^{a}$ $=$ $\mathfrak{f}_{k}|_{V^{a}}$
for local expression on open charts $V^{a}$ $\subset $ $M.$ Fix $p\in U^{0}$ 
$\subset $ $M$ and write $\mathfrak{f}_{k}^{0}$ $=$ $\varpi _{k}^{0}s^{0}$
where $s^{0}$ $=$ ($e^{0})^{-(k+1)}$ around $p=p^{0}$ (if $U^{0}$ $=$ $U$
then $\varpi _{k}^{0}$ is the above $\varpi _{k});$ similar notation applies
on $U^{a},$ $\overline{U^{a}}$ $\subset $ $V^{a}.$ Two basic properties are
introduced as follows. If $\lambda \circ _{M}p$ $=$ $p^{(a)}$ $\in $ $U^{a},$
with $\tilde{\sigma}_{M}(\lambda ^{-1})s^{a}(p^{(a)})$ $=:$ $\tau (\lambda
^{-1})s^{0}(p)$ (the subscript $-(k+1)$ in $\tilde{\sigma}_{M,-(k+1)}$
dropped throughout) we obtain the first property:%
\begin{equation}
(\tilde{\sigma}_{M}(\lambda ^{-1})\mathfrak{f}_{k})(p)=(\sigma _{M}(\lambda
)^{\ast }(\varpi _{k}^{a}(p^{(a)})))\tau (\lambda ^{-1})s^{0}(p).
\label{10b}
\end{equation}

\noindent For the second property suppose $\nu \circ (\lambda \circ p)$ $=$ $%
\tilde{p}^{(a)}$ $\in $ $U^{a},$ i.e. $p$ $=$ $\lambda ^{-1}\circ (\nu
^{-1}\circ \tilde{p}^{(a)})$ with $\lambda \circ p$ $\in $ $U^{a}$ as above$%
. $ One sees that%
\begin{equation}
\tau (\lambda ^{-1}\nu ^{-1})=\nu ^{l(k+1)}\tau (\lambda ^{-1})  \label{10c}
\end{equation}

\noindent since $\nu ^{-1}\circ s^{a}$ $=$ $\nu ^{l(k+1)}s^{a}$ by (\ref{9-0}%
) and (\ref{9.18-5a}). Here we need not restrict ourselves to $\nu \sim 1$
although the set $\{\nu \in $ $\mathbb{C}^{\ast }$ $|$ $\nu \circ (\lambda
\circ p)$ $\in $ $U^{a}\}$ may be far from being a connected set. As the
reasoning is similar to Remark \ref{10-2-1}, we leave it to the reader.

The following arguments constitute a refined treatment of those in (\ref%
{A9-1})-(\ref{A10-1}). Choose $c\in \mathbb{C}^{\ast }$ with $|c|$ $\neq $ $%
1.$ Suppose $|c|$ $<$ $1$ and $l$ $>$ $0$ (the remaining three cases will be
similar). For the sequence $\{c^{i}\circ _{M}p\}_{i=1,2,\cdot \cdot \cdot }$
there exists some fixed open chart $U^{a}$ and a subsequence $%
\{c^{n(i)}\circ _{M}p\}_{i=1,2,\cdot \cdot \cdot }$ such that $%
\{c^{n(i)}\circ _{M}p\}_{i}$ $\subset $ $U^{a}$ by the finiteness assumption
on $\{U^{a}\}_{a}$ as stated in the proposition$.$ Denote $c^{n(1)}$ by $%
c_{1},$ $c^{n(2)-n(1)}$ by $c_{2},$ etc. and set $p_{i}$ $=$ $c_{i}\circ
c_{i-1}\circ $ $\cdot \cdot \circ $ $c_{1}\circ p$ $(=c^{n(i)}\circ p)$ $=$ $%
c_{i}\circ p_{i-1}$ $\in $ $U^{a}$ with $p_{0}$ $:=$ $p.$ For instance, by $%
c_{1}^{-1}\circ c_{2}^{-1}\circ p_{2}$ $=$ $p$ with $p_{2}$ $\in $ $U^{a}$
and from (\ref{10a}), (\ref{10b})%
\begin{equation*}
\varpi _{k}^{0}(p)s^{0}(p)=\mathfrak{f}_{k}^{0}(p)=(\tilde{\sigma}%
_{M}(c_{1}^{-1}c_{2}^{-1})\mathfrak{f}_{k})(p)=\sigma _{M}(c_{2}c_{1})^{\ast
}(\varpi _{k}^{a}(p_{2}))\tau (c_{1}^{-1}c_{2}^{-1})s^{0}(p).
\end{equation*}

\noindent This yields $\sigma _{M}(c_{2}c_{1})^{\ast }(\varpi
_{k}^{a}(p_{2}))$ $=$ $\tau (c_{1}^{-1}c_{2}^{-1})^{-1}\varpi _{k}^{0}(p)$
and in turn, using (\ref{10c}) for $\tau (c_{1}^{-1}c_{2}^{-1})^{-1},$ $%
\sigma _{M}(c_{2}c_{1})^{\ast }(\varpi _{k}^{a}(p_{2}))$ $=$ $%
c^{(n(1)-n(2))l(k+1)}\tau (c_{1}^{-1})^{-1}\varpi _{k}^{0}(p).$

Similarly, for $i\geq 2$ one has 
\begin{equation*}
\sigma _{M}((c_{i}c_{i-1}\cdot \cdot \cdot c_{2})c_{1})^{\ast }(\varpi
_{k}^{a}(p_{i}))=c^{(n(1)-n(i))l(k+1)}\tau (c_{1}^{-1})^{-1}\varpi
_{k}^{0}(p).
\end{equation*}%
\noindent It is trivial that $\sigma _{M}((c_{i}c_{i-1}\cdot \cdot \cdot $ $%
c_{2})c_{1})^{\ast }(\varpi _{k}^{a}(p_{i}))$ $=$ ($\sigma
_{M}((c_{i}c_{i-1}\cdot \cdot \cdot $ $c_{2})c_{1})^{\ast }\varpi
_{k}^{a})(p).$ Thus%
\begin{equation}
(\sigma _{M}(c^{n(i)-n(1)}c_{1})^{\ast }\varpi
_{k}^{a})(p)=c^{(n(1)-n(i))l(k+1)}\tau (c_{1}^{-1})^{-1}\varpi _{k}^{0}(p)%
\text{ }(i=2,3,\cdot \cdot \cdot ).  \label{10d}
\end{equation}%
\noindent As an illustration, suppose $p,$ $c\circ p$ $\in $ $U^{0}$ for the
simplest case (corresponding to taking $n(1)$ $=$ $0$ and $n(2)$ $=$ $1).$
Then 
\begin{equation}
(\sigma _{M}(c)^{\ast }\varpi _{k}^{0})(p)=c^{-l(k+1)}\varpi _{k}^{0}(p)
\label{10e}
\end{equation}

\noindent reproducing (\ref{9-12.5}) above.

Now a contradiction follows from (\ref{10d}) by letting $i\rightarrow \infty 
$ and using the similar argument indicated in (\ref{A10-1}) since $|c|$ $<$ $%
1$ and $l$ $>$ $0$ by assumption. For other cases one may consider $%
\{c^{-n(i)}\circ p\}_{i=1,2,\cdot \cdot \cdot };$ we omit the details. Hence 
$\varpi _{k}\equiv 0$ as asserted by the lemma.

\endproof%

\proof
(\textbf{of Lemma \ref{10.1-5}}) By Lemma \ref{10.1-6} and (\ref{9.3f}), one
has 
\begin{equation}
\omega _{k}^{\prime }(1^{\prime },x,dx)=c(x)^{-k}\omega _{k}(1,x,dx)\text{ \ 
}k\geq 1  \label{wk}
\end{equation}

\noindent proving the first assertion of the lemma. The $\mathbb{C}^{\ast }$%
-action $\tilde{\sigma}_{M,-k}$ on $L^{-k}$ has been indicated in the proof
of Lemma \ref{10.1-6} (see (\ref{9.18-5a})).

\endproof%

Using the similar method as in the previous lemma one proves the analogous
statement:

\begin{lemma}
\label{10.2} (Vanishing result (II)) In (\ref{9.3f}), $\omega _{k}=\omega
_{k}(1,x,dx)\equiv 0,$ $k\geq 1.$
\end{lemma}

\proof
Recall that $\omega _{k}$ is a global $L^{-k}$-valued holomorphic $p$-form
on $M$ by Lemma \ref{10.1-5} (see (\ref{wk})). By comparing $\omega _{k}$ in
(\ref{wk}) with $\varpi _{k}$ in (\ref{9.3e}) and $\omega _{k}$ in (\ref{A10}%
) with $\varpi _{k}$ in (\ref{10e}) or (\ref{9-12.5}), it is perhaps a
tedious matter but not a difficult one to convince oneself that one can
formally run the arguments similar to the vanishing of $\varpi _{k}$ in
Lemma \ref{10.1-6} (which shall not be repeated here), giving the desired
vanishing of this lemma.

\endproof%

\textit{Proof of \textbf{Proposition \ref{propA3}} continued}: In (\ref{A16}%
) for $\tilde{\omega}_{k}(1,x,dx,d\zeta ),$ we obtain $\tilde{\omega}%
_{k}\equiv 0$ ($k\geq 1)$ from the vanishing results Lemmas \ref{10.1-6} and %
\ref{10.2} above. For $k=0,$ $\tilde{\omega}_{0}(1,x,dx,d\zeta )=\omega
_{0}(1,x,dx)$ since $\eta _{0}$ in (\ref{A16}) vanishes by Lemma \ref{10.1-6}%
. By these vanishings, we thus reduce (\ref{A9}) to

\begin{equation}
\pi _{2}^{\ast }\omega _{B}(\zeta \circ 1,x)=\omega _{0}(1,x,dx).
\label{A15}
\end{equation}

That the linear map $\psi $ of (\ref{9-4.5}) is well-defined has been shown
in the first half of the proof. The injectivity of $\psi $ follows from (\ref%
{A15}) since the pull-back $\pi _{2}^{\ast }$ of the surjective map $\pi
_{2} $ is injective$.$ It remains to prove the surjectivity of $\psi .$

Given an element $\omega _{0}$ $\in $ $H_{0}^{0}(M,\Omega _{M}^{p}),$ let $%
\tilde{\omega}_{B}$ $:=$ $\pi _{1}^{\ast }\omega _{0}$ on $P$ where $\pi
_{1} $ $:$ $P$ $\rightarrow $ $M$ is the natural projection, so that locally 
$\tilde{\omega}_{B}(\zeta \circ 1,x)$ $=$ $\omega _{0}(1,x,dx)$ on $P.$ We
claim $\sigma _{d}(\lambda )^{\ast }\tilde{\omega}_{B}$ $=$ $\tilde{\omega}%
_{B}.$ For, $\omega _{0}$ is $\sigma _{M}$-invariant by assumption and this
yields the claim:%
\begin{equation*}
\sigma _{d}(\lambda )^{\ast }\tilde{\omega}_{B}=\sigma _{d}(\lambda )^{\ast
}\pi _{1}^{\ast }\omega _{0}=\pi _{1}^{\ast }\sigma _{M}^{\ast }(\lambda
)\omega _{0}=\pi _{1}^{\ast }\omega _{0}=\tilde{\omega}_{B}
\end{equation*}%
\noindent using $\pi _{1}\circ \sigma _{d}(\lambda )$ $=$ $\sigma
_{M}(\lambda )\circ \pi _{1}$ by (\ref{9-0}). So $\tilde{\omega}_{B}$
descends on $P/\sigma _{d}=B$ to an element $\omega _{B}$ $\in $ $%
H^{0}(B,\Omega _{B}^{p})$ in the sense that $\pi _{2}^{\ast }\omega _{B}$ $=$
$\tilde{\omega}_{B}.$ This, together with $\tilde{\omega}_{B}$ $=$ $\omega
_{0}(1,x,dx)$ as just mentioned, implies that $\omega _{0}$ lies in the
image of $\psi $ in (\ref{9-4.5}) (cf. (\ref{A9}), (\ref{9-1.1}) and (\ref%
{A15})). We have proved the surjectivity of $\psi ,$ and hence the
isomorphism of $\psi .$

\endproof%

The following remarks will be used in the proof of Theorem \ref{BM0}.

\begin{remark}
\label{PLMB} In the notation of Proposition \ref{propA3}, associated with $P$
is the holomorphic line bundle $L_{M}$ (resp. $L_{B}$) over $M$ (resp. $B$)
by%
\begin{equation}
L_{M}:=P\times _{\sigma _{s}}\mathbb{C}=P\times \mathbb{C}/\sim _{s}\text{ }(%
\text{resp. }L_{B}:=P\times _{\sigma _{d}}\mathbb{C}=P\times \mathbb{C}/\sim
_{d}\text{)}  \label{10.30-5}
\end{equation}%
\noindent where $(u,\zeta )\sim _{s}(\sigma _{s}(\lambda )u,\lambda
^{-1}\zeta )$ (resp. $(u,\zeta )\sim _{d}(\sigma _{d}(\lambda )u,\lambda
^{-1}\zeta )$) for $\lambda \in \mathbb{C}^{\ast }.$ By the map $%
u\rightarrow \lbrack (u,1)]$ via $\sigma _{s}$ (resp. $\sigma _{d})$ we have
an embedding $P$ $\cong $ $L_{M}\backslash \{0$-section\} (resp. $%
L_{B}\backslash \{0$-section\}) into $L_{M}$ (resp. $L_{B}).$ We have the
following linear isomorphisms:%
\begin{equation}
H_{m,\sigma _{s}}^{q}(P,\mathcal{O}_{P}\mathcal{)}\simeq
H^{q}(M,(L_{M}^{\ast })^{\otimes m}),\text{ }H_{m,\sigma _{d}}^{q}(P,%
\mathcal{O}_{P}\mathcal{)}\simeq H^{q}(B,(L_{B}^{\ast })^{\otimes m}).
\label{9.17-5}
\end{equation}%
This fact is nothing but a restatement of Proposition \ref{p-gue2-2} adapted
to the present context. Let us just be brief. Let $\Omega _{m,\sigma
_{s}}^{0,q}(P)$ (resp. $\Omega _{m,\sigma _{d}}^{0,q}(P)$) denote $\Omega
_{m}^{0,q}(P)$ (see Definition \ref{2m}) with respect to $\sigma _{s}$
(resp. $\sigma _{d}$). The map from $\eta \otimes (e^{\ast })^{\otimes m}$ $%
\in $ $\Omega ^{0,q}(M,(L_{M}^{\ast })^{\otimes m})$ (resp. $\Omega
^{0,q}(B,(L_{B}^{\ast })^{\otimes m})$) to $\omega $ $\in $ $\Omega
_{m,\sigma _{s}}^{0,q}(P)$ (resp. $\Omega _{m,\sigma _{d}}^{0,q}(P)$)$,$
defined by%
\begin{equation*}
\omega (z,\zeta e)=\eta \otimes (e^{\ast })^{\otimes m}(z,\zeta e)=\eta
(z)(e^{\ast }(\zeta e))^{m}=\eta (z)\zeta ^{m}
\end{equation*}%
induces a linear isomorphism between $H_{m,\sigma _{s}}^{q}(P,\mathcal{O}_{P}%
\mathcal{)}$ and $H^{q}(M,(L_{M}^{\ast })^{\otimes m})$ (resp. $H_{m,\sigma
_{d}}^{q}(P,\mathcal{O}_{P}\mathcal{)}$ and $H^{q}(B,(L_{B}^{\ast
})^{\otimes m})$)$.$
\end{remark}

\begin{remark}
\label{9.3} (Generalization to locally free case, used in Example \ref{E-IF}
too) Suppose $\sigma _{s}$, $\sigma _{d}$ on $P$ are only locally free in
the sense that they satisfy conditions as stated in Theorem \ref%
{main_theorem} except possibly the compactness of the quotient. Then the
isomorphisms in (\ref{9.17-5}) with $L_{M}^{\ast },$ $L_{B}^{\ast }$
considered as orbifold line bundles, remain valid. For, one can use
Proposition \ref{dualLm} via the $\mathbb{C}^{\ast }$-equivariant line
bundle $L_{\Sigma }^{\ast }$ (with $\Sigma $ $=$ $P$ here) and translate the 
$\mathbb{C}^{\ast }$-invariant condition there ($L_{\Sigma }^{\ast }$%
-valued) into an appropriate setting (cf. $L_{\Sigma }^{\ast }/\sigma ,$ $%
\sigma $ $=$ $\sigma _{s},$ $\sigma _{d})$ associated with orbifolds ($%
\Sigma /\sigma $ $=$ $M,$ $B)$ as $\mathbb{C}^{\ast }$-quotients of $\Sigma $
$=$ $P$ here. For the tensor product ($L_{M}^{\ast })^{\otimes m}$ of
orbifold line bundles, see for instance \cite[p.14]{ALR}. Compare the
introductory paragraph in the proof of Theorem \ref{BM0} after Lemma \ref%
{10.4-5} below.
\end{remark}

We turn now to the locally free case. The following lemma is perhaps well
known (at least for the topological setting with compact group action), but
we are unable to find it (i.e. the complex analytic setting with complex
group action) in the literature.

\begin{lemma}
\label{10.4-5} In the notation of Theorem \ref{BM0}, for the projection $\pi
_{1}:P$ $\rightarrow $ $M$ and small open subsets $V$ of $M,$ $\pi
_{1}^{-1}(V)$ $\subset $ $P$ is biholomorphically of the form ($\mathbb{C}%
^{\ast }\times \tilde{U})/\Gamma ,$ where $\Gamma $ $(\subset $ $S^{1}$ $%
\subset $ $\mathbb{C}^{\ast })$ is some finite group and $\tilde{U}$ is some
open domain in $\mathbb{C}^{\dim P-1}.$
\end{lemma}

\begin{remark}
\label{10-9-1} The arguments in the proof below may be simplified if one
uses Remark \ref{R-2-11}.
\end{remark}

\proof
\textbf{(of Lemma \ref{10.4-5})} We follow the last paragraph in the proof
of Theorem \ref{thm2-1} (with its $\Sigma $ set to be $P$ here). Given $x$ $%
\in $ $\tilde{U}$ and $h$ $\in $ $\Gamma $, where $\tilde{U}$ $\times $ $%
(-\delta ,\delta )$ $\times $ $\mathbb{R}^{+}$ is a sufficiently small
cone-like open neighborhood of a given point $p$ $=$ $(z,0,1)$ $\in $ $P$
with the isotropy group $\Gamma _{p}$ $=:$ $\Gamma $ $\subset $ $S^{1}$ $%
\subset $ $\mathbb{C}^{\ast },$ let $\tau (h)(x)$ $=$ $shx$ for some $s$
(depending on $x)$ near $1$ $\in $ $S^{1},$ and let $h$ act freely on $%
\mathbb{C}^{\ast }\times \tilde{U}$ by $\phi (h)$ $:$ $(\xi ,x)$ $%
\rightarrow $ $(\xi h^{-1}s^{-1},$ $\tau (h)(x))$ with $\tau $ as in (\ref%
{2-11a})$.$ It can be verified that $\phi (h_{1}h_{2})$ $=$ $\phi
(h_{1})\phi (h_{2})$ mainly because $\tau $ satisfies the similar group-law
property$.$

Form the quotient ($\mathbb{C}^{\ast }\times \tilde{U})/\Gamma $ via $\phi $
and define a holomorphic map $\mu $ $:$ ($\mathbb{C}^{\ast }\times \tilde{U}%
)/\Gamma $ $\rightarrow $ $P$ by $[(\xi ,x)]$ $\rightarrow $ $\xi \circ
_{\sigma _{1}}x$ $=:$ $\xi x.$ It comes down to proving that $\mu $ is
injective. This reduces to the assertion that for $p_{1}$ $=$ $(z_{1},0,1),$ 
$p_{2}$ $=$ $(z_{2},0,1)$ in $\tilde{U}$ $\times $ $\{0\}$ $\times $ $\{1\}$
and $\xi $ $\in $ $\mathbb{C}^{\ast }$ such that $\xi p_{1}$ $=$ $p_{2},$ $%
\xi $ must be of the form $\xi $ $=$ $hs$ for some $h$ $\in $ $\Gamma $ and $%
s$ near $1,$ i.e. $\tau (h)(p_{1})$ $=$ $p_{2}.$ This assertion follows from
the facts that in this case $\xi $ $\in $ $S^{1}$ by using Lemma \ref{A} $i)$
with (\ref{1-0}), and in turn that $\xi $ is near the isotropy group $\Gamma 
$ since $\xi p$ $\sim $ $p$ follows from the condition $\xi p_{1}$ $=$ $%
p_{2} $ where $p_{1}$ $\sim $ $p,$ $p_{2}$ $\sim $ $p$ ($\tilde{U}$ small
around $z=z(p)).$

\endproof%

\proof
\textbf{(of Theorem \ref{BM0}) }In the locally free case, from Lemma \ref%
{10.4-5} (and Theorem \ref{thm2-1}) it follows that $P\rightarrow M$ $=$ $%
P/\sigma _{1}$ via $\pi _{1}$ is regarded as the total space of an orbifold
principal $\mathbb{C}^{\ast }$-bundle on the complex orbifold $M.$ (We omit
the detailed check on the transition functions; see Remark \ref{9.3} for
some relevant background material.) Assume that $M$ is compact, or in case
it is noncompact assume that it can be covered by finitely many open subsets 
$V$ as specified in Proposition \ref{propA3}; that is, $\pi ^{-1}(V)$ $\cong 
$ ($\mathbb{C}^{\ast }\times \tilde{U})/\Gamma $ can be considered as local
trivializations in the sense of orbifold bundle. As usual the local sections
in the orbifold sense correspond to the $\Gamma $-invariant local sections
of $\mathbb{C}^{\ast }\times \tilde{U}$ $\rightarrow $ $\tilde{U}.$ With the
preceding assumption on $M,$ take now $B$ = $P/\sigma _{2}.$ Here it does
not concern us whether $B$ is compact or not. Then, after examination of the
preceding proof for the globally free case, this additional condition
imposed by the $\Gamma $-invariance does not obstruct the main lines of the
argument there. We leave the details to the reader.%

An alternative approach is to restrict oneself to the regular part $M_{reg}$
of $M$ where the action $\sigma _{1}$ on $\pi _{1}^{-1}(M_{reg})$ is
globally free (here the action $\sigma _{M}$ on $M$ induced by $\sigma _{2}$
also maps $M_{reg}$ to $M_{reg}$ using the commutativity of $\sigma _{1}$
and $\sigma _{2}$). Now that $M_{reg}$ is necessarily noncompact (as long as 
$M\neq M_{reg}),$ a straightforward use of Proposition \ref{propA3} on $%
M_{reg}$ would require a finite-covering condition on $M_{reg}$ as assumed
in the statement there. Imposing the assumption of this finiteness (e.g. $%
M_{reg}$ being quasi-projective with algebraic principal $\mathbb{C}^{\ast }$%
-bundle $\pi _{1}^{-1}(M_{reg})$ $\rightarrow $ $M_{reg})$ and using the
Hartogs extension on normal analytic spaces (see Theorem \ref{thm2-1}) for
the normality of $M)$ the holomorphic $p$-forms on $M_{reg}$ can
holomorphically extend to $M$ (see Footnote$^{2}$ in the Introduction for $p$%
-forms on orbifolds)$,$ leading to the desired isomorphism map (see the next
paragraph). We omit the details.

In sum, as in (\ref{9-4.5}) of Proposition \ref{propA3} the map $\omega _{B}$
$\rightarrow $ $\omega _{0}$ still gives us a linear isomorphism:%
\begin{equation}
H^{0}(B,\Omega _{B}^{p})\simeq H_{0,\sigma _{M}}^{0}(M,\Omega _{M}^{p}).
\label{9.18-5}
\end{equation}

\noindent We have proved (\ref{BM0-1}) of Theorem \ref{BM0}.

We turn now to (\ref{BM0-2}) of Theorem \ref{BM0}. Here we assume that $B$
is smooth, compact and K\"{a}hler. We have%
\begin{equation}
\sum_{p=0}^{\dim B}(-1)^{p}\dim H^{0}(B,\Omega _{B}^{p})=\sum_{p=0}^{\dim
B}(-1)^{p}\dim H^{p}(B,\mathcal{O}_{B})  \label{A17-1}
\end{equation}

\noindent since $H^{0}(B,\Omega _{B}^{p})$ $\cong $ $H^{p}(B,\mathcal{O}%
_{B}) $ in the K\"{a}hler case. By Remarks \ref{PLMB} and \ref{9.3} (with $%
m=0)$ identifying $H^{\ast }(B,\mathcal{O}_{B})$ and $H_{0,\sigma
_{2}}^{\ast }(P,\mathcal{O}_{P}\mathcal{)}$ so that $H_{0,\sigma _{2}}^{\dim
P}(P,\mathcal{O}_{P}\mathcal{)}$ $=$ $H^{\dim P}(B,\mathcal{O}_{B})$ which
is $0$ since $\dim P$ $=$ $1+\dim B,$ we have:%
\begin{equation}
\sum_{p=0}^{\dim B}(-1)^{p}\dim H^{p}(B,\mathcal{O}_{B})=\sum_{p=0}^{\dim
P}(-1)^{p}\dim H_{0,\sigma _{2}}^{p}(P,\mathcal{O}_{P}\mathcal{)}.
\label{A18}
\end{equation}

\noindent Now (\ref{BM0-2}) of Theorem \ref{BM0} follows from (\ref{9.18-5}%
), (\ref{A17-1}) and (\ref{A18}).

\endproof%

\bigskip

\proof
(\textbf{of Corollary \ref{Cor4-1}}) Theorem \ref{propA2-1} proved in the
preceding section and Theorem \ref{BM0} just proved clearly yield this
corollary provided that Remark \ref{mero} is used to take care of the
meromorphic extension condition needed in Theorem \ref{BM0}.

\endproof%

\end{document}